\documentclass[10pt]{smfart}

\RequirePackage[T1]{fontenc}
\RequirePackage{amsfonts,latexsym,amssymb}
\RequirePackage[frenchb]{babel}
\addto\extrasfrenchb{\bbl@nonfrenchitemize
\bbl@nonfrenchspacing}

\RequirePackage{mathrsfs}
\let\mathcal\mathscr

\usepackage{array}
\usepackage{float}

\usepackage[all]{xy}
\usepackage{url}
\usepackage{fge}
\newcommand{\moins}{\mathbin{\fgebackslash}}

\numberwithin{equation}{section}
\theoremstyle{plain}
\newtheorem{prop}[equation]{\propname}
\newtheorem{theo}[equation]{\theoname}
\newtheorem{conj}[equation]{\conjname}
\newtheorem{coro}[equation]{\coroname}

\newtheorem{lemm}[equation]{\lemmname}
\theoremstyle{definition}
\theoremstyle{remark}
\newtheorem{defi}[equation]{\definame}
\newtheorem{rema}[equation]{\remaname}
\newtheorem{exem}[equation]{\exemname}

\newtheorem{ques}[equation]{Question}

\usepackage[pagebackref]{hyperref}

\let\cal\mathcal
\let\goth\mathfrak

\newcommand{\eet}{\operatorname{\acute{e}t} }
 \newcommand{\proet}{\operatorname{pro\acute{e}t} }

\def\nekovar{Nekov\'a\v{r}}
\def\paskunas{Pa\v{s}k\={u}nas}

\def\oGamma{{\overline\Gamma}}
\def\rg{{\rm R}\Gamma}

\def\FGet{{\Phi\Gamma^{\rm et}}}

\def\point{{\cdot}}

\def\Qbar{\overline{\bf Q}}
\def\Q{{\bf Q}} \def\Z{{\bf Z}}
\def\C{{\bf C}}
\def\N{{\bf N}}
\def\R{{\bf R}}
\def\O{{\cal O}}
\def\dual{{\boldsymbol *}}
\def\bmu{{\boldsymbol\mu}}

\def\Qbar{{\overline{{\bf Q}}}}
\def\Kbar{{\overline{K}}}
\def\epsilon{\varepsilon}

\def\oed#1{{\oe^{\dagger,#1}}}

\def\bdr{{\bf B}_{{\rm dR}}}
\def\Bdr{{\mathbb B}_{\rm dR}}

\def\acris{{\bf A}_{{\rm cris}}}

 \def\A{{\bf A}} \def\B{{\bf B}}
\def\tE{\widetilde{\bf E}} \def\tA{\widetilde{\bf A}} \def\tB{\widetilde{\bf B}}
\def\tbA{\widetilde{\mathbb A}}
\def\Ai{{\bf A}^{]\infty[}} \def\Aidu{{\bf A}^{]\infty[,\dual}}
\def\Aip{{\bf A}^{]\infty,p[}}
\def\Aipdu{{\bf A}^{]\infty,p[,\dual}}

\def\edag{{\cal E}^\dagger}
\def\oe{\O_{\cal E}}

\def\piqp{{{\bf P}^1}}

\def\reso{{\text{r\'es}_0}}
\def\resl{{\text{r\'es}_{t=0}}}

\def\matrice#1#2#3#4{{\big(\begin{smallmatrix}#1&#2\\ #3&#4\end{smallmatrix}\big)}}
\def\vecteur#1#2{{\big(\begin{smallmatrix}#1\\#2\end{smallmatrix}\big)}}
\def\wotimes{\widehat{\otimes}}

\def\lra{\longrightarrow}
\newcommand{\hooklra}{\lhook\joinrel\longrightarrow}

\def\iO{\underrightarrow{{\cal O}}{\phantom{}}}
\def\iH{\underrightarrow{H}{\phantom{}}}
\def\pH{\underleftarrow{H}{\phantom{}}}

\let\emptyset\varnothing

\def\G{{\cal G}} 
\def\GG{{\mathbb G}} \def\BB{{\mathbb B}} \def\PP{{\mathbb P}} \def\UU{{\mathbb U}}
\def\TT{{\mathbb T}}

\def\dalg{{\cal D}_{\rm alg}}
\def\cdo{{\rm Mes}}
\def\cy{{\boldsymbol\epsilon}_\A}
\def\cyp{{\boldsymbol\epsilon}_p}

\def\cZ{{\widehat\Z}}
\def\cZp{{\widehat\Z^{]p[}}}
\def\wGamma{{\widehat\Gamma}}

\def\premiers{${\cal P}$}
\def\ZZZ{{$\cZ$, $\Z_S$, $\Z^{]S[}$, $\Z[\frac{1}{p}]$}}
\def\QQQ{{$\Q_S$, $\Q^{\rm cycl}$}}
\def\adeles{{$\A$, $\A^{]\infty[}$, $\A^{]S[}$, $\delta_\A$, $|\ |_\A$}}
\def\exponent{{${\bf e}_\A$, ${\bf e}_\infty$, ${\bf e}_\ell$}}
\def\caract{{$\chi^{(p)}$, $\chi_{\rm Gal}$}}
\def\cyclo{{$\cy$, $\cyp $}}
\def\gauss{{$G(\chi)$}}
\def\galois{{$G_K$, $G_K^{\rm ab}$, $G_{\Q,S}$}}
\def\weil{{${\rm W}_{\Q_\ell}$, ${\rm WD}_{\Q_\ell}$}}
\def\GGG{{$\GG$, $G$}}
\def\BBB{{$\BB$}}
\def\PPP{{$\PP$, $P$, $P^+$}}
\def\UUU{{$\UU$, $U$}}
\def\HHH{{${\cal H}$, ${\cal H}^\pm$}}
\def\Gam{{$\wGamma(N)$, $\wGamma_0(N)$, $\Gamma(N)$, $\Gamma_0(N)$}}
\def\signa{{${\rm sign}$}}
\def\omegapi{{$\omega_\pi$}}
\def\vpi{{$v_\pi$, $v'_\pi$, $v_{\pi,\ell}$, $v_\pi^{]\ell[}$, $v_\pi^{]S[}$}}
\def\Tp{{$T_\ell$, $T_p$}}
\def\Pip{{$\Pi_p$, $\Pi_p^{\rm alg}$, $\Pi_\ell^{\rm cl}$}}
\def\PIQP{{$\piqp$, ${\rm Div}^0(\piqp)$}}
\def\WWW{{$W_{k,j}$, $W_{k,j}^\dual$, $W_{\rm tot}$, $W_k$}}
\def\eee{{$e_1$, $e_2$, $e_1^\dual$, $e_2^\dual$}}
\def\actions{actions $*$, $\star$, $|_{k,j}$}
\def\rmalg{{${\rm Alg}$}} 
\def\rml{{${\rm LC}$, ${\rm LP}$}} 
\def\calC{{${\cal C}$, ${\cal C}_c$}}
\def\mesure{{${\rm Mes}$}}
\def\MMM{{$M_{k,j}^{\rm cl}$, $M_{k,j}^{\rm par,\,cl}$, $M_{k}^{\rm cl}$}}
\def\Mcong{{$M^{\rm cong}$, $M^{\rm cong}_S$, $M_k^{\rm cong}$}}
\def\Mkj{{$M_{k,j}$, $M_{k,j}^{\rm qh}$, $M_{k,j}^{\rm par}$, $M_{k,j}^{\rm par,\,qh}$}}
\def\PIK{{$\Pi_K$, $\Pi'_K$, $\Pi_\Q$, $\Pi_\Q'$, $\Pi_{\Q,S}$, $\Pi'_{\Q,S}$}}
\def\calA{{${\cal A}$, ${\cal A}^{\pm}$, ${\cal A}_{\rm par}$, ${\cal A}_{\rm par}^{\pm}$}}
\def\KKK{{${\cal K}(\phi,u)$, ${\cal K}(\phi,u,X)$, ${\cal K}_v$}}
\def\twis{twists $\otimes\chi$, $\otimes|\ |_\A^a$, $\otimes\cyp ^i$}
\def\fonctionL{{$L(\phi,s)$}}
\def\EEE{{${\cal E}$, $\oe$, $\oe^+$}}
\def\Ln{{$L_n$}} 
\def\FGET{{$\FGet$}}
\def\sigm{{$\sigma_\ell$ $\sigma_a$, $[\sigma_a]$}}
\def\Gamm{{$\Gamma$}}
\def\DDD{{$D_{\rm dif}$, $D_{{\rm dif},n}$, $D_{\rm dR}$, $\check D$}}
\def\lamb{{$\Lambda$, $\Lambda_0$, $\overline\Lambda_0$}}
\def\EXP{{$\exp^\dual$, ${\rm Exp}^\dual$}}
\def\epsi{{$\epsilon$}}
\def\ttt{{$t$}}

\def\AAA{{$\tA$, $\tA^+$, $\tA^{++}$, $\tA^-$}}
\def\bbb{{$\tB^-$}}
\def\Hc{{$H^1_c$, $H^1_{\rm par}$, $H^1_{{\rm truc}\sharp}$}}
\def\www{{$W_{k,j}^{\rm B}$, $W_{k,j}^{\rm dR}$, $W_{k,j}^{\eet}$}}
\def\YYY{{$Y(K)$, $Y(N)$}}
\def\XXX{{$X(K)$, $X(K)^\times$, $X(N)$}}
\def\zzz{{$Z(1)$, $Z(N)$}}
\def\inft{{$\infty_a$}}
\def\ome{{$\omega$, $\omega^{k,j}$}}

\def\IH{{$\iH$ $\pH$, $\iH(-)_S$, $\pH(-)_S$}}
\def\nabl{{$\nabla$}}
\def\hol{{${\rm Hol}$}}
\def\BDR{{$\Bdr$, $\O\Bdr$}}
\def\tq{{$\tilde q$}}
\def\vv{{$v_1$, $v_2$}}
\def\RHf{{$\rho_{f,j+1}$}}
\def\piell{{$\pi$, $\pi_\ell$, $\pi^{\rm alg}$, $\check\pi$}}
\def\pif{{$\pi_{f,j+1}$}}
\def\mpi{{$m(\pi)$, $m_{\rm B}(\pi)$, $m_{\eet}(\pi)$, $m_{\rm dR}(\pi)$}}
\def\lambdapi{{$\lambda(\pi)$}}
\def\Omegapi{{$\Omega_\pi^\pm$}}
\def\Vi{{$V(i)$}}
\def\iotaES{{$\iota_{\rm ES}$, $\iota_{\rm ES}^\pm$, $\iota_{{\rm ES},p}$}}
\def\iotan{{$\iota_n$}}
\def\iotadr{{$\iota_{{\rm dR},\pi}^+$, $\iota_{{\rm dR},\pi}^-$, $\iota_{\rm dR}^-$}}
\def\iotapi{{$\iota_\pi$, $\iota_{\pi,p}$, $\iota_{\pi,S}$}}

\def\iotaq{{$\iota_q$}}
\def\iotalin{{$\iota^{\rm lin}$, $\iota^{\rm semi}$}}
\def\iotaA{{$\iota_{\tA}$}}
\def\facteur{{$\epsilon(\pi)$, $\epsilon(\pi_\ell)$}}
\def\tAAA{{$\tbA$, $\tbA^+$, $\tbA^{++}$}}
\def\boulef{{$B$, $B^\times$, $B_\infty$, $\overline B$}}
\def\bouleo{{$B^-$, $B^{-,\times}$, $B^-_\infty$, $B^{-,\times}_\infty$}}
\def\GB{{$G_B$, $H_B$}}
\def\bdrq{{$\bdr^\pm\{\tilde q\}$, $\bdr\{\tilde q\}$, $\bdr\{\tilde q^{1/p^\infty}\}$}}
\def\loB{{$\log_B$}}
\def\kdr{{$\kappa_{\rm dR}$}}

\def\vell{{$v_{T,\ell}$, $v_{T,\ell}'$}}
\def\TTm{{$T_{\goth m}$, $T_0$, $T$}}
\def\TTM{{$T_M$, $T_{0,M}$}}

\def\rhoT{{$\rho_T$, $\rho_T^{\diamond}$, $\check\rho_T$, $\rho_{T_0}$}}
\def\rhoM{{$\rho_{T_M}$, $\rho_{T_{0,M}}$}}
\def\calX{{${\cal X}$, ${\cal X}^{\rm cl}$, ${\cal X}^{{\rm cl},+}$, ${\cal X}_0$, ${\cal X}_0^{\rm cl}$, ${\cal X}_0^{{\rm cl},+}$}}
\def\Hrho{{$H^1_c[\rho_T]_S$}}

\def\lambdat{{$\lambda_T$}}
\def\xxxo{{$X(Np^\infty)$, $X(0)$, $\widehat X(0)$, $\widehat X^{(p)}(0)$}}
\def\zzzo{{$Z(Np^\infty)$, $Z(0)$, $\widehat Z(0)$, $\widehat Z^{(p)}(0)$}}
\def\taq{{$\tA^-[[\tilde q^{\Q_+}]]\boxtimes\Z_p^\dual$}}
\def\tGG{{$\widetilde G_\Q$, $\widetilde G_{\Q_p}$}}
\def\RES{{${\rm Res}$}}
\def\HH{{$H$, $H_0$}}
\def\calW{{${\cal W}$, ${\cal W}_0$}}
\def\oi{{$(0,\infty)$, $(a,b)$}}
\def\oits{{$(0,\infty)_{T,S}$, $(0,\infty)_{\pi,S}^{c,d}$}}
\def\MAT{{${\bf M}'_2$, ${\bf M}'_{2,1}$}}
\def\sO{{$\iO$}}
\def\piS{{$\pi_S$, $\pi_S[\eta]$}}
\def\zm{{${\bf z}^S_{{\rm Iw},M}(\rho_T)$, ${\bf z}_M^S(\rho_T)$, ${\bf z}^S_{{\rm Iw},M}(\rho_{T_0})$}}
\def\zpi{{${\bf z}(\pi)$, ${\bf z}_{\rm Iw}(\pi)$, ${\bf z}_{(0,\infty)}(\pi)$}}
\def\zmpi{{${\bf z}_{\rm Iw}(m_{\eet}(\pi))$, ${\bf z}(m_{\eet}(\pi))$, ${\bf z}_{\rm Iw}(m_{\eet}(\pi))_p$}}
\def\zemp{{${\bf z}^S_{\rm Iw}(\rho_T,\iota_{\rm Em})_p$, ${\bf z}^S(\rho_T,\iota_{\rm Em})_p$, ${\bf z}^S_{\rm Iw}(\rho_{T_0},\iota_{\rm Em})_p$}}
\def\zem{{${\bf z}^S_{\rm Iw}(\rho_T,\iota_{\rm Em})$, ${\bf z}^S(\rho_T,\iota_{\rm Em})$, ${\bf z}^S_{\rm Iw}(\rho_{T_0},\iota_{\rm Em})$}}
\def\zmem{{${\bf z}^S_{{\rm Iw},M}(\rho_T,\iota_{\rm Em})_p$, ${\bf z}_M^S(\rho_T,\iota_{\rm Em})_p$, ${\bf z}^S_{{\rm Iw},M}(\rho_{T_0},\iota_{\rm Em})_p$}}
\def\zeis{{${\bf z}_{\rm Eis}$, ${\bf z}'_{\rm Eis}$, $\tilde{\bf z}_{\rm Eis}$, $\tilde{\bf z}'_{\rm Eis}$}}
\def\zsie{{${\bf z}_{\rm Siegel}$}}
\def\zkato{{${\bf z}_{\rm Kato}$, ${\bf z}_{\rm Kato}^{c,d}$, ${\bf z}_{\rm Kato}^{S,c,d}$}}
\def\zkatokj{{${\bf z}_{\rm Kato}^{c,d}(k,j)$, ${\bf z}_{\rm Kato}^{S,c,d}(k,j)$}}
\def\Zmpi{{${\bf z}^{S}(\check m_{\eet}(\pi))$, ${\bf z}_{\rm Kato}^{S,c,d}(\check m_{\eet}(\pi))$}}
\def\zkatol{{${\bf z}_{\rm Kato}^{S}(\Lambda\otimes\rho_T)$}}
\def\zb{{$\zeta_{\rm B}$, $\zeta_{\rm dR}$}}

\makeindex
\begin{document}
\title[Factorisation du syst\`eme de Beilinson-Kato]
{Une factorisation de la cohomologie compl\'et\'ee et du syst\`eme de Beilinson-Kato}
\author{Pierre Colmez}
\address{C.N.R.S., IMJ-PRG, Sorbonne Universit\'e, 4 place Jussieu,
75005 Paris, France}
\email{pierre.colmez@imj-prg.fr}
\author{Shanwen Wang}
\address{School of mathematics, Renmin University of China, 59 ZhongGuanCun Street, 100872 Beijing, P.R. China}
\email{s\_wang@ruc.edu.cn}
\dedicatory{\`A la m\'emoire de Jo\"el Bella\"{\i}che et Jan \nekovar}
\begin{abstract}
Nous montrons que le symbole modulaire $(0,\infty)$, vu comme \'el\'ement
du dual de la cohomologie compl\'et\'ee, interpole en famille le syst\`eme d'Euler
de Kato, et nous en d\'eduisons une factorisation du syst\`eme de Beilinson-Kato
comme un produit de deux symboles $(0,\infty)$ (un avatar alg\'ebrique
de la m\'ethode de Rankin). La preuve utilise
la correspondance de Langlands locale $p$-adique pour ${\bf GL}_2(\Q_p)$ et
la factorisation d'Emerton de la cohomologie compl\'et\'ee de la tour des courbes modulaires,
dont nous donnons une preuve nouvelle reposant 
sur la construction d'un mod\`ele de Kirillov pour la cohomologie compl\'et\'ee, et
 que nous raffinons en
imposant des conditions aux points classiques; l'existence d'un tel raffinement 
traduit une propri\'et\'e d'analyticit\'e des p\'eriodes $p$-adiques de formes modulaires.
\end{abstract}
\begin{altabstract}
We show that the modular symbol $(0,\infty)$, considered as an element of the dual of
Emerton's completed cohomology, interpolates Kato's Euler system at classical points,
and we deduce from this a factorisation of Beilinson-Kato's system as a product
of two symbols $(0,\infty)$ (an algebraic analog of Rankin's method).
The proof uses the $p$-adic local Langlands correspondence for ${\bf GL}_2(\Q_p)$
and Emerton's factorization of the completed cohomology of the tower of modular curves
for which we provide a new proof resting upon the construction of a Kirillov
model for the completed cohomology, and
which we refine by imposing
conditions at classical points; the existence of such a refinement is a manifestation
of an analyticity property for $p$-adic periods of modular forms.
\end{altabstract}
\thanks{Pendant l'\'elaboration de cet article, la recherche de S.W. a \'et\'e subventionn\'ee par
the Fundamental Research Funds for the Central Universities, and the Research Funds of Renmin University 
of China \no 20XNLG04 et the National Natural Science Foundation of China (Grant \no 11971035); 
P.C. \'etait membre des projets ANR Percolator puis Coloss.}
\setcounter{tocdepth}{2}

\maketitle


\section*{Introduction}

Le point de d\'epart de cet article \'etait l'espoir que l'on pouvait expliquer
l'apparition d'un produit de deux valeurs sp\'eciales de fonctions~$L$ dans les formules
explicites li\'ees au syst\`eme d'Euler de Kato (ce produit de deux valeurs est ce qui sort de la
m\'ethode de Rankin) par une factorisation du syst\`eme de Beilinson-Kato comme
un produit de deux symboles modulaires
$(0,\infty)$, chacun fournissant une des deux 
valeurs sp\'eciales de fonctions~$L$.
C'est ce que nous prouvons, mais cela demande d'interpr\'eter correctement les termes en pr\'esence.

\Subsection{La factorisation de la cohomologie compl\'et\'ee}\label{cro1}
\subsubsection{La cohomologie compl\'et\'ee}\label{cro2}
Soit $L$ une extension finie de $\Q_p$ d'anneau des entiers~$\O_L$ et corps r\'esiduel~$k_L$, 
et soient $\A$ (resp.~$\Ai$, $\Aip$) l'anneau
des ad\`eles (resp.~ad\`eles finis, resp.~ad\`eles finis sans la composante $p$-adique) de $\Q$
et $\cZ$, $\cZ^{]p[}$ les anneaux des entiers de $\Ai$, $\Aip$.
On note ${\cal P}$ l'ensemble des nombres premiers.

Notons $\GG$ le groupe ${\bf GL}_2$ et $(x_v)_v$ les composantes
de $x\in\GG(\A)$ (avec $v\in{\cal P}\cup\{\infty\}$).
On dispose d'actions de $\GG(\Q)$ et $\GG(\A)$ sur
les fonctions $\phi:\GG(\A)\to L$ d\'efinies par
$$(\gamma*\phi)(x)=\phi(\gamma^{-1}x),\ {\text{si $\gamma\in\GG(\Q)$,}}\quad
(g\star\phi)(x)=\phi(xg),\ {\text{si $g\in\GG(\A)$.}}$$
Ces deux actions commutent, ce qui fait que, si $X$ est un espace de fonctions stable par ces deux actions,
les groupes $H^i(\GG(\Q),X)$ sont munis d'une action de $\GG(\A)$.
Parmi les espaces $X$ possibles, mentionnons les espaces
${\cal C}\supset{\cal C}^{(p)}\supset{\rm LA}\supset{\rm LP}$, o\`u
\begin{align*}
{\cal C}(\GG(\A),L)&:=\{\phi:\GG(\A)\to L,\ {\text{$\phi$ continue}}\}\\
{\cal C}^{(p)}(\GG(\A),L)&:=\{\phi\in {\cal C}(\GG(\A),L),
\ {\text{$\phi$ localement lisse pour l'action de $\GG(\cZ^{]p[})$}}\}\\ 
{\rm LA}(\GG(\A),L)&:=\{\phi\in {\cal C}^{(p)}(\GG(\A),L),
\ {\text{$\phi$ localement analytique en $x_p$}}\}\\
{\rm LP}(\GG(\A),L)&:=\{\phi\in {\cal C}^{(p)}(\GG(\A),L),
\ {\text{$\phi$ localement alg\'ebrique en $x_p$}}\}
\end{align*}
\begin{rema}\phantomsection\label{intox1}
(i) Comme $L$ est totalement discontinu, les $\phi$ ci-dessus se factorisent par $\GG(\A)/\GG(\R)_+$,
o\`u $\GG(\R)_+$ est la composante connexe de $1\in\GG(\R)$; la composante archim\'edienne
n'intervient donc qu'\`a travers le groupe de ses composantes connexes $\{\pm1\}$.

Il s'ensuit que les $H^0$ ne sont pas des groupes tr\`es passionnants car une fonction
invariante par $\GG(\Q)$ et $\GG(\R)^+$ se factorise par le d\'eterminant (plus exactement,
par $\Aidu/\Q^\dual\R_+^\dual\cong\cZ^\dual$). Par contre les $H^1$ sont des groupes tr\`es
int\'eressants, et les $H^i$, pour $i\geq 2$, sont nuls.

(ii) Soit $\Gamma:={\bf SL}_2(\Z)$. Le lemme de Shapiro fournit un isomorphisme
$$H^1(\GG(\Q),{\cal C}^{(p)}(\GG(\A),L))\cong H^1(\Gamma,{\cal C}^{(p)}(\GG(\cZ),L))=
L\otimes_{\O_L}H^1(\Gamma,{\cal C}^{(p)}(\GG(\cZ),\O_L))$$
Par ailleurs, 
$${\cal C}^{(p)}(\GG(\cZ),\O_L)=\varinjlim\nolimits_{(N,p)=1}\varprojlim\nolimits_k
\big(\varinjlim\nolimits_n {\cal C}(\GG(\Z/Np^n),\O_L/p^k)\big)$$
et le lemme de Shapiro, coupl\'e avec les th\'eor\`emes de comparaison ``cohomologie du $\pi_1$-Betti''
et ``Betti-\'etale'', fournit des isomorphismes
$$H^1(\Gamma, {\cal C}(\GG(\Z/Np^n),\O_L/p^k))\cong H^1_{\rm B}(Y(Np^n)(\C),\O_L/p^k)\cong
H^1_{\eet}(Y(Np^n)_\Qbar,\O_L/p^k)$$
o\`u $Y(M)$ est la courbe modulaire de niveau $M$ (connexe sur $\Q$ mais pas g\'eom\'etriquement
connexe).  
Cela munit les $H^1$ d'une action de $G_\Q:={\rm Gal}(\Qbar/\Q)$, et permet de montrer
que $H^1(\GG(\Q),{\cal C}^{(p)}(\GG(\A),L))$ est la cohomologie compl\'et\'ee d'Emerton~\cite{Em06}.

(iii) On a aussi, si $C$ est un corps alg\'ebriquement clos, complet pour la valuation $p$-adique,
\begin{align*}
H^1(\GG(\Q),{\cal C}(\GG(\A),\O_L)) &\cong H^1_{\rm proet}(\widehat Y(0)_C,\O_L)\\
H^1(\GG(\Q),{\cal C}^{(p)}(\GG(\A),L))&\cong H^1_{\rm proet}(\widehat Y(0)^{(p)}_C,\O_L)
\end{align*}
o\`u $H^1_{\rm proet}(\widehat Y(0)_C,\O_L)$ est la cohomologie pro\'etale
g\'eom\'etrique de la courbe perfecto\"{\i}de
$\widehat Y(0)$ obtenue en compl\'etant
la tour des courbes modulaires de tous niveaux~\cite{Sz2}
tandis que 
$\widehat Y(0)^{(p)}$ est la limite projective pour $(N,p)\neq 1$ des
$\widehat Y(Np^\infty)$, o\`u $\widehat Y(Np^\infty)$ est la compl\'et\'ee
de la tour des courbes modulaires de niveaux $Np^n$, pour $n\in\N$ (i.e.~on compl\`ete uniquement
en $p$).  Cette interpr\'etation, qui nous sera tr\`es utile, a d\'ej\`a \'et\'e utilis\'ee
avec profit par Pan~\cite{pan2}.

(iv) On peut montrer que $H^1(\GG(\Q),{\cal C}^{(p)}(\GG(\A),L))$ est l'ensemble des
vecteurs $\GG(\cZ^{]p[})$-lisses de $H^1(\GG(\Q),{\cal C}(\GG(\A),L))$ 
(ce n'est pas une tautologie: le
m\^eme \'enonc\'e pour la restriction des scalaires de $K$ \`a $\Q$
de ${\bf GL}_1$ \'equivaut \`a la conjecture
de Leopoldt pour $K$).

(v) 
On a $H^1(\GG(\Q),{\rm LP}(\GG(\A),L))=L\otimes H^1(\GG(\Q),{\rm LP}(\GG(\A),\Q))$
et la th\'eorie d'Eichler-Shimura exprime $\C\otimes H^1(\GG(\Q),{\rm LP}(\GG(\A),\Q))$
en termes de formes modulaires classiques.
De mani\`ere analogue:

$\bullet$ Les espaces $\C_p\wotimes_L H^1(\GG(\Q),{\cal C}(\GG(\A),L))$
et $\C_p\wotimes_L H^1(\GG(\Q),{\cal C}^{(p)}(\GG(\A),L))$ sont reli\'es aux formes modulaires $p$-adiques
(cf.~th.\,\ref{intox4} pour un tel lien).

$\bullet$
$\C_p\wotimes_L H^1(\GG(\Q),{\rm LA}(\GG(\A),L))$ est reli\'e aux formes modulaires $p$-adiques 
surconvergentes~\cite{pan2,pan3}.
\end{rema}

\subsubsection{Le module $H^1[\rho_T]$}\label{cro3}
Notons $G_{\Q_\ell}$ le groupe de Galois absolu de $\Q_\ell$.
On fixe un plongement $\Qbar\hookrightarrow\Qbar_\ell$, et donc une injection
$G_{\Q_\ell}\hookrightarrow G_\Q$.

Fixons $N$ premier \`a $p$. Soient 
\begin{align*}
\wGamma(Np^\infty)&:=\GG(\cZ)\cap(1+N{\bf M}_2(\cZ^{]p[}))\\
\widehat H^1(N)&:=H^1(\GG(\Q),{\cal C}(\GG(\A)/\wGamma(Np^\infty),\O_L))
\subset H^1(\GG(\Q),{\cal C}(\GG(\A),\O_L))
\end{align*}
(On a aussi $\widehat H^1(N)=H^1_{\proet}(\widehat Y(Np^\infty)_C,\O_L)$, et
$\widehat H^1(N)$ est stable par $G_\Q$.)

On note $T(Np^\infty)$ l'alg\`ebre de Hecke agissant sur $\widehat H^1(N)$ (engendr\'ee par les
op\'erateurs de Hecke $T_\ell$, pour $\ell\nmid Np$); c'est une alg\`ebre semi-locale.
On fixe un id\'eal maximal ${\goth m}$ de $T(Np^\infty)$ et on note $T$ le localis\'e 
$T(Np^\infty)_{\goth m}$; c'est une alg\`ebre locale d'id\'eal maximal ${\goth m}$ et, 
quitte \`a agrandir $L$, on peut supposer que $T/{\goth m}=k_L$.
Soit $${\cal X}={\rm Spec}(T)(\O_{\overline L})$$ 

On dispose d'un pseudo-caract\`ere $t_{\goth m}:G_\Q\to T$, de dimension $2$.
On suppose que ${\goth m}$ est {\it non-eisenstein} (i.e.~la r\'eduction modulo ${\goth m}$ de
$t_{\goth m}$ est la trace d'une repr\'esentation absolument irr\'eductible); 
il existe alors une repr\'esentation $\rho_T:G_\Q\to{\bf GL}_2(T)$ dont la trace est $t_{\goth m}$. 

Si $x\in{\cal X}$, on note
${\goth p}_x$ l'id\'eal premier de $T$ qui lui correspond et $\rho_x$
la sp\'ecialisation de $\rho_T$ en $x$ 
(i.e. la repr\'esentation $(T/{\goth p}_x)[\frac{1}{p}]\otimes_T\rho_T$).
On dit que {\it $x$ est classique}, si $\rho_x$ est la repr\'esentation 
associ\'ee \`a une forme modulaire $f$ primitive de poids~$\geq 2$ \`a torsion
pr\`es par une puissance enti\`ere du caract\`ere cyclotomique.
Les points classiques sont zariski-denses dans
la fibre g\'en\'erique de ${\cal X}$.

Soient $S$ l'ensemble des nombres premiers divisant $Np$ et
$G_{\Q,S}$ le groupe de Galois de l'extension maximale de $\Q$, non ramifi\'ee en dehors de $S$.
Alors $\rho_T$ se factorise \`a travers $G_{\Q,S}$ et, si $N$ est suffisamment divisible
par les $\ell\in S\moins\{p\}$,
les th\'eor\`emes {\og big $R$ = big $T$ \fg} identifient $T$ 
\`a l'anneau des d\'eformations universelles
de $\overline\rho_T$ et $\rho_T:G_{\Q,S}\to {\bf GL}_2(T)$ \`a la d\'eformation universelle
(cf.~\cite{boc} et~\cite[\S\,7.3]{Em08}).

Le localis\'e $\widehat H^1(N)_{\goth m}$ de $\widehat H^1(N)$ est un facteur direct de
$\widehat H^1(N)$. On note $H^1[\rho_T]$ le sous-$T[\GG(\A)]$-module de
$H^1(\GG(\Q),{\cal C}(\GG(\A),\O_L)$ qu'il engendre. On cherche \`a d\'ecrire
$H^1[\rho_T]$ en tant que $T[G_\Q\times\GG(\Ai)]$-module.

\subsubsection{La factorisation}\phantomsection\label{intro1}
Soit 
$\rho_T^\diamond={\rm Hom}_T(\rho, T)$.
Via les correspondances de Langlands locales $p$-adiques en famille,
on sait associer \`a $\rho_T^\diamond$, pour tout $\ell\in{\cal P}$, 
une $T$-repr\'esentation $\Pi_\ell(\rho_T^\diamond)$
de $\GG(\Q_\ell)$:

\vskip1mm
$\bullet$ Si $\ell\neq p$, la repr\'esentation que nous utilisons est une variante de celle
fournie par la th\'eorie d'Emerton-Helm~\cite{EH}; c'est un $T$-module sans torsion, muni d'une action
lisse de $\GG(\Q_\ell)$, et si $x\in{\cal X}$, alors $(T/{\goth p}_x)\otimes_T\Pi_\ell(\rho_T^\diamond)$
est (g\'en\'eriquement) un r\'eseau de $\Pi_\ell^{\rm cl}(\rho_x^\dual)$, o\`u $\Pi_\ell^{\rm cl}(\rho_x^\dual)$
est la repr\'esentation lisse de $\GG(\Q_\ell)$ associ\'ee \`a $\rho_x^\dual$ via la correspondance 
de Langlands locale classique.

\vskip1mm
$\bullet$ Si $\ell=p$, l'existence de $\Pi_p(\rho_T^\diamond)$ est une cons\'equence
de la fonctorialit\'e de la correspondance de Langlands locale $p$-adique~\cite{gl2,Ki1,Pas}.
La repr\'esentation $\Pi_p(\rho_T^\diamond)$ est la boule unit\'e d'un $L$-banach, et 
est munie d'une action continue de $\GG(\Q_p)$. En tant que $T$-module,
c'est un $T$-module {\og de torsion\fg} au sens o\`u la $T$-torsion est dense dans $\Pi_p(\rho_T^\diamond)$
(si $x\in{\cal X}$, 
le sous-module de ${\goth p}_x$-torsion de $\Pi_p(\rho_T^\diamond)$ s'identifie \`a un r\'eseau de
$\Pi_p(\rho_{x}^\dual)$, o\`u $\Pi_p(\rho_{x}^\dual)$ est la repr\'esentation unitaire
 de $\GG(\Q_p)$ associ\'ee \`a $\rho^\dual_x$ via la correspondance
de Langlands locale $p$-adique).

\vskip2mm
On d\'efinit $\Pi(\rho_T^\diamond)$ comme le produit tensoriel ext\'erieur restreint
des $\Pi_\ell(\rho_T^\diamond)$, pour $\ell\in{\cal P}$.
Si $x$ n'est pas trop pathologique (\no\ref{pathol} et th.\,\ref{glob111}),
alors
$$\big(L\otimes_{\O_L}\Pi(\rho_T^\diamond)\big)[{\goth p}_x]
\cong\Pi_p(\rho_x^\dual)\otimes\big(\otimes'_{\ell\neq p} \Pi_\ell^{\rm cl}(\rho_x^\dual)\big)$$
La repr\'esentation $\rho_T$ n'est bien d\'etermin\'ee qu'\`a multiplication
pr\`es par une unit\'e de $T$,
mais la fonctorialit\'e de la correspondance de Langlands locale $p$-adique et le passage
au dual font que $\rho_T\otimes_T\Pi_p(\rho_T^\diamond)$ ne d\'epend pas du choix de $\rho_T$
(de la m\^eme mani\`ere que $\rho_T^\diamond\otimes_T \rho_T={\rm End}_T(\rho_T)$ n'en d\'epend pas).

\begin{theo}\phantomsection\label{Intro2}
Si ${\goth m}$ est non-eisenstein,
il existe un isomorphisme naturel, $T[G_{\Q,S}\times \GG(\Ai)]$-\'equivariant:
$$\iota_T:\rho_T\otimes_T\Pi(\rho_T^\diamond)[\tfrac{1}{p}]\overset{\sim}{\to} L\otimes_{\O_L}H^1[\rho_T]$$
\end{theo}

\begin{rema}\phantomsection\label{Intro3}
{\rm (i)} 
Si ${\goth m}$ est {\it g\'en\'erique} (i.e.~non-eisenstein et la
restriction de $\overline\rho_T$ \`a $G_{\Q_p}$
pas de la forme\footnote{Si $p=2$, il faut aussi supposer que cette restriction n'est pas irr\'eductible,
mais cette condition doit pouvoir s'\'eliminer avec un peu plus d'efforts.}
 $\matrice{1}{0}{0}{1}\otimes\delta$), nous obtenons un isomorphisme au niveau entier:
$$\iota_T:\rho_T\otimes_T\Pi(\rho_T^\diamond)\overset{\sim}{\to} H^1[\rho_T]$$
Si ${\goth m}$ est 
non-eisenstein et si la restriction de $\overline\rho_T$
\`a $G_{\Q_p}$
n'est pas de la forme $\matrice{1}{*}{0}{1}\otimes\delta$,
Emerton~\cite[th.\,6.2.13]{Em08} a montr\'e l'existence d'un isomorphisme ``abstrait'' 
au niveau entier
entre les deux membres
\`a ceci pr\`es que notre action de $\GG(\A)$ et celle d'Emerton diff\`erent de $g\mapsto{^tg^{-1}}$,
et donc que le membre de gauche est $\rho_T\otimes_T\Pi(\rho_T)$ chez Emerton.

{\rm (ii)} Notre preuve est tr\`es diff\'erente de celle d'Emerton et fournit un isomorphisme bien
d\'etermin\'e.
Elle repose sur la construction
de mod\`eles de Kirillov pour les deux membres, \`a valeurs
dans
${\cal C}(\Aidu,\check{T}\otimes_{\Z_p}\tA^-)$, o\`u $\check{T}={\rm Hom}_{\O_L}(T,\O_L)$
et $\tA^-=\tA/\tA^+$, avec $\tA=W(C^\flat)$ et $\tA^+=W(\O_C^\flat)$ (l'apparition
de $\tA^-$ est un peu surpenante vu qu'il n'entre nulle part dans la d\'efinition des
deux membres; voir les ${\rm n}^{\rm os}$~\ref{cro4} et~\ref{cro5}
 pour les d\'etails de cette apparition). 
Une fois ces mod\`eles de Kirillov construits, la strat\'egie consiste \`a prouver que:

$\bullet$ ${\cal K}_{\rm Aut}:\rho_T\otimes_T\Pi(\rho_T^\diamond)
\to {\cal C}(\Aidu,\check{T}\otimes_{\Z_p}\tA^-)$ est une isom\'etrie sur son image
(quasi-tautologique vu la d\'efinition de $\Pi(\rho_T^\diamond)$).

$\bullet$ (dans le cas g\'en\'erique) ${\cal K}_{H}:H^1[\rho_T]
\to {\cal C}(\Aidu,\check{T}\otimes_{\Z_p}\tA^-)$ est une isom\'etrie sur son image
(d\'elicat: utilise de la th\'eorie de Hodge $p$-adique enti\`ere et des th\'eor\`emes de multiplicit\'e~$1$
modulo~$p$
qui font d\'efaut dans le cas non g\'en\'erique --- la preuve de ce cas (th.\,\ref{NE2}) est plus alambiqu\'ee
et donne un r\'esultat moins pr\'ecis, i.e.~on est forc\'e d'inverser $p$).

$\bullet$ Les images des vecteurs localement alg\'ebriques des deux membres sont les m\^emes
(cf.~diag.\,(\ref{intox33}), qui repose sur des
formules explicites pour les th\'eor\`emes de comparaisons ``de Rham-\'etale $p$-adique'' pour les
formes modulaires).

On conclut alors en utilisant la densit\'e des vecteurs localement alg\'ebriques des deux membres.
\end{rema}

\begin{rema}\phantomsection\label{Intro3.1}
{\rm (i)}
On d\'eduit du th\'eor\`eme, pour $x\in{\cal X}$ pas trop pathologique, un isomorphisme naturel:
$$\rho_x\otimes\big(\Pi_p(\rho_x^\dual)\otimes\big(\otimes'_{\ell\neq p}
\Pi_\ell^{\rm cl}(\rho_x^\dual)\big)\big)\overset{\sim}{\to} (L\otimes_{\O_L}H^1[\rho_T])[{\goth p}_x]$$
qui a pour cons\'equence l'\'enonc\'e suivant, peut-\^etre plus parlant:
\begin{equation}\label{loglo}
{\rm Hom}_{G_\Q}(\rho_x,L\otimes_{\O_L}H^1_{\proet}(\widehat Y(0)^{(p)}_C,\O_L))
\cong \Pi_p(\rho_x^\dual)\otimes\big(\otimes'_{\ell\neq p}
\Pi_\ell^{\rm cl}(\rho_x^\dual)\big)
\end{equation}

{\rm (ii)} Selon la conjecture de compatibilit\'e local-global
d'Emerton~\cite[conj.\,1.1.1]{Em06b} l'isomorphisme~(\ref{loglo}) devrait \^etre
valable pour toute repr\'esentation de dimension~$2$ de $G_\Q$, impaire, absolument irr\'eductible,
non ramifi\'ee en dehors d'un nombre fini de nombres premiers.
\end{rema}

\Subsection{Mod\`eles de Kirillov}\label{intox23}
\subsubsection{Le mod\`ele de Kirillov de la cohomologie compl\'et\'ee}\label{cro4}
Le mod\`ele de Kirillov ${\cal K}_H$ pour $H^1[\rho_T]$ est la restriction d'un mod\`ele
de Kirillov pour $H^1_{\proet}(\widehat Y(0)_C,\O_L)=\O_L\otimes_{\Z_p} H^1_{\proet}(\widehat Y(0)_C,\Z_p)$.
On note $\TT$ l'alg\`ebre de Hecke qui agit: c'est la limite projective des alg\`ebres
de Hecke de tous les niveaux finis et elle agit \`a travers son quotient $T$ sur $H^1[\rho_T]$.

On construit le mod\`ele de Kirillov pour $H^1_{\proet}(\widehat Y(0)_C,\Z_p)$
en restreignant les classes de cohomologie 
\`a une composante connexe du voisinage multiplicatif $\widehat Z(0)_C$ de la pointe $\infty$
 --  une boule ouverte perfecto\"{\i}de -- 
et en utilisant une description
de la cohomologie pro\'etale de la boule ouverte perfecto\"{\i}de 
(formules~(\ref{bato40.2}) et~(\ref{bato40.4})).

La boule ouverte est une r\'eunion croissante 
de boules ferm\'ees et, si $B_\infty$ est une boule
ferm\'ee perfecto\"{\i}de, on a 
$H^1_{\proet}(B_\infty,\Z_p)\cong
W((\O(B_\infty)^\flat)/(\varphi-1)$.  
Le membre de droite se calcule facilement et fournit, en passant \`a la limite,
une injection naturelle\footnote{Ceci demande que $C$ soit sph\'eriquement complet.}
 de la cohomologie pro\'etale de la boule ouverte perfecto\"{\i}de
dans $\prod_{i>0,\,v_p(i)=0}\tA^- \tilde q^i$ o\`u $\tilde q=[q^\flat]$ et $q$ est la
coordonn\'ee locale usuelle sur le lieu multiplicatif de la courbe
modulaire de niveau~$1$.

En prenant le terme correspondant \`a $i=1$, compos\'e avec la restriction
au voisinage de la pointe $\infty$, cela fournit une fl\`eche naturelle
$\alpha:H^1_{\proet}(\widehat Y(0)_C,\Z_p)\to \tA^-$;
c'est de l\`a que sort le $\tA^-$.  Ensuite, on pose
$$\langle{\cal K}_{H,v}(x),\lambda\rangle=
\alpha(\matrice{x}{0}{0}{1}\cdot \lambda v),\quad{\text{si $v\in H^1_{\proet}(\widehat Y(0)_C,\Z_p)$,
si $x\in\Aidu$ et si $\lambda\in \TT$.}}$$
Ce mod\`ele de Kirillov n'est que $G_{\Q_p}$-\'equivariant car le voisinage
multiplicatif de $\infty$ n'est d\'efini que sur $\Q_p$.

\begin{rema}\phantomsection\label{intox3}
(i) Le point d\'elicat est de prouver que ce mod\`ele de Kirillov est injectif sur $H^1[\rho_T]$
(apr\`es avoir sym\'etris\'e sous l'action de $G_\Q$ pour perdre le moins d'information possible). 
Nous utilisons
deux approches de domaines de validit\'e diff\'erents:

$\bullet$ La premi\`ere approche utilise la th\'eorie de Hodge $p$-adique enti\`ere, et permet
de prouver l'injectivit\'e en restriction au localis\'e en un id\'eal maximal g\'en\'erique
(et m\^eme mieux, on obtient une isom\'etrie sur l'image (th.\,\ref{cdn3})
comme il est mentionn\'e ci-dessus).

$\bullet$ La seconde approche utilise la description
des vecteurs $\matrice{1}{\Z_p}{0}{1}$-invariants de $C\wotimes H^1_{\proet}(\widehat Y(0),\Z_p)$
(ce groupe peut se calculer en termes de la cohomologie pro\'etale du faisceau
$\widehat\O$ gr\^ace au th\'eor\`eme de comparaison primitif de Scholze~\cite[th.\,IV.2.1]{Sz2}, ce qui
est aussi le point de d\'epart des travaux de Pan~\cite{pan2,pan3} 
dont nous nous sommes inspir\'es).
Cette seconde approche permet (th.\,\ref{igu8.2})
de prouver l'injectivit\'e sur tout sous-espace ferm\'e $W$, 
stable par
$\TT$, $G_\Q$ et $\GG(\Ai)$, tel que $W[{\goth p}_x]\neq 0$ 
implique que
$\rho_{x}$ est irr\'eductible et l'op\'erateur de Sen de la restriction de $\rho_x$
\`a $G_{\Q_p}$ n'est pas scalaire.
(On peut supprimer cette derni\`ere condition par ``prolongement analytique'', cf.~preuve du
th.\,\ref{NE2}.)
\end{rema}

Un sous-produit de la seconde approche est le th\'eor\`eme de d\'ecomposition de Hodge-Tate suivant
(th.\,\ref{igu2}, prop.\,\ref{igu7} et rem.\,\ref{igu8})
inspir\'e par les travaux de Pan~\cite{pan2} et Howe~\cite{howe}.
Soient $U=\matrice{1}{\Z_p}{0}{1}$ et $B=\matrice{\Z_p^\dual}{\Z_p}{0}{\Z_p^\dual}$.
On note $\widehat{\rm Ig}(0)$ la compl\'etion de la tour 
d'Igusa de tous niveaux. 
L'espace $\O(\widehat{\rm Ig}(0)/U)_\kappa$ qui appara\^{\i}t dans le th\'eor\`eme ci-dessous est
l'espace des formes modulaires $p$-adiques de poids $\kappa$.

\begin{theo}\phantomsection\label{intox4}
On a une suite exacte $\GG(\Aip)\times B(\Z_p)\times G_{\Q_p}$-\'equivariante:
$$0\to H^1(\widehat Y(0)_C/U)\to C\wotimes H^1_{\proet}(\widehat Y(0)_C,\Z_p))^U
\to \O({\rm Ig}(0)_C^\times/U)(-1)\to 0$$
De plus, si $\kappa$ est un caract\`ere continu de $B$, 
avec $\kappa\big(\matrice{a}{b}{0}{d}\big)=\kappa_1(a)\kappa_2(d)$, 
et si on passe aux vecteurs $\kappa$ isotypiques, la suite
reste exacte et l'op\'erateur de Sen sur le membre de gauche 
{\rm (}resp.~de droite{\rm )} est la multiplication
par $-w(\kappa_2)$ {\rm (}resp.~$-1-w(\kappa_1)${\rm )}.
\end{theo}

\begin{rema}\phantomsection\label{intox5}
{\rm (i)}
L'injectivit\'e du mod\`ele de Kirillov pour $H^1[\rho_T]$, coupl\'ee
aux propri\'et\'es de la correspondance de Langlands locale $p$-adique pour $\GG(\Q_p)$,
permet de prouver {\it directement} un certain nombre de r\'esultats qu'Emerton d\'eduit
de sa factorisation de la cohomologie compl\'et\'ee (comme la compatibilit\'e avec
la correspondance de Langlands locale classique~\cite[\S\,7.4]{Em08} 
ou la conjecture de Fontaine-Mazur pour les
repr\'esentations pro-modulaires~\cite[th.\,7.1.1]{Em08}, cf.~th.\,\ref{Ki119} et~\ref{Ki121}).  

(ii) L'approche d'Emerton fait grand usage
d'un r\'esultat
de Berger-Breuil~\cite{BB}, sp\'ecifique \`a $\GG(\Q_p)$, \`a savoir l'irr\'eductibilit\'e,
si $\rho$ est cristalline \`a poids de Hodge-Tate distincts, du
compl\'et\'e universel de la repr\'esentation localement alg\'ebrique de $\GG(\Q_p)$ associ\'e
\`a $\rho$ par la correspondance classique (modifi\'ee pour encoder les poids de Hodge-Tate).
Par contraste, notre preuve de la factorisation de la cohomologie compl\'et\'ee n'utilise pas cette
irr\'eductibilit\'e.
\end{rema}
\subsubsection{Le mod\`ele de Kirillov de $\rho_T\otimes_T\Pi(\rho_T^\diamond)$}\label{cro5}
Il est clair que $\tA^-$ ne peut sortir que de $\Pi_p(\rho_T)$ car 
si $\ell\neq p$, $\Pi_\ell(\rho_T)$ est, par construction, un sous-module de
${\rm LC}(\Q_\ell^\dual,\Z[\bmu_{\ell^\infty}]\otimes_\Z T)$, o\`u ${\rm LC}$ d\'esigne
l'espace des fonctions localement constantes
(i.e.~on identifie
$\Pi_\ell(\rho_T)$ a son image par le mod\`ele de Kirillov ${\cal K}_\ell: \Pi_\ell(\rho_T)\to
{\rm LC}(\Q_\ell^\dual,\Z[\bmu_{\ell^\infty}]\otimes_\Z T)$). 

Maintenant, si $V$ est une repr\'esentation
de $G_{\Q_p}$, de dimension~$2$, et si $\Pi_p(V)$ est la repr\'esentation unitaire de $\GG(\Q_p)$
qui lui est associ\'ee, on a une suite exacte $\PP(\Q_p)$-\'equivariante~\cite[cor.\,II.2.9]{gl2}
$$0\to (\tA\otimes V)^H/(\tA^+\otimes V)^H\to \Pi_p(V)\to J(V)\to 0$$
o\`u $J(V)$ est de dimension finie et fixe par $\matrice{1}{\Q_p}{0}{1}$ (un analogue
$p$-adique du module de Jacquet). Le membre de gauche s'injecte dans $(\tA^-\otimes V)^H$, 
et on montre (prop.\,\ref{cano1})
que cette injection s'\'etend de mani\`ere unique $\PP(\Q_p)$-\'equivariante \`a $\Pi_p(V)$.
Cela fournit une injection naturelle
$\Pi_p(\rho_T^\diamond)\hookrightarrow \tA^-\otimes\rho_T^\dual$, o\`u $\rho_T^\dual$
est le $\O_L$-dual de $\rho_T$ (ce passage du $T$-dual au $\O_L$-dual est d\^u \`a la mani\`ere
dont on d\'efinit $\Pi_p$).  

On en d\'eduit une fl\`eche naturelle
$\beta:\rho_T\otimes\Pi_p(\rho_T^\diamond)\to \tA^-$, puis un mod\`ele de Kirillov
$v\mapsto{\cal K}_{p,v}$
pour $\rho_T\otimes\Pi_p(\rho_T^\diamond)$, \`a valeurs dans ${\cal C}(\Q_p^\dual,\tA^-\otimes\check{T})$
 en posant
$$\langle{\cal K}_{p,v}(x),\lambda\rangle=\beta\big(\matrice{x}{0}{0}{1}\cdot \lambda v\big),
\quad{\text{si $x\in\Q_p^\dual$ et si $\lambda\in T$.}}$$
Le mod\`ele de Kirillov ${\cal K}_{\rm Aut}$ est obtenu en prenant le produit tensoriel
ext\'erieur restreint des ${\cal K}_\ell$ puis en prenant un r\'eseau bien choisi
(on demande que la r\'eduction modulo~${\goth m}_L$ ait un socle g\'en\'erique et irr\'eductible).

\subsubsection{Les vecteurs localement alg\'ebriques}\label{cro6}
Emerton~\cite[(4.3.4)]{Em06}, \cite[th.\,7.4.2]{Em06b}, 
 a prouv\'e que les vecteurs alg\'ebriques de la cohomologie
compl\'et\'ee proviennent de formes modulaires classiques: on a
$$H^1(\GG(\Q),{\cal C}^{(p)}(\GG(\A),L))^{\rm alg}=H^1(\GG(\Q),{\rm LP}(\GG(\A),L))$$
Il a aussi prouv\'e que les vecteurs $\GG(\Z_p)$-alg\'ebriques sont denses~\cite[prop.\,5.4.1]{Em08}
(voir aussi la prop.~\ref{cen5});
c'est aussi le cas des vecteurs localement alg\'ebriques de poids fix\'e.

Soit $\pi$ une $L$-repr\'esentation localement alg\'ebrique de $\GG(\Ai)$, absolument irr\'eductible,
telle que 
$$m_{\eet}(\pi):={\rm Hom}(\pi,H^1(\GG(\Q),{\cal C}^{(p)}(\GG(\A),L)))$$
soit non nul
(si $\pi$ n'est pas de la s\'erie principale, c'est une $L$-repr\'esentation de $G_\Q$, de dimension~$2$;
il y a une bijection entre l'ensemble des $\pi$ comme ci-dessus et celui des 
twists $f\otimes \Z(j)$ des formes modulaires primitives, les s\'eries d'Eisenstein correspondant
aux $\pi$ de la s\'erie principale).

La compatibilit\'e entre les correspondances de Langlands locales classique et $p$-adique et
un th\'eor\`eme de Carayol~\cite{cara} fournissent un isomorphisme abstrait
$\pi\cong \Pi(m_{\eet}(\pi)^\dual)^{\rm alg}$, si $m_{\eet}(\pi)$ est de dimension~$2$.
On peut fixer cet isomorphisme en utilisant 
un g\'en\'erateur naturel $\iota_{\rm dR}^-$ du second terme du gradu\'e de 
$D_{\rm dR}(m_{\eet}(\pi)^\dual)$ 
fourni par le th\'eor\`eme de comparaison pour les formes modulaires et la dualit\'e avec
le premier terme du gradu\'e pour la repr\'esentation duale (cf.~\no\ref{como12.41}), ainsi que
la description du mod\`ele de Kirillov des
vecteurs localement alg\'ebriques de $\Pi_p(V)$ si $V$ 
est de Rham (\no\ref{valg3} et formule~(\ref{emoins})).

Par ailleurs, si $m^0_{\eet}(\pi)$ est un r\'eseau de $m_{\eet}(\pi)$
obtenu par sp\'ecialisation de
$\rho_T$, alors $m^0_{\eet}(\pi)\otimes\Pi(m^0_{\eet}(\pi)^\dual)$ s'injecte naturellement dans
$\rho_T\otimes_T\Pi(\rho_T^\diamond)$ (ici encore, $\rho_T\to m^0_{\eet}(\pi)$
n'est d\'efini qu'\`a $L^\dual$ pr\`es mais cette ind\'etermination dispara\^{\i}t
quand on combine $m^0_{\eet}(\pi)$ et $\Pi(m^0_{\eet}(\pi)^\dual)$).  
On a alors un diagramme commutatif
\begin{equation}\label{intox33}
\xymatrix@C=7mm@R=6mm{
m_{\eet}(\pi)\otimes\pi\ar@{^{(}->}[r]\ar@{^{(}->}[d]^-{\iota_{\rm dR}^-}
& L\otimes_{\O_L}H^1[\rho_T]\ar[r]^-{{\cal K}_H}&L\otimes_{\O_L}{\cal C}(\Aidu,\tA^-\otimes\check{T})\\
m_{\eet}(\pi)\otimes\Pi(m_{\eet}(\pi)^\dual)\ar@{^{(}->}[rr]
&&L\otimes_{\O_L}(\rho_T\otimes_T\Pi(\rho_T^\diamond))\ar[u]_-{{\cal K}_{\rm Aut}}
}\end{equation}
La preuve de la commutativit\'e de ce diagramme combine
une loi de r\'eciprocit\'e explicite locale (th.\,\ref{bato22}) et
des ingr\'edients de la preuve de la conjecture $C_{\rm dR}$ pour les formes modulaires
(th.\,\ref{Ybato13} et~\ref{Ybato21}).

\Subsection{Interpolation des \'el\'ements de Kato}\label{cro7}
\subsubsection{\'El\'ement de Kato et symboles modulaires}
On suppose maintenant que ${\goth m}$ est g\'en\'erique pour pouvoir
disposer de l'isomorphisme du th.\,\ref{Intro2} au niveau entier.
(Le th.\,\ref{intro2}
ci-dessous est valable, plus g\'en\'eralement, sous l'hypoth\`ese que $\rho$ est $\Pi_p$-compatible
(cf.~\no\ref{como17}),
 qui inclut
le cas $\rho_{|G_{\Q_p}}$ irr\'eductible, sans restriction sur la repr\'esentation r\'esiduelle.)

 Si $x\in{\cal X}$ est classique, Kato~\cite[th.\,12.5]{Ka4} a construit un \'el\'ement
$${\bf z}_{\rm Iw}(\rho_x)\in \rho_x^\dual\otimes H^1(G_{\Q,S},\Lambda\otimes\rho_x),$$
o\`u $\Lambda=\Z_p[[{\rm Gal}(\Q(\bmu_{p^\infty})/\Q)]]$ est l'alg\`ebre d'Iwasawa.
Cet \'el\'ement est compl\`etement d\'etermin\'e par son image par la localisation
${\rm loc}_p$ en $p$ (i.e.~sa restriction \`a $G_{\Q_p}$) qui, elle-m\^eme, est d\'etermin\'ee
par les exponentielles duales de ses sp\'ecialisations aux caract\`eres
localement alg\'ebriques de ${\rm Gal}(\Q(\bmu_{p^\infty})/\Q)$ dans la {\og bande critique\fg},
et ces exponentielles duales font intervenir les valeurs sp\'eciales de la fonction $L$
de $f$ et de ses tordues par des caract\`eres de Dirichlet.

Le symbole modulaire $(0,\infty)$, vu comme \'el\'ement du dual de la cohomologie compl\'et\'ee\footnote{
C'est en fait un \'el\'ement du dual de la cohomologie compl\'et\'ee \`a support compact mais
comme on localise en un id\'eal non-eisenstein, cela revient au m\^eme.},
produit\footnote{On \'evalue $(0,\infty)$ sur $\rho_x\otimes v^{]p[}\otimes\Pi_p(\rho_x^\dual)$
o\`u $v^{]p[}$ est le produit tensoriel des nouveaux vecteurs normalis\'es en les $\ell\neq p$.}
 un \'el\'ement de $\rho_x^\dual\otimes \Pi_p^\dual(\rho_x^\dual)$
qui est invariant par $\matrice{p}{0}{0}{1}\in \GG(\Q_p)$.
De plus, l'\'evaluation de $(0,\infty)$ sur les vecteurs localement alg\'ebriques
de $\Pi_p^\dual(\rho_x^\dual)$ produit des valeurs sp\'eciales de la fonction $L$
de $f$ et de ses tordues (cf.~\cite{Em05} ou prop.\,\ref{ES11} 
et \no\ref{es2} ci-apr\`es).

Soit $D(\rho_x)$ le $(\varphi,\Gamma)$-module associ\'e \`a $\rho_x$
(ou plut\^ot sa restriction $\rho_{x,p}$ \`a $G_{\Q_p}$)
par l'\'equivalence de cat\'egories de Fontaine~\cite{Fo04}.
On dispose~\cite{CC99,gl2,poids} d'isomorphismes (o\`u 
$\zeta_{\rm B}$ est un g\'en\'erateur privil\'egi\'e du twist de Tate $\Z(1)$):
$$\Pi_p^\dual(\rho_x^\dual)^{\matrice{p}{0}{0}{1}=1}
\cong D(\rho_x(1))^{\psi=1}\cong H^1(G_{\Q_p},\Lambda\otimes\rho_x(1))
=H^1(G_{\Q_p},\Lambda\otimes\rho_x)\otimes\zeta_{\rm B}.$$
Comme $(0,\infty)$ est invariant par $\matrice{p}{0}{0}{1}$, cela fournit:
$${\bf z}_{(0,\infty)}(\rho_x)\in \rho_x^\dual\otimes H^1(G_{\Q_p},\Lambda\otimes\rho_x)\otimes\zeta_{\rm B}.$$
\begin{theo}\phantomsection\label{intro2}
Si $x\in{\cal X}$ est classique, alors
$${\rm loc}_p({\bf z}_{\rm Iw}(\rho_x))={\bf z}_{(0,\infty)}(\rho_x)\otimes\zeta_{\rm B}^{-1}.$$
\end{theo}
\begin{rema}\phantomsection\label{intro3}
(i)
Si $\rho_{x,p}$ est absolument irr\'eductible, ce r\'esultat est une cons\'equence directe
de la comparaison des formules faisant intervenir les valeurs sp\'eciales de fonctions~$L$
sus-mentionn\'ees et de la formule de la prop.\,\ref{ES20bis} 
(qui r\'esulte de~\cite[prop.\,VI.5.12]{gl2} ou \cite[prop.\,2.13]{poids})
qui relie exponentielles duales de sp\'ecialisations de ${\bf z}_{(0,\infty)}(\rho_x)$
et \'evaluations de $(0,\infty)$ sur les vecteurs localement alg\'ebriques.
Cela a aussi \'et\'e remarqu\'e par Yiwen Zhou~\cite{Zhou}.

(ii) Si $\rho_{x,p}$ est une extension de deux caract\`eres, le r\'esultat s'obtient
par {\og prolongement analytique\fg} gr\^ace au th.\,\ref{intro4} ci-dessous.

(iii) On peut utiliser l'\'equation fonctionnelle de $(0,\infty)$
pour en d\'eduire une \'equation fonctionnelle pour l'\'el\'ement
de Kato (th.\,\ref{neqf12}, extension de~\cite[th.\,4.7]{naka1} au cas o\`u $\rho_{x,p}$
n'est pas n\'ecessairement absolument irr\'eductible).
\end{rema}

\subsubsection{Interpolation en famille}
Comme ci-dessus, la forme lin\'eaire $(0,\infty)\circ\iota_T$ sur\footnote{Plus pr\'ecis\'ement,
sur $\rho_T(-1)\otimes v^{]p[}\otimes\Pi_p(\check\rho_T(1))$ o\`u $v^{]p[}$
est un \'el\'ement bien choisi de $\otimes_{\ell\neq p}\Pi_\ell(\check\rho_T(1))$, cf.~\no\ref{EUL12}
(cet \'el\'ement d\'epend en fait de $S$ de mani\`ere simple).}
$\rho_T(-1)\otimes\Pi_p(\rho^\diamond_T(1))$ est invariante par $\matrice{p}{0}{0}{1}$
et donc donne naissance \`a
$${\bf z}_{(0,\infty)}(\rho_T)\in \rho_T^{\diamond}\otimes_T H^1(G_{\Q_p},\Lambda\wotimes\rho_T)\otimes\zeta_{\rm B}.$$
Notons que, si $x$ est classique, 
la sp\'ecialisation
de ${\bf z}_{(0,\infty)}(\rho_T)$ en $x$ n'est autre que ${\bf z}_{(0,\infty)}(\rho_x)$;
autrement dit, ${\bf z}_{(0,\infty)}(\rho_T)$ interpole les ${\bf z}_{(0,\infty)}(\rho_x)$, 
pour $x\in{\cal X}$ classique.
\begin{theo}\phantomsection\label{intro4}
{\rm (i)}
Si ${\goth m}$ est g\'en\'erique, il existe 
$${\bf z}_{\rm Iw}(\rho_T)\in 
\rho_T^{\diamond}\otimes_T H^1(G_{\Q,S},\Lambda\wotimes\rho_T),$$
unique, tel que ${\rm loc}_p({\bf z}_{\rm Iw}(\rho_T))={\bf z}_{(0,\infty)}(\rho_T)\otimes\zeta_{\rm B}^{-1}$.

{\rm (iii)} Si $x$ est classique, 
la sp\'ecialisation de ${\bf z}_{\rm Iw}(\rho_T)$ en $x$ est ${\bf z}_{\rm Iw}(\rho_x)$.
Autrement dit, ${\bf z}_{\rm Iw}(\rho_T)$ interpole les \'el\'ements de Kato.
\end{theo}
\begin{rema}\phantomsection\label{intro5}
{\rm (i)} La construction de l'\'el\'ement ${\bf z}_{\rm Iw}(\rho_T)$ se fait en deux temps.
On commence par construire ${\bf z}_{\rm Iw}(\rho_T)$ dans
${\rm Fr}(T)\otimes_T\big(\rho_T^{\diamond}\otimes_T H^1(G_{\Q,S},\Lambda\wotimes\rho_T)\big)$,
cf.~chap.\,\ref{chapi3},
puis on supprime les d\'enominateurs (th.\,\ref{facto2}).

{\rm (ii)} L'id\'ee de la construction est d'utiliser le fait qu'on sait, gr\^ace au th.\,\ref{intro2},
que les sp\'ecialisations aux points classiques de ${\bf z}_{(0,\infty)}(\rho_T)$ proviennent de
classes globales. On peut esp\'erer en d\'eduire, via la suite exacte de Poitou-Tate
et la zariski-densit\'e des points classiques, que ${\bf z}_{(0,\infty)}(\rho_T)$ lui-m\^eme
provient d'une classe globale (l'unicit\'e r\'esulte formellement des r\'esultats de Kato).

$\bullet$ 
Si $H^2(G_{\Q,S},\Lambda\wotimes\rho_x)$ est de type fini sur $\Z_p$ {\rm (condition
$\mu=0$)}, cette strat\'egie marche parfaitement car les
obstructions \'eventuelles s'annulent (\S\,\ref{EUL17}).

$\bullet$ Sans l'hypoth\`ese $\mu=0$, tout ce qu'on obtient
par cette m\'ethode (\S\,\ref{geni0})
est que l'obstruction est de $T$-torsion, ce qui explique l'introduction de ${\rm Fr}(T)$.

(iii) L'\'elimination des d\'enominateurs utilise
la factorisation du syst\`eme des \'el\'ements de Beilinson-Kato du th.\,\ref{Intro4} 
(cf.~(iii) de la rem.\,\ref{Intro5}).

(iv) On peut se demander s'il n'y aurait pas une preuve plus directe du fait
que $(0,\infty)$ fournit une classe de cohomologie globale: dans l'\'etat, cela utilise
les r\'esultats de Kato sur la th\'eorie d'Iwasawa des formes modulaires, 
ceux d'Emerton sur la factorisation de la cohomologie
compl\'et\'ee, ainsi que des propri\'et\'es fines de la correspondance de Langlands locale $p$-adique.

(v) La m\^eme m\'ethode construit un syst\`eme d'Euler dont ${\bf z}_{\rm Iw}(\rho_T)$
est la base.

(vi) Nakamura~\cite{naka2} construit directement un \'el\'ement
${\bf z}'_{\rm Iw}(\rho_T)$ qui interpole les \'el\'ements de Kato en les points
classiques \`a partir de la factorisation
du syst\`eme des \'el\'ements de Beilinson-Kato. Cette propri\'et\'e d'interpolation
et la zariski-densit\'e des points classiques
impliquent que son \'el\'ement co\"{\i}ncide avec le notre.
\end{rema}

\Subsection{Factorisation du syst\`eme de Beilinson-Kato}\label{intro6}

\subsubsection{Le syst\`eme de Beilinson-Kato}\label{cro8}
Le syst\`eme de Beilinson-Kato~\cite{Ka4,bbk-Kato,shanwen} est un \'el\'ement 
$${\bf z}_{\rm Kato}\in H^2(\Pi_\Q,{\cal D}_{\rm alg}({\bf M}_2'(\A^{]\infty[}),\Q_p)(2))$$
construit \`a partir des unit\'es des Siegel. Dans le membre de droite,

$\bullet$ $\A^{]\infty[}=\Q\otimes\cZ$ est le groupe des ad\`eles finies de $\Q$,

$\bullet$ ${\bf M}'_2$ est l'ensemble des $\matrice{a}{b}{c}{d}$ avec $\vecteur{a}{c}\neq\vecteur{0}{0}$
et $\vecteur{b}{d}\neq\vecteur{0}{0}$, 

$\bullet$ ${\cal D}_{\rm alg}({\bf M}_2'(\A^{]\infty[}),\Q_p)$
est l'ensemble des distributions alg\'ebriques, i.e.~des 
formes lin\'eaires sur l'espace ${\rm LC}_c({\bf M}_2'(\A^{]\infty[}))$
des fonctions localement constantes \`a support compact dans
${\bf M}_2'(\A^{]\infty[})$, 

$\bullet$ $\Pi_\Q$ est le groupe fondamental arithm\'etique de la courbe modulaire $Y(1)/\Q$
(vue comme un champs alg\'ebrique: on a une suite exacte
$1\to \widehat{{\bf SL}_2(\Z)}\to \Pi_\Q\to G_\Q\to 1$, o\`u
$\widehat{{\bf SL}_2(\Z)}$ est le compl\'et\'e profini de ${\bf SL}_2(\Z)$ et $G_\Q$
le groupe de Galois absolu de $\Q$).

\vskip.2cm
Si $\phi\in {\rm LC}_c({\bf M}_2'(\A^{]\infty[}))$, il existe $N$ tel que $\phi$
soit fixe par $\wGamma(N)$ et alors le cup-produit
de ${\bf z}_{\rm Kato}$ et $\phi$ fournit un \'el\'ement
$$\int\phi\,{\bf z}_{\rm Kato}\in H^2_{\eet}(Y(N)_{\Q(\bmu_N)},\Q_p(2))
\cong H^1(G_{\Q(\bmu_N)},H^1_{\eet}(Y(N)_{\Qbar},\Q_p(2))),$$
le dernier isomorphisme provenant de la suite spectrale de Hochschild-Serre.
Autrement dit, ${\bf z}_{\rm Kato}$ est une machine \`a produire des classes de cohomologie
galoisienne 
v\'erifiant des relations de distribution,
\`a valeurs dans la cohomologie \'etale g\'eom\'etrique
des courbes modulaires.

Maintenant, on peut multiplier ${\bf z}_{\rm Kato}$ par un op\'erateur $B_p^{c,d}$, avec $c,d\in\cZ^\dual$,
de mani\`ere \`a supprimer les d\'enominateurs (i.e.~remplacer $\Q_p(2)$ par $\Z_p(2)$,
ce qui fournit des mesures au lieu de distributions alg\'ebriques,
et donc permet d'int\'egrer des fonctions continues et pas seulement localement constantes); cela
introduit des facteurs parasites simples dans les formules.

Si $S$ est un ensemble fini de nombres premiers contenant~$p$,
on peut se restreindre aux fonctions de la forme ${\bf 1}_{{\bf M}_2(\cZ{]S[})}\otimes\phi_S$,
o\`u $\cZ^{]S[}=\prod_{\ell\notin S}\Z_\ell$, et $\phi_S$ est une fonction continue
\`a support compact dans $\GG(\Q_S)$, avec $\Q_S=\prod_{\ell\in S}\Q_\ell$.
Cela fournit
$${\bf z}_{{\rm Kato}}^{S,c,d}\in H^1(G_{\Q,S}, \widehat H^1_S(2)),$$
o\`u $\widehat H^1_S=H^1(\GG(\Z[\frac{1}{S}]),\cdo (\GG(\Q_S),L))$, et
$\widehat H^1_S(2)$ est le dual de Tate de $\widehat H^1_{c,S}$,
o\`u $$\widehat H^1_{c,S}:=H^1_c(\GG(\Z[\tfrac{1}{S}]),{\cal C}(\GG(\Q_S),L))$$
est la cohomologie compl\'et\'ee de la tour des formes modulaires de niveaux
divisibles par les $\ell\in S$.

\subsubsection{La factorisation}\label{cro9}
On note $T'$ l'alg\`ebre de Hecke correspondant \`a la repr\'esentation $\rho_T^\diamond(2)$,
et on pose ${\cal X}'={\rm Spec}(T')$. L'application
$\rho\mapsto\rho\otimes (\det\rho)^{-1}\chi_{\rm cycl}^2$ induit des isomorphismes
${\cal X}\cong{\cal X}'$ et $T'\cong T$, ce qui permet de voir tous les $T'$-modules associ\'es
\`a $\rho_T^\diamond(2)$ comme des $T$-modules.
On dispose de 
$$\iota_{T'}:\rho_T^\diamond(1)\otimes_T\Pi_p(\rho_T(-1))\to \widehat H^1_{c,S}$$
et donc, en composant avec $(0,\infty)$, cela fournit une forme lin\'eaire
sur $\rho_T^\diamond(1)\otimes_T\Pi_p(\rho_T(-1))$,
et donc un \'el\'ement
$$(0,\infty)_{T',S}\in \rho_T(-1)\otimes_T\Pi_p^\dual(\rho_T(-1)).$$

Par dualit\'e, $\iota_{T'}$ fournit une surjection
$\widehat H^1_S(2)\to \rho_T \otimes_T\Pi_p^\dual(\rho_T(-1))$.
On note $${\bf z}_{{\rm Kato}}^{S,c,d}(\Lambda\wotimes\rho_T)
\in H^1(G_{\Q,S}, \Lambda\wotimes\rho_T\otimes_T\Pi_p^\dual(\rho_T(-1)))
\cong H^1(G_{\Q,S}, \Lambda\wotimes\rho_T)\otimes_T\Pi_p^\dual(\rho_T(-1))$$
l'image de ${\bf z}_{{\rm Kato}}^{S,c,d}$.

Enfin, on dispose de ${\bf z}_{\rm Iw}(\rho_T)\in \rho_T^\diamond\otimes_TH^1(G_{\Q,S},\Lambda\wotimes\rho_T)$
(gr\^ace au th.\,\ref{intro4}),
et d'un accouplement naturel $\rho_T(-1)\otimes_T \rho_T^\diamond\to T$ 
(on fixe une orientation du motif de Tate, cf.~\no\ref{como112.2}).
\begin{theo}\phantomsection\label{Intro4}
Dans $\Pi_p^\dual(\rho_T(-1))\otimes_T H^1(G_{\Q,S}, \Lambda\wotimes\rho_T)$,
on a la factorisation suivante:
$${\bf z}_{{\rm Kato}}^{S,c,d}(\Lambda\wotimes\rho_T)=
\big(\tfrac{1}{2}\big(1+\matrice{-1}{0}{0}{1}\big)\star B_p^{c,d}
\star(0,\infty)_{T',S}\big)\otimes {\bf z}_{\rm Iw}(\rho_T).$$
\end{theo}
\begin{rema}\phantomsection\label{Intro5}
(i) On d\'eduit du th\'eor\`eme que l'on peut diviser ${\bf z}_{{\rm Kato}}^{S,c,d}(\Lambda\wotimes\rho_T)$
par le facteur $B_p^{c,d}$ que l'on avait d\^u introduire pour des raisons d'int\'egralit\'e;
cela fournit 
$${\bf z}_{{\rm Kato}}^{S}(\Lambda\wotimes\rho_T)
\in H^1(G_{\Q,S}, \Lambda\wotimes\rho_T)\otimes_T\Pi_p^\dual(\rho_T(-1))$$

(ii) Par construction, la localisation  en $p$ de ${\bf z}_{\rm Iw}(\rho_T)$
est $(0,\infty)_{T,S}$ et on obtient une factorisation
$${\rm loc}_p({\bf z}_{{\rm Kato}}^{S}(\Lambda\wotimes\rho_T))=
\tfrac{1}{2}\big(1+\matrice{-1}{0}{0}{1}\big)\star\big((0,\infty)_{T',S}\otimes (0,\infty)_{T,S}\big)$$
sous la forme que l'on esp\'erait.

(iii) La d\'emonstration repose sur les ingr\'edients suivants:

$\bullet$ On prouve (th.\,\ref{fact2})
la factorisation du th\'eor\`eme 
en restriction \`a un point classique. Cela consiste \`a
\'evaluer les deux membres sur des fonctions tests bien choisies; le membre de gauche
a \'et\'e calcul\'e par Kato~\cite{Ka4} (cf.~aussi th.\,\ref{shi4} ci-apr\`es)
et fait intervenir un produit de deux 
valeurs sp\'eciales de fonctions $L$ de formes
modulaires, et chacun des termes du membre de droite fait intervenir une
des valeurs sp\'eciales.

$\bullet$ On d\'eduit (lemme~\ref{facto1}) 
de cette factorisation ponctuelle la factorisation du th\'eor\`eme.

$\bullet$
Le th\'eor\`eme d'Ash-Stevens~\cite{AS} implique que $(0,\infty)_{T',S}$ 
{\og n'a pas de z\'eros\fg} et on en d\'eduit (th.\,\ref{facto2}) que ${\bf z}_{\rm Iw}(\rho_T)$
n'a pas de p\^ole, ce qui conclut la preuve
du th.\,\ref{intro4}.  
\end{rema}

\Subsection{C\'eoukonfaikoi}\label{cro10}
Cet article comporte quatre parties de th\'ematiques assez diff\'erentes.

\vskip2mm
\noindent $\bullet$ La partie I est consacr\'ee aux correspondances de Langlands pour ${\rm GL}_2$ (sur $\Q$).

$\diamond$
Le chap.\,\ref{chapi1} rassemble les notations utilis\'ees dans tout l'article.

$\diamond$
Le chap.\,\ref{llp0} 
fait un r\'esum\'e des propri\'et\'es de la correspondance de Langlands locale $p$-adique
pour ${\bf GL}_2(\Q_p)$, et donne quelques compl\'ements (comme l'injection
naturelle $\Pi(V)\hookrightarrow \tA^-\otimes V$ de la prop.\,\ref{cano1} ou
la prop.\,\ref{kirp3} \`a la base des r\'esultats mentionn\'es dans le (i) de la rem.\,\ref{intox5}).

$\diamond$
Le chap.\,\ref{EUL2} est consacr\'e \`a la construction de $\Pi(\rho_T)$
et de son mod\`ele de Kirillov.  Il contient une \'etude de la correspondance de
Langlands locale en famille pour ${\bf GL}_2(\Q_\ell)$.

\vskip2mm
\noindent $\bullet$ La partie II est consacr\'ee \`a la cohomologie du groupe $\GG(\Q)$ et ses avatars.

$\diamond$
Le chap.\,\ref{cup0} \'etudie le cohomologie de $\GG(\Q)$ agissant
sur divers espaces fonctionnels ad\'eliques, et fait le lien avec la cohomologie
compl\'et\'ee d'Emerton.

$\diamond$
Le chap.\,\ref{fma0} ad\'elise la th\'eorie des formes modulaires classiques
et \'etudie le lien avec la cohomologie compl\'et\'ee, tandis que
le chap.\,\ref{como1} fait le pont entre le point de vue cohomologie des groupes et le point de vue
g\'eom\'etrique; ce chapitre
 inclut en particulier la d\'efinition de l'application d'Eichler-Shimura $p$-adique. 

$\diamond$
Le chap.\,\ref{chapi2} \'etudie les multiplicit\'e des repr\'esentations
lisses de $\GG(\Ai)$ dans divers incarnations de la cohomologie compl\'et\'ee.

\vskip2mm
\noindent $\bullet$ La partie III est consacr\'ee \`a la factorisation de la cohomologie compl\'et\'ee.

$\diamond$
Le chap.\,\ref{qq8} est consacr\'e au calcul de la cohomologie \'etale des boules unit\'e ouvertes
usuelle et perfecto\"{\i}de, avec en particulier la d\'efinition de l'application $\iota_{\tA}$
et une loi de r\'eciprocit\'e explicite qui joue un grand r\^ole dans la suite.

$\diamond$
Le chap.\,\ref{modkir} d\'efinit le mod\`ele de Kirillov de la cohomologie compl\'et\'ee,
et les chap.\,\ref{rCDN0} et~\ref{IG0} \'etudie l'injectivit\'e de ce mod\`ele;
le chap.\,\ref{Ki114} explore les cons\'equences directes de cette injectivit\'e.

$\diamond$
Le chap.\,\ref{YEUL1} est consacr\'e \`a la preuve du th.\,\ref{Intro2}.

\vskip2mm
\noindent $\bullet$ La partie IV est consacr\'ee \`a l'\'etude du syst\`eme de Beilinson-Kato.

$\diamond$
Le chap.\,\ref{Symbo1} d\'emontre le th.\,\ref{intro2}.

$\diamond$
Le chap.\,\ref{chapi3} prouve le th.\,\ref{intro4} (apr\`es extension des scalaires \`a ${\rm Fr}(T)$).

$\diamond$
Le chap.\,\ref{chapi4} est consacr\'e \`a des rappels sur le syst\`eme de Beilinson-Kato
et le chap.\,\ref{chapi5} \`a la preuve du th.\,\ref{Intro4}.

\subsubsection*{Remerciements}
Ce projet a d\'ebut\'e en 2014 et s'est poursuivi au gr\'e des invitations entre la France et la Chine,
dont deux mois et demi en Chine \`a l'automne 2016, 
entre le BICRM de P\'ekin, l'universit\'e Fudan
de Shanghai et le centre de conf\'erences de Sanya. 
P.C. voudrait remercier en particulier
le SCMS, le Fudan Scholar programm et l'universit\'e de Fudan pour leurs invitations
\`a Shanghai en mars 2018 et d\'ecembre~2018; 
S.W. voudrait remercier les universit\'es d'Essen et Regensburg et le CRM de Montr\'eal
pour leur hospitalit\'e de mars 2014 \`a juillet 2015, ainsi que Laurent Berger, Gabriel Dospinescu,
l'ENS-Lyon et l'IMJ-PRG pour des invitations en avril 2018, novembre 2018 et ao\^ut 2019.

Une premi\`ere version a \'et\'e d\'epos\'ee sur arXiv en avril 2021 (2104.09200 [math.NT]).  
La preuve de la factorisation du syst\`eme de Beilinson-Kato y \'etait nettement plus
tortueuse car nous n'avions pas r\'eussi \`a prouver l'injectivit\'e du mod\`ele
de Kirillov de la cohomologie compl\'et\'ee.  C'est Vincent Pilloni qui nous a sugg\'er\'e,
\`a la fin d'un expos\'e \`a Lyon en juin 2021,
une strat\'egie utilisant les travaux de Lue Pan (une variante de la seconde approche de
la rem.\,\ref{intox3}) pour prouver cette injectivit\'e. La r\'ealisation de cette strat\'egie
a b\'en\'efici\'e de conversations avec Gabriel Dospinescu, Lue Pan et Juan-Esteban Rodriguez-Camargo
et de l'hospitalit\'e du MSRI de Berkeley lors du premier trimestre de 2023 et du HIM de Bonn
lors du second semestre.  Nous voudrions remercier tous les acteurs sus-nomm\'es de leur aide.

{\Small
\tableofcontents
}

\part{Correspondances de Langlands locales}
\section{Pr\'eliminaire}\label{chapi1}
Le but de ce chapitre est de fixer un certain nombre de notations et de normalisations.
\Subsection{Ad\`eles}\label{prelim1}
\subsubsection{Notations}\label{prelim2}
On note ${\cal P}$ \index{premiers@\premiers}l'ensemble des nombres premiers.  
Si $S\subset{\cal P}$, on dit qu'un entier $N$ est {\it \`a support dans $S$} si tous ses diviseurs
premiers appartiennent \`a $S$.

Si $v\in {\cal P}\cup\{\infty\}$
est une place de $\Q$, on note $\Q_v$ le compl\'et\'e de $\Q$ en $v$ (et donc $\Q_\infty=\R$).
On note $\A$ l'anneau des \index{adeles@\adeles}ad\`eles de $\Q$, produit restreint des $\Q_v$:
si $\cZ=\prod_{p\in{\cal P}}\Z_p$ est le compl\'et\'e profini de $\Z$, on a
$\A=\R\times(\Q\otimes\widehat\Z)$.

Si $S\subset {\cal P}\cup\{\infty\}$ est fini, \index{Q@\QQQ}on \index{Z@\ZZZ}pose
$$\Z_S=\prod_{\ell\in S\cap{\cal P}}\Z_\ell, 
\quad\Z[\tfrac{1}{S}]=\Z[\tfrac{1}{\ell},\,\ell\in S\cap{\cal P}],
\quad \Q_S=\prod_{v\in S}\Q_v$$
On note $\A^{]S[}$ le produit restreint des $\Q_v$ pour $v\notin S$.
Par exemple, $\A^{]\infty[}=\Q\otimes\cZ$. On a $\A=\Q_S\times \A^{]S[}$ pour tout $S$.

Si $x$ est un objet ad\'elique (i.e., un \'el\'ement de $\A$, $\A^\dual$, ${\bf GL}_n(\A)$,
etc.), et si 
$v$ est une place de $\Q$, on note $x_v$ la composante de $x$ en $v$ (et donc $x=(x_v)_v$).
Si
$S$ est un ensemble fini de places de $\Q$, on note $x_S=(x_v)_{v\in S}$ la composante de $x$ 
en les \'el\'ements de $S$ et
$x^{]S[}=(x_v)_{v\notin S}$ la composante de $x$ hors de $S$ (et donc $x=(x_S,x^{]S[})$.

Si $\GG$ est un groupe alg\'ebrique sur $\Q$, les projections $\GG(\A)\to \GG(\Q_v)$ ont
des sections naturelles qui permettent de consid\'erer les
$\GG(\Q_v)$ comme des sous-groupes de $\GG(\A)$.

De m\^eme, si $S$ est un ensemble fini de places de $\Q$, alors $\GG(\Q_S)=\prod_{v\in S}\GG(\Q_v)$
est naturellement un sous-groupe de $\GG(\A)$ et tout \'el\'ement de $\GG(\A)$ peut
s'\'ecrire de mani\`ere unique sous la forme $x_Sx^{]S[}$ avec $x_S\in \GG(\Q_S)$ et
$x^{]S[}\in \GG(\A^{]S[})$; de plus $x_S$ et $x^{]S[}$ commutent.

\subsubsection{Caract\`eres additifs}\label{prelim3}
On \index{expo@\exponent}note ${\bf e}_\infty:\C\to\C^\dual$ le caract\`ere du groupe additif
$${\bf e}_\infty(\tau)=e^{-2i\pi\,\tau}.$$
On note ${\bf e}_\A:\A\to\C^\dual$ le caract\`ere du groupe additif
se factorisant \`a travers $\A/\Q$ et dont la restriction \`a $\R\subset\A$
est ${\bf e}_\infty$.
Si $\ell$ est un nombre premier, on note ${\bf e}_\ell$ la restriction de ${\bf e}_\A$
\`a $\Q_\ell\subset \A$.  Alors ${\bf e}_\ell$ se factorise \`a travers $\Q_\ell/\Z_\ell=
\Z[\frac{1}{\ell}]/\Z\subset \R/\Z$ et on a
$${\bf e}_\ell(x_\ell)=e^{2i\pi\,\overline x_\ell}$$
 o\`u $\overline x_\ell$ est
l'image de $x$ dans $\R/\Z$ par l'inclusion ci-dessus.
On a alors
$${\bf e}_\A((x_v)_v)=\prod_v{\bf e}_v(x_v).$$

\subsubsection{Caract\`eres multiplicatifs}\label{prelim4}
Soit $\omega:\A^\dual\to \C^\dual$ un caract\`ere continu. Si $v$ est
une place de $\Q$, on note $\omega_v$ la restriction de $\omega$ \`a $\Q_v^\dual$.
Il existe un entier $N\geq 1$ et un caract\`ere $\tilde\omega:(\Z/N)^\dual\to \C^\dual$
tel que 
la restriction de $\omega$ \`a $\cZ^\dual$ s'obtienne en composant $\tilde\omega$
avec la projection naturelle $\cZ^\dual\to(\Z/N)^\dual$.

Si $\omega$ se factorise \`a travers $\A^\dual/\Q^\dual$, et
si $p\nmid N$, alors $\omega_p(p)=\tilde\omega^{-1}(p)\omega_\infty(p)^{-1}$.

\vskip.1cm
\noindent $\bullet$ {\it Le caract\`ere norme $|\ |_\A$}.
Si $v$ est une place de $\Q$, on note $|\ |_v$ 
\index{adeles@\adeles}la norme sur $\Q_v$ (si $v$ est la place
correspondant \`a un nombre premier $\ell$, on a $|\ell|_v=\ell^{-1}$).
On d\'efinit $|\ |_\A$ par la formule
$$|x|_\A=\prod_v|x_v|_v.$$
La formule du produit implique que $|\ |_\A$ se factorise \`a travers $\A^\dual/\Q^\dual$.

\vskip.1cm
\noindent $\bullet$ {\it Le caract\`ere $\delta_\A $}.
On note $\delta_\A :\A^\dual\to \C^\dual$ le caract\`ere
$$x\mapsto \delta_\A (x)=x_{\infty}^{-1}|x|_\A,$$ 
et donc $\delta_\A $ est localement constant, \`a valeurs
dans $\Q^\dual$.
De plus,
$$\delta_{\A,\infty}={\rm sign},\quad \delta_{\A,\ell}=|\ |_\ell,
\quad
\delta_\A =|\ |_\A\ {\text{sur $\Aidu$}}.$$ 

\vskip.1cm
\noindent $\bullet$ {\it Caract\`eres alg\'ebriques}.
Un caract\`ere continu $\chi:\A^\dual/\Q^\dual\to \C^\dual$ est dit {\it alg\'ebrique,
de poids $a\in\Z$}, si $\chi(x_\infty)=x_\infty^a$ pour tout $x_\infty\in\R_+^\dual$.
Par exemple, le caract\`ere $|\ |_\A$
est alg\'ebrique de poids $1$, et
$|\ |_\A^a$ est alg\'ebrique de poids $a$.

De m\^eme, un caract\`ere continu $\chi:\A^\dual/\Q^\dual\to \Qbar_p^\dual$ est dit {\it alg\'ebrique,
de poids $a\in\Z$}, s'il existe $n\geq 1$ tel que $\chi(x_p)=x_p^a$ pour tout $x_p\in 1+p^n\Z_p$.
Notons qu'un caract\`ere continu $\chi:\A^\dual/\Q^\dual\to \Qbar_p^\dual$ est automatiquement
unitaire car il se factorise par $\A^\dual/\Q^\dual\R_+^\dual$ qui est compact (isomorphe
\`a $\cZ^\dual$).

Si $\chi$ est alg\'ebrique de poids $a$, alors $\chi=|\ |_\A^a\eta$, o\`u
$\eta:\A^\dual/\Q^\dual\to \C^\dual$ est d'ordre fini, et donc \`a valeurs
dans $\Qbar$ (en fait dans une extension finie de $\Q$).  
Le caract\`ere $x\mapsto x_\infty^{-a}\chi(x)$ est aussi
\`a valeurs dans $\Qbar$ et on \index{chi@\caract}d\'efinit
$\chi^{(p)}:\A^\dual/\Q^\dual\to \Qbar_p^\dual$
par la formule $$x_p^{-a}\chi^{(p)}(x)=x_\infty^{-a}\chi(x)=\eta\delta_\A ^a(x),$$ 
o\`u tous les membres
appartiennent \`a $\Qbar$.  Le caract\`ere $\chi^{(p)}$ est alg\'ebrique
de poids $a$ et $\chi\mapsto\chi^{(p)}$ est une bijection
de l'ensemble des caract\`eres alg\'ebriques $\A^\dual/\Q^\dual\to 
\C^\dual$ sur celui des caract\`eres alg\'ebriques de $\A^\dual/\Q^\dual\to
\Qbar_p^\dual$; cette bijection pr\'eserve le poids des caract\`eres.
Si $\chi_\infty(x_\infty)={\rm sign}(x_\infty)^e|x_\infty|^a$, avec $e\in\{0,1\}$, alors
$$\chi^{(p)}_\infty={\rm sign}^{a+e},\quad \chi^{(p)}_p(x_p)=x_p^a\chi(x_p)
\quad{\rm et}\quad \chi^{(p)}_\ell=\chi_\ell,\ {\text{si $\ell\neq \infty,p$.}}$$
\subsubsection{Groupes de Galois, de Weil et de Weil-Deligne}\label{prelim6}
On note $\Qbar$ la cl\^oture alg\'ebrique de $\Q$ dans $\C$ et on fixe
un plongement de $\Qbar$ dans $\Qbar_\ell$, pour tout $\ell$.

Si $K=\Q,\Q_\ell$, on note $G_K={\rm Gal}(\overline K/K)$ le groupe de \index{galois@\galois}Galois
absolu de $K$ et $G_K^{\rm ab}$ son ab\'elianis\'e.  
On note ${\rm W}_{\Q_\ell}\subset G_{\Q_\ell}$ le groupe de \index{weil@\weil}Weil
et ${\rm WD}_{\Q_\ell}$ le groupe de Weil-Deligne.

\subsubsection{L'extension cyclotomique}\label{prelim7}
On note $\Q^{\rm cycl}$ l'extension \index{Q@\QQQ}cyclotomique de $\Q$
obtenue en rajoutant toutes les racines de l'unit\'e
(c'est aussi l'extension
ab\'elienne maximale de $\Q$, i.e. ${\rm Gal}(\Q^{\rm cycl}/\Q)=G_\Q^{\rm ab}$).  
Le caract\`ere \index{epsil@\cyclo}cyclotomique $\cy: G_\Q\to \cZ^\dual$
fournit une identification
$$\cy:G_\Q^{\rm ab}=\cZ^\dual\quad {\text{(d'inverse $u\mapsto\sigma_u$)}}.$$
Il fournit aussi des identifications ${\rm Gal}(\Q(\bmu_M)/\Q)=(\Z/M)^\dual$,
pour tout $M\geq 1$.
\begin{rema}\phantomsection\label{cyclo10}
{\rm (i)} L'application
$x\mapsto\phi_x$, o\`u $\phi_x(u)=\sigma_u(x)$ si $u\in\cZ^\dual$,
induit un isomorphisme (${\rm LC}$ d\'esigne l'espace des \index{rml@LC}fonctions localement constantes)
$$\Q^{\rm cycl}\overset{\sim}{\longrightarrow}{\rm LC}(\cZ^\dual,\Q^{\rm cycl})^{\cZ^\dual},$$
o\`u $a\in \cZ^\dual$ agit sur $\phi$ par $(\sigma_a\cdot\phi)(u)=\sigma_a(\phi(a^{-1}u))$.

{\rm (ii)} Si $K\supset\Q^{\rm cycl}$, l'application
$x\otimes y\mapsto (u\mapsto x\sigma_u(y))$ induit un isomorphisme
$$K\otimes_\Q\Q^{\rm cycl}\overset{\sim}{\longrightarrow}{\rm LC}(\cZ^\dual,K)$$
de $K$-alg\`ebres.  De plus, $\phi\mapsto \int_{\cZ^\dual}\phi$, de ${\rm LC}(\cZ^\dual,K)$ dans $K$,
correspond via cet isomorphisme \`a $1\otimes {\rm Tr}$, o\`u ${\rm Tr}:\Q^{\rm cycl}\to\Q$
est la trace normalis\'ee (\'egale \`a $\frac{1}{[F:\Q]}{\rm Tr}_{F/\Q}$ sur $F$, si
$F\subset\Q^{\rm cycl}$ est une extension finie de $\Q$).
\end{rema}

\subsubsection{Th\'eorie du corps de classes}\label{prelim8}
L'inclusion de $\Qbar$ dans $\Qbar_\ell$ induit une 
injection $G_{\Q_\ell}\hookrightarrow G_\Q$ ainsi qu'une fl\`eche
$G_{\Q_\ell}^{\rm ab}\to G_\Q^{\rm ab}$ qui se trouve \^etre injective
car l'extension ab\'elienne maximale de $\Q_\ell$ est aussi l'extension cyclotomique.

On \index{sigma@\sigm}note $\sigma_\ell$ l'\'el\'ement de $G_{\Q_\ell}^{\rm ab}$
agissant trivialement sur $\Q_\ell(\bmu_{\ell^\infty})$ et
par $x\mapsto x^\ell$ sur $\overline{\bf F}_\ell$ (i.e.~$\sigma_\ell$
est le frobenius arithm\'etique dans $G_{\Q_\ell}^{\rm ab}$).
L'image de $\sigma_\ell$ dans $G_\Q^{\rm ab}=\cZ^\dual$ est~$\ell^{]\infty,\ell[}$.

Par ailleurs, l'inclusion $\cZ^\dual\hookrightarrow \A^\dual$
induit un isomorphisme $\cZ^\dual\cong \A^\dual/\R_+^\dual\Q^\dual$.
L'image de $\sigma_\ell$ dans $\A^\dual/\R_+^\dual\Q^\dual$ est aussi
\'egale \`a $\ell_\ell^{-1}=\ell^{]\infty,\ell[}\ell_\infty\ell^{-1}$.
Comme l'image du sous-groupe d'inertie de $G_{\Q_\ell}^{\rm ab}$
dans $\cZ^\dual$ est $\Z_\ell^\dual$, on voit
que $G_{\Q_\ell}^{\rm ab}$ s'envoie dans l'adh\'erence de $\Q_\ell^\dual$
dans $\A^\dual/\R_+^\dual\Q^\dual$ et que l'image de $W_{\Q_\ell}^{\rm ab}$
est $\Q_\ell^\dual$.

Le diagramme suivant est commutatif:
$$\xymatrix{W_{\Q_\ell}\ar[r]\ar[d]&G_{\Q_\ell}\ar[r]&G_\Q\ar[r]^-{\cy}&\cZ^\dual\ar[d]^-{\wr}\\
\Q_\ell^\dual\ar[rrr]&&&\A^\dual/\R_+^\dual\Q^\dual}$$

Si $[L:\Q_p]<\infty$, on d\'eduit de ce qui pr\'ec\`ede des identifications
entre:

$\bullet$ caract\`eres continus $\chi_{\rm Gal}:W_{\Q_\ell}\to L^\dual$
et caract\`eres continus $\chi:\Q_\ell^\dual\to L^\dual$,

$\bullet$ caract\`eres continus $\chi_{\rm Gal}:G_{\Q_\ell}\to \O_L^\dual$,
caract\`eres continus $\chi_{\rm Gal}:W_{\Q_\ell}\to \O_L^\dual$
et caract\`eres continus $\chi:\Q_\ell^\dual\to \O_L^\dual$

$\bullet$ caract\`eres continus $\omega:\A^\dual/\Q^\dual\to\O_L^\dual$
et caract\`eres continus $\omega_{\rm Gal}:G_\Q\to\O_L^\dual$.

Dans les identifications \index{chi@\caract}ci-dessus, $\chi_{\rm Gal}(\sigma_\ell)=\chi(\ell_\ell^{-1})$
et $\omega_{\rm Gal}(\sigma)=\omega(\cy(\sigma))$,
et si $\omega_\ell$ est la restriction de $\omega$ \`a $\Q_\ell^\dual$,
alors $\omega_{{\rm Gal},\ell}=\omega_{\ell,{\rm Gal}}$

\subsubsection{Le caract\`ere cyclotomique}\label{prelim9}
Par exemple, le caract\`ere $p$-cyclotomique
$$\cyp :G_{\Q_\ell}^{\rm ab}\to \Z_p^\dual,$$
obtenu par action sur $\bmu_{p^\infty}$,
correspond:

$\bullet$
au caract\`ere $|\ |_\ell$ de $\Q_\ell^\dual$, si $\ell\neq p$: on a
$\cyp (\sigma_\ell)=|\ell_\ell^{-1}|_\ell=\ell$;

$\bullet$ 
au caract\`ere
$x\mapsto x|x|_p$ de $\Q_p^\dual$, si $\ell=p$.

\subsubsection{Sommes de Gauss globales}\label{prelim10}
Si $\chi=\eta|\ |_\A^a$, o\`u $\eta$ est de conducteur $N$, et si $\tilde\eta:(\Z/N)^\dual\to\C^\dual$
est le caract\`ere de Dirichlet associ\'e, d\'efinissons la somme de \index{gauss@\gauss}Gauss $G(\chi)$
par 
$$G(\chi)=G(\tilde\eta)=\sum_{b\in(\Z/N)^\dual}\tilde\eta(b){\bf e}_\infty(\tfrac{b}{N}).$$
\begin{rema}\phantomsection\label{prelim10.1}
{\rm (i)} Si $\chi$ est \`a valeurs dans $L$, alors $G(\chi)\in(L\otimes \Q^{\rm cycl})^\dual$
et, si $a\in\cZ^\dual$, alors
$$\sigma_a(G(\chi))=\sum_{b\in(\Z/N)^\dual}\tilde\eta(b){\bf e}_\infty(\tfrac{ab}{N})=
\chi^{-1}(a)G(\chi).$$
{\rm (ii)} On d\'eduit du (i) que, 
$\frac{G(\chi_1\chi_2)}{G(\chi_1)G(\chi_2)}\in L^\dual\otimes 1\subset (L\otimes \Q^{\rm cycl})^\dual$,
si $\chi_1,\chi_2$ sont \`a valeurs dans $L$;
en particulier, $G(\chi)G(\chi^{-1})\in L^\dual\otimes 1$.
\end{rema}
\subsubsection{Sommes de Gauss locales}\label{prelim11}
Si $\eta:\Z_\ell^\dual\to L^\dual$ est un caract\`ere de conducteur~$\ell^n$,
on note $G(\eta)\in L\otimes\Z(\bmu_{\ell^\infty})$, la somme de Gauss
$$G(\eta):=\sum_{x\in(\Z/\ell^n\Z)^\dual}\eta(x)\otimes {\bf e}_\ell(\tfrac{x}{\ell^n}).$$
Si $a\in\Z_\ell^\dual$, on a 
$$\sigma_a(G(\eta))=\eta(a)^{-1}G(\eta).$$

\Subsection{Le groupe ${\bf GL}_2(\A)$}\label{prelim12}
Soit $$\GG:={\bf GL}_2$$
et \index{G@\GGG}soient
$$\UU:=\matrice{1}{*}{0}{1}\subset\PP:=\matrice{*}{*}{0}{1}\subset \BB:=\matrice{*}{*}{0}{*}\subset\GG$$
\index{U@\UUU}l'unipotent, \index{B@\BBB}le mirabolique, et \index{P@\PPP}le borel.
\subsubsection{Le groupe $\GG(\R)$}  \label{prelim13}
Si $g\in \GG(\R)$, on \index{sign@\signa}d\'efinit 
$${\rm sign}(g)\in\{\pm 1\}$$
comme le signe de $\det g$.
L'application $\matrice{a}{b}{-b}{a}\mapsto a+ib$ induit une identification
$$\C^\dual=
\big\{\matrice{a}{b}{-b}{a},\ a+ib\neq 0\big\}.$$
L'application $\matrice{a}{b}{c}{d}\mapsto\frac{ai+b}{ci+d}$ induit un isomorphisme
$$\GG(\R)/\C^\dual\simeq
{\cal H}:=\piqp(\C)\moins\piqp(\R)
={\cal H}^+\sqcup{\cal H}^-,$$ 
\index{H@\HHH}o\`u $${\cal H}^+=\{\tau \in\C,\ {\rm Im}(\tau )>0\},
\quad
{\cal H}^-=\{\tau \in\C,\ {\rm Im}(\tau )<0\}.$$

Le normalisateur de $\C^\dual$ agit par multiplication \`a droite (i.e. $(g,\tau)\mapsto \tau g$)
sur ${\cal H}$.
Ce normalisateur est $\C^\dual\sqcup \C^\dual \matrice{-1}{0}{0}{1}$ et
$\C^\dual$ agit trivialement sur ${\cal H}$ tandis que $\C^\dual \matrice{-1}{0}{0}{1}$
agit par 
$\tau \mapsto 
\overline\tau $ puisque $\matrice{a}{b}{c}{d}\matrice{-1}{0}{0}{1}=\matrice{-a}{b}{-c}{d}$
et $\frac{-ai+b}{-ci+d}$ est le conjugu\'e complexe de $\frac{ai+b}{ci+d}$.
On prolonge l'action du normalisateur de $\C^\dual$ en une action $(g,\tau)\mapsto \tau g$ de $\GG(\R)$
agissant \`a travers $\GG(\R)/\GG(\R)_+$, \`a ne pas confondre avec l'action naturelle
$(g,\tau)\mapsto g\cdot\tau$
par multiplication \`a gauche sur ${\cal H}=\GG(\R)/\C^\dual$:
$$
\matrice{a}{b}{c}{d}\cdot\tau=\frac{a\tau+b}{c\tau+d}
\quad{\rm et}\quad
\tau g=\begin{cases} \tau &{\text{si $g\in \GG(\R)_+$,}}\\ 
\overline\tau &{\text{si $g\notin \GG(\R)_+$.}}\end{cases}
$$

\subsubsection{Sous-groupes de congruence}\label{prelim14}
Si $N\in \N$, on \index{Gamma2@\Gam}note
\begin{align*}
\wGamma(N)={\rm Ker}(\GG(\cZ)\to \GG(\Z/N)),&\quad
\wGamma_0(N)=\big\{\matrice{a}{b}{c}{d}\in \GG(\cZ),\ c\in N\cZ\big\},\\
\Gamma(N)= {\rm Ker}({\rm SL}_2(\Z)\to {\rm SL}_2(\Z/N)),&\quad
\Gamma_0(N)=\big\{\matrice{a}{b}{c}{d}\in {\rm SL}_2(\Z),\ c\in N\Z\big\}
\end{align*}
On a donc:
\begin{align*}
\Gamma(N)={\rm SL}_2(\Z)\cap \wGamma(N),\quad & \Gamma_0(N)={\rm SL}_2(\Z)\cap \wGamma_0(N)\\
\Gamma_0(1)=\Gamma(1)={\rm SL}_2(\Z),\quad & \wGamma(1)=\GG(\cZ)
\end{align*}

$\bullet$ Tout \'el\'ement de $\GG(\A)$ peut s'\'ecrire $\gamma^{-1}g_\infty\kappa$
avec $\gamma\in \GG(\Q)$, $g_\infty\in \GG(\R)$ (resp.~$g_\infty\in \GG(\R)_+$)
et $\kappa\in \GG(\cZ)$, et cette \'ecriture est unique \`a
$(\gamma,g_\infty,\kappa)\mapsto(\alpha\gamma,\alpha g_\infty,\alpha\kappa)$ pr\`es,
avec $\alpha\in \GG(\Z)$ (resp.~$\alpha\in\Gamma(1)$).

\vskip.1cm
$\bullet$ Plus g\'en\'eralement, tout
\'el\'ement de $\GG(\A)$ peut s'\'ecrire $\gamma^{-1}g_\infty\kappa$
avec $\gamma\in \GG(\Q)$, $g_\infty\in \GG(\R)_+$
et $\kappa\in \wGamma_0(N)$, et cette \'ecriture est unique \`a
$(\gamma,g_\infty,\kappa)\mapsto(\alpha\gamma,\alpha g_\infty,\alpha\kappa)$ pr\`es,
avec 
$\alpha\in\Gamma_0(N)$.

\vskip.1cm
$\bullet$ Soit $S\subset {\cal P}\cup\{\infty\}$, fini. 

--- Si $\infty\in S$, tout \'el\'ement de $\GG(\Q_S)$ peut s'\'ecrire $\gamma^{-1}g_\infty\kappa$
avec $\gamma\in \GG(\Z[\frac{1}{S}])$, $g_\infty\in \GG(\R)$ 
et $\kappa\in \GG(\cZ)$, et cette \'ecriture est unique \`a
$(\gamma,g_\infty,\kappa)\mapsto(\alpha\gamma,\alpha g_\infty,\alpha\kappa)$ pr\`es,
avec $\alpha\in \GG(\Z)$.

--- Si $\infty\notin S$, tout \'el\'ement de $\GG(\Q_S)$ peut s'\'ecrire $\gamma^{-1}\kappa$
avec $\gamma\in \GG(\Z[\frac{1}{S}])$
et $\kappa\in \GG(\cZ)$, et cette \'ecriture est unique \`a
$(\gamma,\kappa)\mapsto(\alpha\gamma,\alpha\kappa)$ pr\`es,
avec $\alpha\in \GG(\Z)$.

\subsubsection{Repr\'esentations lisses de $\GG(\Q_\ell)$.}\label{prelim15}
Soit $L$ un corps de caract\'eristique~$0$ (que l'on consid\`ere muni de la topologie discr\`ete).
Si $\pi$ est une $L$-repr\'esentation lisse de $\GG(\Q_\ell)$, 
admettant un caract\`ere
\index{omegapi@\omegapi}central $\omega_\pi$, un {\it mod\`ele
de Kirillov} pour $\pi$ est une injection $\BB(\Q_\ell)$-\'equivariante
$$\pi\hookrightarrow
{\rm LC}(\Q_\ell^\dual,L\otimes\Q[\bmu_{\ell^\infty}])^{\Z_\ell^\dual}$$
dans le $L$-espace des $\phi:\Q_\ell^\dual\to L\otimes\Q(\bmu_{\ell^\infty})$,
localement constantes,
v\'erifiant $\sigma_a(\phi(x))=\phi(ax)$, pour tous $x\in\Q_\ell^\dual$ et $a\in\Z_\ell^\dual$
(o\`u $\sigma_a$ agit sur $\bmu_{\ell^\infty}$ par $\zeta\mapsto\zeta^a$),
muni de l'action de $\BB(\Q_\ell)$ donn\'ee par
$$\big(\matrice{a}{b}{0}{d}\cdot\phi\big)(x)=
\omega_{\pi}(d){\bf e}_\ell(\tfrac{bx}{d})\phi(\tfrac{ax}{d})$$

Si $\chi_1,\chi_2$ sont des caract\`eres localement constants de $\Q_\ell^\dual$,
on note $\chi_1\otimes|\ |_\ell^{-1}\chi_2$ le caract\`ere
$\matrice{a}{b}{0}{d}\mapsto \chi_1(a)\chi_2(d)|d|_\ell^{-1}$ de $\BB(\Q_\ell)$
et
${\rm Ind}_{\BB(\Q_\ell)}^{\GG(\Q_\ell)}(\chi_1\otimes|\ |_\ell^{-1}\chi_2)$
la repr\'esentation lisse induite, i.e.
$$\{\phi\in{\rm LC}(\GG(\Q_\ell),L),\
\phi\big(\matrice{a}{b}{0}{d}x\big)=\chi_1(a)\chi_2(d)|d|_\ell^{-1}\phi(x),
\ \forall x\in \GG(\Q_\ell),\ \matrice{a}{b}{0}{d}\in \BB(\Q_\ell)\},$$
l'action de $\GG(\Q_\ell)$ \'etant $(g\star\phi)(x)=\phi(xg)$.
Le mod\`ele de Kirillov $\phi\mapsto {\cal K}_\phi$ 
de $\pi={\rm Ind}(\chi_{1}\otimes \chi_{2}|\ |_\ell^{-1})$
est donn\'e par (la suite est constante pour $n$ assez grand):
$${\cal K}_\phi(y)=\lim_{n\to+\infty}\tfrac{1}{G(\chi_2)}\int_{\ell^{-n}\Z_\ell}
\phi\big(\matrice{0}{1}{-1}{0}\matrice{y}{x}{0}{1}\big){\bf e}_\ell(-x)\,dx.$$
Si $\chi_1$ et $\chi_2$ sont non ramifi\'es,
$H^0(\GG(\Z_\ell), {\rm Ind}_{\BB(\Q_\ell)}^{\GG(\Q_\ell)}(\chi_1\otimes|\ |_\ell^{-1}\chi_2))$
est de dimension~$1$.  Dans le mod\`ele de Kirillov,
cet espace est engendr\'e par la \index{vpi@\vpi}fonction
$v_{\pi}$ 
d\'efinie par
$$v_{\pi}(x)=
\begin{cases}
0& {\text{si $x\notin\Z_\ell$,}}\\
\chi_{1}(\ell)^n+\chi_{1}(\ell)^{n-1}\chi_{2}(\ell)+\cdots+
\chi_{2}(\ell)^n & {\text{si $x\in \ell^n\Z_\ell^\dual$ et $n\geq 0$.}}
\end{cases}$$
C'est {\it le nouveau vecteur normalis\'e}
(par la condition $v_{\pi}(1)=1$)
de $\pi$;
il engendre $\pi$ en tant que $\PP(\Q_\ell)$-module.

\index{Tp@\Tp}Si 
$T_\ell=\frac{1}{\ell}\big(\matrice{1}{0}{0}{\ell}_\ell+\sum_{b=0}^{\ell-1}\matrice{\ell}{b}{0}{1}_\ell\big)$,
on a
$$T_\ell\star v_{\pi}=(\chi_1(\ell)+\chi_2(\ell))v_{\pi}.$$

On note $v'_\pi$ la \index{vpi@\vpi}fonction
$$v'_\pi={\bf 1}_{\Z_\ell^\dual}\in
\pi;$$
 c'est un g\'en\'erateur du foncteur de Kirillov~\cite[\S\,4.1]{Em08} de $\pi$,
i.e.~le module
\begin{center}
$\big\{v\in \pi,\ \matrice{\Z_\ell^\dual}{\Z_\ell}{0}{1}\cdot v=v,
\ \sum_{i=0}^{\ell-1}\matrice{\ell}{i}{0}{1}_\ell\cdot v=0\big\}.$
\end{center}
Un petit calcul fournit la relation
\begin{equation}
\label{eu11}
\big(1-(\chi_{1}(\ell)+\chi_{2}(\ell))\matrice{\ell^{-1}}{0}{0}{1}_\ell
+(\chi_{1}(\ell)\chi_{2}(\ell))\matrice{\ell^{-2}}{0}{0}{1}_\ell\big)\star
v_{\pi}=v'_\pi
\end{equation}

\subsubsection{La correspondance de Langlands locale classique}\label{prelim16}
Soit $L$ un corps de caract\'eristique $0$ (muni de la topologie discr\`ete).
Soit $\rho:{\rm WD}_{\Q_\ell}\to {\bf GL}_2(L)$ une $L$-repr\'esentation continue
(c'est la donn\'ee
d'une repr\'esentation $\rho:{\rm W}_{\Q_\ell}\to {\bf GL}_2(L)$ et de $N\in{\bf M}_2(L)$ tels
que $N\rho(g)=\ell^{\deg g}\rho(g)N$, pour tout $g\in{\rm W}_{\Q_\ell}$).
On sait associer \`a $\rho$
\index{Pi@\Pip}une $L$-repr\'esentation $\Pi^{\rm cl}_\ell(\rho)$
de $\GG(\Q_\ell)$; la correspondance $\rho\mapsto\Pi^{\rm cl}_\ell(\rho)$ est normalis\'ee
(via la compatibilit\'e local-global)
pour que les fonctions $L$ globales co\"{\i}ncident.  Elle v\'erifie les
propri\'etes suivantes:

$\bullet$ $\Pi^{\rm cl}_\ell(\rho)$ est lisse, admissible, admet un caract\`ere central
$\omega_{\Pi_\ell(\rho)}$, et on a:
$$\omega_{\Pi_\ell(\rho)}=|\ |_\ell^{-1}\det\rho.$$

$\bullet$ $\Pi^{\rm cl}_\ell(\rho)$ admet
un mod\`ele de Kirillov.
De plus, dans ce mod\`ele, $\Pi^{\rm cl}_\ell(\rho)$ contient
${\rm LC}_{\rm c}(\Q_\ell^\dual,L\otimes\Q[\bmu_{\ell^\infty}])^{\Z_\ell^{\dual}}$
et le quotient est de dimension $0$, $1$ ou $2$ sur $L$ suivant que
$\rho$ est irr\'eductible, que $N$ agit non trivialement ou que $\rho$ n'est pas irr\'eductible
et $N$ agit trivialement.

$\bullet$ Si $N$ agit trivialement et si
$$\rho^{\rm ss}=\chi_1\oplus\chi_2,$$ avec
$\chi_2\neq|\ |_\ell\chi_1$ (ce que l'on peut imposer en \'echangeant $\chi_1$ et $\chi_2$),
alors
$$\Pi^{\rm cl}_\ell(\rho)={\rm Ind}_{\BB(\Q_\ell)}^{\GG(\Q_\ell)}(\chi_1\otimes|\ |_\ell^{-1}\chi_2).$$

$\bullet$ Si $N$ agit non trivialement, alors $\rho$ est une extension non triviale de
$\chi_1$ par $\chi_2=|\ |_\ell\chi_1$ et $\Pi^{\rm cl}_\ell(\rho)={\rm St}\otimes\chi_1$
o\`u ${\rm St}$ est la steinberg.

\section{La correspondance de Langlands locale $p$-adique pour ${\bf GL}_2(\Q_p)$}\label{llp0}
Dans ce chapitre, on rappelle succinctement un certain nombre de propri\'et\'es de la correspondance
de Langlands locale $p$-adique pour $\GG(\Q_p)$, ses liens avec la th\'eorie d'Iwasawa...,
et on donne quelques compl\'ements:
en particulier une formule explicite (prop.\,\ref{ES20bis}) exprimant exponentielles
duales de classes de cohomologie galoisienne pour une repr\'esentation de de Rham en termes
de la repr\'esentation de $G$ associ\'ee, et un mod\`ele de Kirillov (\no\ref{kirp0}) pour
$V^\dual\otimes\Pi_p(V)$
qui joue un grand r\^ole dans la suite.
\Subsection{$(\varphi,\Gamma)$-modules}\label{llp1}
\subsubsection{Quelques anneaux}\label{llp2}
Soit $\oe^+=\O_L[[T]]$.
On d\'efinit $\oe$ \index{E@\EEE}comme le compl\'et\'e
de $\oe^+[\frac{1}{T}]$ pour la topologie $p$-adique;
c'est l'anneau des
s\'eries de Laurent $\sum_{k\in\Z}a_kT^k$, avec $a_k\in \O_L$ et
$\lim_{k\to-\infty}v_p(a_k)=+\infty$.
On note
${\cal E}=\oe[\frac{1}{p}]$ le corps des fractions de $\oe$.

Si $h\geq 1$, on pose $r_h=v_p(\zeta_{p^h}-1)=\frac{1}{(p-1)p^{h-1}}$
(resp.~$r_h=\frac{1}{2^h}$, si $p=2$), et $n_h=\frac{1}{r_h}$.
On note $\oed{h}$ le compl\'et\'e
de $\oe^+[\frac{p}{T^{n_h}}]$ pour la topologie $p$-adique,
et on pose ${\cal E}^{(0,r_h]}=\oed{h}[\frac{1}{p}]$;
on d\'efinit alors le sous-corps $\edag$ des \'el\'ements
{\it surconvergents} de ${\cal E}$, comme la r\'eunion des
${\cal E}^{(0,r_h]}$.

\smallskip

Si $n\geq 1$, on \index{Ln@\Ln}pose $L_n=L\otimes_{\Q_p} \Q_p(\zeta_{p^n})$,
et on \index{iotan@\iotan}note $\iota_n$ l'injection de
${\cal E}^{(0,r_n]}$ dans $L_n[[t]]$ envoyant
$f$ sur $f(\zeta_{p^n}e^{-t/p^n}-1)$.

\subsubsection{Actions de $\varphi$ et $\Gamma$}\label{llp3}
\index{Gamma1@\Gamm}Soit $\Gamma={\rm Gal}(\Q_p(\bmu_{p^\infty})/\Q_p)$.
Le caract\`ere\footnote{On note simplement $\chi$ le caract\`ere cyclotomique dans ce chapitre.}
 \index{epsil@\cyclo}cyclotomique $\chi:=\cyp $ induit un isomorphisme
$\chi:\Gamma\overset{\sim}{\to}\Z_p^\dual$, et on 
\index{sigma@\sigm}note $a\mapsto\sigma_a$
l'isomorphisme inverse.

On munit les $L$-alg\`ebres
topologiques ${\cal E},\edag$ d'actions continues de $\varphi$ et $\Gamma$
d\'efinies par $\varphi(T)=(1+T)^p-1$ et $\sigma_a(T)=(1+T)^a-1$.
L'injection $\iota_n$ ci-dessus est $\Gamma$-\'equivariante si on fait agir
$\sigma_a$ sur $t$ par $\sigma_a(t)=at$.

\subsubsection{$(\varphi,\Gamma)$-modules \'etales}\label{llp4}
Si $\Lambda=\oe, {\cal E},\edag$, un $(\varphi,\Gamma)$-module $D$
sur $\Lambda$ est un $\Lambda$-module libre, de rang fini, munis
d'actions semi-lin\'eaires de $\varphi$ et $\Gamma$ commutant entre elles.
On dit que {\it $D$ est \'etale} si $\varphi$ est de pente~$0$.
On \index{FG@\FGET}note $\FGet(\Lambda)$ la cat\'egorie des $(\varphi,\Gamma)$-modues
\'etales sur $\Lambda$. (On peut aussi consid\'erer des $\oe$-modules de torsion, et beaucoup
de preuves utilisent des d\'evissages les faisant intervenir, mais nous n'en aurons pas vraiment besoin.)

\subsubsection{L'op\'erateur $\psi$}\label{llp5}
Si $D\in\FGet(\Lambda)$, tout \'el\'ement $x$ de $D$ peut s'\'ecrire,
de mani\`ere unique, sous la forme $x=\sum_{i=0}^{p-1}\varphi(x_i)(1+T)^i$.
Cela permet de d\'efinir un inverse \`a gauche
$\psi:D\to D$ de $\varphi$ par la formule $\psi(x)=x_0$;
de plus $\psi$ commute \`a $\Gamma$.

\subsubsection{Surconvergence}\label{llp6}
Si $D\in \FGet({\cal E})$, on note
$D^\dagger$ le module des \'el\'ements
surconvergents~\cite{CC98,BC08} de $D$: 
si $n\in\N$,
soit
$D^{(0,r_n]}$ le plus grand sous-${\cal E}^{(0,r_n]}$-module $M$ de type fini
de $D$ tel que $\varphi(M)$ soit inclus dans
le sous-${\cal E}^{(0,r_{n+1}]}$-module
de $D$ engendr\'e par $M$.
Il
existe un entier $m(D)\geq 1$ tel
que~$D^{(0,r_{m(D)}]}$ soit de rang~$d$ sur ${\cal E}^{(0,r_{m(D)}]}$ et
alors
$D^{(0,r_n]}={\cal E}^{(0,r_n]}\otimes_{{\cal E}^{(0,r_{m(D)}]}}
D^{(0,r_{m(D)}]}$,
 pour tout
$n\geq m(D)$, et $D^\dagger$ est la limite inductive des $D^{(0,r_n]}$.
On a $D^\dagger\in \FGet(\edag)$ et $D={\cal E}\otimes_{\edag}D^\dagger$.

\subsubsection{Les modules $D_{{\rm dif},n}$, $D_{{\rm dif},n}^+$ et $D_{{\rm dif},n}^-$}\label{llp7}
Si $n\geq 1$, on \index{D@\DDD}pose
$$ D_{{\rm dif},n}^+=L_n[[t]]\otimes  D^{(0,r_n]}$$
o\`u l'on voit $L_n[[t]]$, pour $n\geq m( D)$, comme une
${\cal E}^{(0,r_n]}$-alg\`ebre via
le morphisme
$\iota_n:{\cal E}^{(0,r_n]}\to L_n[[t]]$.
Soient
$$L_\infty[[t]]:=\varinjlim_n L_n[[t]]\quad{\rm et}\quad
D^+_{\rm dif}:=\varinjlim_n D^+_{{\rm dif},n}.$$
Alors $ D^+_{\rm dif}$ est un $L_\infty[[t]]$-module libre de rang~$d$
muni d'une action de $\Gamma$.
On pose $D_{\rm dif}= D^+_{\rm dif}[\frac{1}{t}]$; c'est un $L_\infty((t)):=L_\infty[[t]][t^{-1}]$-module
muni d'une filtration d\'ecroissante par les $D^i_{\rm dif}:=t^iD_{\rm dif}^+$, stable par $\Gamma$.
Le quotient $ D^-_{\rm dif}= D_{\rm dif}/ D^+_{\rm dif}$
est un $L_\infty[t]$-module de torsion,
muni d'une action localement analytique de $\Gamma$ (limite inductive
de repr\'esentations de dimension finie).

\subsubsection{$(\varphi,\Gamma)$-module de\,Rham}\label{llp8}
On dit que $ D$ est {\it de\,Rham} 
si le $L$-espace $D_{\rm dR}=(D_{\rm dif})^{\Gamma}$
est de \index{D@\DDD}dimension~$d$.  Si c'est le cas, alors
$D_{\rm dif}=L_\infty((t))\otimes_LD_{\rm dR}$ et, si on munit
$D_{\rm dR}$ de la filtration $(D_{\rm dR}^i,\ i\in\Z)$ induite par 
la filtration naturelle de $D_{\rm dif}$, on a
$D_{\rm dif}^+=\sum_{i\in \Z}t^{-i}L_\infty[[t]]\otimes D_{\rm dR}^i$.
Si $D$ est de\,Rham,
{\it les poids de $D$} sont les oppos\'es des sauts de la filtration,
i.e.~les entiers $i$ tels que
$D^{-i}_{\rm dR}\neq D^{1-i}_{\rm dR}$.

\Subsection{$(\varphi,\Gamma)$-modules et repr\'esentations de $\GG(\Q_p)$}\label{llp9}
Posons:
$$G=\GG(\Q_p),\quad P=\PP(\Q_p),\quad U=\UU(\Q_p),$$
\index{U@\UUU}et \index{G@\GGG}notons $P^+$ \index{P@\PPP}le 
sous-semi-groupe $P^+=\matrice{\Z_p-\{0\}}{\Z_p}{0}{1}$ de $P$. 
\subsubsection{La correspondance $D\mapsto\Pi_p(D)$}\label{llp10}
On renvoie \`a~\cite{gl2,CDP} ou~\cite{CD} pour ce qui suit.
Soit 
$D\in\FGet({\cal E})$ ou $\FGet(\oe)$, de rang $2$.
Notons
$\omega_D$ le caract\`ere
$$\omega_D=(x|x|)^{-1}\det D.$$

$\bullet$ {\it Le faisceau $U\mapsto D\boxtimes U$}.---
$D$ donne naissance
\`a un faisceau $P^+$-\'equivariant $U\mapsto D\boxtimes U$
sur $\Z_p$, avec $D\boxtimes\Z_p=D$, et  
$\matrice{a}{b}{0}{1}\in P^+$ agit sur $x\in\Z_p$ par $\matrice{a}{b}{0}{1}\cdot x=ax+b$,
et sur $D$ par $\matrice{p^ka}{b}{0}{1}\cdot x=(1+T)^b\varphi^k\circ\sigma_a(x)$.
Ce faisceau se prolonge, de mani\`ere unique,
en un faisceau $G$-\'equivariant $U\mapsto D\boxtimes U$
(o\`u $U$ d\'ecrit les ouverts compacts) sur $\piqp=\piqp(\Q_p)$,
tel que
le centre de $G$ agisse par $\omega_D$, et il existe
 une repr\'esentation
\index{Pi@\Pip}unitaire $\Pi_p(D)$ de $G$ telle que l'on ait une suite exacte
\begin{align*}
0\to\Pi_p(D)^\dual\otimes\omega_D\to &D\boxtimes\piqp\to\Pi_p(D)\to 0
\end{align*}
\begin{rema}\phantomsection\label{pathop}
(i) Si $D$ n'est pas une extension de ${\cal E}((x|x|)^{-1}\delta)$ par ${\cal E}(\delta)$,
la repr\'esentation $\Pi_p(D)$ ci-dessus n'a ni sous-objet ni quotient de dimension finie.

(ii) Dans le cas pathologique o\`u
 $D$ est une telle extension, la repr\'esentation ci-dessus
admet $\delta$ comme sous-objet et n'est pas celle que l'on veut.
Le quotient
$\Pi^{\rm min}_p(D)$ par $\delta$ n'a ni sous-objet ni quotient de dimension finie, et
${\rm Ext}^1(\delta,\Pi^{\rm min}_p(D))$ est de dimension~$2$; on d\'efinit
$\Pi_p(D)$ comme l'extension universelle: on a une suite exacte
$$0\to \Pi^{\rm min}_p(D)\to \Pi_p(D)\to \delta\oplus\delta\to 0$$
\end{rema}

$\bullet$ {\it Dualit\'e}.---
On fait agir $\Gamma$ et $\varphi$ sur $\frac{dT}{1+T}$
par $\sigma_a\big(\frac{dT}{1+T}\big)=a\frac{dT}{1+T}$ et
$\varphi\big(\frac{dT}{1+T}\big)=\frac{dT}{1+T}$.  Cela nous fournit
le $(\varphi,\Gamma)$-module
$\Lambda\frac{dT}{1+T}\in\FGet(\Lambda)$, si $\Lambda=\oe,{\cal E}$.
On pose \index{D@\DDD}alors $\check D={\rm Hom}_\Lambda(D,\Lambda\frac{dT}{1+T})$; c'est
un objet de $\FGet(\Lambda)$ et $V(\check D)$ est le dual de Tate $\check V(D)$ de $V(D)$
(i.e. ${\rm Hom}_{\O_L}(V(D),\O_L(1))$ (resp.~${\rm Hom}_{L}(V(D),L(1))$), 
si $D\in \FGet(\oe)$ (resp.~$D\in \FGet({\cal E})$).

On note $\{\ ,\ \}:\check D\times D\to L$ l'accouplement d\'efini
par $$\{\check x,y\}=\reso(\langle\sigma_{-1}(\check x),y\rangle),$$
o\`u $\langle\ ,\ \rangle:\check D\otimes D\to {\cal E}\frac{dT}{1+T}$
est l'accouplement naturel, et $\reso\,fdT=a_{-1}$, si 
$f=\sum_{k\in\Z}a_kT^k$.
Cet accouplement se prolonge (de mani\`ere unique) en un 
accouplement parfait, $G$-\'equivariant,
$$
\{\ ,\ \}:(\check D\boxtimes\piqp)\otimes( D\boxtimes\piqp)\to L$$
tel que $\check D\boxtimes U$
et $D\boxtimes V$ sont orthogonaux si $U\cap V=\emptyset$.
On peut r\'e\'ecrire la suite exacte ci-dessus sous la forme
\begin{align*}
0\to\Pi_p(\check D)^\dual\to &D\boxtimes\piqp\to\Pi_p(D)\to 0\\
\end{align*}

\subsubsection{Th\'eorie d'Iwasawa}\label{wasi1}
\index{Lam@\lamb}Soit $\Lambda=\Z_p[[\Gamma]]$; on note $[\sigma_a]\in \Lambda$ l'\'el\'ement correspondant
\index{sigma@\sigm}\`a $\sigma_a$ par l'injection naturelle $\Gamma\hookrightarrow \Lambda^\dual$.

Soient $\Lambda=\oe,{\cal E}$ et $D\in \FGet(\Lambda)$.
Si $V:=V(D)$ est la repr\'esentation de $ G_{\Q_p}$ associ\'ee \`a $D$ par l'\'equivalence
de cat\'egories de Fontaine, on dispose~\cite[th.\,II.1.3]{CC99}
d'un isomorphisme \index{expo2@\EXP}de $\Lambda$-modules
$${\rm Exp}^\dual:H^1( G_{\Q_p},\Lambda\otimes V)\overset{\sim}{\to}
D^{\psi=1}.$$
Par ailleurs, l'application ${\rm Res}_{\Z_p}$ induit une fl\`eche
$${\rm Res}_{\Z_p}:(\Pi_p(\check D)^\dual)^{\matrice{p}{0}{0}{1}=1}{\to} D^{\psi=1},
\quad {\text {avec ${\rm Res}_{\Z_p}\big(\matrice{a}{0}{0}{1}\star\mu\big)=\sigma_a({\rm Res}_{\Z_p}(\mu))$,}}
$$
qui est un isomorphisme en g\'en\'eral\footnote{C'est le cas si $D$ est irr\'eductible
ou, plus g\'en\'eralement, si $D^{\varphi=1}=\check D^{\varphi=1}=0$, cf.~\cite[rem.\,V.14]{CD}}.
On en d\'eduit le r\'esultat suivant:
\begin{prop}\phantomsection\label{wasi2}
Si $\mu\in \Pi_p(\check D)^\dual$ est fixe par $\matrice{p}{0}{0}{1}$,
il existe un unique $\lambda(\mu)\in H^1( G_{\Q_p},\Lambda\otimes V)$ tel que
${\rm Exp}^\dual(\lambda(\mu))={\rm Res}_{\Z_p}\mu$.
De plus, si $a\in\Z_p^\dual$,
$$\lambda\big(\matrice{a}{0}{0}{1}\star\mu\big)=[\sigma_a]\cdot\lambda(\mu).$$
\end{prop}

Enfin, {\it si $D$ est de Rham},
on a la loi de r\'eciprocit\'e explicite suivante~\cite[th.\,IV.2.1]{CC99}
 o\`u $\exp^\dual$
est l'application \index{expo2@\EXP}exponentielle duale
de Bloch-Kato: dans la proposition,
$\int_{1+p^n\Z_p}\hskip-.2cm x^{j}\lambda\in H^1(G_{\Q_p(\bmu_{p^n})},V\otimes\chi^{j})$
et $\exp^\dual\big(\int_{1+p^n\Z_p}\hskip-.2cm x^{j}\lambda\big)\in t^{-j}L_n\otimes_L{\rm Fil}^jD_{\rm dR}$,
\index{t@\ttt}\index{epsi@\epsi}o\`u 
$$t=-\log[\epsilon]\in\bdr^+\ {\text{est le $2i\pi$ $p$-adique de Fontaine,
avec $\epsilon=(1,\zeta_p,\dots,\zeta_{p^n},\dots)$.}}$$
(Si $V$ est une repr\'esentation de de Rham, $\exp^\dual$ 
envoie $H^1(G_K,V)$ dans ${\rm Fil}^0D_{\rm dR}(V)$.)

\begin{prop}\phantomsection\label{ES15}
Si $\lambda\in H^1( G_{\Q_p},\Lambda\otimes V)$, et si $n$ est assez grand,
$$p^{-n}\iota_n({\rm Exp}^\dual(\lambda))=\sum_{j\in\Z}
\exp^\dual\big(\int_{1+p^n\Z_p}x^{j}\lambda\big)\in 
{\rm Fil}^0(L_n[[t]]\otimes_L D_{\rm dR}).$$
\end{prop}

\begin{rema}\phantomsection\label{ES15.1}
Soit $\phi(x)=\sum_{j\in J}\phi_j(x)x^j$, o\`u $J\subset \Z$ est fini, et $\phi_j\in{\rm LC}(\Z_p^\dual,
L_\infty((t)))$.  Il existe donc $n\geq 1$ tel que $\phi_j$ soit constante modulo $p^n$, pour tout~$j$.
On d\'efinit alors $\exp^*(\int_{\Z_p^\dual}\phi(x)\lambda)$ par la formule
$$\exp^*(\int_{\Z_p^\dual}\phi(x)\lambda)=\sum_{j\in J}\sum_{a\,{\rm mod}\,p^n}
\phi_j(a)\exp^*(\int_{a+p^n\Z_p}x^j\lambda)\in L_\infty((t))\otimes_LD_{\rm dR}.$$
(On a $\int_{a+p^n\Z_p}x^j\lambda\in H^1(G_{\Q_p(\bmu_{p^n})},V\otimes\chi^j)$.)

Maintenant, si $\phi(ax)=\sigma_a(\phi(x))$ pour tous $a,x\in\Z_p^\dual$, alors
$\phi(x)=\sum_{j\in J}\sigma_x(\alpha_j)(tx)^j$, avec $\alpha_j\in L_n$.
Comme $t^j\exp^*\int_{a+p^n\Z_p}x^j\lambda=\sigma_a(t^j\exp^*\int_{1+p^n\Z_p}x^j\lambda)$,
on a
$$\exp^*(\int_{\Z_p^\dual}\phi(x)\lambda)=
\sum_{a\,{\rm mod}\,p^n}\sigma_a\Big(\sum_{j\in J}\alpha_jt^j\exp^*\int_{1+p^n\Z_p}x^j\lambda\Big)
\in D_{\rm dR}\subset L_\infty((t))\otimes_L D_{\rm dR}.$$
\end{rema}

\subsubsection{R\'esultats de densit\'e}
Si $c,d\in\Z_p^\dual$, soit $B_p^{c,d}=\big(c^2-\matrice{c}{0}{0}{1}\big)\big(d^2-\matrice{1}{0}{0}{d}\big)$,
et soit $(B_p^{c,d})^\dual$ l'op\'erateur adjoint (obtenu en rempla\c{c}ant $c$ et $d$ par leurs inverses
dans les matrices).
\begin{lemm}\phantomsection\label{facto3}
Si $c$ est un g\'en\'erateur de $\Z_p^\dual$, si $d=c^{-1}$ et si $V=V(D)$
n'a pas de quotient non nul sur lequel l'inertie agit par $\chi^{-1}$ 
ou par $\chi^2\det V$, 
le sous-espace engendr\'e par les $(B_p^{c,d})^\dual\star v$, pour $v\in \Pi_p(D)$,
 est dense dans $\Pi_p(D)$.
\end{lemm}
\begin{proof}
Par dualit\'e, cela revient \`a prouver que, si $\mu\in\Pi_p(D)^\dual$ est tu\'e par $B_p^{c,d}$,
alors $\mu=0$.

On a $\Pi_p(D)^\dual\subset \check D\boxtimes\piqp$, et 
on dispose de ${\rm Res}_{\Z_p}:\check D\boxtimes\piqp\to \check D$ qui commute \`a l'action
de $\matrice{\Z_p^\dual}{0}{0}{1}$.
Soit $\mu\in \Pi_p^\dual(D)$, et soient $z_1={\rm Res}_{\Z_p}\mu$ et $z_2={\rm Res}_{\Z_p}w\star \mu$.
Alors $B_p^{c,d}\star\mu=0$ si et seulement si $B_p^{c,d}\star z_1=0$ et $B_p^{c,d}\star z_2=0$.
Or :
\begin{align*}
B_p^{c,d}\star z_1&=(c^2-\sigma_c)(c^{-2}-\delta(c^{-1})\sigma_c)\cdot z_1\\
B_p^{c,d}\star z_2&=(c^2-\delta(c)\sigma_{c^{-1}})(c^{-2}-\sigma_{c^{-1}})\cdot z_2
\end{align*}
o\`u $\delta=(x|x|)(\det V)^{-1}$.
Si $B_p^{c,d}\star\mu=0$, les $\Gamma$-modules engendr\'es par $z_1$ et $z_2$
sont de rang~$\leq 2$ sur $\O_L$.  On en d\'eduit pour commencer,
 en utilisant \cite[prop.\,III.4.8]{mira} (combin\'ee
avec l'isomorphisme $\check D^{\rm nr}={\bf D}^{\rm nr}(\check V)$ de la preuve de 
\cite[prop.\,II.2.2]{mira}
et le (ii) de \cite[rem.\,II.1.2]{mira}), que $z_1$ et $z_2$ appartiennent au sous-module
$(W(\overline{\bf F}_p)\otimes\check V^{H'})^H$ de $\check D$, 
o\`u $H'={\rm Gal}(\Qbar_p/\Q_p^{\rm ab})$.  Puis que $\check V$
a une droite stable par $H'$ et tu\'ee par $\sigma_c-c^2$ ou par $\sigma_c-c^{-2}\delta(c)$,
et enfin que $V$ a une droite quotient stable par $H'$ et tu\'ee par $\sigma_c-c^{-1}$ ou par
$\sigma_c-c^3\delta^{-1}(c)$.

   Ceci implique le r\'esultat annonc\'e.
\end{proof}

\Subsection{Vecteurs localement alg\'ebriques}\label{valg1}
Si $D\in\FGet({\cal E})$ est de rang~$2$, on \index{Pi@\Pip}note $\Pi_p(D)^{\rm alg}$
le sous-espace de $\Pi_p(D)$ des vecteurs localement alg\'ebriques.
Alors $\Pi_p(D)^{\rm alg}\neq 0$ si et seulement si
$D$ est de\,Rham, \`a poids distincts~\cite[th.\,0.20]{gl2}, \cite{Dosp1} ou~\cite[th.\,2.5]{poids}.
\begin{rema} \label{valg2}
Soit $D\in\FGet({\cal E})$, de Rham, de rang~$2$, \`a poids de Hodge-Tate distincts.

{\rm (i)}
Si $D$ est irr\'eductible, alors $\Pi_p(D)^{\rm alg}$ est dense dans
$\Pi_p(D)$.

{\rm (ii)} Si $D$ n'est pas irr\'eductible, l'adh\'erence $\Pi_1$
de $\Pi_p(D)^{\rm alg}$ est une sous-repr\'esentation stricte de
$\Pi_p(D)$: c'est soit une s\'erie principale, auquel cas $\Pi_p(D)/\Pi_1$
est aussi une s\'erie principale, soit une tordue de la steinberg continue
auquel cas $\Pi_p(D)/\Pi_1$ est une extension non scind\'ee d'une s\'erie principale
par un caract\`ere.  

(iii) Si $D$ n'est pas irr\'eductible mais
n'est pas la somme directe de deux caract\`eres, 
alors ${\rm End}(D)={\rm End}(\Pi_p(D))=L$ car la condition sur les poids
implique que $D$ n'est pas l'extension d'un objet de rang~$1$ par lui-m\^eme; 
en particulier, un \'el\'ement
de ${\rm End}(\Pi_p(D))$ est compl\`etement d\'etermin\'e par sa restriction
\`a $\Pi_p(D)^{\rm alg}$ m\^eme si ce dernier espace n'est pas dense.
\end{rema}

\subsubsection{Mod\`ele de Kirillov}\label{valg3}
Si $D$ est de\,Rham, \`a poids $k_1<k_2$, le sous-espace $\Pi_p(D)^{\rm alg}$
des vecteurs localement alg\'ebriques de $\Pi_p(D)$ est non nul
et de la forme $\Pi^{\rm lisse}\otimes {\rm Sym}^{k-1}\otimes \det^{k_1}$,
o\`u $k=k_2-k_1$ et ${\rm Sym}^k$ est la puissance sym\'etrique
de la repr\'esentation (de dimension~$2$) naturelle de $G$.

On note $ D_{\rm dR}$ le $L$-espace
$ D_{{\rm dif}}^\Gamma$.  Il est muni de la filtration d\'ecroissante
induite par celle de $ D_{{\rm dif}}$,
et on a $ D_{\rm dR}^{-k_2}= D_{\rm dR}$,
$ D_{\rm dR}^{-k_1+1}=0$, tandis que
$ D_{\rm dR}^{-k_2+1}= D_{\rm dR}^{-k_1}$ est une droite.

On a $ D_{{\rm dif}}^+=t^{k_2}L_\infty[[t]]\otimes D_{\rm dR}^{-k_2}
+t^{k_1}L_\infty[[t]]\otimes D_{\rm dR}^{-k_1}$.
On pose
\begin{align*}
X_\infty^-=(t^{k_1}L_\infty[[t]]\otimes D_{\rm dR}^{-k_2})/ D_{{\rm dif}}^+=
(t^{k_1}L_\infty[t]/t^{k_2}L_\infty[t])\otimes  (D_{\rm dR}^{-k_2}/ D_{\rm dR}^{-k_1})
\\
X_\infty^+= D_{{\rm dif}}^+/(t^{k_2}L_\infty[[t]]\otimes D_{\rm dR}^{-k_2})=
(t^{k_1}L_\infty[t]/t^{k_2}L_\infty[t])\otimes  D_{\rm dR}^{-k_1}
\end{align*}

On dispose (cf.~\cite[prop.\,VI.5.6]{gl2} ou~\cite[prop.\,2.10]{poids}) 
d'un mod\`ele de \index{K@\KKK}Kirillov pour $\Pi_p(D)^{\rm alg}$, i.e.~d'une injection
$v\mapsto {\cal K}_v$, naturelle, $P$-\'equivariante, de $\Pi_p(D)^{\rm alg}$
dans 
l'espace
${\rm LA}_{\rm rc}(\Q_p^\dual,X_\infty^-)^\Gamma$ des $\phi:\Q_p^\dual\to X_\infty^-$

\quad --- \`a support
compact dans $\Q_p$ (c'est la signification du {\og rc\fg}),

\quad --- v\'erifiant
$\sigma_a(\phi(x))=\phi(ax)$, pour tout $a\in\Z_p^\dual$ (invariance par $\Gamma$),

muni de l'action suivante de $P$:
$$\big(\matrice{a}{b}{0}{1}\cdot\phi\big)(x)=[\epsilon^{bx}]\phi({ax})
\quad({\text{Si $x\in\Q_p$, alors $[\epsilon^x]={\bf e}_p(x)e^{-tx}$.}})$$

De plus, l'image de $\Pi_p(D)^{\rm alg}$ par $v\mapsto {\cal K}_v$ 
contient ${\rm LA}_c(\Q_p^\dual,X_\infty^-)^\Gamma$ et cet espace est de codimension
finie dans l'image: le quotient est $J(\Pi^{\rm lisse})\otimes{\rm Sym}^{k-1}\otimes\det^{k_1}$,
o\`u $J(\Pi^{\rm lisse})$ est le module de Jacquet de $\Pi^{\rm lisse}$ et est de
dimension~$0$ si  
$\Pi^{\rm lisse}$ est supercuspidale, de dimension~$1$ si $\Pi^{\rm lisse}$ est un twist de la steinberg,
et de dimension~$2$ si $\Pi^{\rm lisse}$ est une 
s\'erie principale (i.e.~l'induite d'un caract\`ere du borel) irr\'eductible.

\subsubsection{Mod\`ele de Kirillov et dualit\'e}\label{valg4}
Le $(\varphi,\Gamma)$-module $\check D$ est aussi de\,Rham, \`a poids de Hodge-Tate $1-k_2, 1-k_1$.
On \index{accouplements!{$\{\ ,\ \}$, $\{\ ,\ \}_{\rm dif}$}}note $\langle\ ,\ \rangle:\check D_{\rm dif}\times D_{\rm dif}\to
L_\infty((t))dt$ l'accouplement naturel (on a $dt=\frac{dT}{1+T}$ car $t=\log(1+T)$).
Alors 
$$\langle \check x, y\rangle_{\rm dR}=
{\text{r\'es}}_{t=0}\langle \check x, y\rangle\,,$$
o\`u $\resl(\sum a_nt^ndt)=a_{-1}$, 
induit un accouplement parfait sur $\check D_{\rm dR}\times D_{\rm dR}$.

On note $\{\ ,\ \}_{\rm dif}:\check D_{\rm dif}\times D_{\rm dif}\to L$ 
l'accouplement bilin\'eaire
d\'efini par
$$\{\check x,y\}_{\rm dif}=
\lim_{n\to+\infty}p^{-n}{\rm Tr}_{L_n/L}\big(\resl(\langle\sigma_{-1}(\check x),y\rangle)\big).$$
On a $\{\ ,\ \}_{\rm dif}=\frac{p-1}{p}\langle \ ,\  \rangle_{\rm dR}$ 
sur $\check D_{\rm dR}\times D_{\rm dR}$.

L'accouplement $\{\ ,\ \}_{\rm dif}$ est identiquement nul sur 
$\check D_{\rm dif}^+\times D_{\rm dif}^+$
et induit un accouplement
$\{\ ,\ \}_{\rm dif}:\check X_\infty^+\times X_\infty^-\to L$.

Le r\'esultat suivant~\cite[prop.\,VI.5.12]{gl2} ou~\cite[prop.\,2.13]{poids}  sera crucial.
\begin{prop}\phantomsection\label{ES14}
Si $v\in\Pi_p(D)^{\rm alg}$ est tel que ${\cal K}_v$ est \`a support compact dans $\Q_p^\dual$,
et si
$\mu\in\Pi_p( D)^\dual$,
alors
\begin{equation}\label{acc}
\{\mu,v\}=\sum_{i\in\Z}\big\{\iota_N\big({\rm Res}_{\Z_p}
\big(\matrice{p^{i+N}}{0}{0}{1}\mu\big)\big),
{\cal K}_v(p^i)\big\}_{\rm dif}\,,
\end{equation}
pour tout $N$ assez grand {\rm (d\'ependant de $v$; l'expression
a un sens
pour $N\geq m( D)$ car
${\rm Res}_{\Z_p}\Pi_p(D)^\dual\subset \check D^{(0,m( D)]}$)}. 
\end{prop}

\subsubsection{Une formule explicite}\label{valg5}
Comme $\Q_p^\dual=\sqcup_{n\in\Z}p^{-n}\Z_p^\dual$, on a
$${\rm LP}_c(\Q_p^\dual,X_\infty^-)^{\Gamma}=
\oplus_{n\in\Z}\matrice{p^n}{0}{0}{1}\cdot
{\rm LP}(\Z_p^\dual,X_\infty^-)^{\Gamma}.$$

Soit $\check e^+$ une base de $\check D_{\rm dR}^{k_2-1}$.  Alors 
$\check e^+$
induit un isomorphisme 
$X_\infty^-(D)\overset{\sim}{\to}t^{k_1}L_\infty[t]/t^{k_2}L_\infty[t]$, avec
$\check e^+(x)\frac{dt}{t}=\langle \check e^+,x\rangle$,
et une injection
$$\check e^+:\Pi_p(D)^{\rm alg}\to {\rm LP}(\Q_p^\dual,\tfrac{t^{k_1}L_\infty[t]}{t^{k_2}L_\infty[t]})^{\Gamma}.$$
Si $e^-$ est la base de $D_{\rm dR}^{-k_2}/D_{\rm dR}^{-k_1}$ duale
de $\check e^+$,
on a 
\begin{equation}\label{emoins}
{\cal K}_v=\check e^+(v)\,e^-.
\end{equation}
\begin{prop}\phantomsection\label{ES20bis}
 Soient:

\quad $\bullet$ $\mu\in \Pi_p(D)^\dual$ invariant
par $\matrice{p}{0}{0}{1}$, 

\quad $\bullet$
$\lambda\in H^1( G_{\Q_p},\Lambda\otimes \check V(D))$
v\'erifiant ${\rm Res}_{\Z_p}(\mu)=
\sigma_{-1}\cdot {\rm Exp}^\dual(\lambda)$.
 
Si $v\in\Pi_p(D)^{\rm alg}$, alors
$$
\exp^*\big(\int_{\Z_p^\dual}\check e^+(v)\,\lambda\big)=
\big\{\mu,v\big\}\check e^+
$$
\end{prop}
\begin{proof}
La formule \`a d\'emontrer est \'equivalente \`a
$$\big\{\mu,v\big\}=
\big\langle\exp^*\big(\int_{\Z_p^\dual}\check e^+(v)\,\lambda\big),
e^-\big\rangle_{\rm dR},$$
et on peut utiliser la formule (\ref{acc}) pour \'evaluer le membre de gauche.
Notons $\phi$ la fonction $\check e^+(v)$; son invariance par $\Gamma$ implique
qu'il existe $\alpha_{k_1},\dots,\alpha_{k_2-1}\in L_\infty$ tels que l'on ait
$$\phi(x)=
\sigma_x(\alpha_{k_1})(tx)^{k_1}+\cdots+\sigma_x(\alpha_{k_2-1})(tx)^{k_2-1}.$$
L'invariance de $\mu$ par $\matrice{p}{0}{0}{1}$ et la prop.\,\ref{ES14}
fournissent, pour $N$ assez grand, la formule
\begin{align*}
\big\{\mu,v(\phi)\big\}=&\ 
\big\{\iota_N(\sigma_{-1}\cdot {\rm Exp}^\dual(\lambda)),\phi(1) e^-\}_{\rm dif}\\
=&\ p^{-N}{\rm Tr}_{L_N/L}\big({\text{r\'es}}_{t=0}
(\langle \iota_N({\rm Exp}^\dual(\lambda)),\phi(1)e^-\rangle)\big),
\end{align*}
la seconde \'egalit\'e venant de la d\'efinition
de $\{\ ,\ \}_{\rm dif}$. 
D'apr\`es la prop.~\ref{ES15},
$$p^{-N}\iota_N({\rm Exp}^\dual\lambda)=\sum_{i\in\Z}
\exp^\dual\int_{1+p^N\Z_p}x^i\lambda, \quad{\text{avec 
$\exp^\dual\int_{1+p^N\Z_p}x^i\lambda\in t^{-i}L_N\otimes \check D_{\rm dR}$.}}$$
Donc
$${\text{r\'es}}_{t=0}
(\langle p^{-N}\iota_N({\rm Exp}^\dual(\lambda)),\phi(1)e^-\rangle)=
\sum_{i=k_1}^{k_2-1}\alpha_i\langle t^i \exp^\dual\int_{1+p^N\Z_p}x^i\lambda,e^-\rangle_{\rm dR}.$$
Il reste \`a prendre la trace, i.e. \`a appliquer l'op\'erateur
$\sum_{a\in(\Z/p^N)^\dual}\sigma_a$, et on conclut en utilisant
la formule
$$\sigma_a\big(t^i\exp^*\int_{1+p^N\Z_p}\hskip-.5cm
x^{i}\lambda\big)=
t^i\exp^*\int_{a+p^N\Z_p}\hskip-.5cm
x^{i}\lambda.\qedhere$$
\end{proof}

\begin{rema}\phantomsection\label{ES21}
La formule ci-dessus pour l'action de $\Gamma$ sur $\exp^*$ se justifie comme suit.
Soit $V_n=\Z_p[G_n]\otimes V$, o\`u $G_n= G_{\Q_p}/ G_{\Q_p(\zeta_{p^n})}\cong(\Z/p^n)^\dual$.
Le lemme de Shapiro
fournit un isomorphisme $H^1( G_{\Q_p(\zeta_{p^n})},V)\cong H^1( G_\Q,V_n)$:
si $g\mapsto c_g$ est un $1$-cocycle sur $ G_{\Q_p}$ \`a valeurs dans $V_n$,
on peut \'ecrire $c_g$ sous la forme $\sum_{b\in (\Z/p^n)^\dual}[\sigma_b]\otimes v_{b,g}$, 
et $g\mapsto v_{b,g}$
est un $1$-cocycle sur $ G_{\Q_p(\zeta_{p^n})}$ \`a valeurs dans $V$, et l'application
du lemme de Shapiro envoie $g\mapsto c_g$ sur $g\mapsto v_{1,g}$.
Le cocycle $g\mapsto v_{b,g}$ correspond \`a $\int_{b+p^n\Z_p}c$.

Par d\'efinition de $\exp^\dual$, il existe $C\in \bdr^+\otimes V_n$
et $v=\exp^\dual(c)\in (\bdr^+\otimes V_n)^{ G_{\Q_p}}$ tels que
$c_g=(g-1)\cdot C+\log\chi(g)v$.  Si on \'ecrit $v$ sous la forme
$\sum_{b\in (\Z/p^n)^\dual}[\sigma_b]\otimes v_b$, avec $v_b\in (\bdr^+\otimes V)^{ G_{\Q_p(\zeta_{p^n})}}$,
alors
$v_b=\exp^*(\int_{b+p^n\Z_p}c)$, et l'invariance de $v$ par $ G_{\Q_p}$
se traduit par $\sum_{b\in (\Z/p^n)^\dual}[\sigma_{\chi(g)b}]\otimes g(v_b)=
\sum_{b\in (\Z/p^n)^\dual}[\sigma_b]\otimes v_b$.  Si $\chi(g)=a$, identifiant
les coefficients de $[\sigma_a]$, on obtient $\sigma_a(v_1)=v_a$, ce que l'on voulait.
\end{rema}

\Subsection{Le $\BB(\Q_p)$-module $\Pi_p(D)$}\label{qq7}
On \index{AA@\AAA}rappelle que 
$\tA=W(C^\flat)$, $\tA^+=W(\O_{C^\flat})$ et $\tA^{++}=W({\goth m}_{C^\flat})$.
On \index{b@\bbb}pose 
$$\tA^-:=\tA/\tA^+,\quad \tB^-:=\Q_p\otimes_{\Z_p}\tA^-$$
On note $\tilde p=[p^\flat]\in\tA^{++}$, avec $p^\flat=(p,p^{1/p},\dots)\in C^\flat$.

\subsubsection{L'injection de $\Pi_p(V)$ dans $\tA^-\otimes V$}
Soit $D\in\FGet(\Lambda)$, de rang~$2$, avec $\Lambda=\oe,{\cal E}$, et soit $V=V(D)$.
\index{D@\DDD}Soient $\widetilde D=(\tA\otimes V)^H$ et $\widetilde D^+=(\tA^+\otimes V)^H$.
D'apr\`es~\cite[cor.\,II.2.9]{gl2} et~\cite[rem.\,III.27]{CD}, on a une suite exacte 
de $P$-modules\footnote{L'action de $P$ sur $\widetilde D, \widetilde D^+$
est $\matrice{p^ka}{b}{0}{1}\cdot v=[\epsilon^b]\,\varphi^k(\sigma_a(v))$, si $k\in\Z$,
$a\in\Z_p^\dual$ et $b\in\Q_p$.}
$$0\to \widetilde D/\widetilde D^+\to \Pi_p(D)\to M\to 0,$$
o\`u $M$ est un sous-quotient de $D^\sharp/D^\natural$
et $U$ agit trivialement sur $D^\sharp/D^\natural$ qui est de type fini sur $\Z_p$ ou $\Q_p$.
On peut compl\'eter cet \'enonc\'e de la mani\`ere suivante.
\begin{prop}\phantomsection\label{cano1}
On dispose d'une injection $P$-\'equivariante\footnote{M\^eme action que ci-dessus sur
le membre de droite.}
$$\Pi_p(D)\hookrightarrow (\tA^-\otimes V)^H$$
prolongeant l'inclusion naturelle $\widetilde D/\widetilde D^+\hookrightarrow (\tA^-\otimes V)^H$.
\end{prop}
\begin{proof}
Soit $z\in M$.
Choisissons $\tilde z\in\Pi_p(D)$ relevant $z$.  Comme
$U$ agit trivialement sur $M$, on en d\'eduit un $1$-cocycle
$u\mapsto z_u=(u-1)\cdot\tilde z$ sur $U$, 
\`a valeurs dans $\widetilde D/\widetilde D^+\subset \tA^-\otimes V$.
D'apr\`es le lemme~\ref{cano2} ci-dessous, il existe $c\in \tA^-\otimes V$, unique, tel que
$z_u=(u-1)\cdot c$.  Comme $\sigma(z_u)=z_u$, pour tout $\sigma\in H$, l'unicit\'e
de $c$ implique $c\in (\tA^-\otimes V)^H$, ce qui permet de conclure.
\end{proof}

\begin{lemm}\phantomsection\label{cano2} On a
$$H^0(U,\tA^-)=0\quad{\rm et}\quad H^1(U,\tA^-)=0.$$
\end{lemm}
\begin{proof} On utilise la suite exacte $0\to\tA^+\to \tA\to \tA^-\to 0$ et la suite
exacte longue de cohomologie associ\'ee.

$\bullet$
Comme $U$ est la limite inductive des $\UU(p^{-n}\Z_p)$ qui sont procycliques,
on a $H^2(U,M)=0$ pour tout module $M$ muni d'une action continue de $U$.

$\bullet$
L'action de $u(b)=\matrice{1}{b}{0}{1}$ sur $\tA$ est la multiplication par $[\epsilon^b]$.
On en d\'eduit que $H^0(U,\tA)=0$.  

$\bullet$
Si $u\mapsto z_u$ est un $1$-cocycle
\`a valeurs dans $\tA$, on a $([\epsilon^a]-1)z_{u(b)}=([\epsilon^b]-1)z_{u(a)}$ pour tous $a,b\in\Q_p$.
Il s'ensuit que $z=\frac{1}{[\epsilon^a]-1}z_{u(a)}$ ne d\'epend pas du choix de $a$, et $u\mapsto z_u$
est le bord de $z$. On en d\'eduit que $H^1(U,\tA)=0$.  

$\bullet$
Si $z_u\in \tA^+$ pour tout $u$,
cela implique que $([\epsilon^a]-1)z\in\tA^+$ pour tout $a\in\Q_p$.  En reduisant modulo~$p$,
et en prenant $a=p^{-n}$, on en d\'eduit que $v_E((\epsilon-1)^{p^{-n}}\overline z)\geq 0$,
pour tout $n$; en faisant tendre $n$ vers $+\infty$, on en tire $\overline z\in\tE^+$.
On peut appliquer ce qui pr\'ec\`ede \`a $\frac{1}{p}(z-[\overline z])$ et une r\'ecurrence imm\'ediate
suivie d'un passage \`a la limite permet d'en d\'eduire $z\in \tA^+$.  Il en r\'esulte que $H^1(U,\tA^+)=0$

Le r\'esultat s'en d\'eduit via la suite exacte longue de cohomologie.
\end{proof}

\subsubsection{Le mod\`ele de Kirillov de $\Pi_p(V)\otimes V^\dual$}\label{kirp0}
D'apr\`es la prop.\,\ref{cano1},
$\Pi_p(V)$ est naturellement un sous-$P$-module de $(\tA^-\otimes V)^H$.
On dispose donc d'une application naturelle 
$$\omega:\Pi_p(V)\otimes V^\dual\to\tB^-,
\quad
(a\otimes v)\otimes\check v\mapsto \langle \check v,v\rangle\,a$$ 
(Si $V$ est une $\O_L$-repr\'esentation, alors $\omega$ est \`a valeurs dans $\tA^-$.)

On peut utiliser $\omega:\Pi_p(V)\otimes V^\dual\to\tB^-$
pour fabriquer un mod\`ele de Kirillov: 
on pose
$${\cal K}_{p,v\otimes\check v}(x)=\big\langle\check v, \matrice{x}{0}{0}{1}\cdot v\big\rangle,
\quad{\text{si $v\in\Pi_p(V)$ et $\check v\in V^\dual$}}$$
Cela fournit une application $P$-\'equivariante (non n\'ecessairement injective):
$${\cal K}_p:\Pi_p(V)\otimes V^\dual\to{\cal C}(\Q_p^\dual,\tB^-),
\quad\big(\matrice{a}{b}{0}{1}\phi\big)(x)=[\epsilon^{bx}]\phi(ax)$$
De plus, on a les relations suivantes:
\begin{align*}
\sigma({\cal K}_{p,v\otimes\check v}(x))=
\big\langle\sigma(\check v), &\matrice{\chi(\sigma)x}{0}{0}{1}\cdot v\big\rangle
={\cal K}_{p,\sigma(v\otimes\check v)}(\chi(\sigma)x)\\
\varphi({\cal K}_{p,v\otimes\check v}(x))=
\big\langle\check v, &\matrice{px}{0}{0}{1}\cdot v\big\rangle
={\cal K}_{p,v\otimes\check v}(px)
\end{align*}
Si on fait agir $\varphi$ et $G_{\Q_p}$ sur ${\cal C}(\Q_p^\dual,\tB^-)$ par
$$(\varphi\cdot\phi)(x)=\varphi(\phi(p^{-1}x)),\quad
(\sigma\cdot\phi)(x)=\sigma(\phi(\chi(\sigma)^{-1}x))$$
les relations ci-dessus se traduisent par le fait que ${\cal K}_p$ est $\varphi$ et $G_{\Q_p}$-\'equivariante.

\medskip
Si $V$ est la restriction d'une repr\'esentation $\widetilde V$ de $G_\Q$, on peut
\'etendre ${\cal K}_p$ en une application $G_\Q\times G$-\'equivariante
$$\widetilde{\cal K}_p:\Pi_p(V)\otimes V^\dual\to {\rm Ind}_{G_{\Q_p}\times P}^{G_\Q\times G}
{\cal C}(\Q_p^\dual,\tB^-),\quad \widetilde{\cal K}_{p,w}(\sigma,g)={\cal K}_{p,(\sigma\otimes g)\cdot w}$$
\begin{prop}\phantomsection\label{kirp1}
Si $\widetilde V$ est absolument irr\'eductible, alors $\widetilde{\cal K}_p$ est injective.
\end{prop}
\begin{proof}
Le noyau est stable par $G_\Q\times P$, et donc est de la forme $\Pi\otimes \widetilde V^\dual$, o\`u $\Pi$
est une sous-$P$-repr\'esentation de $\Pi_p(V)$, puisque $\widetilde V$ est suppos\'ee absolument irr\'eductible.
On veut prouver que $\Pi=0$; supposons le contraire. 
Quitte \`a remplacer $\Pi$ par une sous-$P$-repr\'esentation,
on peut supposer que $\Pi$ est une composante du $P$-socle de $\Pi_p(V)$.  

$\bullet$ Si $V$ est irr\'eductible, alors $\Pi=\Pi_p(V)$
et on obtient une contradiction car $\omega$ n'est pas identiquement nulle
sur $\Pi_p(V)\otimes V^\dual$.

$\bullet$ Si la semi-simplifi\'ee de $V$ est $\chi_1\oplus\chi_2$, alors il existe $i=1,2$
tel que
$\Pi\subset B(\chi_i,\chi_{3-i})$ (le quotient est de dimension~$\leq 1$ sur $L$)
et alors $\chi_i^{-1}$ est un quotient de $V^\dual$.
Dans ce cas, la restriction de ${\cal K}_p$ \`a $\Pi\otimes V^\dual$
se factorise \`a travers $\Pi\otimes \chi_i^{-1}$ et l'application
induite est celle obtenue en rempla\c{c}ant $V$ par $\chi_i$ (i.e.~$\Pi\subset (\tB^-\otimes\chi_i)^H$
et ${\cal K}_p$ envoie $(a\otimes\chi_i)\otimes\chi_i^{-1}$ sur $a$); cette application est injective,
ce qui conduit \`a une contradiction.
\end{proof}

\subsubsection{Cons\'equences de l'existence d'un mod\`ele de Kirillov pour $\Pi\otimes V$}
Les r\'esultats de ce $\no$ nous serviront pour \'etablir une compatibilit\'e local-global
pour la correspondance de Langlands locale $p$-adique~(th.\,\ref{Ki117}).
\begin{lemm}\phantomsection\label{kirp2}
Soit $W$ une $\O_L$-repr\'esentation de $G_{\Q_p}$, de rang fini.
S'il existe une fl\`eche $\tA^+$-lin\'eaire ${\cal K}:\tA^-\otimes W\to \tA^-$, 
non nulle, commutant aux actions de $G_{\Q_p}$ et $\varphi$, alors $(W^\dual)^{G_{\Q_p}}\neq 0$.
\end{lemm}
\begin{proof}
Soit $e_i$, $i\in I$, une base de $W$. Si $n\geq 1$, soit $x_n\in\tA$ un rel\`evement
de ${\cal K}(\tilde p^{-n}\otimes e_i)$. Alors $\tilde p\, x_{n+1}-x_n\in\tA^+$ et donc
$\tilde p^n x_n$
converge vers un \'el\'ement ${\cal K}(e_i)\in\tA^+$. On a alors ${\cal K}(\tilde p^{-n}\otimes e_i)=
{\cal K}(e_i)\tilde p^{-n}$; on en d\'eduit que
${\cal K}(\sum_ia_ie_i)=\sum_i{\cal K}(e_i)a_i$,
si les $a_i$ appartiennent \`a $\tA^-$ (si $k\geq 1$, 
il existe $n$ tel que $\tilde p^na_i\in\tA^+$ modulo
$p^k\tA$ et on d\'eduit le r\'esultat en passant \`a la limite sur $k$).  Autrement dit,
${\rm Hom}_{\tA^+}(\tA^-\otimes W,\tA^-)=\tA^+\otimes W^\dual$.
L'invariance de ${\cal K}$ par $\varphi$ se traduit par son appartenance \`a
$(\tA^+\otimes W^\dual)^{\varphi=1}=W^\dual$ et son invariance par $G_{\Q_p}$
par son appartenance \`a $(W^\dual)^{G_{\Q_p}}$. 
\end{proof}

\begin{prop}\phantomsection\label{kirp3}
Soit $V$ une $L$-repr\'esentation de $G_{\Q_p}$, 
et soit $\Pi$ une repr\'esentation unitaire de $\GG(\Q_p)$. 
S'il existe une fl\`eche
$G_{\Q_p}\times \matrice{p^\Z}{\Q_p}{0}{1}$-\'equivariante\footnote{
{\rm $G_{\Q_p}$ agissant par $\matrice{\chi(\sigma)}{0}{0}{1}\otimes\sigma$ sur le terme de gauche}.}
 $\Pi\otimes V\to\tB^-$, non nulle,
alors il existe des fl\`eches $G_{\Q_p}$-\'equivariantes $V\to {\bf V}(\Pi)^\dual$
et ${\bf V}(\Pi)\to V^\dual$, non nulles.
\end{prop}
\begin{proof}
Quitte \`a multiplier notre fl\`eche initiale par $p^k$,
on peut choisir des $\O_L$-r\'eseaux $V_0$ et $\Pi_0$ de $V$ et $\Pi$, stables par $G_{\Q_p}$
et $\PP(\Q_p)$, tels que la restriction \`a $\Pi_0\otimes V_0$ soit
\`a valeurs dans $\tA^-$ (car $\Pi$ est de longueur finie comme $\PP(\Q_p)$-module topologique).
On a $\Pi_0\subset (\tA^-\otimes {\bf V}(\Pi_0))^H$ et la commutation \`a l'action
de $\matrice{1}{\Q_p}{0}{1}$ \'equivaut \`a ce que la fl\`eche 
soit $W((\Z_p[\bmu_{p^\infty}])^\flat)=\tA^+_{\Q_p}$-lin\'eaire, 
ce qui permet de l'\'etendre en une fl\`eche $\tA^+$-lin\'eaire
$\tA^-\otimes(V_0\otimes {\bf V}(\Pi_0))\to\tA^-$ qui commute \`a $G_{\Q_p}$ et \`a $\varphi$ (qui encode
l'action de $\matrice{p}{0}{0}{1}$).
On peut donc appliquer le lemme~\ref{kirp2} 
\`a $W=V_0\otimes{\bf V}(\Pi_0)$ pour en d\'eduire le r\'esultat.
\end{proof}

\begin{rema}\phantomsection\label{kirp4}
Si $\Pi$ est irr\'eductible, ${\bf V}(\Pi)$ est irr\'eductible (en effet,
$\Pi$ est isomorphe \`a $\Pi_p({\bf V}(\Pi))$ \`a des repr\'esentations de dimension
finie pr\`es, et l'irr\'eductibilit\'e de ${\bf V}(\Pi)$ r\'esulte de~\cite[rem.\,II.2.3]{gl2}).
La prop.\,\ref{kirp3} restreint donc fortement les $V$ pour lesquelles
$\Pi\otimes V$ admet une fl\`eche \'equivariante non nulle vers~$\tB^-$.
\end{rema}

\section{Correspondances de Langlands en famille}\label{EUL2}
Dans ce chapitre, on construit la repr\'esentation $\Pi(\rho_T)$ de $\GG(\Ai)$
associ\'ee \`a une $T$-repr\'esentation de $G_\Q$.
Cette repr\'esentation est obtenue en prenant un r\'eseau naturel du mod\`ele de
Whittaker du produit tensoriel restreint de repr\'esentations $\Pi_\ell(\rho_T)$
de $\GG(\Q_\ell)$ associ\'ees aux restrictions de $\rho_T$ aux $G_{\Q_\ell}$.
La construction des $\Pi_\ell(\rho_T)$ est diff\'erente selon que $\ell=p$ (cf.~\S\,\ref{EUL3})
ou $\ell\neq p$ (cf.~\S\,\ref{EUL4}).  On termine par la d\'efinition
du mod\`ele de Kirillov de $\rho_T\otimes\Pi(\rho_T^\diamond)$.

\Subsection{Quelques notations}
\subsubsection{L'espace ${\cal X}$}
Soient $T$ une $\O_L$-alg\`ebre noeth\'erienne
locale\footnote{Ce qui suit s'applique
\`a une alg\`ebre semi-locale: il suffit de l'\'ecrire comme produit
d'alg\`ebres locales et de prendre le produit des correspondances pour
chaque composante.}
compl\`ete, r\'eduite, d'id\'eal maximal ${\goth m}_T$
et de corps r\'esiduel $k_L\cong {\bf F}_q$.

Si $M$ est un $T$-module, on d\'efinit des duaux $M^\dual$ et $M^\diamond$ par:
$$M^\dual:={\rm Hom}_{\O_L}(M,\O_L),\quad M^\diamond:={\rm Hom}_T(M,T)$$
\begin{rema}\phantomsection\label{connexe}
On peut d\'ecomposer $T[\frac{1}{p}]$ sous la forme $\prod_{i\in I}\Lambda_i$, o\`u $I$ est
fini et $\Lambda_i$ est connexe. Soit $e_i\in T[\frac{1}{p}]$ l'idempotent correspondant \`a $\Lambda_i$,
et soit $T_i:=e_i T$. Alors
 $\Lambda_i=T_i[\frac{1}{p}]$, et $T\hookrightarrow \prod_iT_i$ et le quotient
est tu\'e par une puissance de $p$, ce qui permet de ramener beaucoup de questions au cas o\`u
$T$ est connexe.
\end{rema}
Soient ${\cal X}:={\rm Spec}\,T$ et  
${\cal X}':={\rm Spec}\,T[\frac{1}{p}]$. Alors ${\cal X}'$ est un ouvert de ${\cal X}$
dont les composantes connexes sont les ${\cal X}'_i:={\rm Spec}\,T_i[\frac{1}{p}]$.

Les composantes irr\'eductibles de ${\cal X}$ et ${\cal X}'$ sont en bijection avec
les id\'eaux premiers minimaux de $T$; on note ${\cal X}_{\goth a}$ et ${\cal X}'_{\goth a}$ les composantes
correspondant \`a ${\goth a}$ (i.e.~${\cal X}_{\goth a}={\rm Spec}(T/{\goth a})$)

Si ${\goth p}\in{\cal X}$, on note
$\kappa({\goth p})$ son corps r\'esiduel. On dit que ${\goth p}$ est {\it un point de ${\cal X}$} 
si $\kappa({\goth p})$ est une extension finie de $L$; ces points correspondent donc aux id\'eaux
maximaux de~$T[\frac{1}{p}]$.

\subsubsection{La repr\'esentation~$\rho_T$}\label{pathol}
 Soit $\rho_T:G_{\Q,S}\to
{\bf GL}_2(T)$ une $T$-repr\'esentation continue de~$G_{\Q,S}$.
Si $T\to\Lambda$ est un morphisme d'anneaux, on pose
$\rho_\Lambda:=\Lambda\otimes_T\rho_T$, et on note simplement
$\overline\rho_T$ la repr\'esentation $\rho_{T/{\goth m}_T}=k_L\otimes_T\rho_T$.

\vskip1mm
On dit qu'un point ${\goth p}$ de ${\cal X}$ est 
{\it $\ell$-probl\'ematique} (resp.~{\it $\ell$-pathologique}), si
la restriction de $\rho_{\kappa({\goth p})}$ \`a $G_{\Q_\ell}$ est 
$\matrice{\cyp }{0}{0}{1}\otimes\delta$ et
si celle de
$\rho_{\kappa({\goth a})}$ est $\matrice{\cyp }{*}{0}{1}\otimes\delta$, avec $*$ non scind\'e,
 pour au moins un (resp.~pour tout) id\'eal premier minimal
${\goth a}$ contenu dans ${\goth p}$. Les points $\ell$-probl\'ematiques sont contenus dans un ferm\'e
de Zariski de codimension~$\geq 1$.

On dit que ${\goth p}$ est {\it $p$-pathologique} si la restriction de $\rho_{\kappa({\goth p})}$ \`a $G_{\Q_p}$
est de la forme $\matrice{1}{*}{0}{\cyp }\otimes\delta$.

\vskip1mm
Si $\ell$ est un nombre premier, on sait associer \`a $\rho_T$ (ou plut\^ot
\`a sa restriction \`a $G_{\Q_\ell}$) une $T$-repr\'esentation $\Pi_\ell(\rho_T)$
et une $k_L$-repr\'esentation $\Pi_\ell(\overline\rho_T)$
\index{Pi@\Pip}de $\GG(\Q_\ell)$.
La recette est diff\'erente suivant que $\ell=p$ ou que $\ell\neq p$, et les sp\'ecialisations
se comportent mal aux points probl\'ematiques ce qui nous conduit \`a introduire les notions suivantes.

Si $\ell\neq p$, on note $\Pi_\ell^{\rm min}(\rho_{\kappa({\goth p})})$ le socle
de $\Pi_\ell^{\rm cl}(\rho_{\kappa({\goth p})})$; on a $\Pi_\ell^{\rm min}(\rho_{\kappa({\goth p})})=
\Pi_\ell^{\rm cl}(\rho_{\kappa({\goth p})})$ sauf si la restriction de $\rho_{\kappa({\goth p})}$ \`a $G_{\Q_\ell}$ est  
$\matrice{\cyp }{0}{0}{1}\otimes\delta$ auquel
cas le quotient est de dimension~$1$.

Si $\ell=p$, on pose $\Pi_p^{\rm min}(\rho_{\kappa({\goth p})})=
\Pi_p(\rho_{\kappa({\goth p})})$ si ${\goth p}$ n'est pas $p$-pathologique.
Si ${\goth p}$ est $p$-pathologique, $\Pi_p^{\rm min}(\rho_{\kappa({\goth p})})$
est la repr\'esentation d\'efinie dans la rem.\,\ref{pathop}.
\Subsection{La correspondance  modulo~$p$}\label{glob103}
Commen\c{c}ons par la d\'efinition de $\Pi_\ell(\overline\rho_T)$.

\vskip2mm
\noindent $\bullet$ 
Si $\ell\neq p$, il existe (\cite[th.\,4.3.1]{Em08}, \cite[th.\,5.1.5]{EH}, \cite{Helm})
une unique $k_L$-repr\'esentation $\Pi_\ell(\overline\rho_T)$ de $\GG(\Q_\ell)$
v\'erifiant:

(i) $\Pi_\ell(\overline\rho_T)$ est de socle g\'en\'erique et irr\'eductible, et le quotient par le socle
est de dimension finie sur $k_L$.

(ii) Si $\rho:G_{\Q_\ell}\to{\rm GL}_2(\O_K)$ est un rel\`evement de $k_K\otimes\overline\rho_T$, et si
$\Pi_\ell(\rho)$ est l'unique r\'eseau de $\Pi_\ell(K\otimes_{\O_K}\rho)$ tel que $k_K\otimes\Pi_\ell(\rho)$
v\'erifie (i), alors $k_K\otimes\Pi_\ell(\rho)\hookrightarrow \Pi_\ell(\overline\rho_T)$.

(iii) $\Pi_\ell(\overline\rho_T)$ est minimale relativement aux conditions (i) et (ii).

\vskip1mm
On note $\Pi_\ell^{\rm min}(\overline\rho_T)$ le socle de $\Pi_\ell(\overline\rho_T)$.

\vskip1mm
{\it Si $\overline\rho_T$ est non ramifi\'ee}, on d\'efinit
$\Pi_\ell^{\rm nr}(\overline\rho_T)$ comme $k_L\otimes_{\O_L}\Pi_\ell(\rho)$
pour n'importe quel rel\`evement non ramifi\'e $\rho$ de $\overline\rho_T$.
On a 
$$\Pi^{\rm min}_\ell(\overline\rho_T)\hookrightarrow
\Pi_\ell^{\rm nr}(\overline\rho_T)\hookrightarrow \Pi_\ell(\overline\rho_T)$$

\vskip2mm
\noindent $\bullet$ 
Si $\ell=p$, il existe une unique $k_L$-repr\'esentation $\Pi_p(\overline\rho_T)$ de $\GG(\Q_p)$
v\'erifiant:

(i) $\Pi_p(\overline\rho_T)$ n'a pas de sous-$k_L[\GG(\Q_p)]$-module non nul, de dimension finie sur $k_L$.

(ii) ${\bf V}(\Pi_p(\overline\rho_T))=\overline\rho_T$.

(iii) $\Pi_p(\overline\rho_T)$ est maximale relativement aux conditions (i) et (ii).

\vskip1mm
On note $\Pi^{\rm min}_p(\overline\rho_T)$ la repr\'esentation minimale v\'erifiant (ii).

\begin{rema}\phantomsection\label{patibulaire}
Dans tous les cas (y compris $\ell=p$), le quotient 
$\Pi_\ell(\overline\rho_T)/\Pi^{\rm min}_\ell(\overline\rho_T)$ est de dimension finie
sur $k_L$.  En fait, on a 

$$\Pi^{\rm min}_\ell(\overline\rho_T)=\Pi_\ell(\overline\rho_T)$$
sauf dans les cas suivants:

$\bullet$ $\ell\neq p$ et $\overline\rho_T\cong \matrice{\cyp }{*}{0}{1}\otimes\delta$,
o\`u le quotient est de dimension $1$ ou $2$ (voire $3$ si $p=2$), suivant la valeur de $*$ et la classe de
$\ell$ modulo~$p$.

$\bullet$ $\ell=p$ et $\overline\rho_T\cong \matrice{1}{*}{0}{\cyp }\otimes\delta$,
o\`u le quotient est de dimension $2$ ($3$ si $p=2$).
\end{rema}

\Subsection{Le cas $\ell=p$}\label{EUL3}
 Le $T$-module
$\Pi_p(\rho_T)$ \index{Pi@\Pip}est la boule unit\'e d'un $L$-banach,
le $\O_L$-dual $\Pi_p^\dual(\rho_T)$
de $\Pi_p(\rho_T)$ est un $T$-module compact et m\^eme un $T[[K]]$-module de type fini,
si $K$ est un sous-groupe ouvert compact de $\GG(\Q_p)$.

Soit $T_n:=T/{\goth m}_T^n$; c'est une $\O_L$-alg\`ebre de longueur finie.
Soit $\rho_n=T_n\otimes_T\rho_T^\diamond$, et soit $\rho_n^\vee$ son dual
de pontryagin. On a une suite exacte $0\to\rho_1^{\oplus r_n}\to \rho_n\to\rho_{n-1}\to 0$,
et donc une suite exacte $0\to \rho_{n-1}^\vee\to \rho_n^\vee\to (\rho_1^\vee)^{\oplus r_n}\to 0$.
Posons\footnote{Par construction, $\Pi_p(\rho_n^\vee)$ est la plus petite repr\'esentation $\Pi$
de $\GG(\Q_p)$ v\'erifiant ${\bf V}(\Pi)=\rho_n^\vee$.}
$$\Pi_p(\rho_n^\vee):=(D(\rho_n^\vee)\boxtimes\piqp)/(D(\rho_n^\vee)^\sharp\boxtimes\piqp)$$
Si $\overline\rho_T$ n'est pas de la forme $\matrice{1}{*}{0}{\epsilon}$, alors
on a une suite exacte
$$0\to \Pi_p(\rho_{n-1}^\vee)\to \Pi_p(\rho_n^\vee)\to \Pi_p(\rho_1^\vee)^{\oplus r_n}\to 0$$
Si $\overline\rho_T$ est de la forme ci-dessus, on a une suite exacte
$$0\to \Pi_p(\rho_{n-1}^\vee)'\to \Pi_p(\rho_n^\vee)\to \Pi_p(\rho_1^\vee)^{\oplus r_n}\to 0$$
o\`u $\Pi_p(\rho_{n-1}^\vee)'$ contient $\Pi_p(\rho_{n-1}^\vee)$ et le quotient est
de longueur finie sur $\O_L$.
Soit alors $\Pi_p((\rho_T^\diamond)^\vee):=\varinjlim_n \Pi_p(\rho_n^\vee)$, et soit
($T_p$ d\'esigne le module de Tate)
$$\Pi_p(\rho_T):=T_p(\Pi_p((\rho_T^\diamond)^\vee))$$
En tant que $T$-module, $\Pi_p(\rho_n^\vee)$ est une somme directe de copies de $T_n^\vee$
(au moins dans le cas g\'en\'erique; dans le cas exceptionnel, c'est le cas \`a un $\O_L$-module
de longueur finie pr\`es). Il s'ensuit
que $\Pi_p(\rho_T)$ est une somme directe compl\'et\'ee de copies de $\check{T}:={\rm Hom}_{\O_L}(T,\O_L)$,
et donc que la $T$-torsion est dense dans $\Pi_p(\rho_T)$ (car elle est dense dans
$\check{T}$).

\begin{rema}\phantomsection\label{cano1.31}
(i)
Soit $\rho_T^\clubsuit=(\rho_T^\diamond)^\dual$ (c'est aussi le module de Tate
de $(\rho_T^\diamond)^\vee$, dual de Pontryagin de $\rho_T^\diamond$).
La prop.\,\ref{cano1} s'\'etend \`a ce cadre: on dispose d'une injection $\PP(\Q_p)$-\'equivariante
$$\Pi_p(\rho_T)\hookrightarrow (\tA^-\wotimes\rho_T^\clubsuit)^H.$$
(On a $\Pi_p(\rho_n^\vee)\hookrightarrow (\tA^-\wotimes\rho_n^\vee)$, et donc
$\Pi_p((\rho_T^\diamond)^\vee)\hookrightarrow (\tA^-\wotimes(\rho_T^\diamond)^\vee)$,
et on conclut en prenant le module de Tate des deux membres.)

(ii)
Le caract\`ere central de $\Pi_p(\rho_T)$ est 
$$\omega_{\Pi_p}=(x|x|_p)^{-1}\det\rho_T.$$

(iii)
Si ${\goth p}$ est un point de ${\cal X}$,
alors $\Pi_p(\rho_T)[{\goth p}]\hookrightarrow\Pi_p(\rho_{\kappa({\goth p})})$ avec
\'egalit\'e sauf, peut-\^etre, si ${\goth p}$ est $p$-pathologique o\`u l'on peut juste assurer
que $\Pi^{\min}_p(\rho_{\kappa({\goth p})})\hookrightarrow \Pi_p(\rho_T)[{\goth p}]$.

(iv) On a $\Pi_p(\rho_T)^\dual\hookrightarrow D(\rho_T^\diamond(1))^\natural\boxtimes\piqp$,
avec \'egalit\'e sauf, peut-\^etre, si $\overline\rho_T\cong\matrice{1}{*}{0}{\cyp }$
en restriction \`a $G_{\Q_p}$, auquel cas, le quotient est de type fini sur $T$.
\end{rema}

{
\subsubsection{Le mod\`ele de Kirillov}\label{glob7}
Soit $\check{T}:={\rm Hom}_{\O_L}(T,\O_L)$.
On peut utiliser l'injection
$\Pi_p(\rho_T)\hookrightarrow (\tA^-\wotimes\rho_T^\clubsuit)^H$
et l'accouplement tautologique $\langle\ ,\rangle:\rho_T^\diamond\times \rho_T^\clubsuit\to \O_L$
pour fabriquer un mod\`ele de Kirillov 
(non n\'ecessairement injectif), $T[\PP(\Q_p)\times G_{\Q_p}]$-\'equivariant
$${\cal K}_p:\rho_T^\diamond\otimes_T\Pi_p(\rho_T)\to {\cal C}(\Q_p^\dual,\check{T}\otimes_{\Z_p}\tA^-)$$
v\'erifiant:
$$\big\langle{\cal K}_{p,\check v\otimes v}(x),\lambda\big\rangle
:=\big\langle \lambda\check v,\matrice{x}{0}{0}{1}\cdot v\big\rangle,
=\big\langle \check v,\matrice{x}{0}{0}{1}\cdot \lambda v\big\rangle,
\quad{\text{ si $\lambda\in T$, $\check v\in \rho_T^\diamond$ et $v\in \Pi_p(\rho_T)$}}$$
Soit alors ${\cal K}^\GG_p$ l'application $T[\GG(\Q_p)\times G_{\Q}]$-\'equivariante\footnote{Avec
action naturelle de $T$ sur le membre de droite, 
et les actions de $\PP(\Q_p)$ et $G_{\Q_p}$ du \no\ref{kirp0}.}:
\begin{align*}
{\cal K}^\GG_p:\rho_T^\diamond\otimes_T\Pi_p(\rho_T)\to 
{\rm Ind}_{\PP(\Q_p)\times G_{\Q_p}}^{\GG(\Q_p)\times G_\Q}{\cal C}(\Q_p^\dual,\O_L\otimes\tA^-)
\quad{\cal K}^\GG_{p,v\otimes\check v}
(g,\sigma)={\cal K}_{p,(g,\sigma)\cdot(v\otimes\check v)}
\end{align*}
\begin{prop}\phantomsection\label{glob4}
Si $\overline\rho_T$ est absolument irr\'eductible, 
${\cal K}^\GG_p$ est une isom\'etrie sur son image.
\end{prop}
\begin{proof}
Il suffit de prouver que ${\cal K}^\GG_p$ est injective modulo ${\goth m}_L$.
Le noyau est un $T[\GG(\Q_p)\times G_{\Q}]$-module, et il suffit donc de prouver que
${\cal K}^\GG_p$ est injective sur le $T[\GG(\Q_p)\times G_{\Q}]$-socle de 
$M:=k_L\otimes_{\O_L}(\rho_T^\diamond\otimes_T\Pi_p(\rho_T))$.

Le $T[\GG(\Q_p)\times G_{\Q}]$-socle de $M$ est le $\GG(\Q_p)\times G_{\Q}$-socle de $M[{\goth m}_T]$.
Or 
$M[{\goth m}_T]=(\rho_T^\diamond/{\goth m}_T)\otimes(k_L\otimes_{\O_L}\Pi_p(\rho_T))[{\goth m}_T]
\hookrightarrow
\overline\rho_T^\vee\otimes \Pi_p(\overline\rho_T)$ et son
$\GG(\Q_p)\times G_{\Q}$-socle
n'est autre
que $\overline\rho_T^\vee\otimes {\rm soc}_{\GG(\Q_p)}(\Pi_p(\overline\rho_T))$ 
car $\overline\rho_T^\vee$ est suppos\'ee
irr\'eductible.  Mais ${\rm soc}_{\GG(\Q_p)}(\Pi_p(\overline\rho_T^\vee))$ est
une somme directe de repr\'esentations irr\'eductibles de dimension infinie,
et on est ramen\'e \`a v\'erifier que
${\cal K}_p:V(\pi)^\vee\otimes\pi\to {\cal C}(\Q_p^\dual,\tE^-)$ n'est pas identiquement nulle
si $\pi$
est une repr\'esentation irr\'eductible de $\GG(\Q_p)$, de dimension infinie.
Ceci est clair. 
\end{proof}

}

\Subsubsection{Un r\'esultat de densit\'e}
\begin{theo}\phantomsection\label{intox2}
Soit $X\subset{\cal X}$. Les conditions suivantes sont \'equivalentes:

{\rm (i)} $X$ est zariski-dense dans ${\cal X}$.

{\rm (ii)} $\sum_{x\in X}\Pi_p(\rho_T)[{\goth p}_x]$ est dense dans $\Pi_p(\rho_T)$.
\end{theo}
\begin{proof}
Par dualit\'e, on est ramen\'e \`a prouver que 
$\cap_{x\in X} {\goth p_x}\cdot (D(\rho_T^\diamond(1))^\natural\boxtimes\piqp)=0$
si et seulement si $X$ est zariski-dense. Cela r\'esulte de ce que 
$D(\rho_T^\diamond(1))^\natural\boxtimes\piqp$ est un $T$-module compact, sans $T$-torsion.
\end{proof}

On dit qu'un point ${\goth p}$ de ${\cal X}$ est sympathique si 
la restriction \`a $G_{\Q_p}$ de $\rho_{\kappa({\goth p})}$
est irr\'eductible, de Rham, \`a poids de Hodge-Tate distincts.
De mani\`ere \'equivalente, ${\goth p}$ est sympathique si et seulement si
$\Pi_p(\rho_{\kappa({\goth p})})^{\rm alg}$ est dense dans $\Pi_p(\rho_{\kappa({\goth p})})$.
\begin{coro}\phantomsection\label{intox9}
Si les points sympathiques sont zariski-denses dans ${\cal X}$, alors
$(\Pi_p(\rho_T)[\frac{1}{p}])^{\rm alg}$ est dense dans $\Pi_p(\rho_T)[\frac{1}{p}]$.
\end{coro}

\Subsection{Le cas $\ell\neq p$}\label{EUL4}
\subsubsection{Le $\PP(\Q_\ell)$-module $\widetilde{\rm LC}_c(\Q_\ell^\dual,\Lambda)$}\label{EH3}
Si $\Lambda$ est un $\Z_p$-module, on note 
$\widetilde{\rm LC}(\Q_\ell^\dual,\Lambda)$
l'espace 
$$\widetilde{\rm LC}(\Q_\ell^\dual,\Lambda):=
\left\{\phi:\Q_\ell^\dual\to\Z[\bmu_{\ell^\infty}]\otimes\Lambda,\hskip2mm 
 {\text{\parbox{4.6cm}{\Small localement
constantes,\\ \`a support compact dans $\Q_\ell$,\\ $\sigma_a(\phi(x))=\phi(ax)$ 
$\forall a\in\Z_\ell^\dual$, $x\in\Q_\ell^\dual$}}}\right\}$$ 
o\`u $\sigma_a$ agit sur $\Z[\bmu_{\ell^\infty}]$.
On note $$\widetilde{\rm LC}_c(\Q_\ell^\dual,\Lambda)\subset \widetilde{\rm LC}(\Q_\ell^\dual,\Lambda)$$
le sous-espace des fonctions \`a support compact dans $\Q_\ell^\dual$.
On munit $\widetilde{\rm LC}(\Q_\ell^\dual,\Lambda)$ de l'action de $\PP(\Q_\ell)$ donn\'ee
par 
$$\big(\matrice{a}{b}{0}{1}\phi\big)(x):={\bf e}_\ell(bx)\phi(ax)$$
Le sous-espace $\widetilde{\rm LC}_c(\Q_\ell^\dual,\Lambda)$
de $\widetilde{\rm LC}(\Q_\ell^\dual,\Lambda)$ est stable par $\PP(\Q_\ell)$.

\begin{lemm}\phantomsection\label{eh3}
Soient $\Lambda\subset\Lambda'$ des $\Z_p$-modules.
Si $M$ est un sous-$\Z_p$-module de $\widetilde{\rm LC}(\Q_\ell^\dual,\Lambda')$
stable par $\PP(\Q_\ell)$, 
et si $M\cap \widetilde{\rm LC}_c(\Q_\ell^\dual,\Lambda')\subset
\widetilde{\rm LC}_c(\Q_\ell^\dual,\Lambda)$, alors
$M\subset \widetilde{\rm LC}(\Q_\ell^\dual,\Lambda)$
\end{lemm}
\begin{proof}
Si $\phi\in M$, alors $\big(\matrice{1}{\ell^{-n}}{0}{1}-1\big)\phi(x)=({\bf e}_\ell(\ell^{-n}x)-1)\phi(x)
\in \Lambda$, pour tout $n\in\Z$. 
Or ${\bf e}_\ell(\ell^{-n}x)-1$ est inversible dans $\Z_p\otimes\Z[\bmu_{\ell^\infty}]$ si $v_\ell(x)\leq n-1$.
Le r\'esultat s'en d\'eduit.
\end{proof}
\subsubsection{Mod\`ele de Kirillov et mod\`ele de Whittaker}
Si $\pi$ est une $T$-repr\'esentation $\pi$ de $\GG(\Q_\ell)$, un {\it mod\`ele de Kirillov
pour $\pi$}
est une injection $\PP(\Q_\ell)$-\'equivariante ${\cal K}:\pi\to \widetilde {\rm LC}(\Q_\ell^\dual,T)$.
Il arrive (rarement) que l'on dise que ${\cal K}$ est un mod\`ele de Kirillov m\^eme
si ${\cal K}$ n'est pas injectif.
\begin{rema}\phantomsection\label{whit1}
(i) Posons
$$\widetilde{\rm LC}(\GG(\Q_\ell), T):={\rm Ind}_{\PP(\Q_\ell)}^{\GG(\Q_\ell)}
\widetilde{\rm LC}(\Q_\ell^\dual,T)$$
Comme 
${\rm LC}(\Q_\ell^\dual,\Z[\bmu_\ell]\otimes_{\Z} T)=
\Z[\bmu_\ell]\otimes_{\Z}\widetilde {\rm LC}(\Q_\ell^\dual,T)$ est l'induite
de $\UU(\Q_\ell)$ \`a $\PP(\Q_\ell)$ du caract\`ere ${\bf e}_\ell$,
on a $$\Z_p[\bmu_\ell]\otimes_{\Z}\widetilde{\rm LC}(\GG(\Q_\ell), T)=
{\rm Ind}_{\UU(\Q_\ell)}^{\GG(\Q_\ell)}{\bf e}_\ell$$
(induite \`a coefficients dans $\Z_p[\bmu_\ell]\otimes_{\Z}T$).

(ii) Si ${\cal K}:\pi\to \widetilde {\rm LC}(\Q_\ell^\dual,T)$ est un mod\`ele de Kirillov
d'une $T$-repr\'esentation $\pi$ de $\GG(\Q_\ell)$, on peut \'etendre ${\cal K}$
en un mod\`ele $\GG(\Q_\ell)$-\'equivariant ({\it mod\`ele de Whittaker})
$${\cal K}^\GG:\pi\to 
\widetilde{\rm LC}(\GG(\Q_\ell), T),\quad
{\text{avec ${\cal K}^\GG_v(g)={\cal K}_{g\cdot v}(1)$.}}$$
On r\'ecup\`ere ${\cal K}_v$ par ${\cal K}_v(x)={\cal K}^\GG_v\big(\matrice{x}{0}{0}{1}\big)$.

(iii) Le mod\`ele de Kirillov est plus pertinent pour les applications arithm\'etiques \`a
cause de son lien avec les $q$-d\'eveloppements de formes modulaires; le mod\`ele de Whittaker
est plus facile \`a utiliser pour les purs probl\`emes de repr\'esentations car il
est $\GG(\Q_\ell)$-\'equivariant et pas seulement $\PP(\Q_\ell)$-\'equivariant (ou
$\BB(\Q_\ell)$-\'equivariant en pr\'esence d'un caract\`ere central).

Si ${\cal K}_v$ n'est pas divisible par ${\goth m}_L$ dans $\widetilde{\rm LC}(\Q_\ell^\dual,T)$,
alors ${\cal K}^\GG_v$ n'est pas divisible par ${\goth m}_L$ dans $\widetilde{\rm LC}(\GG(\Q_\ell),T)$,
mais la r\'eciproque peut \^etre fausse  et donc $\pi$ peut \^etre $p$-satur\'ee dans
$\widetilde{\rm LC}(\GG(\Q_\ell),T)$ sans l'\^etre dans $\widetilde{\rm LC}(\Q_\ell^\dual,T)$ 
(c'est une des raisons qui font que le mod\`ele de Whittaker
est plus commode pour les purs probl\`emes de repr\'esentations).
\end{rema}

\subsubsection{Une construction possible de $\Pi_\ell(\rho_T)$}
Notre construction de $\Pi_\ell(\rho_T)$, pour $\ell\neq p$, est une variante de celle
de Emerton et Helm~\cite{EH}, (cf.~\no\ref{EH2} pour 
la comparaison avec la th\'eorie de Emerton-Helm). 
On demande \`a $\Pi_\ell(\rho_T)$ d'\^etre une $T$-repr\'esentation lisse de $\GG(\Q_\ell)$, sans $T$-torsion,
interpolant la correspondance classique au sens que:

$\bullet$  $\kappa({\goth a})\otimes_T\Pi_\ell(\rho_T)
=\Pi_\ell^{\rm cl}(\rho_{\kappa({\goth a})})$
si ${\goth a}$ est un id\'eal premier minimal de $T$.

$\bullet$ $\kappa({\goth p})\otimes_T\Pi_\ell(\rho_T)
=\Pi_\ell^{\rm cl}(\rho_{\kappa({\goth p})})$,
si ${\goth p}$ est un point de ${\cal X}$,
sauf si ${\goth p}$ est $\ell$-probl\'ematique o\`u
$\kappa({\goth p})\otimes_T\Pi_\ell(\rho_T)
=\Pi_\ell^{\rm min}(\rho_{\kappa({\goth p})})$.

\vskip1mm
La mani\`ere la plus rapide de construire un tel $\Pi_\ell(\rho_T)$ est de 
commencer par
d\'efinir $\Pi_\ell^{\rm cl}(\rho_{{\rm Fr}(T)})$ 
comme $\prod_{\goth a}\Pi_\ell^{\rm cl}(\rho_{\kappa({\goth a})})$,
le produit portant sur les id\'eaux minimaux de $T$ (l'anneau total des fractions
${\rm Fr}(T)$ est le produit des $\kappa({\goth a})$) identifi\'e \`a son mod\`ele de Whittaker
(cf.~rem.\,\ref{whit1}). Il suffit alors de
poser $$\Pi_\ell(\rho_T):=
\Pi_\ell(\rho_{{\rm Fr}(T)})\cap \widetilde{\rm LC}(\GG(\Q_\ell),T)$$
l'intersection \'etant prise
\`a l'int\'erieur de
$\widetilde{\rm LC}(\GG(\Q_\ell), {\rm Fr}(T))$.

Comme cette d\'efinition est peu explicite, nous allons en donner une variante avec
laquelle il est plus facile de faire des calculs. En particulier,
nous aurons besoin d'une description pr\'ecise
dans le cas non ramifi\'e, et nous traitons ce cas en d\'etail
(${\rm n}^{\rm os}$~\ref{EH6} et~\ref{EH7}, les calculs sont parfaitement classiques).
{

\subsubsection{Mod\`ele de Kirillov des induites paraboliques}\label{EH6}
Soient $\chi_1,\chi_2:\Q_\ell^\dual\to T^\dual$ des caract\`eres continus.
Soit $I(\chi_2,\chi_1)$ l'induite parabolique 
$$I(\chi_2,\chi_1):={\rm Ind}_B^G\,(\chi_1\otimes|\ |_\ell^{-1}\chi_2)$$ 
d\'efinie par:
$$I(\chi_2,\chi_1)=\big\{\phi:\GG(\Q_\ell)\to T,\ 
\phi\big(\matrice{a}{b}{0}{d}x\big)=\chi_1(a)\chi_2(d)|d|_\ell^{-1}\phi(x)\big\}$$
muni de l'action $(g\star\phi)(x)=\phi(xg)$.
Nous allons \'etudier le mod\`ele de Kirillov de $I(\chi_2,\chi_1)$.

Si $\phi\in I(\chi_2,\chi_1)$, on note $\bar\phi:\Q_\ell\to T$ la
fonction 
$$x\mapsto \bar\phi(x)=\phi\big(\matrice{0}{1}{-1}{0}\matrice{1}{-x}{0}{1}\big)=
\phi\big(\matrice{0}{1}{-1}{x}\big)$$
L'identit\'e
$$\matrice{0}{1}{-1}{x}\matrice{a}{b}{c}{d}=\matrice{\frac{ad-bc}{a-cx}}{-c}{0}{a-cx}
\matrice{0}{1}{-1}{\frac{dx-b}{a-cx}}$$
fournit le r\'esultat suivant.
\begin{lemm}\phantomsection\label{eh8}
L'application $\phi\mapsto\bar\phi$ induit un isomorphisme 
$$I(\chi_2,\chi_1)
\overset{\sim}{\to}
{\rm LC}_c(\Q_\ell,T)\oplus T\cdot \phi_\infty,\quad
{\text{o\`u $\phi_\infty:={\bf 1}_{\Q_\ell\moins\ell\Z_\ell}|\ |_\ell^{-1}\chi_2\chi_1^{-1}$}}$$
et l'action de $G$ sur l'espace de droite est donn\'ee par
$$\matrice{a}{b}{c}{d}\bar\phi(x)=\chi_1(\tfrac{ad-bc}{a-cx})
|a-cx|_\ell^{-1}\chi_2(a-cx)\bar\phi\big(\tfrac{dx-b}{a-cx}\big)$$
\end{lemm}

Si $\phi\in I(\chi_2,\chi_1)$, on note ${\cal K}_\phi$ la fonction d\'efinie
par
\begin{align*}
{\cal K}_\phi(y):=\tfrac{1}{G(\chi_2)}\int_{\Q_\ell}\phi\big(\matrice{0}{1}{-1}{0}\matrice{y}{x}{0}{1}\big)
{\bf e}_\ell(-x)\,dx
\end{align*}
L'identit\'e
$$\matrice{0}{1}{-1}{0}\matrice{y}{x}{0}{1}=
\matrice{-1}{0}{0}{y}\matrice{0}{1}{-1}{0}\matrice{1}{x/y}{0}{1}$$
nous donne
\begin{align*}
{\cal K}_\phi(y)={\cal K}_{\bar\phi}(y)
:=\tfrac{1}{G(\chi_2)}\chi_1(-1)\chi_2(y)|y|_\ell^{-1}\int_{\Q_\ell}\bar\phi(-x/y)
{\bf e}_\ell(-x)\,dx
\end{align*}
L'application $\phi\mapsto {\cal K}_\phi$ ou $\bar\phi\mapsto {\cal K}_{\bar\phi}$ 
est le mod\`ele de Kirillov de 
$I(\chi_2,\chi_1)$. On note ${\cal K}(\chi_1,\chi_2)$
l'image dans
$\widetilde{\rm LC}(\Q_\ell,T)$ de l'application $\phi\mapsto{\cal K}_\phi$.

\begin{lemm}\phantomsection\label{eh9}
\'Ecrivons $\chi_2\chi_1^{-1}=\eta\,{\rm nr}_\beta$ avec $\eta(p)=1$ 
et ${\rm nr}_\beta(x)=\beta^{v_\ell(x)}$.
Alors\footnote{
Avec la convention $\frac{\chi_1-\beta\chi_2}{1-\beta}(x)=\chi_1(x)(1+\beta+\cdots+\beta^{v_\ell(x)})$
qui permet d'inclure le cas o\`u $\beta-1$ n'est pas inversible.}
$${\cal K}_{\phi_\infty}=\begin{cases}
\frac{\chi_2(-1)}{G(\chi_2)}\big(\tfrac{\ell\beta-1}{\ell\beta}
\,{\bf 1}_{\Z_\ell}\,\frac{\chi_1-\beta\chi_2}{1-\beta}
-\frac{1}{\ell}{\bf 1}_{\ell^{-1}\Z_\ell}\chi_2\big)
 &{\text{$\eta=1$,}}\\
\beta^{-N}\frac{G(\chi_2\chi_1^{-1})G(\chi_1)}{G(\chi_2)}
\chi_1(-1){\bf 1}_{\ell^{-N}\Z_\ell}\frac{\chi_1}{G(\chi_1)}
&{\text{$\eta$ ramifi\'e, de conducteur $\ell^N$.}}
\end{cases}$$
\end{lemm}
\begin{proof}
On a
\begin{align*}
{\cal K}_{\phi_\infty}(y)&=\tfrac{1}{G(\chi_2)}\chi_1(-1)\chi_2(y)|y|_\ell^{-1}
\int_{\Q_\ell\moins\ell y\Z_\ell}(|\ |_\ell^{-1}\chi_2\chi_1^{-1})(-x/y){\bf e}_\ell(-x)\,dx\\ &=
\chi_2(-1)\tfrac{1}{G(\chi_2)}\chi_1(y)
\int_{\Q_\ell\moins\ell y\Z_\ell}|x|_\ell^{-1}\chi_2\chi_1^{-1}(x){\bf e}_\ell(-x)\,dx
\end{align*}
On peut d\'ecouper cette int\'egrale en $\sum_{n\leq v_\ell(y)}\int_{\ell^n\Z_\ell^\dual}$.

\vskip1mm
\noindent $\bullet$
Si $\eta$
est ramifi\'e de conducteur $\ell^N$, la seule de ces int\'egrales qui est non nulle 
est 
$$\int_{\ell^{-N}\Z_\ell^\dual}\chi_2\chi_1^{-1}(x){\bf e}_\ell(-x)\,dx=
\beta^{-N}\sum_{x\in (\Z/\ell^N)^\dual}\eta(x){\bf e}_\ell(-x/\ell^N)=
\beta^{-N}\eta(-1)G(\eta)$$
Comme $\ell^{-N}\Z_\ell^\dual\subset\Q_\ell\moins\ell y\Z_\ell$, si et seulement
$\Z_\ell^\dual\subset\Q_\ell\moins\ell^{N+1} y\Z_\ell$, et donc si et seulement si $v_\ell(y)+N+1\geq 1$,
on a
$\alpha=\beta^{-N}\eta(-1)G(\eta)G(\chi_2)^{-1}{\bf 1}_{\ell^{-N}\Z_\ell}\chi_1$.
Par ailleurs, $G(\eta)G(\chi_2)^{-1}G(\chi_1)\in\Q^\dual$ et est une unit\'e de $\Z_p$;
il s'ensuit que 
$$\alpha=\alpha_0{\bf 1}_{\ell^{-N}\Z_\ell}\tfrac{\chi_1}{G(\chi_1)},\quad
{\text{o\`u $\alpha_0=\eta(-1)\tfrac{G(\chi_2\chi_1^{-1})G(\chi_1)}{G(\chi_2)}\beta^{-N}\in T^\dual$.}}$$
On en d\'eduit la seconde formule.

\vskip1mm
\noindent $\bullet$
Si $\chi_2\chi_1^{-1}={\rm nr}_\beta$, alors
$$\int_{\ell^n\Z_\ell^\dual}=\begin{cases}
0 &{\text{si $n\leq -2$,}}\\
 -\ell^{-1} \beta^{-1} &{\text{ si $n=-1$,}}\\
(1-\tfrac{1}{\ell})\beta^n&{\text{ si $n\geq 0$.}}
\end{cases}$$
Si $v_\ell(y)=-1$, seul $n=-1$ contribue, tandis que si $v_\ell(y)\geq 0$, on obtient
\begin{align*}
\chi_1(y)\big(-\ell^{-1}\beta^{-1}+\sum_{n=0}^{v_\ell(y)}(1-\tfrac{1}{\ell})\beta^n\big)&=
\chi_1(y)\big(-\ell^{-1}\beta^{-1}+(1-\tfrac{1}{\ell})\tfrac{1- \beta\chi_2\chi_1^{-1}(y)}{1-\beta}\big)\\
&=
\tfrac{\ell\beta-1}{\ell\beta(1-\beta)}\chi_1(y)-\tfrac{(\ell-1)\beta^2}{\ell\beta(1-\beta)}\chi_2(y)
\end{align*}
D'o\`u, en utilisant le fait que $\chi_1=\beta\chi_2$ sur $\ell^{-1}\Z_\ell$,
$$
{\cal K}_{\phi_\infty}=
\tfrac{\chi_2(-1)}{G(\chi_2)}\big({\bf 1}_{\Z_\ell}\cdot\big(\tfrac{\ell\beta-1}{\ell\beta(1-\beta)}\chi_1-
\tfrac{(\ell-1)\beta^2}{\ell\beta(1-\beta)}\chi_2\big)-\tfrac{1}{\ell}{\bf 1}_{\ell^{-1}\Z_\ell^\dual}\chi_2\big)
$$
On en d\'eduit la premi\`ere formule en utilisant 
l'identit\'e ${\bf 1}_{\ell^{-1}\Z_\ell^\dual}+{\bf 1}_{\Z_\ell}={\bf 1}_{\ell^{-1}\Z_\ell}$
et l'identit\'e
$$(\ell\beta-1)\chi_1-(\ell-1)\beta^2\chi_2=(\ell\beta-1)(\chi_1-\beta\chi_2)-\beta(1-\beta)\chi_2\qedhere$$
\end{proof}

\begin{prop}\phantomsection\label{eh10}
{\rm (i)} Si $\eta\neq 1$, alors
$${\cal K}(\chi_1,\chi_2)=\widetilde{\rm LC}_c(\Q_\ell^\dual,T)
\oplus T\cdot {\bf 1}_{\Z_\ell}\tfrac{\chi_1}{G(\chi_1)}
\oplus T\cdot {\bf 1}_{\Z_\ell}\tfrac{\chi_2}{G(\chi_2)}$$

{\rm (ii)} Si $\eta=1$ {\rm(et donc $G(\chi_1)=G(\chi_2)$)}, alors
$${\cal K}(\chi_1,\chi_2)=\widetilde{\rm LC}_c(\Q_\ell^\dual,T)
\oplus T\cdot {\bf 1}_{\Z_\ell}\tfrac{(\ell\beta-1)}{\beta-1}\big(\tfrac{\chi_1}{G(\chi_1)}-\beta\tfrac{\chi_2}{G(\chi_2)}\big)
\oplus T\cdot {\bf 1}_{\Z_\ell}\tfrac{\chi_2}{G(\chi_2)}$$
\end{prop}
\begin{proof}
Le changement de variable $x=-yu$ nous donne
\begin{align*}
{\cal K}_{\bar\phi}(y)=
\chi_1(-1)\tfrac{\chi_2(y)}{G(\chi_2)}\int_{\Q_\ell}\bar\phi(u){\bf e}_\ell(yu)\,du
\end{align*}
Comme la transform\'ee de Fourier induit un isomorphisme
de ${\rm LC}_c(\Q_\ell,T)$ sur $\widetilde{\rm LC}_c(\Q_\ell,T)$,
on en d\'eduit que ${\cal K}(\chi_1,\chi_2)$ est la somme directe de
$\tfrac{\chi_2}{G(\chi_2)}\widetilde{\rm LC}_c(\Q_\ell,T)$ et de $T\cdot{\cal K}_{\phi_\infty}$.
Comme la multiplication par $\tfrac{\chi_2}{G(\chi_2)}$
est bijective sur $\widetilde{\rm LC}_c(\Q_\ell^\dual,T)$, le lemme~\ref{eh9} permet de conclure.
\end{proof}
\begin{rema}
Si $\chi_2=|\ |_\ell\chi_1$, et donc $\eta=1$ et $\beta=\ell^{-1}$, le facteur $\ell\beta-1$ s'annule,
ce qui traduit le fait que le mod\`ele de Kirillov n'est pas injectif (ce qui est rassurant car
${\rm Ind}_B^G\,(\chi_1\otimes\chi_1)$ contient le caract\`ere $\chi_1$, et $\widetilde{\rm LC}(\Q_\ell^\dual,T)$
ne poss\`ede pas de sous-module de type fini stable par $\PP(\Q_\ell)$).
\end{rema}

\subsubsection{Induites non ramifi\'ees}\label{EH7}
On suppose dans ce qui suit que $\ell\notin S$,
et donc que $\chi_1$ et $\chi_2$ sont non ramifi\'es
(et $G(\chi_1)=G(\chi_2)=1$ et $\chi_1(-1)=\chi_2(-1)=1$).
\begin{lemm}\phantomsection\label{eh11}
$\alpha:=\phi_\infty+{\bf 1}_{\ell\Z_\ell}$ est invariante par $\GG(\Z_\ell)$
et ${\cal K}_\alpha=\frac{\ell\beta-1}{\ell\beta}\,{\bf 1}_{\Z_\ell}\frac{\chi_1-\beta\chi_2}{1-\beta}$.
\end{lemm}
\begin{proof}
Soit $\alpha':={\bf 1}_{\ell\Z_\ell}$, et donc $\alpha=\phi_\infty+\alpha'$.
Pour v\'erifier l'invariance sous l'action de $\GG(\Z_\ell)$, il suffit de v\'erifier
celle par $w=\matrice{0}{1}{1}{0}$, par $\matrice{1}{b}{0}{1}$ pour $b\in\Z_\ell$ et
par $\matrice{a}{0}{0}{1}$ pour $a\in\Z_\ell^\dual$.
Pour cela, on utilise la formule du lemme~\ref{eh8}. On a $w\cdot\phi_\infty={\bf 1}_{\Z_\ell}$
et $w\cdot{\bf 1}_{\Z_\ell^\dual}={\bf 1}_{\Z_\ell^\dual}$, et comme $\alpha=\phi_\infty+{\bf 1}_{\Z_\ell}-{\bf 1}_{\Z_\ell^\dual}$
on en d\'eduit l'invariance par $w$; celle par $\matrice{a}{0}{0}{1}$ est imm\'ediate
(et $\phi_\infty$ et $\alpha'$ sont invariantes); pour celle par
$\matrice{1}{b}{0}{1}$, on r\'e\'ecrit $\alpha'+\phi_\infty$ sous la forme ${\bf 1}_{\Z_\ell}+
{\bf 1}_{\Q_\ell\moins\Z_\ell}\phi_\infty$, et les deux morceaux sont alors invariants.

La transform\'ee de Fourier de $\alpha'$ est
$\ell^{-1}{\bf 1}_{\ell^{-1}\Z_\ell}$
et donc ${\cal K}_{\alpha'}=\ell^{-1}{\bf 1}_{\ell^{-1}\Z_\ell}\chi_2$.
Le r\'esultat est alors une cons\'equence directe du lemme~\ref{eh9}.
\end{proof}

\begin{prop}\phantomsection\label{eh13}
Si $\chi_2\chi_1^{-1}={\rm nr}_\beta$, on peut prolonger {\rm (de mani\`ere unique)}
l'action de $\GG(\Q_\ell)$ sur ${\cal K}(\chi_1,\chi_2)$ \`a
$$\Pi_\ell(\chi_1,\chi_2):=\widetilde{\rm LC}_c(\Q_\ell^\dual,T)\oplus
T\cdot {\bf 1}_{\Z_\ell}\chi_2\oplus T\cdot{\bf 1}_{\Z_\ell}\tfrac{\chi_1-\beta\chi_2}{1-\beta}$$
de telle sorte que l'action de $\PP(\Q_\ell)$ soit induite par celle existant sur
$\widetilde{\rm LC}(\Q_\ell^\dual,T)$ et 
${\bf 1}_{\Z_\ell}\tfrac{\chi_1-\beta\chi_2}{1-\beta}$ soit fixe par $\GG(\Z_\ell)$.
\end{prop}
\begin{proof}
L'unicit\'e est claire.  
Posons $v={\bf 1}_{\Z_\ell}\tfrac{\chi_1-\beta\chi_2}{1-\beta}$.
Si $\ell\beta-1$ n'est pas un diviseur de $0$, l'action ci-dessus est induite
par celle existant sur ${\rm Fr}(T)\otimes I(\chi_2,\chi_1)$; il suffit donc de v\'erifier qu'elle
stabilise $\Pi_\ell(\chi_1,\chi_2)$ et il suffit de v\'erifier que c'est le cas pour les actions de $w$
et de $\PP(\Q_\ell)$.  Comme $\Pi_\ell(\chi_1,\chi_2)={\cal K}(\chi_1,\chi_2)
+T\,v$ la stabilit\'e par $w$ est claire puisque
les deux morceaux sont stables. Comme $\matrice{1}{b}{0}{1}\cdot v-v$ est \`a support compact
dans $\Q_\ell^\dual$, cela montre la stabilit\'e par $\matrice{1}{\Q_\ell}{0}{1}$.
Comme celle par $\matrice{\Z_\ell^\dual}{0}{0}{1}$ est claire, il ne reste que celle
par $\matrice{\ell^{-1}}{0}{0}{1}$ \`a v\'erifier; celle-ci suit de l'identit\'e
$$\chi_1(\ell)\matrice{\ell^{-1}}{0}{0}{1}v-v={\bf 1}_{\Z_\ell}\chi_2$$
et de l'appartenance de ${\bf 1}_{\Z_\ell}\chi_2$ \`a ${\cal K}(\chi_1,\chi_2)$.
Le cas o\`u $\ell\beta-1$ est un diviseur de~$0$ s'obtient par sp\'ecialisation (les formules pour 
l'action ci-dessus sont polynomiales en~$\beta$).
\end{proof}
\begin{rema}\phantomsection\label{eh14}
Si $\ell\beta-1=0$ (i.e. $\chi_2=|\ |_\ell\chi_1$), la repr\'esentation de $G$ que l'on obtient
est ${\rm Ind}_B^G(|\ |_\ell\chi_1\otimes|\ |_\ell^{-1}\chi_1)$ alors que l'on est parti de
${\rm Ind}_B^G(\chi_1\otimes\chi_1)$, la division par $0$ ayant permis de faire passer le
caract\`ere $\chi_1\circ\det$ de sous-objet \`a quotient.
\end{rema}

\subsubsection{Construction de $\Pi_\ell(\rho_T)$: le cas connexe}\label{whit2}
Si $T[\frac{1}{p}]$ est connexe, la restriction de la
semi-simplifi\'ee $\rho_T^{\rm ss}$ de $\rho_T$ au sous-groupe d'inertie est constante,
et il y a trois possibilit\'es:

\vskip2mm
$\bullet$ $\rho_T=\rho_0\otimes_{\O_L}\chi$, o\`u $\rho_0:G_{\Q_\ell}\to {\rm GL}_2(\O_L)$ est 
absolument irr\'eductible et $\chi:\Q_\ell^\dual\to T^\dual$
est un caract\`ere localement constant, auquel cas
$\Pi_\ell(\rho_T):=\chi\otimes_{\O_L}\Pi(\rho_0)$ et $\Pi(\rho_0)$ est 
un r\'eseau d'une repr\'esentation supercuspidale dont le mod\`ele
de Kirillov est $\widetilde{\rm LC}_c(\Q_\ell^\dual,\O_L)$.

\vskip2mm
$\bullet$ $\rho_T=\chi_1\oplus\chi_2$ avec $\chi_2\chi_1^{-1}$ ramifi\'e.
Dans ce cas, on pose $\Pi_\ell(\rho_T):=I(\chi_2,\chi_1)$.

\vskip2mm
$\bullet$ 
$\rho_T^{\rm ns}=\chi_1\oplus\chi_2$ avec $\chi_2\chi_1^{-1}$ non ramifi\'e.
Dans ce cas, on d\'efinit $\Pi_\ell(\rho_T^{\rm ns})$ comme la repr\'esentation
$\Pi_\ell(\chi_1,\chi_2)$ de la prop.\,\ref{eh13}.  

$\diamond$ Si $\rho_T\otimes\chi_1^{-1}$ est non ramifi\'ee,
on pose $\Pi_\ell(\rho_T):=\Pi_\ell(\rho_T^{\rm ss})$.

\vskip1mm
$\diamond$ Si $\rho_T\otimes\chi_1^{-1}$ est ramifi\'ee, quitte \`a \'echanger $\chi_1$ et $\chi_2$,
il existe des id\'eaux premiers minimaux
${\goth a}$ de $T$ tels que $\kappa({\goth a})\otimes(\rho_T\otimes\chi_1^{-1})$ soit une extension non
triviale de $1$ par $|\ |_\ell$. 
On doit avoir $\Pi(\rho_T)=\chi_1\otimes\Pi(\rho_T\otimes\chi_1^{-1})$ et donc,
quitte \`a tordre par $\chi_1^{-1}$, 
on peut supposer que $\rho_T^{\rm ss}=1\oplus\chi$, avec $\chi$ non ramifi\'e,
et qu'il existe ${\goth a}$ premier, minimal,
tel que $\kappa({\goth a})\otimes\rho_T$ soit une extension non
triviale de $1$ par $|\ |_\ell$.

Soit $I_1$ l'ensemble des ces ${\goth a}$, et soit $I_2$ son compl\'ementaire (ce sont des ensembles finis).
Soit $T'_i=\prod_{{\goth a}\in I_i}T/{\goth a}$, si $i=1,2$. L'application naturelle
$T\to T'_1\times T'_2$ est injective; on note ${\goth b}_1$ l'image inverse de $\{0\}\times T'_2$
et ${\goth b}_2$ celle de $T'_1\times\{0\}$ et $T_i=T/{\goth b}_i$, si $i=1,2$.
L'application naturelle $T\to T_1\times T_2$ est encore injective, et $\rho_{T_2}$
est non ramifi\'ee alors que, pour tout id\'eal minimal ${\goth a}$ de $T_1$,
$\kappa({\goth a})\otimes\rho_T$ est une extension non
triviale de $1$ par $|\ |_\ell$.

On a une suite exacte $\GG(\Q_\ell)$-\'equivariante 
$$0\to T_1\otimes {\rm St}\to \Pi_\ell(\rho_{T_1}^{\rm ss})\to T_1\to 0$$
et une surjection $\Pi_\ell(\rho_T^{\rm ss})\to \Pi_\ell(\rho_{T_1}^{\rm ss})$
(induite par la surjection $T\to T_1$).  On d\'efinit
$\Pi_\ell(\rho_T)$ comme le noyau de la fl\`eche compos\'ee
$\Pi_\ell(\rho_T^{\rm ss})\to T_1$; on a une suite exacte
$\GG(\Q_\ell)$-\'equivariante 
\begin{equation}\label{tor1}
0\to \Pi_\ell(\rho_T)\to \Pi_\ell(\rho_T^{\rm ss})\to T_1\to 0
\end{equation}
\begin{rema}\phantomsection\label{whit3}
(i)
Au niveau des mod\`eles de Kirillov on a, dans les notations
de la prop.\,\ref{eh13},
un isomorphisme de $T$-modules
$$\Pi_\ell(\rho_T)\cong \widetilde{\rm LC}_c(\Q_\ell^\dual,T)\oplus T\cdot {\bf 1}_{\Z_\ell}\chi_2
\oplus {\goth b}_2\cdot {\bf 1}_{\Z_\ell}\tfrac{\chi_1-\beta\chi_2}{1-\beta}$$
(le ${\goth b}_2$ appara\^{\i}t comme noyau de $T\to T_1$);
en particulier, $\Pi_\ell(\rho_T)$ n'est pas n\'ecessairement libre sur $T$ (mais pas loin).

(ii) En appliquant $M\mapsto T_2\otimes_T M$ \`a la suite exacte~(\ref{tor1}), on obtient une suite exacte
$T_2\otimes_T\Pi_\ell(\rho_T)\to \Pi_\ell(\rho_{T_2})\to T/({\goth b}_1,{\goth b}_2)\to 0$.
En particulier, la fl\`eche naturelle $\Pi_\ell(\rho_T)\to\Pi_\ell(\rho_{T_2})$
n'est pas n\'ecessairement surjective bien que $\rho_T\to \rho_{T_2}$ le soit.
\end{rema}

\begin{rema}\phantomsection\label{connexe2}
Dans tous les cas, $\Pi_\ell(\rho_T)$ contient $\widetilde{\rm LC}_c(\Q_\ell^\dual,T)$
et le quotient $J_\ell(\rho_T)$ est un $T$-module sans torsion engendr\'e par au plus deux \'el\'ements,
et libre sauf dans le dernier cas.
\end{rema}

\subsubsection{Construction de $\Pi_\ell(\rho_T)$: le cas g\'en\'eral}\label{whit4}
Si $T\hookrightarrow \prod_iT_i$ o\`u les $T_i[\frac{1}{p}]$ sont connexes
comme dans la rem.\,\ref{connexe},
on d\'efinit $\Pi_\ell(\rho_T)$ par
$$\Pi_\ell(\rho_T)=\big(\prod\nolimits_i\Pi_\ell(\rho_{T_i})\big)\cap \widetilde{\rm LC}(\GG(\Q_\ell),T)$$
l'intersection \'etant prise dans 
$\widetilde{\rm LC}(\GG(\Q_\ell),\prod_iT_i)$.

Si $p^N$ tue $(\prod_iT_i)/T$, alors $p^N$ tue $\big(\prod_i\Pi_\ell(\rho_{T_i})\big)/\Pi_\ell(\rho_T)$.
En particulier, dans le mod\`ele de Kirillov,
$\Pi_\ell(\rho_T)$ contient $p^N\widetilde{\rm LC}_c(\Q_\ell^\dual,T)$.

\begin{rema}\phantomsection\label{whit10}
On d\'eduit de la classification~\cite{vigne} des repr\'esentations de $\GG(\Q_\ell)$ modulo~$p$ que
$k_L\otimes_{T_i}\Pi_\ell(\rho_{T_i})$ a une et une seule composante de Jordan-H\"older
de dimension infinie qui n'est autre, pour tout $i$, que
$\Pi^{\rm min}_\ell(\overline{\rho_T})$.
\end{rema}

\begin{prop}\phantomsection\label{problem}
Si ${\goth p}$ est un point de ${\cal X}$, alors
$$\kappa({\goth p})\otimes_T\Pi_\ell(\rho_T)\cong 
\begin{cases} 
\Pi_\ell^{\rm cl}(\rho_{\kappa({\goth p})}) &{\text{si ${\goth p}$ n'est pas $\ell$-probl\'ematique,}}\\
\Pi_\ell^{\rm min}(\rho_{\kappa({\goth p})}) &{\text{si ${\goth p}$ est $\ell$-probl\'ematique.}}
\end{cases}$$
\end{prop}
\begin{proof}
Comme l'\'enonc\'e est rationnel, on peut supposer que $T[\frac{1}{p}]$ est connexe, et alors
le r\'esultat est imm\'ediat sur la construction.
\end{proof}

\subsubsection{Comparaison avec la th\'eorie d'Emerton et Helm}\label{EH2}
Notons $\Pi_\ell^{\rm EH}(\rho_T)$ le $p$-satur\'e de l'image
de $\check{T}\otimes_T\Pi_\ell(\rho_T)$ dans 
$\widetilde{\rm LC}(\GG(\Q_\ell),\check{T})$.  Donc $\Pi_\ell^{\rm EH}(\rho_T)$ est de $T$-torsion
(la $T$-torsion est dense dans $\check{T}$).

\begin{exem}\phantomsection\label{whit5}
Si $T=\O_L$, alors $\Pi^{\rm EH}_\ell(\rho_T)$ est l'unique (\`a homoth\'etie pr\`es) $\O_L$-r\'eseau
de $\Pi_\ell^{\rm cl}(L\otimes_{\O_L}\rho_T)$
dont la r\'eduction modulo~${\goth m}_L$ a un socle g\'en\'erique (que $\Pi_\ell(\rho_T)$
v\'erifie cette condition est un cas particulier du (i) du th.\,\ref{whit7} ci-dessous).
\end{exem}

\begin{rema}\phantomsection\label{whit6}
(i) Si on remplace $\Pi_\ell(\rho_T)$ par un sous-r\'eseau dans la d\'efinition
de $\Pi_\ell^{\rm EH}(\rho_T)$, on obtient le m\^eme r\'esultat puisqu'on $p$-sature.
Cela permet de se ramener au cas connexe pour beaucoup d'arguments:
si $T\hookrightarrow\prod_iT_i$ comme dans la rem.\,\ref{connexe},
on peut remplacer $\Pi_\ell(\rho_T)$ par $\oplus_i p^N\Pi_\ell(\rho_{T_i})$.

(ii)
Si ${\goth n}$ est un id\'eal de $T$ (un tel id\'eal est automatiquement ferm\'e, et ${\goth n}[\frac{1}{p}]$
est ferm\'e dans $T[\frac{1}{p}]$)
et si $T_{\goth n}:=T/{\goth n}$ est sans $p$-torsion, alors
$\check{T}[{\goth n}]=\check{T_{\goth n}}$.
Il s'ensuit que
$$\Pi_\ell^{\rm EH}(\rho_T)[\goth n]=\Pi_\ell^{\rm EH}(\rho_T)
\cap\widetilde{\rm LC}(\GG(\Q_\ell),\check{T_{\goth n}})$$
\end{rema}

\begin{theo}\phantomsection\label{whit7}
{\rm (i)}
Le socle de
$k_L\otimes_{\O_L}\Pi_\ell^{\rm EH}(\rho_T)$ est irr\'eductible, de dimension infinie sur $k_L$.

{\rm (ii)} Si ${\goth p}$ est un point de ${\cal X}$,
alors $\Pi_\ell^{\rm EH}(\rho_T)[{\goth p}]$ est le r\'eseau $\Pi_ell^{\rm EH}(\rho_{T/{\goth p}})$ de
$\Pi_\ell^{\rm cl}(\rho_{\kappa({\goth p})})$ sauf si
${\goth p}$ est $\ell$-pathologique o\`u c'est un r\'eseau 
de $\Pi_\ell^{\rm min}(\rho_{\kappa({\goth p})})$.
\end{theo}
\begin{proof}
Par construction, 
$$k_L\otimes_{\O_L}\Pi_\ell^{\rm EH}(\rho_T)\hookrightarrow
\widetilde{\rm LC}(\GG(\Q_\ell),k_L\otimes_{\O_L}\check{T})$$
Son socle ${\bf soc}$ 
est tu\'e par ${\goth m}_T$ car $k_L\otimes_{\O_L}\check{T}$ est de ${\goth m}_T^\infty$-torsion,
et comme $(k_L\otimes_{\O_L}\check{T})[{\goth m}_T]=(T/{\goth m}_T)^\vee=k_L$,
ce socle s'injecte dans $\widetilde{\rm LC}(\GG(\Q_\ell),k_L)$; en particulier, il
est constitu\'e de repr\'esentations g\'en\'eriques. Maintenant, les composantes
de Jordan-H\"older de $k_L\otimes_{\O_L}\Pi_\ell^{\rm EH}(\rho_T)$ sont parmi
celles des $k_L\otimes_{T_i}\Pi_\ell(\rho_{T_i})$ et 
on a donc ${\bf soc}\cong \Pi^{\rm min}_\ell(\overline\rho_T)^{\oplus r}$ pour un certain $r\geq 1$
(cf.~rem.\,\ref{whit10}). Mais le th\'eor\`eme
de multiplicit\'e $1$ affirme que $\Pi^{\rm min}_\ell(\overline\rho_T)$ est de multiplicit\'e~$\leq 1$ dans 
$\widetilde{\rm LC}(\GG(\Q_\ell),k_L)$. On a donc $r=1$, ce qui prouve le (i).

Pour prouver le (ii), on peut supposer que $T[\frac{1}{p}]$ est connexe, et alors le r\'esultat est clair
sur la description de $\Pi_\ell(\rho_T)$, sauf dans le dernier cas 
($\rho_T^{\rm ss}\otimes\chi_1^{-1}$
non ramifi\'e et $\rho_T\otimes\chi_1^{-1}$ ramifi\'ee) dont nous reprenons les notations.

Commen\c{c}ons par prouver que $\Pi_\ell^{\rm EH}(\rho_T)[{\goth b}_1]= \Pi_\ell^{\rm EH}(\rho_{T_2})$.
Il r\'esulte du (i) de la rem.\,\ref{whit3} et du (ii) de la rem.\,\ref{whit6}
que $\Pi_\ell^{\rm EH}(\rho_T)[{\goth b}_1]$ contient comme sous-$T_2$-module
$$\widetilde{\rm LC}_c(\Q_\ell^\dual,\check T_2)\oplus
\check{T}_2{\bf 1}_{\Z_\ell}\chi_2
\oplus {\goth b}_2\check{T}\cdot {\bf 1}_{\Z_\ell}\tfrac{\chi_1-\beta\chi_2}{1-\beta}$$
Or ${\goth b}_2\check{T}$ contient un r\'eseau de $\check{T}_2$ d'apr\`es le lemme~\ref{whit21}
ci-dessous. 
On en d\'eduit que $\Pi_\ell^{\rm EH}(\rho_T)[{\goth b}_1]$ contient un r\'eseau
de $\Pi_\ell^{\rm EH}(\rho_{T_2})$ et donc lui est \'egal par $p$-saturation.

On en d\'eduit le r\'esultat pour tous les ${\goth p}$ appartenant \`a ${\rm Spec}\,T_2$
(vu comme sous-espace de ${\rm Spec}\,T$). Si ${\goth p}$ appartient au compl\'ementaire
(qui est inclus dans ${\rm Spec}\,T_1$),
on a $\Pi_\ell^{\rm EH}(\rho_T)[{\goth p}]=\chi_1\otimes{\rm St}$ qui n'est un r\'eseau
de $\Pi_\ell^{\rm cl}(\rho_{\kappa({\goth p})})$ que si 
$\rho_{\kappa({\goth p})}\otimes\chi_1^{-1}$ est ramifi\'ee, i.e.~si ${\goth p}$
n'est pas $\ell$-pathologique.

Ceci permet de conclure.
\end{proof} 

\begin{lemm}\phantomsection\label{whit21}
Soit ${\goth n}\neq 0$ un id\'eal de $T$. Soit $A({\goth n})$ l'ensemble des id\'eaux premiers
minimaux ${\goth a}$ tels que ${\goth n}$ soit inclus dans l'id\'eal annulateur de la composante
irr\'eductible ${\cal X}_{\goth a}$ de ${\cal X}$ d\'efinie par ${\goth a}$, et soient
${\goth n}'$ l'id\'eal annulateur de $\cup_{{\goth a}\notin A({\goth n})}{\cal X}_{\goth a}$
et $T'=T/{\goth n}'$. Alors il existe $c\in\N$ tel que 
$p^c\check T'\subset {\goth n}\cdot\check T\subset\check T'$.
\end{lemm}
\begin{proof}
On a ${\goth n}{\goth n}'=0$ et donc ${\goth n}\cdot\check T\subset\check{T}[{\goth n}']=\check T'$,
ce qui prouve l'une des inclusions.

Si $\alpha\in T'$,
alors $\alpha T'[\frac{1}{p}]$ est ferm\'e dans $T'[\frac{1}{p}]$
et donc ${\rm Hom}(T'[\frac{1}{p}],L)\to {\rm Hom}(\alpha T'[\frac{1}{p}],L)$ est surjective.
Maintenant, par construction, ${\goth n}$, vu comme id\'eal
de $T'$ contient des \'el\'ements qui ne sont pas des diviseurs de $0$.
Choisissons un tel $\alpha$.

Si $\mu\in{\rm Hom}(T'[\frac{1}{p}],L)$, on d\'efinit $\lambda$ sur $\alpha T'[\frac{1}{T}]$ par
$\lambda(x)=\mu(\frac{x}{\alpha})$ ce qui est possible car $x\mapsto \alpha x$ est une bijection
de $T'[\frac{1}{p}]$ sur $\alpha T'[\frac{1}{T}]$ gr\^ace \`a la condition
mise sur $\alpha$. On peut alors relever $\lambda$ en $\lambda'\in {\rm Hom}(T'[\frac{1}{p}],L)$,
et on a $\alpha\lambda'=\mu$,
i.e. la multiplication par $\alpha$ est surjective sur ${\rm Hom}(T'[\frac{1}{p}],L)=\check{T'}[\frac{1}{p}]$.
Le th\'eor\`eme de l'image ouverte permet d'en d\'eduire que $\alpha \check{T}'\supset p^c\check{T}'$ avec $c\in\N$;
a fortiori ${\goth n}\cdot\check{T}'\supset p^c\check{T}'$.
\end{proof}

\begin{rema}
On d\'eduit du th.\,\ref{whit7} que $\Pi^{\rm EH}_\ell(\rho_T)$ est
le module consid\'er\'e par Emerton et Helm (\cite[th.\,4.4.1]{Em08}, \cite{EH}), ou son
$\O_L$-dual lisse~\cite{HM} si on pr\'ef\`ere le point de vue co-whittaker.
\end{rema}

\subsubsection{Foncteur de Kirillov et nouveau vecteur}\label{whit9}
Rappelons que $L\otimes\Pi_\ell(\rho_T)$ contient $\widetilde{\rm LC}(\Q_\ell^\dual,T)$
(si $T[\frac{1}{p}]$ est connexe, c'est d\'ej\`a le cas de $\Pi_\ell(\rho_T)$).
On \index{vell@\vell}note $ v'_{T,\ell}$ la fonction
$$ v'_{T,\ell}={\bf 1}_{\Z_\ell^\dual}\in
L\otimes\Pi_\ell(\rho_T);$$
 c'est un g\'en\'erateur du {\it foncteur de Kirillov}~\cite[\S\,4.1]{Em08} de $L\otimes\Pi_\ell(\rho_T)$,
i.e.~le $T[\frac{1}{p}]$-module
\begin{center}
$\big\{v\in L\otimes\Pi_\ell(\rho_T),\ \matrice{\Z_\ell^\dual}{\Z_\ell}{0}{1}\cdot v=v,
\ \sum_{i=0}^{\ell-1}\matrice{\ell}{i}{0}{1}_\ell\cdot v=0\big\}.$
\end{center}
(Si $v'_{T,\ell}\in \Pi_\ell(\rho_T)$, c'est 
un g\'en\'erateur du {foncteur de Kirillov} de $\Pi_\ell(\rho_T)$.)
\begin{rema}\phantomsection\label{nodeno}
A priori, l'image de $\check{T}\otimes_Tv'_{T,\ell}$ est 
seulement incluse dans $L\otimes_{\O_L}\Pi_\ell^{\rm EH}(\rho_T)$. 
Mais, si ${\goth p}$ est un point non $\ell$-pathologique de ${\cal X}$, l'image de la $\goth p$-torsion
est incluse dans $\Pi_\ell^{\rm EH}(\rho_T)[\goth p]=\Pi_\ell^{\rm EH}(\rho_{T/{\goth p}})$
car ${\bf 1}_{\Z_\ell}$ appartient \`a son mod\`ele de Kirillov.
Comme le $\O_L$-module engendr\'e par les $\check{T}[\goth p]$ est dense dans $\check T$, on
en d\'eduit que l'image de $\check{T}\otimes_Tv'_{T,\ell}$ est incluse dans $\Pi_\ell^{\rm EH}(\rho_T)$.
\end{rema}

\index{vell@\vell}Si $\ell\notin S$, notons $v_{T,\ell}\in\Pi_\ell(\rho_T)$ la fonction
${\bf 1}_{\Z_\ell}\tfrac{\chi_1-\beta\chi_2}{1-\beta}$ de la prop.\,\ref{eh13}.
De mani\`ere explicite, on a
$$v_{T,\ell}(x)=
\begin{cases}
0& {\text{si $x\notin\Z_\ell$,}}\\
\chi_{1}(\ell)^n+\chi_{1}(\ell)^{n-1}\chi_{2}(\ell)+\cdots+
\chi_{2}(\ell)^n & {\text{si $x\in \ell^n\Z_\ell^\dual$ et $n\geq 0$.}}
\end{cases}$$
C'est {\it le nouveau vecteur normalis\'e}
(par la condition $v_{T,\ell}(1)=1$)
de $\Pi_\ell(\rho_T)$;
il 
engendre $\Pi_\ell(\rho_T)$ en tant que $T[\PP(\Q_\ell)]$-module.

Un petit calcul fournit la relation
\begin{equation}
\label{eu1}
\big(1-(\chi_{1}(\ell)+\chi_{2}(\ell))\matrice{\ell^{-1}}{0}{0}{1}_\ell
+(\chi_{1}(\ell)\chi_{2}(\ell))\matrice{\ell^{-2}}{0}{0}{1}_\ell\big)\star
v_{T,\ell}= v'_{T,\ell}
\end{equation}

\begin{rema}\phantomsection\label{eu100}
Si $T=\O_L$, la repr\'esentation $\Pi_\ell(\rho_T)$ poss\`ede un nouveau vecteur
normalis\'e $v_{T,\ell}$ pour tout $\ell\neq p$.  Il existe $P_\ell\in 1+X\O_L[X]$,
unique, tel que
$$P_\ell\big(\matrice{\ell^{-1}}{0}{0}{1}_\ell\big)\star v_{T,\ell}= v'_{T,\ell}.$$
De plus, $\deg P_\ell\leq 2$, et $P_\ell=1$ si $\Pi_\ell(\rho_T)$
est supercuspidale.
La relation (\ref{eu1}) se traduit, si $\ell\notin S$, par la formule
$$P_\ell=1-(\chi_{1}(\ell)+\chi_{2}(\ell))X+(\chi_{1}(\ell)\chi_{2}(\ell))X^2.$$
Dans tous les cas, $P_\ell(\ell^{-s})^{-1}$ est le facteur d'Euler en $\ell$ de la fonction $L$
associ\'ee \`a~$\rho_T$, i.e.~$P_\ell(X)=\det((1-X\sigma_{\ell}^{-1})_{|\rho_T^{I_\ell}})$.
\end{rema}

\Subsection{Torsion par un caract\`ere}\label{EUL6}
Si $F$ est un sous-corps de $\Q^{\rm ab}$, alors 
$$H:={\rm Gal}(F/\Q)$$
est un quotient de $G_\Q$ et m\^eme de
${\rm Gal}(\Q^{\rm ab}/\Q)=\cZ^\dual$ (l'identification se faisant via le caract\`ere cyclotomique).
Comme la fl\`eche naturelle $\cZ^\dual\to\A^\dual/\Q^\dual\R_+^\dual$
est un isomorphisme, 
$ H $ est aussi, naturellement, un quotient de $\A^\dual$.
On note $$\gamma\mapsto\overline\gamma
\quad{\rm et}\quad
a\mapsto\overline a$$ les morphismes
$G_\Q\to  H $ et $\A^\dual\to  H $
ci-dessus.

On peut voir l'alg\`ebre de groupe compl\'et\'ee
$\Z_p[[ H ]]$ aussi comme l'alg\`ebre des mesures
sur $ H $, \`a valeurs dans $\Z_p$.
Si $h\in  H $, on note 
$[h]$ la masse de Dirac en $h$ (i.e. l'\'el\'ement $h$ vu
comme \'el\'ement de l'alg\`ebre de groupe).
On munit
$\Z_p[[H]]$ d'actions $\Z_p[[H]]$-lin\'eaires de
$G_\Q$ et de $\GG(\A)$ en posant:
\begin{equation}\label{EUL6.5}
\gamma\cdot \lambda=[{\overline\gamma}]\,\lambda
\quad{\rm et}\quad
g\star\lambda=[{\overline{\det g}^{-1}}]\,\lambda,\quad
{\text{si $\gamma\in G_\Q$ et $g\in \GG(\A)$.}}
\end{equation}
Par dualit\'e,  ${\cal C}( H ,\Z_p)$ est aussi
muni d'actions $\Z_p[[H]]$-lin\'eaires de
$G_\Q$ et de $\GG(\A)$.

Soit ${\cal W}_F={\rm Spec}\,\Z_p[[H]]$.  Si $[L:\Q_p]<\infty$,
alors ${\cal W}_F(\O_L)$ est l'ensemble des caract\`eres continus
$\eta:  H \to \O_L^\dual$.

Soit maintenant $T$ une alg\`ebre locale comme ci-dessus.
On suppose que $H$ est fini ou isomorphe au produit de $\Z_p$ par un groupe fini.
Alors $$T[[H]]:=\Z_p[[H]]\wotimes_{\Z_p}T$$ est une alg\`ebre semi-locale.
Si ${\cal X}_T={\rm Spec}\,T$ et ${\cal X}_{T[[H]]}={\rm Spec}\,T[[H]]$ 
et si $[L:\Q_p]<\infty$, alors
${\cal X}_{T[[H]]}(\O_L)={\cal X}_T(\O_L)\times {\cal W}_F(\O_L)$.

Si $\rho_T:G_\Q\to {\bf GL}_2(T)$ est une $T$-repr\'esentation de $G_\Q$,
on note $\rho_{T[[H]]}$ la $T[[H]]$-repr\'esentation
$$\rho_{T[[H]]}:=\Z_p[[H]]\wotimes_{\Z_p}\rho_T,$$ 
o\`u $G_\Q$ agit diagonalement.
Si $x\in {\cal X}_T(\O_L)$, on note $\rho_x$ la repr\'esentation
$$\rho_x:=(T/{\goth p}_x)\otimes_T\rho_T,$$ 
o\`u ${\goth p}_x$ est l'id\'eal
de $T$ correspondant \`a $x$.
La repr\'esentation associ\'ee \`a $(x,\eta)\in {\cal X}_T(\O_L)\times {\cal W}_F(\O_L)=
{\cal X}_{T[[H]]}(\O_L)$
est la tordue $\rho_x\otimes\tilde \eta$ de $\rho_x$ 
par le caract\`ere $\tilde \eta$ de $G_\Q$ d\'efini par
$\tilde\eta(\gamma)=\eta(\overline\gamma)$.
En d'autres termes, la repr\'esentation $\rho_T$ interpole analytiquement
les $\rho_x$, pour $x\in{\cal X}_T(L)$ et $\rho_{T[[H]]}$
interpole analytiquement
les tordues des $\rho_x$, pour $x\in{\cal X}_T(L)$, par les caract\`eres continus de~$H$.

La compatibilit\'e des correspondances de Langlands locales avec
la torsion par un caract\`ere fournit des identifications
de $T[[H]]$-modules:
\begin{align*}
\Pi_p^\dual(\rho_{T[[H]]})= \Z_p[[H]]\wotimes_{\Z_p}\Pi_p^\dual(\rho_T)
\quad&{\rm et}\quad
\Pi_p(\rho_{T[[H]]})={\cal C}(H,\Z_p)\wotimes_{\Z_p}\Pi_p(\rho_T)\\
\Pi_\ell(\rho_{T[[H]]})=&\ \Z_p[[H]]\otimes_{\Z_p}\Pi_\ell(\rho_T),
\ {\text{si $\ell\neq p$,}}
\end{align*}
et ces identifications commutent aux actions de $\GG(\Q_\ell)$
en faisant agir $\GG(\Q_\ell)$ diagonalement sur les membres de droite.

\Subsection{Globalisation}\label{glob1}
Rappelons que $\rho_T$ est non ramifi\'ee hors de $S$. 
\subsubsection{La repr\'esentation $\Pi(\rho_T)$ de $\GG(\Ai)$}\label{glob101}
Posons
\begin{align*}
\Pi^{]p[}(\rho_T)&:=\Pi_{S\moins\{p\}}(\rho_T)\otimes_T\Pi^{]S[}(\rho_T),\\
\Pi_{S\moins\{p\}}(\rho_T):=
\big(\otimes_{\ell\in S\moins\{p\}}\Pi_\ell(\rho_T)\big)/&T{\text{-torsion}},
\quad  \Pi^{]S[}(\rho_T):=\otimes_{\ell\notin S}(\Pi_\ell(\rho_T),v_{T,\ell})
\end{align*}
o\`u les produits tensoriels de la seconde ligne sont au-dessus de $T$ 
et $\otimes_{\ell\notin S}(\Pi_\ell(\rho_T),v_{T,\ell})$
est le produit tensoriel restreint des $\Pi_\ell(\rho_T)$ relativement aux $v_{T,\ell}$;
en ce qui concerne $\Pi_{S\moins\{p\}}(\rho_T)$,
comme les $\Pi_\ell(\rho_T)$, pour $\ell\in S\moins\{p\}$, ne sont pas forc\'ement libres,
le produit tensoriel peut introduire de la $T$-torsion.

En faisant le produit tensoriel\footnote{En fait de produit tensoriel, c'est plut\^ot un produit
ext\'erieur: $(\boxtimes_\ell\phi_\ell)((x_\ell)_{\ell}):=\prod_\ell\phi_\ell(x_\ell)$.}
 des mod\`eles de Kirillov locaux, on fabrique un mod\`ele
de Kirillov $T[\PP(\Aip)]$-\'equivariant:
$${\cal K}^{]p[}:\Pi^{]p[}(\rho_T)\to\widetilde{\rm LC}(\A^{]\infty,p[,\dual},T)$$
Comme on a quotient\'e par la $T$-torsion en faisant le produit tensoriel, ce mod\`ele
est injectif (c'est la raison principale pour quotienter par la torsion).

On a une injection $\PP(\Q_p)$-\'equivariante $\iota_p:\Pi_p(\rho_T)\hookrightarrow 
(\tA^-\wotimes\rho_T^\clubsuit)^H$.
Soit 
$$\iota:={\cal K}^{]p[}\otimes\iota_p:\Pi^{]p[}(\rho_T)\otimes_T\Pi_p(\rho_T)\to
\widetilde{\rm LC}(\Aipdu,(\tA^-\wotimes\rho_T^\clubsuit)^H)$$ et soit
$\Pi(\rho_T)_{\rm ns}$ 
(le ``ns'' en indice signifie ``non satur\'e'')
l'image de $\iota$.
On peut prolonger l'injection $T[\PP(\Ai)]$-\'equivariante
$\Pi(\rho_T)_{\rm ns}\hookrightarrow \widetilde{\rm LC}(\Aipdu,(\tA^-\wotimes\rho_T^\clubsuit)^H)$
en une injection $T[\GG(\Ai)]$-\'equivariante $\iota^\GG$ dans l'induite \`a $\GG(\Ai)$,
i.e.~le mod\`ele de Whittaker
$$\widetilde{\rm LC}(\GG(\Aip),(\tA^-\wotimes\rho_T^\clubsuit)^H):=
{\rm Ind}_{\PP(\Ai)}^{\GG(\Ai)}\widetilde{\rm LC}(\Aipdu,(\tA^-\wotimes\rho_T^\clubsuit)^H)$$

Soit $\Pi(\rho_T)$ le satur\'e de $\Pi(\rho_T)_{\rm ns}$ dans
cette induite; i.e.
\begin{equation}\label{globule}
\Pi(\rho_T):=\iota^\GG\big(\Pi(\rho_T)_{\rm ns}[\tfrac{1}{p}]\big)\cap
\widetilde{\rm LC}(\GG(\Aip),(\tA^-\wotimes\rho_T^\clubsuit)^H)
\end{equation}
On munit $\Pi(\rho_T)$ de la topologie induite par celle de $\tA^-\wotimes\rho_T^\clubsuit$
(i.e. la topologie sur l'induite est obtenue par limite inductive). Le (i) du th.\,\ref{glob2}
ci-dessous montre que la topologie induite sur $\Pi(\rho_T)_{\rm ns}$ est la topologie
initiale.
\begin{rema}\phantomsection\label{glob0}
On peut d\'efinir de la m\^eme mani\`ere $\Pi_S(\rho_T)_{\rm ns}$ et $\Pi_S(\rho_T)$,
et on a 
$$\Pi(\rho_T)=\Pi_S(\rho_T)\otimes_T\Pi^{]S[}(\rho_T),\quad
\Pi(\rho_T)_{\rm ns}=\Pi_S(\rho_T)_{\rm ns}\otimes_T\Pi^{]S[}(\rho_T)$$ 
(et $\Pi^{]S[}(\rho_T)$ est un $T$-module libre).
\end{rema}
\begin{theo}\phantomsection\label{glob2}
{\rm (i)}
Il existe $c\in\N$ tel que $\Pi(\rho_T)_{\rm ns}\subset \Pi(\rho_T)\subset p^{-c}\Pi(\rho_T)_{\rm ns}$.

{\rm (ii)}
On a
$$(k_L\otimes_{\O_L}\Pi(\rho_T))[{\goth m}_T]=\big((k_L\otimes_{\O_L}\Pi_S(\rho_T))[{\goth m}_T]\big)\otimes_{k_L}
\big(\Pi^{]S[}(\rho_T)/{\goth m}_T\big)$$
De plus, $(k_L\otimes_{\O_L}\Pi_S(\rho_T))[{\goth m}_T]$ est g\'en\'erique, et contient 
$\otimes_{\ell\in S}\Pi^{\rm min}_\ell(\overline\rho_T)$, le quotient \'etant non g\'en\'erique,
et $\Pi^{]S[}(\rho_T)/{\goth m}_T\cong\otimes_{\ell\notin S}'\Pi_\ell^{\rm nr}(\overline\rho_T)$.
\end{theo}
\begin{proof}
Pour prouver le (i),
la rem.\,\ref{glob0} permet de remplacer $\Pi(\rho_T)$ par $\Pi_S(\rho_T)$.

On a $\Pi_S(\rho_T)_{\rm ns}\subset \Pi_S(\rho_T)$ par construction.
Pour prouver l'existence de $c$, on se ram\`ene au cas o\`u $T[\frac{1}{p}]$ est connexe
(si $T\hookrightarrow\prod_iT_i$ comme dans la rem.~\ref{connexe}, avec
$(\prod_iT_i)/T$ tu\'e par $p^{c_1}$, la formule~(\ref{globule}) permet de montrer
que $\prod_i\Pi(\rho_{T_i})/\Pi(T)$ est tu\'e par $p^{c_1}$).

On identifie $\Pi_S(\rho_T)_{\rm ns}$ \`a son mod\`ele de Kirillov, i.e.~on fait tous les
calculs dans $\widetilde{\rm LC}(\Q_p^\dual,(\tA^-\wotimes\rho_T^\clubsuit)^H)$.
Soit $x=\sum_i x_i^{]p[}\otimes x_{i,p}\in \Pi_S(\rho_T)_{\rm ns}$, non divisible par $p$. 
On doit exhiber $c$, ne d\'ependant pas de $x$, tel que $\iota(x)$ ne soit pas divisible
par $p^{c}$.

Soit $S':=S\moins\{p\}$ et
soit $T':=\Z[\bmu_{\ell^\infty},\,\ell\in S']\otimes T$; il suffit de prouver le r\'esultat
apr\`es extension des scalaires \`a $T'$ ce qui permet de s'affranchir de la condition d'invariance
par $\Z_{S'}^\dual$.
Le sous-$T'$-module engendr\'e par les $x_i^{]p[}$ est de rang fini, et il est donc
inclus dans un module du type $\oplus_i {\goth n}_i\cdot e_i$, o\`u ${\goth n}_i$
est un id\'eal de $T'$ (produit de ${\goth b}_{i,\ell}$, o\`u les ${\goth b}_{i,\ell}$ sont du type
de ceux apparaissant dans la rem.\,\ref{whit3}; en particulier, ${\goth n}_i$ varie dans un ensemble fini
puisque $S$ est fini et les ${\goth b}_{i,\ell}$ varient dans un ensemble fini), et o\`u $e_i$ est de la forme
$\otimes_{\ell\in S'}e_{i,\ell}$ et $e_{i,\ell}$ est d'une
des formes suivantes:

\vskip1mm
$\bullet$
${\bf 1}_{a(1+\ell^{n_\ell}\Z_\ell)}$, avec 
$n_\ell>0$ ne d\'ependant que de $\ell$, et $a\in\Q_\ell^\dual/(1+p^{n_\ell}\Z_\ell)$.

$\bullet$ ${\bf 1}_{\Z_\ell}({\chi_{\ell,1}}-\chi_{\ell,2})$
ou ${\bf 1}_{\Z_\ell}(\chi_{\ell,2}(a){\chi_{\ell,1}}-\chi_{\ell,1}(a)\chi_{\ell,2})$
pour $a\in\Z_\ell^\dual$ minimisant $v_p({\chi_{\ell,1}}(x)-\chi_{\ell,2}(x))$ 
(si $(\rho_{T,\ell})^{\rm ss}=\chi_{\ell,1}\oplus\chi_{\ell,2}$ et 
$\chi_{\ell,1}\neq \chi_{\ell,2}$ sur $\Z_\ell^\dual$).
On a alors $v_p({\chi_{\ell,1}}(a)-\chi_{\ell,2}(a))=0$ sauf si 
${\chi_{\ell,1}}^{-1}{\chi_{\ell,2}}$ est d'ordre $p^k$, o\`u l'on obtient
$v_(\zeta_{p^k}-1)$.

$\bullet$ ${\bf 1}_{\Z_\ell}\chi_1$ ou $v_{T,\ell}$ (si $\chi_{\ell,1}=\chi_{\ell,2}$
sur $\Z_\ell^\dual$).

\vskip1mm
Il existe alors une collection de $z_i\in\Q_{S'}^\dual$ tels que $e_i(z_j)=0$ si $j\neq i$,
et $e_i(z_i)$ est une unit\'e de $T'$ 
sauf pour les fonctions du second point (o\`u il faut \'evaluer
en $\ell^{n!}$ et $a\ell^{n!}$ pour $n$ assez grand -- la limite
est de la forme $\zeta-1$, o\`u $\zeta\neq 1$ est une racine de l'unit\'e) 
et $v_{T,\ell}$ du troisi\`eme (o\`u il faut \'evaluer
en $\ell^{n!}-1$ pour $n$ assez grand -- la limite est $1$).

Cette quasi-orthogonalit\'e permet de supposer que $x=e^{]p[}\otimes x_p$, o\`u $e^{]p[}$ est une des fonctions
ci-dessus et $x_p\in{\goth n}\Pi_p(\rho_T)$ n'est pas divisible par $p$ dans ${\goth n}\Pi_p(\rho_T)$,
avec ${\goth n}$ un des ${\goth n}_i$ ci-dessus.  
Mais alors $x_p$
n'est pas divisible par $p$ dans $\tA^-\otimes{\goth n}\rho_T^\clubsuit$, et il existe $c({\goth n})$
tel que $x_p$ ne soit pas divisible par $p^{c({\goth n})}\tA^-\otimes\rho_T^\clubsuit$ d'apr\`es
le lemme~\ref{whit21} (car $\rho_T^\clubsuit\cong \check T\oplus\check T$ en tant que $T$-module). 
Si $c'=\sup_{\goth n}c({\goth n})$, il s'ensuit que $\iota(x)$ n'est pas
divisible par $p^{|S|+c'}$ (car $v_p(\zeta-1)\leq 1$, si $\zeta\neq 1$ est une racine de l'unit\'e,
et il y a au plus $|S'|$ termes de cette forme, provenant de fonctions du second point, pour $\ell\in S'$). 
Si $c=c'+|S'|$, alors $x$ n'est pas divisible par $p^c$ dans le mod\`ele de Kirillov, et donc aussi,
a fortiori, dans le mod\`ele de Whittaker.  Cela prouve le (i).

Passons \`a la preuve du (ii). La factorisation r\'esulte de la rem.\,\ref{glob0}, et le reste de l'\'enonc\'e
se d\'emontre en adaptant la preuve du (i) du th.\,\ref{whit7}, les points cl\'es \'etant que
$k_L\otimes_{\O_L}\Pi_S(\rho_T)\hookrightarrow \widetilde{\rm LC}(\GG(\Q_S),\tA^-\wotimes_{\Z_p}\rho_T^\clubsuit)$
et que $\rho_T^\clubsuit\cong \check T\oplus\check T$.
\end{proof}

\subsubsection{Sp\'ecialisation}\label{glob102}
Soit $\overline S$ l'ensemble des nombres premiers en lesquels $\overline\rho_T$ est ramifi\'ee (en incluant $p$).
Si $\overline S\subset\Sigma\subset S$, notons $T(\Sigma)$ le plus grand quotient de $T$ tel que
$\rho_{T(\Sigma)}$ soit non ramifi\'ee en dehors de $\Sigma$, et soit ${\cal X}_\Sigma:={\rm Spec}\,T(\Sigma)$;
c'est un ferm\'e de ${\cal X}$ (qui n'est autre que ${\cal X}_S$).

On dit que $\rho_T$ est {\it \'equilibr\'ee} si ${\cal X}_\Sigma$ est une r\'eunion de composantes irr\'eductibles
de ${\cal X}$ pour tout $\overline S\subset\Sigma\subset S$.
Si $\rho_T$ est \'equilibr\'ee, les points de ${\cal X}$ ne sont $\ell$-pathologiques pour aucun $\ell\neq p$.
\begin{exem}
Si la restriction de $\overline\rho_T$ \`a $G_{\Q(\bmu_p)}$ est irr\'eductible, et si
la restriction \`a $G_{\Q_p}$ n'est pas de la forme $\matrice{1}{*}{0}{\cyp }\otimes\delta$
ou $\matrice{1}{*}{0}{1}\otimes\delta$, alors la d\'eformation universelle
de $\overline\rho_T$ non ramifi\'ee en dehors de $S$ est \'equilibr\'ee
(il ressort des travaux de B\"ockle~\cite{boc}, Diamond-Flach-Guo~\cite{DFG} et Fouquet-Wan~\cite{FW22},
 que, dans ce cas,
 les ${\cal X}_\Sigma$ sont purement de dimension relative~$3$ sur $\Z_p$).
\end{exem}

Si ${\goth p}$ est un point de ${\cal X}$, on d\'efinit $\Pi(\rho_{\kappa({\goth p})})$
et $\Pi^{\rm min}(\rho_{\kappa({\goth p})})$ 
comme les produits tensoriels restreints (au-dessus de $\kappa({\goth p})$):
\begin{align*}
\Pi(\rho_{\kappa({\goth p})})
&:=\Pi_p(\rho_{\kappa({\goth p})})\otimes
\big(\otimes'_{\ell\neq p}\Pi_\ell^{\rm cl}(\rho_{\kappa({\goth p})})\big)\\
\Pi^{\rm min}(\rho_{\kappa({\goth p})})
&:=\big(\otimes_{\ell\in S}\Pi_\ell^{\rm min}(\rho_{\kappa({\goth p})})\big)
\otimes
\big(\otimes'_{\ell\notin S}\Pi_\ell^{\rm cl}(\rho_{\kappa({\goth p})})\big)
\end{align*}
Notons que le quotient $\Pi(\rho_{\kappa({\goth p})})/\Pi^{\rm min}(\rho_{\kappa({\goth p})})$
est petit (i.e.~non g\'en\'erique).
\begin{theo}\phantomsection\label{glob111} 
Soit ${\goth p}$ un point de ${\cal X}$.

{\rm (i)} On a des injections $\Pi^{\rm min}(\rho_{\kappa({\goth p})})\hookrightarrow
\Pi(\rho_T)[{\goth p}]\hookrightarrow \Pi(\rho_{\kappa({\goth p})})$.

{\rm (ii)} Si $\rho_T$ est \'equilibr\'ee, on a un isomorphisme
$\Pi(\rho_T)[{\goth p}]\overset{\sim}{\hookrightarrow} \Pi(\rho_{\kappa({\goth p})})$
sauf, peut-\^etre, si ${\goth p}$ est $p$-pathologique.
\end{theo} 
\begin{proof}
Le (i) r\'esulte de ce que
$\Pi(\rho_T)[{\goth p}]=(\Pi^{]p[}(\rho_T)/{\goth p})\otimes \Pi_p(\rho_T)[{\goth p}]$
et $\Pi_p(\rho_T)[{\goth p}]=\Pi_p(\rho_{\kappa({\goth p})})$ (sauf dans le cas $p$-pathologique
o\`u on a juste une inclusion), et de la prop.\,\ref{problem}.

Le (ii) se d\'emontre par les m\^emes techniques que le (ii) du th.\,\ref{whit7}, en utilisant
le fait que $\Pi_p(\rho_T)$ se comporte comme $\check T$:
ces techniques fournissent une injection de 
$\Pi(\rho_{T(\Sigma)})$ dans $\Pi(\rho_T)$, pour tout $\Sigma$;
on choisit alors $\Sigma$ \'egal \`a l'ensemble des $\ell$ en lesquels $\rho_{\kappa({\goth p})}$
est ramifi\'ee.
\end{proof}

\begin{conj}\phantomsection\label{glob112} 
Si $\overline\rho:G_{\Q}\to{\bf GL}_2(k)$ est irr\'eductible, non ramifi\'ee en dehors de $S$, et si
$T$ est l'anneau des d\'eformations universelles de $\overline\rho$ non ramifi\'ees
en dehors de $S$ et $\rho_T:G_{\Q,S}\to{\bf GL}_2(T)$ est la d\'eformation universelle,
on a un isomorphisme de $k_L[\GG(\Ai)]$-modules
$$(k_L\otimes_{\O_L}\Pi(\rho_T))[{\goth m}_T]\cong
\big(\otimes_{\ell\in S}\Pi_\ell(\overline\rho)\big)\otimes\big(\otimes'_{\ell\notin S}
\Pi_\ell^{\rm nr}(\overline\rho)\big)$$
\end{conj}

\subsubsection{Le mod\`ele de Kirillov de $\rho_T\otimes_T \Pi(\rho_T^\diamond)$}\label{glob104}
On va s'int\'eresser \`a la repr\'esentation $\rho_T\otimes_T \Pi(\rho_T^\diamond)$;
en particulier, cela demande de remplacer $\rho_T$ par $\rho_T^\diamond$ dans les \'enonc\'es
pr\'ec\'edents et $\rho_T^\clubsuit$ par $(\rho_T^\diamond)^\clubsuit=\rho_T^\dual$.

Fixons un plongement $\Z[\bmu^{]p[}]\hookrightarrow\Qbar_p$ et donc aussi un plongement
dans $\tA^+$.
En faisant le produit tensoriel de ${\cal K}^{]p[}$ d\'efini ci-dessus et
de ${\cal K}_p$ du \no\ref{glob7}, on fabrique un mod\`ele
de Kirillov global, $T[\PP(\Ai)\times G_{\Q_p}]$-\'equivariant:
$$\xymatrix@C=3mm@R=4mm{{\cal K}_{\rm Aut}^T:\rho_T\otimes_T\Pi(\rho_T^\diamond)\ar@{=}[r]&
\Pi^{]p[}(\rho_T^\diamond)\otimes_T(\rho_T\otimes_T\Pi_p(\rho_T^\diamond))\ar[d]\\
&{\rm LC}(\A^{]\infty,p[,\dual},\Z[\bmu^{]p[}]\otimes T)\otimes_T {\cal C}(\Q_p^\dual,\check{T}\otimes_{\Z_p}\tA^-)\ar[d]\\
&{\cal C}(\Aidu,\check{T}\otimes_{\Z_p}\tA^-)}$$
On promeut ${\cal K}_{\rm Aut}^T$ en un morphisme $T[\GG(\Ai)\times G_{\Q,S}]$-\'equivariant
$${\cal K}^\GG_{\rm Aut}:\rho_T\otimes_T\Pi(\rho_T^\diamond)\to
{\rm Ind}_{\PP(\Ai)\times G_{\Q_p}}^{\GG(\Ai)\times G_{\Q,S}}
{\cal C}(\Aidu,\check{T}\otimes_{\Z_p}\tA^-),\quad
{\cal K}^\GG_{{\rm Aut},v}(g)={\cal K}^T_{{\rm Aut},{g\cdot v}}$$

\begin{prop}\phantomsection\label{glob3}
Si $\overline\rho_T$ est irr\'eductible, 
${\cal K}^\GG_{\rm Aut}$ est une isom\'etrie sur son image.
\end{prop}
\begin{proof}
C'est une cons\'equence imm\'ediate de la prop.\,\ref{glob4} et de la d\'efinition
de $\Pi(\rho_T^\diamond)$.
\end{proof}


\part{Cohomologie d'espaces fonctionnels ad\'eliques et cohomologie compl\'et\'ee}
\section{Cohomologie de $\GG(\Q)$ et cohomologie compl\'et\'ee}\label{cup0}
Dans ce chapitre, on explique (rem.~\ref{cup7.1}, \ref{cup7.2} et~\ref{cup7.3})
le lien entre la cohomologie compl\'et\'ee
d'Emerton et la cohomologie de $\GG(\Q)$ \`a valeurs dans certains espaces
fonctionnels ad\'eliques comme ${\cal C}(\GG(\A))$.
Le lecteur souhaitant plus de d\'etails est invit\'e \`a consulter~\cite{adele}.
Enfin, on \'etablit (cor.\,\ref{ES6.4}) une dualit\'e entre
$H^1_c(\GG(\Q),{\cal C}(\GG(\A)))$ et $H^1(\GG(\Q),{\rm Mes}(\GG(\A)))$.
\Subsection{Cohomologie \`a support compact}\label{cup1}
\subsubsection{Cocycles nuls sur le borel}
Si $M$ est un $\GG(\Q)$-module,
on d\'efinit la cohomologie \`a support compact $H^\bullet_c(\GG(\Q),M)$
comme la cohomologie du c\^one $[\rg(\GG(\Q),M)\to\rg(\BB(\Q),M)]$.  Elle est donc calcul\'ee
par le complexe\footnote{${\cal C}$ d\'esigne les fonctions continues, mais comme $\GG(\Q)$ est discret,
toutes les fonctions sont continues.}
(dans lequel on note simplement $B$ et $G$ les groupes $\BB(\Q)$ et $\GG(\Q)$)
$$M\to {\cal C}(G,M)\oplus M\to {\cal C}(G\times G,M)\oplus {\cal C}(B,M)\to
{\cal C}(G\times G\times G,M)\oplus {\cal C}(B\times B,M)\to\cdots\,,$$
o\`u les fl\`eches ${\cal C}(H^i,M)\to {\cal C}(H^{i+1},M)$, pour $H=\GG(\Q),\BB(\Q)$, sont
les diff\'erentielles usuelles, les fl\`eches ${\cal C}(\GG(\Q)^i,M)\to {\cal C}(\BB(\Q)^i,M)$
sont les restrictions, et les autres fl\`eches sont nulles\footnote{On d\'efinit de m\^eme
$H^1_c(\Gamma(1),M)$.}.  En particulier,
$$H^1_c(\GG(\Q),M)=\frac{\{((c_\sigma)_{\sigma\in \GG(\Q)},c_B),\ c_{\sigma\tau}=\sigma\cdot c_\tau+c_\sigma,\ 
c_\sigma=(\sigma-1)\cdot c_B,\ {\text{si $\sigma\in \BB(\Q)$}}\}}
{\{(((\sigma-1)\cdot a)_{\sigma\in \GG(\Q)},a),\ a\in M\}}.$$
On note $Z^1(\GG(\Q),\BB(\Q),M)$ le module
des $1$-cocycles $(c_\sigma)_{\sigma\in \GG(\Q)}$ sur $\GG(\Q)$, 
\`a valeurs dans $M$, qui sont identiquement nuls
sur $\BB(\Q)$.
On dispose d'une application naturelle
$Z^1(\GG(\Q),\BB(\Q),M){\to} H^1_c(\GG(\Q),M)$ envoyant
$(c_\sigma)_{\sigma\in \GG(\Q)}$ sur la classe de $((c_\sigma)_{\sigma\in \GG(\Q)},0)$.
\begin{lemm}\phantomsection\label{ES2}
Cette application induit un isomorphisme naturel
$$Z^1(\GG(\Q),\BB(\Q),M)\overset{\sim}{\to} H^1_c(\GG(\Q),M).$$
\end{lemm}
\begin{proof}
Cela r\'esulte de ce que
$$\phantom{XXXXXX}((c_\sigma)_{\sigma},c_B)=((c_\sigma-(\sigma-1)\cdot c_B)_{\sigma},0)+
(((\sigma-1)\cdot c_B)_{\sigma},c_B)\phantom{XXXX}\qedhere$$
\end{proof}

\subsubsection{Le symbole modulaire $(0,\infty)$}\label{cup2}
Soit $\piqp=\piqp(\Q)$, et \index{P1@\PIQP}soient ${\rm Div}(\piqp)$ le $\Z$-module
libre de base $\piqp$ et ${\rm Div}^0(\piqp)$ le sous-$\Z$-module
des $\sum_{x\in\piqp}n_x(x)$ v\'erifiant $\sum_{x\in\piqp}n_x=0$.
Si $a,b\in\piqp$, on note $(a,b)$ l'\'el\'ement $(b)-(a)$ de
${\rm Div}^0(\piqp)$.  Les $(a,b)$ forment une famille g\'en\'eratrice
de ${\rm Div}^0(\piqp)$ et on a les relations $(a,a)=0$ et $(a,b)+(b,c)+(c,a)=0$
pour tous $a,b,c\in\piqp$.

Le groupe $\GG(\Q)$ agit sur $\piqp$ par $\matrice{a}{b}{c}{d} x=\frac{ax+b}{cx+d}$.
Cette action est transitive et
le stabilisateur de $\infty$ est le groupe $\BB(\Q)=\matrice{\Q^\dual}{\Q}{0}{\Q^\dual}$.

\begin{lemm}\phantomsection\label{ES3}
Si $M$ un $\GG(\Q)$-module,
on a un isomorphisme naturel
$$Z^1(\GG(\Q),\BB(\Q),M)\cong H^0(\GG(\Q),{\rm Hom}({\rm Div}^0(\piqp),M)).$$
\end{lemm}
\begin{proof}
Si $\alpha\in{\rm Hom}({\rm Div}^0(\piqp),M)$ est invariant par $\GG(\Q)$, i.e.~si
$$\gamma\cdot \alpha(a,b)=\alpha(\gamma\cdot a,\gamma\cdot b),\quad{\text{pour tous $a,b\in\piqp$,}}$$
on pose $c_\sigma=\alpha(\infty,\sigma\cdot\infty)$.
Alors $c_\sigma=0$ si $\sigma\in \BB(\Q)$ puisque $\sigma\cdot\infty=\infty$,
et 
$$c_{\sigma\tau}=\alpha(\infty,\sigma\cdot\infty)+\alpha(\sigma\cdot\infty,\sigma\tau\cdot\infty)=
c_\sigma+\sigma(\alpha(\infty,\tau\cdot\infty))=c_\sigma+\sigma\cdot c_\tau,$$ 
ce qui prouve
que $(c_\sigma)_{\sigma\in \GG(\Q)}\in Z^1(\GG(\Q),\BB(\Q),M)$.

R\'eciproquement, si
$(c_\sigma)_{\sigma\in \GG(\Q)}\in Z^1(\GG(\Q),\BB(\Q),M)$, et si $a,b\in\piqp$, on choisit
$\sigma_a,\sigma_b\in \GG(\Q)$ avec $\sigma_a\cdot \infty=a$ et $\sigma_b\cdot\infty=b$,
et on pose $\alpha(a,b)=c_{\sigma_b}-c_{\sigma_a}$ (notons que $c_{\sigma_x}$ ne d\'epend que de $x$,
et pas de $\sigma_x$ car si $\sigma\cdot\infty=\sigma'\cdot\infty$, il existe $\gamma\in \BB(\Q)$
tel que $\sigma=\sigma'\gamma$ et alors $c_{\sigma}=\sigma'\cdot c_\gamma+c_{\sigma'}=
c_{\sigma'}$ puisque $c_\gamma=0$ par hypoth\`ese).
Alors $$\alpha(\gamma\cdot a,\gamma\cdot b)=
c_{\gamma\sigma_b}-c_{\gamma\sigma_a}=\gamma\cdot c_{\sigma_b}-c_{\gamma}-
\gamma\cdot c_{\sigma_a}+c_{\gamma}=\gamma\cdot\alpha(a,b).$$
Les deux fl\`eches construites ci-dessus sont 
 inverses l'une de l'autre.
\end{proof}

\begin{rema}\phantomsection\label{ES4}
{\rm (i)} En combinant les lemmes~\ref{ES2} et~\ref{ES3},
on obtient un isomorphisme naturel
$$H^0(\GG(\Q),{\rm Hom}({\rm Div}^0(\piqp),M))\cong H^1_c(\GG(\Q),M).$$

{\rm (ii)} L'isomorphisme ci-dessus
fournit, si
$a,b\in\piqp$, une application naturelle 
$$(a,b): H^1_c(\GG(\Q),M)\to M.$$
Cette application est d\'efinie comme suit: si
$c\in H^1_c(\GG(\Q),M)$ et si $(c_\sigma)_{\sigma\in \GG(\Q)}$ est l'\'el\'ement
de $Z^1(\GG(\Q),\BB(\Q),M)$ correspondant \`a $c$, alors
$$\langle (a,b), c\rangle=c_{\alpha_b}-c_{\alpha_a}, \quad{\text{o\`u $\alpha_a,\alpha_b\in\\G(\Q)$
v\'erifient $\alpha_a\cdot\infty=a$ et $\alpha_b\cdot\infty=b$.}}$$
En particulier,
$$\langle (0,\infty), c\rangle=-c_S,\quad {\text {avec $S=\matrice{0}{1}{-1}{0}$}}.$$
\end{rema}

\Subsection{Espaces fonctionnels ad\'eliques}\label{fma3}
\subsubsection{Repr\'esentations alg\'ebriques}\label{fma4}
On \index{W@\WWW}note $W_{k,j}^\dual$ la $\Q$-repr\'esentation 
$$W_{k,j}^\dual={\rm Sym}^k\otimes{\det}^{-j}$$ de $\GG(\Q)$,
o\`u ${\rm Sym}^k$ est la puissance sym\'etrique de la repr\'esentation standard de dimension $2$
de $\GG(\Q)$, \index{W@\WWW}et $W_{k,j}$ sa duale.  
La repr\'esentation ${\rm Sym}^1$ est donc la repr\'esentation standard: i.e. c'est l'espace
$\Q e_1^\dual\oplus\Q e_2^\dual$ \index{e@\eee}muni de \index{action@\actions}l'action
$$\matrice{a}{b}{c}{d} * e_1^\dual=a e_1^\dual+c e_2^\dual,\quad
\matrice{a}{b}{c}{d} * e_2^\dual=b e_1^\dual+d e_2^\dual.$$
La repr\'esentation $W_{k,j}^\dual$ admet comme base les $(e_1^\dual)^i(e_2^\dual)^{k-i}(e_1^\dual\wedge e_2^\dual)^{-j}$
(avec action \'evidente de $\GG(\Q)$), pour $0\leq i\leq k$.
Si $e_1,e_2$ est \index{e@\eee}la base de $({\rm Sym}^1)^\dual$ duale de $e_1^\dual,e_2^\dual$, alors
$$\matrice{a}{b}{c}{d} * e_1=\frac{de_1-be_2}{ad-bc},\quad
\matrice{a}{b}{c}{d} * e_2=\frac{-ce_1+ae_2}{ad-bc},\quad 
\matrice{a}{b}{c}{d} * (e_1\wedge e_2)=\frac{e_1\wedge e_2}{ad-bc}$$
et $W_{k,j}$ admet comme base les $(e_1)^i(e_2)^{k-i}(e_1\wedge e_2)^{-j}$,
pour $0\leq i\leq k$.

Les $W_{k,j}$, pour $k\in\N$ et $j\in\Z$, sont toutes les repr\'esentations
alg\'ebriques irr\'eductibles de $\GG(\Q)$ et on a
$$W_{k,j}^\dual\cong W_{k,k-j}.$$

Si $L$ est un corps de caract\'eristique $0$, notons $W_{k,j}^\dual(L)$ la repr\'esentation
alg\'ebrique de $\GG(L)$
$$W_{k,j}^\dual(L)=L\otimes_\Q W_{k,j}^\dual,$$
et \index{W@\WWW}posons
$$W_{\rm tot}(L)=\oplus_{k,j}W_{k,j}\otimes W_{k,j}^\dual(L).$$
On munit $W_{\rm tot}(L)$ d'une action de $\GG(\Q)\times \GG(L)$
en faisant agir $(h_1,h_2)\in \GG(\Q)\times \GG(L)$
par $h_1$ sur $W_{k,j}$ et $h_2$ sur $W_{k,j}^\dual(L)$.

Soit ${\rm Alg}(\GG,L)$ \index{alg@\rmalg}l'espace des fonctions alg\'ebriques
sur $\GG(L)$, \`a valeurs dans $L$ (i.e. l'espace des fonctions polynomiales, \`a coefficients
dans $L$, en $a,b,c,d,(ad-bc)^{-1}$, si $g=\matrice{a}{b}{c}{d}\in \GG(L)$).
On fait agir $\GG(\Q)\times \GG(L)$ sur ${\rm Alg}(\GG,L)$ par
$$(h_1,h_2)\cdot\phi(g)=\phi(h_1^{-1}gh_2).$$

Si $v\in W_{k,j}$ et $\check v\in W_{k,j}^\dual(L)$, la fonction 
$$g\mapsto\phi_{\check v,v}(g)
=\langle g\cdot\check v,v\rangle$$
est un \'el\'ement de ${\rm Alg}(\GG,L)$, et l'application 
$\check v\otimes v\mapsto \phi_{\check v,v}$, 
de $W_{k,j}^\dual(L)\otimes W_{k,j}$ dans ${\rm Alg}(\GG,L)$, 
est $\GG(\Q)\times \GG(L)$-\'equivariante car $\langle h_1^{-1}gh_2\cdot\check v,v\rangle
=\langle gh_2\cdot\check v,h_1\cdot v\rangle$.  
Ceci fournit un isomorphisme $\GG(\Q)\times \GG(L)$-\'equivariant
$$W_{\rm tot}(L)\cong {\rm Alg}(\GG,L).$$

\subsubsection{Fonctions localement alg\'ebriques}\label{fma5}
Si $\Lambda$ est un anneau,
\index{LC@\rml}soit ${\rm LC}(\GG(\A),\Lambda)$ l'espace 
des $\phi:\GG(\A)\to \Lambda$ localement constantes.
Notons qu'une telle fonction est constante sur les classes modulo~$\GG(\R)_+$ puisque
ce groupe est connexe.

L'espace ${\rm LC}(\GG(\A),L)$ est muni \index{action@\actions}d'actions de $\GG(\Q)$ 
et de $\GG(\A)$ donn\'ees par:
$$(\gamma*\phi)(x)=\phi(\gamma^{-1}x),\ {\text{si $\gamma\in \GG(\Q)$}},\quad
(g\star\phi)(x)=\phi(xg),\ {\text{si $g\in \GG(\A)$}}.$$
Ces deux actions commutent (et donc d\'efinissent une action
de $\GG(\Q)\times \GG(\A)$).

Si $L$ est un corps de caract\'eristique~$0$,
cela munit \index{LC@\rml}l'espace $${\rm LP}(\GG(\A),L)={\rm LC}(\GG(\A),L)\otimes_L {\rm Alg}(\GG,L)$$
des fonctions localement alg\'ebriques sur $\GG(\A)$
d'actions de $\GG(\Q)$ (action diagonale des actions de $\GG(\Q)$)
et de $\GG(\A)\times \GG(L)$.

\subsubsection{Fonctions continues et mesures}\label{fma6.1}
Si $\Lambda$ est un anneau topologique,
on \index{C@\calC}note ${\cal C}(\GG(\A),\Lambda)$ l'espace des fonctions continues
sur $\GG(\A)$ \`a valeurs dans $\Lambda$.  Si $\Lambda$ est totalement discontinu, une telle fonction
est constante sur les classes modulo $\GG(\R)_+$ puisque ce groupe est connexe.
On munit ${\cal C}(\GG(\A),\Lambda)$ des actions de $\GG(\Q)$
et de $\GG(\A)$ comme ci-dessus:
$$(\gamma*\phi)(x)=\phi(\gamma^{-1}x),\ {\text{si $\gamma\in \GG(\Q)$}},\quad
(g\star\phi)(x)=\phi(xg),\ {\text{si $g\in \GG(\A)$}}.$$
Ces actions commutent.

\vskip.2cm
Si $\Lambda$ est totalement discontinu,
on note
${\cal C}_c(\GG(\A),\Lambda)$ \index{C@\calC}l'espace des fonctions continues \`a support compact modulo $\GG(\R)_+$,
et ${\rm Mes}(\GG(\A),\Lambda)$ \index{mesures@\mesure}son $\Lambda$-dual topologique 
(l'espace des {\it mesures sur $\GG(\A)$
\`a valeurs dans $\Lambda$}). On munit ${\cal C}_c(\GG(\A),\Lambda)$ des actions de $\GG(\Q)$
et de $\GG(\A)$ sur ${\cal C}(\GG(\A),\Lambda)$, et ${\rm Mes}(\GG(\A),\Lambda)$ des actions qui s'en 
d\'eduisent par dualit\'e.
\subsubsection{Injection dans les fonctions continues}\label{fma6}
Si $v$ est une place de $\Q$ et si $L$ est une extension finie
de $\Q_v$, 
on dispose d'un morphisme naturel
$\GG(\A)\to \GG(\Q_v)\hookrightarrow \GG(L)$ et on peut faire agir $\GG(\A)$ sur
${\rm LP}(\GG(\A),L)$ via le morphisme diagonal $\GG(\A)\to \GG(\A)\times \GG(L)$
et l'action de $\GG(\A)\times \GG(L)$ sur ${\rm LP}(\GG(\A),L)$.
Par ailleurs, l'application $\phi\otimes P\mapsto \phi P$ induit une injection
$${\rm LP}(\GG(\A),L)\hookrightarrow{\cal C}(\GG(\A),L)\,,$$ 
qui est $\GG(\Q)\times \GG(\A)$-\'equivariante
si on munit ${\rm LP}(\GG(\A),L)$ de l'action d\'efinie par ce qui pr\'ec\`ede.

\Subsection{La cohomologie compl\'et\'ee}\label{cup4}
\subsubsection{Induction de $\Gamma(1)$ \`a $\GG(\Q)$}
On note $\Gamma$ le groupe $\Gamma(1)={\bf SL}_2(\Z)$.
Soit $L$ une extension finie de $\Q_p$.
\begin{lemm} \label{cup5}
On a un isomorphisme de $\GG(\Q)$-modules:
$${\cal C}(\GG(\A),L)={\rm Ind}_\Gamma^{\GG(\Q)}{\cal C}(\GG(\cZ),L).$$
\end{lemm}
\begin{proof}
Si $\phi\in {\cal C}(\GG(\A),L)$, alors $\phi$ est constante sur les classes
modulo $\GG(\R)_+$. Si $\gamma\in \GG(\Q)$, soit $\phi_\gamma\in {\cal C}(\GG(\cZ),L)$
d\'efinie par $\phi_\gamma(\kappa)=\phi(\gamma^{-1} x_\infty\kappa)$ pour n'importe quel
choix de $x_\infty\in \GG(\R)_+$ (le r\'esultat ne d\'epend pas de ce choix d'apr\`es ce qui pr\'ec\`ede).

Si $\alpha\in\Gamma$, alors
\begin{align*}
\phi_{\alpha\gamma}(\kappa)=\phi(\gamma^{-1}\alpha^{-1}x_\infty\kappa)&=
\phi(\gamma^{-1}(\alpha_\infty^{-1}x_\infty)((\alpha^{]\infty[})^{-1}\kappa))\\ 
&= \phi_\gamma((\alpha^{]\infty[})^{-1}\kappa)=(\alpha*\phi_\gamma)(\kappa),
\end{align*}
ce qui prouve que $(\phi_\gamma)_{\gamma\in \GG(\Q)}\in {\rm Ind}_\Gamma^{\GG(\Q)}{\cal C}(\GG(\cZ))$.

R\'eciproquement, si $(\phi_\gamma)_{\gamma\in \GG(\Q)}\in {\rm Ind}_\Gamma^{\GG(\Q)}{\cal C}(\GG(\cZ))$,
on d\'efinit $\phi\in {\cal C}(\GG(\A),L)$, en posant $\phi(\gamma^{-1}x_\infty\kappa)=
\phi_\gamma(\kappa)$: tout \'el\'ement $\GG(\A)$ peut s'\'ecrire sous la forme
$\gamma^{-1} x_\infty \kappa$, avec $\gamma\in \GG(\Q)$,
$x_\infty\in \GG(\R)_+$ et $\kappa\in \GG(\cZ)$, et une telle \'ecriture est unique
\`a changements simultan\'es $\gamma\mapsto \alpha\gamma$, $x_\infty\mapsto \alpha_\infty x_\infty$
et $\kappa\mapsto \alpha^{]\infty[}\kappa$, avec $\alpha\in \Gamma$; la condition
$\phi_{\alpha\gamma}=\alpha*\phi_\gamma$ est exactement la condition qu'il faut pour
que $\phi$ soit bien d\'efinie.
\end{proof}

\begin{rema}\phantomsection\label{cup6}
(i) On a ${\rm LC}(\GG(\cZ),L)=\varinjlim_N {\cal C}(\GG(\Z/N),L)$, et la preuve ci-dessus permet
de montrer que
$${\rm LC}(\GG(\A),L)=\varinjlim_N {\rm Ind}_\Gamma^{\GG(\Q)}{\cal C}(\GG(\Z/N),L).$$

(ii) Pour les m\^emes raisons, on a:
$${\rm Mes}(\GG(\A),L)={\rm Ind}_\Gamma^{\GG(\Q)} {\rm Mes}(\GG(\cZ),L).$$
\end{rema}

\subsubsection{Descente de $\GG(\Q)$ \`a $\Gamma(1)$}
Compte-tenu du lemme~\ref{cup5} et de la rem.~\ref{cup6},
le lemme de Shapiro nous donne:
\begin{prop}\phantomsection\label{cup6.5}
Si $X={\cal C}$, ${\rm Mes}$, ${\rm LC}$, ${\rm LP}$, on a des
isomorphismes naturels:
$$H^1(\GG(\Q),X(\GG(\A),L))\cong H^1(\Gamma, X(\GG(\cZ),L)).$$
\end{prop}

\begin{rema}\phantomsection\label{cup7.1}
{\rm (i)} On a des isomorphismes:

$\bullet$ $H^1(\Gamma, {\cal C}(\GG(\cZ),L))=L\otimes_{\O_L}H^1(\Gamma, {\cal C}(\GG(\cZ),\O_L))$.

$\bullet$ $H^1(\Gamma, {\cal C}(\GG(\cZ),\O_L))=\varprojlim_n H^1(\Gamma,{\cal C}(\GG(\cZ),\O_L/p^n))$.

$\bullet$
$H^1(\Gamma,{\cal C}(\GG(\cZ),\O_L/p^n))
=\varinjlim_N H^1(\Gamma,{\cal C}(\GG(\Z/N),\O_L/p^n))$.

$\bullet$
$H^1(\Gamma,{\cal C}(\GG(\Z/N),\O_L/p^n))=H^1(Y(N)_\C,\O_L/p^n)$, o\`u $Y(N)$ est la courbe modulaire
de niveau $N$ du chap.\,\ref{como1}.

On en d\'eduit, en utilisant la prop.\,\ref{cup6.5},
 que $H^1(\GG(\Q),{\cal C}(\GG(\A),L))$ est la cohomologie compl\'et\'ee
de la tour des courbes modulaires de tous niveaux.
De m\^eme $H^1_c(\GG(\Q),{\cal C}(\GG(\A),L))$ est la cohomologie \`a support compact compl\'et\'ee
de la tour des courbes modulaires de tous niveaux.

{\rm (ii)} On a une suite exacte de $\GG(\A)$-modules (o\`u ${\cal C}(-)={\cal C}(-,L)$ ou ${\cal C}(-,\O_L)$)
$$0\to {\cal C}(\cZ^\dual)\to{\rm Ind}_{\BB(\A)}^{\GG(\A)}{\cal C}(\cZ^\dual\times\cZ^\dual)\to
H^1_c(\GG(\Q),{\cal C}(\GG(\A)))\to H^1(\GG(\Q),{\cal C}(\GG(\A)))\to 0$$
o\`u $\cZ^\dual$ est vu comme le quotient
$\A^\dual/\R_+^\dual\times\Q^\dual$ et $\GG(\A)$ agit \`a travers le d\'eterminant sur ${\cal C}(\cZ^\dual)$
et $\BB(\A)$ \`a travers le tore diagonal sur ${\cal C}(\cZ^\dual\times\cZ^\dual)$.
Il en r\'esulte que la diff\'erence entre cohomologie compl\'et\'ee et cohomologie compl\'et\'ee
\`a support compact est tr\`es petite; en particulier, les deux espaces deviennent isomorphes
quand on localise en un id\'eal non-eisenstein.

{\rm (iii)} L'isomorphisme $H^1(Y(N)_\C,\O_L/p^n)\cong H^1_{\eet}(Y(N)_{\Qbar},\O_L/p^n)$
(et de m\^eme pour la cohomologie \`a support compact)
permet de munir tous les groupes ci-dessus d'une action de $G_\Q$.
Une autre mani\`ere de le faire est d'utiliser le fait que
la cohomologie de $\Gamma$ \`a valeurs dans des groupes comme ci-dessus
est aussi celle de son compl\'et\'e profini $\widehat\Gamma$ car $\Gamma$ contient
un sous-groupe d'indice fini qui est un groupe
libre. Or on dispose d'une extension
non triviale $\Pi_\Q$ de $G_\Q$ par $\widehat\Gamma$ (cf.~\no\ref{s23}), ce qui munit
la cohomologie de $\widehat\Gamma$ d'une action de $G_\Q$.
Les deux actions pr\'ec\'edentes co\"{\i}ncident car $\Pi_\Q$ est le groupe fondamental
arithm\'etique de $Y(1)$ vu comme champs alg\'ebrique, et $\widehat\Gamma$ en est
le groupe fondamental g\'eom\'etrique (et sa cohomologie co\"{\i}ncide avec 
la cohomologie \'etale puisque les courbes sont des $K(\pi,1)$).
\end{rema}

\begin{rema}\phantomsection\label{cup7.2}
(i) Plut\^ot que la cohomologie compl\'et\'ee de la tour de tous niveaux, Emerton fixe le niveau
$N$ hors de $p$, compl\`ete le long de la tour des courbes de niveau $Np^k$, pour $k\in\N$,
et prend la limite
inductive sur $N$. La cohomologie compl\'et\'ee $\widehat H^1(N)$
le long de la tour des courbes de niveau $Np^k$
est $$\widehat H^1(N) = H^1(\Gamma,{\cal C}(\GG(\Z_p\times(\Z/N)),L)).$$ 
Il s'ensuit, via le lemme de Shapiro,
que ce que consid\`ere Emerton est 
$$H^1(\Gamma, {\varinjlim}_N {\cal C}(\GG(\Z_p\times(\Z/N)),L))
\cong H^1(\GG(\Q),{\cal C}^{(p)}(\GG(\A),L)),$$ 
o\`u ${\cal C}^{(p)}(\GG(\A),L)\subset {\cal C}(\GG(\A),L)$ est le sous-espace des fonctions
localement lisses
pour~$\GG(\cZ^{]p[})$.

(ii) Si $M$ est un $\Gamma$-module, il r\'esulte de la description 
du \S\,\ref{cup3} ci-dessous, que les groupes
$H^1_c(\Gamma,M)$ et $H^1(\Gamma,M)$ sont des sous-quotients de $M$.
Comme le $\GG(\Z_p)$-module ${\cal C}(\GG(\Z_p\times(\Z/N)),L)$ est une repr\'esentation admissible
de $\GG(\Z_p)$, il en est de m\^eme de $\widehat H^1(N)$ et de sa variante \`a support compact
(cas particulier de~\cite[th.\,2.1.5]{Em06}).

(iii)
L'injection
$\GG(\A)$-\'equivariante
$ {\rm LP}(\GG(\A),L)\hookrightarrow {\cal C}(\GG(\A),L)$ 
fournit, d'apr\`es Emerton~\cite[(4.3.4)]{Em06} et \cite[th.\,7.4.2]{Em06b}, des identifications
\begin{align}\label{VV1}
H^1(\GG(\Q),{\cal C}(\GG(\A),L))^{\rm alg}&=
H^1(\GG(\Q),{\rm LP}(\GG(\A),L))\\
&=\oplus_{k,j} H^1(\GG(\Q),{\rm LC}(\GG(\A),L)\otimes_\Q W_{k,j})\otimes W_{k,j}^\dual(L)\notag
\end{align}
o\`u $^{\rm alg}$ d\'esigne l'espace des vecteurs localement alg\'ebriques.
(Dans la d\'ecomposition ci-dessus, l'action de $\GG(\A)$ 
sur $H^1(\GG(\Q),{\rm LC}(\GG(\A),\Q_p)\otimes_\Q W_{k,j})$ est lisse, et alg\'ebrique
sur $W_{k,j}^\dual(L)$ sur lequel $\GG(\A)$ agit \`a travers $\GG(\Q_p)$.)
\end{rema}

\begin{rema}\phantomsection\label{cup7.3}
Soit $S\subset{\cal P}\cup\{\infty\}$ fini.
Si $\infty\in S$ (resp.~$\infty\notin S$),
 tout \'el\'ement
de $\GG(\Q_S)$ s'\'ecrit sous la forme $\gamma^{-1}x_\infty\kappa$ (resp.~$\gamma^{-1}\kappa$),
avec $\gamma\in \GG(\Z[\frac{1}{S}])$, $x_\infty\in \GG(\R)_+$ et $\kappa\in \GG(\Z_S)$,
et une telle \'ecriture est unique \`a changements
simultan\'es $\gamma\mapsto \alpha\gamma$, $x_\infty\mapsto \alpha_\infty x_\infty$
et $\kappa\mapsto \alpha_{S\moins\{\infty\}}\kappa$, avec $\alpha\in\Gamma$ (resp.~$\alpha\in \GG(\Z)$).
On en d\'eduit des isomorphismes de $\GG(\Z[\frac{1}{S}])$-modules, si $X={\cal C}$, ${\rm LC}$, ${\rm LP}$,
${\rm Mes}$:
$$ X(\GG(\Q_S),L)=\begin{cases}
{\rm Ind}_{\Gamma}^{\GG(\Z[\frac{1}{S}])}X(\GG(\Z_S),L) &{\text{si $\infty\in S$,}}\\
{\rm Ind}_{\GG(\Z)}^{\GG(\Z[\frac{1}{S}])}X(\GG(\Z_S),L) &{\text{si $\infty\notin S$}}
\end{cases}$$
En particulier, si $\infty\in S$, alors
$H^1(\GG(\Z[\frac{1}{S}]),{\cal C}(\GG(\Q_S),L))$ est la cohomologie compl\'et\'ee
de la tour des courbes modulaires de niveaux dont les facteurs premiers appartiennent \`a $S$,
et $H^1_c(\GG(\Z[\frac{1}{S}]),{\cal C}(\GG(\Q_S),L))$ en est la cohomologie \`a support
compact compl\'et\'ee.
\end{rema}

\subsubsection{Descente de $\GG(\Q)$ \`a $\Gamma(1)$ pour la cohomologie \`a support compact}
Le lemme de Shapiro n'a, a priori, pas de raison
d'\^etre vrai pour la cohomologie \`a support compact (on a quand m\^eme 
une fl\`eche naturelle $H^i_c(\GG(\Q),{\rm Ind}_\Gamma^{\GG(\Q)}M)\to H^i_c(\Gamma,M)$
obtenue en \'evaluant les fonctions $\phi:\GG(\Q)\to M$ en $1$).
Mais on a le r\'esultat suivant:
\begin{prop}\phantomsection\label{cup7}
L'application naturelle induit un isomorphisme
$$H^1_c(\GG(\Q),{\cal C}(\GG(\A),L))\overset{\sim}{\to} H^1_c(\Gamma,{\cal C}(\GG(\cZ),L))$$
\end{prop}
\begin{proof}
On a un diagramme commutatif \`a lignes exactes, o\`u 
$$B=\BB(\Q),\hskip2mm, G=\GG(\Q),\hskip2mm U=\BB(\Q)\cap\Gamma, \hskip2mm
{\cal C}_\A={\cal C}(\GG(\A),L),\hskip2mm {\cal C}_\cZ={\cal C}(\GG(\cZ),L)$$
et les deux isomorphismes verticaux r\'esultent du lemme de Shapiro:
$$\xymatrix@R=.4cm@C=.5cm{
H^0(G,{\cal C}_{\A})\ar[d]^\wr\ar[r]& H^0(B,{\cal C}_{\A})\ar[d]\ar[r]
&H^1_c(G,{\cal C}_{\A})\ar[d]\ar[r] 
&H^1(G,{\cal C}_{\A})\ar[d]^\wr\ar[r]& H^1(B,{\cal C}_{\A})\ar[d]\\
H^0(\Gamma,{\cal C}_{\cZ})\ar[r]& H^0(U,{\cal C}_{\cZ})\ar[r]
&H^1_c(\Gamma,{\cal C}_{\cZ})\ar[r] 
&H^1(\Gamma,{\cal C}_{\cZ})\ar[r]& H^1(U,{\cal C}_{\cZ})}$$
Pour conclure, gr\^ace au lemme des 5, il suffit de v\'erifier
que les fl\`eches $H^i(B,{\cal C}_\A)\to H^i(U,{\cal C}_{\cZ})$, pour $i=0,1$,
sont des isomorphismes.

On a $\GG(\A)=\BB(\Q)\cdot(\GG(\cZ)\GG(\R)_+)$ et $\BB(\Q)\cap(\GG(\cZ)\GG(\R)_+)=U$.
Il en r\'esulte que ${\cal C}_\A\cong{\rm Ind}_U^B {\cal C}_\cZ$,
et on conclut
en utilisant le lemme de Shapiro.
\end{proof}

\begin{rema}\phantomsection\label{cup8.5}
{\rm (i)} On montre de m\^eme que
$H^1_c(\GG(\Z[\frac{1}{S}]),{\cal C}(\GG(\Q_S),L))\to H^1_c(\Gamma,{\cal C}(\GG(\Z_S),L))$
est un isomorphisme.

{\rm (ii)} On a les m\^emes r\'esultats en rempla\c{c}ant ${\cal C}$ par ${\cal C}^{(p)}$.
\end{rema}

\Subsection{Dualit\'e}\label{cup3}
Soient $\Gamma={\rm SL}_2(\Z)$ et $\oGamma={\rm SL}_2(\Z)/\{\pm I\}$, 
o\`u $I=\matrice{1}{0}{0}{1}$.
Si $M$ est un $\Gamma$-module, on a $H^i(\Gamma,M)=H^i(\oGamma,M^{\{\pm I\}})$ \`a $2$-torsion pr\`es.
Le groupe $\oGamma$ est engendr\'e par $S=\matrice{0}{1}{-1}{0}$ et $U=\matrice{0}{1}{-1}{1}$,
les seules relations \'etant $S^2=1$ et $U^3=1$.  De plus, si $\overline B\subset\oGamma$ est l'image
de $B=\matrice{1}{\Z}{0}{1}$, alors $SU$ est un g\'en\'erateur de $\overline B$
(c'est l'image de $\matrice{1}{-1}{0}{1}$).

Si $M$ est un $\Z_p[\oGamma]$-module, on note $M^\dual$ son dual (suivant les cas, cela
peut \^etre le dual de Pontryagin, le $\Z_p$-dual, ou le $\Q_p$-dual (ou $L$-dual) topologique).
Le groupe $H^1_c(\oGamma,M^\dual)$ est le groupe des $1$-cocycles $\sigma\mapsto c^\dual_\sigma$
sur $\oGamma$, \`a valeurs dans $M^\dual$, qui sont identiquement nuls sur $\overline B$.
Un tel cocycle est enti\`erement d\'etermin\'e par $c^\dual_U$ et $c^\dual_S$, et on a
$(1+S)c^\dual_S=0$ et $(1+U+U^2)c^\dual_U=0$ (\`a cause des relations $S^2=1$ et $U^3=1$)
et la relation $c^\dual_{SU}=0$ impose en plus que $Sc^\dual_U+c^\dual_S=0$, ou encore
$Sc^\dual_U-Sc^\dual_S=0$, et donc $c^\dual_U=c^\dual_S$.
Autrement dit, on a une identification
$$H^1_c(\oGamma,M^\dual)=(M^\dual)^{1+U+U^2=0}\cap (M^\dual)^{1+S=0}.$$

Maintenant, un $1$-cocycle $\sigma\mapsto c_\sigma$ sur $\oGamma$, \`a valeurs dans $M$,
est uniquement d\'etermin\'e par $c_U$ et $c_S$ soumis aux relations
$(1+S)c_S=0$ et $(1+U+U^2)c_U=0$. La classe de cohomologie de ce cocycle
ne change pas si on remplace $(c_U,c_S)$ par $(c_U-(U-1)c,c_S-(S-1)c)$, avec
$c\in M$.  Cela permet, \`a $2$-torsion pr\`es, de supposer que $c_S=0$, et alors
la classe de cohomologie est d\'etermin\'ee par $c_U$ \`a addition pr\`es de
$(U-1)c$, avec $c\in M^{S=1}$. On a donc un isomorphisme
\begin{equation}\label{cogamma}
H^1(\oGamma, M)\cong M^{1+U+U^2=0}/(U-1)M^{S=1}
\end{equation}

Notons $\langle\ ,\ \rangle:M^\dual\times M\to \Lambda$, 
o\`u $\Lambda=\Z_p$ ou $\Q_p/\Z_p$ ou $\Q_p,L$..., l'accouplement
naturel.  
On d\'efinit un accouplement
$$H^1_c(\oGamma,M^\dual)\times H^1(\oGamma, M)\to\Lambda$$
en utilisant les identification pr\'ec\'edentes:
si
$$c^\dual\in\big((M^\dual)^{1+U+U^2=0}\cap (M^\dual)^{1+S=0}\big),\quad
c\in\big(M^{1+U+U^2=0}/(U-1)M^{S=1}\big),$$
on pose
$$c^\dual\cup c=2\langle c^\dual,(1-U^2)c\rangle.$$
\begin{prop}\phantomsection\label{ES5}
L'accouplement $\cup$ induit un isomorphisme
 (\`a $6$-torsion pr\`es)
$$H^1_c(\oGamma,M^\dual)\cong H^1(\oGamma,M)^\dual$$
\end{prop}
\begin{proof}
Les espaces $(M^\dual)^{1+U+U^2=0}$ et $M^{1+U+U^2=0}$ sont en dualit\'e pour l'accouplement
naturel $M^\dual\times M\to\Lambda$, et $c\mapsto (1-U^2)c$ est un isomorphisme
de $M^{1+U+U^2=0}$ \`a $3$-torsion pr\`es (noyau et conoyau sont tu\'es par $3$).
De plus, l'orthogonal de $(U-1)M^{S=1}$ (dans $(M^\dual)^{1+U+U^2=0}$) pour l'accouplement
$\langle c^\dual,(1-U^2)c\rangle$ est l'ensemble des $c^\dual$ tels que
$\langle c^\dual,(1-U^2)(U-1)c\rangle=0$ pour tout $c\in M^{S=1}$.
Or $(1-U^2)(U-1)=-(U^2+U+1)-3$ et l'adjoint de $U^2+U+1$ est $U^{-2}+U^{-1}+1=U+U^2+1$
qui tue $c^\dual$.  L'orthogonal de $(U-1)M^{S=1}$ est donc l'ensemble
des $c^\dual$ v\'erifiant
$\langle c^\dual,3c\rangle=0$ pour tout $c\in M^{S=1}$; c'est donc
$(M^\dual)^{S+1=0}$ (\`a $6$-torsion pr\`es).
\end{proof}
\begin{rema}\phantomsection\label{ES6}
(i)
L'accouplement ci-dessus est l'accouplement naturel \`a valeurs
dans $H^2_c(\oGamma,\Lambda)\cong\Lambda$.

(ii) On en d\'eduit une dualit\'e \`a $12$-torsion pr\`es 
$H^1_c(\Gamma(1),M^\dual)\times H^1(\Gamma(1), M)\to\Lambda$.
\end{rema}

\begin{coro}\phantomsection\label{ES6.4}
Les groupes $H^1_c(\GG(\Q),{\cal C}(\GG(\A),L))$ et $H^1(\GG(\Q),{\rm Mes}(\GG(\A),L))$
sont en dualit\'e.
\end{coro}
\begin{proof}
D'apr\`es les prop.~\ref{cup6.5} et~\ref{cup7},
on a des isomorphismes:
\begin{align*}
H^1_c(\GG(\Q),{\cal C}(\GG(\A),L))&\cong H^1_c(\Gamma(1),{\cal C}(\GG(\cZ),L))\\
H^1(\GG(\Q),{\rm Mes}(\GG(\A),L))&\cong H^1(\Gamma(1),{\rm Mes}(\GG(\cZ),L))
\end{align*}
La prop.\,\ref{ES5} (cf.~rem.\,\ref{ES6}) implique donc une dualit\'e dans un sens.
Pour en d\'eduire celle dans l'autre sens, il suffit de v\'erifier que
$H^1(\Gamma,{\rm Mes}(\GG(\cZ),\O_L))$ est sans $p$-torsion, ce qui
r\'esulte de ce que $H^0(\Gamma,{\rm Mes}(\GG(\cZ),k_L))=0$ car une mesure invariante par
$\Gamma$ l'est aussi par $\matrice{1}{\Z_p}{0}{1}\subset\matrice{1}{\cZ}{0}{1}$ par continuit\'e
de l'action de $\matrice{1}{\cZ}{0}{1}$,
et donc est nulle (pas de mesure de Haar en $p$-adique).
\end{proof}
\begin{rema}\phantomsection\label{dual1}
On montre de m\^eme que les groupes
$H^1_c(\GG(\Z[\frac{1}{S}]),{\cal C}(\GG(\Q_S),L))$ et $H^1(\GG(\Z[\frac{1}{S}]),{\rm Mes}(\GG(\Q_S),L))$
sont en dualit\'e.
\end{rema}

{
\Subsection{Action du centre}\label{cen1}
Si $a\in \A^\dual$, soit $z(a)=\matrice{a}{0}{0}{a}$. 
\begin{lemm}\phantomsection\label{cen2}
Si $a\in\Q^\dual$ plong\'e diagonalement
dans $\A^\dual$, alors $z(a)$ agit trivialement sur $H^1(\GG(\Q),{\cal C}(\GG(\A),\O_L))$
et  $H^1_c(\GG(\Q),{\cal C}(\GG(\A),\O_L))$.
\end{lemm}
\begin{proof}
Soit $\phi\in H^1_c(\GG(\Q),{\cal C}(\GG(\A),\O_L))$, et soit
 $\gamma\mapsto \phi_\gamma$ un $1$-cocycle repr\'esentant $\phi$. Alors $z(a)\star \phi$
est repr\'esent\'e par le cocycle $\gamma\mapsto z(a)\star \phi_\gamma$. Comme $z(a)$
est dans le centre de $\GG(\A)$ et aussi \'el\'ement de $\GG(\Q)$, on a aussi, en posant $\alpha:=z(a)^{-1}$,
$$z(a)\star\phi_\gamma=\alpha*\phi_\gamma=\phi_{\alpha\gamma}-\phi_\alpha=\phi_{\gamma\alpha}-\phi_\alpha=
(\gamma-1)*\phi_\alpha+\phi_\gamma$$
et donc $\gamma\mapsto z(a)\star \phi_\gamma$ diff\`ere de $\gamma\mapsto \phi_\gamma$ par un cobord.
Cela prouve le r\'esultat pour~$H^1$.  

Pour $H^1_c$, on a en plus la condition 
$\phi_\gamma=(\gamma-1)*\phi_B$, si $\gamma\in\BB(\Q)$, mais 
$((\gamma\mapsto z(a)\star \phi_\gamma),z(a)\star\phi_B)-
((\gamma\mapsto \phi_\gamma),\phi_B)$ est le bord de $\phi_\alpha$
plus $((\gamma\mapsto 0),(\alpha-1)*\phi_B-\phi_\alpha)$, et ce dernier terme est nul
puisque $\alpha\in\BB(\Q)$.
\end{proof}

\begin{coro}\phantomsection\label{cen3}
Le centre de $\GG(\A)$ {\rm (identifi\'e \`a $\A^\dual$)}
 agit \`a travers $\A^\dual/\R_+^\dual\Q^\dual\cong\cZ^\dual$
sur $H^1(\GG(\Q),{\cal C}(\GG(\A),\O_L))$
et $H^1_c(\GG(\Q),{\cal C}(\GG(\A),\O_L))$.
\end{coro}

}

{
\Subsection{Densit\'e des vecteurs alg\'ebriques}\label{cen4}
On rappelle que $S=\matrice{0}{1}{-1}{0}$ et $U=\matrice{0}{1}{-1}{1}$ sont les g\'en\'erateurs
habituels de $\Gamma(1)$.
Si $M$ est un $\Gamma(1)$-module, on a une suite exacte (\`a $2$-torsion pr\`es),
cf.~(\ref{cogamma})
$$0\to H^0(\Gamma(1),M)\to M^{S=1}\overset{U-1}{\longrightarrow}M^{1+U+U^2=0}\to
H^1(\Gamma(1),M)\to 0$$
et $M^{S=1}=(S-1)((-I)+1)M$ (\`a $2$-torsion pr\`es),
$M^{1+U+U^2=0}=(U-1)((-I)+1)M$ (\`a $6$-torsion pr\`es).
En particulier $H^1(\Gamma(1),M)$ est un quotient de $M$ (\`a $6$-torsion pr\`es).

Soit $N$ premier \`a $p$, 
et soit $\widehat H^1(N):=H^1(G_\Q,{\cal C}(\GG(\A)/\wGamma(Np^\infty),\O_L))$.
\begin{prop}\phantomsection\label{cen5}
{\rm (Emerton, \cite[prop.\,5.4.1]{Em08})}
Les vecteurs $\GG(\Z_p)$-alg\'ebriques sont denses dans $\widehat H^1(N)$.
\end{prop}
\begin{proof}
Cela r\'esulte de ce que,
si $M:={\cal C}(\GG(\Z_p)\times\GG(\Z/N),\O_L)$, alors
 $\widehat H^1(N)=H^1(\Gamma(1),M)$
est un quotient de
$M$ et de ce que 
les vecteurs $\GG(\Z_p)$-alg\'ebriques sont denses
dans ${\cal C}(\GG(\Z_p),\O_L)$ (Stone-Weierstrass, 
densit\'e des polyn\^omes dans les fonctions continues 
sur un compact).
\end{proof}

\begin{prop}\phantomsection\label{cen6}
{\rm (Emerton, \cite[prop.\,5.3.15]{Em08})}
Supposons ${\goth m}$
non-eisenstein.

{\rm (i)} Si $p\geq 5$, alors
$\widehat H^1(N)_{\goth m}$ est injectif comme $\GG(\Z_p)$-module.

{\rm (ii)} Si $K_p$ est un sous-groupe ouvert de $\GG(\Z_p)$ qui se surjecte sur $\Z_p^\dual$
et tel que $\Gamma(1)\cap K_p\wGamma(Np^\infty)$ est sans torsion,
alors
$\widehat H^1(N)_{\goth m}$ est injectif comme $K_p$-module.
\end{prop}
\begin{proof}
{\rm (i)}
Si $M={\cal C}(\GG(\Z_p)\times\GG(\Z/N),\O_L)$, alors $M^{1+U+U^2=0}$ et $M^{S=1}$
sont des facteurs directs de $M$ (\`a $6$-torsion pr\`es, d'o\`u l'hypoth\`ese $p\geq 5$), 
et leurs localis\'es en ${\goth m}$
aussi puisque $T_{\goth m}$ est un facteur direct de $T$.
Comme $M$ est injectif, il en est de m\^eme de $M^{1+U+U^2=0}_{\goth m}$ et $M^{S=1}_{\goth m}$.

Maintenant $H^0(\Gamma(1),M)_{\goth m}=0$ car ${\goth m}$
est non-eisenstein. On a donc une suite exacte
$0\to M^{S=1}_{\goth m}\to M^{1+U+U^2}_{\goth m}\to \widehat H^1(N)_{\goth m}\to 0$,
qui pr\'esente $\widehat H^1(N)_{\goth m}$ comme un quotient de modules injectifs.
Le r\'esultat s'en d\'eduit.

{\rm (ii)} L'hypoth\`ese implique que $\Gamma_U:=\Gamma(1)\cap K_p\wGamma(Np^\infty)$
est libre de type fini.  Si $M$ est un $\Gamma_U$-module et si $I$ est un ensemble
de g\'en\'erateurs de $\Gamma_U$, on a une suite exacte
$0\to H^0(\Gamma_U,M)\to M\to M^I\to H^1(\Gamma_U,M)\to 0$.
Par ailleurs, le lemme de Shapiro fournit un isomorphisme
$\widehat H^1(N)\cong H^1(\Gamma_U,{\cal C}(K_p\times\GG(\Z/N),\O_L))$.
On conclut comme ci-dessus, avec $M:={\cal C}(K_p\times\GG(\Z/N),\O_L)$. 
\end{proof}

}

\section{Formes modulaires ad\'eliques}\label{fma0}
Le but de ce chapitre est d'ad\'eliser la th\'eorie classique
des formes modulaires.

\Subsection{Formes modulaires classiques}\label{fma1}

\subsubsection{L'alg\`ebre des formes modulaires}\label{fma2}
Si $k\in\N$ et $j\in\Z$, on d\'efinit une \index{action@\actions}action \`a droite $(\gamma,f)\mapsto f_{|_{k,j}}\gamma$ de
$\GG(\R)_+$ sur les $f:{\cal H}^+\to \C$, par la formule
$$(f_{|_{k,j}}\gamma)(\tau)=\tfrac{(ad-bc)^{k-j}}{(c\tau+d)^{k}}
f\big(\tfrac{a\tau+b}{c\tau+d}\big),\quad{\text{si $\gamma=\matrice{a}{b}{c}{d}$.}}$$
Si $\Gamma$ est un sous-groupe d'indice fini de $\Gamma(1)={\bf SL}_2(\Z)$,
on \index{M1@\MMM}note $M^{\rm cl}_{k,j}(\Gamma,\C)$ (resp.~$M^{\rm par,\,cl}_{k,j}(\Gamma,\C)$), 
le $\C$-espace vectoriel {\it des formes
modulaires de poids~$(k,j)$ pour $\Gamma$}, c'est-\`a-dire, l'ensemble
des $f:{\cal H}^+\to\C$, holomorphes, \`a croissance lente
(resp.~\`a d\'ecroissance rapide) \`a
l'infini, v\'erifiant $f_{|_{k,j}}\gamma=f$ quel que soit $\gamma\in\Gamma$
(l'espace ne d\'epend pas de $j$ puisque
$ad-bc=1$, mais l'action de $\GG(\Q)_+$ en d\'epend; on note cet espace simplement
$M^{\rm cl}_k(\Gamma,\C)$ si on ne veut pas pr\'eciser l'action de $\GG(\Q)_+$).

Si $f\in M^{\rm cl}_k(\Gamma,\C)$, alors $f$ est p\'eriodique de p\'eriode~$N$
pour un certain entier $N\geq 1$, et $f$ est somme de sa s\'erie de Fourier
$$f(\tau)=\sum_{n=0}^{+\infty}a_{n/N}e^{2i\pi n\tau/N}=
\sum_{n=0}^{+\infty}a_{n/N}q^{n/N},\quad
{\text{avec $q=e^{2i\pi\tau}$}}.$$
La s\'erie $\sum_{n\in\Q_+}a_{n}q^{n}$ s'appelle
le {\it $q$-d\'eveloppement} de $f$.  Si $A$ est un sous-anneau
de $\C$, on note $M^{\rm cl}_k(\Gamma,A)$ le sous-$A$-module de $M^{\rm cl}_k(\Gamma,\C)$
des formes dont le $q$-d\'eveloppement est \`a coefficients dans $A$
et $M^{\rm cl}(\Gamma,A)$ la $A$-alg\`ebre des formes modulaires
pour $\Gamma$ \`a coefficients dans $A$,
somme directe des $M^{\rm cl}_k(\Gamma,A)$, pour $k\geq 0$.
Finalement, on note $M^{\rm cl}_k(A)$ (resp.~$M^{\rm cl}(A)$) la r\'eunion des
$M^{\rm cl}_k(\Gamma,A)$ (resp.~$M^{\rm cl}(\Gamma,A)$), o\`u
$\Gamma$ d\'ecrit l'ensemble des sous-groupes d'indice fini de
$\Gamma(1)$.  

\subsubsection{Les groupes $\Pi_\Q$, $\Pi'_\Q$, $\Pi_{\Q,S}$ et $\Pi'_{\Q,S}$}\label{s23}
Si $K$ est un corps de caract\'eristique~$0$, on pose
$M^{\rm cl}(\Gamma(1),K)=K\otimes_\Q M^{\rm cl}(\Gamma(1),\Q)$ et
$M^{\rm cl}(\overline K)=\overline K\otimes_{\Qbar}M^{\rm cl}(\Qbar)$.
On \index{PiK@\PIK}note $\Pi_K$ le groupe
des automorphismes de la $K$-alg\`ebre $M^{\rm cl}(\Kbar)$ au-dessus de $M^{\rm cl}(\Gamma(1),K)$;
c'est un groupe profini.  Chacun des $M^{\rm cl}_k(\overline K)$ est stable
par $\Pi_K$.

Si $K$ est alg\'ebriquement clos, alors $\Pi_K$ est le compl\'et\'e
profini de $\Gamma(1)$ (qui est beaucoup plus gros que
${\bf SL}_2(\cZ)$).  Dans le cas g\'en\'eral, on dispose
de la suite exacte
$$\xymatrix@C=.4cm{
1\ar[r] &\Pi_\Kbar\ar[r] & \Pi_K\ar[r] & G_K\ar[r] &1,}$$
scind\'ee, $ G_K$ agissant sur les coefficients du
$q$-d\'eveloppement des formes modulaires.

Par ailleurs, l'alg\`ebre $M^{\rm cl}(\Qbar)$ est stable sous l'action de
$\GG(\Q)_+$ d\'efinie par $f\to f_|\gamma$, si $f\in M^{\rm cl}(\Qbar)$
et $\gamma\in \GG(\Q)_+$, avec $f_|\gamma=f_{|_{k,0}}\gamma$ si $f\in M^{\rm cl}_k(\Qbar)$. 
 Il
en est donc de m\^eme de $M^{\rm cl}(\overline K)$ et on \index{PiK@\PIK}note
$\Pi'_K$ le sous-groupe des automorphismes de $M^{\rm cl}(\Kbar)$
engendr\'e par $\Pi_K$ et $\GG(\Q)_+$, et encore $f\to f_|\gamma$ 
l'action de $\gamma\in\Pi'_K$ sur $f\in M^{\rm cl}(\overline K)$.
Si $k\in\N$ et $j\in\Z$, on \index{action@\actions}note $f\mapsto f_{|_{k,j}}\gamma$ l'action
de $\gamma\in\Pi'_K$ sur $M^{\rm cl}_k(\overline K)\subset M^{\rm cl}(\overline K)$
co\"{\i}ncidant avec l'action d\'efinie au \no\ref{fma2} si $\gamma\in \GG(\Q)_+$.

Si $K$ est alg\'ebriquement clos, $\Pi'_K$ est le compl\'et\'e
de $\GG(\Q)$ pour la topologie de groupe dont une base de voisinage de~$1$
est constitu\'ee des sous-groupes d'indice fini de $\Gamma(1)$.
Dans le cas g\'en\'eral, on dispose
de la suite exacte, scind\'ee
$$\xymatrix@C=.4cm{
1\ar[r] &\Pi'_\Kbar\ar[r] & \Pi'_K\ar[r] & G_K\ar[r] &1.}$$

\subsubsection{Sous-groupes de congruence}\label{s23.5}
On \index{M2@\Mcong}note $M^{\rm cong}$ (resp.~$M^{\rm cong}_S$) la r\'eunion des
$M^{\rm cl}(\Gamma(N),\Q(\bmu_N))$, o\`u $N$ d\'ecrit les entiers~$\geq 1$
(resp.~les entiers~$\geq 1$ \`a support dans $S$).
La sous-alg\`ebre $M^{\rm cong}$ est stable par
$\Pi_\Q$ et $\Pi'_\Q$ qui agissent \`a travers
$\GG(\cZ)$ et $\GG(\A^{]\infty[})$ respectivement.
On a le diagramme commutatif de groupes suivant:
$$\xymatrix@R=.6cm{
1\ar[r] &\Pi_{\Kbar}\ar[d]\ar[r] &\Pi_K\ar[r]\ar[d]^{\rho_{\rm cong}}\ar[r]
& G_K\ar[d]^{\cy}\ar[r] &1\\
1\ar[r]&{\bf SL}_2(\cZ)\ar[r]&\GG(\cZ)\ar[r]^-{\det} &\cZ^\dual
\ar[r] &1.}$$
La section de $ G_\Q$ dans $\Pi_\Q$ d\'ecrite ci-dessus
induit une section de l'application d\'eterminant $\GG(\cZ)\to
\cZ^\dual$; c'est celle qui envoie $u$ sur la matrice $\matrice{u}{0}{0}{1}$:
l'action $f_{|_{k,j}}\gamma$ de $G(\A^{]\infty[})$ sur $M_k^{\rm cong}$
est celle obtenue par continuit\'e \`a partir de celle du \no\ref{fma2}
si $\det\gamma\in\Q_+^\dual$, et l'action de $\matrice{u}{0}{0}{1}$,
si $u\in\cZ^\dual$ est celle de $\sigma_u$ sur les coefficients du $q$-d\'eveloppement
(en accord avec la th\'eorie g\'eom\'etrique, cf.~prop.\,\ref{como5}).

Le noyau $H$ de l'application naturelle $\Pi_\Q'\to \GG(\A^{]\infty[})$ est
inclus dans $\Pi_\Q$ puisqu'il fixe $M^{\rm cl}(\Gamma(1),\Q)$. On a donc un diagramme commutatif
$$\xymatrix@R=.5cm{
1\ar[r] &H\ar@{=}[d]\ar[r] &\Pi_\Q\ar[r]\ar[d]\ar[r]^-{\rho_{\rm cong}}
&\GG(\cZ)\ar[d]\ar[r] &1\\
1\ar[r]&H\ar[r]&\Pi'_\Q\ar[r]^-{\rho_{\rm cong}}&\GG(\A^{]\infty[})
\ar[r] &1.}$$

Soit $M_S^{{\rm cong},+}$ la cl\^oture int\'egrale de $M^{\rm cl}(\Gamma(1),\Z[\frac{1}{S}])$
dans $M_S^{\rm cong}$ et soit $M_S^+$ l'extension \'etale maximale
de $M_S^{{\rm cong},+}$.  Alors $M_S^{{\rm cong},+}$ est contenue
dans $M^{\rm cl}(\Qbar)$ et stable par $\Pi_\Q$; 
le \index{PiK@\PIK}groupe $\Pi_{\Q,S}={\rm Aut}(M_S^+/M^{\rm cl}(\Gamma(1),\Z[\frac{1}{S}]))$
est un quotient de $\Pi_{\Q}$.  Par ailleurs, $M_S^{{\rm cong},+}$
est stable par $\GG(\Z[\frac{1}{S}])_+\subset \GG(\Q)_+$, et donc $M_S^+$ aussi.
On note $\Pi'_{\Q,S}$ le sous-groupe de
${\rm Aut}(M_S^+)$ engendr\'e par $\Pi_{\Q,S}$ et $\GG(\Z[\frac{1}{S}])_+$.
On a des diagrammes commutatifs:
$${\xymatrix@R=.5cm{
&\Pi_\Q\ar[r]\ar[d]\ar[r]
& G_\Q\ar[d]\\
&\Pi_{\Q,S}\ar[r] &G_{\Q,S}}}
\hskip1cm
{\xymatrix@R=.5cm{
1\ar[r] &H_S\ar@{=}[d]\ar[r] &\Pi_{\Q,S}\ar[r]\ar[d]\ar[r]^-{\rho_{\rm cong}}
&\GG(\Z_S)\ar[d]\ar[r] &1\\
1\ar[r]&H_S\ar[r]&\Pi'_{\Q,S}\ar[r]^-{\rho_{\rm cong}}&\GG(\Q_S)
\ar[r] &1.}}$$

\Subsection{Formes modulaires ad\'eliques quasi-holomorphes}\label{fma8}
\subsubsection{Fonctions harmoniques}\label{fma7}
On \index{A@\calA}note ${\cal A}$ l'anneau des fonctions $\phi:\GG(\A)\to\C$ v\'erifiant les propri\'et\'es
suivantes:

$\bullet$ $\phi$ se factorise \`a travers ${\cal H}\times \GG(\A^{]\infty[})$,

$\bullet$ si $g^{]\infty[}\in \GG(\A^{]\infty[})$, alors $\tau\mapsto\phi(\tau,g^{]\infty[})$
est harmonique, \`a croissance lente,

$\bullet$ 
il existe $K_\phi\subset \GG(\A^{]\infty[})$, ouvert, tel que
$\phi(g_\infty,g^{]\infty[}\kappa)=\phi(g_\infty,g^{]\infty[})$ si $\kappa\in K_\phi$.

\smallskip On note ${\cal A}^+$ le sous-espace de ${\cal A}$ des fonctions
holomorphes en $\tau$ et ${\cal A}^-$ celui des fonctions antiholomorphes.
Comme une fonction harmonique est somme d'une fonction holomorphe et d'une fonction
antiholomorphe et que les constantes sont les seules fonctions \`a la fois holomorphes et antiholomorphes
sur un domaine connexe,
on a une suite exacte
$$0\to {\rm LC}(\GG(\A),\C)\to{\cal A}^+\oplus{\cal A}^-\to{\cal A}\to 0.$$

\smallskip
On \index{A@\calA}note ${\cal A}_{\rm par}$ l'id\'eal de ${\cal A}$ des fonctions \`a d\'ecroissance rapide \`a l'infini
(le {\og par\fg} signifie {\og parabolique\fg}),
et ${\cal A}_{\rm par}^+$, ${\cal A}_{\rm par}^-$ ses intersections avec ${\cal A}^+$ et ${\cal A}^-$.
On a ${\cal A}_{\rm par}={\cal A}_{\rm par}^+\oplus {\cal A}_{\rm par}^-$ car les constantes ne sont pas \`a d\'ecroissance rapide.

\smallskip
On fait agir $\gamma\in \GG(\Q)$ et $g\in \GG(\A^{]\infty[})$ sur ${\cal A}$ et ${\cal A}_{\rm par}$ par les formules
$$(\gamma * \phi)(x)=\phi(\gamma^{-1}x)
\quad{\rm et}\quad
(g \star\phi)(x)=\phi(xg).$$
Les actions de $\GG(\Q)$ et $\GG(\A^{]\infty[})$ ainsi d\'efinies commutent.
L'action de $\GG(\A^{]\infty[})$ est lisse; on la prolonge en une action
 de $\GG(\A)$, qui commute
encore \`a celle de $\GG(\Q)$, en faisant agir $\GG(\R)$ sur ${\cal H}$ comme expliqu\'e
au \no\ref{prelim13}.  De mani\`ere explicite,
$$(g_\infty\star \phi)(\tau,x^{]\infty[})=\begin{cases}\phi(\tau,x^{]\infty[}) &
{\text{si ${\rm sign}(g_\infty)=1$,}}\\
\phi(\overline\tau,x^{]\infty[})& 
{\text{si ${\rm sign}(g_\infty)=-1$.}}\end{cases}$$
L'action de $\GG(\R)\subset \GG(\A)$ se factorise donc \`a travers $\GG(\R)/\GG(\R)_+$
et, si $g_\infty$ appartient au normalisateur de $\C^\dual$, alors
$(g_\infty\star \phi)(\tau,x^{]\infty[})=
\phi(\tau g_\infty,x^{]\infty[})$.

Les espaces ${\cal A}^+$, ${\cal A}^-$, ${\cal A}_{\rm par}^+$ et ${\cal A}_{\rm par}^-$
sont stables par $\GG(\A^{]\infty[})$, mais $\GG(\R)/\GG(\R)_+$ \'echange
${\cal A}^+$ et ${\cal A}^-$, ainsi que ${\cal A}_{\rm par}^+$ et ${\cal A}_{\rm par}^-$.

\subsubsection{D\'efinition}\label{fma9}
Le ${\cal A}$-module ${\cal A}\otimes W_{k,j}$ est muni d'actions
(diagonales) de $\GG(\Q)$ et de $\GG(\A^{]\infty[})$ 
qui commutent.  

Comme $\matrice{a}{b}{c}{d}\cdot \tau=\tfrac{a\tau+b}{c\tau+d}$,
on a $$\big(\matrice{a}{b}{c}{d}*\phi\big)(\tau)=
\phi\big(\matrice{a}{b}{c}{d}^{-1}\cdot\tau\big)=\phi\big(\tfrac{d\tau-b}{-c\tau+a}\big)$$
et donc, si $\matrice{a}{b}{c}{d}\in \GG(\Q)$,
$$\matrice{a}{b}{c}{d} * (e_1-\tau e_2)=\tfrac{1}{ad-bc}
\big((de_1-be_2)-\tfrac{d\tau-b}{-c\tau+a}(-ce_1+ae_2)\big)=\tfrac{1}{-c\tau+a}(e_1-\tau e_2).$$
Il s'ensuit que
$$\matrice{a}{b}{c}{d}*\tfrac{(\tau e_2-e_1)^k}{(e_1\wedge e_2)^j}=\tfrac{(ad-bc)^j}{(-c\tau+a)^k}
\,\tfrac{(\tau e_2-e_1)^k}{(e_1\wedge e_2)^j}.$$

Si $r\leq k$ le sous-${\cal A}^+$-module
$${\rm Fil}^{k-r-j}({\cal A}^+\otimes W_{k,j})=\oplus_{\ell=0}^r
\big({\cal A}^+\otimes \tfrac{(\tau e_2-e_1)^{k-\ell}e_1^\ell}{(e_1\wedge e_2)^j}\big)$$
est stable par $\GG(\Q)$ et par $\GG(\A^{]\infty[})$; 
ceci d\'efinit donc une filtration
d\'ecroissante de ${\cal A}^+\otimes W_{k,j}$ par des sous-$\GG(\Q)\times \GG(\A^{]\infty[})$-modules.

D\'efinissons l'espace des {\it formes modulaires ad\'eliques, quasi-holomorphes,
de poids $(k,j)$ et profondeur $r$} par:
$$M_{k,j,r}^{\rm qh}(\C)=H^0(\GG(\Q),{\rm Fil}^{k-r-j}({\cal A}^+\otimes W_{k,j})).$$
On \index{M3@\Mkj}note simplement $M_{k,j}(\C)$ (resp.~$M_{k,j}^{\rm qh}(\C)$) l'espace
$$M_{k,j}(\C)=M_{k,j,0}^{\rm qh}(\C) 
\quad{\text{(resp.~$M_{k,j}^{\rm qh}(\C)=M_{k,j,k}^{\rm qh}(\C)$)}}$$
des {\it formes modulaires ad\'eliques de poids $(k,j)$}
(resp.~des {\it formes modulaires ad\'eliques quasi-holomorphes de poids $(k,j)$
et de profondeur arbitraire}).
On note
$$M_{k,j}^{\rm par}(\C),\quad M_{k,j}^{\rm qh,\,par}(\C),\quad\cdots$$
les espaces des {\it formes paraboliques} obtenus en rempla\c{c}ant ${\cal A}^+$
par ${\cal A}_{\rm par}^+$.

Tous ces espaces sont munis d'actions lisses de $\GG(\A^{]\infty[})$.
\begin{lemm}\phantomsection\label{carcen}
Soit $\phi\in M_{k,j}(\C)$, non nulle.  Si 
$\matrice{a}{0}{0}{a}\star\phi=\omega(a)\phi$, pour tout $a\in\Aidu$, alors
$\omega$ est la restriction \`a $\Aidu$ d'un caract\`ere
de $\A^\dual/\Q^\dual$, dont la restriction \`a $\R^\dual$ est $x_\infty\mapsto x_\infty^{k-2j}$,
et donc est alg\'ebrique, de poids $k-2j$.
\end{lemm}
\begin{proof}
Si $a\in\Q^\dual$, alors
$$\matrice{a}{0}{0}{a}*\big(\matrice{a}{0}{0}{a}^{]\infty[}\star\phi\big)=\omega(a^{]\infty[})\phi.$$
Mais on a aussi (le $a^{2j-k}$ vient de l'action sur $\tfrac{(\tau e_2-e_1)^k}{(e_1\wedge e_2)^j}$)
\begin{align*}
\matrice{a}{0}{0}{a}*\big(\matrice{a}{0}{0}{a}^{]\infty[}\star\phi\big)(\tau,g^{]\infty[})=&\\
a^{2j-k}\phi(\matrice{a}{0}{0}{a}^{-1}\cdot\tau,
\big(\matrice{a}{0}{0}{a}^{]\infty[}\big)^{-1}g^{]\infty[}\matrice{a}{0}{0}{a}^{]\infty[})
=&\ a^{2j-k}\phi(\tau,g^{]\infty[}).
\end{align*}
Donc $\omega(a^{]\infty[})=a^{2j-k}$, et $\omega$ est la restriction \`a $\Aidu$ 
du caract\`ere alg\'ebrique
de $\A^\dual/\Q^\dual$ d\'efini par $\omega(x_\infty,x^{]\infty[})=\omega(x^{]\infty[})x_\infty^{k-2j}$,
qui est de poids $k-2j$.
\end{proof}
\subsubsection{Lien avec la th\'eorie classique}\label{class}
Si $\phi\in M_{k,j}(\C)$, soit $f_\phi:{\cal H}^+\to \C$ d\'efinie par
$$\phi=\phi_0\otimes \tfrac{(\tau e_2-e_1)^k}{(e_1\wedge e_2)^j}
\quad{\rm et}\quad
f_\phi(\tau)=\phi_0(\tau,1^{]\infty[}).$$

\begin{lemm}\phantomsection\label{emrt1}
{\rm (i)} Soit $\gamma \in \GG(\Q)_+$
 et soit $\gamma ^{]\infty[}$ l'image de $\gamma $
dans $\GG(\A^{]\infty[})$. 
Alors on a
$$f_{\gamma ^{]\infty[}\star \phi}=(f_\phi)_{|_{k,j}}\gamma^{-1}.$$
{\rm (ii)}
Si $\phi$ est invariante par $\wGamma(N)$, alors $f_\phi\in M^{\rm cl}_{k,j}(\Gamma(N),\C)$.
Plus pr\'ecis\'ement, si
$\matrice{a}{b}{c}{d}\star\phi=\omega(d)\,\phi$
pour tout $\matrice{a}{b}{c}{d}\in\wGamma_0(N)$
alors $f_{|_{k,j}}\matrice{a}{b}{c}{d}=\tilde\omega^{-1}(d)f_\phi$, pour
tout $\matrice{a}{b}{c}{d}\in\Gamma_0(N)$.
\end{lemm}
\begin{proof}
Le (ii) est une cons\'equence imm\'ediate du (i);
prouvons le (i).
Soit $\gamma =\matrice{a}{b}{c}{d}\in \GG(\Q)$ et donc $\gamma^{-1}=\frac{1}{ad-bc}\matrice{d}{-b}{-c}{a}$.
On a $\gamma ^{]\infty[}\star \phi=\gamma *(\gamma ^{]\infty[}\star \phi)$, et donc
\begin{align*}
f_{\gamma ^{]\infty[}\star \phi}(\tau)\otimes\tfrac{(\tau e_2-e_1)^k}{(e_1\wedge e_2)^j}
=&\ (\gamma *(\gamma ^{]\infty[}\star \phi))(\tau,1^{]\infty[})\\= &
\phi_0(\gamma ^{-1}\cdot \tau,(\gamma ^{-1})^{]\infty[}1^{]\infty[}\gamma ^{]\infty[})
\otimes \tfrac{(ad-bc)^{j}}{(-c\tau+a)^k}\tfrac{(\tau e_2-e_1)^k}{(e_1\wedge e_2)^j}\\=&\
\phi_0(\gamma ^{-1}\cdot \tau,1^{]\infty[})
\otimes \tfrac{(ad-bc)^{j}}{(-c\tau+a)^k}\tfrac{(\tau e_2-e_1)^k}{(e_1\wedge e_2)^j}\\= &
\tfrac{(ad-bc)^{j}}{(-c\tau+a)^k}f_{\phi}\big(\tfrac{d\tau-b}{-c\tau+a}\big)
\otimes\tfrac{(\tau e_2-e_1)^k}{(e_1\wedge e_2)^j}
\end{align*}
\qedhere
\end{proof}
\begin{rema}
 Une autre possibilit\'e pour transformer
l'action classique \`a droite en action \`a gauche serait de consid\'erer $f_{|k,j}{^{t}\gamma}$
au lieu de $f_{|k,j}\gamma^{-1}$, ce qui correspond \`a l'action
ad\'elique $(g\star\phi)(x)=\phi(x{^{\hskip.7mm t}g^{-1}})$.
 C'est ce que fait Emerton (en fixant $j$ en fonction de $k$).
\end{rema}

\begin{rema}
Le lemme~\ref{emrt1} ci-dessus fabrique une forme modulaire classique \`a partir
d'une forme modulaire ad\'elique.  Dans l'autre sens,
si $f\in M_{k,j}^{\rm cl}(\C)$ il existe
$\phi_{f}\in M_{k,j}(\C)$, unique, telle que
$$\phi_{f}\big(\tau,\matrice{u}{0}{0}{1}\big)=f(\tau)\otimes \tfrac{(\tau e_2-e_1)^k}{(e_1\wedge e_2)^j},
\quad{\text{ si $u\in\cZ^\dual$.}}$$
Par construction $\phi_f$ est invariante par $\matrice{\cZ^\dual}{0}{0}{1}$
et on a $f_{\phi_f}=f$.  Mais, si on supprime la premi\`ere propri\'et\'e,
on peut multiplier $\phi_f$ par n'importe quelle
fonction de la forme $\alpha\circ\det$, o\`u $\alpha:\A^\dual\to \C$ prend la valeur $1$
en~$1$ et se factorise \`a travers $\A^\dual/\R_+^\dual\Q^\dual$ (par exemple un caract\`ere
d'ordre fini de $\A^\dual/\Q^\dual$).
\end{rema}

\begin{rema}\phantomsection\label{fma10}
Pla\c{c}ons nous sous l'hypoth\`ese du (ii) du lemme~\ref{emrt1} et
supposons que $(p,N)=1$, et donc que $\phi$ est invariante par $\GG(\Z_p)$.
Soit $T_p$ \index{Tp@\Tp}l'op\'erateur
$$T_p=\tfrac{1}{p}\big(\matrice{1}{0}{0}{p}_p+\sum_{b=0}^{p-1}\matrice{p}{b}{0}{1}_p\big).$$
On d\'eduit du calcul ci-dessus, de l'identit\'e $\matrice{1}{0}{0}{p}_p=\matrice{p^{-1}}{0}{0}{1}_p
\matrice{p}{0}{0}{p}_p$, et de ce que $\matrice{p}{0}{0}{p}_p$ agit par
$\omega_p(p)=\tilde\omega^{-1}(p)p^{2(j+1)-(k+2)},$
que l'on a
\begin{align*}
T_p\star\phi(\tau,1^{]\infty[})=&\
\tfrac{1}{p}\Big(\tilde\omega^{-1}(p)p^{2j-k}p^{-(j-k-1)}f_\phi(p\tau)+\sum_{b=0}^{p-1}
p^{j-k-1}f_\phi(\tfrac{\tau-b}{p})\Big)\tfrac{(\tau e_2-e_1)^k}{(e_1\wedge e_2)^j}d\tau\\=&\
p^{j-k-1}\Big(\tilde\omega^{-1}(p)p^{k+1}f_\phi(p\tau)+\tfrac{1}{p}\sum_{b=0}^{p-1}
f_\phi(\tfrac{\tau-b}{p})\Big)\tfrac{(\tau e_2-e_1)^k}{(e_1\wedge e_2)^j}d\tau.
\end{align*}
Si on suppose que $\matrice{a}{b}{c}{d}\star\phi=\omega(a)\phi$
au lieu de $\matrice{a}{b}{c}{d}\star\phi=\omega(d)\phi$, les m\^emes calculs
nous donnent:
$$T_p\star\phi(\tau,1^{]\infty[})=
\tilde\omega^{-1}(p)p^{j-k-1}\Big(\tilde\omega(p)p^{k+1}f_\phi(p\tau)+\tfrac{1}{p}\sum_{b=0}^{p-1}
f_\phi(\tfrac{\tau-b}{p})\Big)\tfrac{(\tau e_2-e_1)^k}{(e_1\wedge e_2)^j}d\tau
$$
\end{rema}

\subsubsection{$q$-d\'eveloppement et mod\`ele de Kirillov}\label{fma11}
Soit $\phi=\phi_0\otimes\frac{(\tau e_1-e_2)^k}{(e_1\wedge e_2)^j}\in M_{k,j}(\C)$.  
L'invariance de $\phi$ par $\matrice{n^{-1}}{0}{0}{1}\in \GG(\Q)$ se traduit
par l'\'equation fonctionnelle
$$\phi_0\big(n\tau,\matrice{nu}{0}{0}{1}\big)=n^{j-k}
\phi_0\big(\tau,\matrice{u}{0}{0}{1}\big).$$
Maintenant, $\phi$ est somme de sa s\'erie de Fourier,
et il existe des fonctions $u\mapsto a(n,u)$, pour $n\in\Q_+$, telles que l'on ait
$$\phi_0\big(\tau,\matrice{u}{0}{0}{1}\big)=\sum_{n\in\Q_+}a(n,u){\bf e}_\infty(-n\tau).$$
\index{K@\KKK}Si ${\cal K}(\phi,u)=a(1,u)$, l'\'equation fonctionnelle ci-dessus se traduit
par la relation 
$$a(n,u)=n^{k-j}{\cal K}(\phi,nu),$$
et donc
$$\phi\big(\tau,\matrice{u}{0}{0}{1}\big)=
\big(a(0,u)+\sum_{n\in\Q_+^\dual}n^{k-j}{\cal K}(\phi,n^{]\infty[}u){\bf e}_\infty(-n\tau)\big)\otimes
\tfrac{(\tau e_2-e_1)^k}{(e_1\wedge e_2)^j}.$$
\begin{lemm}\phantomsection\label{fma11.2}
On a 
$${\cal K}\big({\matrice{a^{]\infty[}}{b^{]\infty[}}{0}{1}\star\phi},u\big)=
{\bf e}^{]\infty[}(b^{]\infty[}u){\cal K}(\phi,a^{]\infty[}u).$$
\end{lemm}
\begin{proof}
Si $n\in\Q$, on a
$\matrice{1}{n}{0}{1}*\big(\matrice{a^{]\infty[}}{b^{]\infty[}}{0}{1}\star\phi_0\big)=
\matrice{a^{]\infty[}}{b^{]\infty[}}{0}{1}\star\phi_0$,
et donc
\begin{align*}
\big(\matrice{a^{]\infty[}}{b^{]\infty[}}{0}{1}\star\phi_0\big)\big(\tau\matrice{u}{0}{0}{1}\big)
=&\ \phi_0\big(\tau-n,\matrice{1}{-n}{0}{1}\matrice{u}{0}{0}{1}\matrice{a^{]\infty[}}{b^{]\infty[}}{0}{1}\big)
\\=&\ \phi_0\big(\tau-n,\matrice{a^{]\infty[}u}{b^{]\infty[}u-n}{0}{1}\big).
\end{align*}
Si on choisit $n$ assez proche de $b^{]\infty[}u$, alors
$$\phi_0\big(\tau-n,\matrice{a^{]\infty[}u}{b^{]\infty[}u-n}{0}{1}\big)=
\phi_0\big(\tau-n,\matrice{a^{]\infty[}u}{0}{0}{1}\big),\quad
{\bf e}_\infty(-n)={\bf e}^{]\infty[}(n)={\bf e}^{]\infty[}(b^{]\infty[}u).$$
En comparant les d\'eveloppements de Fourier des deux membres, on en d\'eduit le r\'esultat
annonc\'e.
\end{proof}

\subsubsection{Questions de rationalit\'e}\label{ratio0}
On d\'eduit de la rem.\,\ref{cyclo10} des identifications naturelles
$${\rm LC}(\Aidu,\C)=
{\rm LC}(\Aidu\times\cZ^\dual,\C)^{\cZ^\dual}
=\big({\rm LC}(\Aidu,\C)\otimes\Q^{\rm cycl}\big)^{\cZ^\dual},$$
o\`u:

$\bullet$ $\cZ^\dual$ agit sur ${\rm LC}(\Aidu\times\cZ^\dual,\C)$
par $(a\cdot\phi)(x,u)=\phi(a^{-1}x,au)$ et sur
${\rm LC}(\Aidu,\C)\otimes\Q^{\rm cycl}$ par
$(a\cdot (\phi\otimes\alpha))(x)=\phi(a^{-1}x)\otimes\sigma_a(\alpha)$,

$\bullet$ La premi\`ere fl\`eche $\phi\mapsto\tilde\phi$ est donn\'ee par
$\tilde\phi(x,u)=\phi(ux)$ et l'inverse 
$\phi\otimes\alpha\mapsto\tilde\phi$ de la seconde par
$\tilde\phi(x,u)=\sigma_u(\alpha)\phi(x)$.

\medskip
Si $\Lambda$ est un sous-anneau de $\C$, on dit que {\it $\phi$ est d\'efinie sur $\Lambda$},
si $(u\mapsto {\cal K}(\phi,u))
\in {\rm LC}\big(\Aidu,\Lambda\otimes\Z^{\rm cycl}\big)^{\cZ^\dual}$.
On \index{M3@\Mkj}note $M_{k,j}(\Lambda)$ le sous-anneau de $M_{k,j}(\C)$ des formes d\'efinies sur $\Lambda$.
Par exemple, $\phi\in M_{k,j}(\Q)$, si ${\cal K}(\phi,u)\in \Q^{\rm cycl}$ 
et si $\sigma_a({\cal K}(\phi,u))={\cal K}(\phi,au)$, pour tous $a\in\cZ^\dual$
et $u\in\Aidu$.
\begin{rema}\phantomsection\label{ratio1}
L'application $\phi\mapsto f_\phi$ du lemme~\ref{emrt1}
induit un isomorphisme $M_{k,j}(\Q)\overset{\sim}{\to} M_{k,j}^{\rm cong}$,
l'isomorphisme inverse associe \`a $f\in M_{k,j}^{\rm cong}$ l'\'el\'ement
$\phi_f$ de $M_{k,j}(\Q)$ caract\'eris\'e par le fait que
$\phi_f\big(\tau,\matrice{a}{0}{0}{1}^{]\infty[}\big)=(\sigma_a(f))(\tau)$,
si $a\in\cZ^\dual$.
On a $f_{g\star\phi}=(f_\phi)_{|_{k,j}}g^{-1}$ pour tout $g\in \GG(\A^{]\infty[})$
(cela r\'esulte du lemme~\ref{emrt1} et 
de la description de l'action $f_{|_{k,j}}\gamma$, cf.~\no\ref{s23.5}).
\end{rema}

\Subsubsection{Torsion par un caract\`ere alg\'ebrique}\label{fma12}
\begin{rema}\phantomsection\label{ES7}
{\rm (i)} Si $\eta:\A^\dual/\Q^\dual\to\C^\dual$ est un caract\`ere continu,
d'ordre fini, alors la multiplication par $\eta\circ\det$ induit
un isomorphisme $M_{k,j,r}^{\rm qh}(\C)\cong M_{k,j,r}^{\rm qh}(\C)$,
et on a $g\star(\eta\circ\det\cdot\phi)=\eta\circ\det(g)
\cdot (\eta\circ\det\cdot\phi)$, si $g\in \GG(\A^{]\infty[})$.

{\rm (ii)} 
Le produit tensoriel par $(\delta_\A \circ\det)^a(e_1\wedge e_2)^a$
induit
un isomorphisme $M_{k,j,r}^{\rm qh}(\C)\cong M_{k,j-a,r}^{\rm qh}(\C)$
car 
$$\gamma * (\delta_\A \circ\det)=\det\gamma\,(\delta_\A \circ\det)
\quad{\rm et}\quad
\gamma *(e_1\wedge e_2)=(\det\gamma)^{-1}(e_1\wedge e_2),$$
et, si $g\in \GG(\A^{]\infty[})$, on a 
\begin{align*}
\big(g\star(\phi\otimes (\delta_\A \circ\det)^a(e_1\wedge e_2)^a)\big)(x)=&\ \delta_\A ^a\circ\det(g)
\cdot (\phi(xg)\otimes (\delta_\A \circ\det(x))^a(e_1\wedge e_2)^a)\\ 
=&\ \delta_\A ^a\circ\det(g)
\cdot ((g\star\phi)(\delta_\A \circ\det)^a(e_1\wedge e_2)^a)(x).
\end{align*}

{\rm (iii)} 
Si $\chi=\eta|\ |_\A^a:\A^\dual/\Q^\dual\to \C^\dual$ est alg\'ebrique de poids $a$, 
le produit tensoriel par
$(\eta\delta_\A ^a\circ \det)(e_1\wedge e_2)^a$ induit
un isomorphisme 
$$\xymatrix@C=3.5cm{
M_{k,j}(\C)\ar[r]^-{\otimes (\eta\delta_\A ^a\circ \det)(e_1\wedge e_2)^a}_-{\sim}
& M_{k,j-a}(\C)\otimes\chi},$$
o\`u l'on note $M\otimes\chi$ le $\GG(\A^{]\infty[})$-module $M$ avec action
de $\GG(\A^{]\infty[})$ multipli\'ee par $\chi\circ\det$.
\end{rema}

Notons \index{twi@\twis}simplement $\phi\mapsto\phi\otimes\chi$ l'application 
$\phi\mapsto\phi\otimes (\eta\delta_\A ^a\circ \det)(e_1\wedge e_2)^a$ ci-dessus.
\begin{lemm}\phantomsection\label{eqf1}
{\rm (i)}
Si $L$ est un sous-corps de $\C$ et si $\chi=\eta|\ |_\A^a$ est \`a valeurs dans~$L$,
alors $\phi\mapsto G(\chi^{-1})\phi\otimes\chi$ induit un isomorphisme
$M_{k,j}(L)\overset{\sim}{\longrightarrow} M_{k,j-a}(L)\otimes\chi$ de $G(\A^{]\infty[})$-modules.

{\rm (ii)} On a un diagramme commutatif 
$$\xymatrix@C=1.8cm@R=.6cm{M_{k,j}(L)\ar[r]^-{\otimes G(\chi^{-1})\chi}\ar[d]^-{\cal K}
&M_{k,j-a}(L)\otimes\chi\ar[d]^-{\cal K}\\
{\rm LC}(\Aidu,L\otimes\Q^{\rm cycl})^{\cZ^\dual}\ar[r]^-{\times G(\chi^{-1})\chi}
&{\rm LC}(\Aidu,L\otimes\Q^{\rm cycl})^{\cZ^\dual}}
$$
\end{lemm}
\begin{proof}
On a ${\cal K}(\phi\otimes\chi,u)=\chi(u){\cal K}(\phi,u)$, si $u\in\Aidu$.
Si $\phi\in M_{k,j}(L)$ et si $\chi$ est \`a valeurs dans $L$, on
a ${\cal K}(\phi\otimes\chi,au)=\eta(a)\sigma_a({\cal K}(\phi\otimes\chi,u))$,
si $a\in\cZ^\dual$.
Comme 
$\sigma_a(G(\chi^{-1}))=\eta(a)G(\chi^{-1})$, on a
$$G(\chi^{-1}){\cal K}(\phi\otimes\chi,au)
=\sigma_a(G(\chi^{-1}){\cal K}(\phi\otimes\chi,u)),$$
et donc $G(\chi^{-1})\phi\otimes\chi\in M_{k,j}(L)$.

Le lemme s'en d\'eduit.
\end{proof}

\Subsection{L'application d'Eichler-Shimura}\label{fma13}
\subsubsection{D\'efinition}\label{fma14}
Notons $\tau=x+iy$ le param\`etre naturel de ${\cal H}$ et 
$d\tau$ la base naturelle de $\Omega^1({\cal H})$ sur $\O({\cal H})$.
Si $k\in\N$ et $j\in\Z$, on a des suites exactes, 
induites par la diff\'erentielle $f\mapsto df$ en la place $\infty$,
\begin{align*}
0\to {\rm LC}(\GG(\A)/\C^\dual,\C)\otimes W_{k,j}\to &\ {\cal A}^+\otimes  W_{k,j}\to
{\cal A}^+d\tau\otimes  W_{k,j}\to 0\\
0\to {\rm LC}(\GG(\A)/\C^\dual,\C)\otimes W_{k,j}\to&\ {\cal A}\otimes  W_{k,j}\to
\big({\cal A}^+d\tau\oplus{\cal A}^-d\overline\tau\big)\otimes  W_{k,j}\to 0
\end{align*}
Notons qu'une fonction localement constante sur $\GG(\A)$ est constante sur les classes
modulo $\C^\dual$, et donc que
${\rm LC}(\GG(\A)/\C^\dual,\C)={\rm LC}(\GG(\A),\C)$.
On appelle {\it applications d'Eichler-Shimura} les \index{iotaES@\iotaES}applications $\iota_{\rm ES}$
de connexion
\begin{align*}
H^0\big(\GG(\Q),{\cal A}^+d\tau\otimes  W_{k,j}\big)\to&\ H^1\big(\GG(\Q),{\rm LC}(\GG(\A),\C)\otimes W_{k,j}\big)\\
H^0\big(\GG(\Q),\big({\cal A}^+d\tau\oplus{\cal A}^-d\overline\tau\big)\otimes  W_{k,j}\big)\to&\ 
H^1\big(\GG(\Q),{\rm LC}(\GG(\A),\C)\otimes W_{k,j}\big)
\end{align*}
ainsi que celles obtenues en tensorisant par $W_{k,j}^\dual(\C)$ et
en utilisant l'injection $\GG(\Q)\times \GG(\C)$-\'equivariante
${\rm LC}(\GG(\A),\Q)\otimes W_{k,j}\otimes W_{k,j}^\dual(\C)\hookrightarrow {\rm LP}(\GG(\A),\C)$
\begin{align*}
H^0\big(\GG(\Q),{\cal A}^+d\tau\otimes  W_{k,j}\big)\otimes W_{k,j}^\dual(\C)\to&\ 
 H^1\big(\GG(\Q),{\rm LP}(\GG(\A),\C)\big)\\
H^0\big(\GG(\Q),\big({\cal A}^+d\tau\oplus{\cal A}^-d\overline\tau\big)\otimes  W_{k,j}\big)\otimes W_{k,j}^\dual(\C)\to&\ 
H^1\big(\GG(\Q),{\rm LP}(\GG(\A),\C)\big)
\end{align*}

\begin{rema} \label{ES9}
{\rm (i)} Les deux membres des secondes lignes
 ont une action de $\GG(\A)\times \GG(\C)$, l'action de
$\GG(\A)$ \'etant localement constante (et donc celle de $\GG(\R)\subset \GG(\A)$
se factorise
\`a travers $\GG(\R)/\GG(\R)_+\cong\{\pm 1\}$) tandis que celle de $\GG(\C)$ est alg\'ebrique
(triviale si on ne tensorise pas par $W_{k,j}^\dual(\C)$);  
le membre de gauche des premi\`eres lignes a seulement une action
de $\GG(\A^{]\infty[})\times \GG(\C)$.
L'application $\iota_{\rm ES}$ des secondes lignes est $\GG(\A)\times \GG(\C)$-\'equivariante 
(celle des premi\`eres est seulement $\GG(\A^{]\infty[})\times \GG(\C)$-\'equivariante).

{\rm (ii)} L'action de $\GG(\A)\times \GG(\Q)\subset \GG(\A)\times \GG(\C)$ sur
$$H^1(\GG(\Q),{\rm LP}(\GG(\A),\C))=\C\otimes_\Q 
H^1(\GG(\Q),{\rm LP}(\GG(\A),\Q)),$$
stabilise le
sous-$\Q$-espace $H^1(\GG(\Q),{\rm LP}(\GG(\A),\Q))$.
\end{rema}

Maintenant,
$$\matrice{a}{b}{c}{d} * d\tau=\tfrac{d}{d\tau}
\big(\tfrac{d\tau-b}{-c\tau+a}\big)\,d\tau=\tfrac{ad-bc}{(-c\tau+a)^2}\,d\tau,$$
et donc $d\tau$ se transforme comme $\frac{(\tau e_2-e_1)^2}{e_1\wedge e_2}$
sous l'action de $\GG(\Q)$.  En identifiant\footnote{\label{Koda}
L'isomorphisme de Kodaira-Spencer $\omega^2\cong \Omega^1({\text{log-cusp}})$ n'est
pas $\GG(\Q)$-\'equivariant; celui qui l'est est
$\omega^2\cong \Omega^1({\text{log-cusp}})\otimes{\cal H}^2_{\rm dR}$.
Il envoie $\big(\frac{dt}{t}\big)^2$ sur $\frac{dq}{q}\otimes\zeta_{\rm dR}^{-1}$.
Or on a $\frac{dt}{t}=2i\pi(\tau e_2-e_1)$, $\frac{dq}{q}=2i\pi\,d\tau$
et $\zeta_{\rm dR}^{-1}=2i\pi\,\zeta_{\rm B}^{-1}=-2i\pi\, e_1\wedge e_2$,
et donc $\frac{(\tau e_2-e_1)^2}{e_1\wedge e_2}=-d\tau$ 
(cf.~${\rm n}^{\rm os}$\,\ref{geo6} et~\ref{como112.2} pour les notations).}
$d\tau$ \`a $-\frac{(\tau e_2-e_1)^2}{e_1\wedge e_2}$, cela
identifie
${\cal A}^+d\tau\otimes  W_{k,j}$ \`a ${\rm Fil}^2({\cal A}^+\otimes W_{k+2,j+1})$.
En passant \`a la suite exacte de cohomologie, on obtient une suite exacte
$$M_{k,j}^{\rm qh}(\C)\overset{d}{\longrightarrow}
M_{k+2,j+1,k}^{\rm qh}(\C)\longrightarrow 
H^1(\GG(\Q),{\rm LC}(\GG(\A))\otimes W_{k,j}).$$
Par ailleurs, l'inclusion naturelle
$M_{k+2,j+1}(\C)\hookrightarrow M_{k+2,j+1,k}^{\rm qh}(\C)$ induit un isomorphisme
$$M_{k+2,j+1}(\C)\cong M_{k+2,j+1,k}^{\rm qh}(\C)/dM_{k,j}^{\rm qh}(\C).$$
D'o\`u une injection (d'Eichler-Shimura), $\GG(\A)$-\'equivariante,
$$\iota_{\rm ES}: M_{k+2,j+1}(\C)\hookrightarrow
H^1(\GG(\Q),{\rm LC}(\GG(\A),\C)\otimes W_{k,j}).$$
En faisant la somme sur les couples $(k,j)$, on en d\'eduit une injection
$$\iota_{\rm ES}: \oplus_{k,j}(M_{k+2,j+1}(\C)\otimes W_{k,j}^\dual)\hookrightarrow
H^1(\GG(\Q),{\rm LP}(\GG(\A),\C)).$$

\subsubsection{D\'ecomposition suivant l'action de $\GG(\R)/\GG(\R)_+$}\label{fma15}
Si $X(\GG(\A))$ est un espace de fonctions sur $\GG(\A)$, stable par translation
\`a droite par $\GG(\A)$ et par translation \`a gauche par $\GG(\Q)$, et sur lequel
$\GG(\R)_+\subset \GG(\A)$ agit trivialement,
on note $H^1(\GG(\Q),X(\GG(\A)))^\pm$ le sous-espace
de $H^1(\GG(\Q),X(\GG(\A)))$ sur lequel $\matrice{-1}{0}{0}{1}_\infty$
agit par $\pm 1$.  Tout \'el\'ement $\gamma$ de $H^1(\GG(\Q),X(\GG(\A)))$
s'\'ecrit, de mani\`ere unique, $\gamma=\gamma^++\gamma^-$, avec
$\gamma^{\pm}\in H^1(\GG(\Q),X(\GG(\A)))^{\pm}$.

On \index{iotaES@\iotaES}note $\iota^{\pm}_{\rm ES}(f)$ l'\'el\'ement $\iota_{\rm ES}(f)^{\pm}$.
Cela d\'ecompose l'application $\iota_{\rm ES}$ sous la forme
$$\iota_{\rm ES}=\iota_{\rm ES}^++\iota_{\rm ES}^-.$$

\subsubsection{Torsion par un caract\`ere alg\'ebrique}\label{fma16}
Comme compl\'ement aux r\'esultats du \no\ref{fma12}, on a
la remarque suivante.
\begin{rema}\phantomsection\label{ES7.2}
Soit $\chi=\eta|\ |_\A^a:\A^\dual/\Q^\dual\to \C^\dual$ alg\'ebrique de poids $a$.
Notons $H^1(M)$ le groupe $H^1(\GG(\Q),M)$
et, si $X={\rm LC}$, ${\cal C}$ et $F=\C,L$, notons
$X_F$ l'espace $X(\GG(\A),F)$.

{\rm (i)}
On a un diagramme commutatif $\GG(\A^{]\infty[})$-\'equivariants:
$$\xymatrix@C=3.5cm{
M_{k+2,j+1}(\C)\ar[r]^-{\otimes (\eta\delta_\A ^a\circ \det)(e_1\wedge e_2)^a
}\ar[d]^-{\iota_{\rm ES}}
& M_{k+2,j-a+1}(\C)\otimes\chi\ar[d]^-{\iota_{\rm ES}}\\
H^1({\rm LC}_\C\otimes W_{k,j})
\ar[r]^-{\otimes (\eta\delta_\A ^a\circ\det)(e_1\wedge e_2)^a}
&H^1({\rm LC}_\C\otimes W_{k,j-a})\otimes\chi}$$

{\rm (ii)}
Si $L$ est une extension finie de $\Q_p$, on a un diagramme $\GG(\A)$-\'equivariant:
$$\xymatrix@C=3.5cm{
H^1({\rm LC}_L\otimes W_{k,j})\otimes W_{k,j}^\dual
\ar[r]^-{\otimes (\eta\delta_\A ^a\circ\det)(e_1\wedge e_2)^a
(e_1^\dual\wedge e_2^\dual)^a}\ar[d]
&H^1({\rm LC}_L\otimes W_{k,j-a})\otimes W_{k,j-a}^\dual\otimes\chi^{(p)}\ar[d]\\
H^1({\cal C}_L)\ar[r]^-{\times \chi^{(p)}\circ\det}&H^1({\cal C}_L)
\otimes\chi^{(p)}
}$$
En particulier, si $\chi=|\ |_\A^a$, alors $\chi^{(p)}((-1)_\infty,1^{]\infty[})=(-1)^a$,
et donc les twists du diagramme ci-dessus multiplient l'action de
$\matrice{-1}{0}{0}{1}_\infty$ par $(-1)^a$.
\end{rema}

\section{Cohomologie des courbes modulaires}\label{como1}
Dans ce chapitre, on passe en revue les diff\'erents th\'eor\`emes de comparaison
pour les formes modulaires et la cohomologie de $\GG(\Q)$; en particulier,
on donne (cf.~\no\ref{qq21}) une description de l'application d'Eichler-Shimura $p$-adique, cruciale
pour la factorisation de la cohomologie compl\'et\'ee du chap.\,\ref{YEUL1}.
\Subsection{L'espace $\GG(\Q)\backslash ({\cal H}\times \GG(\A^{]\infty[}))$}\label{geo1}
\subsubsection{Espaces de modules}\label{geo2}
Soit $g=(g_\infty,g^{]\infty[})\in \GG(\A)$.
Si $g_\infty=\matrice{a_\infty}{b_\infty}{c_\infty}{d_\infty}\in \GG(\R)$, notons
$\vecteur{\omega_1(g)}{\omega_2(g)}$ le vecteur
$g_\infty\cdot\vecteur{i}{1}=\vecteur{a_\infty i+b_\infty}{c_\infty i+d_\infty}\in\C^2$,
et $V(g)$ le sous-$\Q$-espace de $\C$ engendr\'e par ${\omega_1(g)},{\omega_2(g)}$.
Alors $\R\otimes V(g)\to \C$ est un isomorphisme, et
$g^{]\infty[}$ fournit une base $\lambda_{g,1},\lambda_{g,2}$ du $\A^{]\infty[}$-module
${\rm Hom}(V(g),\A^{]\infty[})$ par la formule $g^{]\infty[}=\matrice{\lambda_{g,1}(\omega_1(g))}{\lambda_{g,2}(\omega_1(g))}
{\lambda_{g,1}(\omega_2(g))}{\lambda_{g,2}(\omega_2(g))}$.

Si $\gamma\in \GG(\Q)$, alors $\omega_1(\gamma g),\omega_2(\gamma g)$ est une base
de $V(g)$, et $\lambda_{\gamma g,j}=\lambda_{g,j}$, si $j=1,2$.
On en d\'eduit une bijection naturelle:
$$\GG(\Q)\backslash \GG(\A)\longleftrightarrow \{V\subset\C,\ \R\otimes V\overset{\sim}\to\C + {\text{base $(\lambda_1,\lambda_2)$ de
 ${\rm Hom}(V,\A^{]\infty[})$}}\}.$$
Maintenant, si $(V,\lambda_1,\lambda_2)$ est un \'el\'ement du membre de droite, on peut lui associer
$(V^+,\lambda_1,\lambda_2)$, o\`u $V^+\subset V$ est l'ensemble
des $v$ tels que $\lambda_1(v),\lambda_2(v)\in\cZ$.  Alors $V^+$ est un $\Z$-r\'eseau de $V$ (et donc un r\'eseau de $\C$)
et $(\lambda_1,\lambda_2)$ est une base du $\cZ$-module ${\rm Hom}(V^+,\cZ)$. On en d\'eduit des bijections
naturelles
$$\xymatrix@R=.4cm@C=-.2cm{
\GG(\Q)\backslash \GG(\A)\ar@{-}[r]^-{\sim}\ar@{-}[d]^-{\wr}& \{V\subset\C,\ \R\otimes V\overset{\sim}\to\C + {\text{base $(\lambda_1,\lambda_2)$ de
 ${\rm Hom}(V,\A^{]\infty[})$}}\}\ar@{-}[d]^-{\wr}\\
\GG(\Z)\backslash (\GG(\R)\times \GG(\cZ))\ar@{-}[r]^-{\sim}& \{{\text{$\Lambda$ r\'eseau de $\C$}} + {\text{base $(\lambda_1,\lambda_2)$ de
 ${\rm Hom}(\Lambda,\cZ)$}}\}}$$
Ceci fournit un isomorphisme d'espaces analytiques
$$\GG(\Q)\backslash({\cal H}\times \GG(\A^{]\infty[}))\cong \GG(\Z)\backslash ({\cal H}\times \GG(\cZ)),$$
qui explique pourquoi la cohomologie de $\GG(\Z)\backslash ({\cal H}\times \GG(\cZ))$ est naturellement munie
d'une action de $\GG(\A^{]\infty[})$ et pas seulement de $\GG(\cZ)$.
\begin{rema}
Se donner une base $(\lambda_1,\lambda_2)$ de
 ${\rm Hom}(V,\A^{]\infty[})$ est \'equivalent \`a se donner une base de $\cZ\otimes V^+$
(prendre la base duale).
\end{rema}

\subsubsection{Quotients de $\GG(\Q)\backslash ({\cal H}\times \GG(\A^{]\infty[}))$}\label{geo3}

Si $K$ est un sous-groupe ouvert compact de $\GG(\A^{]\infty[})$, \index{Y@\YYY}soit
$$Y(K)(\C):=\GG(\Q)\backslash({\cal H}\times \GG(\A^{]\infty[}))/K
=\GG(\Z)\backslash ({\cal H}\times \GG(\cZ))/K.$$
C'est une surface de Riemann non compacte dont l'ensemble des composantes connexes
est $\cZ^\dual/\det K$.

Si $K=\wGamma(N)$, on note simplement $Y(N)(\C)$ la surface de Riemann $Y(K)(\C)$.
On a un isomorphisme de surfaces de Riemann (on note encore $a$ un rel\`evement 
arbitraire de $a\in(\Z/N)^\dual$
dans $\Z$):
$$(\Gamma(N)\backslash{\cal H}^+)\times(\Z/N)^\dual\overset{\sim}{\to}Y(N)(\C),\quad
(\tau,a)\mapsto (\C/(\Z\oplus\Z\tau),\tfrac{a}{N},\tfrac{\tau}{N}).$$
Par ailleurs, 
on a des isomorphismes de surfaces de Riemann
(on note encore $a$ un rel\`evement arbitraire de $a\in(\Z/N)^\dual$ dans $\cZ^\dual$)
\begin{align*}
&(\Gamma(N)\backslash{\cal H}^+)\times(\Z/N)^\dual \overset{\sim}{\to} \GG(\Q)\backslash ({\cal H}\times \GG(\A^{]\infty[}))/\wGamma(N),
\quad (\tau,a)\mapsto\big(\tau,\matrice{a}{0}{0}{1}\big),\\
&(\Gamma(N)\backslash{\cal H}^+)\times(\Z/N)^\dual \overset{\sim}{\to} \GG(\Q)\backslash \GG(\A)/
(\C^\dual\times\wGamma(N)),
\quad (x+iy,a)\mapsto\big(\matrice{y}{x}{0}{1},\matrice{a}{0}{0}{1}\big)
\end{align*}
Le groupe
$\GG(\A^{]\infty[})$ agit sur la limite projective  
$\varprojlim_K Y(K)(\C)$,
pour $K$ d\'ecrivant les sous-groupes ouverts compacts de $\GG(\A^{]\infty[})$
(la conjugaison par un \'el\'ement de $\GG(\A^{]\infty[})$ est un automorphisme
du syst\`eme des
sous-groupes compacts de $\GG(\A^{]\infty[})$), et donc aussi sur sa cohomologie
(qui, par d\'efinition, est la limite inductive des cohomologies
de $Y(K)(\C)$).

\subsubsection{Syst\`emes locaux}\label{geo5}
La repr\'esentation $W_{k,j}$ du \no\ref{fma4}
d\'efinit un syst\`eme local $W_{k,j}^{\rm B}$ sur $Y(K)(\C)$
(le B est l'initiale de Betti).
La cohomologie de ce syst\`eme local est reli\'ee \`a la cohomologie de $\Gamma(1)\cap K$,
et peut s'exprimer en termes de celle de $\Gamma(1)$ via le lemme de Shapiro:
$$H^i(Y(K)(\C),W_{k,j}^{\rm B})=H^i(\Gamma(1),{\rm LC}(\GG(\cZ)/K,\Z)\otimes W_{k,j}).$$
En passant \`a la limite sur $K$, on obtient des isomorphismes
\begin{align*}
\varinjlim H^i(Y(K)(\C),W_{k,j}^{\rm B})&=H^i(\Gamma(1),{\rm LC}(\GG(\cZ),\Z)\otimes W_{k,j})\\
&=H^i(\GG(\Q),{\rm LC}(\GG(\A),\Z)\otimes W_{k,j}),
\end{align*}
le second isomorphisme r\'esultant de ce que 
${\rm LC}(\GG(\A),\Z)={\rm Ind}_{\Gamma(1)}^{\GG(\Q)}{\rm LC}(\GG(\cZ),\Z)$.
On a les m\^emes r\'esultats pour les cohomologies $H^i_c$ (\`a \index{H1@\Hc}support compact) et
$H^1_{\rm par}$ (parabolique).

Tensoriser $W_{k,j}$ par le faisceau structural produit un fibr\'e \index{W2@\www}vectoriel $W_{k,j}^{\rm dR}$ muni
de la connexion $\nabla$ induite par la diff\'erentielle $d:\O\to\Omega^1$ sur le faisceau structural.
Les formes modulaires quasi-holomorphes ad\'eliques de poids $(k,j)$
s'identifient aux sections globales, i.e.~$\varinjlim H^0(Y(K)(\C),W_{k,j}^{\rm dR})$,
qui sont \`a croissance lente \`a l'infini.
Le th\'eor\`eme de comparaison Betti-de Rham fournit, si $\sharp\in\{\ ,c,{\rm par}\}$, 
\index{H1@\Hc}un isomorphisme:
$$\varinjlim H^i_{{\rm dR}\sharp}(Y(K)(\C),W_{k,j}^{\rm dR})\cong
\C\otimes_\Q\varinjlim H^i_{{\rm B}\sharp}(Y(K)(\C),W_{k,j}^{\rm B}).$$

\subsubsection{Interpr\'etation g\'eom\'etrique}\label{geo6}
Si $\Gamma$ est un sous-groupe de congruence de $\Gamma(1)$,
soit $\Gamma'$ le produit semi-direct de $\Gamma$ par $\Z^2$ 
des matrices 
$$\gamma'=\Big(\begin{smallmatrix}a&b&0\\c&d&0\\u&v&1\end{smallmatrix}\Big),\quad{\text{$\matrice{a}{b}{c}{d}\in\Gamma$
et $(u,v)\in\Z^2$.}}$$
On fait agir $\gamma'$ sur ${\cal H}\times\C$ par
$\gamma'\cdot(\tau,z)=\big(\tfrac{a\tau+b}{c\tau+d},\tfrac{z+u\tau+v}{c\tau+d}\big)$.
Alors $\Gamma'\backslash({\cal H}\times\C)$ se surjecte sur $\Gamma\backslash{\cal H}$ via $(\tau,z)\mapsto\tau$
et la fibre en $\tau$ est le tore complexe $E_\tau=\C/(\Z\oplus\Z\tau)$.
Le syst\`eme local $(W_{1,0}^\dual)^{\rm B}$ est l'homologie $H_1(E_\tau,\Z)\cong \Z\oplus\Z\tau$;
les \'el\'ements $e_1^\dual$ et $e_2^\dual$ correspondent respectivement \`a $-1$ et $\tau$.
Si $\gamma=\matrice{a}{b}{c}{d}\in\Gamma$, l'action de la monodromie correspondant \`a $\gamma$ sur $v$ se calcule en suivant
$v$ le long d'un chemin reliant $\tau$ \`a $\gamma^{-1}\cdot\tau$ (le groupe fondamental agit naturellement \`a droite,
et pour transformer notre action \`a gauche en action \`a droite, il faut faire
agir $\gamma$ par $\gamma^{-1}$).  La base $(e_1^\dual, e_2^\dual)=(-1,\tau)$ se transforme
en $\big(-1,\frac{d\tau-b}{-c\tau+a}\big)\sim (c\tau-a,d\tau-b)=(a e_1^\dual+c e_2^\dual,b e_1^\dual+ d e_2^\dual)$,
ce qui est effectivement l'action sur $W_{1,0}^\dual$.

Le syst\`eme local $W_{1,0}^{\rm B}$ est donc la cohomologie de Betti de la famille des $E_\tau$;
le fibr\'e $U\mapsto W_{1,0}^{\rm dR}(U)=\O(U)\otimes W_{1,0}$ 
en est la cohomologie de de Rham; la connexion $\nabla=d\otimes 1$ ci-dessus est
la connexion de Gauss-Manin.
Les $W_{k,j}^{\rm B}$ et $W_{k,j}^{\rm dR}$ s'obtiennent \`a partir de 
$W_{1,0}^{\rm B}$ et $W_{1,0}^{\rm dR}$ par des op\'erations tensorielles.
La forme $\tau e_2-e_1$ est la forme $dz$ sur $E_\tau$ (ces deux formes
prennent les m\^emes valeurs sur $e_1^\dual$ et $e_2^\dual$, \`a savoir
$-1$ et~$\tau$).

Il est d'usage de poser 
$$t=e^{2i\pi\,z},\quad q=e^{2i\pi\,\tau}.$$
On a un isomorphisme (induit par $z\mapsto t$) et des identifications
$$E_\tau=\C/(\Z\oplus\Z\tau)\cong \C^\dual/q^\Z,\quad \tfrac{dt}{t}=2i\pi\,dz=2i\pi (\tau e_2-e_1),
\quad \tfrac{dq}{q}=2i\pi\,d\tau.$$

\Subsection{Les courbes modulaires}\label{como2}
\subsubsection{Espaces de modules de courbes elliptiques}\label{geo7}
Si $K$ est un sous-groupe ouvert compact de $\GG(\A^{]\infty[})$, on \index{X@\XXX}note
$X(K)(\C)$ la compactifi\'ee (lisse) de $Y(K)(\C)$ obtenue en rajoutant
\`a $Y(K)(\C)$ des {\og pointes\fg}.
Alors $X(K)(\C)$ est l'ensemble des $\C$-points d'une
courbe alg\'ebrique $X(K)$
qui admet un mod\`ele connexe (mais pas g\'eom\'etriquement
connexe) sur $\Q$.
On note $X(K)^\times$ la
courbe $X(K)$ avec une structure logarithmique aux pointes.

Si $N\geq 1$, on note simplement $X(N)$ la courbe $X(\wGamma(N))$, compactifi\'ee de $Y(N)$.
Alors $Y(N)$
param\`etre les triplets $(E,e_1,e_2)$ \`a isomorphisme pr\`es, o\`u $E$
est une courbe elliptique, et $(e_1,e_2)$ une base sur $\Z/N$ du sous-groupe $E[N]$ de $N$-torsion de $E$.
La courbe $Y(N)$ est d\'efinie sur $\Q$, mais n'est pas g\'eom\'etriquement connexe:
le corps des constantes de $\O(Y(N))$ est $\Q(\zeta_N)$, et les composantes connexes
de $Y(N)_{\Q(\zeta_N)}$ sont en bijection avec l'ensemble $\bmu_N^\dual$ des
racines primitives $N$-i\`emes de l'unit\'e (si $\eta\in\bmu_N^\dual$, la composante
correspondant \`a $\eta$ param\`etre les triplets $(E,e_1,e_2)$ comme ci-dessus
tels que l'accouplement de Weil $\langle e_1,e_2\rangle$ de $e_1$ et $e_2$ soit \'egal \`a $\eta$).
On identifie $\bmu_N^\dual$ \`a $(\Z/N)^\dual$ en envoyant $a\in(\Z/N)^\dual$ sur
$\zeta_N^a=e^{2i\pi\,a/N}$.

On \index{inf@\inft}note $\infty_a$ la pointe \`a l'infini
de $(\Gamma(N)\backslash{\cal H}^+)\times\{a\}\subset Y(N)(\C)$ (cf.~\no\ref{geo3}). 
 C'est un point de $X(N)(\Q(\zeta_N))$, et si
$u\in \cZ^\dual$ et $\sigma_u\in{\rm Gal}(\Q^{\rm cycl}/\Q)$ est l'image inverse de $u$ par
le caract\`ere cyclotomique, alors $\sigma_u(\infty_a)=\infty_{ua}$.

\subsubsection{Le voisinage $p$-adique de la pointe $\infty$}\label{como2.9}
Soit $C=\C_p$ (ou, plus g\'en\'eralement, un corps alg\'ebriquement clos, complet
pour $v_p$).
Soit $Z(1)_C\subset X(1)_C$ \index{Z2@\zzz}le lieu param\'etrant les courbes elliptiques \`a r\'eduction
multiplicative (i.e.~$v_p(j)<0$).  C'est une boule ouverte unit\'e 
(de param\`etre~$\frac{1}{j}$ ou $q$).

Si $N\geq 1$, notons $\pi_N:X(N)\to X(1)$ l'application naturelle, et $Z(N)_C$
la composante connexe de $\pi_N^{-1}(Z(1)_C)$ stable par $\matrice{1}{\Z/N}{0}{1}\subset \GG(\Z/N)$.
Alors $Z(N)_C$ est une boule ouverte de param\^etre local 
$q^{1/N}$ et $\pi_N^*:\O^+(Z(1)_C)\subset\O^+(Z(N)_C)$ 
est l'inclusion $\O_C[[q]]\hookrightarrow \O_C[[q^{1/N}]]$.

Enfin, on \index{Z3@\zzzo}note $Z(0)_C$ la limite projective des $Z(N)_C$.

\subsubsection{Syst\`emes locaux}\label{como3}
Les faisceaux du \no\ref{geo6} ont des avatars alg\'ebriques: le tore $E_\tau$ est l'ensemble des
$\C$-points d'une courbe elliptique, et les
courbes $Y(K)$ param\`etrent des courbes elliptiques avec structure de niveau.
La cohomologie de la courbe elliptique universelle au-dessus de $Y(K)$ fournit
donc des faisceaux ${\cal H}^i_{\rm truc}$ sur $Y(K)$ ou son analytifi\'ee.  

Posons $F^{\rm B}=\Q$, $F^{\rm dR}=\Q$ et $F^{\eet}=\Q_p$
(avec B = Betti, dR = de\,Rham, ${\eet}$ = \'etale). 
Si ${\rm truc}\in\{{\rm B},{\rm dR},{\eet}\}$
et si $k\in\N$ et $j\in\Z$, on \index{W2@\www}note
$W_{k,j}^{\rm truc}$ le $F^{\rm truc}$-faisceau
$$W_{k,j}^{\rm truc}={\rm Sym}^k{\cal H}^1_{\rm truc}\otimes ({\cal H}^2_{\rm truc})^{-j}.$$
Notons que ${\cal H}^2_{\rm truc}=\det {\cal H}^1_{\rm truc}$ est le twist de Tate $(-1)$
dans chacune des cohomologies consid\'er\'ees et que l'on a
$({\cal H}^1_{\rm truc})^\dual\cong {\cal H}^1_{\rm truc}\otimes ({\cal H}^2_{\rm truc})^{-1}$. 
On en d\'eduit un isomorphisme
$$(W_{k,j}^{\rm truc})^\dual\cong W_{k,k-j}^{\rm truc}.$$ 

Le fibr\'e ${\cal H}^1_{\rm dR}$ admet comme sous-fibr\'e de rang~$1$, le
\index{ome@\ome}fibr\'e $\omega$ (des formes diff\'erentielles holomorphes sur la courbe elliptique
universelle), et donc une filtration
$${\rm Fil}^0{\cal H}^1_{\rm dR}={\cal H}^1_{\rm dR},\quad
{\rm Fil}^1{\cal H}^1_{\rm dR}=\omega,\quad
{\rm Fil}^2{\cal H}^1_{\rm dR}=0.$$
Comme ${\cal H}^2_{\rm dR}=\wedge^2 {\cal H}^1_{\rm dR}$, le fibr\'e ${\cal H}^2_{\rm dR}$
est muni de la filtration
$$
{\rm Fil}^1{\cal H}^2_{\rm dR}={\cal H}^2_{\rm dR},\quad
{\rm Fil}^2{\cal H}^1_{\rm dR}=0,$$
et les fibr\'es $W_{k,j}^{\rm dR}$ sont munis de la filtration obtenue par produit
tensoriel: 
$${\rm Fil}^{k-j-r}W_{k,j}^{\rm dR}={\rm Fil}^{k-r}W_{k,0}^{\rm dR},\quad
{\rm Fil}^{k-r}W_{k,0}^{\rm dR}=\omega^{k-r}\otimes {\rm Sym}^r{\cal H}^1_{\rm dR}.$$

\subsubsection{Le motif $\Q(i)$ et ses r\'ealisations}\label{como112.2}
Soit $\Q(1)$ le motif de Tate.  On note $\Q(1)_{\rm B}$, $\Q_p(1)$ et $\Q(1)_{\rm dR}$
ses r\'ealisations de Betti, \'etale $p$-adique, et de de\,Rham.
Alors $\Q(1)_{\rm B}$ et $\Q(1)_{\rm dR}$ sont des $\Q$-espaces de dimension~$1$,
admettant des bases \index{zzeta@\zb}naturelles $\zeta_{\rm B}$ et~$\zeta_{\rm dR}$.
On a des isomorphismes de comparaison:
$$\C\otimes \Q(1)_{\rm B}\cong\C\otimes\Q(1)_{\rm dR},
\quad
\Q_p\otimes \Q(1)_{\rm B}\cong\Q_p(1),
\quad
\bdr\otimes \Q(1)_{\rm B}\cong\bdr\otimes\Q(1)_{\rm dR}$$
et, via ces isomorphismes, $\zeta_{\rm B}$ et $\zeta_{\rm dR}$ sont reli\'es par:
$$1\otimes \zeta_{\rm B}=2i\pi\otimes \zeta_{\rm dR}\ {\text{(dans $\C\otimes \Q(1)_{\rm B}$)}},
\quad
1\otimes \zeta_{\rm B}=t\otimes \zeta_{\rm dR}\ {\text{(dans $\bdr\otimes \Q(1)_{\rm B}$)}}.$$
On note encore $\zeta_{\rm B}$ (ou $\zeta_{\eet}$) l'\'el\'ement $1\otimes \zeta_{\rm B}$ de $\Q_p(1)$;
on a donc $\sigma(\zeta_{\rm B})=
\cyp (\sigma)\zeta_{\rm B}$, si $\sigma\in G_\Q$.

Plus g\'en\'eralement, si $i\geq 0$, on note $\Q(i)$ la puissance tensorielle $i$-i\`eme
de $\Q(1)$ et, si $i\leq 0$, on note $\Q(i)$ le dual de $\Q(-i)$.
Alors $\Q(i)_{\rm B}$ et $\Q(i)_{\rm dR}$ sont des $\Q$-espaces de dimension~$1$ admettant
$\zeta_{\rm B}^i$ et $\zeta_{\rm dR}^i$ comme base.

Notons que $e_1^\dual, e_2^\dual$ est une base anti-orient\'ee du $H_1$
de la courbe elliptique universelle, et donc que 
$$e_1^\dual\wedge e_2^\dual=-\zeta_{\rm B}
\quad{\rm et}\quad
e_1\wedge e_2=-\zeta_{\rm B}^{-1}.$$
Il s'ensuit que 
\begin{align*}
\tfrac{(\tau e_2-e_1)^k}{(e_1\wedge e_2)^j}d\tau&=
-\tfrac{(\tau e_2-e_1)^{k+2}}{(e_1\wedge e_2)^{j+1}}=(-1)^j\zeta_{\rm B}^{j+1}(dz)^{k+2}\\
(-2i\pi)^{k+1-j}\tfrac{(\tau e_2-e_1)^k}{(e_1\wedge e_2)^j}d\tau &=
(-1)^{k+1}((2i\pi)^{-1}\zeta_{\rm B})^{j+1}(2i\pi\,dz)^{k+2}=
(-1)^{k+1}\zeta_{\rm dR}^{j+1}(\tfrac{dt}{t})^{k+2}
\end{align*}

\subsubsection{Cohomologie}\label{como7}
Si ${\rm truc}\in\{{\rm B},{\rm dR},{\eet}\}$,
on note $H^\bullet_{\rm truc}$ la cohomologie correspondante
(la cohomologie \'etale est la g\'eom\'etrique, i.e.~celle
des $X(K)^\times\times\Qbar$,
la structure logarithmique a pour effet de permettre les rev\^etements profinis ramifi\'es en les pointes:
ce qu'on obtient est souvent appel\'e {\it cohomologie log-Kummer-(pro)\'etale}),
$H^\bullet_{{\rm truc},c}$ la cohomologie \`a support compact et
$H^\bullet_{\rm truc,\,par}$ la cohomologie parabolique
(i.e.~l'image de $H^\bullet_{{\rm truc},c}\to H^\bullet_{{\rm truc}}$).
Alors $H^1_{\rm truc,\,par}$ est naturellement un quotient de
$H^1_{{\rm truc},c}$, mais le principe de Manin-Drinfeld fournit un scindage
naturel, ce qui permet de consid\'erer $H^1_{\rm truc,\,par}$ comme un sous-objet de
$H^1_{{\rm truc},c}$.

Les $H^1_{\rm dR\sharp}(X(K)^\times,W^{\rm dR}_{k,j})$, pour $\sharp\in\{\ ,c,{\rm par}\}$,
sont munis d'une filtration induite par celle
de $W^{\rm dR}_{k,j}$ dont le gradu\'e associ\'e n'a que deux termes non nuls
$H^{1,0}_{\rm dR\sharp}(X(K)^\times,W^{\rm dR}_{k,j})$ (sous-objet de $H^1_{\rm dR\sharp}(X(K)^\times,W^{\rm dR}_{k,j})$)
et $H^{0,1}_{\rm dR\sharp}(X(K)^\times,W^{\rm dR}_{k,j})$ (le quotient).

La th\'eorie de Hodge fournit une d\'ecomposition 
(o\`u $w_\infty\hskip.5mm {=}\hskip.5mm \matrice{-1}{0}{0}{1}_\infty$
et $H_\C\hskip.5mm {=}\hskip.5mm \C{\otimes}_\Q H$)
$$H^1_{\rm dR,\,par}(X(K)^\times,W^{\rm dR}_{k,j})_\C=
H^{1,0}_{\rm dR,\,par}(X(K)^\times,W^{\rm dR}_{k,j})_\C\oplus
H^{0,1}_{\rm dR,\,par}(X(K)^\times,W^{\rm dR}_{k,j})_\C,$$
les repr\'esentants harmoniques \'etant $M_{k+2,j+1}^{\rm par}(K,\C)$ pour $H^{1,0}$
et $w_\infty\star M_{k+2,j+1}^{\rm par}(K,\C)$ pour $H^{0,1}$.
On a aussi $\C\otimes H^{1,0}_{\rm dR}(X(K)^\times,W^{\rm dR}_{k,j})=M_{k+2,j+1}(K,\C)$.

On a un diagramme commutatif
d'isomorphismes de comparaison (entre les lignes 1 et 2, il s'agit de la comparaison entre
cohomologies de Betti et de de Rham combin\'ee avec la comparaison entre cohomologie de de Rham alg\'ebrique et
cohomologie de de Rham classique, entre les lignes 2 et 3, c'est la th\'eorie de Hodge (et l'existence
de repr\'esentants harmoniques), entre les 3 et 4, c'est l'isomorphisme d'Eichler-Shimura,
entre les 4 et 1, c'est la comparaison entre la cohomologie du groupe fondamental et celle des syst\`emes locaux):
$$\xymatrix@R=.4cm{
\C\otimes_\Q H^1_{\rm B,par}(X(K)^\times(\C),W_{k,j}^{\rm B})\ar@<5mm>[d]^-{\iota_{\rm B}^{\rm dR}}\\
\C\otimes_\Q H^1_{\rm dR,par}(X(K)^\times,W_{k,j}^{\rm dR})\ar@<5mm>[u]^-{\iota_{\rm dR}^{\rm B}}\ar@<5mm>[d]^-{\iota_{\rm dR}^{\rm H}}\\
M^{\rm par}_{k+2,j+1}(K,\C)\oplus \big(w_\infty \star M^{\rm par}_{k+2,j+1}(K,\C)\big)
\ar@<5mm>[u]^-{\iota_{\rm H}^{\rm dR}}\ar[d]^-{\iota_{\rm ES}}\\
\C\otimes_\Q H^1_{\rm par}(\Gamma(1),{\cal C}(\GG(\cZ)/K,W_{k,j}))\ar@/_70pt/@<-1.5cm>[uuu]
}$$

\subsubsection{Dualit\'e}\label{COM1}
On dispose d'une application trace 
$${\rm Tr}:H^2_{{\rm truc},c}(X(K)^\times,W_{0,0}^{\rm truc})\to
F^{\rm truc}\otimes\zeta_{\rm truc}^{-1}$$ et, si $K\subset K'$, d'un diagramme commutatif
(qui se d\'eduit du diagramme analogue pour les $H^0$):
$$\xymatrix{
H^2_{{\rm truc},c}(X(K)^\times,W_{0,0}^{\rm truc})\ar@<.1cm>[d]^-{\rm cor}\ar[r]^-{\rm Tr} &F^{\rm truc}\otimes\zeta_{\rm truc}^{-1}\ar@<.1cm>[d]^-{\rm id}\\
H^2_{{\rm truc},c}(X(K')^\times,W_{0,0}^{\rm truc})\ar@<.1cm>[u]^-{\rm res}\ar[r]^-{\rm Tr} &F^{\rm truc}\otimes\zeta_{\rm truc}^{-1}\ar@<.1cm>[u]^-{[K':K]}
}$$

En composant le cup-produit avec la projection naturelle $W_{k,k-j}\otimes W_{k,j}\to W_{0,0}$
puis avec l'application trace ci-dessus, cela fournit des dualit\'es:
\begin{align*}
\langle\ ,\ \rangle_{{\rm truc},K}:H^1_{{\rm truc}}(X(K)^\times,W_{k,k-j}^{\rm truc})
\times H^1_{{\rm truc},c}(X(K)^\times,W_{k,j}^{\rm truc})\to F^{\rm truc}\otimes\zeta_{\rm truc}^{-1}\\
\langle\ ,\ \rangle_{{\rm truc},K}:H^1_{{\rm truc,\,par}}(X(K)^\times,W_{k,k-j}^{\rm truc})
\times H^1_{{\rm truc,\,par}}(X(K)^\times,W_{k,j}^{\rm truc})\to F^{\rm truc}\otimes\zeta_{\rm truc}^{-1}
\end{align*}
De plus, on a
\begin{equation}
\langle\mu_K,{\rm res}(\phi_{K'})\rangle_{{\rm truc},K}
=\langle{\rm cor}(\mu_K),\phi_{K'}\rangle_{{\rm truc},K'}\label{dua1}
\end{equation}
pour tous $\mu_K\in H^1_{{\rm truc}}(X(K)^\times,W_{k,k-j}^{\rm truc})$ et
$\phi_{K'}\in H^1_{{\rm truc},c}(X(K')^\times,W_{k,j}^{\rm truc})$.

\subsubsection{Compatibilit\'e des accouplements de dualit\'e}
Soit $N\geq 1$.
On dispose \index{accouplements!{$\langle\ ,\ \rangle$, $\langle\ ,\ \rangle_{\rm B}$, $\langle\ ,\ \rangle_{\rm dR}$}}d'accouplements:
\begin{align*}
&\langle\ ,\ \rangle:H^1(\Gamma(1),{\cal C}(\GG(\Z/N),W_{k,k-j}))\times
H^1_{\rm par}(\Gamma(1),{\cal C}(\GG(\Z/N),W_{k,j}))\to \Q\\
&\langle\ ,\ \rangle_{\rm B}:H^1_{\rm B}(Y(N)(\C),W_{k,k-j}^{\rm B})\times
H^1_{\rm B,par}(Y(N)(\C),W_{k,j}^{\rm B})\to\Q\otimes\zeta_{\rm B}^{-1}\\
&\langle\ ,\ \rangle_{\rm dR}:H^1_{\rm dR}(X(N)^\times,W_{k,k-j}^{\rm dR})\times
H^1_{\rm dR,par}(X(N)^\times,W_{k,j}^{\rm dR})\to\Q\otimes\zeta_{\rm dR}^{-1}\\
&\langle\ ,\ \rangle_{\rm H}:\big(M_{\blacksquare}(N)\oplus w_\infty\star M^{\rm par}_{\blacksquare}(N)\big)\times
\big(M^{\rm par}_{\square}(N)\oplus w_\infty\star M^{\rm par}_{\square}(N)\big)\to\C\\
&\hskip2.3cm \blacksquare=k+2,k+1-j,\quad \square=k+2,j+1
\end{align*}
Les accouplements $\langle\ ,\ \rangle_{\rm B}$ et $\langle\ ,\ \rangle_{\rm dR}$ sont
obtenus via $H^2_c$; l'accouplement $\langle\ ,\ \rangle$ est celui du \S\,\ref{cup3},
et $\langle\ ,\ \rangle_{\rm H}$ est obtenu en \'evaluant le cup-produit des deux formes
diff\'erentielles le long de la classe fondamentale de $H^2_{{\rm B},c}$, i.e.
$$\langle\phi_1 ,\phi_2 \rangle_{\rm H}=\int_{Y(N)(\C)}\phi_1\wedge\phi_2$$
(Quand on d\'eveloppe $\frac{(\overline\tau e_1-e_2)^k}{(e_1\wedge e_2)^{k-j}}d\overline\tau\wedge
\frac{(\tau e_1-e_2)^k}{(e_1\wedge e_2)^{j}}d\tau$ et les expressions analogues \'echangeant
les r\^oles de $\tau$ et $\overline\tau$, on voit appara\^{\i}tre la forme $y^kdxdy$ habituelle
intervenant dans le produit scalaire de Petersson.)

Ces accouplements se correspondent (apr\`es extension des scalaires \`a $\C$)
via les isomorphismes de comparaison (et l'identit\'e $\zeta_{\rm B}=2i\pi\,\zeta_{\rm dR}$):
\begin{align}\label{Compa}
\langle\iota_{\rm ES}(\phi_1),\iota_{\rm ES}(\phi_2)\rangle&=\langle\phi_1,\phi_2\rangle_{\rm H}\\
\langle\iota_{\rm dR}^{\rm B}(\phi_1),\iota_{\rm dR}^{\rm B}(\phi_2)\rangle_{\rm B}\otimes\zeta_{\rm B}&=
2i\pi\,\langle\phi_1,\phi_2\rangle_{\rm dR}\otimes\zeta_{\rm dR} \notag \\
\langle\iota_{\rm H}^{\rm dR}(\phi_1),\iota_{\rm H}^{\rm dR}(\phi_2)\rangle_{\rm dR}\otimes\zeta_{\rm dR}&=
\tfrac{1}{2i\pi}\langle\phi_1,\phi_2\rangle_{\rm H}\notag
\end{align}

On en d\'eduit la formule explicite suivante.
\begin{lemm}\phantomsection\label{shi2}
Soient $\Phi_1\in H^1_{{\rm dR},{\rm par}}(X(N)^\times,W^{\rm dR}_{k,j})$ 
et $\Phi_2\in H^1_{\rm dR}(X(N)^\times,W^{\rm dR}_{k,k-j})$, 
et soient $\Phi_1^{\rm Har}=\iota_{\rm dR}^{\rm H}(\Phi_1)$ 
et $\Phi_2^{\rm Har}=\iota_{\rm dR}^{\rm H}(\Phi_2)$
les repr\'esentants harmoniques de $\Phi_1$ et~$\Phi_2$. 
Alors
\begin{align*}
\langle \Phi_1,& \Phi_2\rangle_{{\rm dR},Y(N)}\otimes\zeta_{\rm dR}
=\tfrac{1}{2i\pi}\int_{Y(N)(\C)}\Phi^{\rm Har}_1\wedge\Phi^{\rm Har}_2\\
=&\ \tfrac{1}{2i\pi}
\int_{\Gamma(N)\backslash{\cal H}^+}\sum_{a\in(\Z/N)^\dual}(\Phi^{\rm Har}_1\wedge\Phi^{\rm Har}_2)(\tau,\matrice{1}{0}{0}{a}^{]\infty[})\\
=&\ \tfrac{1}{2i\pi}
\int_{\Gamma(N)\backslash{\cal H}^+}\sum_{a\in(\Z/N)^\dual}
\Big(\big(\matrice{1}{0}{0}{a}^{]\infty[}\star\Phi^{\rm Har}_1\big)\wedge\big(\matrice{1}{0}{0}{a}^{]\infty[}\star\Phi^{\rm Har}_2\big)\Big)
(\tau,1^{]\infty[})
\end{align*}
\end{lemm}

\Subsection{Cohomologie de la tour des courbes modulaires}

Si ${\rm truc}\in\{{\rm B},{\rm dR},{\eet}\}$, et si $\sharp\in\{\ ,c,{\rm par}\}$, 
\index{H2@\IH}posons:
\begin{align*}
&\pH^1_{\rm truc\sharp}(W^{\rm truc}_{k,j}):=\varprojlim_K
H^1_{\rm truc\sharp}(X(K)^\times, W^{\rm truc}_{k,j}),\\
&\iH^1_{{\rm truc}\sharp}(W^{\rm truc}_{k,j}):=\varinjlim_K
H^1_{{\rm truc}\sharp}(X(K)^\times, W^{\rm truc}_{k,j}),
\end{align*}
o\`u $K$ parcourt les sous-groupes ouverts compacts de $\GG(\A^{]\infty[})$ (on peut
se contenter des $\wGamma(N)$ puisque ceux-ci forment un syst\`eme cofinal);
les fl\`eches de transition sont les corestrictions pour $\pH$
et les restrictions pour $\iH$.

De m\^eme, si $S$ est un ensemble fini de nombres premiers, on pose
\begin{align*}
&\pH^1_{\rm truc\sharp}(W^{\rm truc}_{k,j})_S:=\varprojlim_K
H^1_{\rm truc\sharp}(X(K)^\times, W^{\rm truc}_{k,j}),\\
&\iH^1_{{\rm truc}\sharp}(W^{\rm truc}_{k,j})_S:=\varinjlim_K
H^1_{{\rm truc}\sharp}(X(K)^\times, W^{\rm truc}_{k,j}),
\end{align*}
o\`u $K$ parcourt les sous-groupes ouverts compacts de $\GG(\Q_S)$.

Les $\iH^1_{\rm truc\sharp}(W^{\rm truc}_{k,j})$  
et les $\pH^1_{\rm truc\sharp}(W^{\rm truc}_{k,j})$  
(resp.~les $\iH^1_{\rm truc\sharp}(W^{\rm truc}_{k,j})_S$
et les $\pH^1_{\rm truc\sharp}(W^{\rm truc}_{k,j})_S$) sont munis
d'actions de $\GG(\A^{]\infty[})$ (resp.~de $\GG(\Q_S)$), lisses pour les $\iH^1$.

\subsubsection{Comparaison avec la cohomologie de $\GG(\Q)$}\label{como8}
On a des isomorphismes commutant aux actions de
$\GG(\A^{]\infty[})$, 
pour $\sharp\in\{\ ,c,{\rm par}\}$
et $\star\in\{\ ,S\}$,
\begin{align*}
\Q_p\otimes_\Q \iH^1_{{\rm B}\sharp}(W^{\rm B}_{k,j})_\star\cong&\  
\iH^1_{{\eet}\sharp}(W^{\eet}_{k,j})_\star,\\
\C\otimes_\Q \iH^1_{{\rm B}\sharp}(W^{\rm B}_{k,j})_\star\cong&\  
\C\otimes_\Q \iH^1_{{\rm dR}\sharp}(W^{\rm dR}_{k,j})_\star,\\
\bdr\otimes_\Q \iH^1_{{\rm B}\sharp}(W^{\rm B}_{k,j})_\star\cong&\  
\bdr\otimes_\Q \iH^1_{{\rm dR}\sharp}(W^{\rm dR}_{k,j})_\star.
\end{align*}
Et de m\^eme pour $\pH^1$.
Pour m\'emoire, rappelons les isomorphismes d\'ej\`a mentionn\'es.
\begin{align*}
\iH^1_{{\rm B}\sharp}(W^{\rm B}_{k,j})\cong&\  H^1_\sharp(\GG(\Q),{\rm LC}(\GG(\A^{]\infty[}),\Q)\otimes W_{k,j}),\\
\iH^1_{{\rm B}\sharp}(W^{\rm B}_{k,j})_S\cong&\  H^1_\sharp(\GG(\Z[\tfrac{1}{S}]),{\rm LC}(\GG(\Q_S),\Q)\otimes W_{k,j}).
\end{align*}

\subsubsection{Dualit\'e}
La formule~(\ref{dua1}) montre que l'accouplement passe \`a la limite:
si
\begin{align*}
\mu=(\mu_K)_K&\in \varprojlim
H^1_{\rm truc}(X(K)^\times, W^{\rm truc}_{k,k-j})\\
 \phi=(\phi_K)_K&\in \varinjlim
H^1_{{\rm truc},c}(X(K)^\times, W^{\rm truc}_{k,j})
\end{align*}
on pose
\begin{equation}\label{limite}
\langle\mu,\phi\rangle_{\rm truc}:=\lim\nolimits_K\, \langle\mu_K,\phi_K\rangle_{{\rm truc},K}
\end{equation}
(la suite est en fait stationnaire pour $K$ assez petit, et il suffit de prendre $K$
de la forme $\wGamma(N)$).
L'accouplement
$$\langle\ ,\ \rangle_{\rm truc}:
\pH^1_{\rm truc}(W^{\rm truc}_{k,k-j})_\star\times \iH^1_{{\rm truc},c}(W^{\rm truc}_{k,j})_\star
\to F^{\rm truc}\otimes\zeta_{\rm truc}^{-1}$$
ainsi d\'efini commute aux actions de $\GG(\A^{]\infty[})$ ou $\GG(\Q_S)$ 
(agissant trivialement sur~$\zeta_{\rm truc}$),
et est une dualit\'e si $\star\in\{\ ,S\}$.

Il en est de m\^eme de sa restriction \`a
$\pH^1_{\rm truc,\,par}(W^{\rm truc}_{k,k-j})_\star\times 
\iH^1_{{\rm truc},\,{\rm par}}(W^{\rm truc}_{k,j})_\star$.

\subsubsection{Formes modulaires quasi-holomorphes g\'eom\'etriques}\label{como4}
\index{ome@\ome}Soit
$$\omega^{k,j}=\omega^k\otimes({\cal H}^2_{\rm dR})^{-j}.$$
C'est un fibr\'e en droites engendr\'e, localement pour la topologie analytique, 
par $\zeta_{\rm dR}^j(\frac{dt}{t})^k=
(-1)^j(2i\pi)^{k-j}\frac{(dz)^k}{(e_1\wedge e_2)^j}=
(-1)^j(2i\pi)^{k-j}\tfrac{(\tau e_2-e_1)^k}{(e_1\wedge e_2)^j}$.
\index{H2@\IH}Posons
$$\iH^0(\omega^{k,j})=\varinjlim H^0(X(K)^\times,\omega^{k,j})=\varinjlim H^0(X(N)^\times,\omega^{k,j}),$$
o\`u $K$ d\'ecrit les sous-groupes ouverts compacts de $\GG(\A^{]\infty[})$ et $N$ les entiers~$\geq 1$.
On a (en notant $\O^{\rm temp}({\cal H}^+)$ les fonctions holomorphes sur ${\cal H}^+$, \`a croissance lente
\`a l'infini, i.e.~aux pointes)
\begin{align*}
\C\otimes_\Q H^0(X(N)^\times,\omega^{k,j})&
=\C\otimes_\Q H^0(X(1)^\times,{\rm LC}(\GG(\Z/N),\Z)\otimes\omega^{k,j})\\
&= H^0(\Gamma(1),{\rm LC}(\GG(\Z/N),\Z)\otimes \O^{\rm temp}({\cal H}^+)\tfrac{(\tau e_2-e_1)^k}{(e_1\wedge e_2)^j})
\end{align*}
On en d\'eduit que (le passage de $\Gamma(1)$ \`a $\GG(\Q)$ se fait via le lemme de Shapiro)
\begin{align*}
\C\otimes_\Q \iH^0(\omega^{k,j})&=
H^0(\Gamma(1),{\rm LC}(\GG(\cZ),\Z)\otimes \O^{\rm temp}({\cal H}^+)\tfrac{(\tau e_2-e_1)^k}{(e_1\wedge e_2)^j})\\
&= H^0(\GG(\Q),{\cal A}^+\otimes\tfrac{(\tau e_2-e_1)^k}{(e_1\wedge e_2)^j})=M_{k,j}(\C)
\end{align*}
Plus g\'en\'eralement,
$$\C\otimes_\Q\iH^0({\rm Fil}^{k-j-r}W_{k,j}^{\rm dR})=M^{\rm qh}_{k,j,r}(\C).$$

\begin{prop}\phantomsection\label{como5}
L'image de $\iH^0(\omega^{k,j})$ dans $M_{k,j}(\C)$ est $(2i\pi)^{k-j}M_{k,j}(\Q)$.
\end{prop}
\begin{proof}
Le $(2i\pi)^{k-j}$ vient du passage de $\zeta_{\rm dR}^j(\frac{dt}{t})^k$ \`a
$\tfrac{(\tau e_2-e_1)^k}{(e_1\wedge e_2)^j}$.  
Maintenant,
si $g\in H^0(X(N)^\times,\omega^k)$, 
on note $g_b=\sum_{n\in\frac{1}{N}\N}c_{b,n}(g)q^n$ son $q$-d\'eveloppement
en $\infty_b$.  
Si $\phi_g\in M_{k,j}(\C)$ correspond \`a $(2i\pi)^{j-k}g$, on a
${\cal K}(\phi_g,n^{]\infty[}b)=c_{b,n}(g)$, si $n>0$ et $b\in\cZ^\dual$ (en d\'efinissant
$c_{b,n}(g)$ comme $c_{\overline b,n}(g)$, si $b$ a comme image $\overline b$ dans $(\Z/N)^\dual$).
Or $\sigma_a(g_b)=g_{ab}$ (i.e.~$\sigma_a(c_{b,n}(g))=c_{ab,n}(g)$, pour tout $n\in\frac{1}{N}\N$),
si $a\in\cZ^\dual$, puisque $\sigma_a(\infty_b)=\infty_{ab}$.
Il s'ensuit que $\sigma_a({\cal K}(\phi_g,u))={\cal K}(\phi_g,au)$,
si $a\in\cZ^\dual$ et $u\in\Aidu$, ce qui permet de conclure.
\end{proof}

\Subsection{La conjecture $C_{\rm dR}$ pour les formes modulaires}\label{qq17}
\subsubsection{Formes modulaires quasi-holomorphes g\'eom\'etriques}\label{qq18}
On pourra consulter~\cite{urban} pour des compl\'ements sur ce qui suit.

Compte-tenu de l'isomorphisme de Kodaira-Spencer 
$\omega^{2,1}\overset{\sim}{\longrightarrow}\Omega^1$,
on peut voir la
connexion de Gauss-Manin comme
une \index{nab@\nabl}connexion
$$\nabla:W_{k,j}^{\rm dR}\to W_{k+2,j+1}^{\rm dR}.$$
En passant aux sections globales,
on obtient une d\'ecomposition
$$\iH^0(W_{k+2,j+1}^{\rm dR})=\nabla\cdot \iH^0(W_{k,j}^{\rm dR})\oplus
\iH^0(\omega^{k+2,j+1}),$$
Autrement dit, l'application naturelle
$$\iH^0(\omega^{k+2,j+1})\to \frac{\iH^0(W_{k+2,j+1}^{\rm dR})}{\nabla\cdot \iH^0(W_{k,j}^{\rm dR})}$$
est un isomorphisme (tout ceci est d\'ej\`a vrai \`a un niveau fini).
La \index{Hol@\hol}projection $${\rm Hol}:\iH^0(W_{k+2,j+1}^{\rm dR})\to \iH^0(\omega^{k+2,j+1})$$
qui
se d\'eduit de la d\'ecomposition ci-dessus est la 
{\it projection holomorphe} (si on \'etend les scalaires \`a $\C$, on obtient la projection
orthogonale de $M^{\rm qh}_{k+2,j+1}(\C)$ sur $M_{k+2,j+1}(\C)$ 
pour le produit scalaire de Petersson).

Les $\iH^i_{\rm dR}(W_{k,j}^{\rm dR})$ sont les groupes d'hypercohomologie
du complexe de faisceaux $W_{k,j}^{\rm dR}\overset{\nabla}{\longrightarrow} W_{k+2,j+1}^{\rm dR}$.
La filtration de Hodge qui se d\'eduit de la d\'efinition prend la forme de la suite
exacte suivante:
\begin{equation}\label{cano109}
0\to \frac{\iH^0(W_{k+2,j+1}^{\rm dR})}{\nabla\cdot \iH^0(W_{k,j}^{\rm dR})}
\to \iH^1_{\rm dR}(W_{k,j}^{\rm dR})\to \iH^1(W_{k,j}^{\rm dR})^{\nabla=0}\to 0
\end{equation}

\subsubsection{Cohomologie pro\'etale et formes quasi-holomorphes}\label{qq19}
\index{bdr@\BDR}Notons $\Bdr$, $\Bdr^+$, $\O\Bdr$ et $\O\Bdr^+$ les faisceaux pro\'etales en anneaux de Fontaine
relatifs, et $\bdr^+$ l'anneau $\Bdr^+(\C_p)$.
Le th\'eor\`eme de comparaison basique de Scholze~\cite{Sz1,DLLZ} fournit un isomorphisme
$$\bdr^+\otimes \iH^1_{\eet}(W_{k,j}^{\eet})\overset{\sim}{\to}
\iH^1_{\eet}(\Bdr^+\otimes W_{k,j}^{\eet}).$$
La conjecture $C_{\rm dR}$ pour les formes modulaires r\'esulte alors
du calcul du membre de droite:
$$\iH^1_{\eet}(\Bdr^+\otimes W_{k,j}^{\eet})\cong{\rm Fil}^0(\bdr\otimes \iH^1_{\rm dR}(W_{k,j}^{\rm dR})).$$
Nous aurons besoin d'expliciter une partie de cet isomorphisme.
On note $$\nabla:\O\Bdr^+\to\O\Bdr^+\otimes_\O\Omega^1$$ la connexion
$\Bdr^+$-lin\'eaire, co\"{\i}ncidant avec $d$ sur $\O$.
L'application induite par $\nabla$ sur
$\bdr\otimes H^0(U, W_{k,j}^{\rm dR})$, si $U$ est un ouvert d'une courbe modulaire,
 n'est autre que la connexion de Gauss-Manin.
On a une suite exacte de faisceaux pro\'etales
$$0\to\Bdr^+\to {\rm Fil}^0(\O\Bdr)\overset{\nabla}{\to}{\rm Fil}^0(\O\Bdr\otimes\Omega^1)\to 0.$$
Localement, on a
\begin{align*}
H^0_{\proet}(U_C,{\rm Fil}^0(\O\Bdr)\otimes_{\Q_p} W_{k,j}^{\eet})&={\rm Fil}^0(\bdr\otimes_{\Q}
H^0(U, W_{k,j}^{\rm dR}))\\
 H^1_{\proet}(U_C,{\rm Fil}^0(\O\Bdr)\otimes_{\Q_p} W_{k,j}^{\eet})&=0\\
H^0_{\proet}(U_C,{\rm Fil}^0(\O\Bdr\otimes_{\O}\Omega^1)\otimes_{\Q_p} W_{k,j}^{\eet})&={\rm Fil}^0(\bdr\otimes_{\Q}
H^0(U, W_{k,j}^{\rm dR}\otimes\Omega^1)),\\
H^1_{\proet}(U_C,{\rm Fil}^0(\O\Bdr\otimes_{\O}\Omega^1)\otimes_{\Q_p} W_{k,j}^{\eet})&=0.
\end{align*}
On en d\'eduit des isomorphismes
\begin{align*}
\iH^0_{\proet}({\rm Fil}^0(\O\Bdr)\otimes_{\Q_p} W_{k,j}^{\eet})&={\rm Fil}^0(\bdr\otimes_{\Q}
\iH^0(W_{k,j}^{\rm dR}))\\
\iH^0_{\proet}({\rm Fil}^0(\O\Bdr\otimes_\O\Omega^1)\otimes_{\Q_p} W_{k,j}^{\eet})
&={\rm Fil}^0(\bdr\otimes_{\Q} \iH^0(W_{k+2,j+1}^{\rm dR}))\\
\iH^1_{\proet}({\rm Fil}^0(\O\Bdr)\otimes_{\Q_p} W_{k,j}^{\eet})&={\rm Fil}^0(\bdr\otimes_{\Q}
\iH^1(W_{k,j}^{\rm dR}))\\
\iH^1_{\proet}({\rm Fil}^0(\O\Bdr\otimes_\O\Omega^1)\otimes_{\Q_p} W_{k,j}^{\eet})&={\rm Fil}^0(\bdr\otimes_{\Q}
\iH^1(W_{k+2,j+1}^{\rm dR}))
\end{align*}
et une suite exacte que l'on pourra comparer avec~(\ref{cano109}).
$$0\to 
\frac{{\rm Fil}^0(\bdr\otimes
\iH^0(W_{k+2,j+1}^{\rm dR}))}{\nabla\cdot {\rm Fil}^0(\bdr\otimes
\iH^0(W_{k,j}^{\rm dR}))}\to
\iH^1_{\proet}(\Bdr^+\otimes_{\Q_p} W_{k,j}^{\rm et})\to
{\rm Fil}^0(\bdr\otimes \iH^1(W_{k,j}^{\rm dR}))^{\nabla=0}\to 0.$$

\subsubsection{L'application d'Eichler-Shimura $p$-adique}\label{qq21}
La \index{iotaES@\iotaES}fl\`eche 
$$\iota_{{\rm ES},p}:
\frac{{\rm Fil}^0(\bdr\otimes
\iH^0(W_{k+2,j+1}^{\rm dR}))}{\nabla\cdot {\rm Fil}^0(\bdr\otimes
\iH^0(W_{k,j}^{\rm dR}))}
\to \iH^1_{\rm proet}(\Bdr^+\otimes_{\Q_p} W_{k,j}^{\eet})$$
admet la description suivante.

Soit $K$ un sous-groupe ouvert de $\GG(\cZ)$.
Soit $f\in {\rm Fil}^0(\bdr\otimes H^0(X(K)^\times,W_{k+2,j+1}^{\rm dR}))$,
et soit $U$ un ouvert affino\"{\i}de de $X(K)$, de $\pi_1$ g\'eom\'etrique $G_U$.
Si $U$ est assez petit, on peut r\'esoudre l'\'equation $\nabla g_U=f$ sur $U$,
avec 
$$g_U\in {\rm Fil}^0(\O\Bdr\otimes  W_{k,j}^{\rm dR})= ({\rm Fil}^0\O\Bdr)\otimes W_{k,j}^{\eet}.$$
Si $\sigma\in G_{U}$, alors $\nabla((\sigma-1)g_U)=(\sigma-1)\cdot\nabla g_U=(\sigma-1)f=0$,
et donc $\sigma\mapsto (\sigma-1)g_U$ est un $1$-cocycle sur $G_{U}$ \`a valeurs
dans $H^0(U_C,\Bdr^+\otimes W_{k,j}^{\eet})$ (car $({\rm Fil}^0\O\Bdr)\cap\Bdr=\Bdr^+$).
Comme $g_U$ est unique \`a addition pr\`es d'un \'el\'ement
de $H^0(U_C,\Bdr^+\otimes W_{k,j}^{\eet})$, la classe $[f]_U$ de ce cocycle
dans $H^1_{\proet}(U_C,\Bdr^+\otimes W_{k,j}^{\eet})$ ne d\'epend que de $f$.
Pour les m\^emes raisons, si $U_1,U_2$ sont des affino\"{\i}des,
les restrictions de $[f]_{U_1}$ et $[f]_{U_2}$ \`a ${U_1\cap U_2}$ co\"{\i}ncident
(car $\nabla g_{U_1}-\nabla g_{U_2}=0$), et donc les
$[f]_U$ se recollent pour donner naissance \`a un \'el\'ement $\iota_{{\rm ES},p}(f)$ de
$H^1_{\proet}(X(K)^\times_C,\Bdr^+\otimes W_{k,j}^{\eet})$.

\subsubsection{$q$-d\'eveloppement}\label{qq22}
Si $q$ d\'esigne le param\`etre naturel du voisinage $Z(1)$ de la pointe $\infty$ (cf.~\no\ref{como2.9}),
 la courbe elliptique universelle est ${\bf G}_m/q^\Z$, et $e_1^\dual,e_2^\dual$ se sp\'ecialisent
en la base
naturelle sur $\cZ$ du module de Tate:
$e_1^\dual=(\zeta_N^{-1})_N$ et $e_2^\dual=(q^{1/N})_N$.
L'action du $\pi_1$ g\'eom\'etrique se factorise \`a travers
${\rm Aut}(Z(0)_C/Z(1)_C)\cong \UU(\cZ)$, et on peut
choisir le g\'en\'erateur $\gamma:=\matrice{1}{1}{0}{1}$ 
de $\UU(\cZ)$ de telle sorte que $\gamma(e_2^\dual)=e_2^\dual-e_1^\dual$
(et on a $\gamma(e_1^\dual)=e_1^\dual$).
Comme $e_1,e_2$ est la base duale
de $e_1^\dual,e_2^\dual$, on a $\gamma(e_1)=e_1+e_2$ et $\gamma(e_2)=e_2$.

\index{q@\tq}Soient 
$$\tilde q=[(q,q^{1/p},q^{1/p^2},\dots)]\in H^0(Z(1),\Bdr^+)
\quad{\rm et}\quad
u=\log({q}/{\tilde q})\in H^0(Z(1),\O\Bdr^+).$$
On a $$\gamma(u)=u+t\quad{\rm et}\quad \nabla u=du=\tfrac{dq}{q}.$$
\index{v2@\vv}Soient 
$$v_1=ue_2-te_1\quad{\rm et}\quad v_2=e_2,$$ 
de telle sorte que $v_1,v_2$ est une base sur $\O(Z(1))$ de
$H^0(Z(1),W_{1,0}^{\rm dR})$ (vu comme sous-module de 
$H^0_{\proet}(Z(1)_C,\O\Bdr^+\otimes W_{1,0}^{\rm et})$) et que $v_1$ est une base de ${\rm Fil}^1$.
Si $T$ d\'esigne la variable sur ${\bf G}_m$, alors $\frac{dT}{T}$
est un g\'en\'erateur de $H^0(Z(1),{\rm Fil}^1W_{1,0}^{\rm dR}))$,
et on a $v_1=\frac{dT}{T}$ (au signe pr\`es... suivant la normalisation
dans l'accouplement des p\'eriodes $p$-adiques).
Par ailleurs (note~\ref{Koda}),
\begin{equation}\label{qq22.1}
du=\tfrac{dq}{q}=\zeta_{\rm dR}(\tfrac{dT}{T})^2=\zeta_{\rm dR}v_1^2,
\quad \zeta_{\rm dR}=-t^{-1}\zeta_{\rm B}=-t^{-1}(e_1\wedge e_2)^{-1}.
\end{equation}

Si $f$ est une forme modulaire quasi-holomorphe
g\'eom\'etrique,
 de poids $(k{+}2,j{+}1)$,
i.e.~$f\in\iH^0(W_{k+2,j+1}^{\rm dR})$, 
la restriction de $f$ \`a $Z(0)_C$ peut s'\'ecrire, de mani\`ere unique,
sous la forme $$f=\sum_{i=0}^kf_i\otimes v_1^{k+2-i}v_2^i\zeta_{\rm dR}^{j+1},$$
avec $f_i\in\O(Z(0)_C)=C\otimes_{\O_C}\varinjlim_N\O_C[[q^{1/N}]]$.

\subsubsection{Projection holomorphe}\label{qq23}

Soit $\partial =\frac{\nabla}{du}$.  On a 
$$\partial v_1=v_2=e_2,\quad \partial v_2=0,\quad
\partial f_i=\partial_q f_i,{\text{ avec $\partial_q=q\frac{d}{dq}$.}}$$
\begin{prop}\phantomsection\label{bato14}
Si $f=\sum_{i=0}^kf_i\otimes v_1^{k+2-i}v_2^i\zeta_{\rm dR}^{j+1}\in \iH^0(W_{k+2,j+1}^{\rm dR})$,
la projection holomorphe ${\rm Hol}(f)$ de $f$ sur $\iH^0(\omega^{k+2,j+1})$ est
$${\rm Hol}(f)=\big(\sum_{i=0}^k(-1)^i\tfrac{(k-i)!}{k!}\partial_q^if_i\big)\otimes  v_1^{k+2}\zeta_{\rm dR}^{j+1}.$$
\end{prop}
\begin{proof}
Comme $du=\zeta_{\rm dR}v_1^2$,
On en d\'eduit que 
$$\nabla(g\otimes v_1^{k+1-i}v_2^{i-1}\zeta_{\rm dR}^{j})=
\partial_qg \otimes v_1^{k-i+3}v_2^{i-1}\zeta_{\rm dR}^{j+1}+(k+1-i)g\otimes v_1^{k-i+2}v_2^{i}\zeta_{\rm dR}^{j+1},$$
Une r\'ecurrence imm\'ediate montre que,
modulo l'image de $\nabla$, 
$$g \otimes v_1^{k-i+2}v_2^i\zeta_{\rm dR}^{j+1}=
\tfrac{-1}{k+1-i}\partial_qg\otimes v_1^{k-i+3}v_2^{i-1}\zeta_{\rm dR}^{j+1}
=(-1)^i\tfrac{(k-i)!}{k!}\partial_q^ig \otimes v_1^{k+2}\zeta_{\rm dR}^{j+1}.$$
Le r\'esultat s'en d\'eduit.
\end{proof}

\section{Vecteurs localement alg\'ebriques de la cohomologie compl\'et\'ee}\label{chapi2}
Dans ce chapitre, on rappelle les r\'esultats de Carayol~\cite{cara} et Emerton~\cite{Em06,Em06b} 
sur la d\'ecomposition de l'espace des vecteurs
localement alg\'ebriques de la cohomologie compl\'et\'ee en somme directe de repr\'esentations de $\GG(\A)$,
et on explicite le dictionnaire entre formes modulaires primitives et repr\'esentations de $\GG(\A)$. 

On rappelle que l'on a fix\'e des plongements $\Qbar\hookrightarrow\C$ et $\Qbar\hookrightarrow\Qbar_\ell$,
pour tout $\ell$.
\Subsection{Repr\'esentations attach\'ees \`a une forme primitive}\label{chapi2.5}
Soit $f\in M_{k+2,j+1}^{\rm par,\,cl}(\Gamma_0(N),\chi)$ primitive (cf.~\no\ref{fma2}),
et soient $f=\sum_{n\geq 1}a_nq^n$ son $q$-d\'eveloppement et $\Q(f):=\Q(a_n,\,n\geq 1)\subset \C$;
alors $\Q(f)$ est un corps de nombres muni d'un plongement dans $\Qbar\subset\C$.

Soit $\Q_p(f)\subset\Qbar_p$ le compos\'e de $\Q_p$ et $\Q(f)$.
On peut associer \`a $f$ une repr\'esentation $\pi_{f,j+1}$ de ${\rm GL}_2(\A^{]\infty[})$ et,
si $S$ contient les nombres premiers divisant $Np$,
une repr\'esentation $\rho_{f,j+1}:G_{\Q,S}\to {\rm GL}_2(\Q_p(f))$.
La repr\'esentation $\rho_{f,j+1}$ est \index{rhof@\RHf}caract\'eris\'ee par le fait que
$$\det(1-\ell^{-s}\rho_{f,j+1}(\sigma_\ell^{-1}))=1-a_\ell \ell^{j-k-1-s}+\chi(\ell)\ell^{2j-k-1-2s},
\quad{\text{si $\ell\nmid Np$.}}$$
La repr\'esentation $\pi_{f,j+1}$ est \index{pif@\pif}celle engendr\'ee par les $g\star \phi_{f,j+1}$
(cf.~\no\ref{class}),
pour $g\in \GG(\A^{]\infty[})$.
L'application associant \`a $\phi\in\pi_{f,j+1}$ la fonction ${\cal K}_\phi:={\cal K}(\phi,\ )$
du lemme~\ref{fma11.2}
permet d'identifier $\pi_{f,j+1}$ \`a son mod\`ele de Kirillov.
Le nouveau vecteur de ce mod\`ele de Kirillov est la fonction $v_{f,j+1}$
d\'efinie par
$$v_{f,j+1}(x)=\begin{cases}
0 & {\text{si $x\notin\cZ$,}}\\
n^{j-k-1}a_n & {\text{si $x\in n\cZ^\dual$,
et $n$ est un entier~$\geq 1$.}}
\end{cases}$$

La repr\'esentation $\pi_{f,j+1}$
se d\'ecompose sous la forme d'un produit tensoriel restreint
$$\pi_{f,j+1}=\otimes'_\ell\pi_{f,j+1,\ell},$$
et $v_{f,j+1}$ comme le produit tensoriel des nouveaux vecteurs $v_{f,j+1,\ell}$
des $\pi_{f,j+1,\ell}$:
$$v_{f,j+1}=\otimes_\ell v_{f,j+1,\ell},\quad
{\text{i.e.~$v_{f,j+1}(u)=\prod_\ell v_{f,j+1,\ell}(u_\ell)$, si $u=(u_\ell)_\ell$.}}$$
Si $\ell\nmid N$,
alors $$\pi_{f,j+1,\ell}\cong{\rm Ind}_{\BB(\Q_\ell)}^{\GG(\Q_\ell)}
(\chi_{\ell,1}\otimes|\ |_\ell^{-1}\chi_{\ell,2}),$$
o\`u $\chi_{\ell,1}$ et $\chi_{\ell,2}$ sont des caract\`eres
non ramifi\'es de $\Q_\ell^\dual$, et
$$(1-\chi_{\ell,1}(\ell)X) (1-\chi_{\ell,2}(\ell)X)=
1-\ell^{j-k-1}a_\ell X+\chi(\ell)\ell^{2j-k-1}X^2.$$
On a $(\rho_{f,j+1})_{|G_{\Q_\ell}}=\chi_{\ell,1}\oplus\chi_{\ell,2}$.
Plus g\'en\'eralement, $\Pi^{\rm cl}_\ell(\rho_{f,j+1,\ell})=\pi_{f,j+1,\ell}$, pour tout~$\ell$.

On a $$L(\rho_{f,j+1},s)=L(\pi_{f,j+1},s)=L(v_{f,j+1},s)=L(f,s+k+1-j)$$
(o\`u $k+1-j=(k+2)-(j+1)$).
Donc $$\rho_{f,j+1}=\rho_f\otimes\cyp ^{k+1-j}.$$
Les poids de Hodge-Tate de $\rho_{f,j+1}$ sont $-j=-(k+1)+(k+1-j)$ et $k+1-j$.

\begin{theo}\phantomsection\label{caray1}
{\rm (Carayol~\cite{cara})}
$${\rm Hom}_{\GG(\A^{]\infty[})}(\pi_{f,j+1},\iH^1_{{\eet},c}(W_{k,j}^{\eet}))=\rho_{f,j+1}^\dual.$$
\end{theo}

\Subsection{Repr\'esentations cohomologiques}\label{como9}
Soit $\pi$ une $\Q$-repr\'esentation de $\GG(\A^{]\infty[})$, lisse, irr\'eductible.
On dit que {\it $\pi$ est cohomologique} 
s'il existe $(k,j)$
tel \index{mpi@\mpi}que 
$$m(\pi):={\rm Hom}_{\Q[\GG(\A^{]\infty[})]}(\pi,H^1_{\rm par}(\GG(\Q),{\rm LC}(\GG(\A),\Q)\otimes W_{k,j}))\neq 0.$$
Notons que cela implique que {\it $\pi$ est cuspidale},
que ${\rm End}_{\GG(\A)}\pi$ est un corps de nombres $\Q(\pi)$,
 et que $\pi$ admet un caract\`ere 
\index{omegapi@\omegapi}central $\omega_\pi$ \`a valeurs dans $\Q(\pi)^\dual$.
Le couple $(k,j)$ est alors uniquement d\'etermin\'e et on dit que
{\it $\pi$ est cohomologique de poids $(k+2,j+1)$}.

\subsubsection{Multiplicit\'es des repr\'esentations lisses}\label{como10}
Soit $\pi$ cohomologique de poids $(k+2,j+1)$. On fixe un plongement $\Q(\pi)\hookrightarrow \Qbar$ et on note
$\Q_p(\pi)\subset\Qbar_p$ le compos\'e de $\Q_p$ et $\Q(\pi)$.

On rappelle que $F^{\rm truc}=\Q$ si ${\rm truc}\in\{{\rm B},{\rm dR}\}$ et $F^{\rm truc}=\Q_p$ si ${\rm truc}={\eet}$.
Si
${\rm truc}\in\{{\rm B},{\rm dR},\eet\}$, on pose
$$m_{\rm truc}(\pi)={\rm Hom}_{F^{\rm truc}[\GG(\A^{]\infty[})]}(\pi, \iH^1_{{\rm truc,\,par}}(W^{\rm truc}_{k,j})).$$
(Comme $\pi$ est cuspidale, on obtiendrait le m\^eme module en rempla\c{c}ant la cohomologie
parabolique par l'usuelle ou celle \`a support compact.)

Les $m_{\rm truc}(\pi)$ sont des $F^{\rm truc}(\pi)$-modules gr\^ace \`a l'action de $\Q(\pi)$ sur $\pi$,
et ils sont de rang~$2$.
On a des isomorphismes de comparaison:
\begin{align*}
m_{\rm B}(\pi)=&\ m(\pi)\\
m_{\eet}(\pi)=&\ \Q_p(\pi)\otimes_{\Q(\pi)}m_{\rm B}(\pi),\\
\C\otimes_{\Q(\pi)}m_{\rm B}(\pi)\cong &\ \C\otimes_{\Q(\pi)} m_{\rm dR}(\pi),\\
\bdr\otimes_{{\Q(\pi)}}m_{\rm B}(\pi)\cong &\ \bdr\otimes_{\Q(\pi)} m_{\rm dR}(\pi).
\end{align*}
En particulier, $m_{\rm dR}$ est un $\Q(\pi)$-module filtr\'e
de rang~$2$: 
on note $m_{\rm dR}^+(\pi)$ le cran non trivial de la filtration sur $m_{\rm dR}(\pi)$,
$$m_{\rm dR}^+(\pi)={\rm Hom}_{\Q[\GG(\A^{]\infty[})]}(\pi, \iH^0(\omega^{k+2,j+1})),$$
et $m_{\rm dR}^-(\pi)$ le quotient, de telle sorte que le gradu\'e associ\'e est
$${\rm gr}(m_{\rm dR}(\pi))=m_{\rm dR}^+(\pi)\oplus m_{\rm dR}^-(\pi).$$
De m\^eme,
$m_{\eet}(\pi)$ est \index{mpi@\mpi}une repr\'esentation de $ G_\Q$
de rang~$2$ sur $\Q_p(\pi)$;
sa restriction \`a $ G_{\Q_p}$ est de\,Rham \`a poids de Hodge-Tate $j$ et $j-k-1$,
et 
$$D_{\rm dR}(m_{\eet}(\pi))=\Q_p(\pi)\otimes_{\Q(\pi)} m_{\rm dR}(\pi)$$ 
en tant que $\Q_p(\pi)$-module filtr\'e.

La th\'eorie de Hodge fournit un scindage naturel de la filtration de Hodge, apr\`es extension
des scalaires \`a $\C$:
$$\C\otimes_{\Q(\pi)} m_{\rm dR}(\pi)=m^{(k+1-j,-j)}(\pi)\oplus m^{(-j,k+1-j)}(\pi),$$
avec 
$$m^{(k+1-j,-j)}(\pi)=\C\otimes_{\Q(\pi)} m_{\rm dR}^+(\pi)
\quad{\rm et}\quad 
m^{(-j,k+1-j)}(\pi)\overset{\sim}{\to}\C\otimes_{\Q(\pi)} m_{\rm dR}^-(\pi).$$
De plus, $\matrice{-1}{0}{0}{1}_\infty$ \'echange $m^{(-j,k+1-j)}(\pi)$
et $m^{(k+1-j,-j)}(\pi)$.

\subsubsection{Dualit\'e}\label{como11}
On \index{pi@\piell}note $\check\pi$ la {\it contragr\'ediente de $\pi$}, i.e.~l'espace des
vecteurs $\GG(\A^{]\infty[})$-lisses de ${\rm Hom}(\pi,\Q)$.
On a $\omega_{\check\pi}=\omega_\pi^{-1}$, et $\check\pi\cong\pi\otimes\omega_\pi^{-1}$.
On en d\'eduit, en utilisant le (i) du lemme~\ref{carcen},
 que $\check\pi$ est de poids $(k+2,k+1-j)$.

Si ${\rm truc}\in\{{\rm B},{\eet},{\rm dR}\}$, on a une application naturelle
de $\iH^1_{{\rm truc,\,par}}(W^{\rm truc}_{k,k-j})$ dans
$\pH^1_{{\rm truc,\,par}}(W^{\rm truc}_{k,k-j})$, 
envoyant $(x_N)_N$ sur $(\frac{1}{[\wGamma(1):\wGamma(N)]}x_N)_N$,
 et donc un accouplement naturel
$$\langle\ ,\ \rangle_{\rm truc}:\iH^1_{{\rm truc,\,par}}(W^{\rm truc}_{k,k-j})
\times \iH^1_{{\rm truc,\,par}}(W^{\rm truc}_{k,j})\to
F^{\rm truc}\otimes\zeta_{\rm truc}^{-1}$$
On en d\'eduit des dualit\'es naturelles
$$\langle\ ,\ \rangle_{\rm truc}:m_{\rm truc}(\check\pi)\times m_{\rm truc}(\pi)\to 
F^{\rm truc}\otimes\zeta_{\rm truc}^{-1},$$
d\'efinies par
$$\langle\check\gamma(\check v) ,\gamma(v) \rangle_{\rm truc}=
\langle\check\gamma ,\gamma \rangle_{\rm truc}\cdot\langle\check v,v\rangle,$$
pour tous $\check v\in\check\pi$, $v\in\pi$.

\subsubsection{Encadrement et mod\`ele de Kirillov}\label{como12}
Un {\it encadrement de $\pi$} est un mod\`ele de Kirillov pour $\pi$,
i.e.~une injection $\PP(\A^{]\infty[})$-\'equivariante
$$\pi\hookrightarrow {\rm LC}\big(\Aidu,\Q(\pi)\otimes_{\Q}\Q^{\rm cycl}\big)^{\cZ^\dual},$$
o\`u $\matrice{a}{b}{0}{1}\in \PP(\A^{]\infty[})$ agit par $(\matrice{a}{b}{0}{1}\star\phi)(x)=
{\bf e}^{]\infty[}(bx)\phi(ax)$, et l'invariance par $\cZ^\dual$ se traduit par $\phi(ax)=\sigma_a(\phi(x))$,
si $a\in\cZ^\dual$ et $x\in\Aidu$ (et $\sigma_a$ agit sur $\Q^{\rm cycl}$).
Toute repr\'esentation cohomologique irr\'eductible admet un encadrement (unique \`a 
multiplication pr\`es par un \'el\'ement de $\Q(\pi)^\dual$).

Une {\it repr\'esentation encadr\'ee} de $\GG(\A^{]\infty[})$ est un sous-$\PP(\A^{]\infty[})$-module 
$$\pi\subset {\rm LC}\big(\Aidu,\Q(\pi)\otimes_{\Q}\Q^{\rm cycl}\big)^{\cZ^\dual},$$
muni d'une action lisse de $\GG(\A^{]\infty[})$ \'etendant celle de $\PP(\A^{]\infty[})$.
\subsubsection{Encadrement de repr\'esentations localement alg\'ebriques}\label{como12.0}
Soit 
$${\rm LP}^{[-j,k-j]}(\Aidu,\Q_p(\pi)\otimes_{\Q}\Q^{\rm cycl})$$ 
l'espace des $\sum_{i=-j}^{k-j}\phi_i X^i$, avec $\phi_i\in
{\rm LC}(\Aidu,\Q_p(\pi)\otimes_{\Q}\Q^{\rm cycl})$.
On munit cet espace d'une action de $\PP(\A^{]\infty[})$
par la formule
$$\big(\matrice{a}{b}{0}{1}\star\phi)(u,X)={\bf e}(bx)e^{b_pX}\phi(au,a_pX).$$

Si $\pi^{\rm alg}=\pi\otimes_{\Q(\pi)} W_{k,j}^\dual(\Q_p(\pi))$,
{\it un encadrement de $\pi^{\rm alg}$}
est un mod\`ele de Kirillov pour $\pi^{\rm alg}$, i.e.~une injection $\PP(\A^{]\infty[})$-\'equivariante
$$\pi^{\rm alg}\hookrightarrow 
{\rm LP}^{[-j,k-j]}(\Aidu,\Q_p(\pi)\otimes\Q^{\rm cycl})^{\cZ^\dual},$$
l'invariance par $\cZ^\dual$ se traduisant par $\phi(au,X)=\sigma_a(\phi(u,X))$.

Si $\pi$ est encadr\'ee, on fabrique un encadrement de $\pi^{\rm alg}$
par:
$$\sum_{i=0}^k\phi_i\otimes\tfrac{(e_2^\dual)^{k-i}(e_1^\dual)^i}{(k-i)!\,(e_1^\dual\wedge e_2^\dual)^j}
\mapsto \sum_{i=0}^k\phi_i X^{i-j}.$$

\subsubsection{Nouveau vecteur}\label{como12.1}
Si $\pi$ est de conducteur $N$,
il existe $v_\pi\in\pi$, \index{vpi@\vpi}unique, le {\it nouveau vecteur} de $\pi$, tel que
$$\matrice{a}{b}{c}{d}\star v_\pi=\omega_\pi(d)v_\pi,{\text{ pour tout
$\matrice{a}{b}{c}{d}\in\widehat{\Gamma}_0(N)$,\quad $v_\pi(x)=1$, si $x\in\cZ^\dual$.}}$$
Notons que $v_\pi$ est \`a support dans $\cZ$ (par invariance par $\matrice{1}{\cZ}{0}{1}$),
et que $v_\pi(x)\in\Q(\pi)$ (par invariance par $\cZ^\dual$)
et ne d\'epend que de $|x|_\A$ (par invariance par $\matrice{\cZ^\dual}{0}{0}{1}$).

\subsubsection{Factorisation de $\pi$}\label{como12.2}
La fonction $v_\pi$ admet une factorisation naturelle sous la forme 
$$v_\pi(u^{]\infty[})=\prod_\ell v_{\pi,\ell}(u_\ell),
\quad{\text{o\`u }}
v_{\pi,\ell}\in {\rm LC}\big(\Q_\ell^\dual,\Q(\pi)\otimes_\Q\Q(\bmu_{\ell^\infty})\big)^{\Z_\ell^\dual}
\ {\text{et $v_{\pi,\ell}(1)=1$.}}$$

Les translat\'es de $v_\pi$ sous l'action de $\GG(\Q_\ell)$
sont de la \index{vpi@\vpi}forme $v_\pi^{]\ell[}\otimes\phi_\ell$,
o\`u 
$$v_\pi^{]\ell[}=\otimes_{q\neq\ell}v_{\pi,q},
\quad{\text{i.e.~$v_\pi^{]\ell[}(u^{]\infty,\ell[})=\prod_{q\neq\ell}v_{\pi,q}(u_q)$,}}$$
et $\phi_\ell\in {\rm LC}(\Q_\ell^\dual,\Q(\pi)\otimes_\Q\Q(\bmu_{\ell^\infty}))^{\Z_\ell^\dual}$.
Le $\Q(\pi)$-espace $\pi_\ell$ engendr\'e par les $\phi_\ell$ h\'erite  
de l'action de $\GG(\Q_\ell)$; c'est donc une repr\'esentation de $\GG(\Q_\ell)$,
naturellement encadr\'ee, et dont le nouveau vecteur
est $v_{\pi,\ell}$.

L'application $\otimes_\ell(\pi_\ell,v_{\pi,\ell})\to
{\rm LC}\big(\Aidu,\Q(\pi)\otimes_\Q\Q^{\rm cycl}\big)^{\cZ^\dual}$
envoyant $\otimes_\ell\phi_\ell$, o\`u $\phi_\ell=v_{\pi,\ell}$ pour presque tout $\ell$,
sur la fonction $u^{]\infty[}\mapsto\prod_\ell\phi_\ell(u_\ell)$, induit
un isomorphisme $\GG(\A^{]\infty[})$-\'equivariant
$$\otimes_\ell(\pi_\ell,v_{\pi,\ell})\cong (\pi,v_\pi).$$

\subsubsection{L'\'el\'ement $\iota_{{\rm dR},\pi}^+$ de $m_{\rm dR}(\pi)$
et la forme modulaire $f_\pi$}\label{como12.3}
L'espace $m_{\rm dR}^+(\pi)$ admet une base naturelle $\iota_{{\rm dR},\pi}^+$
\index{iotadr@\iotadr}d\'efinie par
\begin{align*}
\big(\iota_{{\rm dR},\pi}^+(v)\big)(\tau,\matrice{u}{0}{0}{1})
&=(-2i\pi)^{k+1-j}\big(\sum_{n\in\Q_+^\dual}n^{(k+2)-(j+1)}v(nu){\bf e}_\infty(-n\tau)\big)
\otimes\tfrac{(\tau e_2-e_1)^k}{(e_1\wedge e_2)^j}d\tau\\
&=(-1)^{k}\big(\sum_{n\in\Q_+^\dual}n^{(k+2)-(j+1)}v(nu){\bf e}_\infty(-n\tau)\big)
\otimes\zeta_{\rm dR}^{j+1}\big(\tfrac{dt}{t}\big)^{k+2}
\end{align*}
La forme modulaire $f_\pi$ d\'efinie par
$$\big(\iota_{{\rm dR},\pi}^+(v_\pi)\big)(\tau,1^{]\infty[})=f_\pi(\tau) 
\otimes(-2i\pi)^{k+1-j}\tfrac{(\tau e_2-e_1)^k}{(e_1\wedge e_2)^j}d\tau$$
est primitive, et appartient \`a $S_{k+2}(\Gamma_0(N),\tilde\omega_\pi^{-1})$.
Son $q$-d\'eveloppement est 
$$f_\pi=\sum_{n\geq 1}a_nq^n,\quad{\text{avec $a_n=n^{k+1-j}v_\pi(n^{]\infty[})$.}}$$
On note $f_\pi^\dual$ la forme modulaire dont le $q$-d\'eveloppement est 
$$f_\pi^\dual=\sum_{n\geq 1}\overline{a_n}q^n, \quad{\text{et donc
$f_\pi^\dual(\tau)=\overline{f_\pi(-\overline\tau)}$}}.$$
On a aussi 
$$f_\pi^\dual=f_\pi\otimes\tilde\omega_\pi\quad{\text{(i.e.~$\overline{a_n}=\tilde\omega_\pi(n)a_n$).}}$$
Alors $\pi$ et ses multiplicit\'es s'expriment en termes de $f_\pi$
et des repr\'esentations associ\'ees:
$$\pi\cong \pi_{f_\pi,j+1},\quad m_{\eet}(\pi)\cong \rho_{f_\pi,j+1}^\dual\cong
\rho_{f_\pi^\dual}\otimes\cyp ^j.$$

\subsubsection{L'\'el\'ement $\iota_{{\rm dR},\pi}^-$ de $m_{\rm dR}(\pi)$
et la p\'eriode $\lambda(\pi)$}\label{como12.41}
L'accouplement $\langle\ ,\ \rangle_{\rm dR}$ est identiquement nul sur
$m_{\rm dR}^+(\check\pi)\times m_{\rm dR}^+(\pi)$; il induit donc des
dualit\'es sur $m_{\rm dR}^-(\check\pi)\times m_{\rm dR}^+(\pi)$
et $m_{\rm dR}^+(\check\pi)\times m_{\rm dR}^-(\pi)$.
On \index{iotadr@\iotadr}note $\iota_{{\rm dR},\check\pi}^-$ et $\iota_{{\rm dR},\pi}^-$
les \'el\'ements de $m_{\rm dR}^-(\check\pi)$ et $m_{\rm dR}^-(\pi)$
v\'erifiant
$$\langle\iota_{{\rm dR},\check\pi}^- ,\iota_{{\rm dR},\pi}^+ \rangle_{\rm dR}=\zeta_{\rm dR}^{-1},
\quad
\langle\iota_{{\rm dR},\check\pi}^+ ,\iota_{{\rm dR},\pi}^- \rangle_{\rm dR}=\zeta_{\rm dR}^{-1},$$
et on note encore $\iota_{{\rm dR},\check\pi}^-$ et $\iota_{{\rm dR},\pi}^-$
leurs repr\'esentants harmoniques dans $m^{(j-k,j+1)}(\check\pi)$ et $m^{(-j,k+1-j)}(\pi)$.

Il \index{lambdapi@\lambdapi}existe $\lambda(\pi), \lambda(\check\pi)\in\C$ tels que
$$\iota_{{\rm dR},\pi}^-=\lambda(\pi)\,\matrice{-1}{0}{0}{1}_\infty\star \iota_{{\rm dR},\pi}^+,
\quad
\iota_{{\rm dR},\check\pi}^-=\lambda(\check\pi)\,\matrice{-1}{0}{0}{1}_\infty\star \iota_{{\rm dR},\check\pi}^+$$
\begin{lemm}\phantomsection\label{como11.2}
On a $\lambda(\pi)=-\lambda(\check\pi)$.
\end{lemm}
\begin{proof}
Posons
\begin{align*}
\phi_\pi^+=\iota_{{\rm dR},\pi}^+(v_\pi),\quad \phi_\pi^-=\iota_{{\rm dR},\pi}^-(v_\pi),\quad
\check\phi_\pi^+=\iota_{{\rm dR},\check\pi}^+(v_{\check\pi}),
\quad \check\phi_\pi^-=\iota_{{\rm dR},\check\pi}^-(v_{\check\pi})
\end{align*}
Par d\'efinition,
$$\langle\iota_{{\rm dR},\check\pi}^+ ,\iota_{{\rm dR},\pi}^- \rangle_{\rm dR}=
\langle\check\phi_\pi^+,\phi_\pi^-\rangle_{\rm dR}=\zeta_{\rm dR}^{-1},\quad
\langle\iota_{{\rm dR},\check\pi}^- ,\iota_{{\rm dR},\pi}^+ \rangle_{\rm dR}=
\langle\check\phi_\pi^-,\phi_\pi^+\rangle_{\rm dR}=\zeta_{\rm dR}^{-1}.$$
Posons $w_\infty=\matrice{-1}{0}{0}{1}_\infty$.
On a
$$\phi_\pi^-=\lambda(\pi)\,w_\infty\star\phi_\pi^+,\quad
\check\phi_\pi^-=\lambda(\check\pi)\,w_\infty\star\check\phi_\pi^+.$$
Comme $w_\infty^2=1$ et $\langle w_\infty\star\check\phi, w_\infty\star\phi\rangle_{\rm dR}=
-\langle\check\phi,\phi\rangle_{\rm dR}$ car $\tau\mapsto\overline\tau$ renverse
l'orientation, il r\'esulte de ce qui pr\'ec\`ede que
\begin{align*}
\lambda(\pi)=
-\langle w_\infty\star \check\phi_\pi^-,\lambda(\pi) w_\infty\star \phi_\pi^+\rangle_{\rm dR}\otimes\zeta_{\rm dR}&=
-\langle w_\infty\star \check\phi_\pi^-,\phi_\pi^-\rangle_{\rm dR}\otimes\zeta_{\rm dR}\\ 
&= -\langle \lambda(\check\pi)\check\phi^+,\phi_\pi^-\rangle_{\rm dR}\otimes\zeta_{\rm dR}=
-\lambda(\check\pi)
\end{align*}
\qedhere
\end{proof}

\begin{lemm}\phantomsection\label{lambda1}
\index{vpi@\vpi}Soit $v_\pi^{]S[}=\otimes_{\ell\notin S}v_{\pi,\ell}$.
Si $\eta:\Z_S^\dual\to L^\dual$ est un caract\`ere localement constant, alors:
\begin{align*}
(\iota_{{\rm dR},\pi}^+(v_\pi^{]S[}\otimes\eta))\big(\tau,\matrice{u}{0}{0}{1}\big)&=
(-2i\pi)^{k+1-j}\eta(u)(f_\pi\otimes {\bf 1}_{\Z_S^\dual}\eta)(\tau)\otimes
\tfrac{(\tau e_2-e_1)^k}{(e_1\wedge e_2)^j}d\tau\\
(\iota_{{\rm dR},\pi}^-(v_\pi^{]S[}\otimes\eta))\big(\tau,\matrice{u}{0}{0}{1}\big)&=
\eta(-1)\lambda(\pi)
(2i\pi)^{k+1-j}\eta(u)(f_\pi\otimes {\bf 1}_{\Z_S^\dual}\eta)(-\overline\tau)\otimes
\tfrac{(\overline\tau e_2-e_1)^k}{(e_1\wedge e_2)^j}d\overline\tau
\end{align*}
\end{lemm}
\begin{proof}
La premi\`ere formule est imm\'ediate sur la d\'efinition.
La seconde se d\'emontre en utilisant la formule
$\iota_{{\rm dR},\pi}^-=\lambda(\pi)w_\infty\star\iota_{{\rm dR},\pi}^+$ (cela fait appara\^{\i}tre
le $\lambda(\pi)$ et remplace $\tau$ par $\overline\tau$), puis en utilisant l'invariance
par $\matrice{-1}{0}{0}{1}\in \GG(\Q)$ (cela change $e_1$ en $-e_1$, $u$ en $-u$ et $\overline\tau$
en $-\overline\tau$).
\end{proof}

\subsubsection{Les p\'eriodes $\Omega^{\pm}_\pi$}\label{como12.42}
Apr\`es extension des scalaires, $m(\pi)$ acquiert une base naturelle:
$$\C\otimes_{\Q(\pi)}m(\pi)=\C\,\iota_{\rm ES}^+\circ\iota_{{\rm dR},\pi}^+
\oplus \C\,\iota_{\rm ES}^-\circ\iota_{{\rm dR},\pi}^+.$$
On \index{Ompi@\Omegapi}note $\Omega_\pi^{\pm}(\gamma)$ les coordonn\'ees de $\gamma\in m(\pi)$ dans cette base
(ce sont des nombres transcendants):
$$\gamma=\Omega_\pi^+(\gamma)\,\iota_{\rm ES}^+\circ\iota_{{\rm dR},\pi}^+
+\Omega_\pi^-(\gamma)\,\iota_{\rm ES}^-\circ\iota_{{\rm dR},\pi}^+.$$

\begin{lemm}\phantomsection\label{como11.1}
Si $\gamma\in m(\pi)$ et $\check\gamma\in m(\check\pi)$, alors
$$\langle\check\gamma,\gamma\rangle_{\rm B}\otimes\zeta_{\rm B}=
2i\pi\,\frac{\Omega_{\check\pi}^+(\check\gamma)\Omega_{\pi}^-(\gamma)-
\Omega_{\check\pi}^-(\check\gamma)\Omega_{\pi}^+(\gamma)}{2\,\lambda(\check\pi)}=
2i\pi\,\frac{\Omega_{\check\pi}^-(\check\gamma)\Omega_{\pi}^+(\gamma)-
\Omega_{\check\pi}^+(\check\gamma)\Omega_{\pi}^-(\gamma)}{2\,\lambda(\pi)}.$$
\end{lemm}
\begin{proof}
Par d\'efinition, 
$\langle\check\gamma,\gamma\rangle_{\rm B}=\langle\check\gamma(v_{\check\pi}),\gamma(v_\pi)\rangle_{\rm B}$.
On \'ecrit $\gamma=\gamma^++\gamma^-$ et $\check\gamma=\check\gamma^++\check\gamma^-$.
Alors, par d\'efinition des p\'eriodes $\Omega^{\pm}$ et de $\phi_\pi^{\pm}$, $\phi_{\check\pi}^\pm$
(rem.\,\ref{como11.2}), on a
$$\gamma^{\pm}(v_\pi)=\Omega_\pi^{\pm}(\gamma)\iota_{\rm ES}\big(\tfrac{1}{2}\big(\phi^+_\pi\pm
 (w_\infty\star\phi^+_\pi)\big)\big),
\quad
\check\gamma^{\pm}(v_{\check\pi})=\Omega_{\check\pi}^{\pm}(\check\gamma)\iota_{\rm ES}\big(\tfrac{1}{2}\big(\check\phi^+_\pi\pm
 (w_\infty\star\check\phi^+_\pi)\big)\big).$$
On utilise les formules
\begin{align*}
& \langle \iota_{\rm ES}(\check\phi),\iota_{\rm ES}(\phi)\rangle_{\rm B}\otimes\zeta_{\rm B}=
2i\pi\,\langle\check\phi,\phi\rangle_{\rm dR}\otimes\zeta_{\rm dR}\\
&\big\langle\check\phi_\pi^-,\phi_\pi^+\big\rangle_{\rm dR}=\zeta_{\rm dR}^{-1}=
-\langle w_\infty\star\check\phi_\pi^-, w_\infty\star\phi_\pi^+\rangle_{\rm dR}\\
&\langle w_\infty\star\check\phi_\pi^-,\phi_\pi^+\rangle_{\rm dR}=
\langle\check\phi_\pi^-, w_\infty\star\phi_\pi^+\rangle_{\rm dR}=0
\end{align*}
(La premi\`ere ligne provient de~(\ref{Compa}), la seconde vient des d\'efinitions
et la troisi\`eme 
de ce que l'on a affaire au cup-produit de deux 1-formes holomorphes ou antiholomorphes.)
On en d\'eduit l'identit\'e (avec $\langle\ ,\ \rangle=\langle\ ,\ \rangle_{\rm dR}$)
\begin{align*}
\big\langle \check\phi^+_\pi+\epsilon
 (w_\infty\star\check\phi^+_\pi),
\phi^+_\pi+\epsilon'
 (w_\infty\star\phi^+_\pi)\big\rangle&=
\frac{1}{\lambda(\check\pi)}\big\langle 
 (w_\infty\star\check\phi^-_\pi)+\epsilon \check\phi^-_\pi,
\phi^+_\pi+\epsilon'
 (w_\infty\star\phi^+_\pi)\big\rangle\\
&=\epsilon-\epsilon'
\end{align*}
et le r\'esultat.
\end{proof}

\subsubsection{Multiplicit\'e des repr\'esentations localement alg\'ebriques}\label{como10.1}
Soit 
$$\pi^{\rm alg}=\pi\otimes_{\Q(\pi)}W_{k,j}^\dual(L(\pi)).$$
Alors $\pi^{\rm alg}$ \index{pi@\piell}est une repr\'esentation localement alg\'ebrique de $\GG(\A^{]\infty[})$.
On \index{mpi@\mpi}pose 
$$m(\pi^{\rm alg})={\rm Hom}_{L[\GG(\A^{]\infty[})]}\big(\pi^{\rm alg},H^1_{\rm par}(\GG(\Q),{\cal C}(\GG(\A),L))\big).$$
L'application naturelle $({\rm LC}(\GG(\A),\Q)\otimes W_{k,j})\otimes_{\Q}W_{k,j}^\dual(L)\to
{\cal C}(\GG(\A),L)$ induit,
d'apr\`es~\cite[th.\,7.4.2]{Em06b} et~\cite{cara},
un isomorphisme $L(\pi)$-\'equivariant
$$m_{\eet}(\pi)=L(\pi)\otimes_{\Q(\pi)}m(\pi)\overset{\sim}{\to} m(\pi^{\rm alg})$$
On a alors une d\'ecomposition $G_\Q\times\GG(\Ai)$-\'equivariante, pour ${\rm truc}\in\{\ ,c,{\rm par}\}$
\begin{equation}\label{VV2}
H^1_{\rm truc}(\GG(\Q),{\cal C}(\GG(\A),L))^{\rm alg}=
\big(\bigoplus\nolimits_\Pi\big(m_{\eet}(\Pi)\otimes_{L(\Pi)}\Pi^{\rm alg}\big)
\big)\bigoplus
\big(\raisebox{-1mm}{$\overset{\mbox{\footnotesize{contribution}}}{\mbox{\footnotesize{des pointes}}}$}\big)
\end{equation}
o\`u $\Pi$ parcourt les $L$-repr\'esentations cohomologiques irr\'eductibles,
les $m_{\eet}(\Pi)$ sont deux \`a deux non isomorphes, et {\og contribution des pointes\fg}
est une somme de repr\'esentations de la forme $\chi\otimes I$, 
o\`u $\chi$ est un caract\`ere de
$G_\Q$ et $I$ est une repr\'esentation de $\GG(\A)$ de la s\'erie principale (induite d'un caract\`ere du borel);
ce terme est nul si ${\rm truc}={\rm par}$, et dispara\^{\i}t dans les autres cas
quand on localise en un id\'eal non-eisenstein.

\begin{prop}\phantomsection\label{VV3}
Si $S\subset{\cal P}$ est fini et contient $p$, alors
$${\rm Hom}_{L[\GG(\A^{]\infty,S[})]}\big(L(\pi)\otimes_{\Q(\pi)}\pi^{]S[},H^1(\GG(\Q),{\cal C}(\GG(\A),L))^{\rm alg}\big)=
m_{\eet}(\pi)\otimes_{L(\pi)}\pi_S^{\rm alg}$$
\end{prop}
\begin{proof}
Cela r\'esulte de la d\'ecomposition~(\ref{VV2}) et du th\'eor\`eme de multiplicit\'e~$1$:
la repr\'esentation $\pi^{]S[}$ d\'etermine $m_{\eet}(\pi)$ par le th\'eor\`eme de densit\'e
de \v{C}ebotarev, qui elle-m\^eme d\'etermine les $\pi_\ell$ pour $\ell\in S$ ainsi que
$\pi_p^{\rm alg}$ car la composante alg\'ebrique est encod\'ee dans les poids de Hodge-Tate
de la restriction de $m_{\eet}(\pi)$ \`a $G_{\Q_p}$.
\end{proof}

\Subsection{Torsion par un caract\`ere}\label{como112.1}
\subsubsection{Torsion d'une repr\'esentation encadr\'ee}\label{como112.3}
Si $\pi$ est une repr\'esentation encadr\'ee de $\GG(\A^{]\infty[})$, et si $\chi:\A^\dual/\Q^\dual\to
L^\dual$ est un caract\`ere lisse, on \index{twi@\twis}note $\pi\otimes\chi$ la repr\'esentation encadr\'ee
dont l'espace est\footnote{Voir le lemme~\ref{eqf1} pour l'introduction de la somme de Gauss.}
 $\{G(\chi^{-1})\phi\,\chi,\ \phi\in\pi\}$, avec action de $\GG(\A^{]\infty[})$ donn\'ee
par $$g\star (\phi\,\chi)=\chi(\det g)\,(g\star\phi)\,\chi.$$ 
On a $$\big(\matrice{a}{b}{0}{1}\star(\phi\chi)\big)(x)={\bf e}^{]\infty[}(bx)\phi(ax)\chi(ax)=
\chi(a)\big(\big(\matrice{a}{b}{0}{1}\star\phi\big)\chi\big)(x),$$ ce qui prouve que l'action
de $\GG(\A^{]\infty[})$ prolonge bien l'action de $\PP(\A^{]\infty[})$.

\vskip.2cm
\noindent$\bullet${\it Torsion par une puissance de $|\ |_\A$}.---
Si $\pi$ est une repr\'esentation cohomologique encadr\'ee de
poids $(k+2,j+1)$, et si $a\in\Z$, alors  $\pi\otimes|\ |_{\A}^a$ est de poids $(k+2,j+1-a)$,
et on a un isomorphisme naturel
$$m_{\rm B}(\pi\otimes |\ |_\A^a)\cong m_{\rm B}(\pi)\otimes (e_1\wedge e_2)^a,$$
avec $(\gamma\otimes(e_1\wedge e_2)^a)(v)=\gamma(v)\otimes(\delta_\A ^a\circ\det)(e_1\wedge e_2)^a$
comme dans la rem.\,\ref{ES7.2} (modulo l'identification $m_{\rm B}(\pi)=m(\pi)$).

Comme $e_1\wedge e_2=-\zeta_{\rm B}^{-1},$
on a
\begin{align*}
m_{\rm B}(\pi\otimes |\ |_\A^a)=&\ m_{\rm B}(\pi)\otimes \zeta_{\rm B}^{-a},\\
m_{\eet}(\pi\otimes |\ |_\A^a)=&\ m_{\eet}(\pi)\otimes \zeta_{\rm B}^{-a},\\
m_{\rm dR}(\pi\otimes |\ |_\A^a)=&\ m_{\rm dR}(\pi)\otimes \zeta_{\rm dR}^{-a}.
\end{align*}
\begin{lemm}\phantomsection\label{twi1}
Si $\gamma\in m_{\rm B}(\pi)$, alors
$$\Omega_{\pi\otimes|\ |_\A^a}^{\pm}(\gamma\otimes\zeta_{\rm B}^{-a})
=(2i\pi)^{-a}\Omega_\pi^{\pm(-1)^a}\hskip-.1cm (\gamma).$$
\end{lemm}
\begin{proof}
On a $|nu|_\A^a=n^{-a}\delta_\A^a(u)$, et donc
\begin{align*}
n^{(k+2)-(j+1-a)}(v\otimes |\ |_\A^a)(nu)&=n^{(k+2)-(j+1)}\delta_\A^a(u)v(nu)\\
\iota_{{\rm dR},\pi\otimes |\ |_\A^a}^+&=\delta_\A ^a\iota_{{\rm dR},\pi}^+\otimes\zeta_{\rm dR}^{-a}
\end{align*}
On en d\'eduit, via les rem.~\ref{ES7} et~\ref{ES7.2} (en particulier pour l'apparition du $(-1)^a$),
que
$$\iota_{\rm ES}^{\pm}\circ\iota_{{\rm dR},\pi\otimes |\ |_\A^a}^+=
(2i\pi)^a(\iota_{\rm ES}^{\pm(-1)^a}\hskip-.2cm \circ\iota_{{\rm dR},\pi}^+)\otimes\zeta_{\rm B}^{-a}.$$
D'o\`u les relations
\begin{align*}
\Omega_\pi^{\pm}(\gamma)(\iota_{\rm ES}^{\pm}\circ\iota_{{\rm dR},\pi}^+)\otimes\zeta_{\rm B}^{-a}& =
\gamma^{\pm}\otimes\zeta_{\rm B}^{-a}=
\Omega_{\pi\otimes |\ |_\A^a}^{\pm(-1)^a}\hskip-.1cm (\gamma\otimes\zeta_{\rm B}^{-a})
\,\iota_{\rm ES}^{\pm(-1)^a}\hskip-.2cm \circ\iota_{{\rm dR},\pi\otimes |\ |_\A^a}^+\\
&= \Omega_{\pi\otimes |\ |_\A^a}^{\pm(-1)^a}\hskip-.1cm(\gamma\otimes\zeta_{\rm B}^{-a})(2i\pi)^a(\iota_{\rm ES}^{\pm}\circ\iota_{{\rm dR},\pi}^+)\otimes\zeta_{\rm B}^{-a}
\end{align*}
et le r\'esultat.
\end{proof}
\noindent$\bullet${\it Torsion des vecteurs localement alg\'ebriques}.---
Si $\pi$ est cohomologique de poids $(k+2,j+1)$, on note $\pi^{\rm alg}$ la repr\'esentation
localement alg\'ebrique $\pi\otimes W_{k,j}^\dual$.
Si $\chi$ est de poids $a$, alors, en tant que repr\'esentation encadr\'ee, on
$$(\pi\otimes\chi)^{\rm alg}=\{G(\chi)^{-1}X^a\chi^{(p)}\phi,\ \phi\in\pi^{\rm alg}\},$$
o\`u $\chi^{(p)}$ est le caract\`ere $p$-adique associ\'e \`a $\chi$ (cf.~\no\ref{prelim4}).
On a aussi 
$$m_{\eet}(\pi\otimes\chi)=m_{\eet}(\pi)\otimes G(\chi)\zeta_{\rm B}^{-a}$$
(Notons que $\sigma(G(\chi)\zeta_{\rm B}^{-a})=\chi_{\rm Gal}(\sigma)^{-1} G(\chi)\zeta_{\rm B}^{-a}$,
pour tout $\sigma\in G_\Q$, i.e.~$m_{\eet}(\pi\otimes\chi)=m_{\eet}(\pi)\otimes\chi_{\rm Gal}^{-1}$
en tant que repr\'esentation de $G_\Q$.)
L'accouplement naturel 
$$m_{\eet}(\pi\otimes\chi)\otimes (\pi\otimes\chi)^{\rm alg}\to
H^1_c(\GG(\Q),{\cal C}(\GG(\A),L))$$
 est reli\'e \`a celui pour $\pi$ par la formule
$$\langle \gamma\otimes G(\chi)\zeta_{\rm B}^{-a}, G(\chi)^{-1}X^a\chi^{(p)}\phi\rangle=
(\chi^{(p)}\circ\det) \langle\gamma,\phi\rangle$$
(La multiplication par $\chi^{(p)}\circ\det$ commute \`a l'action de $\GG(\Q)$,
et donc fournit un isomorphisme de $H^1_c(\GG(\Q),{\cal C}(\GG(\A),L)$.)

\part{Mod\`ele de Kirillov et factorisation de la cohomologie compl\'et\'ee}
\section{La cohomologie de la boule unit\'e}\label{qq8}
Dans ce chapitre, on rassemble un certain nombre de r\'esultats concernant la
cohomologie pro\'etale de la boule unit\'e ouverte. En particulier, on d\'efinit
un $q$-d\'eveloppement pour la cohomologie \`a valeurs dans certains syst\`emes locaux
(la fl\`eche $\iota_{\tA}$ de (\ref{bato40.4})) et on prouve une loi de r\'eciprocit\'e
explicite (th.\,\ref{bato22})
mettant en sc\`ene une application logarithme qui est un analogue g\'eom\'etrique
d'un inverse de l'exponentielle de Bloch-Kato.

Si $X$ est un espace perfecto\"{\i}de, on \index{At@\tAAA}note $\tbA(X)$, $\tbA^+(X)$, $\tbA^{++}(X)$
les sections globales sur $X$ des faisceaux pro\'etales $\tbA$, $\tbA^+$, $\tbA^{++}$:
si $X$ est un affino\"{\i}de perfecto\"{\i}de, alors
$\tbA(X)=W(\O(X)^\flat)$, $\tbA^+(X)=W(\O^+(X)^\flat)$, $\tbA^{++}(X)=W(\O^{++}(X)^\flat)$.
De m\^eme, on note $\Bdr^+(X)$ les sections globales sur $X$ du faisceau pro\'etale $\Bdr^+$.

Soit $C$ un corps complet pour $v_p$, et alg\'ebriquement clos.
On note $\tA$, $\tA^+$, $\tA^{++}$ et $\bdr^+$ les anneaux de Fontaine correspondants
(on peut les voir comme
 les sections globales des faisceaux ci-dessus sur ${\rm Spa}(C,\O_C)$).

\Subsection{Cohomologie de la boule ferm\'ee}
\subsubsection{La boule ferm\'ee}
Soit $B$ la boule ferm\'ee ${\rm Spa}(C\langle q\rangle,\O_C\langle q\rangle)$; on 
\index{bou1@\boulef}note
$B^\times$ la boule $B$ munie d'une structure logarithmique en $0$.
On note $\overline B$ le rev\^etement universel de $B^\times$:
$\O(\overline{B})$ est le compl\'et\'e de 
l'extension maximale de $\O(B)=C\langle q\rangle$, \'etale en dehors
de $0$.
La boule ferm\'ee perfecto\"{\i}de $B_\infty$ est un quotient de $\overline B$: on a
$B_\infty={\rm Spa}(C\langle q^{p^{-\infty}}\rangle,\O_C\langle q^{p^{-\infty}}\rangle)$ o\`u
$\O_C\langle q^{p^{-\infty}}\rangle$ est le compl\'et\'e $p$-adique de $\varinjlim_n\O_C[q^{p^{-n}}]$,
i.e. l'anneau des $\sum_{i\in I}{a_iq^i}$, $I=p^{-\infty}\N$, avec $a_i\to 0$ quand
$i\to\infty$ --- suivant le filtre des compl\'ementaires des parties finies --- dans~$I$.
On note $B_\infty^\times$ la boule $B_\infty$ munie d'une structure logarithmique en $0$.

On \index{GB@\GB}pose
$$G_B={\rm Aut}(\overline{B}/B)
\quad{\rm et}\quad
H_B={\rm Aut}(\overline{B}/B_\infty).$$
Alors
$${\rm Aut}(B_\infty/B)=G_B/H_B\cong \UU(\Z_p)=\matrice{1}{\Z_p}{0}{1}.$$ 
On note $\gamma_u$ l'\'el\'ement $\matrice{1}{u}{0}{1}$ de $\UU(\Z_p)$,
et on note simplement $\gamma$ le g\'en\'erateur topologique $\gamma_1$ de $ \UU(\Z_p)$.
Sur $\O(B_\infty)$, l'action de $\UU(\Z_p)$ est $\gamma_u^\dual q^i={\bf e}_p(ui) q^i$; elle s'\'etend
en une action de $\UU(\Q_p)$ par la m\^eme formule.

On note $B_{\rm Kum}$ le rev\^etement de Kummer maximal de $B^\times$: on rajoute
les $q^{1/N}$, pour tout $N$, et pas seulement pour $N\mid p^\infty$, et 
$B_{\rm Kum}={\rm Spa}(C\langle q^{\Q_+}\rangle,\O_C\langle q^{\Q_+}\rangle)$.
Alors $B_{\rm Kum}$ est un quotient de $\overline B$ et $\overline B/B_{\rm Kum}$
est \'etale; on note $H_{\rm Kum}$
le groupe de Galois de ce rev\^etement, et on a
${\rm Aut}(B_{\rm Kum}/B)=G_B/H_{\rm Kum}\cong \UU(\cZ)$.
Comme ci-dessus, l'action naturelle de $\UU(\cZ)$ sur $\O(B_{\rm Kum})$
s'\'etend en une action de $\UU(\Ai)$, avec $\gamma_u^\dual q^i=e_\A(u)q^i$
si $u\in \Ai$ et $i\in\Q_+$.

\subsubsection{La cohomologie de la boule unit\'e perfecto\"{\i}de}\label{qq9}
Notons $\O^\flat$ le bascul\'e de $C\langle q^{p^{-\infty}}\rangle$; on a donc
$\O^\flat=C^\flat\langle \tilde q^{p^{-\infty}}\rangle$ (en notant $\tilde q$ l'\'el\'ement
 $(q,q^{1/p},...)$ de $\O^\flat$).  
Maintenant, tout \'el\'ement de $\tbA(B_\infty)=W(\O^\flat)$ s'\'ecrit, de mani\`ere unique,
sous la forme $\sum_{i\in I} a_i \tilde q^i$, avec $a_i\in\tA$ et $a_i\to 0$ quand $i\to\infty$
(i.e. pour tous $n,k$ il n'y a qu'un nombre fini de $i$ tels que $a_i\notin p^n\tA+\tilde p^k\tA^+$).
On d\'ecompose $I$ sous la forme
$$I=\{0\}\sqcup(\sqcup_{k\in\Z}I_k),\quad{\text{avec $I_k=\{p^ka,\ a\in\N, (a,p)=1\}$.}}$$
De m\^eme, notons $\O^\flat_{\rm Kum}$ le bascul\'e de $C\langle q^{\Q_+}\rangle$; on a donc
$\O^\flat_{\rm Kum}=C^\flat\langle \tilde q^{\Q_+}\rangle$.
Tout \'el\'ement de $\tbA(B_{\rm Kum})=W(\O_{\rm Kum}^\flat)$ 
s'\'ecrit, de mani\`ere unique,
sous la forme $\sum_{i\in \Q_+} a_i \tilde q^i$, avec $a_i\in\tA$ et $a_i\to 0$ quand $i\to\infty$.
On d\'ecompose $\Q_+$ sous la forme
$$\Q_+=\{0\}\sqcup(\sqcup_{k\in\Z}I_{{\rm Kum},k}),
\quad{\text{avec $I_{{\rm Kum},k}=\{x\in\Q_+,\ v_p(x)=k\}$.}}$$
\begin{prop}\phantomsection\label{cano3}
{\rm (i)} On a des isomorphismes naturels:
\begin{align*}
H^1_{\proet}(B_{\rm Kum},\Z_p)&\cong H^1(H_{\rm Kum},\Z_p)\cong \tbA(B_{\rm Kum})/(\varphi-1)\\
H^1_{\proet}(B_\infty^\times,\Z_p)&\cong H^1(H_B,\Z_p)\cong \tbA(B_\infty)/(\varphi-1)
\end{align*}
{\rm (ii)}
Les applications naturelles 
$$\widehat\oplus_{i\in I_{{\rm Kum},0}}\tA \tilde q^i\to \tbA(B_{\rm Kum})/(\varphi-1)
\quad{\rm et}\quad
\widehat\oplus_{i\in I_0}\tA \tilde q^i\to \tbA(B_\infty)/(\varphi-1)$$
induisent des isomorphismes
$$\widehat\oplus_{i\in I_{{\rm Kum},0}}(\tA/\tA^{++})\tilde q^i\cong \tbA(B_{\rm Kum})/(\varphi-1)
\quad{\rm et}\quad
\widehat\oplus_{i\in I_0}(\tA/\tA^{++})\tilde q^i\cong \tbA(B_\infty)/(\varphi-1)$$
\end{prop}
\begin{proof}
Pour $B_{\rm Kum}$, le (i) est un cas particulier des r\'esultats de~\cite{Sz1}.
Le cas de $B_\infty^\times$ s'en d\'eduit en utilisant la suite spectrale
de Hochschild-Serre pour le rev\^etement $B_{\rm Kum}/B_\infty$ (de groupe de Galois
$\UU(\cZ^{]p[})$ dont le  ``cardinal'' est premier \`a $p$) qui
fournit un isomorphisme $H^1_{\proet}(B_\infty^\times,\Z_p)=
H^0(\UU(\cZ^{]p[}),H^1_{\proet}(B_{\rm Kum},\Z_p))$: comme $\UU(\cZ^{]p[})$ 
est de  ``cardinal'' premier \`a $p$, les points fixes sous $\UU(\cZ^{]p[})$
de la suite exacte $0\to \tbA(B_{\rm Kum})\overset{\varphi-1}{\lra}
\tbA(B_{\rm Kum})\to \tbA(B_{\rm Kum})/(\varphi-1)\to 0$ forment
encore une suite exacte (c'est aussi visible sur la description du (ii)).

Passons au (ii); la preuve est la m\^eme dans les deux cas, et nous ne traiterons
que le cas de $B_\infty$.
Il s'agit de prouver que 
$\widehat\oplus_{i\in I_0}\tA \tilde q^i\to \tbA(B_\infty)/(\varphi-1)$ est surjective
et que le noyau est $\widehat\oplus_{i\in I_0}\tA^{++} \tilde q^i$.

Commen\c{c}ons par la surjectivit\'e.  Soit $z=\sum_{i\in I} a_i \tilde q^i$ comme ci-dessus.

$\bullet$ Comme $1-\varphi$ est surjectif sur $\tA$, on peut supposer $a_0=0$.

$\bullet$ On peut \'ecrire $a_i=\sum p^n[a_{i,n}]$.  Or, si $x\in {\goth m}_{C^\flat}$,
alors $y=[x]\tilde q^i$ est dans l'image de $1-\varphi$ car la s\'erie $y+\varphi(y)+\varphi^2(y)+\cdots$
converge.  Comme on travaille modulo~$1-\varphi$, on peut supprimer tous les $[a_{i,n}]\tilde q^i$
avec $a_{i,n} \in {\goth m}_{C^\flat}$.  Il ne reste alors, modulo $p^N$, qu'un nombre
fini de $a_{i,n}$ non nuls.

$\bullet$ On a $[a_{i,n}]\tilde q^i=[a_{i,n}^{p^j}]\tilde q^{ip^j}$ modulo $1-\varphi$ et il existe un 
$j\in\Z$ (unique) tel que $ip^j\in I_0$.

On en d\'eduit la surjectivit\'e
et le fait que le noyau contient $\widehat\oplus_{i\in I_0}\tA^{++} \tilde q^i$.
Pour conclure, il suffit donc de v\'erifier que 
$$\big(\widehat\oplus_{i\in I_0}\tA \tilde q^i)\cap \big((\varphi-1)\cdot \tbA(B_\infty)\big)\subset
\widehat\oplus_{i\in I_0}\tA^{++} \tilde q^i.$$
On se ram\`ene \`a v\'erifier le m\^eme \'enonc\'e modulo~$p$.  
Soit $z\in \O^\flat$ tel que $(\varphi-1)z\in \oplus_{i\in I_0}C^\flat \tilde q^i$.
On \'ecrit $z=\sum_{i\in I} a_i \tilde q^i$ et $(\varphi-1)z=\sum_{i\in I}b_i \tilde q^i$.  
Supposons, par l'absurde, qu'au moins un des $a_i$ n'appartient pas \`a ${\goth m}_{C^\flat}$,
et choisissons $i_1$ (resp. $i_2$) tel que $v_p(i_1)$ (resp.~$v_p(i_2)$) r\'ealise le minimum
(resp.~maximum) des $v_p(i)$ pour $i$ tel que $a_i\notin {\goth m}_{C^\flat}$.
Alors au moins une des deux propri\'et\'es suivantes est v\'erifi\'ee: $v_p(i_1)< 0$
ou $v_p(i_2)\geq 0$.  Dans le premier cas, $b_{i_1}=a_{i_1/p}^p-a_{i_1}\equiv -a_{i_1}$ mod ${\goth m}_{C^\flat}$;
dans le second, $b_{p i_2}=a_{i_2}^p-a_{p i_2}\equiv a_{i_2}^p$
mod ${\goth m}_{C^\flat}$.  Dans les deux cas, cela fabrique un $b_i$, avec $i\notin I_0$, tel que
$b_i\neq 0$, et donc $(\varphi-1)z\notin \oplus_{i\in I_0}C^\flat \tilde q^i$ ce qui est contraire \`a l'hypoth\`ese.

On en d\'eduit le r\'esultat.
\end{proof}
\begin{rema}\phantomsection\label{qq10.4}
Pour construire le second isomorphisme du (i), on part
de la suite exacte $0\to \Z_p\to \tbA(\overline B)\overset{\varphi-1}{\longrightarrow}
\tbA(\overline B)\to 0$, et on prend la cohomologie continue de $H_B$ qui, par descente presque \'etale,
 nous donne
une suite exacte 
$$0\to\Z_p\to \tbA(B_\infty)\overset{\varphi-1}{\longrightarrow}
\tbA(B_\infty)\to H^1(H_B,\Z_p)\to 0.$$
En particulier, si $x\in \tbA(B_\infty)$ le $1$-cocycle associ\'e
est $\tau\mapsto (\tau-1)\cdot c$, o\`u $c\in \tbA(\overline B)$ v\'erifie $(\varphi-1)\cdot c=x$.
\end{rema}

\subsubsection{Descente en niveau fini}\label{qq10}
On \index{alg@\rmalg}note ${\rm Alg}$ l'espace des 
fonctions polynomiales sur $\UU(\Z_p)$ \`a valeurs dans $\Z_p$ 
(une telle fonction est de la forme $\phi(\gamma_u)=\sum_{i=0}^ka_i\binom{u}{i}$).
 On fait agir
$\UU(\Z_p)$ par $\big(g\cdot\phi\big)(x)=\phi(g^{-1}x)$; via l'isomorphisme
avec $\Z_p$, cette formule devient $\big(\gamma_u\cdot\phi\big)(\gamma_x)=\phi(\gamma_{x-u})$.
On note ${\cal C}$ l'espace des fonctions continues sur $\UU(\Z_p)$, \`a valeurs dans $\Z_p$.
On \index{W@\WWW}note $W_k\subset {\rm Alg}$ l'espace des polyn\^omes de degr\'e~$\leq k$;
c'est une sous-$\UU(\Z_p)$-repr\'esentation de ${\rm Alg}$ et 
$${\rm Alg}=\varinjlim W_k.$$
On a une suite exacte
de $\UU(\Z_p)$-repr\'esentations:
$$0\to W_k\to{\cal C}\overset{(\gamma-1)^k}{\longrightarrow}{\cal C}\to 0.$$
Si $M$ est un $\UU(\Z_p)$-module, on note $M^{U-{\rm fini}}$ l'ensemble des $x\in M$ tu\'es
par une puissance de $\gamma-1$. En particulier, ${\rm Alg}={\cal C}^{U-{\rm fini}}$.

Comme $\UU(\Z_p)$ est un quotient de $G_B$, on peut voir ${\rm Alg}$ et les $W_k$
comme des syst\`emes locaux pro\'etales sur $B^\times$.
\begin{prop}
On a des isomorphismes
$$H^1_{\proet}(B^\times,{\rm Alg})\cong H^1(G_B,{\rm Alg}),
\quad
H^1_{\proet}(B^\times,W_k)\cong H^1(G_B,W_k).$$
\end{prop}
\begin{lemm}\phantomsection\label{cano103}
L'application qui, \`a un $1$-cocycle $\tau\mapsto\phi_\tau$, associe $\tau\mapsto\phi_\tau(0)$
induit des isomorphismes
$$H^1(H_B,\Z_p)=H^1(G_B,{\cal C})
\quad{\rm et}\quad
H^1(H_B,\Z_p)^{U-{\rm fini}}=H^1(G_B,{\rm Alg}).$$
\end{lemm}
\begin{proof}
Le premier isomorphisme est le lemme de Shapiro. Le second r\'esulte de la suite exacte
longue de cohomologie associ\'ee \`a la suite exacte ci-dessus: 
comme $H^0(G_B,W_k)=H^0(G_B,{\cal C})$ est l'espace des fonctions constantes, 
cette suite fournit la suite exacte
\begin{equation}\label{bato2}
0\to \Z_p\to H^1(G_B,W_k)\to H^1(G_B,{\cal C})^{(\gamma-1)^k=0}\to 0.
\end{equation}
Le $\Z_p$ est tu\'e dans $H^1(G_B,W_{k+1})$ car l'image de $x\in\Z_p$ dans $H^1(G_B,W_k)$
est celle du cocycle $\sigma\mapsto (\sigma-1)\cdot\phi_x$, o\`u $(\gamma-1)^k\cdot\phi_x=x$,
et donc $\phi_x$ est un polyn\^ome de degr\'e~$\leq k+1$, i.e.~$\phi_x\in W_{k+1}$.
Le r\'esultat s'en d\'eduit.
\end{proof}

Soit $W=W_k$.
Soit $\tau\mapsto c_\tau$ un $1$-cocycle sur $G_B$, \`a valeurs dans $W$.
\begin{prop}\phantomsection\label{bato40}
Il existe
$$c\in ([\epsilon^{1/p}]-1)^{-k}\tbA^{++}(\overline B)\otimes W,
\quad
x\in \tbA^+({\O(B_\infty)})\otimes W,$$ tels que 
$$(\varphi-1)c\in (\widehat\oplus_{i\in I_0}([\epsilon]-1)^{-k}\tA^{++}\tilde q^i)\otimes W
\quad{\rm et}\quad
c_\tau=\tfrac{\tau-1}{\gamma-1}x+(\tau-1)c.$$
\end{prop}
\begin{proof}
On peut \'ecrire $c_\tau$ sous la forme $\sum_{i=0}^kc_{i,\tau}\otimes u^i$.
La restriction de $\tau\mapsto c_{i,\tau}$ \`a $H_B$ est un $1$-cocycle \`a valeurs dans $\Z_p$.
Il existe $c_i\in \tbA(\overline B)$ tel que $c_{i,\tau}=(\tau-1)c_i$ si $\tau\in H_B$.
Il s'ensuit que $(\tau-1)\cdot(\varphi-1)c_i=0$ et donc $(\varphi-1)c_i\in \tbA(B_\infty)$.
L'isomorphisme de la prop.~\ref{cano3} permet, quitte \`a retrancher $(\varphi-1)a_i$
\`a $c_i$, avec $a_i\in \tbA(B_\infty)$ de supposer que
$(\varphi-1)c_i\in \widehat\oplus_{i\in I_0}\tA \tilde q^i$.
Le fait que la classe de $\tau\mapsto c_{i,\tau}$ est tu\'ee par $\big(\matrice{1}{1}{0}{1}-1\big)^k$
(cela suit de la suite exacte (\ref{bato2}))
\'equivaut \`a ce que 
$$(\varphi-1)c_i\in \widehat\oplus_{i\in I_0}([\epsilon^i]-1)^{-k}\tA^{++}\tilde q^i=
\widehat\oplus_{i\in I_0}([\epsilon]-1)^{-k}\tA^{++}\tilde q^i.$$
Il r\'esulte du lemme~\ref{bato1} ci-dessous que 
$$c_i\in ([\epsilon^{1/p}]-1)^{-k}\tbA^{++}(\overline B).$$

Soit $c=\sum_{i=1}^kc_i\otimes u^i.$
Alors, par construction $c$ v\'erifie les propri\'et\'es demand\'ees; il reste \`a construire $x$.
La classe de $\tau\mapsto c_\tau$ dans $H^1(H_B,W)$ est invariante
par $G_B$ (agissant par $(\gamma\cdot c)_\tau=\gamma\cdot c_{\gamma^{-1}\tau\gamma}$) puisque
c'est la restriction d'un $1$-cocycle sur $G_B$ (explicitement, $\gamma\cdot c_{\gamma^{-1}\tau\gamma}=
c_\tau+(\tau-1)\cdot c_\gamma$).
Cette classe est d\'etermin\'ee par $(\varphi-1)\cdot c$ modulo $\widehat\oplus_{i\in I_0}\tA^{++}\tilde q^i
\subset\tbA^{++}(B_\infty)$
d'apr\`es le (ii) de la prop.\,\ref{cano3};
on en d\'eduit que
$(\sigma-1)(\varphi-1)c\in \tbA^{++}(B_\infty)\otimes W$, pour tout $\sigma\in G_B$.
Il en r\'esulte, puisque $\varphi-1$ est bijectif sur
$\tbA^{++}(B_\infty)$ (d'inverse $-1-\varphi-\varphi^2-\cdots$), que
$$(\sigma-1)\cdot c\in (\Z_p\oplus\tbA^{++}(B_\infty))\otimes W,
\quad{\text{ pour tout $\sigma\in G_B$.}}$$

Soit $y=(\varphi-1)c\in \frac{1}{([\epsilon]-1)^k}\tbA^{+}(B_\infty)$.
Alors $(\gamma-1)y\in (\Z_p\oplus\tbA^{++}(B_\infty))\otimes W$ 
d'apr\`es ce qui pr\'ec\`ede,
et comme $\varphi-1$ est bijectif sur
$\tbA^{++}(B_\infty)$,
il existe $x_0\in (W(\overline{\bf F}_p)\oplus\tbA^{++}(B_\infty))\otimes W$ (unique \`a $W$ pr\`es),
tel que $(\gamma-1)y+(\varphi-1)x_0=0$.
Alors $\tau\mapsto \tfrac{\tau-1}{\gamma-1}x_0+(\tau-1)c$ est un $1$-cocycle sur $G_B$,
tu\'e par $\varphi-1$ et donc \`a valeurs dans $W$, qui co\"{\i}ncide avec $\tau\mapsto c_\tau$
sur $H_B$.  Le cocycle $c_\tau-(\tfrac{\tau-1}{\gamma-1}x_0+(\tau-1)c)$ provient donc,
par inflation, d'un $1$-cocycle $\sigma\mapsto x_\sigma$ sur $\UU(\Z_p)$, \`a valeurs dans $W$.
Il suffit alors de poser $x=x_0+x_\gamma$ pour conclure.
\end{proof}

\begin{lemm}\phantomsection\label{bato1}
Si 
$$(\varphi-1)x=([\epsilon]-1)^{-k}z,\quad{\text{avec $z\in\tbA^{++}(\overline B)$,}}$$
 alors
$$x=([\epsilon^{1/p}]-1)^{-k}y,\quad{\text{ avec $y\in\tbA^{++}(\overline B)$.}}$$
\end{lemm}
\begin{proof}
L'\'equation devient
$\varphi(y)-\xi^k y=z$, avec $\xi=\frac{[\epsilon]-1}{[\epsilon^{1/p}]-1}\in\tA^{++}$.
Modulo $p$, cette \'equation devient $y_0^p-\alpha y_0=z_0$, et comme $z_0\in\O^{++}(\overline B)^{\flat}$
et $\alpha\in {\goth m}_{C^\flat}$, on a $y_0\in \O^{++}(\overline B)^{\flat}$ (car $v(y_0)>0$ pour toute
valuation sur $\O(\overline B)^\flat$). On peut donc \'ecrire $y=[y_0]+py'$ et recommencer le raisonnement
avec $y'$, etc. Cela permet de prouver que $y=[y_0]+p[y_1]+p^2[y_2]+\cdots$, 
avec $y_i\in \O^{++}(\overline B)^{\flat}$, ce qui permet de conclure.
\end{proof}

\Subsection{Cohomologie de la boule ouverte}\label{qq12}
On suppose $C$ sph\'eriquement complet pour que $C/\O_C\to\varprojlim_{r>0}C/p^{-r}\O_C$
soit un isomorphisme, ce qui
implique en particulier
 que l'on a des isomorphismes
$$\varprojlim_{r>0}\tA/\tilde p^{-r}\tA^+=\tA/\tA^+:=\tA^-,\quad
\varprojlim_{r>0}\tfrac{1}{\tilde p^{ir}([\epsilon]-1)^k}\tA^+/\tfrac{1}{\tilde p^{ir}}\tA^+=
\tfrac{1}{([\epsilon]-1)^k}\tA^+/\tA^+.$$
Soient $B^-$ la boule \index{bou2@\bouleo}unit\'e ouverte et
$B^-_\infty$ la boule unit\'e ouverte perfecto\"{\i}de. Ce sont les r\'eunions croissantes, pour $r>0$,
des boules ferm\'ees (resp.~ferm\'ees perfecto\"{\i}des) 
$$B_r={\rm Spa}(C\langle (q/p^r)\rangle, \O_C\langle (q/p^r)\rangle),\quad
B_{r,\infty}={\rm Spa}(C\langle (q/p^r)^{p^{-\infty}}\rangle, \O_C\langle (q/p^r)^{p^{-\infty}}\rangle).$$
On note $B^{-,\times}$ et $B_\infty^{-,\times}$ les boules
$B^-$ et $B^-_\infty$ munies d'une structure logarithmique en $0$.

Les $H^0$ n'ayant pas de ${\rm R}^1\varprojlim$, on a
$H^1_{\proet}(B^{-,\times}_\infty,\Z_p)=\varprojlim_{r>0}H^1_{\proet}(B^{\times}_{r,\infty},\Z_p)$,
et comme
$\tA^-\to \varprojlim \tA/(\tilde p^{-r}\tA^{++})$
est un isomorphisme,
on obtient, gr\^ace \`a la prop.\,\ref{cano3}, une injection naturelle
\begin{equation}\label{bato40.2}
\iota_{\tA}:H^1_{\proet}(B^{-,\times}_\infty,\Z_p)\hookrightarrow
\prod_{i\in I_0}\tA^-\,\tilde q^i.
\end{equation}

On peut voir ${\rm Alg}$ comme un ind-syst\`eme local sur $B^-$, puisque c'est un ind-syst\`eme
local sur $B_r$, pour tout $r>0$.
On a des isomorphismes
\begin{align*}
H^1_{\proet}(B^{-,\times},{\rm Alg})&=\varinjlim\nolimits_k\varprojlim\nolimits_r H^1_{\proet}(B_r^\times,W_k)\\
&=\varinjlim\nolimits_k\varprojlim\nolimits_r H^1(G_{B_r},W_k)\cong 
\varinjlim\nolimits_k\varprojlim\nolimits_r H^1(H_{B_r},\Z_p)^{(\gamma-1)^k=0}
\end{align*}
Comme $H^1(H_{B_r},\Z_p)^{(\gamma-1)^k=0}\cong \widehat\oplus_{i\in I_0}
(([\epsilon]-1)^{-k}\tilde p^{-ir}\tA^{++}/\tilde p^{-ir}\tA^{++})\tilde q^i$,
et comme $\varprojlim (([\epsilon]-1)^{-k}\tilde p^{-ir}\tA^{++}/\tilde p^{-ir}\tA^{++})=
([\epsilon]-1)^{-k}\tA^+/\tA^+$,
on en d\'eduit une injection \index{iotaA@\iotaA}naturelle
\begin{equation}\label{bato40.4}
\iota_{\tA}:H^1_{\proet}(B^{-,\times},{\rm Alg})\hookrightarrow
\varinjlim_k\prod_{i\in I_0}(([\epsilon]-1)^{-k}\tA^+/\tA^+)\tilde q^i.
\end{equation}

\Subsection{La cohomologie du faisceau $\Bdr^+$}\label{qq13}
\subsubsection{Descente presque \'etale et d\'ecompl\'etion}\label{qq14}
On \index{bdrq@\bdrq}note $\bdr^+\{\tilde q\}$ 
l'anneau des $\sum a_i\tilde q^i$, avec $i\in \N$,
$a_i\in\bdr^+$ et, pour tout $k$ et tout $r>0$, on a $v_{\bdr^+/t^k}(a_i)+ir\to +\infty$
quand $i\to\infty$.

On note $\bdr^+\{\tilde q^{1/p^\infty}\}$ 
l'anneau des $\sum a_i\tilde q^i$, avec $i\in p^{-\infty}\N$,
$a_i\in\bdr^+$ et, pour tout $k$ et tout $r>0$, on a $v_{\bdr^+/t^k}(a_i)+ir\to +\infty$
quand $i\to\infty$ (suivant le filtre des compl\'ementaires des parties finies).
Alors $\bdr^+\{\tilde q^{1/p^\infty}\}$ est le compl\'et\'e
de $\varinjlim_n \bdr^+\{\tilde q^{1/p^n}\}$, et 
$$\Bdr^+(B^-_\infty)=\varprojlim_r\Bdr^+(B_{r,\infty})=\bdr^+\{\tilde q^{1/p^\infty}\}.$$

\begin{lemm}\phantomsection\label{qq14.1}
On a un isomorphisme naturel
$$H^1_{\rm proet}(B^{-,\times},{\rm Alg}\otimes\Bdr^+)\cong 
H^1(\UU(\Z_p),{\rm Alg}\otimes\bdr^+\{\tilde q^{1/p^\infty}\}).$$
\end{lemm}
\begin{proof}
On a
\begin{align*}
H^1_{\rm proet}(B^{-,\times},{\rm Alg}\otimes\Bdr^+) & \cong 
\varprojlim_r H^1_{\rm proet}(B_r^\times,{\rm Alg}\otimes\Bdr^+)\\
&\cong \varprojlim_r H^1(G_{B_r},{\rm Alg}\otimes\Bdr^+(\overline B_r))\\
&\cong \varprojlim_r H^1(\UU(\Z_p),{\rm Alg}\otimes\Bdr^+(B_{r,\infty}))\\
&\cong H^1(\UU(\Z_p),{\rm Alg}\otimes\bdr^+\{\tilde q^{1/p^\infty}\})
\end{align*}
le premier isomorphisme r\'esulte de ce que les $H^0$ v\'erifient la propri\'et\'e de Mittag-Leffler,
le second provient de ce que les $B_r^\times$ sont des $K(\pi,1)$,
le troisi\`eme 
est une application directe des m\'ethodes de descente presque \'etale,
et le dernier provient de ce que les $\Bdr^+(B_{r,\infty})$ n'ont pas de ${\rm R}^1\varprojlim$
car $\Bdr^+(B_{r,\infty})$ est dense dans $\Bdr^+(B_{s,\infty})$ si $r<s$ (Mittag-Leffler topologique).
\end{proof}

Si $P\in \bdr\otimes{\rm Alg}$ et si $k\in\N$, soit $P^{[k]}=(\gamma-1)^kP$. On a
$P^{[k]}=0$ si $k\geq \deg P$.
Un calcul imm\'ediat fournit le lemme suivant.
\begin{lemm}\phantomsection\label{cano105}
Si $i\in p^{-\infty}\N\moins\{0\}$,
on a
$$\tilde q^i\otimes P=(\gamma-1)\cdot\big(\tilde q^i\otimes\big(\tfrac{1}{[\epsilon^i]-1}P+
\tfrac{[\epsilon^i]}{([\epsilon^i]-1)^2}P^{[1]}+ \tfrac{[\epsilon^i]^2}{([\epsilon^i]-1)^3}P^{[2]}
+\cdots\big)\big)$$
\end{lemm}

\begin{lemm}\phantomsection\label{cano106}
L'injection $\bdr^+\{\tilde q\}\hookrightarrow \bdr^+\{\tilde q^{1/p^\infty}\}$
induit un isomorphisme
$$H^1(\UU(\Z_p), \bdr^+\{\tilde q\}\otimes{\rm Alg}) \overset{\sim}{\to}
H^1(\UU(\Z_p), \bdr^+\{\tilde q^{1/p^\infty}\}\otimes{\rm Alg}).$$
\end{lemm}
\begin{proof}
Cela r\'esulte facilement du lemme~\ref{cano105} 
et de la d\'efinition de $\bdr^+\{\tilde q^{1/p^\infty}\}$
\end{proof}

\subsubsection{L'application $\log_B $}\label{qq15}
On \index{bdrq@\bdrq}pose 
$$\bdr\{\tilde q\}=\bdr^+\{\tilde q\}[\tfrac{1}{t}]
\quad{\rm et}\quad
\bdr^-\{\tilde q\}=\bdr\{\tilde q\}/\bdr^+\{\tilde q\}.$$

\begin{lemm}\phantomsection\label{cano107}
{\rm (i)} On a des identit\'es
$$H^0(\UU(\Z_p),{\rm Alg}\otimes \bdr\{\tilde q\})=\bdr,
\quad
H^1(\UU(\Z_p),{\rm Alg}\otimes \bdr\{\tilde q\})=0$$
et une suite exacte
$$0\to\bdr^-\to H^0(\UU(\Z_p),{\rm Alg}\otimes \bdr^-\{\tilde q\})\to
H^1(\UU(\Z_p),{\rm Alg}\otimes \bdr^+\{\tilde q\})\to 0$$

{\rm (ii)} L'application $P\mapsto P(0)$
induit un isomorphisme
$$H^0(\UU(\Z_p),{\rm Alg}\otimes \bdr^-\{\tilde q\})\cong \bdr^-\{\tilde q\}.$$
\end{lemm}
\begin{proof}
Le calcul de $H^0$ est imm\'ediat; celui du $H^1$ r\'esulte du lemme~\ref{cano105} (en remarquant que
$\frac{(it)^k}{([\epsilon^i]-1)^k}$ est {\og entier\fg} et donc que la s\'erie qui sort du lemme converge).
La suite exacte s'en d\'eduit via la suite exacte longue de cohomologie associ\'ee
\`a $0\to\bdr^+\{\tilde q\}\to \bdr\{\tilde q\}\to \bdr^-\{\tilde q\}\to 0$
tensoris\'ee par ${\rm Alg}$.  Ceci prouve le~(i).

Pour prouver le (ii), il suffit de remarquer que si $x\in \bdr^-\{\tilde q\}$,
alors $\phi_x$, d\'efinie par $\phi_x(\tau)=
\tau\cdot x$, est un \'el\'ement de ${\rm Alg}\otimes \bdr^-\{\tilde q\}$,
et que $x\mapsto \phi_x$ est inverse de $P\mapsto P(0)$.
\end{proof}

Soit
$$\bdr^-\{\tilde q\}_0=\{\sum_i a_i\tilde q^i\in\bdr^-\{\tilde q\},
\ {\text {avec $a_0=0$}}\}.$$
On \index{lob@\loB}d\'eduit de ce qui pr\'ec\`ede un 
isomorphisme naturel (r\'eminiscent de l'inverse
de l'exponentielle de Bloch-Kato)
$$\log_B :H^1_{\proet}(B^{-,\times},\Bdr^+\otimes{\rm Alg})\overset{\sim}{\to} \bdr^-\{\tilde q\}_0.$$

\subsubsection{Une loi de r\'eciprocit\'e explicite}\label{qq16}
On pose
\begin{align*}
\tA^-[[\tilde q]]\boxtimes\Z_p^\dual&=
\prod_{i\in I_0}\tA^-\tilde q^i\\
(\tA^-[[\tilde q]]\boxtimes\Z_p^\dual)^{U-{\rm fini}}&=
\varinjlim_k\prod_{i\in I_0}(([\epsilon]-1)^{-k}\tA^+/\tA^+)\tilde q^i.
\end{align*}
On dispose donc, d'apr\`es~(\ref{bato40.4}) d'une application naturelle
$$\iota_{\tA}:H^1_{\proet}(B^{-,\times},{\rm Alg})\to 
(\tA^-[[\tilde q]]\boxtimes\Z_p^\dual)^{U-{\rm fini}}.$$

Par ailleurs,
si $\alpha\in ([\epsilon]-1)^{-k}\tA^+/\tA^+$, alors $\varphi^n(\alpha)=0$ dans $\bdr^-$
si $n< 0$.
Cela permet de d\'efinir une \index{kdr@\kdr}injection
$$\kappa_{\rm dR}:(\tA^-[[\tilde q]]\boxtimes\Z_p^\dual)^{U-{\rm fini}}\to 
\bdr^-\{ \tilde q\}_0, \quad
x \mapsto \sum_{n\in\Z}\varphi^n(x).$$
(La s\'erie converge dans $\bdr^-\{\tilde q\}$ car $t^k\varphi^n(([\epsilon]-1)^k\tilde q^i)\in
p^{-nk}\acris \tilde q^{p^ni}$; on obtient
 une injection car, si $x\in\tA^+$ est tel que $\varphi^n(x)\in t^k\bdr^+$ pour tout $n\geq 0$,
alors $x\in ([\epsilon]-1)^k\tA^+$.)
On en d\'eduit une fl\`eche
$$\kappa_{\rm dR}\circ \iota_{\tA}: H^1_{\proet}(B^{-,\times},{\rm Alg})\to \bdr^-\{\tilde q\}_0$$
(Remarquons que $\kappa_{\rm dR}\circ (\varphi-1)=0$, ce qui
fait que $\kappa_{\rm dR}\circ\iota_{\tA}$ ne d\'epend pas du choix du repr\'esentant fait
dans la d\'efinition de $\iota_{\tA}$.)

\begin{theo}\phantomsection\label{bato22}
$\log_B \circ\  \iota =\kappa_{\rm dR}\circ \iota_{\tA}$.
\end{theo}
\begin{proof}
Il s'agit de prouver que le diagramme suivant commute
$$\xymatrix@C=.4cm@R=5mm{
\bdr^-\{\tilde q\}_0\ar[r]^-{\sim}&H^0(\UU(\Z_p),\bdr^-\{\tilde q\}_0\otimes{\rm Alg})
\ar[r]^-{\sim} &H^1(\UU(\Z_p),\bdr^+\{\tilde q\}\otimes{\rm Alg})\\
(\tA^-[[\tilde q]]\boxtimes\Z_p^\dual)^{U-{\rm fini}}\ar[u]^-{\kappa_{\rm dR}}
&H^1_{\proet}(B^{-,\times},{\rm Alg})\ar[l]_-{\iota_{\tA}}\ar[r]^-{\iota}
&H^1_{\proet}(B^{-,\times},\Bdr^+\otimes{\rm Alg})\ar[u]_-{\wr}
}$$
(i.e.~que les deux fl\`eches 
$H^1_{\proet}(B^{-,\times},{\rm Alg})\to H^1(\UU(\Z_p),\bdr^+\{\tilde q\}\otimes{\rm Alg})$ co\"{\i}ncident).
Soit $\hat c\in H^1_{\proet}(B^{-,\times},{\rm Alg})$.
D'apr\`es la prop.~\ref{bato40}, il existe, pour tout $r>0$,
\begin{align*}
&c_r \in \tfrac{1}{([\epsilon^{1/p}]-1)^k}\tbA^{++}(\overline B_r)\otimes{\rm Alg},\\
&y_r=(\varphi-1)c_r \in \tfrac{1}{([\epsilon]-1)^k}\tbA^{++}(B_{r,\infty})\otimes{\rm Alg},\\
&(\gamma-1)y_r \in \tbA^{++}(B_{r,\infty})\otimes{\rm Alg},\\
&x_r\equiv(1+\varphi+\varphi^2+\cdots)(\gamma-1)y_r 
\in \tbA^{++}(B_{r,\infty})\otimes{\rm Alg}\ {\rm mod}\ {\rm Alg}
\end{align*}
tels que $\tau\mapsto c_{r,\tau}=\frac{\tau-1}{\gamma-1}x_r+(\tau-1)c_r$ est
un $1$-cocycle sur $G_{B_r}$ repr\'esentant l'image de $\hat c$ dans $H^1(G_{B_r},{\rm Alg})$.
Comme
$\tbA^{++}(B_{r,\infty})\subset \Bdr^+(B_{r,\infty})$, on peut voir $x_r$ comme un \'el\'ement
de $\Bdr^+(B_{r,\infty})$, ce qui repr\'esente $\tau\mapsto \frac{\tau-1}{\gamma-1}x_r$ comme
l'inflation d'un cocycle de $H^1(\UU(\Z_p),\Bdr^+(B_{r,\infty})\otimes{\rm Alg})$
qui est
l'image de $\tau\mapsto c_{r,\tau}$ dans $H^1(G_B,\Bdr^+(\overline B_r)\otimes{\rm Alg})$
(car $c_r\in\Bdr^+(\overline B_r)\otimes {\rm Alg}$ puisque
$[\epsilon^{1/p}]-1$ est inversible dans $\bdr^+$).

Dans l'autre sens, 
soit $z_r=(1+\varphi+\varphi^2+\cdots)y_r$
(la s\'erie $(1+\varphi+\varphi^2+\cdots)y_r$ converge dans $\tA[[\tilde q/\tilde p^r]]\otimes{\rm Alg}$
et les coefficients sont dans $\cup_{n\in\N}\frac{1}{([\epsilon^{p^n}]-1)^k}\tA^+[\frac{1}{\tilde p}]$
qui s'injecte dans $t^{-k}\bdr^+$; la s\'erie qui en r\'esulte converge 
dans $\Bdr(B_{r,\infty}^-)\otimes{\rm Alg}$ mais
pas, a priori, dans $\Bdr(B_{r,\infty})\otimes{\rm Alg}$).
L'image de $\hat c$ dans 
$(\tA^-[[\tilde q/\tilde p^r]]\boxtimes\Z_p^\dual)^{U-{\rm fini}}$ est celle de $y_r(0)$
dont l'image par $\kappa_{\rm dR}$ est celle de $z_r(0)$.
Son image dans $H^1(\UU(\Z_p),\Bdr^+(B_{r,\infty}^-)\otimes{\rm Alg})$
est celle de
$\tau\mapsto (\tau-1)z_r$, et comme $x_r=(\gamma-1)z_r$,
le cocycle pr\'ec\'edent est aussi $\tau\mapsto \frac{\tau-1}{\gamma-1}x_r$.

Les deux fl\`eches co\"{\i}ncident donc dans
$\varprojlim_r H^1(\UU(\Z_p),\Bdr^+(B_{r,\infty}^-)\otimes{\rm Alg})$
qui est \'egal \`a $H^1(\UU(\Z_p), \varprojlim_r \Bdr^+(B_{r,\infty}^-)\otimes{\rm Alg})$
pour les m\^emes raisons
 de crit\`ere de Mittag-Leffler topologique que pour la preuve du lemme~\ref{qq14.1}.
Comme $\varprojlim_r \Bdr^+(B_{r,\infty}^-)=\Bdr^+(B^-)=\bdr^+\{\tilde q\}$,
cela permet de conclure.
\end{proof}

Soit $W=W_k$. On voit $W$ comme une repr\'esentation du groupe fondamental
$G_B$  de $B^-$ via le morphisme $G_B\to \UU(\Z_p)$.
La fl\`eche $W\otimes W^\dual\to {\rm Alg}$ d\'efinie par
$v\otimes\check v\mapsto \phi_{\check v,v}$, 
avec $\phi_{\check v,v}(x)=\langle x\cdot\check v,v\rangle$, est $G_B\times \UU(\Z_p)$-\'equivariante
(en faisant agir $G_B$ sur $W$ et $\UU(\Z_p)$ sur $W^\dual$, et les deux
sur ${\rm Alg}$ par la formule habituelle).
Ceci fournit une fl\`eche $\UU(\Z_p)$-\'equivariante
 $$H^1(G_B,W)\otimes W^\dual\to H^1(G_B,{\rm Alg}).$$
\begin{lemm}\phantomsection\label{cano108}
On suppose qu'il existe $c\in\bdr\{\tilde q\}\otimes W$ 
tel que $v_\tau=(\tau-1)\cdot c$ pour tout $\tau\in G_B$.
Alors $\log_B(v\otimes\check v)$ est l'image de 
$\langle \check v,c\rangle$ dans $\bdr^-\{\tilde q\}$.
\end{lemm}
\begin{proof}
Si $w=(\tau\mapsto w_\tau)$ est un $1$-cocycle sur $G_B$, \`a valeurs dans $\bdr\{\tilde q\}\otimes W$, et si
$\check w\in W^\dual$, alors $w\otimes\check w$ est repr\'esent\'e
par le $1$-cocycle
$\tau\mapsto \phi_{\check w,w_\tau}$,
\`a valeurs dans $\bdr\{\tilde q\}\otimes {\rm Alg}$. Maintenant,
$\phi_{\check v,\tau\cdot c}=\tau\cdot\phi_{\check v,c}$, et donc
$v\otimes\check v$ est repr\'esent\'e par $\tau\mapsto (\tau-1)\cdot \phi_{\check v,c}$.
Par d\'efinition, $\log_B(v\otimes\check v)$ est donc l'image de 
$\phi_{\check v,c}(0)=\langle \check v,c\rangle$ dans $\bdr^-\{\tilde q\}$, ce que l'on voulait.
\end{proof}

{
\Subsection{La cohomologie du faisceau $\widehat\O$}\label{luepan0}
Posons $U_0={\mathbb U}(\Z_p)$,
$\gamma_u=\matrice{1}{u}{0}{1}$, $\gamma=\gamma_1$,
 et notons $\nu:U_0\to\Z_p$ le morphisme $\gamma_u\mapsto u$. 
On a $\O^+(B^-)=\O_C[[q]]$, $\O^+(B^-_\infty)=\O[[q^{1/p^\infty}]]$
et $U_0$ agit sur $\O^+( B_\infty^-)$ 
 par $\gamma_u\cdot q^{1/p^n}=\zeta_{p^n}^u q^{1/p^n}$.
On pose:
$$\O^+(B^{-,\times})=\O_C[[q]]_0, \quad \O^+(B^{-,\times}_\infty)=\O[[q^{1/p^\infty}]]_0$$
le $0$ en indice indiquant les s\'eries dont le terme constant est $0$.
Si $X=B^-, B^-_\infty, B^{-,\times}, B^{-,\times}_\infty$, on pose
$\O^b(X)=C\otimes_{\O_C}\O^+(X)$; c'est l'espace des fonctions analytiques born\'ees sur $X$
(de terme constant $0$ dans le cas de $B^{-,\times}, B^{-,\times}_\infty$).
On note $\O(X)$ l'espace des fonctions analytiques (non n\'ecessairement born\'ees).

On a $B^{-,\times}= B_\infty^{-,\times}/U_0$.
Ceci se traduit par les identifications
\begin{align*}
&\O(B^{-,\times}_\infty)^{U_0}=\O(B^{-,\times}),
&&\O(B^-_\infty)^{U_0}=\O(B^-),\\
&\O^b(B^-_\infty)^{U_0}=\O^b(B^-),
&&\O^b(B^{-,\times}_\infty)^{U_0}=\O^b(B^{-,\times})
\end{align*}
Comme ${U_0}$ est procyclique, on en d\'eduit une suite exacte:
$$0\to H^1({U_0},H^0_{\proet}( B_\infty^{-,\times},\widehat\O))\to
H^1_{\proet}(B^{-,\times},\widehat\O)\to H^0({U_0},H^1_{\proet}( B_\infty^{-,\times},\widehat\O))\to 0$$
Maintenant, $ B_\infty^-$ est la r\'eunion croissante de boules 
ferm\'ees perfecto\"{\i}des $\widetilde B_n$,
pour lesquelles $H^1_{\proet}(\widetilde B^\times_n,\widehat\O)=0$ (puisque ce sont 
des affino\"{\i}des perfecto\"{\i}des) et on a ${\rm R}^1\lim\O(\widetilde B^\times_n)=0$
car les $\widetilde B_n$ forment un recouvrement stein de $ B_\infty^-$;
il s'ensuit que $H^1_{\proet}( B_\infty^{-,\times},\widehat\O)=0$
et donc que l'on a un isomorphisme
$$H^1({U_0},H^0_{\proet}( B_\infty^{-,\times},\widehat\O))\cong
H^1_{\proet}(B^{-,\times},\widehat\O)$$
Nous allons donner deux descriptions de $H^1({U_0},\O( B_\infty^{-,\times}))$. 
Un \'el\'ement de $H^1({U_0},\O( B_\infty^{-,\times}))$ est d\'etermin\'e par un cocycle 
$\gamma_u\mapsto c_{u}$
et un tel cocycle est compl\`etement d\'etermin\'e par $c_{1}$ puisque $\Z_p$ est procyclique.

\begin{lemm}\phantomsection\label{luepan1}
$x\mapsto x\cup \nu$, o\`u $\nu$ est vu comme \'el\'ement de $H^1({U_0},\Z_p)$,
induit des isomorphismes
\begin{align*}
\cup\ \nu&:\O^b(B^{-,\times})=H^0({U_0},\O^b( B_\infty^{-,\times}))\overset{\sim}{\to} H^1({U_0},\O^b( B_\infty^{-,\times}))\\
\cup\ \nu&:\O(B^{-,\times})=H^0({U_0},\O( B_\infty^{-,\times}))\overset{\sim}{\to} H^1({U_0},\O( B_\infty^{-,\times}))
\end{align*}
On note $$\exp_B^\dual:H^1({U_0},\O^b( B_\infty^{-,\times}))\to \O^b(B^{-,\times}),
\quad \exp_B^\dual:H^1({U_0},\O( B_\infty^{-,\times}))\to \O(B^{-,\times})$$
les isomorphismes r\'eciproques {\rm (r\'eminiscents de l'exponentielle duale de Bloch-Kato)}.
\end{lemm}
\begin{proof}
C'est parfaitement classique. Plus pr\'ecis\'ement,
le conoyau de $\O^+(B^{-,\times})\to H^1({U_0},\O^+( B_\infty^{-,\times}))$ est tu\'e par $\zeta_p-1$.
(Cela traite le cas de $\O^b$; celui de $\O$ s'en d\'eduit en \'ecrivant $B_\infty$ comme une r\'eunion
croissante de boules plus petites et en passant \`a la limite.)
\end{proof}

\begin{rema}\phantomsection\label{luepan1.1}
Les isomorphismes ci-dessus ne sont pas \'equivariants:
$\sigma\in G_{\Q_p}$ agit par multiplication par
$\cyp (\sigma)^{-1}$ sur $\nu$ car $\sigma\cdot\nu$ est
 le cocycle $u\mapsto\sigma(\nu(\sigma^{-1}u\sigma))$
o\`u le produit se fait dans le groupe $\matrice{G_{\Q_p}}{\Z_p}{0}{1}$ et $\matrice{\sigma}{u}{0}{1}
\matrice{\tau}{v}{0}{1}=\matrice{\sigma\tau}{u+\cyp (\sigma)v}{0}{1}$.
\end{rema}

Soit ${\mathbb B}_2(B_\infty^{-})=\Bdr^+(B_\infty^-)/t^2$ (donc ${\mathbb B}_2(B_\infty^{-})$ est
constitu\'e de s\'eries en $\tilde q^{1/p^{\infty}}$ \`a coefficients dans $\B_2=\bdr^+/t^2$,
v\'erifiant des conditions de croissance).
On note ${\mathbb B}_2(B_\infty^{-,\times})\subset {\mathbb B}_2(B_\infty^{-})$ les s\'eries de terme constant $0$.
\begin{lemm}\phantomsection\label{luepan2}
La suite exacte
$$0\to \O( B_\infty^{-,\times})\to t^{-1}{\mathbb B}_2(B_\infty^{-,\times})\to t^{-1}\O( B_\infty^{-,\times})\to 0$$
induit, en passant \`a la cohomologie de ${U_0}$, un diagramme commutatif
$$\xymatrix@R=4mm@C=5mm{
H^0({U_0},t^{-1}\O( B_\infty^{-,\times}))\ar[r]\ar@{=}[d]&H^1({U_0},\O( B_\infty^{-,\times}))\\
t^{-1}\O(B^{-,\times})\ar[r]^-{t\partial}\ar[r]_-{\sim}& \O(B^{-,\times})\ar[u]_-{\wr}^{\cup\,\nu}
}$$
o\`u l'isomorphisme $\cup\,\nu$ \`a droite est celui du lemme~\ref{luepan1}.
Autrement dit, l'isomorphisme 
$$\log_B^0:H^1({U_0},\O( B_\infty^{-,\times}))\overset{\sim}{\to}t^{-1}\O(B^{-,\times})$$
que l'on d\'eduit, en passant par la gauche du diagramme, v\'erifie la relation
$$\log_B^0=t^{-1}\partial^{-1}\circ\exp^\dual_B$$
\end{lemm}
\begin{proof}
On a 
$$(\gamma-1)\cdot a\,\tilde q^i=([\epsilon^i]-1)\,a\,\tilde q^i,\quad{\text{si $a\in \B_2$ et $i\in p^{-\infty}\N$}}.$$
Si $i\neq 0$, ceci est $0$ dans $\B_2$ si et seulement si $a=0$ ou $i\in\N$ et $a\in t\B_2$.
On en d\'eduit que
$H^0({U_0},t^{-1}{\mathbb B}_2(B_\infty^{-,\times}))=\O(B^{-,\times})=H^0({U_0},\O( B_\infty^{-,\times}))$ 
ce qui fournit une injection
$t^{-1}\O(B^{-,\times})=H^0({U_0},t^{-1}\O( B_\infty^{-,\times}))
\hookrightarrow H^1({U_0},\O( B_\infty^{-,\times}))$.

Si $x=t^{-1}\sum_{n\in\N}a_n\,q^n\in t^{-1}\O(B^{-,\times})$,
si $\tilde a_n\in\bdr^+/t^2$ est un rel\`evement de $a_n$, 
alors l'image de $x$ dans $H^1({U_0},\O( B_\infty^{-,\times}))$ est repr\'esent\'ee par le $1$-cocycle
$\gamma_u\mapsto c_u=\theta((\gamma_u-1)\cdot \tilde x)$, 
avec $\tilde x=t^{-1}\sum_{n\in\N}\tilde a_n \tilde q^n$.
Comme
$\theta(t^{-1}([\epsilon^n]-1))=n$, on obtient
$c_{1}=\sum_n n\,a_n\, q^n=t\,\partial x$.

Le r\'esultat s'en d\'eduit.
\end{proof}

\begin{rema}\phantomsection\label{luepan3}
{\rm (i)}
En reprenant les arguments du \S\,\ref{qq13} et en utilisant le fait
que $\gamma-1:\bdr^+\{\tilde q\}_0\to t\bdr^+\{\tilde q\}_0$ est un isomorphisme
(d'inverse $\sum a_n\,t\tilde q^n\mapsto\sum a_n\tfrac{t}{[\epsilon]^n-1}\tilde q^n$),
on voit que l'application naturelle 
$$\theta:H^1({U_0},\Bdr^+(B_\infty^{-,\times}))\to
H^1({U_0},\O(B_\infty^{-,\times}))$$ est un isomorphisme.  On en d\'eduit un diagramme commutatif
(les fl\`eches verticales sont obtenues en envoyant
les fonctions constantes dans les fonctions polynomiales 
et la fl\`eche de droite est induite par l'inclusion $t^{-1}C\hookrightarrow \bdr^-$):

{\Small
$$\xymatrix@C=5mm@R=3mm{
H^1_{\proet}(B^{-,\times},\Z_p)\ar[dd]\ar[r]^-{\iota}& H^1_{\proet}(B^{-,\times},\widehat\O)
&H^1({U_0},\O( B_\infty^{-,\times}))\ar[l]_-{\sim}\ar[r]^-{\log_B^0}_-{\sim} 
& t^{-1}C\{\tilde q\}_0\ar[dd]\\
&&H^1({U_0},\Bdr^+( B_\infty^{-,\times}))\ar[u]_-{\wr}\ar[d]\\
H^1_{\proet}(B^{-,\times},{\rm Alg})\ar[r]^-{\iota}& H^1_{\proet}(B^{-,\times},\Bdr^+\otimes{\rm Alg})
&H^1({U_0},\Bdr^+( B_\infty^{-,\times})\otimes{\rm Alg})\ar[l]_-{\sim}
\ar[r]^-{\log_B}_-{\sim}\ar[r] & \bdr^-\{\tilde q\}_0}
$$
}

{\rm (ii)} On dispose d'un second diagramme commutatif:
$$\xymatrix@C=4mm@R=4mm{
H^1_{\proet}(B^{-,\times},\Z_p)\ar[d]\ar[r]^-{\iota}& H^1_{\proet}(B_\infty^{-,\times},\Z_p)^{U_0}\ar[r]^-{\iota_{\tA}}\ar[d]
&(\tA^-[[\tilde q]]\boxtimes\Z_p^\dual)^{U_0}\ar[r]^-{\kappa_{\rm dR}}\ar[d]
& t^{-1}C\{\tilde q\}_0\ar[d]\\
H^1_{\proet}(B^{-,\times},{\rm Alg})\ar[r]^-{\iota}& 
H^1_{\proet}(B_\infty^{-,\times},\Z_p)^{{U_0}{\text{-fini}}}\ar[r]^-{\iota_{\tA}}
&(\tA^-[[\tilde q]]\boxtimes\Z_p^\dual)^{{U_0}{\text{-fini}}}\ar[r]^-{\kappa_{\rm dR}}
& \bdr^-\{\tilde q\}_0}
$$
Il r\'esulte du th.\,\ref{bato22} 
que les fl\`eches $H^1_{\proet}(B^{-,\times},\Z_p)\to t^{-1}C\{\tilde q\}_0$ que l'on
d\'eduit des deux diagrammes ci-dessus co\"{\i}ncident.
\end{rema}

}

\section{Un mod\`ele de Kirillov pour la cohomologie compl\'et\'ee}\label{modkir}
Dans ce chapitre, on d\'efinit un mod\`ele de Kirillov pour
$H^1_{\rm proet}(\widehat X(0)_C^\times,\O_L)$.
L'injectivit\'e de ce mod\`ele est \'etudi\'ee dans les chap.\,\ref{rCDN0} et~\ref{IG0},
et les cons\'equences
de son injectivit\'e au chap.\,\ref{Ki114}:
compatibilit\'e entre les correspondances
de Langlands locales $p$-adique et classique (th.\,\ref{Ki119}), 
conjecture de Fontaine-Mazur (th.\,\ref{Ki121}),
repr\'esentations ordinaires (prop.\,\ref{Ki120}).
\Subsection{Cohomologie compl\'et\'ee et alg\`ebres de Hecke}\label{EUL8}
\subsubsection{La tour compl\'et\'ee}\label{qq1.6}
Si $(N,p)=1$, on \index{X3@\xxxo}note $X(Np^\infty)$ la limite projective des $X(Np^k)$,
et $\widehat X(Np^\infty)$ sa compl\'et\'ee (qui est une courbe perfecto\"{\i}de~\cite{Sz2,wein,JE1}).

On note $X(0)$ la limite projective des $X(N)$, $\widehat{X}^{(p)}(0)$
la limite projective des $\widehat X(Np^\infty)$, et $\widehat X(0)$ la compl\'et\'ee
de $X(0)$ 
($\widehat{X}^{(p)}(0)$ est seulement compl\'et\'ee le long de la tour de niveau $p^\infty$).

On rajoute un $C$ en indice (i.e.~$X(N)_C$, $\widehat{X}^{(p)}(0)_C$, etc.) pour d\'enoter
l'extension des scalaires \`a $C$. Il y a une subtilit\'e dans cette op\'eration dans le
cas des tours compl\'et\'ees (cf.~(ii) de la rem.\,\ref{luepan4} ci-dessous).
\begin{rema}\phantomsection\label{luepan4}
{\rm (i)}
Soit $F_\infty=\widehat{\Q_p(\bmu_{p^\infty})}$. 
Alors ${\rm Aut}(F_\infty/\Q_p)=\Z_p^\dual$,
o\`u $a\in\Z_p^\dual$ agit par $\sigma_a$ avec $\sigma_a(\zeta)=\zeta^a$, si $\zeta\in\bmu_{p^\infty}$.
On dispose d'une application naturelle ${\cal F}:C\wotimes F_\infty\to{\cal C}(\Z_p^\dual,C)$, envoyant
$x\otimes y$ sur $\phi_{x\otimes y}$ d\'efinie par $\phi_{x\otimes y}(a)=x\,\sigma_a(y)$.
(La m\^eme formule induit un isomorphisme
$C\otimes \Q_p(\bmu_{p^\infty})\overset{\sim}{\to}{\rm LC}(\Z_p^\dual,C)$.)

Fresnel et de Mathan~\cite{FM1,FM2,FM3}
 ont prouv\'e que ${\cal F}$ est surjective\footnote{Une preuve alternative a r\'ecemment
\'et\'e obtenue par Ophir~\cite{ophir}.} 
et que le noyau\footnote{Ce noyau n'est pas 0; cela peut se voir en regardant l'espace des vecteurs
$\eta$-isotypiques, si $\eta:\Z_p^\dual\to C^\dual$ est un caract\`ere continu: 
dans ${\cal C}(\Z_p^\dual,C)$,
cet espace est de dimension~$1$ sur $C$, engendr\'e par $\eta$, et dans
$C\wotimes F_\infty$ cet espace est $0$ si $\eta$ n'est pas d'ordre fini comme on le voit
en utilisant les traces normalis\'ees $T_n:F_\infty\to\Q_p(\bmu_{p^n})$.}
de ${\cal F}$
est l'adh\'erence des \'el\'ements nilpotents: on a donc envie de dire que 
${\rm Spa}(C\wotimes F_\infty)$ est un \'epaississement infinit\'esimal de $\Z_p^\dual$.

{\rm (ii)}
Le (i) impose de modifier l'extension des scalaires \`a $C$:
il faut quotienter par l'adh\'erence des nilpotents,
ce qui transforme $C\wotimes \Q_p(\mu_{Np^\infty})$ en ${\cal C}((\Z/Np^\infty)^\dual,C)$
(i.e.~l'espace des composantes connexes de $\widehat X(Np^\infty)_C$ est 
$(\Z/Np^\infty)^\dual$ et pas un \'epaississement infinit\'esimal; celui de
$\widehat X(0)_C$ est $\cZ^\dual$).
\end{rema}

\subsubsection{Le voisinage de la pointe~$\infty$}\label{Ki101}
L'espace $\widehat X(0)$ est l'espace de modules
des triplets $(E,e_1^\dual,e_2^\dual)$, o\`u $E$ est une courbe elliptique et
$e_1^\dual,e_2^\dual$ est une base sur $\cZ$ du module de Tate ad\'elique $T_{\cZ}E$ de $E$.
L'action de $\matrice{a}{b}{c}{d}\in\GG(\cZ)$ sur $\widehat X(0)$ envoie
$(E,e_1^\dual,e_2^\dual)$ sur $(E,ae_1^\dual+ce_2^\dual, be_1^\dual+de_2^\dual)$.

On note $\widehat Z(0)\subset \widehat X(0)$ la composante connexe arithm\'etique du lieu multiplicatif
param\^etrant les $(E,e_1^\dual,e_2^\dual)$, o\`u $E\cong {\bf G}_m/q^\Z$, 
$e_1^\dual$ est un g\'en\'erateur de $T_{\cZ}\widehat{\bf G}_m$ et
$e_2^\dual=(q^{1/N})_N$ modulo $T_{\cZ}\widehat{\bf G}_m$.
Autrement dit, $\widehat Z(0)$ est le voisinage (arithm\'etique) usuel de la pointe $\infty$.

On d\'efinit de m\^eme $Z(Np^\infty)\subset X(Np^\infty)$, 
$\widehat Z(Np^\infty)\subset \widehat X(Np^\infty)$, etc.
Alors $\O^+(\widehat Z(0))$ est le compl\'et\'e $(p,q)$-adique 
$\widehat{\Z[\bmu]}[[q^{\Q_+}]]$ de $\Z[\zeta_N,q^{1/N},\,N\geq 1]$
tandis que $\O^+(\widehat Z^{(p)}(0))$ est la limite inductive pour $(N,p)=1$
des $\widehat{\Z[\bmu_{Np^\infty}]}[[q^{1/Np^\infty}]]$.

L'espace $\widehat Z(0)\subset \widehat X(0)$ est stable par $\PP(\Ai)\subset\GG(\Ai)$;
on a, si $i\in\Q_+$ et $\zeta\in\bmu$,
\begin{align*}
&\matrice{u}{0}{0}{1}\cdot\zeta=\zeta^u, \quad\matrice{u}{0}{0}{1}\cdot q^{i}=q^{i},
\quad{\text{si $u\in\cZ^\dual$,}}\\
&\matrice{n}{0}{0}{1}\cdot\zeta=\zeta, \quad\matrice{n}{0}{0}{1}\cdot q^{i}=q^{i/n},
\quad{\text{si $n\in\Q_+^\dual$,}}\\
&\matrice{1}{b}{0}{1}\cdot\zeta=\zeta, \quad \matrice{1}{b}{0}{1}\cdot q^{i}
={\bf e}_{\A}(bi)q^{i},
\quad{\text{si $b\in\Ai$,}}
\end{align*}

\begin{rema}\phantomsection\label{Ki102}
{\rm (i)}
Si $\lambda\in \O^+(\widehat Z(0))$, soit ${\cal K}_\lambda:\Aidu\to \Z_p\wotimes\Z[\bmu]$
la fonction envoyant $x$ sur le terme de degr\'e $1$ de $\matrice{x}{0}{0}{1}\cdot\lambda$.
Si $\lambda=\sum a_iq^i$, 
alors  $\lambda-a_0=\sum_{i>0}{\cal K}_\lambda(i)q^i$; autrement dit
les donn\'ees de $\lambda$ et ${\cal K}_\lambda$ sont presque \'equivalentes.

{\rm (ii)}
Si $x=nu$, avec $u\in\cZ^\dual$ et $n\in\Q_+^\dual$, alors
$\matrice{x}{0}{0}{1}\cdot\lambda=\sum\sigma_u(a_i)q^{i/n}$ et donc
${\cal K}_\lambda(nu)=\sigma_u(a_n)$.
Il s'ensuit que ${\cal K}_\lambda(ux)=\sigma_u({\cal K}_\lambda(x))$ pour tous
$u\in\cZ^\dual$ et $x\in\Aidu$, et donc que
$${\cal K}_\lambda\in {\cal C}(\Aidu, \Z_p\wotimes\Z[\bmu])^{\cZ^\dual}$$
o\`u on fait
agir $u\in\cZ^\dual$ par $(u\cdot\phi)(x)=\sigma_u(\phi(u^{-1}x))$, i.e. par
$\sigma_u\otimes\matrice{u^{-1}}{0}{0}{1}$ si on factorise
${\cal C}(\Aidu, \Z_p\wotimes\Z[\bmu])$ sous la forme
$\Z[\bmu]\wotimes{\cal C}(\Aidu, \Z_p)$.

{\rm (iii)}
Comme $\matrice{x}{0}{0}{1}\matrice{a}{b}{0}{1}=\matrice{1}{bx}{0}{1}\matrice{ax}{0}{0}{1}$,
le terme de degr\'e $1$ de $\matrice{x}{0}{0}{1}\matrice{a}{b}{0}{1}\cdot\lambda$
est ${\bf e}_{\A}(bx)$ fois celui de $\matrice{ax}{0}{0}{1}\cdot\lambda$
(notons que ${\bf e}_\A(bx)\in\Z[\bmu]$),
ce qui se traduit par
$${\cal K}_{\matrice{a}{b}{0}{1}\cdot\lambda}={\bf e}_{\A}(bx){\cal K}_\lambda(ax)$$
Autrement dit, {\it $\lambda\mapsto {\cal K}_\lambda$ est un mod\`ele de Kirillov}
\begin{equation}\label{Ki130}
{\cal K}:\O^+(\widehat Z(0))_0\hookrightarrow
{\cal C}(\Aidu, \Z_p\wotimes\Z[\bmu])^{\cZ^\dual}
\end{equation}
o\`u le $0$ en indice indique les s\'eries (en $q^{\Q_+}$) de terme constant nul.

(iv) On a $\partial(\sum_ia_i q^i)=\sum_iia_iq^i$. On en d\'eduit que
${\cal K}_{\partial\lambda}(nu)=n\sigma_u(a_n)=|nu|_\A^{-1}{\cal K}_\lambda(nu)$, et
donc ${\cal K}_{\partial\lambda}=|\ |_\A^{-1}{\cal K}_\lambda$.
\end{rema}

\subsubsection{La cohomologie de la tour compl\'et\'ee}\label{Ki103}
On rajoute un $\times$ en exposant (e.g. $X(N)_C^\times$, $\widehat Z(Np^\infty)^\times_C$,...)
pour indiquer que l'on met une structure logarithmique
aux pointes (dans le cas des $X$) ou de $0$ (dans le cas des $Z$).
Cela a pour effet de permettre les rev\^etements profinis ramifi\'es en les points
o\`u il y a une structure logarithmique dans la d\'efinition de la cohomologie (pro)\'etale.

Alors $H^1_{\eet}(\widehat X(0)_C^\times,\Z_p)$ est la cohomologie compl\'et\'ee de la tour
(en tant que repr\'esentation de $\GG(\A^{]\infty[})$, c'est $H^1(\GG(\Q),{\cal C}(\GG(\A),\Z_p))$).
Le groupe $H^1_{\eet}(\widehat X(Np^\infty)^\times_C,\Z_p)$ est le sous-groupe 
$H^1(\GG(\Q),{\cal C}(\GG(\A)/\wGamma(Np^\infty),\Z_p))$ des points fixes
par $\Gamma(Np^\infty)$.
Par comparaison, $H^1_{\eet}(X(0)_C^\times,\Z_p)$ est la 
limite inductive des $H^1_{{\eet}}(X(N)_C,\Z_p)$
(en tant que repr\'esentation de $\GG(\A^{]\infty[})$, c'est $H^1(\GG(\Q),{\rm LC}(\GG(\A),\Z_p))$).

\subsubsection{Alg\`ebres de Hecke}\label{Ki104}
Soit $\TT(N)$ la $\O_L$-alg\`ebre de Hecke de niveau $N$ (engendr\'ee
par les $T_\ell$ et les $S_\ell$, pour $\ell\nmid pN$) et $\TT=\varprojlim_N\TT(N)$.
Alors $\TT$ est une alg\`ebre commutative semi-locale, r\'eduite, non noeth\'erienne
(mais limite projective d'alg\`ebres noeth\'eriennes),
qui agit fid\`element sur $H^1(\GG(\Q),{\cal C}(\GG(\A),\O_L))$.
On peut donner une description de $\TT$ via l'isomorphisme
$H^1(\GG(\Q),{\cal C}(\GG(\A),\O_L))\cong H^1_{{\eet}}(X(0)^\times_C,\O_L)$ et l'action de Galois
sur ce dernier module:

Si $\sigma\in G_\Q$, notons $s_\TT(\sigma), 
s_\TT'(\sigma), t_\TT(\sigma)\in{\rm End}(H^1_{{\eet}}(X(0)^\times_C,\O_L))$ 
les op\'erateurs
$$s_\TT(\sigma):=\matrice{\cy(\sigma)^{-1}}{0}{0}{\cy(\sigma)^{-1}},
\quad s'_\TT(\sigma):=\cyp (\sigma)s_\TT(\sigma),\quad
t_\TT(\sigma):=\sigma+s'_\TT(\sigma)\sigma^{-1}$$
(Dans la d\'efinition de $s_\TT(\sigma)$,
 le membre de droite appartient \`a $\GG(\cZ)$ et agit \`a travers l'action de $\GG(\A)$
sur $H^1(\GG(\Q),{\cal C}(\GG(\A),\O_L))$.) 
\begin{prop}\phantomsection\label{Ki105}
{\rm (i)}
$\TT$ est la sous-$\O_L$-alg\`ebre de ${\rm End}(H^1_{{\eet}}(X(0)^\times_C,\O_L))$ 
topologiquement engendr\'ee par les images
des $s_\TT(\sigma)$ et $t_\TT(\sigma)$, pour $\sigma\in G_\Q$.

{\rm (ii)} $\TT$ commute aux actions de $\GG(\A)$ et de $G_\Q$.
\end{prop}
\begin{proof}
Si $\tilde\sigma_\ell\in G_\Q$ est un frobenius arithm\'etique en $\ell$ (i.e.~un rel\`evement
de $\sigma_\ell\in G_\Q^{\rm ab}$) et si $(N,\ell)=1$, les relations d'Eichler-Shimura
se traduisent par les relations
$$s_\TT(\tilde\sigma_\ell)=S_\ell\ {\rm et}\ t_\TT(\tilde\sigma_\ell)=T_\ell
\quad{\text{dans ${\rm End}(H^1_{{\eet}}(X(N)^\times_C,\O_L))$}}$$
Par densit\'e des frobenius, les images de $s_\TT(\sigma),t_\TT(\sigma)$
dans ${\rm End}(H^1_{{\eet}}(X(N)^\times_C,\O_L))$ appartiennent \`a $\TT(N)$, pour tout $N$,
et $\TT(N)$ est la sous-$\O_L$-alg\`ebre de ${\rm End}(H^1_{{\eet}}(X(N)^\times_C,\O_L))$
engendr\'ee par les images des $s_\TT(\sigma)$ et $t_\TT(\sigma)$, pour $\sigma\in G_\Q$.
Le (i) et la commutation \`a l'action de $G_\Q$ s'en d\'eduisent par passage \`a la limite.

La commutation \`a l'action de $\GG(\A)$  se d\'eduit du (i) en remarquant que
$s_\TT(\sigma)$ commute \`a l'action de $\GG(\A)$
puisqu'il est d\'efini comme l'action d'un \'el\'ement du centre de $\GG(\A)$,
et que $t_\TT(\sigma)$ aussi puisque l'action de $G_\Q$ commute \`a celle de $\GG(\A)$.
\end{proof}

\begin{rema}\phantomsection\label{Ki106.1}
(i) D'apr\`es le cor.\,\ref{cen3}, 
le centre de $\GG(\A)$ (identifi\'e \`a $\A^\dual$) agit sur $H^1_{{\eet}}(X(0)^\times_C,\O_L)$ \`a travers 
$\A^\dual/\R_+^\dual\Q^\dual\cong\cZ^\dual$. Par construction, si $u\in\cZ^\dual$,
l'op\'erateur
$\matrice{u}{0}{0}{u}$ d\'efinit un \'el\'ement de $\TT$ si $u\in\cZ^\dual$.
Il en r\'esulte que le centre de $\GG(\A)$ agit par un caract\`ere $\delta_\TT$ \`a valeurs
dans $\TT^\dual$.

(ii) Par construction, $s_\TT(\sigma)\in\TT^\dual$ si $\sigma\in G_\Q$,
et $s_\TT:G_\Q\to \TT^\dual$ est le caract\`ere qui correspond \`a $\delta_\TT$
par la th\'eorie du corps de classes (cf.~\no\ref{prelim8}), 
i.e.~$s_\TT(\sigma)=\delta_\TT(\cy(\sigma))$ pour tout $\sigma\in G_\Q$.
Le caract\`ere $s_\TT':G_\Q\to \TT^\dual$ correspond \`a $x\mapsto (x_p|x_p|_p)\delta_\TT(x)$.

(iii) $\sigma\mapsto t_\TT(\sigma)$ est un pseudo-caract\`ere de $G_\Q$, de dimension~$2$,
 \`a valeurs dans $\TT$, de d\'eterminant $\sigma\mapsto s'_\TT(\sigma)$. Par d\'efinition de
$t_\TT(\sigma)$, on a
$$\sigma^2-t_\TT(\sigma)\sigma+s'_\TT(\sigma)=0\quad
 {\text {dans ${\rm End}(H^1_{{\eet}}(X(0)_C^\times,\O_L))$}}.$$
\end{rema}

\begin{rema}\phantomsection\label{Ki106}
Le pseudo-caract\`ere $t_\TT$ est impair ($t_\TT(\sigma)=0$ si $\sigma$ est une conjugaison complexe) et
il ressort de la conjecture de Serre prouv\'ee par Khare et Wintenberger et des th\'eor\`emes
${\rm big}\,R={\rm big}\,T$ que (au moins en localisant en un id\'eal maximal de $\TT$
g\'en\'erique) c'est le pseudo-caract\`ere universel avec les propri\'et\'es ci-dessus.

De mani\`ere \'equivalente, si $t^{\rm univ}:G_\Q\to R^{\rm univ}$ 
est le pseudo-caract\`ere de dimension~$2$, impair, universel,
il existe un unique morphisme $\alpha:R^{\rm univ}\to \TT$ 
tel que $t_\TT=\alpha\circ t^{\rm univ}$, et $\alpha$
est un isomorphisme (apr\`es localisation en un id\'eal maximal g\'en\'erique).
\end{rema}

{
\Subsection{Mod\`ele de Kirillov des fonctions}\label{Ki107}

Fixons $\iota:\Z[\bmu]\to C$.
\begin{rema}\phantomsection\label{Kiri2}
D'apr\`es la rem.\,\ref{luepan4},
on a une surjection $C\wotimes\Z[\bmu]\to {\cal C}(\cZ^\dual,C)$,
envoyant $\alpha\otimes v$ sur $\phi_{\alpha\otimes v}(x)=\alpha\iota(\sigma_x(v))$.

(i) Si on fait agir $u\in\cZ^\dual$ par $1\otimes\sigma_u$ sur
$C\otimes \Z[\bmu]$, l'action qui s'en d\'eduit 
sur ${\cal C}(\cZ^\dual,C)$ est $\big([u]\cdot\phi\big)(x)=
\phi(xu)$.

(ii)
Si on fait agir $\sigma\in G_{\Q_p}$ par $\sigma\otimes 1$
sur $C\otimes \Z[\bmu]$, l'action qui s'en d\'eduit 
sur ${\cal C}(\cZ^\dual,C)$ est $\big(\sigma\cdot\phi\big)(x)=
\sigma(\phi(\cy(\sigma)^{-1}x))$;
autrement dit c'est $\sigma\otimes [\cy(\sigma)]^{-1}$ dans la factorisation
${\cal C}(\cZ^\dual,C)=C\wotimes {\cal C}(\cZ^\dual,\Q_p)$.
\end{rema}

On peut \'etendre ${\cal K}$ (cf.~\eqref{Ki130}) par $C$-lin\'earit\'e
en une application $\PP(\Ai)$-\'equivariante 
$$C{\otimes}{\cal K}:C\wotimes\O^+(\widehat Z(0))_0\to
{\cal C}(\Aidu, C\wotimes\Z[\bmu])^{\cZ^\dual}$$
On la rend $\PP(\Ai)\times G_{\Q_p}$-\'equivariante en faisant agir
$G_{\Q_p}$ sur $C$. En utilisant $\iota:\Z[\bmu]\to C$ comme dans
la rem.\,\ref{Kiri2}, on en d\'eduit un diagramme commutatif
$$\xymatrix@C=6mm@R=4mm{
C\wotimes\O^+(\widehat Z(0))_0\ar[r]\ar[d]
&{\cal C}(\Aidu, C\wotimes\Z[\bmu])^{\cZ^\dual}\ar[d]\\
\O^b(\widehat Z(0)_C)_0\ar[r]\ar[d]^-{\wr}
&{\cal C}(\Aidu\times\cZ^\dual, C)^{\cZ^\dual}
\ar[r]^-{\sim} & {\cal C}(\Aidu,C)\\
\O^b(\widehat Z(0)^0_C\times\cZ^\dual)_0\ar[d]^-{\wr}\\
C\otimes_{\O_C}{\cal C}(\cZ^\dual,\O_C[[q^{\Q_+}]]_0)\ar[rruu]
}$$
$\PP(\Ai)\times G_{\Q_p}$-\'equivariant, l'action
de $\PP(\Ai)$ sur ${\cal C}(\Aidu,C)$ \'etant donn\'ee par:
$$\big(\matrice{a}{b}{0}{1}\cdot\phi\big)(x)=\iota({\bf e}_\A(bx))\phi(ax)$$

(On passe de la premi\`ere ligne \`a la seconde en quotientant par l'adh\'erence
des nilpotents, cf.~rem.\,\ref{luepan4}; 
l'invariance par $\cZ^\dual$ pour $\phi^{(2)}\in{\cal C}(\Aidu\times\cZ^\dual, C)$
 se traduit par $\phi^{(2)}(xu^{-1},yu)=\phi^{(2)}(x,y)$
pour tous $x\in\Aidu$ et $u,y\in\cZ^\dual$,
et l'isomorphisme ${\cal C}(\Aidu\times\cZ^\dual,C)^{\cZ^\dual}\cong
{\cal C}(\Aidu,C)$ est $\phi^{(2)}(x,y)\mapsto\phi^{(1)}(x):=\phi^{(2)}(x,1)$,
l'isomorphisme inverse \'etant
$\phi^{(2)}(x,y)=\phi^{(1)}(xy,1)$.
La fl\`eche ${\cal C}(\cZ^\dual,C[[q^{\Q_+}]]_0)\to {\cal C}(\Aidu,C)$
envoie $\sum_i\phi_i q^i$, o\`u $\phi_i\in{\cal C}(\cZ^\dual,C)$, sur $\phi:\Aidu\to C$
d\'efinie par $\phi(ui)=\phi_i(u)$ si $i\in\Q_+^\dual$ et $u\in\cZ^\dual$.)

En r\'esum\'e, on a le r\'esultat suivant:

\begin{prop}\phantomsection\label{Ki108}
L'application $C{\otimes}{\cal K}:\O^b(\widehat Z(0)_C)_0\to
{\cal C}(\Aidu,C)$ est un mod\`ele de Kirillov $G_{\Q_p}$-\'equivariant.
\end{prop} 

\Subsection{Mod\`ele de Kirillov de la cohomologie}\label{Ki109}

\subsubsection{Pour $H^1_{\proet}(\widehat Z(0)_C,\Z_p)$}\label{Ki110}
Soient 
$$\tA_{\Q_p}^+:=W(\Z_p(\bmu_{p^\infty})^\flat)
\quad{\rm et}\quad \tA_{\Q_p}:=W(\Q_p(\bmu_{p^\infty})^\flat)$$
Soit aussi $\tilde{\bf e}_\A$ le caract\`ere
$$\tilde{\bf e}_\A:
\Ai\to (\Z_p[\bmu^{]p[}]\wotimes\tA^+_{\Q_p})^\dual,\quad
\tilde{\bf e}_\A(x)={\bf e}^{]p[}(x^{]p[})\otimes [\epsilon^{x_p}]$$
Notons 
$\widetilde\O^+(\widehat Z(0))$ l'anneau $W(\O^+(\widehat Z(0))^\flat)$.
Alors $\widetilde\O^+(\widehat Z(0))$ est un anneau de s\'eries
en $\tilde q^{\Q_+}$ \`a coefficients dans  $\Z_p[\bmu^{]p[}]\wotimes\tA^+_{\Q_p}$, 
et on rajoute un $0$ en indice pour indiquer
le sous-espace des s\'eries de terme constant $0$.
L'action de $\PP(\Ai)$ sur $\O^+(\widehat Z(0))$ en induit une
sur $\widetilde\O^+(\widehat Z(0))$; de mani\`ere explicite:

\quad $\bullet$ $\matrice{n}{b}{0}{1}\cdot\big(\sum a_i \tilde q^i\big)=
\sum a_i\tilde{\bf e}_\A(bi/n)\tilde q^{i/n}$, si $n\in\Q_+^\dual$ et $b\in\Ai$, 

\quad $\bullet$ $\matrice{u}{0}{0}{1}\cdot\big(\sum a_i \tilde q^i\big)=
\sum \sigma_u(a_i) \tilde q^i$, si $u\in\cZ^\dual$.

On dispose de plus  d'une action de $\varphi$, avec:

\quad $\bullet$ $\varphi\cdot\big(\sum a_i \tilde q^i\big)=\sum \varphi(a_i) \tilde q^{pi}$.

\smallskip
Si $\widetilde\O(\widehat Z(0)_C):=W(\O(\widehat Z(0)_C)^\flat)$
et $\widetilde\O(B_{\rm Kum}^-):=W(\O(B_{\rm Kum}^-)^\flat)$,
alors $\widetilde\O(B_{\rm Kum}^-)$ est un anneau de s\'eries
en $\tilde q^{\Q_+}$ \`a coefficients dans $\tA$, muni
d'actions de $\PP(\Ai)$ et $\varphi$ donn\'ees par les formules ci-dessus
\`a part pour le fait qu'il faut remplacer 
$\tilde{\bf e}_\A:\Ai\to \Z_p[\bmu^{]p[}]\wotimes\tA^+_{\Q_p}$
par $\iota\circ\tilde{\bf e}_\A:\Ai\to\tA^+$ (le plongement
$\iota:\Z[\bmu]\to\O_C$ induit un plongement 
$\iota:\Z[\bmu^{]p[}]\to W(\overline{\bf F}_p)\subset\tA^+$). 
Alors
$$\widetilde\O(\widehat Z(0)_C)
\cong{\cal C}(\cZ^\dual,\widetilde\O(B^-_{\rm Kum}))$$
comme $(\PP(\Ai)\times G_{\Q_p})$-module et cette identification commute
\`a l'action de $\varphi$ (l'action de $G_{\Q_p}$ sur $\widetilde\O(\widehat Z(0)_C)$
est induite par celle sur $\O(\widehat Z(0)_C)$ qui, elle-m\^eme, est induite
par celle sur $C$ si on voit $\O(\widehat Z(0)_C)$ comme un quotient de $C\wotimes\O(\widehat Z(0))$).
Les m\^emes formules que ci-dessus fournissent un mod\`ele de Kirillov
$\varphi\times G_{\Q_p}$-\'equivariant
$${\cal K}:\widetilde\O(\widehat Z(0)_C)\to {\cal C}(\Aidu,\tA)$$
les actions de $\PP(\Ai)$, $\varphi$ et $G_{\Q_p}$ sur le membre de droite \'etant
donn\'ees par
$$\big(\matrice{a}{b}{0}{1}\cdot\phi\big)(x)=\iota(\tilde{\bf e}_A(bx))\,\phi(ax),
\quad (\sigma\cdot\phi)(x)=\sigma(\phi(\chi(\sigma)^{-1}x)),
\quad (\varphi\cdot\phi)(x)=\varphi(\phi(p^{-1}x))$$

On a un diagramme commutatif
$$\xymatrix@R=4mm@C=5mm{
H^1_{\proet}(\widehat Z(0)_C,\Z_p)\ar[d]^-{\wr}\ar[r]^-{\sim}
&\widetilde\O(\widehat Z(0)_C)/(\varphi-1)\ar[d]^-{\wr}\\
{\cal C}(\cZ^\dual, H^1_{\proet}(B^-_{\rm Kum},\Z_p))\ar[r]^-{\sim}
&{\cal C}(\cZ^\dual,\widetilde\O(B^-_{\rm Kum})/(\varphi-1))\ar[r]^{\iota_{\tA}}
&{\cal C}(\cZ^\dual,\tA^-[[\tilde q^{\Q_+}]]\boxtimes\Z_p^\dual)
}$$
On en d\'eduit une application
$${\cal K}_H: H^1_{\proet}(\widehat Z(0)_C,\Z_p)
\to {\cal C}(\Aidu,\tA^-)$$
en composant avec l'application
envoyant $\lambda\in {\cal C}(\cZ^\dual,\tA^-[[\tilde q^{\Q_+}]]\boxtimes\Z_p^\dual)$
sur ${\cal K}_\lambda$, o\`u ${\cal K}_\lambda(x)$ est donn\'ee,
si $x=p^knu$ (avec $u\in\cZ^\dual$,
$n\in\Q_+^\dual$, $v_p(n)=0$), par la formule
$${\cal K}_\lambda(x)=\varphi^k(\lambda_n(u)),\quad{\text{
si $\lambda(u,\tilde q)=\sum\nolimits_{i\in\Q_+^\dual\cap\Z_p^\dual}\lambda_i(u)\tilde q^i$.}}$$
(En particulier, ${\cal K}_\lambda(px)=\varphi({\cal K}_\lambda(x))$; i.e.~${\cal K}_\lambda$
est invariante par $\varphi$.)
\begin{prop}\phantomsection\label{Ki111}
L'application 
$${\cal K}_H:H^1_{\proet}(\widehat Z(0)_C,\Z_p)\to {\cal C}(\Aidu,\tA^-)^{\varphi=1}$$
ainsi d\'efinie est:

$\bullet$  $\PP(\Ai)$-\'equivariante si on munit
${\cal C}(\Aidu,\tA^-)$ de l'action donn\'ee par 
$$\big(\matrice{a}{b}{0}{1}\cdot\phi\big)(x)=\widetilde{\bf e}_\A(bx)\phi(ax)$$

$\bullet$
$G_{\Q_p}$ \'equivariante si on munit
${\cal C}(\Aidu,\tA^-)$ de 
$$(\sigma\cdot\phi)(x)=\sigma(\phi(\cy(\sigma)^{-1}x))$$
\end{prop}
\begin{proof}
Si $\lambda\in {\cal C}(\cZ^\dual,\tA^-[[\tilde q^{\Q_+}]]\boxtimes\Z_p^\dual)$,
alors
${\cal K}_\lambda(x)$ est le coefficient du terme de degr\'e~$1$
(en $\tilde q$) de $\big(\matrice{x}{0}{0}{1}\cdot\lambda\big)(1,\tilde q)$,
si on munit $\tA^-[[\tilde q^{\Q_+}]]\boxtimes\Z_p^\dual$ de l'action suivante
de $\matrice{p}{0}{0}{1}$ qui vient de l'identification 
avec $\widetilde\O(B_{\rm Kum}^-)/(\varphi-1)$
(sur lequel $\varphi$ agit trivialement)
$$\matrice{p}{0}{0}{1}\cdot \big(\sum a_i \tilde q^i\big)=
\sum a_i \tilde q^{i/p}=\varphi^{-1}\big(\sum \varphi(a_i) \tilde q^i\big)
\equiv \sum \varphi(a_i) \tilde q^i$$
Le r\'esultat s'en d\'eduit par des calculs sans myst\`ere.
\end{proof}
\begin{rema}\phantomsection\label{Ki111.1}
L'application $v\mapsto{\cal K}_{H,v}(1)$ de $H^1_{\proet}(\widehat Z(0)_C,\Z_p)$ dans $\tA^-$
est $G_{\Q_p}\times\matrice{p^\Z}{\Q_p}{0}{1}$-\'equivariante si on fait agir
$\sigma\in G_{\Q_p}$ par $\matrice{\epsilon_\A(\sigma)}{0}{0}{1}\cdot\sigma$ sur le membre de gauche
et $\matrice{p^k}{b}{0}{1}$ par $x\mapsto [\epsilon^b]\varphi^k(x)$ sur $\tA^-$.
\end{rema}

\subsubsection{Pour $H^1_{\proet}(\widehat X(0)^\times_C,\Z_p)$}\label{Ki112}
En composant ${\cal K}_H$ avec la restriction
$${\rm Res}:H^1_{\rm proet}(\widehat X(0)_C^\times,\O_L)
{\longrightarrow}H^1_{\proet}(\widehat Z(0)^\times_C,\O_L)$$
Cela fournit une fl\`eche $\PP(\Ai)\times G_{\Q_p}$-\'equivariante
$${\cal K}_H:H^1_{\rm proet}(\widehat X(0)_C^\times,\O_L)
\to {\cal C}(\Aidu,\O_L\otimes_{\Z_p}\tA^-)$$
On la promeut en une fl\`eche $\TT[\PP(\Ai)\times G_{\Q_p}]$-\'equivariante
$${\cal K}_H^\TT:H^1_{\rm proet}(\widehat X(0)_C^\times,\O_L)
\to {\cal C}(\Aidu,\check{\TT}\otimes_{\Z_p}\tA^-),\quad \big\langle {\cal K}_{H,v}^\TT(x),\lambda\big\rangle=
{\cal K}_{H,\lambda v}(x)$$
puis en une fl\`eche $\TT[\GG(\Ai)\times G_{\Q}]$-\'equivariante
\begin{align*}
{\cal K}_H^\GG:H^1_{\rm proet}(\widehat X(0)_C^\times,\O_L)\to 
{\rm Ind}_{\PP(\Ai)\times G_{\Q_p}}^{\GG(\Ai)\times G_{\Q}}
{\cal C}(\Aidu,\check{\TT}\otimes_{\Z_p}\tA^-),
\quad {\cal K}_{H, v}^\GG(g)={\cal K}^\TT_{H,g\cdot v}
\end{align*}
Une question fondamentale est la d\'etermination du noyau de ${\cal K}_H^\GG$.
Dans le chap.~\ref{Ki114} nous donnons quelques cons\'equences
de l'injectivit\'e de ${\cal K}_H^\GG$ sur des sous-espaces divers et vari\'es.
Dans les chap.~\ref{rCDN0} et~\ref{IG0} nous \'etablissons cette injectivit\'e
sur de gros sous-espaces de $H^1_{\rm proet}(\widehat X(0)_C^\times,\O_L)$.

\section{Injectivit\'e du mod\`ele de Kirillov: le cas non-eisenstein}\label{rCDN0}
Si ${\goth m}$ est un id\'eal maximal de $\TT$, on note $\TT_{\goth m}$ le localis\'e
de $\TT$ en ${\goth m}$.
Quitte \`a agrandir $L$ on peut supposer, ce que nous ferons, que $k_L=\TT/{\goth m}$.
On dit que ${\goth m}$ est {\it non-eisenstein} si 
la r\'eduction $\overline t_{\goth m}$
modulo ${\goth m}$ du pseudo-caract\`ere $t_\TT$ de la rem.\,\ref{Ki106.1}
est le caract\`ere
d'une repr\'esentation absolument irr\'eductible $\overline\rho_{\goth m}:G_\Q\to
{\rm GL}_2(k_L)$ de d\'eterminant $\overline s'_{\goth m}$ (r\'eduction modulo~${\goth m}$
du caract\`ere $s_\TT'$).
Il existe alors une repr\'esentation
$\rho_{\goth m}:G_\Q\to
{\rm GL}_2(\TT_{\goth m})$ dont la trace est le localis\'e $t_{\goth m}$ de $t_\TT$
(i.e.~le compos\'e de $t_\TT$ et de l'application naturelle $\TT\to\TT_{\goth m}$)
et le d\'eterminant est le localis\'e $s'_{\goth m}$ de $s_\TT'$.

\begin{theo}\phantomsection\label{cdn3}
Si ${\goth m}$ est non-eisenstein,
${\cal K}_H^\GG$ est une isom\'etrie sur son image dans chacun des cas suivants:

$\bullet$ $p\geq 3$ et la restriction de $\overline\rho_{\goth m}$ \`a $G_{\Q_p}$ n'est
pas de la forme $\chi\oplus\chi$.

$\bullet$ $p=2$ et la restriction de $\overline\rho_{\goth m}$ \`a $G_{\Q_p}$ n'est
ni irr\'eductible ni de la forme $\chi\oplus\chi$.
\end{theo}
\begin{proof}
La preuve de ce th\'eor\`eme va demander un peu de pr\'eparation; elle est repouss\'ee
au \S\,\ref{cdn3.1}. Indiquons juste quels ingr\'edients entrent dans la preuve.
Il s'agit de prouver que ${\cal K}_H^\GG$ est une injection modulo ${\goth m}_L$ ou,
de mani\`ere \'equivalente, que son noyau est nul.
Comme ce noyau est stable par $\GG(\Q_p)$, il contient (s'il est non nul)
des \'el\'ements fixes par le prop-$p$-iwahori, et donc
des classes d\'efinies sur
une courbe modulaire de niveau petit en $p$ (et donc avec peu de composantes irr\'eductibles
dans la fibre sp\'eciale d'un mod\`ele semi-stable bien choisi). 
Maintenant, par d\'efinition de ${\cal K}_H^\GG$, ces classes sont nulles dans le tube
des pointes.
Ceci conduit \`a \'etudier, si $Y$ est une courbe
propre, 
le sous-groupe de $H^1_{\eet}(Y,{\bf F}_p)$ des classes nulles dans le tube de suffisamment
de points.  On commence par prouver que ces classes sont tu\'ees par l'application~${\rm dlog}$
(lemme~\ref{cdn10}), puis on prouve que le noyau de ${\rm dlog}$ est petit (prop.\,\ref{cdn11}).
Dans le cas qui nous int\'eresse, cette petitesse implique que ce noyau ne contient pas
de sous-espace stable par $G_\Q$ (\`a part dans certains
cas sp\'eciaux) et comme ${\cal K}_H^\GG$ est, 
par construction, $G_\Q$-\'equivariante, cela montre que son noyau est nul.
\end{proof}

\Subsection{Cohomologie des courbes semi-stables}\label{rCDN2}
Soit $Y$ une courbe propre et lisse d\'efinie sur $C$ et soit $Y_S$ un mod\`ele semi-stable
sur $\O_{C}$ (un tel mod\`ele est \'equivalent \`a la donn\'ee d'une triangulation
$S$ de $Y$ vue comme espace de Berkovich, d'o\`u la notation).
On note $\Omega^1(Y_S)$ l'espace des sections globales du faisceau des formes diff\'erentielles
logarithmiques; c'est un r\'eseau de $\Omega^1(Y)$ qui ne d\'epend pas de $S$.
\subsubsection{L'application ${\rm dlog}$}
On dispose d'une application naturelle\footnote{Philosophiquement, c'est
${\bf F}_p(1)$ plut\^ot que ${\bf F}_p$ mais nous sommes sur un corps alg\'ebriquement clos.}
$${\rm dlog}:H^1_{\rm et}(Y,{\bf F}_p)\to \Omega^1(Y_S)/p$$ 
qui peut se d\'efinir de la mani\`ere suivante\footnote{Une autre d\'efinition possible passe
par la comparaison avec la cohomologie syntomique~\cite{CDN2}.}:
si $x\in H^1_{\rm et}(Y,{\bf F}_p)$, il existe 
$f\in C(Y)^\dual$, ${\rm Div}(f)\in p{\rm Div}(Y)$, 
tel que $x$ soit la classe de Kummer de $f$, et
on a ${\rm dlog}\,x=\frac{df}{f}$.

\begin{lemm}\phantomsection\label{cdn10}
Soit $B$ un ensemble de points lisses de la fibre sp\'eciale $Y_S^{\rm sp}$ de $Y_S$ et, si
$Q\in B$, soit $]Q[$ son tube dans $Y$ {\rm (une boule ouverte)}.
Si l'intersection de $B$ avec toute composante irr\'eductible de $Y_S^{\rm sp}$
est non vide,
on a une injection
$${\rm Ker}\big[\,H^1_{\eet}(Y,{\bf F}_p)\to\oplus_{Q\in B} H^1_{\eet}(]Q[,{\bf F}_p)\,\big]
\hookrightarrow {\rm Ker}\big[\,H^1_{\eet}(Y,{\bf F}_p)\overset{\rm dlog}{\lra}
 \Omega^1(Y_S)/p\,\big]$$
\end{lemm}
\begin{proof}
Si $x\in H^1_{\eet}(Y,{\bf F}_p)$ est la classe de Kummer de $f$ comme ci-dessus
(et donc ${\rm dlog}\,x=\frac{df}{f}$),
 et si $x$ a une image nulle dans
$]Q[$, cela signifie que $f$ est une puissance $p$-i\`eme sur $]Q[$, et donc que la restriction
de $\omega:=\frac{df}{f}$ 
\`a $]Q[$ est nulle (vue comme \'el\'ement de $\Omega^1(]Q[)^+/p\cong (\O_C/p)[[z_Q]]\,dz_Q$, et pas
comme une diff\'erentielle sur la fibre sp\'eciale).
On en d\'eduit que $\omega$, vue comme section globale du faisceau $\Omega^1_{\rm log}/p$
sur le sch\'ema formel associ\'e
\`a $Y_S$, est nulle sur l'ouvert
de lissit\'e de la composante irr\'eductible $Y_Q$ de $Y_S^{\rm sp}$ contenant $Q$.
L'hypoth\`ese selon laquelle $B$ rencontre toutes les composantes irr\'eductibles
de la fibre sp\'eciale implique que $\omega=0$ en dehors des points singuliers
de la fibre sp\'eciale, et donc que $\omega=0$ ce que l'on voulait d\'emontrer.
\end{proof}

\subsubsection{Lien avec l'application de Hodge-Tate de la jacobienne}
Si $G$ est un sch\'ema en groupes fini et plat sur $\O_C$, on dispose d'une application
de Hodge-Tate $\alpha_G:G(C)\to \omega_{G^\vee}$ o\`u $G^\vee$ est le dual de Cartier
de $G$. Cette application est d\'efinie de la mani\`ere suivante: par d\'efinition,
$u\in G(C)$ fournit un morphisme $u:G^\vee\to {\bf G}_m[p^\infty]$, et on d\'efinit
 $\alpha_G(u)=u^\dual\frac{dT}{1+T}$ (o\`u $\frac{dT}{1+T}$ est la 
base standard de l'espace des formes diff\'erentielles invariantes sur $\widehat{\bf G}_m$). 

Le choix de $P_0\in Y$ 
fournit un plongement $\iota:Y_S\to J$ (o\`u $J$ d\'esigne le mod\`ele
de N\'eron de la jacobienne) et un diagramme commutatif (ind\'ependant du choix de $P_0$)
$$\xymatrix@R=6mm@C=12mm{
H^1_{\eet}(J,{\bf F}_p)\ar[d]^-{\wr}_-{\iota^\dual}\ar[r]^-{\alpha_{J[p]}}
 & \Omega^1(J)/p\ar[d]^-{\wr}_-{\iota^\dual}\\
H^1_{\eet}(Y,{\bf F}_p)\ar[r]^-{\rm dlog} & \Omega^1(Y_S)/p
}$$
Pour v\'erifier la commutativit\'e du diagramme, on part de $u\in J[p]$.
Si $\Theta$ est le diviseur th\^eta, il existe $f_u\in C(J)^\dual$, unique \`a multiplication
pr\`es par un constante, de diviseur $p((\Theta\ominus u)-\Theta)$. Si $v\in J[p]$,
alors $f_u(v)/f_u(0)=\langle u,v\rangle_{\rm Weil}$. Il s'ensuit que $\alpha_{J[p]}(u)=\frac{df_u}{f_u}$.
Par ailleurs, $\iota^\dual f_u$ est une fonction sur $Y$ dont la classe de Kummer repr\'esente $u$
(modulo l'identification $H^1_{\eet}(Y,{\bf F}_p)\cong J[p]$), et donc
${\rm dlog}(u)=\frac{d\iota^\dual f_u}{\iota^\dual f_u}=\iota^\dual(\alpha_{J[p]}(u))$.

\subsubsection{Le noyau de l'application ${\rm dlog}$}
Si $G$ est un sch\'ema en groupes fini et plat sur $\O_C$, on d\'efinit sa {\it pente} $\mu(G)$
par $\mu(G):=\frac{{\rm deg}(G)}{{\rm ht}(G)}$, o\`u $|G|=p^{{\rm ht}(G)}$ et $\deg(G)=\sum v_p(x_i)$
si $\omega_G\cong\oplus \O_C/x_i$.
\begin{prop}\phantomsection\label{cdn11}
La pente de
${\rm Ker}\big(H^1_{\eet}(Y,{\bf F}_p)\overset{\rm dlog}{\lra}
 \Omega^1(Y_S)/p\big)$
vu comme sous-sch\'ema en groupes de $J[p]$, est~$\geq 1-\frac{1}{p}$.
\end{prop}
\begin{proof}
Soit $e_1,\dots,e_r$ une base sur ${\bf F}_p$ du noyau $G$ de ${\rm dlog}$
et soit $G_i$ le sous-sch\'ema
en groupes de $J[p]$ engendr\'e par $e_i$. Alors $\mu(G)\geq\mu(G_1\oplus\cdots\oplus G_r)$
d'apr\`es le (4) de~\cite[cor.\,5]{Far2}, et $\mu(G_1\oplus\cdots\oplus G_r)\geq\inf_i\mu(G_i)$
par d\'efinition de $\mu$ et additivit\'e de ${\rm ht}$ et $\deg$ dans une somme directe.
Il suffit donc de prouver que $\mu(G_i)\geq 1-\frac{1}{p}$.
Or $G_i$ est, d'apr\`es Oort-Tate, de la forme ${\rm Spec}(\O_K[T]/(T^p-aT))$, $v_p(a)\in[0,1]$,
et il r\'esulte de \cite[\no1.1.1]{Far1} que ${\rm dlog}(a^{1/(p-1)})=
a^{1/(p-1)}\,dT$ dans $\omega_{G_i}\cong(\O_C/pa^{-1})\,dT$. 
Il s'ensuit que ${\rm dlog}(a^{1/(p-1)})=0$
si et seulement si $\frac{v_p(a)}{p-1}\geq 1-v_p(a)$, c'est \`a dire si $v_p(a)\geq 1-\frac{1}{p}$.
Comme ${\rm ht}(G)=1$, on a $\mu(G)=\deg(G)=
v_p({\rm Ann}(dT))=v_p(a)$, ce qui permet de conclure.
\end{proof}

\begin{ques}\phantomsection\label{cdn12}
Peut-on remplacer $\geq 1-\frac{1}{p}$ par $>1-\frac{1}{p}$ dans
l'\'enonc\'e de la prop.\,\ref{cdn11}?
Cela permettrait de supprimer l'hypoth\`ese ``non irr\'eductible''
pour $p=2$ dans le th.\,\ref{cdn3}.
\end{ques}

\Subsection{Preuve du th.\,\ref{cdn3}}\label{cdn3.1}
Il s'agit de prouver que ${\cal K}_H^\GG $ est une injection modulo~$p$.
On est donc ramen\'e \`a consid\'erer l'application induite
$${\cal K}_H^\GG :H^1_{\proet}(\widehat X(0)_C^\times,k_L)_{\goth m}\to
{\rm Ind}_{\O_L[\PP(\Ai)\times G_{\Q_p}]}^{\TT[\GG(\Ai)\times G_{\Q}]}
{\cal C}(\Aidu,k_L\otimes_{{\bf F}_p}\tE^-)$$
  Le noyau de ${\cal K}_H^\GG $
est un $\TT[\GG(\Ai)\times G_{\Q}]$-module.
Soit $M$ une composante irr\'eductible du $\GG(\Q_p)$-socle de ce noyau.
Alors $M$ est tu\'ee par ${\goth m}$ car $H^1_{\proet}(\widehat X(0)_C^\times,k_L)_{\goth m}$
est de ${\goth m}^\infty$-torsion.
En tant que repr\'esentation de $G_\Q$, le module $M$ est $\overline\rho_{\goth m}$-isotypique
d'apr\`es le lemme~\ref{cdn40} ci-dessous.

Comme $M$ est la limite inductive de ses intersections avec
les $H^1_{\eet}(X(N)^\times_C,k_L)$, il existe $N$ premier \`a $p$ et $k\in\N$
tels que $M':= M\cap H^1_{\eet}(X(Np^k)^\times_C,k_L)$ soit non nul.

Maintenant, $\matrice{1+p\Z_p}{\Z_p}{p\Z_p}{1+p\Z_p}$ est un pro-$p$-groupe qui agit
sur $M'$, et donc $M'$ contient un \'el\'ement non nul fixe par ce sous-groupe.
Il s'ensuit que $M'':=M[{\goth m}]\cap H^1_{\eet}(X(N,p)^\times_C,k_L)\neq 0$.
De plus, $M''$ est stable par $\matrice{0}{1}{p}{0}_p$.

Par ailleurs, $X(N,p)$ a un mod\`ele semi-stable sur $\Z_p[\bmu_p]$, dont la fibre sp\'eciale
a deux composantes connexes et l'une de ces composantes connexes contient la pointe~$\infty$,
l'autre la pointe $\matrice{0}{1}{p}{0}_p\cdot\infty$.
Comme $M''$ est contenu dans le noyau de ${\cal K}_H$, cela implique 
que les \'el\'ements de $M''$ sont triviaux sur le tube de la pointe $\infty$,
et comme $M''$ est
stable par $\matrice{0}{1}{p}{0}_p$,
ils sont aussi triviaux sur le tube de la pointe $\matrice{0}{1}{p}{0}_p\cdot\infty$.  

Soit $J$ la jacobienne de $X(N,p)$ et soit 
$J^{\rm div}$ le groupe $p$-divisible sur $\Q$ associ\'e.
L'alg\`ebre de Hecke $\TT$ agit sur $\O_L\otimes_{\Z_p}J^{\rm div}$.
On note $J^{\rm div}_{\goth m}$ le localis\'e en ${\goth m}$ de $\O_L\otimes_{\Z_p}J^{\rm div}$
(c'est un facteur direct de $\O_L\otimes_{\Z_p}J^{\rm div}$).
On note $J^{\rm mult}_{\goth m}\subset J^{\rm cnx}_{\goth m}$ les plus grands sous-groupes
$p$-divisibles de $J^{\rm div}_{\goth m}$ sur $\Z_p[\bmu_p]$ (le prolongement \`a $\Z_p[\bmu_p]$
se fait en consid\'erant le mod\`ele de N\'eron de $J$ sur $\Z_p[\bmu_p]$), 
de type multiplicatif et connexe.
Alors\footnote{Ou plut\^ot $H^1_{\eet}(X(N,p)^\times_C,k_L(1))$.}
 $H^1_{\eet}(X(N,p)^\times_C,k_L)$ s'identifie \`a $J^{\rm div}[p]$
et cette identification induit des inclusions
$M''\hookrightarrow J^{\rm div}_{\goth m}[{\goth m}]\hookrightarrow J^{\rm div}_{\goth m}[p]$.
Comme les
\'el\'ements de $M''$ sont triviaux dans le voisinage des pointes et
que chaque composante de la fibre sp\'eciale de $X(N,p)$ contient une pointe,
il r\'esulte du lemme~\ref{cdn10} et de la prop.\,\ref{cdn11}
que $M''\subset J^{\rm cnx}_{\goth m}[p]^{1-\frac{1}{p}}$.

$\bullet$ Si $\overline\rho_{\goth m}$ est ordinaire, 
on a $J^{\rm cnx}_{\goth m}[{\goth m}]=J^{\rm mult}_{\goth m}[{\goth m}]$,
et donc $M''\subset J^{\rm mult}_{\goth m}[{\goth m}]$.
Par ailleurs, il r\'esulte de~\cite[prop.\,12.8 et~12.9]{gross} que $G_{\Q_p}$ agit par un caract\`ere
sur $J^{\rm mult}_{\goth m}[{\goth m}]$.
On en d\'eduit que $J^{\rm mult}_{\goth m}[{\goth m}]$ ne contient pas de copie de la
restriction de $\overline\rho_{\goth m}$ \`a $G_{\Q_p}$ (et donc que $M''=0$) si
cette restriction est une extension non triviale de deux caract\`eres (non n\'ec\'essairement distincts)
ou la somme directe de deux caract\`eres distincts.  
Par contre, si cette restriction est de la forme $\chi\oplus\chi$,
on ne peut pas conclure de cette mani\`ere que $M''=0$.

$\bullet$ Si la restriction de $\overline\rho_{\goth m}$ \`a $G_{\Q_p}$
est irr\'eductible, le sch\'ema en groupes $M''$ sur $\Z_p[\bmu_p]$ est compl\`etement d\'etermin\'e
par sa restriction \`a $G_{\Q_p(\bmu_p)}$ puisqu'il ne contient ni sous-groupe multiplicatif 
ni quotient \'etale, \'etant une somme de copies de $\overline\rho_{\goth m}$.
Mais la restriction de $\overline\rho_{\goth m}$ \`a une extension non ramifi\'ee
convenable de ${\Q_p(\bmu_p)}$ est auto-duale. On en d\'eduit que le sch\'ema en groupe associ\'e
est de pente~$\frac{1}{2}$ et donc que $M''$ est de pente~$\frac{1}{2}$.
Comme on sait par ailleurs que $M''$ est de pente~$\geq 1-\frac{1}{p}$, cela prouve que $M''=0$
si $p\geq 3$.

Ceci permet de conclure.

\begin{rema}
Le cas o\`u la restriction
de $\overline\rho_{\goth m}$ \`a $G_{\Q_p}$ est de la forme 
$\chi\oplus\chi$ est le plus d\'elicat \`a bien des \'egards:
en particulier, il r\'esulte de~\cite[cor.\,4.4]{wiese} et de~\cite[prop.\,12.8 et~12.9]{gross}
que $J_{\goth m}^{\rm mult}[\goth m]$ contient des copies de $\overline\rho_{\goth m}$.
\end{rema}

\begin{lemm}\phantomsection\label{cdn40}
Si ${\goth m}$ est non-eisenstein,
l'application naturelle 
$${\rm Hom}_{G_\Q}(\overline\rho_{\goth m}, H^1[{\goth m}])\otimes\overline\rho_{\goth m}\to
H^1[{\goth m}]$$
est un isomorphisme.
\end{lemm}
\begin{proof}
Comme $\overline\rho_{\goth m}$ est irr\'eductible, cette application est injective.
Par ailleurs, si $\sigma\in G_{\Q}$, alors 
$P_{\goth m}(\sigma):=\sigma^2-\overline t_{\goth m}(\sigma)+\overline s'_{\goth m}$ 
tue $H^1[{\goth m}]$
et donc aussi tous ses sous-quotients.
L'irr\'eductibilit\'e
de $\overline\rho_{\goth m}$ impliquent donc que tous les sous-quotients
irr\'eductibles de $\widehat H^1[{\goth m}]$ sont isomorphes \`a $\overline\rho_{\goth m}$, et
il suffit de prouver que $\widehat H^1[{\goth m}]$ n'a aucun sous-quotient qui est
une extension non triviale de $\overline\rho_{\goth m}$ par $\overline\rho_{\goth m}$.  

Une telle extension~$V$
peut \^etre aussi vue comme une repr\'esentation de dimension~$2$ 
sur $k_L[\epsilon]$ (avec $\epsilon^2=0$
et $V/\epsilon V=\overline\rho_{\goth m}$),
et il r\'esulte de ce qui pr\'ec\`ede que le polyn\^ome minimal de tout \'el\'ement
de $G_\Q$ est \`a coefficients dans $k_L$; on veut en d\'eduire
que $V\cong k_L[\epsilon]\otimes_{k_L}\overline\rho_{\goth m}$.
Comme $\overline\rho_{\goth m}$ est irr\'eductible, son image $G$ n'est pas incluse dans un borel
de ${\rm GL}_2(k_L)$ et son image dans ${\rm PGL}_2(k_L)$ n'est pas contenue dans un $p$-sylow 
(car les $p$-sylows de ${\rm PGL}_2(k_L)$ sont les conjugu\'es du sous-groupe des unipotents
sup\'erieurs et leurs images inverses sont contenues dans des borels); on en d\'eduit l'existence
de $g\in G$, semi-simple \`a valeurs propres $\alpha\neq\beta$.

Soit $v_\alpha, v_\beta$ une base de $V/\epsilon V$ constitu\'ee de vecteurs propres pour $g$.
Il r\'esulte de la discussion ci-dessus que 
le polyn\^ome minimal de $g$ agissant sur $V$ est
$(X-\alpha)(X-\beta)$ et on peut relever $v_\alpha$, $v_\beta$ dans $V$, de mani\`ere unique,
en des vecteurs propres pour $g$.  
Montrons que $k_Lv_\alpha\oplus k_Lv_\beta$ est stable
par $G_\Q$, ce qui permettra de conclure puisque $V=k_L[\epsilon]\otimes_{k_L}(k_Lv_\alpha\oplus k_Lv_\beta)$.

Il suffit de prouver que, si $\lambda\in k_L[G_\Q]$ v\'erifie $\lambda\cdot v_\alpha,\lambda\cdot v_\beta\in\epsilon V$,
alors $\lambda\cdot v_\alpha=\lambda\cdot v_\beta=0$. Si $\lambda\cdot v_\alpha=\epsilon(a(\lambda) v_\alpha+c(\lambda)v_\beta)$
et $\lambda\cdot v_\beta=\epsilon(b(\lambda) v_\alpha+d(\lambda)v_\beta)$, alors $a(\lambda)+d(\lambda)=0$
par hypoth\`ese sur la trace. Mais $g\lambda$ v\'erifie les m\^emes hypoth\`eses
et on a $a(g\lambda)=\alpha a(\lambda)$ et $d(g\lambda)=\beta d(\lambda)$. Il s'ensuit
que $a(\lambda)=d(\lambda)=0$.  
On peut aussi appliquer ceci \`a $h\lambda$, o\`u $h\in G$ est tel que $h(v_\alpha)$
et $v(v_\beta)$ ne sont pas colin\'eaires \`a $v_\alpha$ et $v_\beta$ 
(un tel $h$ existe par irr\'eductibilit\'e
de $\overline\rho_{\goth m}$); on en d\'eduit que $b(\lambda)=c(\lambda)=0$, ce qui permet de conclure.
\end{proof}

\section{Injectivit\'e du mod\`ele de Kirillov: le cas g\'en\'eral}\label{IG0}
{

Il r\'esulte du th.\,\ref{cdn3} que ${\cal K}_H^\GG$ est injective 
si on localise en un id\'eal
non-eisenstein (\`a quelques exceptions pr\`es). Dans ce chapitre,
on donne une seconde
preuve de cette injectivit\'e qui permet de contr\^oler le noyau \'eventuel dans le cas eisenstein
et les exceptions du cas non-eisenstein
(le prix \`a payer est que l'on perd le r\'esultat d'int\'egralit\'e du th.\,\ref{cdn3}). 
Cette preuve passe par une analyse des vecteurs de la cohomologie compl\'et\'ee
fixes par $\UU(\Z_p)$ (les m\'ethodes utilis\'ees sont inspir\'ees de Pan~\cite{pan2} -- qui
a fait une analyse d\'etaill\'ee des vecteurs localement analytiques fixes par l'alg\`ebre
de Lie de $\UU(\Z_p)$ -- et de Sean Howe~\cite{howe}).

\Subsection{La tour des courbes d'Igusa}\label{IG2}
On pose
$$\widehat X(]p[,p):=\widehat X(0)/\UU(\Z_p)$$
Si on d\'ecompose $e_i^\dual$ sous la forme $((e_i^\dual)_p,(e_i^\dual)^{]p[})$
en regardant les composantes en $p$ et en-dehors de $p$,
quotienter
$\widehat X(0)$ par $\UU(\Z_p)$ revient \`a remplacer l'information contenue dans
$((e_1^\dual)_p,(e_2^\dual)_p)$ par 
$((e_1^\dual)_p,(e_1^\dual)_p\wedge (e_2^\dual)_p)$; il s'ensuit que
$$\widehat X(]p[,p)=\widehat X(]p[,p)^0\times\Z_p(1)^\times$$
o\`u $\widehat X(]p[,p)^0$ classifie
les triplets $(E,e_1^\dual,(e_2^\dual)^{]p[})$, o\`u $e_1^\dual$ est un \'el\'ement
primitif de $T_{\cZ}E$ et $(e_1^\dual)^{]p[},(e_2^\dual)^{]p[}$
forment une base $T^{]p[}E$, et
o\`u $\Z_p(1)^\times$ est l'ensemble des bases de $\Z_p(1)$ sur $\Z_p$ 
(le $\Z_p(1)^\times$ encode l'accouplement de Weil
$(e_1^\dual)_p\wedge (e_2^\dual)_p$ des composantes en $p$;
la torsion galoisienne se traduit par le fait que $\O(\Z_p(1)^\times)$ est le compl\'et\'e
$p$-adique de $\Z[\bmu_{p^\infty}]$ (et pas ${\cal C}(\Z_p(1)^\times,\Z_p)$)).

Soit ${\rm Ig}(0)$ la composante connexe du lieu ordinaire de $X(0)$
contenant la pointe~$\infty$: ${\rm Ig}(0)$ param\^etre les triplets $(E,e_1^\dual,e_2^\dual)$
o\`u $E$ est une courbe elliptique ordinaire, $(e_1^\dual)_p$ est un g\'en\'erateur du module
de Tate
$T_p\widehat E$ du groupe formel de $E$ (isomorphe \`a $\widehat{\bf G}_m$ sur $\O_C$),
et $e_1^\dual,e_2^\dual$ est une base du module de Tate ad\'elique
sur $\cZ$.  
En particulier, ${\rm Ig}(0)$ contient le voisinage habituel $Z(0)$ de la pointe $\infty$.

Soit ${\rm Ig}(]p[)$ la composante connexe correspondante du lieu ordinaire de $X(]p[,p)$.
On a
$${\rm Ig}(]p[):={\rm Ig}(0)/\UU(\Z_p),\quad {\rm Ig}(]p[)=
{\rm Ig}(]p[)^0\times\Z_p(1)^\times$$
 o\`u ${\rm Ig}(]p[)^0\subset X(]p[,p)^0$
est le lieu des $(E,e_1^\dual,(e_2^\dual)^{]p[})$
o\`u $E$ est une courbe elliptique ordinaire, 
$(e_1^\dual)_p$ est un g\'en\'erateur de $T_p\widehat E$.
\begin{rema}\phantomsection\label{igu1}
Si $\kappa$ est un caract\`ere de $\BB(\Z_p)$, alors
$\O({\rm Ig}(]p[))_\kappa$ est l'espace des formes
modulaires $p$-adiques de poids $\kappa$ (et niveau arbitraire en dehors de $p$).
L'espace $\O({\rm Ig}(]p[))$ est l'espace de toutes les formes
modulaires $p$-adiques.
\end{rema}
L'application de Hodge-Tate $\pi_{\rm HT}:X(0)_C\to \piqp$, commute aux actions 
de $\GG(\Ai)$ et $G_{\Q_p}$ (o\`u $\GG(\Ai)$ agit \`a travers $\GG(\Q_p)$ sur $\piqp$).
Cette application envoie le lieu ordinaire sur $\piqp(\Q_p)$ et le lieu supersingulier
sur le demi-plan de Drinfeld $\piqp\moins\piqp(\Q_p)$, et on a
$${\rm Ig}(0)_C:=\pi_{\rm HT}^{-1}(\{\infty\})$$

\Subsection{D\'ecomposition de Hodge-Tate pour les $\UU(\Z_p)$-invariants}\label{IG3}
\begin{theo}\phantomsection\label{igu2}
{\rm (Comparer avec~\cite{faltings-FM,AIS} et~\cite[th.\,1.0.4]{JE2})}

On a une suite exacte naturelle,
$\GG(\A^{]\infty,p[})\times\BB(\Z_p)\times G_{\Q_p}$-\'equivariante
$$0\to H^1(X(]p[,p)_C^\times,\O)
\to H^1_{{\proet},c}(X(0)_C^\times,C)^{\UU(\Z_p)}\to
\O({\rm Ig}(]p[)^\times_C)(-1)\to 0$$
\end{theo}
\begin{proof}
D'apr\`es le th\'eor\`eme
de comparaison primitif de Scholze~\cite[th.\,IV.2.1]{Sz2},
on a un isomorphisme naturel (avec $\widehat\O^b=\widehat\O^+[\frac{1}{p}]$)
$$H^1_{{\eet},c}(X(0)_C^\times,C)\overset{\sim}{\to}
H^1_{{\proet},c}(X(0)_C^\times,\widehat\O^b)$$
Par ailleurs, comme $H^0_{{\proet},c}(X(0)_C^\times,\widehat\O)=0$, la restriction
fournit un isomorphisme naturel
$$H^1_{{\proet},c}(X(0)_C^\times/\UU(\Z_p),\widehat\O^b)\overset{\sim}{\to}
H^1_{{\proet},c}(X(0)_C^\times,\widehat\O^b)^{\UU(\Z_p)}$$
Soit 
$$X'(0)_C:=X(0)_C\moins {\rm Ig}(0)_C$$
Alors $X(0)_C$ s'obtient en recollant ${\rm Ig}(0)_C$ et $X'(0)_C$ le long de $\partial X'(0)_C$
qui est une limite profinie de cercles fant\^omes perfecto\"{\i}des (ou, si l'on pr\'ef\`ere, de points de type $5$).
Les espaces ${\rm Ig}(0)_C$, $\partial X'(0)_C$ et $X'(0)_C$ sont affines perfecto\"{\i}des
ainsi~\cite[th.\,4.5.2]{CGJ} que
$X'(0)_C/\UU(\Z_p)$ et $\partial X'(0)_C/\UU(\Z_p)$ 
(mais pas ${\rm Ig}(0)_C/\UU(\Z_p)$ ni son bord $\partial{\rm Ig}(0)_C/\UU(\Z_p)$ qui sont simplement affines).

Notons simplement $X,X',{\rm Ig},\partial$ les espaces $X(0)^\times_C, X'(0)^\times_C$,
${\rm Ig}(0)^\times_C$, $\partial X'(0)_C$ et $U$ le groupe $\UU(\Z_p)$
(et $\gamma$ son g\'en\'erateur topologique $\matrice{1}{1}{0}{1}$).
On a une surjection $\partial /U\to \partial{\rm Ig}/U$ et
$X/U$ s'obtient en recollant $X'/U$ et ${\rm Ig}/U$ le
long de $\partial /U$: i.e.~si $V$ est un ouvert de $X/U$, alors
$$\O(V)=\{\phi_1\in\O(V\cap(X'/U)),\phi_2\in\O(V\cap({\rm Ig}/U)),\  \phi_1=\phi_2\ {\rm dans}\ 
\O(V\cap(\partial/U))\}$$
Comme ${\rm Ig}$ est quasi-compacte, on a $\widehat\O^b({\rm Ig})=\widehat\O({\rm Ig})$
et $\widehat\O^b({\partial})=\widehat\O({\partial})$.

On a $\rg_{\proet}({\rm Ig}/U,\widehat\O^b)\simeq(\O({\rm Ig})\overset{\gamma-1}{\lra}\O({\rm Ig}))$
car ${\rm Ig}$ est affine perfecto\"{\i}de et donc 
$\rg_{\proet}({\rm Ig},\widehat\O^b)\simeq(\O({\rm Ig}))$.
Pour les m\^emes raisons, on a
$\rg_{\proet}(X'/U,\widehat\O^b)\simeq(\O^b(X')^U)$
et $\rg_{\proet}(\partial/U,\widehat\O^b)\simeq(\O(\partial)\overset{\gamma-1}{\lra}\O(\partial))$
(on a aussi $\rg_{\proet}(\partial/U,\widehat\O^b)\simeq(\O(\partial)^U)$, 
i.e.~$\O(\partial)\overset{\gamma-1}{\lra}\O(\partial)$ est surjective,  mais nous pr\'ef\'ererons
le premier quasi-isomorphisme).
On en d\'eduit que $\rg_{\proet}(X/U,\widehat\O^b)$ est le complexe simple associ\'e au complexe
double 
$$\xymatrix@R=6mm@C=12mm{\O^b(X')^U\oplus \O({\rm Ig})\ar[d]\ar[r]^-{(0,\gamma-1)}&\O({\rm Ig})\ar[d]\\
\O(\partial)\ar[r]^-{\gamma-1}&\O(\partial)}$$
Notons $Z^1$ le noyau de $\O({\rm Ig})\oplus\O(\partial)\to\O(\partial)$.
Comme $\O(\partial)\overset{\gamma-1}{\lra}\O(\partial)$ est surjective,
on a une suite exacte $0\to \O(\partial)^U\to Z^1\to\O({\rm Ig})\to 0$.
Cette suite exacte
s'inscrit dans un diagramme commutatif
$$\xymatrix@R=5mm@C=6mm{
0\ar[r]& \O^b(X')^U\ar[r]\ar[d]& \O^b(X')^U\oplus\O({\rm Ig})\ar[r]\ar[d]&\O({\rm Ig})\ar[r]\ar[d]^-{\gamma-1}& 0\\
0\ar[r]& \O(\partial)^U\ar[r]& Z^1\ar[r]&\O({\rm Ig})\ar[r]& 0 
}$$
Comme le noyau de $\O^b(X')^U\oplus\O({\rm Ig})\to Z^1$ est nul (car c'est
$H^0_{{\proet},c}(X/U,\widehat\O)$), le lemme du serpent fournit une suite
exacte
$$0\to \O({\rm Ig})^U\to \O(\partial)^U/\O(X')^U\to
H_{\proet}^1(X/U,\widehat\O)\to H^1(U,\O({\rm Ig}))\to 0$$
Comme $X'/U$ et ${\rm Ig}/U$ sont affines, on a une suite exacte
$$0\to \O^b(X')^U\oplus\O({\rm Ig})^U\to \O(\partial)^U\to H^1(X/U,\O)\to 0$$
Le conoyau de $\O({\rm Ig})^U\to \O(\partial)^U/\O^b(X')^U$ s'identifie donc
\`a $H^1(X(]p[,p)_C^\times,\O)$.
On conclut en utilisant l'isomorphisme
$${\cup}\,\nu:\O({\rm Ig}/U)
\cong H^1(U,\O({\rm Ig}))$$
qui permet d'identifier $H^1(U,\O({\rm Ig}))$ \`a $\O({\rm Ig}(]p[)_C^\times)$; le twist $(-1)$
provenant de la rem.\,\ref{luepan1.1}.
\end{proof}

\begin{rema}\phantomsection\label{igu3}
La fl\`eche naturelle $H^1(U,\O({\rm Ig}))\to H_{\proet}^1({\rm Ig}/U,\widehat\O)$ est un isomorphisme
et la compos\'ee avec $H_{\proet}^1(X/U,\widehat\O)\to H^1(U,\O({\rm Ig}))$ est la restriction
(elle est induite par la fl\`eche naturelle $\rg_{\proet}(X/U,\widehat\O)\to
\rg_{\proet}({\rm Ig}/U,\widehat\O)$).
\end{rema}

\Subsection{Le mod\`ele de Kirillov des $\UU(\Z_p)$-invariants}\label{IG5}
La restriction de ${\cal K}_H:H^1_{\proet}(\widehat Z(0)^\times_C,\Z_p)\to {\cal C}(\Aidu,\tA^-)$
aux $\UU(\Z_p)$-invariants
est 
tu\'ee par $[\epsilon]^{x_p}-1$, et donc 
\`a valeurs dans $([\epsilon]^{p^n}-1)^{-1}\tA^+/\tA^+$ sur $p^n\Z_p\times\A^{]p,\infty[,\dual}$.
On peut donc la composer avec la fl\`eche naturelle de
$([\epsilon]^{p^n}-1)^{-1}\tA^+/\tA^+$ dans $t^{-1}\bdr^+/\bdr^+=Ct^{-1}$
et l'\'etendre, par $C$-lin\'earit\'e et continuit\'e en une fl\`eche
$$C{\otimes}{\cal K}_H:H^1_{\proet}(\widehat Z(0)^\times_C,C)^{\UU(\Z_p)}\to {\cal C}(\Aidu,t^{-1}C)$$
Rappelons par ailleurs que l'on dispose aussi d'un mod\`ele de Kirillov $G_{\Q_p}$-\'equivariant
${\cal K}:\O(Z(0)^\times_C)\to{\cal C}(\Aidu,C)$.
\begin{prop}\phantomsection\label{igu5}
On a le diagramme commutatif
$$\xymatrix@R=5mm{
H^1_{\proet}(\widehat X(0)^\times_C,C)^{\UU(\Z_p)}\ar[r]^-{C{\otimes}{\cal K}_H}\ar[d] 
& {\cal C}(\Aidu,t^{-1}C)\ar[dd]^-{\times|\ |_\A}\\
\O({\rm Ig}(]p[)_C^\times)(-1)\ar[d]\\
\O(Z(]p[)^\times_C)(-1)\ar[r]^-{{\cal K}(-1)}&{\cal C}(\Aidu,C(-1))
}$$
{\rm (La premi\`ere fl\`eche verticale \`a gauche
est celle du th.\,\ref{igu2}, la seconde est la restriction,
${\cal K}(-1)$ est obtenu en tordant ${\cal K}$ par $\cyp ^{-1}$ et
on a identifi\'e $C(-1)$ \`a $t^{-1}C$ dans
la fl\`eche verticale \`a droite.)}
\end{prop}
\begin{proof}
Notons simplement $X,{\rm Ig},Z$ et $U$ les espaces $X(0)^\times_C$,
${\rm Ig}(0)^\times_C$, $Z(0)^\times_C$, et le groupe $\UU(\Z_p)$, et
$H^1(-,\Z_p)$, $H^1(-,\widehat\O)$ les groupes $H^1_{\proet}(-,\Z_p)$, $H^1_{\proet}(-,\widehat\O)$.

On a un isomorphisme $H^1(X,\Z_p)^U\cong H^1(X/U,\Z_p)$ et un diagramme
commutatif
$$\xymatrix@R=6mm@C=8mm{
H^1(X/U,\Z_p)\ar[r]\ar[d]&H^1({\rm Ig}/U,\Z_p)\ar[r]\ar[d]&H^1(Z/U,\Z_p)\ar[d]\\
H^1(X/U,\widehat\O)\ar[r]&H^1({\rm Ig}/U,\widehat\O^b)\ar[r]&H^1(Z/U,\widehat\O^b)\\
&H^1(U,\O({\rm Ig}))\ar[u]^-{\wr}\ar[r]&H^1(U,\O^b(Z))\ar[u]^-{\wr}\ar[d]^-{\exp^\dual_B}_-{\wr}\ar[dr]^-{\log_B^0}\\
&\O({\rm Ig})^U(-1)\ar[u]_-{\cup\,\nu}^-{\wr}\ar[r]&\O^b(Z)^U(-1)\ar[d]_-{{\cal K}(-1)}\ar[r]^-{\partial^{-1}}
&t^{-1}\O(Z)^U\ar[d]^-{t^{-1}{\cal K}}\\
&&{\cal C}(\Aidu,C(-1))\ar[r]^-{\times|\ |_\A^{-1}}&{\cal C}(\Aidu,t^{-1}C)
}$$
On passe de la premi\`ere \`a la seconde ligne par l'injection 
naturelle $\Z_p\hookrightarrow\widehat\O$ et de la premi\`ere \`a la seconde
colonne, puis \`a la troisi\`eme, par restriction.  Le carr\'e du milieu commute
d'apr\`es la rem.\,\ref{igu3}; le triangle de droite (troisi\`eme et quatri\`eme lignes)
commute gr\^ace \`a la rem.\,~\ref{luepan3}, et le carr\'e du bas gr\^ace \`a la formule
${\cal K}_{\partial\lambda}=|\ |_\A^{-1}{\cal K}_\lambda$ du (iv) de la rem.\,\ref{Ki102}.

Le mod\`ele de Kirillov pour $H^1(X,\Z_p)^U$ est obtenu en passant par les fl\`eches
ext\'erieures du haut et de droite.
A part la premi\`ere ligne, tous les modules intervenant dans le diagramme
sont des $C$-modules et les fl\`eches sont $C$-lin\'eaires; on peut donc rendre tout le diagramme
$C$-lin\'eaire en rempla\c{c}ant les groupes $H^1(-,\Z_p)$ de la premi\`ere ligne
par les $H^1(-,C)$ correspondant, 
et la fl\`eche $C{\otimes}{\cal K}_H:H^1(X,C)^U\to {\cal C}(\Aidu,t^{-1}C)$
de la proposition est encore obtenue en suivant l'ext\'erieur du diagramme.
On conclut en utilisant les identifications $\O({\rm Ig})^U=\O({\rm Ig}(]p[)_C^\times)$ et
$\O(Z)^U=\O(Z(]p[)_C^\times)$.
\end{proof}

\begin{coro}\phantomsection\label{igu6}
Le noyau de ${\cal K}_H:H^1(X(0)_C^\times,\Z_p)^{\UU(\Z_p)}\to {\cal C}(\Aidu,t^{-1}C)$
est l'intersection avec $H^1(X(]p[,p)_C^\times,\O)$. Le r\'esultat reste vrai
pour le noyau de $C{\otimes}{\cal K}_H$.
\end{coro}
\begin{proof}
Les fl\`eches $\O({\rm Ig}(]p[)_C^\times)\to\O(Z(]p[)_C^\times)$ et
${\cal K}:\O(Z(]p[)_C^\times)\to{\cal C}(\Aidu,C)$ sont injectives (par injectivit\'e
du $q$-d\'eveloppement d'une forme modulaire $p$-adique)
ainsi que la multiplication par $|\ |_\A$. On en d\'eduit que le noyau de
${\cal K}_H$ est le m\^eme que celui de la fl\`eche vers $\O({\rm Ig}(]p[)_C^\times)(-1)$.
Le th.\,\ref{igu2} permet de conclure.
\end{proof}

\Subsection{L'op\'erateur de Sen sur les espaces propres pour $\BB(\Z_p)$}\label{IG6}
Soit $\kappa=\kappa_1\otimes\kappa_2$ un caract\`ere de $\BB(\Z_p)$.
Si $M$ est un $\BB(\Z_p)$-module, on note $M_\kappa$ le sous-espace sur lequel
$\BB(\Z_p)$ agit par $\kappa$.
On d\'eduit du th.\,\ref{igu2} une suite exacte $G_{\Q_p}\times\BB(\Z_p)$-\'equivariante
$$0\to H^1(X(]p[,p)_C^\times,\O)_\kappa\to (C\wotimes H^1_{{\proet},c}(X(0)_C,\Z_p))_\kappa
\to \O({\rm Ig}(]p[)^\times_C)_\kappa(-1)$$
($\O({\rm Ig}(]p[)^\times_C)_\kappa$ est l'espace des formes
modulaires $p$-adiques de poids~$\kappa$.)

\begin{prop}\phantomsection\label{igu7}
L'op\'erateur de Sen est la multiplication par{\rm:}

$\bullet$ $-w(\kappa_2)$ sur $H^1(X(]p[,p)_C^\times,\O)_\kappa$,

$\bullet$ $-1-w(\kappa_1)$ sur $\O({\rm Ig}(]p[)^\times_C)_\kappa(-1)$.
\end{prop}
\begin{proof}
Notons que $\matrice{1}{0}{0}{d}_p\in\matrice{1}{0}{0}{\Z_p^\dual}$ 
agit trivialement sur $X(]p[,p)^0$ 
et agit par $u\mapsto ud$ sur $\Z_p(1)^\times$. 
On voit $\kappa_i$, si $i=1,2$, aussi comme un caract\`ere de $\cZ^\dual$ et de $G_{\Q_p}$
via les projections $\cZ^\dual\to \Z_p^\dual$ et $\cyp :G_{\Q_p}\to\Z_p^\dual$.

On a $H^1(X(]p[,p)_C^\times,\O)= 
H^1(X(]p[,p)^{0,\times}_C,\O)\wotimes{\cal C}(\cZ^\dual,C)$.
Le th.\,\ref{igu2} fournit donc une fl\`eche \'equivariante
pour les actions de $G_{\Q_p}$ et $\BB(\Z_p)$:
$$H^1(X(]p[,p)^{0,\times}_C,\O)\wotimes{\cal C}(\cZ^\dual,C)
\to C\wotimes H^1_{{\proet},c}(X(0)^\times_C,\Z_p)^{\UU(\Z_p)}$$
(Comme $\matrice{1}{0}{0}{d}_p\in\matrice{1}{0}{0}{\Z_p^\dual}$ 
agit trivialement sur $X(]p[,p)^0$
et agit par $u\mapsto ud$ sur $\Z_p(1)^\times$, l'action de
$\matrice{1}{0}{0}{d}_p$ est $\big(\matrice{1}{0}{0}{d}_p\cdot\phi\big)(u)=
\phi(ud)$ sur ${\cal C}(\cZ^\dual,C)$ et est triviale sur $H^1(X(]p[,p)^{0,\times}_C,\O)$.)

On en d\'eduit que, si $v\in (H^1(X(]p[,p)^{0,\times},\O)\otimes{\cal C}(\cZ^\dual,C))_\kappa$,
alors $v$ est de la forme $v_0\otimes\kappa_2$ avec $v_0\in
H^1(X(]p[,p)^{0,\times},\O)\otimes C$, et $\kappa_2$ vu comme un caract\`ere de $\cZ^\dual$.
Alors $\sigma(v)=\kappa_2^{-1}(\sigma)\,\sigma(v_0)\otimes \kappa_2$
(cf.~rem.\,\ref{Kiri2} pour l'action de $G_{\Q_p}$ sur ${\cal C}(\cZ^\dual,C)$). 
Comme l'op\'erateur de Sen est trivial
sur $H^1(X(]p[,p)^{0,\times},\O)\otimes C$, c'est la multiplication par $-w(\kappa_2)$
sur $(H^1(X(]p[,p)^{0,\times},\O)\otimes{\cal C}(\cZ^\dual,C))_\kappa$.

Ceci prouve le premier point;  pour prouver le second, on utilise le mod\`ele de Kirillov:
le cor.\,\ref{igu6} fournit une injection $\PP(\Z_p)\times G_{\Q_p}$-\'equivariante
$$H^1_{\proet}(X(0)_C^\times,C)_\kappa/H^1(X(]p[,p)_C^\times,\O)_\kappa\to
{\cal C}(\Aidu,t^{-1}C)_\kappa$$
On peut d\'ecomposer $\Aidu$ sous la forme $W\times\Z_p^\dual$ (avec $W=\A^{]\infty,p[,\dual}\times
p^\Z$) et alors
${\cal C}(\Aidu,t^{-1}C)_\kappa={\cal C}(W,t^{-1}C)\otimes\kappa_1$ car 
$(\matrice{a}{0}{0}{1}_p\cdot\phi)(w,u)=\phi(w,a\,u)$.
L'action de $\sigma\in G_{\Q_p}$ est $(\sigma\cdot\phi)(w,u)
=\sigma(\phi(w,\cyp (\sigma)^{-1}u))$. Si $\phi=t^{-1}\phi_0\otimes\kappa_1$, avec
$\phi_0\in {\cal C}(W,C)$, on obtient l'action standard sur $\phi_0$ tordue
par $\kappa_1^{-1}(\sigma)\cyp (\sigma)^{-1}$, et comme
l'op\'erateur de Sen est trivial sur ${\cal C}(W,C)$, c'est la multiplication
par $-w(\kappa_1\cyp )=-w(\kappa_1)-1$ sur
${\cal C}(W,t^{-1}C)\otimes\kappa_1$.
\end{proof}

\begin{rema}\phantomsection\label{igu8}
{\rm(Comparer avec~\cite[th.\,1.0.1]{pan2})}
La suite
$$0\to H^1(X(]p[,p)_C^\times,\O)_\kappa\to (C\otimes H^1_{{\proet},c}(X(0)^\times_C,\Z_p))_\kappa
\to \O({\rm Ig}(]p[)^\times_C)_\kappa(-1)\to 0$$
est en fait exacte:
le terme suivant serait
$H^1(\Z_p^\dual\times\Z_p^\dual,H^1(X(]p[,p)^{\times}_C\times\Z_p^\dual,\O))$.
Ce groupe est isomorphe \`a $H^1(\Z_p^\dual,H^1(X(]p[,p)^{\times,0}_C,\O))$
car la cohomologie d'une induite est concentr\'ee en degr\'e $0$.
Mais $H^1(X(]p[,p)^{\times,0}_C,\O)$ est un quotient de $\O(\partial X^{\prime,0}_C)$
et donc $H^1(\Z_p^\dual,H^1(X(]p[,p)^{\times,0}_C,\O))$ est un quotient
de $H^1(\Z_p^\dual, \O(\partial X^{\prime,0}_C))$ car $\Z_p^\dual$ est de dimension cohomologique $1$.
Comme $\partial X^{\prime,0}_C/\Z_p^\dual$ est perfecto\"{\i}de affine, cela implique
que $H^1(\Z_p^\dual, \O(\partial X^{\prime,0}_C))=0$; d'o\`u le r\'esultat.
\end{rema}

\subsection{Injectivit\'e sur l'ouvert d'irr\'eductibilit\'e}

Soit ${\goth p}$ un id\'eal maximal ferm\'e de $\TT[\frac{1}{p}]$, et soit $\rho_{\goth p}$
le repr\'esentation semi-simple de $G_\Q$ dont la trace est $t_{\TT}$ modulo~${\goth p}$ o\`u
$t_{\TT}$ est le pseudo-caract\`ere de la rem.\,\ref{Ki106.1}.
On dit que ${\goth p}$ est {\it g\'en\'erique} si:

$\bullet$  $\rho_{\goth p}$ est irr\'educible,

$\bullet$ l'op\'erateur de Sen de la restriction de $\rho_{\goth p}$ \`a $G_{\Q_p}$
n'est pas scalaire.
\begin{rema}\phantomsection\label{igu8.1}
La premi\`ere condition est v\'erifi\'ee en dehors d'un ferm\'e
de codimension~$1$ et la seconde en dehors d'un ferm\'e de codimension~$2$ (car,
d'apr\`es Pan~\cite[th.\,6.2.2]{pan2}, la seconde condition implique
que $\rho_{\goth p}$ est un twist d'une repr\'esentation d'image finie).
\end{rema}

\begin{theo}\phantomsection\label{igu8.2}
Soit $W\subset H^1_{{\proet},c}(\widehat X^{(p)}(0)^\times_C,L)$ un sous-$L$-module ferm\'e,
stable par $\TT$, $G_\Q$ et $\GG(\A)$. On suppose que, 
$W[{\goth p}]=0$ pour tout ${\goth p}$ non g\'en\'erique. Alors
la restriction de ${\cal K}_H^{\GG}$ \`a $W$ est injective.
\end{theo}
\begin{proof}
Soit $W_0\subset W$ le noyau de ${\cal K}_H^{\GG}$. Alors $W_0$ 
est un sous-$L$-module ferm\'e,
stable par $\TT$, $G_\Q$, $\GG(\Q_p)$ et $\PP(\A^{]p[})$. Si $W_0\neq 0$, il existe $N$ premier \`a
$p$ tel que $W_0(Np^\infty):=W_0\cap \widehat H^1_c(Np^\infty)$ soit non nul.
Alors $W_0(Np^\infty)$ est une repr\'esentation admissible de $\GG(\Q_p)$ et donc
contient une repr\'esentation irr\'eductible~$\pi$. De plus,
$M:={\rm Hom}_{\GG(\Q_p)}(\pi,W_0(Np^\infty))$ est un $\TT$-module ($\TT$ agissant
\`a travers $\TT(Np^\infty)$), de type fini sur $L$ par admissibilit\'e
de $W_0(Np^\infty)$. Il existe donc un id\'eal maximal 
${\goth p}$ de $\TT[\frac{1}{p}]$ tel que $M[{\goth p}]\neq 0$, et ${\goth p}$
est g\'en\'erique gr\^ace \`a l'hypoth\`ese faite sur $W$. En particulier,
$\rho_{\goth p}$ est irr\'eductible, et donc
$M[{\goth p}]$ est $\rho_{\goth p}$-isotypique, en tant que $G_\Q$-module.
On en d\'eduit que $W_0$ contient une copie de $\rho_{\goth p}\otimes\pi$.

La classification~\cite{Pas,CDP} des repr\'esentations unitaires irr\'eductibles
de $\GG(\Q_p)$ et~\cite[\S\,V.4,\,\no 4]{gl2} impliquent qu'il existe
un caract\`ere de $\kappa=\kappa_1\otimes\kappa_2$ de $\BB(\Z_p)$ et $v\in\pi$, non nul,
tels que $\matrice{a}{b}{0}{d}\cdot v=\kappa_1(a)\kappa_2(d)\,v$, pour tout
$\matrice{a}{b}{0}{d}\in\BB(\Z_p)$.
Le cor.\,\ref{igu6} combin\'e avec la prop.\,\ref{igu7} implique alors
que l'op\'erateur de Sen de $\rho_{\goth p}$ est $-w(\kappa_2)$, ce qui est contraire
\`a l'hypoth\`ese selon laquelle ${\goth p}$ est g\'en\'erique
(et donc que cet op\'erateur de Sen n'est pas scalaire).

Ceci permet de conclure.
\end{proof}

\begin{rema}\phantomsection\label{NE1}
On peut prouver le m\^eme \'enonc\'e en supposant seulement que $\rho_{\goth p}$ est irr\'eductible
pour tout ${\goth p}$ tel que $W({\goth p})\neq 0$ mais la preuve (cf.~preuve du th.\,\ref{NE2})
utilise un argument de {\og prolongement analytique\fg} et le r\'esultat de Pan mentionn\'e dans la
rem.\,\ref{igu8.1}.
\end{rema}

}

\section{Quelques cons\'equences de l'injectivit\'e du mod\`ele de Kirillov}\label{Ki114}
Dans ce chapitre, on explique comment d\'eduire de l'injectivit\'e
du mod\`ele de Kirillov une compatibilit\'e local-global (th.\,\ref{Ki117}) pour la correspondance
de Langlands locale $p$-adique.  Cela fournit une compatibilit\'e
entre les correspondances de Langlands locale $p$-adique et classique (th.\,\ref{Ki119})
et une preuve de la conjecture de Fontaine-Mazur (en dimension~$2$, pour les repr\'esentations
impaires) sous des conditions assez faibles (th.\,\ref{Ki121}).

\Subsection{Compatibilit\'e local-global}\label{loglo1}
Soit $N\geq 1$, premier \`a $p$, et soit $S$ l'ensemble des nombres premiers divisant $Np$.
On note $\wGamma(Np^\infty)\subset \GG(\cZ)$ le groupe
$$\wGamma(Np^\infty)=\cap_{k\geq 1}\wGamma(Np^k).$$
C'est un sous-groupe ouvert de $\GG(\cZp)$.
Le groupe 
$$H^1(Np^\infty):=H^1\big(\GG(\Q),{\cal C}\big(\GG(\A)/\wGamma(Np^\infty),\O_L\big)\big)$$ 
s'identifie
\`a la cohomologie compl\'et\'ee de la tour
des courbes modulaires de niveaux $Np^k$, pour $k\geq 1$.
La comparaison avec la cohomologie \'etale le munit d'une action de
$G_\Q$ qui est non ramifi\'ee en dehors de $S$, et donc se factorise
\`a travers le quotient $G_{\Q,S}$ de $G_\Q$.
Il est aussi muni d'une action de $\GG(\Q_p)$ (induite par l'action naturelle sur
${\cal C}(\GG(\A),\O_L)$ par translation \`a droite), et est admissible
en tant que repr\'esentation de $\GG(\Q_p)$.
Ces deux actions commutent.
La limite inductive
$$H^1(0)^{(p)}:=\varinjlim\nolimits_N H^1(Np^\infty)=
H^1\big(\GG(\Q),{\cal C}^{(p)}\big(\GG(\A),\O_L\big)\big)$$
est munie d'actions de $G_\Q$ et $\GG(\Ai)$ qui commutent.

\vskip2mm
Soit $\Gamma$ un sous-groupe de $G_\Q\times\GG(\Ai)$. Si $W$ est un $L[\Gamma]$-module,
et si ${\goth m}$ est un id\'eal maximal de $\TT$,
on pose
\begin{align*}
&{\bf m}(W):={\rm Hom}_\Gamma(W,L\otimes H^1(0)^{(p)}),&&
{\bf m}(W, Np^\infty):={\rm Hom}_\Gamma(W,L\otimes H^1(Np^\infty))\\
&{\bf m}(W)_{\goth m}:={\rm Hom}_\Gamma(W,L\otimes H^1(0)_{\goth m}^{(p)}),&&
{\bf m}(W, Np^\infty)_{\goth m}:={\rm Hom}_\Gamma(W,L\otimes H^1(Np^\infty)_{\goth m})
\end{align*}
Si $\Gamma'$ est un sous-groupe de $G_\Q\times\GG(\Ai)$ commutant \`a $\Gamma$, ces modules
sont des $L[\Gamma']$-modules et on a 
$${\bf m}(W)=\varinjlim_N {\bf m}(W,Np^\infty) 
\quad{\rm et}\quad
{\bf m}(W)_{\goth m}=\varinjlim_N {\bf m}(W,Np^\infty)_{\goth m}$$

\begin{rema}
(i) Si $\rho$ une $L$-repr\'esentation de $G_\Q$, de dimension finie,
alors ${\bf m}(\rho,Np^\infty)$ est une $L$-repr\'esentation admissible de $\GG(\Q_p)$.

(ii) Si $\Pi$ est une $L$-repr\'esentation admissible de $\GG(\Q_p)$, de longueur finie, alors
${\bf m}(\Pi,Np^\infty)$ est une $L$-repr\'esentation de $G_\Q$, de dimension finie.
\end{rema}

\begin{defi}\phantomsection\label{Ki115}
Supposons $W$ absolument irr\'eductible.
On dit que $W$ est:

$\bullet$ {\it promodulaire} (resp.~{\it ${\goth m}$-promodulaire}), si ${\bf m}(W)\neq 0$
(resp.~${\bf m}(W)_{\goth m}\neq 0$); 

$\bullet$ {\it visible}
(resp. {\it ${\goth m}$-visible}) si $W$ est promodulaire (resp.~${\goth m}$-promodulaire)
et ${\cal K}_H^\GG$ est injective 
sur ${\bf m}(W)\otimes W$ (resp.~sur ${\bf m}(W)_{\goth m}\otimes W$) vu comme sous-objet
de $H^1(0)^{(p)}$ par l'\'evaluation.
\end{defi}
\begin{rema}\phantomsection\label{Ki116}
(i) la promodularit\'e de $\rho:G_{\Q}\to {\rm GL}_2(L)$, {\it impaire et non ramifi\'ee
en dehors d'un nombre fini de places}, est quasi-automatique~\cite[th.1.2.3]{Em08}.

(ii)
Il faut penser \`a {\og visible\fg} comme {\og d\'etect\'e par le mod\`ele de Kirillov\fg}.
Les r\'esultats d'injectivit\'e pour ${\cal K}_H^\GG$ 
des chap.~\ref{rCDN0} et~\ref{IG0} montrent que les repr\'esentations invisibles sont difficiles \`a exhiber:

$\bullet$ Si ${\goth m}$ est non-eisenstein, alors
{\og ${\goth m}$-promodulaire $\Rightarrow$ ${\goth m}$-visible\fg}.  

$\bullet$ Si $\rho:G_{\Q}\to {\rm GL}_2(L)$ est absolument irr\'eductible
et si l'op\'erateur de Sen de $\rho_{|G_{\Q_p}}$ n'est pas scalaire, alors
{\og $\rho$ promodulaire $\Rightarrow$ $\rho$ visible\fg}.
(En particulier, si $\rho$ est de Rham, \`a poids de Hodge-Tate distincts, et si
$\rho$ est promodulaire, alors $\rho$ est visible.)
\end{rema}

\begin{theo}\phantomsection\label{Ki117}
Soit $\rho:G_\Q\to {\rm GL}_2(L)$ absolument irr\'eductible, visible,
et soit $\Pi$
une composante irr\'eductible du $\GG(\Q_p)$-socle de ${\bf m}(\rho)$.

{\rm (i)} 
Il existe
une fl\`eche $G_{\Q_p}$-\'equivariante ${\bf V}(\Pi)\to\rho^\dual$, non nulle.

{\rm (ii)} 
De plus:

\quad $\bullet$ 
Si $\rho_{|G_{\Q_p}}$ est absolument irr\'eductible, 
$\rho^\dual\cong {\bf V}(\Pi)$
et $\Pi\cong \Pi_p(\rho^\dual)$.

\quad $\bullet$  Si $\rho_{|G_{\Q_p}}$ est une extension non triviale de $\chi_2$ par $\chi_1$, alors
$\Pi$ est le socle de $\Pi_p(\rho^\dual)$, i.e. 
$B(\chi_1^{-1},\chi_2^{-1})$.

\quad $\bullet$ Si $\rho_{|G_{\Q_p}}=\chi_1\oplus\chi_2$, alors $\Pi$ est une des composantes
du socle de $\Pi_p(\rho^\dual)$, i.e. $B(\chi_1^{-1},\chi_2^{-1})$ ou $B(\chi_2^{-1},\chi_1^{-1})$.

{\rm (Dans les deux cas, si $\chi_2=(x|x|)\chi_1$, 
il faut remplacer $B(\chi_1^{-1},\chi_2^{-1})$ par ${\rm St}\otimes\chi_2^{-1}$.)}
\end{theo}
\begin{proof}
Comme $\rho$ est suppos\'ee visible, ${\cal K}_H^\GG$ est injective 
sur $\Pi\otimes \rho$. En particulier, il existe $v\in \Pi\otimes \rho$ telle que ${\cal K}_{H,v}$
ne soit pas identiquement nulle. Si ${\cal K}_{H,v}(x)\neq 0$, la restriction
de ${\cal K}_H$ \`a $\matrice{x}{0}{0}{1}^{]p[}\star (\Pi\otimes \rho)$
(isomorphe \`a $\Pi\otimes \rho$ comme $G_\Q\times\GG(\Q_p)$-module) n'est pas
identiquement nulle,
et il r\'esulte
de la rem.\,\ref{Ki111.1} que
$v\mapsto {\cal K}_{H,v}(1)$  
fournit
une fl\`eche $G_{\Q_p}\times \matrice{p^\Z}{\Q_p}{0}{1}$-\'equivariante 
$\matrice{x}{0}{0}{1}^{]p[}\star (\Pi\otimes \rho)\to\tB^-$, non nulle.
Le (i) est donc une cons\'equence de la~prop.\,\ref{kirp3}.

Pour prouver le (ii), on utilise le fait que 
${\bf V}(\Pi)$ est irr\'eductible (rem.\,\ref{kirp4}).
On en d\'eduit directement le cas o\`u la restriction \`a $G_{\Q_p}$ est irr\'eductible.

Si la semi-simplifi\'ee de $\rho_{|G_{\Q_p}}$
est $\chi_1\oplus\chi_2$, commen\c{c}ons par remarquer que l'irr\'eductibilit\'e
de $\Pi$ et l'existence d'une fl\`eche non nulle $\rho^\dual\to{\bf V}(\Pi)$ implique
que ${\bf V}(\Pi)$ est de dimension $1$, et donc que $\Pi$ est une composante
de Jordan-H\"older de dimension infinie
d'une s\'erie principale $B(\delta_1,\delta_2)$. Dans le cas d'une extension non triviale, 
${\bf V}(\Pi)$
est le quotient $\chi_1^{-1}$ de $\rho^*$ 
(dans le cas scind\'e, c'est soit $\chi_1^{-1}$, soit $\chi_2^{-1}$), et donc $\delta_1=\chi_1^{-1}$.
On d\'etermine $\delta_2$ en utilisant le lien entre le caract\`ere central et le d\'eterminant
(cf.~rem.\,\ref{Ki106.1}), et donc $\delta_2=\chi_2^{-1}$. 
Le r\'esultat s'en d\'eduit en utilisant le fait que les s\'eries principales sont
irr\'eductibles sauf si $\delta_2=(x|x|)^{-1}\delta_1$, auquel cas la seule composante
de Jordan-H\"older de dimension infinie est ${\rm St}\otimes \delta_2$.
\end{proof}

\Subsection{Applications}\label{loglo2}
\Subsubsection{Adh\'erence des vecteurs localement alg\'ebriques}\label{loglo3}
\begin{prop}\phantomsection\label{Ki118}
Soit $\pi$ cohomologique de poids $(k+2,j+1)$ avec $m_{\eet}(\pi)$ visible.
Soient $\pi^{\rm alg}=\pi\otimes W_{k,j}^\dual$ et
$W$ l'adh\'erence de l'image de 
$$\iota_\pi : m(\pi)\otimes\pi^{\rm alg}
\to H^1(\GG(\Q),{\cal C}^{(p)}(\GG(\A),\Q_p(\pi)))$$
 Alors, en tant que
$G_\Q\times\GG(\Ai)$-repr\'esentation, $W\cong
m_{\eet}(\pi)\otimes\pi^{]p[}\otimes W_p$, o\`u
$W_p$ est la composante du socle de
$\Pi_p(m_{\eet}^\dual(\pi))$ ayant des vecteurs alg\'ebriques non nuls.
En particulier:

$\bullet$ Si $m_{\eet}(\pi)_{|G_{\Q_p}}$ est irr\'eductible,
$W\cong m_{\eet}(\pi)\otimes\pi^{]p[}\otimes \Pi_p(m_{\eet}^\dual(\pi))$.

$\bullet$ Si $m_{\eet}(\pi)_{|G_{\Q_p}}$ est une extension
non triviale de deux caract\`eres, alors
$W\cong m_{\eet}(\pi)\otimes\pi^{]p[}\otimes {\rm soc}(\Pi_p(m_{\eet}^\dual(\pi)))$.
\end{prop}
\begin{proof}
$m_{\eet}(\pi)$ est irr\'eductible, et donc ${\bf m}(m_{\eet}(\pi))\otimes m_{\eet}(\pi)\to
H^1(0)^{(p)}$ est injective, $W$ est inclus dans l'image et est de la forme 
$m_{\eet}(\pi)\otimes \pi^{]p[}\otimes W'$, o\`u $W'$ est une repr\'esentation de
$\GG(\Q_p)$, compl\'et\'ee de $\pi_p^{\rm alg}$, et
$\pi^{]p[}\otimes W'\subset {\bf m}(m_{\eet}(\pi))$.  
Par ailleurs, $W'$ est admissible car $m_{\eet}(\pi)\otimes v_\pi^{]p[}\otimes W'$ s'injecte dans
$H^1(Np^\infty)$ pour $N$ bien choisi, et $H^1(Np^\infty)$ est admissible.
En particulier, $W'$ contient des sous-repr\'esentations irr\'eductibles.

Soit $\Pi$ une composante irr\'eductible du $\GG(\Q_p)$-socle de $W'$.  
Il r\'esulte du th.\,\ref{Ki117} que $\Pi$ est une composante 
du socle de $\Pi_p(m_{\eet}^\dual(\pi))$ (et est \'egal \`a ce socle sauf si la restriction
de $m_{\eet}(\pi)$ \`a $G_{\Q_p}$ est la somme de deux caract\`eres).

Maintenant, comme $m_{\eet}(\pi)$ est de Rham \`a poids de Hodge-Tate distincts, 
$\Pi_p(m_{\eet}^\dual(\pi))^{\rm alg}\neq 0$
et, plus pr\'ecis\'ement, $({\rm soc}(\Pi_p(m_{\eet}^\dual(\pi))))^{\rm alg}\neq 0$;
on a donc $\Pi^{\rm alg}\neq 0$. Comme $m_{\eet}(\pi)\otimes\pi^{]p[}\otimes \Pi^{\rm alg}$
s'injecte dans $H^1(\GG(\Q),{\cal C}(\GG(\A),L))$, on d\'eduit de la prop.\,\ref{VV3}
que $\Pi^{\rm alg}\cong\pi_p^{\rm alg}$.
Comme $\Pi$ est ferm\'ee, 
et $\pi_p^{\rm alg}$
est dense dans $W'$, l'injection de $\Pi$ dans $W'$ est une \'egalit\'e.

Cela prouve le r\'esultat sauf
si la restriction de $m_{\eet}(\pi)$ \`a $G_{\Q_p}$ est la somme de deux caract\`eres.
Dans ce cas,
le socle de $\Pi_p(m_{\eet}^\dual(\pi)))$ a deux composantes $W_1,W_2$, avec $W_1^{\rm alg}\neq 0$
et $W_2^{\rm alg}=0$.  Si le socle de $W'$ a une composante isomorphe \`a $W_1$, on conclut
comme ci-dessus que $W=W_1$. Sinon, le socle de $W'$ est une 
somme de copies de $W_2$ qui est une s\'erie principale $B(\delta_1,\delta_2)$.
Comme $W'$ est admissible,
la th\'eorie des blocs pour les banachs admissibles~\cite{Pas,PT} implique que
les composantes de Jordan-H\"older de $W'$ appartiennent au bloc de
$B(\delta_1,\delta_2)$ qui est constitu\'e de
$B(\delta_1,\delta_2)$ et des composantes 
de Jordan-H\"older de $B(\delta_2,\delta_1)$ (\`a savoir $W_1$ plus un caract\`ere
\'eventuel). Il s'ensuit que $\pi_p^{\rm alg}$ (qui s'injecte dans $W'$ et qui est irr\'eductible)
est isomorphe \`a $W_1^{\rm alg}$ et donc le compl\'et\'e universel de $\pi_p^{\rm alg}$
est $W_1$ (car $W_1$ est une s\'erie principale unitaire). Comme $W'$ est, par construction,
 un quotient de ce compl\'et\'e universel, on en d\'eduit que
$W'=W_1$, ce qui permet de conclure.
\end{proof}

\Subsubsection{Compatibilit\'e classique versus $p$-adique}\label{loglo4}
\begin{theo}\phantomsection\label{Ki119}
{\rm (cf.~\cite[\S\,7.4]{Em08})}
Si $V$ est une $L$-repr\'esentation de dimension~$2$ de $G_{\Q_p}$, de Rham 
\`a poids de Hodge-Tate $a<b$, alors
$$\Pi_p(V)^{\rm alg}\cong \Pi_p(V)^{\rm cl}\otimes W^\dual_{b-a-1,b}$$
\end{theo}
\begin{proof}
Soit $\pi_p=\Pi_p(V)^{\rm cl}$.
Le r\'esultat est prouv\'e dans~\cite[th.\,VI.6.50]{gl2} dans le
cas o\`u $\pi_p$ est de la s\'erie principale (voir aussi~\cite{Cvectan,LXZ}). 
Il suffit donc de le prouver dans le cas $\pi_p$ supercuspidale et,
compte-tenu de la compatibilit\'e
\`a la torsion par un caract\`ere,
il suffit, d'apr\`es~\cite{poids}, pour chaque telle $\pi_p$, d'exhiber une repr\'esentation
$V$, de Rham, v\'erifiant ${\rm LL}(V)=\pi_p\otimes\eta$, pour laquelle 
$\Pi_p(V)^{\rm alg}=\pi_p\otimes\eta\otimes W_{k,j}^\dual$ (pour un choix de $(k,j)$
et de caract\`ere lisse $\eta$).

Quitte \`a remplacer $\pi_p$ par $\pi_p\otimes\eta$,
on peut supposer que $\pi_p$ est la composante
en $p$ d'une repr\'esentation $\pi$ cohomologique.
La prop.\,\ref{Ki118} permet de prendre $V=m^\dual_{\eet}(\pi)$,
\`a condition que
$m_{\eet}(\pi)$ soit visible (ce qui est le cas, d'apr\`es la rem.\,\ref{Ki116}
puisque cette repr\'esentation est modulaire par d\'efinition).
\end{proof}

\subsubsection{Repr\'esentations ordinaires}\label{loglo5}
Si $\Pi$ est une repr\'esentation unitaire de $\GG(\Q_p)$, irr\'eductible et visible,
l'injectivit\'e de ${\cal K}_H^\GG$ sur ${\bf m}(\Pi,Np^\infty)\otimes\Pi$
et l'irr\'eductibilit\'e de ${\bf V}(\Pi)$ impliquent, par les m\^emes m\'ethodes
(en particulier, l'utilisation de la prop.\,\ref{kirp3})
que le $G_{\Q_p}$-cosocle de la restriction de ${\bf m}(\Pi,Np^\infty)$
est une somme finie de copies de ${\bf V}(\Pi)^\dual$.
En particulier:

$\bullet$ ${\bf V}(\Pi)$ est de dimension~$\leq 2$ (cas particulier d'un th\'eor\`eme
de \paskunas~\cite{Pas} -- voir aussi~\cite{CDP}) puisque le 
$G_\Q$-cosocle de ${\bf m}(\Pi,Np^\infty)$
est constitu\'e de repr\'esentations de $G_\Q$ de dimensions~$\leq 2$ (d'apr\`es les relations
d'Eichler-Shimura).

$\bullet$ 
Si $\Pi$ est une s\'erie principale $B(\delta_1,\delta_2)$ (ou ${\rm St}\otimes\delta_1$), alors
le $G_{\Q_p}$-cosocle de ${\bf m}(\Pi,Np^\infty)$ est une somme de copies de $\delta_1^{-1}$.

On d\'emontre de m\^eme le r\'esultat suivant.

\begin{prop}\phantomsection\label{Ki120}
{\rm(cf.~\cite[\S\,5.6]{Em08})}
Soit
$\pi$ cohomologique de poids $(k+2,j)$ tel que $\pi_p$ soit une s\'erie principale
${\rm Ind}_B^G(\delta_1\otimes|\ |_p^{-1}\delta_2)$ {\rm(}ou ${\rm St}\otimes\delta_1$
si $\delta_2=|\ |_p\delta_1${\rm)}, o\`u $\delta_1,\delta_2:\Q_p^\dual\to L^\dual$
sont des caract\`eres lisses avec $v_p(\delta_1(p))=0$ et 
$v_p(\delta_2(p))=k+1$ {\rm (cas ordinaire)}.
Alors, la restriction de $m_{\eet}(\pi)$ \`a $G_{\Q_p}$ est une extension de $x^{-k-1}\delta_2^{-1}$
par $\delta_1^{-1}$.
\end{prop}
\begin{proof}
Soit $\Pi:=\pi_p^{\rm alg}=B(x^{k+1}\delta_2,\delta_1)$, et donc ${\bf V}(\Pi)^\dual=x^{-k-1}\delta_2^{-1}$.
Alors $m_{\eet}(\pi)$ est une sous-repr\'esentation de ${\bf m}(\Pi)$ et donc
${\cal K}_H^\GG$ est injective sur $m_{\eet}(\pi)\otimes\Pi$.  Comme ci-dessus,
on en d\'eduit que 
le cosocle de
la restriction de $m_{\eet}(\pi)$ \`a $G_{\Q_p}$ est $x^{-k-1}\delta_2^{-1}$.
On conclut en utilisant la relation entre
caract\`ere central et d\'eterminant (rem.\,\ref{Ki106.1}).
\end{proof}

\Subsubsection{Conjecture de Fontaine-Mazur}\label{loglo6}
\begin{theo}\phantomsection\label{Ki121}
{\rm(cf.~\cite[\S\,7.1]{Em08})}
Soit $\rho:G_{\Q,S}\to {\rm GL}_2(L)$, absolument irr\'eductible, impaire, visible,
 dont la restriction \`a $G_{\Q_p}$ est de Rham \`a poids
de Hodge-Tate distincts et n'est pas la somme de deux caract\`eres.
Alors $\rho$ est
la repr\'esentation associ\'e \`a une forme modulaire.
\end{theo} 
\begin{proof}
L'hypoth\`ese implique
que le socle de $\Pi_p(\rho^\dual)$ n'a qu'une composante, et l'hypoth\`ese {\og de Rham \`a poids 
de Hodge-Tate distincts\fg} implique~\cite[th.\,V.6.18]{gl2} (voir aussi~\cite{Dosp1,poids})
 que cette composante contient des vecteurs localement alg\'ebriques).

Maintenant, si $\rho$ est visible, le socle
de ${\bf m}(\rho)$ (vu comme repr\'esentation de $\GG(\Q_p)$) contient le socle
de $\Pi_p(\rho^\dual)$ d'apr\`es le th.\,\ref{Ki117}, et donc ${\bf m}(\rho)^{\rm alg}\neq 0$.
Ceci fournit une injection $G_\Q\times\GG(\Q_p)$-\'equivariante
${\bf m}(\rho)^{\rm alg}\otimes\rho\hookrightarrow H^1(0)^{(p)}$.
On conclut gr\^ace \`a la d\'ecomposition~(\ref{VV2}) des
vecteurs localement alg\'ebriques de la cohomologie compl\'et\'ee.
\end{proof}

\begin{rema}
Il faut d'autres outils que le simple mod\`ele de Kirillov pour supprimer la condition
selon laquelle $\rho_{|G_{\Q_p}}$ n'est pas la somme de deux caract\`eres.
On pourrait conclure par passage \`a la limite en approximant $\rho$ par des repr\'esentations
dont un des poids tend vers~$+\infty$ pour montrer que les deux morceaux du socle
de $\Pi_p(\rho^\dual)$ apparaissent dans le socle de ${\bf m}(\rho)$.
Le cas g\'en\'eral a \'et\'e d\'emontr\'e par Pan~\cite{pan1, pan3}.
\end{rema}

}

\section{La cohomologie compl\'et\'ee et sa factorisation}\label{YEUL1}

Dans ce chapitre, on raffine la factorisation d'Emerton~\cite[conj.\,6.1.6 et th.\,6.4.16]{Em08} 
de la cohomologie compl\'et\'ee.
Ce raffinement (th.\,\ref{Ycano110} et~\ref{NE2}) est obtenu en comparant les mod\`eles
de Kirillov des deux membres aux points classiques
(th.\,\ref{Ybato13} et~\ref{Ybato21}, cor.\,\ref{Ybato221}) gr\^ace \`a
la loi de r\'eciprocit\'e explicite du th.\,\ref{bato22}.

\Subsection{Notations}
\subsubsection{Alg\`ebres de Hecke et repr\'esentations galoisiennes}\label{YEUL7}
Soit $T(Np^\infty)$ le quotient \`a travers lequel $\TT$ agit sur $H^1(Np^\infty)$;
c'est l'alg\`ebre de Hecke 
engendr\'ee par les op\'erateurs de Hecke $T_\ell$, pour $\ell\notin S$,
agissant sur ${\cal C}(\GG(\Q_\ell)/\GG(\Z_\ell))$; on ne change pas
$T(Np^\infty)$ en enlevant un nombre fini de ces op\'erateurs de Hecke.
Cette action commute \`a celle de $\GG(\Q_p)$ et de $G_\Q$.

L'alg\`ebre $T(Np^\infty)$ est une alg\`ebre semi-locale. Dans la suite,
nous fixons un id\'eal maximal ${\goth m}$ 
de $T(Np^\infty)$, et \index{Tm@\TTm}notons $T$ la localis\'ee 
$$T:=T(Np^\infty)_{\goth m}$$ 
de $T(Np^\infty)$ en ${\goth m}$
(c'est un facteur direct de $T(Np^\infty)$).
Alors $T$ est une alg\`ebre locale, d'id\'eal maximal ${\goth m}$
et de corps r\'esiduel $T_{\goth m}/{\goth m}$ fini de caract\'eristique~$p$
(et donc de la forme ${\bf F}_q$, avec $q$ une puissance de $p$).

On note $t_T$ la composition du pseudocaract\`ere $t_\TT$ avec la projection
$\TT\to T$ et $t_{\goth m}$ la r\'eduction de $t_T$ modulo~${\goth m}$.
On dit que
${\goth m}$ est {\it non-eisenstein} si $t_{\goth m}$ est la trace d'une repr\'esentation
irr\'eductible de dimension~$2$.  Si ${\goth m}$ est non-eisenstein, il existe
$$\rho_T:G_{\Q,S}\to {\rm GL}_2(T)$$
dont la trace est $t_T$; la trace de la r\'eduction $\overline\rho_T$ de $\rho_T$
modulo~${\goth m}$ est alors $t_{\goth m}$.

On \index{Xcal@\calX}d\'efinit ({\og classique\fg} signifie que $\rho_x$ est la repr\'esentation
associ\'ee \`a une forme modulaire \`a torsion pr\`es 
par une puissance enti\`ere du caract\`ere cyclotomique): 
$${\cal X}:={\rm Spec}(T) 
\quad{\rm et}\quad
{\cal X}^{\rm cl}(\O_L):=\{x\in {\cal X}(\O_L),\ x\ {\rm classique}\}.$$

Soit $H^1(Np^\infty)$ le groupe d\'efini au \S\,\ref{loglo1} et soit
$H^1(Np^\infty)_{\goth m}$ son localis\'e
en ${\goth m}$ (c'est un facteur direct de $H^1(Np^\infty)$). 
On \index{H5@\Hrho}note $H^1[\rho_T]_S$ le sous-$\GG(\Ai)$-module
de $H^1(\GG(\Q),{\cal C}(\GG(\A),\O_L))$ engendr\'e par
$H^1(Np^\infty)_{{\goth m}}$.  
C'est un $G_{\Q,S}\times \GG(\Ai)$-module
\`a l'int\'erieur duquel on r\'ecup\`ere
$H^1(Np^\infty)_{{\goth m}}$ 
en prenant les points fixes par $\wGamma(Np^\infty)$.
La structure de ce $G_{\Q,S}\times \GG(\Ai)$-module a \'et\'e analys\'ee par Emerton~\cite{Em08};
nous donnons une version renforc\'ee du r\'esultat d'Emerton 
dans les th.\,\ref{Ycano110} et~\ref{NE2} ci-dessous.

\subsubsection{Sp\'ecialisation}\label{YEU30}
Si $x\in {\cal X}(\O_L)$, on note ${\goth p}_x$ l'id\'eal
premier correspondant; alors
$\rho_x=(T/{\goth p}_x)\otimes_T\rho_T$.  On a
$$\Pi_\ell(\rho_x^\dual)=(T/{\goth p}_x)\otimes_T\Pi_\ell(\rho_T^\diamond),
\ {\text{si $\ell\neq p$}},\quad
\Pi_p(\rho_x^\dual)=\Pi_p(\rho_T^\diamond)[{\goth p}_x].$$
(Donc $\Pi_p^\dual(\rho_x^\dual)=(T/{\goth p}_x)\otimes_T\Pi_p^\dual(\rho_T^\diamond)$.)
On pose:
$$ v'_{T,S}=\otimes_{\ell\in S\moins\{p\}} v'_{T,\ell},\quad
 v'_{x,S}=\otimes_{\ell\in S\moins\{p\}} v'_{x,\ell}.$$
Donc $ v'_{x,S}$ est la sp\'ecialisation de $ v'_{T,S}$
dans $\otimes_{\ell\in S\moins\{p\}}\Pi_\ell(\rho_x^\dual)$.

\begin{rema}\phantomsection\label{YEU31}
Si $x\in{\cal X}(\O_L)$, alors
$$\rho_x\otimes v'_{x,S} \otimes\Pi_p(\rho_x^\dual)=
\big(\rho_T\otimes  v'_{T,S}\otimes\Pi_p(\rho_T^\diamond)\big)[{\goth p}_x]$$
\end{rema}

\begin{rema}\phantomsection\label{YEU31.5}
Soit $x\in{\cal X}^{\rm cl}(\O_L)$,  correspondant \`a une repr\'esentation cohomologique $\pi$:
on a un isomorphisme $\rho_x\cong m_{\eet}(\pi)$, et $\Pi_\ell(m_{\eet}^\dual(\pi))=\pi_\ell$,
si $\ell\neq p$.

{\rm (i)}  
L'isomorphisme $\rho_x\cong m_{\eet}(\pi)$ n'est unique qu'\`a multiplication
pr\`es par une unit\'e de $L$ (ou $\O_L$ si on fait plus attention), mais 
$\rho_x\otimes\Pi_p(\rho_x^\dual)=m_{\eet}(\pi)\otimes\Pi_p(m_{\eet}^\dual(\pi))$ car le passage
au dual et la fonctorialit\'e de $\Pi_p$ font que les ind\'eterminations se compensent.
On peut donc voir $m_{\eet}(\pi)\otimes\Pi_p(m_{\eet}^\dual(\pi))$ comme un sous-objet
de $\rho_T\otimes_T \Pi_p(\rho_T^\diamond)[\frac{1}{p}]$.

{\rm (ii)}
Si $\pi$ est de poids $(k+2,j+1)$,
on dispose (par d\'efinition de $m_{\eet}(\pi)$) d'une fl\`eche \index{iotapi@\iotapi}naturelle
$G_{\Q,S}\times \GG(\Ai)$-\'equivariante:
$$\iota_{\pi} :m_{\eet}(\pi)\otimes \pi^{\rm alg}\to \iH^1(W_{k,j}^{\eet})\otimes W_{k,j}^\dual
\hookrightarrow H^1(\GG(\Q),{\cal C}(\GG(\A),L)).$$
\end{rema}

\Subsection{Une loi de r\'eciprocit\'e explicite} \label{Yqq24}
Si $c\in H^1_{\proet}(Z(0)^\times_C, \Bdr^+\otimes_{\Q_p}W_{k,j}^{\rm et})$ et
$\check v\in W_{k,j}^\dual$, 
alors 
$$\langle \check v,c\rangle\in H^1_{\proet}(Z(0)^\times_C, \Bdr^+\otimes_{\Q_p}{\rm Alg})$$
et on note $\log_B(c\otimes\check v)$ son image dans $\bdr^-\{\tilde q\}$ 
par l'application du \no\ref{qq15} (un \'el\'ement
de $H^1_{\proet}(Z(0)^\times_C, \Bdr^+\otimes_{\Q_p}{\rm Alg})$ est d\'efini sur $Z(N)$ pour $N$
assez petit (multiplicativement), et $Z(N)$ est une boule ouverte).

\begin{theo}\phantomsection\label{Ybato13}
Si $f \in {\rm Fil}^0(\bdr\otimes \iH^0(W_{k+2,j+1}^{\rm dR}))$,
et si l'on note $\iota_{{\rm ES},p}(f)_{Z(0)}$ la restriction de
$\iota_{{\rm ES},p}(f)\in\iH^1_{\proet}(\Bdr^+\otimes_{\Q_p}W_{k,j}^{\rm et})$ \`a $Z(0)_C$,
alors
$$\log_B\big(\iota_{{\rm ES},p}(f)_{Z(0)}\otimes
\tfrac{(e_1^\dual)^{k-\ell}(e_2^\dual)^\ell}{\ell!(e_1^\dual\wedge e_2^{\dual})^{j}}\big)=
(-1)^{k-j}t^{k-j-\ell}\partial_q^{-1-\ell}
\big({\rm Hol}(f)\otimes v_1^{-k-2}\zeta_{\rm dR}^{-j-1}\big)(\tilde q)$$
dans\footnote{On fera attention qu'il peut y avoir des puissances n\'egatives
de $t$ dans ${\rm Hol}(f)$; plus pr\'ecis\'ement, il faut multiplier
${\rm Hol}(f)$ par $t^{k+1-j}$ pour avoir une s\'erie
\`a coefficients dans $\bdr^+$.} $\bdr^-\{\tilde q\}$.
\end{theo}
\begin{proof}
On choisit $N$ tel que $f$ soit de niveau $N$.
Pour calculer $\iota_{{\rm ES},p}(f)_{Z(0)}$, on doit r\'esoudre l'\'equation
diff\'erentielle $\nabla g=f$ sur $Z(N)_C$.
Une solution naturelle est\footnote{Voir les ${\rm n}^{\rm os}$~\ref{qq22} et~\ref{como112.2}
pour les notations.} 
$$g=-\sum_{n\geq 1}\partial^{n-1}\tfrac{f}{\zeta_{\rm dR}v_1^2}\tfrac{(-u)^n}{n!}$$
car, quand on applique $\nabla$, les termes se compensent deux par deux, et il
ne reste que le morceau $\frac{f}{\zeta_{\rm dR}v_1^2}\frac{du}{1!}=f$ (cf.\,(\ref{qq22.1}))
du terme pour $n=1$.

Le $1$-cocycle sur le $\pi_1$ de $Z(N)_C$ d\'efinissant $\iota_{{\rm ES},p}(f)_{Z(0)}$ est alors
$\sigma\mapsto(\sigma-1)g$, mais il est apparent sur la formule
d\'efinissant $g$ que cette action se factorise par $\UU(N\cZ)$
et que ce cocycle est compl\`etement d\'etermin\'e par sa valeur
sur le g\'en\'erateur $\gamma^N$ de~$\UU(N\cZ)$.

Si $F=\partial ^{-1}\frac{f}{\zeta_{\rm dR}v_1^2}$, on a
\begin{align*}
\gamma^N(g)-g&=\sum_{n\geq 0}\partial ^nF(q)\big(\tfrac{(-u-Nt)^n}{n!}-\tfrac{(-u)^n}{n!}\big)\\
&=F(qe^{-u-Nt})-F(qe^{-u})=F(\tilde q e^{-Nt})-F(\tilde q)=(\gamma^N-1)\cdot F(\tilde q)
\end{align*}
(Le passage de la premi\`ere ligne \`a la seconde est 
la formule de Taylor logarithmique:
$\partial:=\frac{\nabla}{du}=q\frac{dq}{q}$ et $u=\log({q}/{\tilde q})$;
on obtient donc les valeurs en $q e^{-u}=\tilde q$ et $q e^{-u-Nt}=\tilde qe^{-Nt}=\gamma^N(\tilde q)$. 
En particulier, 
$v_1(\tilde q e^{-Nt})=v_1(\tilde q)=-te_1$: comme 
$v_1(q)=ue_2-te_1$, et $v_2(q)=e_2$, on a $\sum_{n\geq 0}(\frac{\nabla}{du})^{n} v_1\frac{(-u)^n}{n!}=
v_1-u\frac{\nabla}{du}v_1=(ue_2-te_1)-ue_2=-te_1$.)  

On peut \'ecrire $f$, sur $Z(0)$, sous la forme
$$f=\sum_{i=0}^kt^{i-(k+1-j)}f_i\otimes v_1^{k+2-i}v_2^i\zeta_{\rm dR}^{j+1},$$
avec $f_i\in\bdr^+\{q\}$.
Une int\'egration par partie
(on int\`egre $\partial_q^{-i}f$ et on d\'erive $v_1^{k-i}v_2^{i}$)
fournit la formule
$$F=\sum_{i=0}^k t^{i-(k+1-j)}\big(\sum_{r=0}^{k-i}(-1)^r\tfrac{(k-i)!}{(k-i-r)!}
\partial_q ^{-1-r}f_i\otimes v_1^{k-i-r}v_2^{i+r}\zeta_{\rm dR}^j\big)$$
Pour \'evaluer en $\tilde q$, on utilise la formule $v_1(\tilde q)=-te_1$ et $v_2=e_2$.

Maintenant, il r\'esulte du lemme~\ref{cano108} que $\log_B(\iota_{{\rm ES},p}(f)_{Z(0)}\otimes \check v)=
\langle\check v, F(\tilde q)\rangle$.
On en d\'eduit, en utilisant les formules
$\langle (e_1^\dual\wedge e_2^{\dual})^{-j},\zeta_{\rm dR}^j\rangle=(-t)^{-j}$,
$$\langle (e_1^\dual)^\ell(e_2^\dual)^{k-\ell},e_1^ie_2^{k-i}\rangle=
\begin{cases} 0& {\text{si $i\neq \ell$,}}\\
\frac{\ell!\,(k-\ell)!}{k!}&{\text{si $i=\ell$,}}\end{cases}$$
(par d\'efinition $\langle \check v_1\cdots\check v_k,v_1\cdots v_k\rangle=\frac{1}{k!}
\sum_{\sigma}\langle \check v_1,v_{\sigma(1)}\rangle\cdots \langle \check v_k,v_{\sigma(k)}\rangle$),
que l'on a
\begin{align*}
\log_B\big(\iota_{{\rm ES},p}(f)_{Z(0)}\otimes
\tfrac{(e_1^\dual)^{k-\ell}(e_2^\dual)^\ell}{\ell!\,(e_1^\dual\wedge e_2^{\dual})^{j}}\big)
=
\tfrac{1}{\ell!}\sum_{i=0}^{\ell}(-1)^{k-i-j}t^{-1-\ell+i}
\tfrac{(k-i)!}{(k-\ell)!}\partial_q^{-1-\ell+i}f_i \tfrac{\ell!\,(k-\ell)!}{k!}.
\end{align*}
(La formule ci-dessus fait que seuls les termes avec $r=\ell-i$ contribuent; l'exposant
de $t$ est $(i-(k+1-j))+(k-i-r)+(-j)$, le second terme venant de $v_1^{k-i-r}$ et le troisi\`eme
de $\langle (e_1^\dual\wedge e_2^{\dual})^{-j},\zeta_{\rm dR}^j\rangle$;
l'exposant de $(-1)$ est $r+(k-i-r)+(-j)$, le second terme venant de $v_1^{k-i-r}$ et le troisi\`eme
de $\langle (e_1^\dual\wedge e_2^{\dual})^{-j},\zeta_{\rm dR}^j\rangle$.)
La somme ci-dessus ne change pas, dans $\bdr^-\{\tilde q\}$, si on remplace
$\sum_{i=0}^{\ell}$ par $\sum_{i=0}^k$ car l'exposant de $t$ est~$\geq 0$ si $i\geq \ell+1$.
Le r\'esultat s'en d\'eduit donc
en comparant la formule ci-dessus (apr\`es simplification)
avec celle de la prop.~\ref{bato14}.
\end{proof}
\begin{rema}
On aurait pu simplifier un peu le calcul en utilisant le fait que $\iota_{{\rm ES},p}$
se factorise \`a travers la projection holomorphe, mais il est rassurant de voir ceci appara\^{\i}tre
sur le calcul.
\end{rema}

\Subsection{Les vecteurs localement alg\'ebriques de la cohomologie compl\'et\'ee}\label{Yqq25}
\subsubsection{$q$-d\'eveloppements}
Soit $\Lambda=L,\O_L$. On \index{AQ@\taq}pose (cf.~\no\ref{qq16}):
$$\tA^-[[\tilde q^{\Q_+}]]\boxtimes\Z_p^\dual:=\hskip-2mm\varinjlim_{(N,p)=1}
\hskip-2mm\tA^-[[\tilde q^{1/N}]]\boxtimes\Z_p^\dual,
\quad
\Lambda\point\tA^-[[\tilde q^{\Q_+}]]\boxtimes\Z_p^\dual
:=\Lambda\otimes_{\Z_p}(\tA^-[[\tilde q^{\Q_+}]]\boxtimes\Z_p^\dual)$$
Le \index{galois2@\tGG}sous-groupe
$$\widetilde G_{\Q}=\{(\sigma,g)\in G_{\Q}\times \PP(\A^{]\infty[}),\ \det g=\cy(\sigma)\}=\big\{
\matrice{\sigma}{b}{0}{1},\ \sigma\in G_{\Q},\ b\in\A^{]\infty[}\big\}$$
de $G_\Q\times\GG(\A^{]\infty[})$ stabilise $Z(0)_C$.
On note $\widetilde G_{\Q_p}$ le sous-groupe de $\widetilde G_{\Q}$,
$$\widetilde G_{\Q_p}=\big\{
\matrice{\sigma}{b}{0}{1},\ \sigma\in G_{\Q_p},\ b\in\Q_p\big\}$$
Notons que $\widetilde G_{\Q_p}$ admet $G_{\Q_p}$ et $\matrice{\Z_p^\dual}{\Q_p}{0}{1}$
comme quotients.

\noindent
$\bullet$ {\it L'application $\iota_q$}.---
D'apr\`es le (i) de la rem.~\ref{cano1.31} (avec $(\rho_T^\diamond)^\clubsuit=\rho_T^\dual$), 
on dispose d'une injection naturelle
$$\iota:\Pi_p(\rho_T^\diamond)\hookrightarrow (\tA^-\otimes_{\Z_p}\rho_T^\dual)^H$$
 qui commute
\`a l'action de $\PP(\Q_p)$:
\begin{equation}\label{Yqq25.4}
\iota\big(\matrice{p^ka}{b}{0}{1}_p\star v)=
[\epsilon^b]\varphi^k\circ\sigma_a(\iota(v)),\quad{\text{si $k\in\Z$, $a\in\Z_p^\dual$ et $b\in\Q_p$}}.
\end{equation}
On en d\'eduit un accouplement naturel
$\langle\ ,\ \rangle:\rho_T\otimes_T\Pi_p(\rho_T^\diamond)\to \O_L\point\tA^-$.
Cela permet de \index{iotaq@\iotaq}d\'efinir
$$\iota_{q}:\rho_T\otimes_T\Pi(\rho_T^\diamond)\to \O_L\point\tA^-[[\tilde q^{\Q_+}]]\boxtimes\Z_p^\dual$$
donn\'ee, si $\gamma\in\rho_T$, $\phi^{]p[}\in\Pi^{]p[}(\rho_T^\diamond)$
et $v\in \Pi_p(\rho_T^\diamond)$, par
$$\gamma\otimes (\phi^{]p[}\otimes v) \mapsto
\sum_{n>0,\,v_p(n)=0}\big\langle\gamma,
\iota\big(\matrice{n}{0}{0}{1}_p\star (\phi^{]p[}(n)v)\big)\big\rangle \,\tilde q^n.$$
\begin{lemm}
Alors $\iota_q$ est $\widetilde G_{\Q_p}$-\'equivariant, si on fait agir
$\widetilde G_{\Q_p}$ \`a travers $G_{\Q_p}$ sur $\rho_T$, 
\`a travers $\matrice{\Z_p^\dual}{\Q_p}{0}{1}\subset\GG(\Q_p)$
sur $\Pi(\rho_T^\diamond)$, et par 
$\matrice{\sigma}{b}{0}{1}\cdot\big(\sum_na_n\tilde q^n\big)=\sum_n\sigma(a_n)[\epsilon^{nb}]\,\tilde q^n$.
\end{lemm}
\begin{proof}
En effet,
$$\matrice{\sigma}{b}{0}{1}\cdot(\gamma\otimes v)=
\sigma(\gamma)\otimes \matrice{\chi(\sigma)}{b}{0}{1}_p \star v$$
tandis que (l'apparition de $\chi(\sigma)=\cyp (\sigma)$ 
vient de la formule~(\ref{Yqq25.4}) ci-dessus)
\begin{align*}
\matrice{\sigma}{b}{0}{1}\cdot\big(\sum\big\langle\gamma,
\iota\big(\matrice{n}{0}{0}{1}_p\star \phi^{]p[}(n)v\big)\big\rangle \,\tilde q^n\big)&=
\sum\big\langle\sigma(\gamma),
\iota\big(\matrice{\chi(\sigma)n}{0}{0}{1}_p\star \phi^{]p[}(n)v\big)\big\rangle 
[\epsilon^{bn}]\,\tilde q^n\\
&=\sum\big\langle\sigma(\gamma),
\iota\big(\matrice{\chi(\sigma)n}{bn}{0}{1}_p\star \phi^{]p[}(n)v\big)\big\rangle \,\tilde q^n
\end{align*}
\end{proof}

\noindent
$\bullet$ {\it Les applications ${\rm Res}$ et $\iota_{\tA}$}.---
Enfin, on \index{res@\RES}note 
$${\rm Res}:H^1_{{\eet}}(\widehat X^{(p)}(0)_C^\times,\Z_p)
\to H^1_{\eet}(\widehat Z^{(p)}(0)^\times_C,\Z_p)$$ l'application
naturelle (de restriction),
et on \index{iotaA@\iotaA}note encore $\iota_{\tA}$
l'application
$$\iota_{\tA}:H^1(\widehat Z^{(p)}(0)^\times_C,\Z_p)\to \tA^-[[\tilde q^{\Q_+}]]\boxtimes\Z_p^\dual$$
d\'eduite de~(\ref{bato40.4}) en passant \`a la limite sur les rev\^etements
$q\mapsto q^N$, pour $(N,p)=1$.  
Alors $\iota_\pi$ est, par d\'efinition, $G_\Q\times\GG(\A^{]\infty[})$-\'equivariante,
${\rm Res}$ est $\widetilde G_\Q$-\'equivariante et $\iota_{\tA}$ est
$\widetilde G_{\Q_p}$-\'equivariante.

\subsubsection{Sp\'ecialisation aux points classiques}
Si $\pi$ est cohomologique, 
la formule~(\ref{emoins}) et la compatibilit\'e entre les correspondances
locales $p$-adique et classique (th.\,\ref{Ki119}) fournissent
un \index{iotadr@\iotadr}isomorphisme
 $$\iota_{{\rm dR}}^-:\pi_p^{\rm alg}\cong \Pi_p(m_{\eet}(\pi)^\dual)^{\rm alg},
\quad v\mapsto v\otimes (\iota_{{\rm dR},\check\pi}^-\otimes\zeta_{\rm dR}).$$
Comme $m_{\eet}(\pi)\otimes \Pi_p(m_{\eet}^\dual(\pi))$ est un sous-objet de
$\rho_T\otimes_T\Pi_p(\rho_T^\diamond)[\frac{1}{p}]$ (cf.~(i) de la rem.\,\ref{YEU31.5}), cela 
fournit $$\iota_q\circ\iota_{\rm dR}^-:m_{\eet}(\pi)\otimes\pi^{\rm alg}\to
L\point\tA^-[[\tilde q^{\Q_+}]]\boxtimes\Z_p^\dual.$$
Si $\pi$ est de poids $(k+2,j+1)$,
on dispose d'une fl\`eche naturelle ((ii) de la rem.\,\ref{YEU31.5})
$$\iota_\pi :m_{\eet}(\pi)\otimes \pi^{\rm alg}\to \iH^1(W_{k,j}^{\eet})\otimes W_{k,j}^\dual\to
H^1_{\eet}(\widehat X^{(p)}(0)_C^\times,\Q_p)^{U-{\rm fini}}.$$

\begin{theo}\phantomsection\label{Ybato21}
Si $\pi$ est cohomologique de poids $(k+2,j+1)$,
alors
$$\iota_{\tA}\circ{\rm Res}\circ\iota_\pi =(-1)^{k-j}\iota_q\circ \iota_{{\rm dR}}^-.$$
\end{theo}
\begin{proof}
Les deux fl\`eches tombent dans $(L\point\tA^-[[\tilde q^{\Q_+}]]\boxtimes\Z_p^\dual)^{U-{\rm fini}}$, 
ce qui permet de les composer avec l'injection $\kappa_{\rm dR}$ du \no\ref{qq16}, et
il suffit de v\'erifier le m\^eme \'enonc\'e apr\`es composition avec $\kappa_{\rm dR}$, auquel cas
on peut \'etendre les deux membres par $\bdr^+$-lin\'earit\'e en des fl\`eches
$\bdr^+\otimes m_{\eet}(\pi)\otimes \pi^{\rm alg}\to L\point\bdr^-\{\tilde q\}$.

Le lemme~\ref{Ybato5} ci-dessous
montre qu'il suffit de v\'erifier que les deux fl\`eches co\"{\i}ncident
sur $t^{j-k-1}\iota_{{\rm dR},\pi}^+\otimes \pi^{\rm alg}$ 
(on a $t^{j-k-1}\iota_{{\rm dR},\pi}^+\in \bdr^+\otimes m_{\eet}(\pi)$ et $t^{j-k-1}\iota_{{\rm dR},\pi}^+$
n'est pas divisible par $t$ dans $\bdr^+\otimes m_{\eet}(\pi)$).
Par lin\'earit\'e, on peut se restreindre \`a $\phi$ de la forme
$\phi^{]p[}\otimes\phi_p\otimes \frac{(e_1^\dual)^{k-\ell}(e_2^\dual)^\ell}{\ell!\,(e_1^\dual\wedge e_2^\dual)^j}$
qui correspond \`a $\phi^{]p[}\otimes\phi_pX^{k-\ell-j}$ dans le mod\`ele de Kirillov de
$\pi^{\rm alg}$.

$\bullet$ Par d\'efinition de l'application $\kappa_{\rm dR}$, on a
(le coefficient de $\tilde q^{p^kn}$ du membre de droite est vu comme un \'el\'ement
de $L\point\bdr^-$ via l'inclusion $\tA^+[\frac{1}{[\epsilon^{p^k}]-1},\,k\in\Z]\hookrightarrow \bdr$):
$$\kappa_{\rm dR}\circ\iota_q(\gamma\otimes\phi^{]p[}\otimes v)=
\sum_{k\in\Z}
\sum_{n>0,\,v_p(n)=0}\big\langle\gamma,\phi^{]p[}(n)
\iota\big(\matrice{p^kn}{0}{0}{1}_pv\big)\big\rangle\, \tilde q^{p^kn}.$$
(L'apparition de $\matrice{p^kn}{0}{0}{1}_p$ vient de ce que
$\varphi^k(\iota(v))=\iota(\matrice{p^k}{0}{0}{1}v)$, si $v\in \Pi_p(\rho_T^\diamond)$,
et celle de $\tilde q^{p^kn}$ de ce que $\varphi^k(\tilde q^n)=\tilde q^{p^kn}$.)
Maintenant, si $v\in \Pi_p(m_{\eet}(\pi)^\dual)^{\rm alg}$, alors
$\iota\big(\matrice{p^kn}{0}{0}{1}_pv\big)={\cal K}_{v}(p^kn)$, par d\'efinition
de ${\cal K}_v$ (cf.~\no\ref{valg3}). On en d\'eduit que
$$\iota\big(\matrice{x}{0}{0}{1}_p\iota_{\rm dR}^-(\phi_pX^{k-\ell-j})\big)=
 (xt)^{k-\ell-j}\phi_p(x)(\iota_{{\rm dR},\check\pi}^-\otimes\zeta_{\rm dR}).$$
Il s'ensuit que
$$\kappa_{\rm dR}\circ\iota_q\circ\iota_{\rm dR}^-(t^{j-k-1}\iota_{{\rm dR},\pi}^+\otimes\phi)=
t^{j-k-1}\langle \iota_{{\rm dR},\pi}^+,\iota_{{\rm dR},\check\pi}^-\otimes\zeta_{\rm dR}\rangle\sum_{n>0}
\phi^{]p[}(n)\phi_p(n)(nt)^{k-\ell-j}\tilde q^n.$$
Par d\'efinition, $\langle \iota_{{\rm dR},\pi}^+,\iota_{{\rm dR},\check\pi}^-\otimes\zeta_{\rm dR}\rangle=1$.
Par ailleurs, 
si $a_n=n^{k+1-j}\phi(n)$, la
s\'erie $\sum_{n>0}a_nq^n$ est le $q$-d\'eveloppement d'une forme
modulaire $f_\phi$, et on a
\begin{align*}
\kappa_{\rm dR}\circ\iota_q\circ\iota_{\rm dR}^-(t^{j-k-1}\iota_{{\rm dR},\pi}^+\otimes\phi)& =
t^{j-k-1}\sum\nolimits_{n>0}a_nn^{j-k-1}(tn)^{k-\ell-j}\tilde q^n\\ &=t^{-\ell-1}\partial_q^{-1-\ell}f_\phi
\end{align*}

$\bullet$ Le th.\,\ref{bato22} permet de remplacer $\kappa_{\rm dR}\circ{\iota_{\tA}}$ par $\log_B $.
L'application $\iota_\pi \circ t^{j-k-1}\iota_{{\rm dR},\pi}^+$ est la compos\'ee de\footnote{Notons
que $t^{j-k-1}\iH^0(\omega^{k+2,j+1})\subset{\rm Fil}^0(\bdr\otimes\iH^0(W_{k+2,j+1}^{\rm dR}))$.}:

\quad $\diamond$ $t^{j-k-1}\iota_{{\rm dR},\pi}^+\otimes{\rm id}:\pi^{\rm alg}=\pi\otimes W_{k,j}^\dual\to
t^{j-k-1}\iH^0(\omega^{k+2,j+1})\otimes W_{k,j}^\dual$,

\quad $\diamond$ $\iota_{{\rm ES},p}\otimes{\rm id}:t^{j-k-1}\iH^0(\omega^{k+2,j+1})\otimes W_{k,j}^\dual\to
\iH^1_{\proet}(\Bdr^+\otimes_{\Q_p}W_{k,j}^{\eet}) \otimes W_{k,j}^\dual$.

La premi\`ere envoie $\phi$ sur 
$\big( t^{j-k-1}f_\phi\otimes v_1^{k+2}\zeta_{\rm dR}^{j+1})
\otimes \frac{(e_1^\dual)^{k-\ell}(e_2^\dual)^\ell}{\ell!\,(e_1^\dual\wedge e_2^\dual)^j}$.
Le th.~\ref{Ybato13} fournit alors la formule
\begin{align*}
{\rm log}_B({\rm Res}(\iota_\pi \circ t^{j-k-1}\iota_{{\rm dR},\pi}^+(\phi)))&=
 t^{j-k-1}(-1)^{k-j}t^{k-\ell-j}\partial_q^{-1-\ell}f_\phi\\ &=
 (-1)^{k-j}t^{-\ell-1}\partial_q^{-1-\ell}f_\phi
\end{align*}
(On a ${\rm Hol}(f_\phi\otimes v_1^{k+2}\zeta_{\rm dR}^{j+1})=
f_\phi\otimes v_1^{k+2}\zeta_{\rm dR}^{j+1}$ puisque cette forme est holomorphe.)

Une comparaison des deux formules permet de conclure.
\end{proof}

On identifie $G_{\Q_p}$ au sous-groupe de $\widetilde G_{\Q_p}$ des $\matrice{\sigma}{0}{0}{1}$.
\begin{lemm}\phantomsection\label{Ybato5}
Si $\alpha:\bdr^+\otimes m_{\eet}(\pi)\otimes\pi_p^{\rm alg}\to L\point\bdr^-$ 
est $\bdr^+$-lin\'eaire et $G_{\Q_p}$-\'equivariant,
et si $\alpha(t^{j-k-1}\iota_{{\rm dR},\pi}^+\otimes\pi_p^{\rm alg})=0$, alors $\alpha=0$.
\end{lemm}
\begin{proof}
Les poids de $m_{\eet}(\pi)$ sont $j$ et $j-k-1$.
L'action de $G_{\Q_p}$ sur $\pi_p^{\rm alg}$ se factorise \`a travers $\matrice{\Z_p^\dual}{0}{0}{1}$,
et
on peut (apr\`es extension de $L$ \`a $L(\bmu_{p^\infty})$)
d\'ecomposer $\pi_p^{\rm alg}$ comme une somme directe de rep\'esentations
de $G_{\Q_p}$ de la forme $\eta\cyp ^{i-j}$, o\`u $\eta$ est un caract\`ere d'ordre
fini et $0\leq i\leq k$.
On en d\'eduit que $m_{\eet}(\pi)\otimes\pi_p^{\rm alg}$ se d\'ecompose en repr\'esentations
de dimension~$2$, de la forme $V_{i,\eta}=m_{\eet}(\pi)\otimes \eta\cyp ^{i-j}$,
dont un poids est~$\geq 0$ et l'autre est $<0$.
La restriction de $\alpha$ \`a $V_{i,\eta}$ appartient
\`a $(\bdr^-\otimes V_{i,\eta}^\dual)^{G_{\Q_p}}$ qui est de dimension~$1$
car $V_{i,\eta}^\dual$ a un poids~$\leq 0$ et l'autre~$>0$.
Il s'ensuit que $\alpha$ est un multiple de 
$t^{i-j}G(\eta)(\iota_{{\rm dR},\check\pi}^-\otimes\zeta_{\rm dR})$.
L'hypoth\`ese implique que $\alpha$ s'annule sur $t^{j-k-1}G(\eta^{-1})\iota_{{\rm dR},\pi}^+\in
\bdr^+\otimes V_{i,\eta}$ et donc $\alpha=0$
car $$\langle t^{i-j}G(\eta)\iota_{{\rm dR},\check\pi}^-\otimes\zeta_{\rm dR},
t^{j-k-1}G(\eta^{-1})\iota_{{\rm dR},\pi}^+\rangle=t^{i-k-1}G(\eta)G(\eta^{-1})$$
n'est pas $0$ dans $\bdr^-$
puisque $i\leq k$.
\end{proof}

\subsubsection{Interpr\'etation en termes de mod\`eles de Kirillov}
Le th.\,\ref{Ybato21} se traduit de mani\`ere agr\'eable en termes des mod\`eles de Kirillov.
\begin{rema}
Soient $1_T$ l'unit\'e de $T$ et $1^{]\infty[}$ l'\'el\'ement neutre de $\Aidu$.

{\rm (i)} Si $v\in \rho_T\otimes\Pi(\rho_T^\diamond)$, on a
$$\iota_q(v)=\sum_{n>0,\,v_p(n)=0}{\cal K}_{\rm Aut}^T\big(\matrice{n}{0}{0}{1}^{]\infty[}\star v\big)
(1_T,1^{]\infty[})\,\tilde q^n$$

{\rm (ii)} Si $v\in H^1_{{\eet}}(\widehat X^{(p)}(0),\Z_p)$, on a
$$\iota_{\tA}\circ{\rm Res}(v)=
\sum_{n>0,\,v_p(n)=0}{\cal K}_{H}^\TT\big(\matrice{n}{0}{0}{1}^{]\infty[}\star v\big)
(1_T,1^{]\infty[})\,\tilde q^n$$
\end{rema}
On d\'eduit de cette remarque et du th.\,\ref{Ybato21} le r\'esultat suivant.
\begin{coro}\phantomsection\label{Ybato221}
Si $\pi$ est cohomologique de poids $(k+2,j+1)$, alors
$${\cal K}_H^\TT\circ\iota_\pi=(-1)^{k-j}{\cal K}_{\rm Aut}^T\circ\iota_{\rm dR}^-$$
\end{coro}
\begin{proof}
Si on applique les deux membres \`a $v\in m_{\eet}(\pi)\otimes\pi^{\rm alg}$,
on obtient des fonctions sur $T\times\Aidu$ (le membre de gauche est une fonction
sur $\TT\times\Aidu$ mais se factorise \`a travers $T\times\Aidu$ sur l'image de $\iota_\pi$).
Le lien entre $q$-d\'eveloppement et mod\`ele de Kirillov, combin\'e avec le th.\,\ref{Ybato21},
montre que ces fonctions co\"{\i}ncident sur $\{1_T\}\times\Aidu$.  On conclut en remarquant
que $T$ agit par le m\^eme caract\`ere (celui permettant d'\'ecrire $m_{\eet}(\pi)$ comme une
sp\'ecialisation de $\rho_T$) sur les images de $\iota_\pi$ et~$\iota_{\rm dR}^-$.
\end{proof}

\Subsection{Une version explicite de la factorisation d'Emerton}\label{Yqq1}
\subsubsection{Le cas g\'en\'erique}
On dit que {\it ${\goth m}$ est g\'en\'erique} (resp.~{\it tr\`es g\'en\'erique})
si ${\goth m}$ est non-eisenstein, et
si la restriction de $\overline\rho_{\goth m}$ \`a $G_{\Q_p}$
n'est pas de la forme $\chi\otimes\matrice{1}{0}{0}{1}$
(resp.~pas de la forme $\chi\otimes\matrice{1}{*}{0}{1}$).

\begin{theo}\phantomsection\label{Ycano110}
{\rm (i)}
Si ${\goth m}$ est g\'en\'erique,
il \index{lat@\lambdat}existe un unique isomorphisme 
$$\iota_T:\rho_T\otimes_T\Pi(\rho_T^\diamond)\overset{\sim}{\to}
H^1[\rho_T]$$
de $T[G_{\Q}\times \GG(\Ai)]$-modules,
tel que ${\cal K}_H^\GG\circ\iota_T={\cal K}_{\rm Aut}^\GG$.

{\rm (ii)} Pour tout $x\in{\cal X}^{\rm cl}$, 
avec $\rho_x=m_{\eet}(\pi)$ et $\pi$ de poids $(k+2,j+1)$, on a un diagramme commutatif
$$\xymatrix@C=.8cm@R=.5cm{
(H^1[\rho_T])[\tfrac{1}{p}]& \rho_T\otimes_T\Pi(\rho_T^\diamond)[\tfrac{1}{p}]\ar[l]_-{\iota_T}\\
m_{\eet}(\pi)\otimes\pi^{\rm alg}\ar[u]^-{(-1)^{k-j}\iota_{\pi}}\ar[r]^-{\iota_{\rm dR}^-}&
m_{\eet}(\pi)\otimes\Pi(m_{\eet}^\dual(\pi))\ar@{^{(}->}[u]}$$
\end{theo}
\begin{proof}
D'apr\`es le th.\,\ref{cdn3}
et la prop.\,\ref{glob3},
${\cal K}_H^\GG$ et ${\cal K}_{\rm Aut}^\GG$ induisent des isom\'etries
sur leurs images respectives; en particulier, ces images sont ferm\'ees puisque
les espaces de d\'epart sont complets.

Par ailleurs, il r\'esulte du cor.\,\ref{Ybato221}
que, si $0\leq j\leq k$, les images des vecteurs localement alg\'ebriques de poids $(k+2,j+1)$
de $H^1[\rho_T]$ et $\rho_T\otimes_T\Pi(\rho_T^\diamond)$
par ${\cal K}_H^\GG$ et ${\cal K}_{\rm Aut}^\GG$
sont les m\^emes. Comme ces vecteurs sont denses dans les deux espaces
(pour $H^1[\rho_T]$ c'est d\'ej\`a le cas en se restreignant \`a $j=k=0$ ou aux vecteurs $\GG(\Z_p)$-alg\'ebriques,
cf.~prop.\,\ref{cen5}, et pour $\Pi(\rho_T^\diamond)$, cela r\'esulte du th.\,\ref{intox2} et de la zariski-densit\'e
des points classiques), et comme les images
de ${\cal K}_H^\GG$ et ${\cal K}_{\rm Aut}^\GG$
sont ferm\'ees, on en d\'eduit que ces images 
sont les m\^emes. 

On peut (et doit) donc d\'efinir $\iota_T$ comme \'etant 
$({\cal K}_{H}^\GG)^{-1}\circ {\cal K}_{\rm Aut}^\GG$, o\`u $({\cal K}_{H}^\GG)^{-1}$
est l'inverse de ${\cal K}_{H}^\GG$
sur l'image de ${\cal K}_{H}^\GG$.
Ceci prouve le (i).

Le (ii) est une r\'e\'ecriture du cor.\,\ref{Ybato221} (en utilisant le fait que 
${\cal K}_H^\GG$ et ${\cal K}_{\rm Aut}^\GG$ sont obtenus en rendant $\GG(\Ai)\times G_\Q$-invariants
${\cal K}_H^\TT$ et ${\cal K}_{\rm Aut}^T$).
\end{proof}

\begin{rema}\phantomsection\label{Ycorr1}
Le th.\,\ref{Ycano110}, coupl\'e avec la construction de $\Pi(\rho_T^\diamond)$
fournit une factorisation de la cohomologie compl\'et\'ee.
Qu'une telle factorisation puisse exister est une intuition d'Emerton~\cite{Em08}.
Les diff\'erences entre les r\'esultats d'Emerton et les notres sont les suivants:

$\bullet$ L'action de $\GG(\Ai)$ diff\`ere de $g\mapsto {^tg}^{-1}$, ce qui fait que
nous obtenons $\rho_T\otimes\Pi(\rho_T^\diamond)$ au lieu de $\rho_T\otimes\Pi(\rho_T)$.

$\bullet$  Emerton obtient un isomorphisme unique \`a isomorphisme pr\`es de chacun des deux membres
alors que le notre est uniquement d\'etermin\'e (une des raisons qui rendent cette unicit\'e
possible est notre choix d'action de $\GG(\Ai)$).

$\bullet$ Les objets qui apparaissent dans notre factorisation de $\Pi(\rho_T^\diamond)$ sont
les duaux de ceux d'Emerton (par exemple, son $\Pi_p(\rho_T^\diamond)$ est sans $T$-torsion alors
que la $T$-torsion est dense dans la notre).  Que l'on obtienne quand-m\^eme le m\^eme objet
est d\^u au fait que, si $A$ et $B$ sont des $T$-modules libres, alors $(A\otimes_TB)^\dual$
peut se d\'ecrire comme \'etant
$A^\diamond\otimes_TB^\dual$ ou $A^\dual\otimes_TB^\diamond$.

$\bullet$ Le th.\,\ref{Ycano110} est un peu plus g\'en\'eral que celui d'Emerton: nous n'avons besoin
que de ce que $\overline\rho_{\goth m}$ soit g\'en\'erique, alors qu'il a besoin
qu'elle soit tr\`es g\'en\'erique.
\end{rema}

\subsubsection{Le cas non-eisenstein}
Si on suppose seulement que ${\goth m}$ est non-eisenstein, on a un r\'esultat un peu plus
faible (i.e.~on est forc\'e d'inverser~$p$).
\begin{theo}\phantomsection\label{NE2}
{\rm (i)}
Si ${\goth m}$ est non-eisenstein,
il \index{lat@\lambdat}existe un unique isomorphisme 
$$\iota_T:\rho_T\otimes_T\Pi(\rho_T^\diamond)[\tfrac{1}{p}]\overset{\sim}{\to}
(H^1[\rho_T])[\tfrac{1}{p}]$$
de $T[G_{\Q}\times \GG(\Ai)]$-modules,
tel que ${\cal K}_H^\GG\circ\iota_T={\cal K}_{\rm Aut}^\GG$.

{\rm (ii)} Pour tout $x\in{\cal X}^{\rm cl}$, 
avec $\rho_x=m_{\eet}(\pi)$ et $\pi$ de poids $(k+2,j+1)$, on a un diagramme commutatif
$$\xymatrix@C=.8cm@R=.5cm{
(H^1[\rho_T])[\tfrac{1}{p}]& \rho_T\otimes_T\Pi(\rho_T^\diamond)[\tfrac{1}{p}]\ar[l]_-{\iota_T}\\
m_{\eet}(\pi)\otimes\pi^{\rm alg}\ar[u]^-{(-1)^{k-j}\iota_{\pi}}\ar[r]^-{\iota_{\rm dR}^-}&
m_{\eet}(\pi)\otimes\Pi(m_{\eet}^\dual(\pi))\ar@{^{(}->}[u]}$$
\end{theo}
\begin{proof}
Comme $\Pi(\rho_T^\diamond)[\frac{1}{p}]$ est une limite inductive de repr\'esentations
admissibles de $\GG(\Q_p)$, la
prop.\,\ref{QQ3} ci-dessous implique que l'image de ${\cal K}_{\rm Aut}^\GG$ est ferm\'ee.
Par ailleurs,
${\cal K}_H^\GG$ induit une isom\'etrie sur son image qui est donc, elle aussi, ferm\'ee.
On en d\'eduit, comme dans la preuve du th.\,\ref{Ycano110}, en utilisant la densit\'e
des vecteurs alg\'ebriques et le cor.\,\ref{Ybato221}, que ces images sont les m\^emes.
On peut donc d\'efinir $\lambda_T:(H^1[\rho_T])[\tfrac{1}{p}]\to \rho_T\otimes_T\Pi(\rho_T^\diamond)[\frac{1}{p}]$
en posant $\lambda_T=({\cal K}_{\rm Aut}^\GG)^{-1}\circ {\cal K}_H^\GG$. Alors
$\lambda_T$ est surjective par construction, et pour prouver le (i), il s'agit de prouver
que $\lambda_T$ est injective (auquel cas, on peut poser $\iota_T=\lambda_T^{-1}$).

Il r\'esulte du th.\,\ref{igu8.2} que le noyau de $\lambda_T$ est support\'e sur
le lieu des ${\goth p}$ tels que l'op\'erateur de Sen de $\rho_{\goth p}$
soit scalaire (la condition $\rho_{\goth p}$ irr\'eductible est assur\'ee par l'hypoth\`ese
que ${\goth m}$ est non-eisenstein). Or il r\'esulte des r\'esultats de Pan~\cite[th.\,6.2.2]{pan2}
que ce lieu est de codimension~$\geq 2$ (c'est l'ensemble des twists de repr\'esentations
associ\'ees \`a des formes modulaires de poids 1, non ramifi\'ees en dehors de $S$).

Maintenant, on a une suite exacte
$$0\to {\rm Ker}\,\lambda_T\to (H^1[\rho_T])[\tfrac{1}{p}]
\to \rho_T\otimes\Pi(\rho_T^\diamond))[\tfrac{1}{p}]\to 0$$
gr\^ace \`a la surjectivit\'e de $\lambda_T$.
On peut dualiser cette suite exacte, et \'ecrire
$(H^1[\rho_T]^\dual)[\frac{1}{p}]$ comme une extension d'un $T[\frac{1}{p}]$-module de torsion 
$M_{\rm Sen}$, dont le support
(lieu o\`u l'op\'erateur de Sen est scalaire) est de codimension~$\geq 2$,
 par un $T[\frac{1}{p}]$-module limite
inductive de modules de la forme\footnote{Si ${\cal X}$ contient un point $p$-pathologique,
il faut remplacer $T^\N$ par un sous-objet dont le quotient est de type fini sur $T$
et support\'e sur le lieu $p$-pathologique. Comme ce lieu est disjoint du lieu
o\`u l'op\'erateur de Sen est scalaire, cela ne change rien \`a ce qui suit.}
$T^\N[\frac{1}{p}]$
(quitte \`a remplacer $T$ par une de ses composantes connexes).
Gr\^ace au th\'eor\`eme de l'image ouverte, on peut trouver un $T$-r\'eseau
de $((H^1[\rho_T])[\tfrac{1}{p}])^\dual$ extension par $T^\N$
d'un $T$-r\'eseau $M^+_{\rm Sen}$ de $M_{\rm Sen}$.
Comme le support est de codimension~$\geq 2$, une telle extension est scind\'ee,
et $M_{\rm Sen}$ s'identifie au sous-module de $T[\frac{1}{p}]$-torsion de 
$((H^1[\rho_T])[\tfrac{1}{p}])^\dual$; le scindage commute donc aux actions
de $G_{\Q}$ et $\GG(\Ai)$.

Il en r\'esulte que la suite exacte ci-dessus est scind\'ee, en tant que
$T[G_{\Q}\times \GG(\Ai)]$-module topologique. Mais ${\rm Ker}\,\lambda_T$ ne contient pas
de vecteurs localement alg\'ebriques non nuls (car les repr\'esentations de $G_\Q$ qui apparaissent
dans le socle sont des repr\'esentations associ\'ees \`a des (twists) de formes modulaires de poids~$1$,
et pas de poids~$\geq 2$), et comme les vecteurs localement alg\'ebriques sont denses
dans $(H^1[\rho_T])[\tfrac{1}{p}]$, on en d\'eduit que ${\rm Ker}\,\lambda_T=0$,
ce qui prouve le (i).

Le (ii) se prouve comme pour le th.\,\ref{Ycano110}.
\end{proof}

\begin{prop}\phantomsection\label{QQ3}
Soient $W_1$, $W_2$ des $L$-banachs munis d'actions continues de $\GG(\Q_p)$ et soit $f:W_1\to W_2$,
$\GG(\Q_p)$-\'equivariante, continue. Si $W_1$ est admissible, alors $f(W_1)$ est ferm\'ee
dans $W_2$.
\end{prop}
\begin{proof}
Quitte \`a quotienter $W_1$ par le noyau de $f$ et remplacer $W_2$ par l'adh\'erence de l'image de $f$,
 on peut supposer que $f$ et $f^*:W_2^\dual\to W_1^\dual$
 sont injectives (et donc, par dualit\'e, gr\^ace au th\'eor\`eme de Hahn-Banach, sont d'images denses).
Maintenant,
comme $W_1$ est admissible, tout sous-$\GG(\Q_p)$-espace
$X$ de $W_1^\dual$ est ferm\'e (l'intersection $X_0$ avec la boule
unit\'e $W_{1,0}^\dual$ de $W_1^\dual$ est un $\O_L[[K]]$-module de type fini ($K$ est un sous-groupe ouvert
compact de $\GG(\Q_p)$) car $W_{1,0}^\dual$ l'est par admissibilit\'e de~$W_1$ et
$\O_L[[K]]$ est noeth\'erienne; il s'ensuit que $X_0$ est compact
et donc $X$ est complet et donc ferm\'e).
Donc $f^\dual$ est surjective et est un isomorphisme de duaux de banachs
par le th\'eor\`eme de l'image ouverte. Il s'ensuit que $f$ est un isomorphisme de banachs, ce qui
permet de conclure.
\end{proof}

\begin{rema}\phantomsection\label{Ycano110.5}
Si $x$ est classique avec $\rho_x=m_{\eet}(\pi)$, et si la restriction de $\rho_x$ \`a $G_{\Q_p}$ n'est pas une
somme de deux caract\`eres, il r\'esulte du (ii) du th.\,\ref{NE2} (ou~\ref{Ycano110})
que la restriction de $\iota_T$ \`a $\rho_x\otimes\Pi_S(\rho_x^\dual)$
est, au signe pr\`es, l'unique (rem.\,\ref{passomme1})
extension de l'application $\iota_{\pi}$ de la rem.\,\ref{YEU31.5}. Si $\rho_x$
est une somme de deux caract\`eres, alors cette restriction fournit
une extension privil\'egi\'ee de $\iota_{\pi}$.
\end{rema}

%
%
%
%

\part{Factorisation du syst\`eme de Beilinson-Kato}
\section{Symboles modulaires et \'el\'ement de Kato}\label{Symbo1}
Ce chapitre est consacr\'e \`a la preuve du th.\,\ref{intro2} comparant
l'\'el\'ement de Kato et le symbole modulaire $(0,\infty)$ pour une forme modulaire de poids~$\geq 2$
(ou plus g\'en\'eralement, une repr\'esentation cohomologique de $\GG(\A^{]\infty[})$).
La preuve (th.\,\ref{zk2}) 
n'est compl\`ete que dans le cas o\`u la repr\'esentation galoisienne associ\'ee
a une restriction \`a $G_{\Q_p}$ qui est irr\'eductible. Le cas g\'en\'eral
s'en d\'eduit par prolongement analytique, cf.\,\no\ref{geni9.4}.

\Subsection{$(0,\infty)$ dans le dual de la cohomologie compl\'et\'ee}\label{symb1}
On peut associer au {\it symbole modulaire}
$(a,b)$ une forme lin\'eaire continue sur
$H^1_c(\GG(\Q),{\cal C}(\GG(\A),\Q_p))$, encore not\'ee $(a,b)$: 
si $\gamma\mapsto \phi_\gamma$ est un $1$-cocycle
sur $\GG(\Q)$ \`a valeurs dans ${\cal C}(\GG(\A),\Q_p)$ et identiquement nul
sur $\BB(\Q)$, on pose 
$$\langle(a,b),\phi\rangle=(\phi_{\alpha_b}-\phi_{\alpha_a})(1),\quad
{\text{o\`u $\alpha_a,\alpha_b\in\GG(\Q)$ v\'erifient $\alpha_a(\infty)=a$, $\alpha_b(\infty)=b$,}}$$
o\`u $1=\matrice{1}{0}{0}{1}\in \GG(\A)$.
C'est donc la compos\'ee du \index{inf1@\oi}symbole modulaire $(a,b)$ du (ii) de la rem.~\ref{ES4}
avec la masse de Dirac en $1$.
En particulier,
$$\langle (0,\infty), (\gamma\mapsto \phi_\gamma)\rangle=-\phi_S(1),\quad
{\text{o\`u $S=\matrice{0}{1}{-1}{0}$.}}$$

\begin{rema}\phantomsection\label{ES10}
{\rm (invariance par extension des scalaires)}

Si $v$ est une place de $\Q$, et si $L$ est une extension finie de $\Q_v$,
alors $(0,\infty)$ sur $H^1_c(\GG(\Q),{\rm LP}(\GG(\A),L))$ est obtenu par
$L$-lin\'earit\'e \`a partir de $(0,\infty)$ sur $H^1_c(\GG(\Q),{\rm LP}(\GG(\A),\Q))$.
En particulier, la restriction de $(0,\infty)$ sur $H^1_c(\GG(\Q),{\rm LP}(\GG(\A),L))$
\`a $H^1_c(\GG(\Q),{\rm LP}(\GG(\A),\Q))$ ne d\'epend pas de $L$, ce qui permet
d'interpr\'eter $p$-adiquement des calculs sur $\C$.
\end{rema}

\begin{rema}\phantomsection\label{ES10.1}
{\rm (invariance par multiplication par un caract\`ere)}  

Soit $X(\GG(\A))$ un espace de fonctions sur $\GG(\A)$,
stable par translations \`a
gauche par $\GG(\Q)$ et \`a droite par $\GG(\A)$.
Si $\eta$ est un caract\`ere continu de $\A^\dual/\Q^\dual$ 
(\`a valeurs dans $\C^\dual$
ou $L^\dual$ suivant la situation), et si la multiplication par $\eta\circ\det$
stabilise $X(\GG(\A))$, elle commute \`a l'action de $\GG(\Q)$ et donc elle induit
un isomorphisme de $H^1_c(\GG(\Q),X(\GG(\A)))$.  Comme $\eta\circ\det(1)=1$,
on voit que $$\langle (0,\infty), (\gamma\mapsto (\eta\circ\det)\phi_\gamma)\rangle=
\langle (0,\infty), (\gamma\mapsto \phi_\gamma)\rangle,$$
i.e.~{\it $(0,\infty)$ est invariante par multiplication par un caract\`ere de
$\A^\dual/\Q^\dual$}.
\end{rema}

\begin{lemm}\phantomsection\label{np=p}
Si $a,b\in\piqp(\Q)$ et si $\gamma\in \GG(\Q)_+$, alors
$$\gamma^{]\infty[}\star(a,b)=(\gamma\cdot a,\gamma\cdot b).$$
En particulier,
$$\matrice{0}{-1}{1}{0}^{]\infty[}\star(0,\infty)=-(0,\infty)\quad{\rm et}\quad
\matrice{a}{0}{0}{1}^{]\infty[}\star(0,\infty)=(0,\infty),\ {\text{si $a\in\Q_+^\dual$.}}$$
\end{lemm}
\begin{proof}
Soit $\gamma\mapsto\phi_\gamma$ un $1$-cocycle sur $\GG(\Q)$, nul sur $\BB(\Q)$, \`a valeurs
dans ${\cal C}(\GG(A),L)$, et soit $\phi$ la classe de cohomologie \`a support compact
correspondante. 
On en d\'eduit:
\begin{align*}
\langle \gamma^{]\infty[}\star(a,b),\phi\rangle & =(\phi_{\alpha_b}-\phi_{\alpha_a})(1,(\gamma^{]\infty[})^{-1})
=(\phi_{\alpha_b}-\phi_{\alpha_a})(\gamma_\infty^{-1},(\gamma^{]\infty[})^{-1})\\
&=\big(\gamma*(\phi_{\alpha_b}-\phi_{\alpha_a})\big)(1)=\langle (\gamma(a),\gamma(b)),\phi\rangle
\end{align*}
(La premi\`ere \'egalit\'e vient de la formule $\langle g\star\mu,\phi\rangle=\langle \mu,g^{-1}\star\phi\rangle$,
de la d\'efinition de l'action $\star$ et de la formule pour $\langle(a,b),-\rangle$, la seconde
vient de ce que $\gamma_\infty^{-1}\in\GG(\R)_+$,  la troisi\`eme r\'esulte
de la d\'efinition de l'action $*$ et la derni\`ere de l'identit\'e
$\gamma*(\phi_{\alpha_b}-\phi_{\alpha_a})=
\phi_{\alpha_{\gamma\cdot b}}-\phi_{\alpha_{\gamma\cdot a}}$ qui est une cons\'equence du lemme~\ref{ES3}.)
\end{proof}
\begin{coro}\phantomsection\label{p=1}
Si $\phi\in H^1_c(\GG(\Q),{\cal C}(\GG(\A),L))$
et si $\ell\in{\cal P}$, alors
$$\big\langle \matrice{\ell}{0}{0}{1}_\ell\star (0,\infty),\phi\big\rangle
=\big\langle (0,\infty),\matrice{\ell}{0}{0}{1}^{]\ell,\infty[}\star\phi\big\rangle,$$
En particulier, si $\matrice{p}{0}{0}{1}^{]\infty,p[}\star\phi=\lambda\,\phi$,
avec $\lambda\in L$, alors
$$\big\langle \matrice{p}{0}{0}{1}_p\star(0,\infty),\phi\rangle=
\lambda\,\langle (0,\infty),\phi\big\rangle,$$
et si $\matrice{\ell}{0}{0}{1}^{]\infty,p[}\star\phi^{]p[}=\lambda\,\phi^{]p[}$, alors
$$\big\langle \matrice{\ell}{0}{0}{1}_\ell\star(0,\infty),\phi^{]p[}\otimes\phi_p\rangle=
\lambda\,\langle \matrice{\ell^{-1}}{0}{0}{1}_p\star(0,\infty),\phi^{]p[}\otimes\phi_p\big\rangle.$$
\end{coro}
\begin{proof}
La premi\`ere formule r\'esulte 
de ce que $\matrice{\ell}{0}{0}{1}^{]\ell,\infty[}\matrice{\ell}{0}{0}{1}_\ell
=\matrice{\ell}{0}{0}{1}^{]\infty[}$,
de ce que
$\big\langle \matrice{\ell}{0}{0}{1}^{]\ell,\infty[}\star\mu,
\matrice{\ell}{0}{0}{1}^{]\ell,\infty[}\star\phi\big\rangle=\langle\mu,\phi\rangle$, 
et du lemme~\ref{np=p}. 
La seconde en est une cons\'equence imm\'ediate,
et la troisi\`eme est une cons\'equence de la formule suivante qui est aussi une
cons\'equence imm\'ediate de la premi\`ere:
$$\phantom{XXXXXX}\big\langle \matrice{\ell}{0}{0}{1}_\ell\star(0,\infty),\phi^{]p[}\otimes\phi_p\rangle=
\lambda\,\langle (0,\infty),\phi^{]p[}\otimes\big(\matrice{\ell}{0}{0}{1}_p\star\phi_p\big)\big\rangle.
\phantom{XXXXXX}\qedhere$$
\end{proof}

\Subsection{$(0,\infty)$ et valeurs sp\'eciales de fonctions~$L$}\label{symb0}
\subsubsection{$(0,\infty)$ et application d'Eichler-Shimura}\label{es1}
Soit $$\phi\in H^0(\GG(\Q),{\cal A}^+_{\rm par}d\tau \otimes W_{k,j})$$
On peut donc \'ecrire $\phi$ sous la forme
$$\phi(\tau,g^{]\infty[})=\big(\sum_{\ell =0}^k\phi_\ell (\tau,g^{]\infty[})
\tfrac{e_1^\ell e_2^{k-\ell }}{(e_1\wedge e_2)^j}\big)\,d\tau,$$
o\`u $\tau\mapsto\phi_\ell (\tau,g^{]\infty[})$
est holomorphe, \`a d\'ecroissance rapide, sur ${\cal H}$.
Si $\check v\in W_{k,j}^\dual$, l'application d'Eichler-Shimura
fournit 
$$\iota_{\rm ES}(\phi\otimes \check v)=\langle\check v,\iota_{\rm ES}(\phi)\rangle\in H^1_c(\GG(\Q),{\rm LP}(\GG(\A),\C)).$$
On fait agir $\matrice{-1}{0}{0}{1}_\infty$ par $\big(\matrice{-1}{0}{0}{1}_\infty,1\big)\in
\GG(\A)\times \GG(\C)$ sur $H^1_c(\GG(\Q),{\rm LP}(\GG(\A),\C))$ et pas par
$\big(\matrice{-1}{0}{0}{1}_\infty,\matrice{-1}{0}{0}{1}\big)$ pour que le r\'esultat
soit compatible avec l'action en $p$-adique (sur $H^1_c(\GG(\Q),{\cal C}(\GG(\A),L))$).
\begin{prop} \label{ES11}
On a\footnote{Avec $\phi(iy,g^{]\infty[})=\big(\sum_{\ell =0}^k\phi_\ell (iy,g^{]\infty[})
\tfrac{e_1^\ell e_2^{k-\ell }}{(e_1\wedge e_2)^j}\big)\,i\,dy$.}
\begin{align*}
\langle (0,\infty),\iota_{\rm ES}(\phi\otimes \check v)\rangle=&\ 
\int_0^{+\infty} 
\langle\check v,\phi(iy,1^{]\infty[})\rangle\\
\langle (0,\infty),\matrice{-1}{0}{0}{1}_\infty\star
\iota_{\rm ES}(\phi\otimes \check v)\rangle=&\ 
\int_0^{+\infty} 
\big\langle\matrice{-1}{0}{0}{1}*\check v,\phi(iy,\matrice{-1}{0}{0}{1}^{]\infty[})\big\rangle
\end{align*}
\end{prop}
\begin{proof}
On note $\tau\mapsto\phi^{(-1)}(\tau,g^{]\infty[})$ la primitive de $\tau\mapsto\phi(\tau,g^{]\infty[})$
s'annulant en $\pm i\infty$.
L'invariance de $\phi$ par $\GG(\Q)$ se traduit par l'existence de constantes
$c^{\pm}_\gamma(\phi,g^{]\infty[})$ telles que l'on ait
$\gamma * \phi^{(-1)}=\phi^{(-1)}+c^{\pm}_\gamma(\phi,g^{]\infty[})$ sur ${\cal H}^\pm$.
En faisant tendre $\tau$ vers $\gamma\cdot\pm i\infty$, 
on obtient $\phi^{(-1)}(\gamma\cdot\pm i\infty,g^{]\infty[})+c^{\pm{\rm sign}(\gamma)}_\gamma(\phi,g^{]\infty[})=0$,
o\`u ${\rm sign}(\gamma)={\rm sign}(\det\gamma)$,
et donc 
$$c^{\pm}_\gamma(\phi,g^{]\infty[})=\int_{\gamma\cdot\pm i\infty}^{\pm{\rm sign}(\gamma) i\infty}
\phi(\tau).$$
Ceci fournit une fonction $c_\gamma(\phi)\in {\rm LC}(\GG(\A))\otimes W_{k,j}$ (avec $c_\gamma(\phi)(g_\infty,g^{]\infty[})=
c^{\pm}_\gamma(\phi,g^{]\infty[})$, si ${\rm sign}(g_\infty)=\pm$),
et $\gamma\mapsto c_\gamma(\phi)$ est un $1$-cocycle sur $\GG(\Q)$ dont la classe dans 
$H^1(\GG(\Q), {\rm LC}(\GG(\A))\otimes W_{k,j})$ est $\iota_{\rm ES}(\phi)$, et qui est
identiquement nul sur $\BB(\Q)$ (puisqu'on int\`egre
de $\pm i\infty$ \`a $\pm i\infty$).
Autrement dit, on a construit l'\'el\'ement de $Z^1(\GG(\Q),\BB(\Q),M)$, avec $M={\rm LC}(\GG(\A))\otimes W_{k,j}$,
repr\'esentant $\iota_{\rm ES}(\phi)$.
Le r\'esultat s'en d\'eduit en revenant \`a la d\'efinition de $(0,\infty)$: il faut \'evaluer
$-\langle\check v,c_S(\phi)\rangle$ en $g=1$.

Pour prouver la seconde formule, on utilise les formules
\begin{align*}
\matrice{-1}{0}{0}{1}_\infty\star(\psi\otimes\check v)
&=\big(\matrice{-1}{0}{0}{1}_\infty\star\psi\big)\otimes \check v\\
\matrice{-1}{0}{0}{1}_\infty\star\iota_{\rm ES}(\phi)
&=\iota_{\rm ES}\big(\matrice{-1}{0}{0}{1}_\infty\star\phi\big)\\
\langle\check v,\phi(\overline\tau,1^{]\infty[})\rangle&=
\big\langle \matrice{-1}{0}{0}{1}*\check v,\phi\big(-\overline\tau,\matrice{-1}{0}{0}{1}^{]\infty[}\big)\big\rangle
\end{align*}
la premi\`ere formule venant du choix fait pour l'action de $\matrice{-1}{0}{0}{1}_\infty$,
la derni\`ere formule venant de ce que
$\matrice{-1}{0}{0}{1}*\phi=\phi$, 
l'action de $\matrice{-1}{0}{0}{1}$ sur $e_1,e_2$ \'etant transf\'er\'ee
sur~$\check v$.
\end{proof}

\subsubsection{Lien avec les valeurs sp\'eciales de fonctions~$L$}\label{es2}
Si $$\phi=\sum_{\ell=0}^k\phi_\ell
\otimes\tfrac{(e_2^\dual)^{k-\ell}(e_1^\dual)^\ell}{(k-\ell)!\,(e_1^\dual\wedge e_2^\dual)^j}
\in M_{k+2,j+1}(\C)\otimes W_{k,j}^\dual,$$ on pose
${\rm LC}={\rm LC}(\Aidu,\C)$, \index{K@\KKK}et
\begin{align*}
{\cal K}(\phi,u,X)= &\ \sum_{\ell=0}^k{\cal K}(\phi_\ell,u)X^{\ell-j}
\in\frac{{\rm LC}[X,X^{-1}]}
{X^{k+1-j}{\rm LC}[X]}\\
L(\phi,s)=L({\cal K}(\phi),s)= &\ \sum_{n\in\Q_+^\dual}{\cal K}(\phi,n^{]\infty[},2i\pi n)\,n^{-s}=
\sum_{\ell=0}^k (2i\pi)^{\ell-j}L(\phi_\ell,s+j-\ell)
\end{align*}
\index{L@\fonctionL}o\`u $L(\phi_\ell,s)=\sum_{n\in\Q_+^\dual}{\cal K}(\phi_\ell,n^{]\infty[})\,n^{-s}$.
Les actions de $\GG(\A^{]\infty[})$ sur $M_{k+2,j+1}(\C)$ et de $\GG(\C)$ sur
$W_{k,j}^\dual$ deviennent, en restriction au mirabolique:
\begin{align*}
{\cal K}(\matrice{a_\infty}{b_\infty}{0}{1}\star\phi,u,X)=&\ e^{b_\infty X}
\,{\cal K}(\phi,u,a_\infty X)\\
{\cal K}(\matrice{a^{]\infty[}}{b^{]\infty[}}{0}{1}\star\phi,u,X)=&\ {\bf e}^{]\infty[}(b^{]\infty[} u)
\,{\cal K}(\phi,a^{]\infty[}u,X)
\end{align*}
Si ${\cal K}\in \frac{{\rm LC}[X,X^{-1}]}
{X^{k+1-j}{\rm LC}[X]}$, on d\'efinit
${\cal K}\circ (-1)$ par la formule
$$({\cal K}\circ (-1))(u,X)={\cal K}(-u,-X).$$

\begin{prop}\phantomsection\label{es3}
Si $\phi\in M_{k+2,j+1}(\C)\otimes W_{k,j}^\dual$, 
\begin{align*}
\langle(0,\infty),\iota_{\rm ES}(\phi)\rangle=&\ \tfrac{1}{(-2i\pi)^{k-j+1}}
L({\cal K}(\phi),0)\\
\langle(0,\infty),\matrice{-1}{0}{0}{1}_\infty\star
\iota_{\rm ES}(\phi)\rangle=&\ \tfrac{1}{(-2i\pi)^{k-j+1}}
L({\cal K}(\phi)\circ (-1),0)
\end{align*}
\end{prop}
\begin{proof}
\'Ecrivons 
$$\phi=\sum_{\ell=0}^k\phi_\ell\otimes\tfrac{(e_2^\dual)^{k-\ell}(e_1^\dual)^\ell}{(k-\ell)!\,(e_1^\dual\wedge e_2^\dual)^j}
\quad{\rm avec}\ \phi_\ell=\phi_{\ell,0}\otimes
\tfrac{(\tau e_2-e_1)^k}{(e_1\wedge e_2)^j}d\tau.$$
Comme $\big\langle\tfrac{(\tau e_2-e_1)^k}{(e_1\wedge e_2)^j},
\tfrac{(e_2^\dual)^{k-\ell}(e_1^\dual)^\ell}{(e_1^\dual\wedge e_2^\dual)^j}\big\rangle=
(-1)^\ell\tau^{k-\ell}$,
la prop.\,\ref{ES11} fournit la formule
$$\langle(0,\infty),\iota_{\rm ES}(\phi)\rangle=
\int_0^\infty\sum_{\ell=0}^k\tfrac{(-1)^\ell}{(k-\ell)!}\phi_{\ell,0}(iy,1^{]\infty[})(iy)^{k-\ell}\,i\,dy,$$
ce qui est la valeur en $s=1$ de
$$\sum_{\ell=0}^k\tfrac{i^{k+\ell+1}}{(k-\ell)!}
\int_0^\infty\phi_{\ell,0}(iy,1^{]\infty[})y^{s+k-\ell}\,\tfrac{dy}{y}.$$
Maintenant\footnote{Le $k+1-j$ est en fait $(k+2)-(j+1)$.},
$\phi_{\ell,0}(iy,1^{]\infty[})=\sum_{n\in\Q_+^\dual}n^{k+1-j}
{\cal K}(\phi_\ell,n^{]\infty[})e^{-2\pi ny}$, et donc
\begin{align*}
\int_0^\infty\phi_{\ell,0}(iy,1^{]\infty[})y^{s+k-\ell}\,\tfrac{dy}{y}=&\ 
\tfrac{\Gamma(s+k-\ell)}{(2\pi)^{s+k-\ell}}
\sum_{n\in\Q_+^\dual}n^{1+\ell-j-s}{\cal K}(\phi_\ell,n^{]\infty[})\\
=&\ \tfrac{\Gamma(s+k-\ell)}{(2\pi)^{s+k-j}}(2\pi)^{\ell-j}
L(\phi_\ell,s+j-\ell-1)
\end{align*}
Donc
\begin{align*}
\langle(0,\infty),\iota_{\rm ES}(\phi)\rangle=&\ 
\sum_{\ell=0}^k\tfrac{i^{k+\ell+1}}{(k-\ell)!}
\tfrac{\Gamma(1+k-\ell)}{(2\pi)^{1+k-j}}(2\pi)^{\ell-j}
L(\phi_\ell,j-\ell)\\
=&\ \tfrac{1}{(-2i\pi)^{1+k-j}}\sum_{\ell=0}^k
(2i\pi)^{\ell-j} L(\phi_\ell,j-\ell)\\=&\ 
\tfrac{1}{(-2i\pi)^{1+k-j}} L(\phi,0)
\end{align*}
Cela prouve la premi\`ere formule.  La seconde s'en d\'eduit en utilisant
la seconde formule de la prop.\,\ref{ES11} et le fait que
$\matrice{-1}{0}{0}{1}*\frac{(e_2^\dual)^{k-\ell}(e_1^\dual)^\ell}{(e_1^\dual\wedge e_2^\dual)^j}=
(-1)^{\ell-j}\frac{(e_2^\dual)^{k-\ell}(e_1^\dual)^\ell}{(e_1^\dual\wedge e_2^\dual)^j}$.
\end{proof}
\begin{coro}\phantomsection\label{es4}
Si 
$${\cal K}^\pm(\phi)=\tfrac{1}{2}\big({\cal K}(\phi)\pm({\cal K}(\phi)\circ (-1))\big),$$ 
alors
$$\langle(0,\infty),\iota_{\rm ES}^\pm(\phi)\rangle= \tfrac{1}{(-2i\pi)^{k-j+1}}
L({\cal K}^\pm(\phi),0).$$
\end{coro}

\begin{rema}\phantomsection\label{es5}
Posons $f_\ell=\sum_{n\in\Q_+^\dual}n^{k+1-j}{\cal K}(\phi_\ell,n^{]\infty[})
e^{2i\pi n\tau}$.  Alors $f_\ell$ est une forme modulaire classique,
et on a $L(\phi_\ell,s)=L(f_\ell,s+k+1-j)$.
Il s'ensuit que
$$L(\phi,0)=\sum_{\ell=0}^k(2i\pi)^{\ell-j} L(f_\ell,k+1-\ell),$$
ce qui exprime $L(\phi,0)$ comme une somme de valeurs critiques
de fonctions $L$ de formes modulaires classiques.
\end{rema}

\Subsection{L'\'el\'ement de Kato}\label{como13}
\subsubsection{T\'etrapilectomie}\label{como12.5}
Soit $V$ une $L$-repr\'esentation de\,Rham de $G_{\Q_p}$.
Si $i\in\Z$, on \index{twi@\twis}note $V\otimes\cyp ^i$ l'espace $V$ muni de l'action de
$G_{\Q_p}$ multipli\'ee par $\cyp ^i$.  La \index{Vi@\Vi}repr\'esentation
$$V(i):=V\otimes\Q_p(i)=V\otimes \zeta_{\rm B}^i$$ 
est isomorphe \`a $V\otimes\cyp ^i$, mais
nous n'identifierons pas ces deux repr\'esentations: en particulier,
\begin{align*}
D_{\rm dR}(V\otimes\cyp ^i)=&\ t^{-i}D_{\rm dR}(V)\subset \bdr\otimes V\\
D_{\rm dR}(V(i))=&\ D_{\rm dR}(V)\otimes t^{-i}\zeta_{\rm B}^i=
D_{\rm dR}(V)\otimes \zeta_{\rm dR}^i
\end{align*}

Si $\mu\in H^1(G_{\Q_p},\Lambda\otimes V)$, alors $\int_{\Z_p^\dual}x^i\mu\in
H^1(G_{\Q_p},V\otimes\cyp ^i)$ et donc
$\exp^*(\int_{\Z_p^\dual}x^i\mu)\in t^{-i}D_{\rm dR}(V)$
(plus pr\'ecis\'ement, $t^{-i}{\rm Fil}^iD_{\rm dR}(V)$).

Utiliser le lemme~\ref{sou3} avec $G'=X'=\{1\}$, $G_p=\Z_p^\dual$,
$H_1=G_{\Q_p}$, $H_2=\Z_p^\dual$, $\iota_1=\cyp $, $\iota_2={\rm id}$,
$W=\Q_p(i)$, fournit un isomorphisme
$\Lambda$-\'equivariant
$$\xymatrix@C=1.5cm{
H^1(G_{\Q_p},\Lambda\otimes V)\otimes \Q_p(i)\ar[r]^-{\sim}_-{\mu\mapsto\mu\otimes \zeta_{\rm B}^i}
& H^1(G_{\Q_p},\Lambda\otimes V\otimes \Q_p(i))}$$
et on a, si $\phi\in{\rm LC}(\Z_p^\dual,L(\bmu_{p^\infty}))^{\Z_p^\dual}$ et $\ell\in\Z$,
$$\int_{\Z_p^\dual}\phi(x)x^\ell\,\mu\otimes \zeta_{\rm B}^i=
\big(\int_{\Z_p^\dual}x^{i+\ell}\phi(x)\mu\big)\otimes \zeta_{\rm B}^i\in H^1(G_{\Q_p},V(i)\otimes\cyp ^\ell).$$
(Si $\sigma\mapsto \mu_\sigma$ est un $1$-cocycle repr\'esentant $\mu$, alors
$\big(\int_{\Z_p^\dual}x^{i+\ell}\phi(x)\mu\big)\otimes \zeta_{\rm B}^i$
est repr\'esent\'e par le $1$-cocycle $\sigma\mapsto
\big(\int_{\Z_p^\dual}x^{i+\ell}\phi(x)\mu_\sigma\big)\otimes \zeta_{\rm B}^i$.)
On a 
$$\exp^*\big(\int_{\Z_p^\dual}\phi(x)x^\ell\,\mu\otimes \zeta_{\rm B}^i\big)\in 
t^{-\ell}D_{\rm dR}(V)\otimes\zeta_{\rm dR}^i.$$

\subsubsection{Les \'el\'ements ${\bf z}_{\rm Iw}(\pi)$ et ${\bf z}(\pi)$}\label{como14}
Soit $f\in M_{k+2,j+1}^{\rm par,\,cl}(\Gamma_0(N),\chi)$ primitive,
et soit $\pi=\pi_{f,j+1}$ (cf.~\S\,\ref{chapi2.5}).

Si $j=0$, on a
$m_{\eet}(\pi)=\rho_{f_\pi^\dual}$.
Dans ce cas, on a le r\'esultat fondamental suivant de Kato,
o\`u l'on a pos\'e $m(f_\pi):=m(\pi)$ et $\iota_{{\rm dR},f_\pi}^+:=\iota_{{\rm dR},\pi}^+$
(cf.~\S\,\ref{como9}, en particulier ${\rm n}^{\rm os}$~\ref{como12.3} et~\ref{como12.42}, ainsi
que \no\ref{wasi1} pour $\Lambda$ et $\exp^\dual$).
\begin{theo}\phantomsection\label{como15}
{\rm (Kato~\cite[th.\,12.5]{Ka4})} 
Il existe un unique
$${\bf z}_{\rm Iw}(f_\pi^\dual)\in m(f_\pi)^\dual\otimes_{\Q(\pi)}H^1( G_{\Q,S},\Lambda\otimes \rho_{f_\pi^\dual})$$ 
qui v\'erifie
$$
\exp^\dual\big(\int_{\Z_p^\dual}\phi(x)(tx)^{k+2- r}\langle{\bf z}_{\rm Iw}(f_\pi^\dual),\gamma\rangle\big)=
\big((2i\pi)^{k- r+1}\Omega_{\pi}^{\pm}(\gamma)L(f_\pi\otimes\phi, r)\big)\iota_{{\rm dR},f_\pi}^+,$$
 si $1\leq  r\leq k+1$, si $\phi\in{\rm LC}(\Z_p^\dual,\Q(\bmu_{p^\infty}))^{\Z_p^\dual}$
v\'erifie $\phi(-x)=\pm(-1)^{k- r-1}\phi(x)$ et si 
$\gamma= \gamma^++\gamma^-
\in m(f_\pi)$, 
 avec $\gamma^{\pm}=\Omega_{\pi}^{\pm}(\gamma)\iota_{\rm ES}^{\pm}\circ\iota_{{\rm dR},f_\pi}^+$.
\end{theo}
\begin{rema}\phantomsection\label{como16}
Le membre de gauche vit dans le monde $p$-adique; celui de droite dans
le monde archim\'edien, mais en fait $(2i\pi)^{k- r+1}\Omega_{\pi}^{\pm}(\gamma)L(f_\pi\otimes\phi, r)\in
\Q(\pi)$.
\end{rema}

Dans le cas $j$ g\'en\'eral, on a 
$$m_{\rm B}(\pi)^\dual=m_{\rm B}(f_\pi)^\dual\otimes \zeta_{\rm B}^{-j}
\quad{\rm et}\quad
H^1(G_{\Q,S},\Lambda\otimes m_{\eet}(\pi))=H^1(G_{\Q,S},\Lambda\otimes \rho_{f_\pi^\dual})\otimes \zeta_{\rm B}^{j},$$
d'o\`u un isomorphisme naturel
$$m(f_\pi)^\dual\otimes_{\Q(\pi)}H^1( G_{\Q,S},\Lambda\otimes \rho_{f_\pi^\dual})\cong
m_{\rm B}(\pi)^\dual\otimes H^1(G_{\Q,S},\Lambda\otimes m_{\eet}(\pi)).$$
On \index{Zaapi@\zpi}note 
\begin{align}\label{ziw0}
{\bf z}_{\rm Iw}(\pi)&\in m_{\rm B}(\pi)^\dual\otimes H^1(G_{\Q,S},\Lambda\otimes m_{\eet}(\pi))\\
{\bf z}(\pi)&\in m_{\rm B}(\pi)^\dual\otimes H^1(G_{\Q,S}, m_{\eet}(\pi))\notag
\end{align}
les images de ${\bf z}_{\rm Iw}(f_\pi^\dual)$ par cet isomorphisme et par
$\Lambda\otimes m_{\eet}(\pi)\to m_{\eet}(\pi)$.
\begin{prop}\phantomsection\label{zk20}
Si $1-j\leq  r\leq k+1-j$, si $\phi\in{\rm LC}(\Z_p^\dual,\Q(\bmu_{p^\infty}))^{\Z_p^\dual}$
v\'erifie $\phi(-x)=\pm(-1)^{ r+1}\phi(x)$, et si $\gamma\in m(\pi)$, alors
\begin{align*}
\exp^\dual\big(\int_{\Z_p^\dual}\phi(x)(tx)^{ r}\langle{\bf z}_{\rm Iw}(\pi),\gamma\rangle\big)
&= \big((2i\pi)^{ r-1}\Omega_{\pi}^{\pm}(\gamma)L(f_\pi\otimes\phi,k+2-j- r)\big)\iota_{{\rm dR},\pi}^+\\
&= \big((2i\pi)^{ r-1}\Omega_{\pi}^{\pm}(\gamma)L(v_\pi^{]p[}\otimes\phi,1- r)\big)\iota_{{\rm dR},\pi}^+\\
\end{align*}
\end{prop}
\begin{proof}
L'\'equivalence des deux \'egalit\'es r\'esulte de la relation
$$
L(f_\pi\otimes\phi,k+2-j- r)= L(v_\pi^{]p[}\otimes\phi,1- r).$$
Par d\'efinition, on a
$$\int_{\Z_p^\dual}\phi(x)(tx)^{ r}\langle{\bf z}_{\rm Iw}(\pi),\gamma\rangle=
t^ r\big(\int_{\Z_p^\dual}\phi(x)x^{ r+j}\langle{\bf z}_{\rm Iw}(f_\pi^\dual),\gamma\otimes
\zeta_{\rm B}^{-j}\rangle
\big)\otimes \zeta_{\rm B}^{j}.$$
D'apr\`es le th.~\ref{como15} (utilis\'e pour $k+2-j- r$ au lieu de $ r$), 
l'image du membre de droite par $\exp^\dual$ est
$$t^ r t^{- r-j}(2i\pi)^{ r+j-1}\Omega_{\pi\otimes|\ |_\A^j}^{\pm(-1)^j}(\gamma\otimes\zeta_{\rm B}^{-j})L(f_\pi\otimes\phi,k+2-j- r)\otimes
\zeta_{\rm B}^{j}\iota^+_{{\rm dR},f_\pi},$$
et on conclut en utilisant les relations (lemme~\ref{twi1} pour la derni\`ere)
\begin{align*}
t^{-j}\zeta_{\rm B}^{j}=\zeta_{\rm dR}^{j},\quad \zeta_{\rm dR}^{j}\iota^+_{{\rm dR},f_\pi}=
\iota_{{\rm dR},\pi}^+,\quad \Omega_{\pi\otimes|\ |_\A^j}^{\pm}(\gamma\otimes\zeta_{\rm B}^{-j})=
(2i\pi)^{-j}\Omega_\pi^{\pm(-1)^j}(\gamma).
\end{align*}
\end{proof}

\begin{rema}\phantomsection\label{como17.1}
D'apr\`es Kato~\cite{Ka4}, la restriction induit une injection\footnote{
Voir le th.\,\ref{KK0} ci-apr\`es pour les r\'esultats de Kato, l'injectivit\'e ci-dessous
s'en d\'eduit comme dans la prop.\,\ref{geni3}.}:
$${\rm loc}_p:m(\pi)^\dual\otimes_{\Q(\pi)}H^1(G_{\Q,S},\Lambda\otimes m_{\eet}(\pi))\hookrightarrow
m(\pi)^\dual\otimes_{\Q(\pi)}H^1(G_{\Q_p},\Lambda\otimes m_{\eet}(\pi)).$$
Maintenant, il y a deux possibilit\'es pour la restriction $m_{\eet}(\pi)_p$
de $m_{\eet}(\pi)$ \`a $G_{\Q_p}$ 
(apr\`es extension \'eventuelle \`a une extension quadratique de $L$):
elle est soit irr\'eductible, soit on a une suite exacte (qui peut \^etre scind\'ee mais, conjecturalement,
cela n'arrive que dans le cas de multiplication complexe) $0\to L(\delta_1)\to m_{\eet}(\pi)_p\to
L(\delta_2)\to 0$ o\`u $\delta_1,\delta_2:\Q_p^\dual\to \O_L^\dual$ sont des caract\`eres
localement alg\'ebriques de poids respectifs $j+1$ et $j-k$.

\hskip.4cm $\bullet$
Dans le premier cas,
les conditions de la prop.\,\ref{zk20} sur les $\exp^\dual$ d\'eterminent compl\`etement l'image
de ${\bf z}_{\rm Iw}(\pi)$ dans $m(\pi)^\dual\otimes_{\Q(\pi)}H^1(G_{\Q_p},\Lambda\otimes m_{\eet}(\pi))$
(c'est d\'ej\`a le cas, d'apr\`es~\cite{Berger}, si on fixe $ r$); cela prouve
directement l'unicit\'e de~${\bf z}_{\rm Iw}(\pi)$. 

\hskip.4cm $\bullet$ Dans le second cas, 
les conditions sur les $\exp^*$ ne d\'eterminent
l'image
de ${\bf z}_{\rm Iw}(\pi)$ dans $m(\pi)^\dual\otimes_{\Q(\pi)}H^1(G_{\Q_p},\Lambda\otimes m_{\eet}(\pi))$,
qu'\`a addition pr\`es
d'un \'el\'ement de $m(\pi)^\dual\otimes_{\Q(\pi)}H^1(G_{\Q_p},\Lambda\otimes L(\delta_1))$ 
et, pour prouver l'unicit\'e de ${\bf z}_{\rm Iw}(\pi)$, il faut utiliser le fait que
$L\otimes_{\O_L}H^1(G_\Q,\Lambda\otimes m_{\eet}(\pi))$ est,
toujours d'apr\`es Kato, de rang~$1$ sur $\Lambda[\frac{1}{p}]$
(et alors, comme ci-dessus, fixer $ r$ garantit d\'ej\`a l'unicit\'e).
\end{rema}

\subsubsection{L'\'el\'ement ${\bf z}_{(0,\infty)}(\pi)$}\label{como17}
Soit $\pi_p^{\rm alg}=\pi_p\otimes W_{k,j}^\dual$; c'est
une repr\'esentation localement alg\'ebrique de $\GG(\Q_p)$ que l'on encadre
par la recette du \no\ref{como12.0}.
Soit 
$$\iota_{\pi,p}:m(\pi)\otimes v_\pi^{]p[}\otimes \Pi_p(m_{\eet}^\dual(\pi))^{\rm alg}\to
H^1_c(\GG(\Q),{\cal C}(\GG(\A),\Q_p(\pi)))$$
la compos\'ee de la fl\`eche naturelle $\iota_\pi:m(\pi)\otimes_{\Q(\pi)}\pi^{\rm alg}
\to H^1_c(\GG(\Q),{\cal C}(\GG(\A),\Q_p(\pi)))$
et de \index{iotadr@\iotadr}l'isomorphisme
$$\iota_{\rm dR}^-:\pi_p^{\rm alg}\overset{\sim}{\to} \Pi_p(m_{\eet}^\dual(\pi))^{\rm alg},
\quad v\mapsto v\otimes(\iota_{{\rm dR},\check\pi}^-\otimes\zeta_{\rm dR})$$
On dit que $\pi$ {\it est $\Pi_p$-compatible} si $\iota_{\pi,p}$ s'\'etend \`a
$m(\pi)\otimes v_\pi^{]p[}\otimes \Pi_p(m_{\eet}^\dual(\pi))$.
\begin{rema}\phantomsection\label{passomme1}
(i) Si $m_{\eet}^\dual(\pi)$ v\'erifie la conjecture de compatibilit\'e local-global
d'Emerton, alors $\pi$ est $\Pi_p$-compatible.

(ii) Si $m_{\eet}^\dual(\pi)_{|G_{\Q_p}}$ est absolument irr\'eductible, alors
$\pi$ est $\Pi_p$-compatible d'apr\`es la prop.\,\ref{Ki118}.

(iii) Si $m_{\eet}^\dual(\pi)$ est r\'esiduellement g\'en\'erique, alors
$\pi$ est $\Pi_p$-compatible d'apr\`es le (ii) du th.\,\ref{Ycano110}.

(iv) L'extension de $\iota_{\pi,p}$ n'est pas forc\'ement unique:

$\bullet$ Si la restriction de $m_{\eet}(\pi)$ \`a $G_{\Q_p}$ n'est pas somme de deux
caract\`eres, on a ${\rm End}(m_{\eet}(\pi))=L$ car $m_{\eet}(\pi)$ a deux poids
distincts et donc ne peut pas \^etre une extension d'un caract\`ere par lui-m\^eme.
L'extension de $\iota_{\pi,p}$ est alors unique
(cf.~rem.\,\ref{valg2}).

$\bullet$ Si cette restriction
est somme de deux caract\`eres, $\Pi_p(m_{\eet}^\dual(\pi))$
est somme de deux s\'eries principales et les vecteurs localement alg\'ebriques
sont inclus dans l'une des s\'eries principales, d'o\`u une ind\'etermination
au niveau de l'autre s\'erie principale.
\end{rema}

{\it On suppose que $\pi$ est $\Pi_p$-compatible.}
Comme $(0,\infty)$ est fixe par $\matrice{p}{0}{0}{1}_p$
sur l'image de $\iota_{\pi,p}$ d'apr\`es le cor.~\ref{p=1}s,
il lui correspond, d'apr\`es la prop.~\ref{wasi2}, 
$${\bf z}_{(0,\infty)}(\pi) \in
m(\pi)^\dual\otimes_{\Q(\pi)} H^1(G_{\Q_p},\Lambda\otimes (m_{\eet}(\pi)\otimes\zeta_{\rm B})),$$ 
ainsi que, par \index{Zaapi@\zpi}torsion,
$${\bf z}_{(0,\infty)}\otimes\zeta_{\rm B}^{-1}\in m(\pi)^\dual\otimes_{\Q(\pi)} H^1(G_{\Q_p},\Lambda\otimes (m_{\eet}(\pi))).$$
\begin{prop}\phantomsection\label{ZK11}
Si $1-j\leq r\leq k+1-j$, si $\phi_ r\in{\rm LC}(\Z_p^\dual,\Q(\bmu_{p^\infty}))^{\Z_p^\dual}$
v\'erifie $\phi_ r(-x)=\pm (-1)^{ r+1}\phi_ r(x)$, et si $\gamma\in m(\pi)$,
alors
\begin{align*}
\exp^\dual\big(\int_{\Z_p^\dual}\phi_ r(x)(tx)^{ r}\, &{\bf z}_{(0,\infty)}(\pi)\otimes\zeta_{\rm B}^{-1}\otimes\gamma\big)\\
=&\ \big({(2i\pi)^{ r-1}} \Omega_\pi^{\pm}(\gamma)L(v_\pi^{]p[}\otimes\phi_ r,1- r)\big)\,\iota_{{\rm dR},\pi}^+.
\end{align*}
\end{prop}
\begin{proof}
On a
\begin{align*}
\exp^\dual\big(\int_{\Z_p^\dual}\phi_ r(x)(tx)^{ r}\, &{\bf z}_{(0,\infty)}(\pi)\otimes\zeta_{\rm B}^{-1}\otimes\gamma^{\pm}\big)\\
=&\
\exp^\dual\big(\int_{\Z_p^\dual}\phi_ r(x)(tx)^{ r-1}\, {\bf z}_{(0,\infty)}(\pi)\otimes\gamma^{\pm}\big)\otimes\zeta_{\rm dR}^{-1}
\end{align*}
Par ailleurs,
d'apr\`es la prop.~\ref{ES20bis},
$$
\exp^\dual\big(\int_{\Z_p^\dual}\phi(x,tx)\, {\bf z}_{(0,\infty)}(\pi)\otimes\gamma^{\pm}\big)=
\langle (0,\infty),\gamma^{\pm}\otimes(v_\pi^{]p[}\otimes\phi)\rangle\,
\iota_{{\rm dR},\pi}^+\otimes\zeta_{\rm dR}
$$
pour tout $\phi\in \pi_p^{\rm alg}$, \`a support dans $\Z_p^\dual$
(vu comme polyn\^ome en $X$ \`a coefficients
dans ${\rm LC}(\Z_p^\dual)$: on pose $\phi(x,tx)=\sum_i\phi_i(x)(tx)^i$,
si $\phi=\sum_i\phi_i X^i$).
Or, d'apr\`es le cor.~\ref{es4},
on a (le $\frac{1}{(-2i\pi)^{k-j+1}}$ apparaissant dans ce corollaire est compens\'e
par le $(-2i\pi)^{k+1-j}$ entrant dans la d\'efinition de $\iota_{{\rm dR},\pi}^+$):
\begin{equation}\label{ZK19}
\big\langle (0,\infty), \gamma^{\pm}\otimes(v_\pi^{]p[}\otimes\phi)\big\rangle=
\Omega_\pi^{\pm}(\gamma)L(v_\pi^{]p[}\otimes\phi^{\pm},0),
\end{equation}
o\`u $\phi^{\pm}(x,tx)=\frac{1}{2}(\phi(x,tx)\pm\phi(-x,-tx))$.
On peut appliquer ce qui pr\'ec\`ede \`a $\phi=\phi_ r X^{ r-1}$,
pour laquelle la condition mise sur $\phi_ r$ fait que $\phi=\phi^{\pm}$, 
et on obtient
\begin{align}\label{ZK20}
\exp^\dual\big(\int_{\Z_p^\dual}\phi_ r(x)(tx)^{ r-1}\, &{\bf z}_{(0,\infty)}(\pi)\otimes\gamma\big)\\
=&\ {(2i\pi)^{ r-1}} \Omega_\pi^{\pm}(\gamma)
L(v_\pi^{]p[}\otimes\phi_ r,1- r)\,\iota_{{\rm dR},\pi}^+\otimes\zeta_{\rm dR}. \notag
\end{align}
Ceci permet de conclure.
\end{proof}

\subsubsection{Comparaison entre ${\bf z}_{\rm Iw}(\pi)$ et $(0,\infty)$}\label{ZK99}
\index{Zbmpi@\zmpi}Posons (cf.~(\ref{ziw0}))
\begin{align*}
{\bf z}_{\rm Iw}(m_{\eet}(\pi))_p:=&\ {\bf z}_{(0,\infty)}(\pi)\otimes\zeta_{\rm B}^{-1} \in
m_{\eet}(\pi)^\dual\otimes H^1(G_{\Q_p},\Lambda\otimes m_{\eet}(\pi))\\
{\bf z}_{\rm Iw}(m_{\eet}(\pi)):=&\ {\bf z}_{\rm Iw}(\pi)\in
m_{\eet}(\pi)^\dual\otimes H^1(G_{\Q,S},\Lambda\otimes m_{\eet}(\pi))\\
{\bf z}(m_{\eet}(\pi)):=&\ {\bf z}(\pi)\in
m_{\eet}(\pi)^\dual\otimes H^1(G_{\Q,S}, m_{\eet}(\pi))
\end{align*}

\begin{theo}\phantomsection\label{zk2}
On a l'identit\'e suivante dans $m(\pi)^\dual\otimes H^1(G_{\Q_p},\Lambda\otimes m_{\eet}(\pi))$
modulo
\footnote{Ce module est de rang~$\leq 2$ sur $\Q_p(\pi)$ 
et est nul sauf dans le cas o\`u $\pi$ est de poids $(2,j)$ et $\pi_p$ est la steinberg (\`a torsion pr\`es
par un caract\`ere de $\Q_p^\dual$ trivial sur $p$).} 
$m(\pi)^\dual\otimes H^0(G_{\Q_p(\bmu_{p^\infty})},m_{\eet}(\pi))$: 
$${\bf z}_{\rm Iw}(m_{\eet}(\pi))_p={\rm loc}_p({\bf z}_{\rm Iw}(m_{\eet}(\pi))).$$
En particulier, ${\bf z}_{(0,\infty)}(\pi)\otimes\zeta_{\rm B}^{-1}$ est dans l'image de ${\rm loc}_p$.
\end{theo}
\begin{proof}
Si la restriction $m_{\eet}(\pi)_p$ de $m_{\eet}(\pi)$ \`a $G_{\Q_p}$ est irr\'eductible, 
cela r\'esulte directement (cf.~rem.\,\ref{como17.1}) de la comparaison de la formule de la prop.\,\ref{ZK11}
avec celle de~la prop.\,\ref{zk20},
et $H^0(G_{\Q_p(\bmu_{p^\infty})},m_{\eet}(\pi))=0$. 

Si $m_{\eet}(\pi)_p$ n'est pas irr\'eductible,
il faut proc\'eder par {\og prolongement analytique\fg}, voir la rem.~\ref{ZK} ci-dessous.
\end{proof}

\begin{rema}\phantomsection\label{ZK}
Dans le cas o\`u $m_{\eet}(\pi)_p$ n'est pas irr\'eductible,
on a une suite exacte $0\to L(\delta_1)\to m_{\eet}(\pi)_p\to L(\delta_2)\to 0$,
avec $\delta_1$ et $\delta_2$ de poids $j+1$ et $j-k$ respectivement. La suite
$$0\to H^1(G_{\Q_p},\Lambda\otimes L(\delta_1))\to H^1(G_{\Q_p},\Lambda\otimes m_{\eet}(\pi)_p)
\to H^1(G_{\Q_p},\Lambda\otimes L(\delta_2))$$
est exacte et
$${\bf z}_{\rm Iw}(m_{\eet}(\pi))_p-{\rm loc}_p({\bf z}_{\rm Iw}(m_{\eet}(\pi)))
\in m_{\eet}(\pi)^\dual\otimes H^1(G_{\Q_p},\Lambda\otimes L(\delta_1)).$$
De plus, s'il existe 
$z\in m_{\eet}(\pi)^\dual\otimes H^1(G_{\Q,S},\Lambda\otimes m_{\eet}(\pi))$ et $\alpha\in\Lambda$,
non diviseur de~$0$, tels que
$\alpha \,{\bf z}_{\rm Iw}(m_{\eet}(\pi))_p={\rm loc}_p(z)$, alors
$z=\alpha\,{\bf z}_{\rm Iw}(m_{\eet}(\pi))$ puisqu'on est en rang~$1$,
et donc
$\alpha\cdot\big({\bf z}_{\rm Iw}(m_{\eet}(\pi))_p-{\rm loc}_p({\bf z}_{\rm Iw}(m_{\eet}(\pi)))\big)=0$,
ce qui implique le th.\,\ref{zk2} pour $\pi$.

Nous produirons un tel couple $(z,\alpha)$ par prolongement analytique (prop.~\ref{geni9}).
\end{rema}

\Subsection{\'Equation fonctionnelle de l'\'el\'ement de Kato}\label{neqf6.1}
Le th.\,\ref{neqf12} ci-dessous est une extension de~\cite[th.\,4.7]{naka1}
(on se d\'ebarrasse de la condition selon laquelle la
restriction de $m_{\eet}(\pi)$ \`a $G_{\Q_p}$ est irr\'eductible).
Signalons l'existence d'un analogue local de cette \'equation fonctionnelle~\cite{joaquin}.
\subsubsection{Facteurs $\epsilon$}\label{neqf7}
On note $N$ le conducteur de $\pi$ et $N=\prod_\ell\ell^{n_\ell}$ sa 
factorisation en un produit de nombres premiers.
\begin{rema}\phantomsection\label{neqf7.1}
(i) On a un isomorphisme $\check\pi\cong\pi\otimes(\omega_\pi\circ\det)^{-1}$ 
de repr\'esentations de $\GG(\A^{]\infty[})$.
Pour des raisons de rationnalit\'e, l'isomorphisme de repr\'esentations encadr\'ees qui
lui correspond est, \`a multiplication pr\`es par un \'el\'ement de $\Q(\pi)^\dual$,
$\phi\mapsto G(\omega_\pi)\omega_\pi^{-1}\phi$.
(que l'on notera 
$\phi\mapsto \phi\otimes G(\omega_\pi)\omega_\pi^{-1}$).

(ii) De m\^eme, $\phi\mapsto G(\omega_{\pi,\ell})\omega_{\pi,\ell}^{-1}\phi$ 
induit un isomorphisme
de $\pi_\ell$ sur $\check\pi_\ell$ (qui sont tous les deux des sous-espaces de ${\rm LC}(\Q_\ell^\dual,
\Q(\pi)\otimes\Q(\bmu_{\ell^\infty}))$).
\end{rema}

\begin{prop}\phantomsection\label{neqf8}
{\rm (i)}
Pour tout $\ell$, il \index{epsi2@\facteur}existe 
$$\epsilon(\pi_\ell)\in \Q(\pi)^\dual\cdot G(\omega_{\pi,\ell})^{-1}\subset
\Q(\pi)\otimes\Q(\bmu_{\ell^\infty})$$
tel que
$$(\epsilon(\pi_\ell)\omega_{\pi,\ell})^{-1}\big(\matrice{0}{-1}{1}{0}_\ell\star  v_{\pi,\ell}\big)=
\matrice{M}{0}{0}{1}_\ell\star v_{\check\pi,\ell},
\quad{\text{si $M\in \ell^{n_\ell}\Z_\ell^\dual$.}}$$
De plus, $\epsilon(\pi_\ell)=1$ si $\ell\nmid N$
et $\epsilon(\pi_\ell)\epsilon(\check\pi_\ell)=\omega_{\pi,\ell}(-1)$.

{\rm (ii)} Si $\epsilon(\pi)=\prod_\ell\epsilon(\pi_\ell)$,
alors $\epsilon(\pi)\in \Q(\pi)^\dual\cdot G(\omega_\pi)^{-1}$
et
$$
 \big(\matrice{0}{-1}{1}{0}^{]\infty[}\star v_\pi\big)\otimes(\epsilon(\pi)\omega_\pi)^{-1}=
\matrice{N}{0}{0}{1}^{]\infty[}\star v_{\check\pi}.$$
De plus,
$\epsilon(\pi)\epsilon(\check\pi)=(-1)^k$.
\end{prop}
\begin{proof}
Par d\'efinition de $v_{\pi,\ell}$, on a $\matrice{a}{b}{Mc}{d}_\ell\star v_{\pi,\ell}=
\omega_{\pi,\ell}(d) v_{\pi,\ell}$, si $\matrice{a}{b}{Mc}{d}_\ell\in\wGamma_0(M)_\ell\subset\GG(\Z_\ell)$,
et la restriction de $\omega_{\pi,\ell}$ \`a $\Z_\ell^\dual$ est triviale si $\ell\nmid N$,
et se factorise \`a travers $(\Z_\ell/M)$ si $\ell\mid N$; en particulier, 
$$\omega_{\pi,\ell}
\big(\det \matrice{a}{b}{Mc}{d}_\ell\big)=\omega_{\pi,\ell}(a)\omega_{\pi,\ell}(d)$$
Maintenant,
$$\matrice{0}{-1}{M}{0}^{-1}=\matrice{0}{1/M}{-1}{0},
\quad
\matrice{0}{1}{-1}{0}\matrice{a}{Mb}{c}{d}\matrice{0}{-1}{1}{0}=\matrice{d}{-c}{-Mb}{a}$$
On en d\'eduit que
$$\matrice{a}{Mb}{c}{d}_\ell\star\big(\matrice{0}{-1}{1}{0}_\ell\star v_{\pi,\ell}\big)=
\omega_{\pi,\ell}(a)\,
\matrice{0}{-1}{1}{0}_\ell\star v_{\pi,\ell},$$ 
et donc, 
si $\matrice{a}{b}{Mc}{d}_\ell\in\wGamma_0(M)_\ell$,
$$\matrice{a}{b}{Mc}{d}_\ell\star
\big(\big(\matrice{0}{-1}{1}{0}_\ell\star v_{\pi,\ell}\big)\otimes\omega_{\pi,\ell}^{-1}\big)=
\omega_{\pi,\ell}^{-1}(d)
\big(\matrice{0}{-1}{1}{0}_\ell\star v_{\pi,\ell}\big)\otimes\omega_{\pi,\ell}^{-1}$$
De m\^eme $\matrice{M^{-1}}{0}{0}{1}_\ell\matrice{a}{Mb}{c}{d}_\ell\matrice{M}{0}{0}{1}_\ell=\matrice{a}{b}{Mc}{d}_\ell$
et donc $$\matrice{a}{Mb}{c}{d}_\ell\star\big(\matrice{M}{0}{0}{1}_\ell\star v_{\check\pi,\ell}\big)=
\omega_{\check\pi,\ell}(d)\matrice{M}{0}{0}{1}_\ell\star v_{\check\pi,\ell}$$ 
On a $$\omega_{\check\pi,\ell}=\omega_{\pi,\ell}^{-1}$$
Comme $\big\{\phi\in \check\pi_\ell,\ 
\matrice{a}{Mb}{c}{d}_\ell\star\phi=\omega_{\check\pi,\ell}(d)\phi,  
\ {\rm si}\ \matrice{a}{b}{Mc}{d}_\ell\in\wGamma_0(M)_\ell\big\}$ est de dimension~$1$,
on en d\'eduit l'existence de $\epsilon(\pi_\ell)$.  

L'appartenance de $\epsilon(\pi_\ell)$
\`a $\Q(\pi)G(\omega_{\pi,\ell}^{-1})$ r\'esulte, comme dans la preuve
du lemme~\ref{eqf1}, de ce que $ v_\pi$ et $ v_{\check\pi}$
sont d\'efinies sur $\Q(\pi)$;
le fait que $\epsilon(\pi_\ell)=1$ si $\ell\nmid N$ r\'esulte de la normalisation
de $v_{\pi,\ell}$ et $v_{\check\pi,\ell}$ (on a impos\'e que ces fonctions vaillent $1$ sur $\Z_p^\dual$).

Comme
$\omega_{\check\pi,\ell}=\omega_{\pi,\ell}^{-1}$,
\begin{align*}
(\omega_{\check\pi,\ell}^{-1}\circ\det)
\matrice{0}{-1}{1}{0}_\ell\star\big((\omega_{\pi,\ell}^{-1}\circ\det)^{-1}
\big(\matrice{0}{-1}{1}{0}_\ell\star v_{\pi,\ell}\big)\big)= 
\matrice{-1}{0}{0}{-1}_\ell\star v_{\pi,\ell}
= \omega_{\pi,\ell}(-1) v_{\pi,\ell}
\end{align*}
D'apr\`es ce qui pr\'ec\`ede, le
membre de gauche ${\rm MdG}$ de l'identit\'e ci-dessus v\'erifie aussi:
\begin{align*}
{\rm MdG}=&\ 
\epsilon({\pi_\ell})(\omega_{\pi,\ell}\circ\det)
\big(\matrice{0}{-1}{1}{0}_\ell\matrice{M}{0}{0}{1}_\ell\star v_{\check\pi,\ell}\big)\\
=&\ 
\epsilon({\pi_\ell})\omega_{\pi,\ell}(M)^{-1}\matrice{1}{0}{0}{M}_\ell\star
\big((\omega_{\check\pi,\ell}^{-1}\circ\det)\matrice{0}{-1}{1}{0}_\ell\star v_{\check\pi,\ell}\big)\\
=&\ \epsilon(\pi_\ell)\omega_{\pi,\ell}(M)^{-1}\epsilon(\check\pi_\ell)
\matrice{M}{0}{0}{M}_\ell\star v_{\pi,\ell}\\
=&\ \epsilon(\pi_\ell)\epsilon(\check\pi_\ell) v_{\pi,\ell}.
\end{align*}
Ceci termine la preuve du (i).

Le (ii) se d\'eduit du (i) gr\^ace aux factorisations
\begin{align*}
\matrice{0}{-1}{1}{0}^{]\infty[}\star  v_\pi=
\otimes_\ell\big(\matrice{0}{-1}{1}{0}_\ell\star  v_{\pi,\ell}\big),
\quad\omega_\pi=\otimes_\ell\omega_{\pi,\ell}
\end{align*}
et \`a la formule $\prod_{\ell}\omega_{\pi,\ell}(-1)=\omega_{\pi,\infty}^{-1}(-1)=(-1)^k$
(qui r\'esulte du lemme~\ref{carcen}).
\end{proof}

\Subsubsection{Action de l'involution sur l'\'el\'ement de Kato}

\begin{rema}\phantomsection\label{naka5}
{\rm (i)}
Si $\omega_\pi^{(p)}$ est le caract\`ere $p$-adique associ\'e \`a $\omega_\pi$,
on a $\omega_\pi^{(p)}=(\otimes_{\ell\neq p}\omega_{\pi,\ell})\otimes\omega_{\pi,p}^{(p)}$
et $\omega_{\pi,p}^{(p)}(x_p)=x_p^{k-2j}\omega_{\pi,p}(x_p)$.

{\rm (ii)}
Les repr\'esentations $\check\pi$ et $\pi\otimes\omega_\pi^{-1}$ de $\GG(\A^{]\infty[})$
sont isomorphes,
et on peut remplacer les sommes de Gauss intervenant dans les isomorphismes du \no\ref{como112.3}
(point sur la torsion des vecteurs localement alg\'ebriques)
par les facteurs $\epsilon$:
\begin{align*}
\check\pi^{\rm alg}&=\{(\epsilon(\pi)\omega_\pi^{(p)})^{-1}X^{2j-k}\phi,\ \phi\in\pi^{\rm alg}\}\\
\check\pi_p^{\rm alg}&=\{(\epsilon(\pi_p)\omega_{\pi,p}^{(p)})^{-1}X^{2j-k}\phi,\ \phi\in\pi_p^{\rm alg}\}\\
\Pi_p(m^\dual_{\eet}(\check\pi))&=
\{(\epsilon(\pi_p)\omega_{\pi,p}^{(p)})^{-1}\phi,\ \phi\in \Pi_p(m^\dual_{\eet}(\pi))\}\\
m_{\eet}(\check\pi)&=m_{\eet}(\pi)\otimes \epsilon(\pi)\zeta_{\rm B}^{-a}
\end{align*}
\end{rema}

On voit
${\bf z}_{\rm Iw}(\pi)\otimes\zeta_{\rm B}$ et ${\bf z}_{\rm Iw}(\check\pi)\otimes\zeta_{\rm B}$
comme des formes lin\'eaires sur $m_{\eet}(\pi)\otimes \Pi_p(m^\dual_{\eet}(\pi))$
et $m_{\eet}(\check\pi)\otimes \Pi_p(m^\dual_{\eet}(\check\pi))$ respectivement.

(iii) Il r\'esulte de la rem.\,~\ref{naka5}
que l'on a une identification naturelle
$$\big(m_{\eet}(\check\pi)\otimes \Pi_p(m^\dual_{\eet}(\check\pi))\big) 
\cong \big(m_{\eet}(\pi)\otimes \Pi_p(m^\dual_{\eet}(\pi))\big)\otimes \tfrac{\epsilon(\pi)}{\epsilon(\pi_p)}
(\omega_{\pi,p}^{(p)})^{-1};$$
l'isomorphisme
$\big(m_{\eet}(\pi)\otimes \Pi_p(m^\dual_{\eet}(\pi))\big)
\overset{\sim}{\to} \big(m_{\eet}(\check\pi)\otimes \Pi_p(m^\dual_{\eet}(\check\pi))\big)$
qui en r\'esulte est not\'e 
$\phi\mapsto \phi\otimes \tfrac{\epsilon(\pi)}{\epsilon(\pi_p)}
(\omega_{\pi,p}^{(p)})^{-1}$ .
Par dualit\'e, cela fournit
$$\big(m_{\eet}(\check\pi)\otimes \Pi_p(m^\dual_{\eet}(\check\pi))\big)^\dual \overset{\sim}{\to}
\big(m_{\eet}(\pi)\otimes \Pi_p(m^\dual_{\eet}(\pi))\big)^\dual$$
not\'e $\mu\mapsto \mu \otimes \tfrac{\epsilon(\pi)}{\epsilon(\pi_p)}
(\omega_{\pi,p}^{(p)})^{-1}$.

\begin{theo}\phantomsection\label{neqf12}
Si $\pi$ est $\Pi_p$-compatible, 
dans $\big(m_{\eet}(\pi)\otimes \Pi_p(m^\dual_{\eet}(\pi))\big)^\dual$, on a l'identit\'e suivante{\rm:}
$$\matrice{0}{-1}{1}{0}_p\star({\bf z}(\pi)\otimes\zeta_{\rm B})=
-\big(\matrice{M}{0}{0}{1}_p\star({\bf z}(\check\pi)\otimes\zeta_{\rm B})\big)
\otimes\tfrac{\epsilon(\pi)}{\epsilon(\pi_p)}(\omega^{(p)}_{\pi,p})^{-1},$$
o\`u $M=\prod_{\ell\neq p}\ell^{n_\ell}$.
\end{theo}
\begin{proof}
Il s'agit de prouver que les deux membres prennent la m\^eme valeur sur tout 
$\gamma\otimes \phi\in m_{\eet}(\pi)\otimes \Pi_p(m^\dual_{\eet}(\pi))$. 
D'apr\`es le th.\,\ref{zk2}, on a ${\bf z}(\pi)\otimes\zeta_{\rm B}=(0,\infty)\otimes v_\pi^{]p[}$
et ${\bf z}(\check\pi)\otimes\zeta_{\rm B}=(0,\infty)\otimes v_{\check\pi}^{]p[}$.
Le membre de gauche devient donc
\begin{equation}\label{naka7}
\big\langle \matrice{0}{-1}{1}{0}_p\star(0,\infty),\gamma\otimes v_\pi^{]p[}\otimes\phi\big\rangle=
-\big\langle (0,\infty),\gamma\otimes \big(\matrice{0}{1}{-1}{0}^{]\infty,p[}\star v_\pi^{]p[}\big)\otimes\phi\big\rangle
\end{equation}
(Le membre de gauche est la d\'efinition, et l'\'egalit\'e avec le membre de droite vient du lemme~\ref{np=p}.)
Maintenant, il r\'esulte de la prop.~\ref{neqf8} que
$$\matrice{0}{1}{-1}{0}^{]\infty,p[}\star v_\pi^{]p[}=\otimes_{\ell\neq p}
\big(\big(\matrice{M}{0}{0}{1}_\ell\star v_{\check\pi,\ell}\big)\otimes(\epsilon(\pi_\ell)\omega_{\pi,\ell})\big)$$
L'invariance de $(0,\infty)$ par multiplication par un caract\`ere fait que l'on ne change
pas le membre de droite de (\ref{naka7}) en multipliant par $(\omega_\pi^{(p)})^{-1}$.
En utilisant alors les identit\'es
$$\prod\nolimits_{\ell\neq p}\epsilon(\pi_\ell)=\tfrac{\epsilon(\pi)}{\epsilon(\pi_p)},
\quad \omega_\pi^{(p)}=\big(\prod\nolimits_{\ell\neq p}\omega_{\pi,\ell}\big)\omega_{\pi,p}^{(p)}$$
cela permet de mettre le membre de droite de (\ref{naka7}) sous la forme
$$-\big\langle (0,\infty),\big(\gamma\otimes \big(\otimes_{\ell\neq p}\big(\matrice{M}{0}{0}{1}_\ell\star v_{\check\pi,\ell}\big)\big)
\otimes \phi\big)\otimes \tfrac{\epsilon(\pi)}{\epsilon(\pi_p)}(\omega^{(p)}_{\pi,p})^{-1} \big\rangle$$
On conclut en utilisant le fait (lemme~\ref{np=p})
que $\matrice{M}{0}{0}{1}^{]\infty[}\star(0,\infty)=(0,\infty)$.
\end{proof}

\section{Interpolation des \'el\'ements de Kato}\label{chapi3}
Dans ce chapitre, on associe (th.\,\ref{EU10}, rem.\,\ref{geni8}, th.\,\ref{geni21})
\`a la repr\'esentation\footnote{En particulier, l'id\'eal ${\goth m}$ est
non-eisenstein; comme on veut pouvoir utiliser la factorisation d'Emerton au niveau
entier (th.\,\ref{Ycano110}), on suppose que ${\goth m}$ 
est g\'en\'erique (gr\^ace au th.\,\ref{NE2}
les constructions
qui suivent restent valables sans cette restriction si on veut bien se permettre d'introduire des
d\'enominateurs, mais il est difficile de controler les d\'enominateurs en question).}
 $\rho_T$ du \no\ref{YEUL7}
un syst\`eme d'Euler ${\bf z}_M^S(\rho_T)$, pour $(M,Np)=1$,
se sp\'ecialisant en les \'el\'ements de Kato du \no\ref{ZK99} aux points classiques
(multipli\'es par des facteurs d'Euler en les $\ell$ divisant $N$).
On utilise la factorisation de la cohomologie compl\'et\'ee
pour voir $(0,\infty)$ comme une forme lin\'eaire sur\footnote{Voir le
d\'ebut du \S\,\ref{EUL11} pour le passage de $\rho_T^\diamond$ \`a
$\check\rho_T$.} $\rho_T\otimes_T\Pi_p(\check\rho_T)$.
Cette forme est invariante par $\matrice{p}{0}{0}{1}_p$, et donne naissance \`a un \'el\'ement
${\bf z}_{\rm Iw}^S(\rho_T)_p$
de $\rho_T^\diamond\otimes H^1(G_{\Q_p},\Lambda\wotimes\rho_T)$ d'apr\`es la prop.\,\ref{wasi2}.
En appliquant ce qui pr\'ec\`ede \`a $\rho_{T_M}:=\Z_p[(\Z/M)^\dual]\otimes\rho_T$,
on obtient (prop.\,\ref{EU6}) un syst\`eme d'\'el\'ements locaux.

Par construction, ${\bf z}_{\rm Iw}^S(\rho_T)_p$ 
se sp\'ecialise (prop.\,\ref{ZK1}) en un multiple explicite de l'\'el\'ement 
${\bf z}_{\rm Iw}(\rho_x)_p$ du \no\ref{ZK99} qui est le localis\'e d'un \'el\'ement
global (th.\,\ref{zk2}). Sous la condition $\mu=0$, la suite exacte de Poitou-Tate et
des techniques de th\'eorie d'Iwasawa permettent d'en d\'eduire (th.\,\ref{EU10}) que
${\bf z}_{\rm Iw}^S(\rho_T)_p$ est le localis\'e
d'un \'el\'ement global ${\bf z}_{\rm Iw}^S(\rho_T,\iota_{\rm Em})$;
sans la condition $\mu=0$, on obtient un \'el\'ement global 
(rem.\,\ref{geni8}), mais seulement apr\`es extension des scalaires \`a ${\rm Fr}(T)$.
On prouve au chap.~\ref{chapi5} que l'extension des scalaires \`a ${\rm Fr}(T)$ est en fait inutile
(th.\,\ref{facto2}).

\vskip.2cm
{\it On suppose $p\neq 2$ dans ce qui suit.} 

\vskip.2cm
Comme on veut pouvoir tordre par des caract\`eres de $\Z_p^\dual$ et pas seulement par
des caract\`eres de $1+p\Z_p$, on 
va remplacer l'alg\`ebre de Hecke $T=T(Np^\infty)_{\goth m}:=T({\goth m})$ du \no\ref{YEUL7}
par
$T:=T({\goth m})[(\Z/p)^\dual]$. On a des 
\index{Tm@\TTm}identifications
$$T=T({\goth m})[(\Z/p)^\dual]=\prod\nolimits_{\eta}T({{\goth m}\otimes\eta}),$$
o\`u $\eta$ d\'ecrit les caract\`eres $(\Z/p)^\dual\to\Z_p^\dual$,
et $T({{\goth m}\otimes\eta})$ est la localis\'ee de $T(Np^\infty)$
en l'id\'eal maximal ${\goth m}\otimes\eta$ tel que $\overline\rho_{T({\goth m}\otimes\eta)}=
\overline\rho_{T({\goth m})}\otimes\eta$ (i.e.~${\rm Tr}(\rho_{T(\goth m\otimes\eta)}(\sigma_\ell))=\eta(\ell)
{\rm Tr}(\rho_{T({\goth m})}(\sigma_\ell))$, si $\ell\nmid Np$).
On dispose alors \index{rhoT@\rhoT}de
$$\rho_T:G_{\Q,S}\to {\bf GL}_2(T)\quad 
({\text{on a $\rho_T=\Z_p[(\Z/p)^\dual]\otimes_{\Z_p}\rho_{T({\goth m})}$}}).$$ 
On pose alors\index{Xcal@\calX}
$${\cal X}:={\rm Spec}(T) 
\quad{\rm et}\quad
{\cal X}^{\rm cl}(\O_L):=\{x\in {\cal X}(\O_L),\ x\ {\rm classique}\}.$$
On d\'efinit $H^1_c[\rho_T]$ comme le sous-$\GG(\Ai)$-module
de $H^1_c(\GG(\Q),{\cal C}(\GG(\A),\O_L))$ engendr\'e par
$\oplus_\eta H^1_c(Np^\infty)_{{\goth m}\otimes\eta}$.
Comme on a localis\'e en des id\'eaux non-eisenstein, il n'y a pas de diff\'erence
entre cohomologie compl\'et\'ee et cohomologie compl\'et\'ee
\`a support compact (mais on veut voir $H^1_c[\rho_T]$ comme un sous-module de la
cohomologie \`a support compact pour pouvoir l'accoupler avec $(0,\infty)$, d'o\`u la notation), 
et on obtient
la somme directe des twists de l'ancien $H^1[\rho_T]$ par les $\eta$ ci-dessus.
Le th.\,\ref{Ycano110} se transpose verbatim en faisant la somme directe sur les $\eta$
(si on voulait inclure $p=2$, il faudrait faire un peu plus attention).

\Subsection{\'Elimination d'une variable}\label{EUL7.3}
\subsubsection{La $\Z_p$-extension cyclotomique}
Soit $\Q_\infty$ la $\Z_p$-extension cyclotomique de~$\Q$:
c'est le sous corps de $\Q(\bmu_{p^\infty})$ fix\'e par
le sous-groupe de torsion de
${\rm Gal}(\Q(\bmu_{p^\infty})/\Q)$.
On note $H$ et $H_0$ les groupes de Galois sur $\Q$ de
$\Q(\bmu_{p^\infty})$ et $\Q_\infty$
(il serait plus normal de les noter $\Gamma$ et $\Gamma_0$
mais $\Gamma$ est d\'ej\`a utilis\'e pour les groupes de congruence...).
Le caract\`ere cyclotomique fournit des \index{Hgal@\HH}identifications:
$$H=\Z_p^\dual,\quad {\rm Gal}(\Q(\bmu_p)/\Q)=\bmu_{p-1},$$
et la d\'ecomposition $\Z_p^\dual=\bmu_{p-1}\times (1+p\Z_p)$
fournit une identification 
$$\iota:H_0\overset{\sim}{\to} 1+p\Z_p,$$
et $\iota$ peut aussi \^etre vu comme un caract\`ere privil\'egi\'e de $H_0$.

Soient $\Lambda$ et $\Lambda_0$ les alg\`ebres de \index{Lam@\lamb}groupes compl\'et\'ees
$$\Lambda_0=\Z_p[[H_0]]
\quad{\rm et}\quad
\Lambda=\Z_p[[H]]=\Lambda_0\otimes_{\Z_p}\Z_p[{\rm Gal}(\Q(\bmu_{p})/\Q)].$$
\index{Wcal@\calW}Posons 
$${\cal W}={\rm Spec}\,\Lambda,\quad {\cal W}_0={\rm Spec}\,\Lambda_0
\quad{\rm et}\quad
{\cal W}^{(p-1)}={\rm Spec}\big(\Z_p[{\rm Gal}(\Q(\bmu_{p})/\Q)]\big).$$
Alors, si $L$ est une extension finie de $\Q_p$,
${\cal W}(\O_L)$ (resp.~${\cal W}_0(\O_L)$) est l'ensemble
des caract\`eres continus de $H$ (resp.~$H_0$) \`a valeurs
dans $\O_L^\dual$ et ${\cal W}^{(p-1)}(\O_L)={\cal W}^{(p-1)}(\Z_p)$
est l'ensemble des caract\`eres $\eta:{\rm Gal}(\Q(\bmu_{p})/\Q)\to\bmu_{p-1}$.
On a ${\cal W}={\cal W}_0\times{\cal W}^{(p-1)}$.

Si $A=\Lambda,\Lambda_0, \Z_p[{\rm Gal}(\Q(\bmu_{p})/\Q)]$, alors
$A$ est naturellement un $A[G_\Q]$-module puisque
$H,H_0$ et ${\rm Gal}(\Q(\bmu_{p})/\Q)$ sont des quotients de $G_\Q$.
De plus, en tant que $\Z_p[G_\Q]$-module,
$\Z_p[{\rm Gal}(\Q(\bmu_{p})/\Q)]=\oplus_{\eta\in {\cal W}^{(p-1)}}\Z_p(\eta)$.
Il s'ensuit que, si $V$ est un $G_{\Q,S}$-module,
on a une d\'ecomposition
$$\Lambda\otimes V=\oplus_{\eta\in {\cal W}^{(p-1)}}\Lambda_0\otimes V(\eta),$$
ce qui permet de ramener l'\'etude de $\Lambda\otimes V$ \`a celle
des $\Lambda_0\otimes V(\eta)$; l'int\'er\^et \'etant que
$\Lambda_0$ est plus simple que $\Lambda$ d'un point de vue alg\'ebrique:
le groupe $H_0$ est isomorphe \`a $\Z_p$, et le choix
d'un g\'en\'erateur $\gamma$ fournit un isomorphisme
$\Gamma_0=\Z_p[[\gamma-1]]$.

\subsubsection{Descente de $T$ \`a $T_0$}\label{EUL10}
Le caract\`ere cyclotomique fournit une identification $G_{\Q,S}^{\rm ab}=\Z_S^\dual$.
Maintenant, on peut d\'ecomposer $\Z_S^\dual$, de mani\`ere naturelle, sous la forme
$\Delta\times (1+p\Z_p)$, 
o\`u $\Delta$ est le produit d'un groupe fini
par un groupe profini d'ordre premier \`a $p$ (i.e. $\Delta$ est le produit
de $\bmu_{p-1}\subset\Z_p^\dual$ et des $\Z_\ell^\dual$, pour $\ell\in S\moins\{p\}$).
Il s'ensuit que, si $A$ est un quotient de $T$, 
un caract\`ere continu $\eta:\Z_S^\dual\to A^\dual$ peut
s'\'ecrire, de mani\`ere unique, sous la forme
$\eta_0\eta_p^2$, o\`u $\eta_0$ est d'ordre fini, trivial sur $1+p\Z_p$, 
et $\eta_p$ est trivial sur $\Delta$ (et \`a valeurs dans $1+{\goth m}_A$).
On en d\'eduit que toute repr\'esentation $\rho:G_{\Q,S}\to {\bf GL}_2(A)$
peut se factoriser, de mani\`ere unique, sous la forme
$\rho=\rho_0\otimes\eta_p$, o\`u le caract\`ere de $\Z_S^\dual$ correspondant
\`a $\det\rho_0$ est trivial sur $1+p\Z_p$ et celui correspondant
\`a $\eta_p$ est trivial sur $\Delta$.
D'o\`u une \index{Tm@\TTm}factorisation:
$$T_{\goth m}=\Lambda_0\wotimes_{\Z_p} T_0,\quad{\text{(et donc $T=\Lambda\wotimes_{\Z_p} T_0$)}}$$
qui, elle-m\^eme, fournit une \index{rhoT@\rhoT}factorisation
$$\rho_{\goth m}=\Lambda_0\wotimes_{\Z_p}\rho_{T_0}, 
\quad{\text{(et donc $\rho_T=\Lambda\wotimes_{\Z_p}\rho_{T_0}$)}}$$
et, d'apr\`es le \no\ref{EUL6},
des factorisations
\begin{align*}
\Pi_p^\dual(\rho_T)= \Lambda\wotimes_{\Z_p}\Pi_p^\dual(\rho_{T_0})
\quad & {\rm et}\quad
\Pi_p(\rho_T)={\cal C}(H,\Z_p)\wotimes_{\Z_p}\Pi_p(\rho_{T_0})\\
\Pi_\ell(\rho_T)=& \Lambda\otimes_{\Z_p}\Pi_\ell(\rho_{T_0}),\ {\text{si $\ell\neq p$.}}
\end{align*}

\subsubsection{Points classiques}
\index{Xcal@\calX}Posons ${\cal X}_0={\rm Spec}\,T_0$.
Si $L$ est une extension finie de~$\Q_p$, et si $x\in{\cal X}_0(\O_L)$,
on note ${\goth p}_x$ l'id\'eal premier de $T_0$ qui lui correspond,
et $\rho_x$ la repr\'esentation $(T_0/{\goth p}_x)\otimes_{T_0}\rho_{T_0}$.
On dit que {\it $x$ est classique} si $\rho_x$ est la tordue par un caract\`ere d'une repr\'esentation
attach\'ee \`a une forme modulaire primitive de poids~$\geq 2$.
\begin{rema}
Par d\'efinition
de $T_0$, la somme des poids de Hodge-Tate de $\rho_x$ est $0$. Il s'ensuit que 
$x$ est classique si et seulement si les poids de Hodge-Tate de $\rho_x$
sont de la forme $\frac{1-k_x}{2},\frac{k_x-1}{2}$,
avec $k_x\geq 2$, et\footnote{$\iota:H_0\to 1+\Z_p$
est le caract\`ere privil\'egi\'e d\'efini ci-dessus.}
$\rho_x\otimes \iota^{(1-k_x)/2}$ est la repr\'esentation
associ\'ee \`a une forme modulaire $f_x$ de poids $k_x$.
\end{rema}
On note ${\cal X}_0^{\rm cl}$ l'ensemble des points classiques et
${\cal X}_0^{{\rm cl},+}$ le sous-ensemble des $x\in {\cal X}_0^{\rm cl}$,
tels que la restriction $\rho_{x,p}$ de $\rho_x$ \`a $G_{\Q_p}$ est
irr\'eductible. Alors, si ${\goth m}$ est g\'en\'erique,
 ${\cal X}_0^{\rm cl}$ et ${\cal X}_0^{{\rm cl},+}$ sont zariski-denses
dans ${\rm Spec}\,T_0[\frac{1}{p}]$, gr\^ace aux th\'eor\`emes {\og big $R$ = big $T$\fg}
(m\'ethode de la foug\`ere infinie).

\Subsection{Construction d'un syst\`eme d'Euler local}\label{EUL11}
On cherche \`a interpoler les \'el\'ements de Kato via $(0,\infty)$.
Comme,
d'apr\`es le th.\,\ref{zk2},
 ${\rm loc}_p({\bf z}_{\rm Iw}(m_{\eet}(\pi)))=
{\bf z}_{(0,\infty)}(\pi)\otimes\zeta_{\rm B}^{-1}$, 
on est amen\'e \`a tordre
la cohomologie compl\'et\'ee par $\zeta_{\rm B}$ pour faire dispara\^itre le $\zeta_{\rm B}^{-1}$
ci-dessus.

On \index{rhoT@\rhoT}pose $\check\rho_T=\rho_T^\diamond(1)$.
La factorisation du th.\,\ref{Ycano110} fournit une identification:
$$\rho_T\otimes_T\Pi(\check\rho_T)\equiv H^1_c[\rho_T(-1)]\otimes\zeta_{\rm B}.$$
A priori, le produit tensoriel devrait \^etre au-dessus de l'alg\`ebre de Hecke $T'$ correspondant
\`a $\rho_T(-1)$, mais $T'\cong T$ car ${\cal X}'={\rm Spec}\,T'$ et ${\cal X}={\rm Spec}\,T$ sont tous
deux isomorphes \`a ${\cal X}_0\times {\cal W}$, et l'isomorphisme $T'\cong T$ correspond
\`a $(\rho,\eta)\mapsto(\rho,\eta\cyp ^{-1})$ de ${\cal X}$ dans ${\cal X}'$.
\vskip.1cm
On renvoie au \S\,\ref{EUL3} pour les notations ci-dessous ($v_{T,\ell}$, $v'_{T,\ell}$,
$\chi_{\ell,i}$, etc.), avec $\check\rho_T$ au lieu de $\rho_T$. En particulier, les
polyn\^omes $P_\ell$ de la rem.\,\ref{eu100}
sont reli\'es aux facteurs d'Euler de $\check\rho_T$
et pas ceux de $\rho_T$.

\subsubsection{L'\'el\'ement de base ${\bf z}^S_{\rm Iw}(\rho_{T_0})_p$}\label{EUL12}
On prend pour $S$ l'ensemble des nombres premiers divisant $Np$.

\vskip.2cm
$\bullet$ {\it Les \'el\'ements ${\bf z}^S_{\rm Iw}(\rho_{T})_p$ et ${\bf z}^S(\rho_{T})_p$}.---
Soit 
$$\psi_S =v_T^{]S[}\otimes\big(\otimes_{\ell\in S\moins\{p\}}v'_{T,\ell}\big).$$ 
\index{inf2@\oits}Notons $(0,\infty)_{T,S} $
la forme bilin\'eaire
$$v\otimes \phi\mapsto\langle(0,\infty),(v\otimes\zeta_{\rm B}^{-1})\otimes\psi_S \otimes\phi\rangle$$
sur $\rho_T\otimes \Pi_p(\check\rho_T)$.
On a donc\footnote{Si $A$ est une $\O_L$-alg\`ebre et si $M$ est un $A$-module, on pose
$M^\diamond={\rm Hom}_A(M,A)$ et $M^\dual={\rm Hom}_{\O_L}(M,\O_L)$.
S'il y a ambigu\"{\i}t\'e sur $A$, on \'ecrit $M^{A{\text -}\dual}$ au lieu de $M^\diamond$.}
$$(0,\infty)_{T,S} \in
\rho_T^{\diamond}\otimes \Pi_p^\dual(\check\rho_T).$$
Maintenant, $\matrice{p}{0}{0}{1}^{]\infty,p[}\star\psi_S =\psi_S $ puisque
$v'_{T,\ell}$ est fixe par $\matrice{\Z_\ell^\dual}{\Z_\ell}{0}{1}$.
Il r\'esulte du cor.\,\ref{p=1} que $(0,\infty)_{T,S} $
est invariante par $\matrice{p}{0}{0}{1}_p$; elle d\'efinit donc, 
via la prop.~\ref{wasi2},
un \index{Zbemp@\zemp}\'el\'ement 
$${\bf z}^S_{\rm Iw}(\rho_{T})_p\in \rho_{T}^\diamond\otimes_{T}
 H^1(G_{\Q_p},\Lambda\otimes\rho_{T}).$$
On note 
$${\bf z}^S(\rho_T)_p\in \rho_{T}^\diamond\otimes_{T}
 H^1(G_{\Q_p},\rho_{T})$$
l'image de ${\bf z}^S_{\rm Iw}(\rho_{T})_p$ via l'application naturelle
$\Lambda\otimes\rho_{T}\to \rho_T$.

\vskip.2cm
$\bullet$ {\it L'\'el\'ement ${\bf z}^S_{\rm Iw}(\rho_{T_0})_p$}.---
Par ailleurs,
$$\rho_T\otimes\Pi_p(\check\rho_T)=
{\cal C}(\Z_p^\dual,\O_L)\wotimes_{\O_L}(\rho_{T_0}\otimes_T\Pi_p(\check\rho_{T_0}))$$
Il s'ensuit que
$$
\rho_T^{\diamond}\otimes_T \Pi_p^\dual(\check\rho_T)=
\Lambda\wotimes_{\Z_p}(\rho_{T_0}^{\diamond}\otimes_{T_0} \Pi_p^\dual(\check\rho_{T_0})).$$
La multiplication par un caract\`ere de $\Z_p^\dual$ sur
${\cal C}(\Z_p^\dual,\Z_p)\wotimes_{\Z_p}(\rho_{T_0}\otimes_T\Pi_p(\check\rho_{T_0}))\subset
H^1(\GG(\Q),{\cal C}(\GG(A),\O_L))$ est induite par celle sur
${\cal C}(\Z_p^\dual,\Z_p)$.
On en d\'eduit que
l'invariance de $(0,\infty)$ par torsion par un caract\`ere (cf.~rem.\,\ref{ES10.1})
implique que 
$$(0,\infty)_{T,S} =1\otimes\mu(\rho_{T_0}),\quad{\text{avec 
$\mu(\rho_{T_0})\in
\rho_{T_0}^{\diamond}\otimes_{T_0}
\Pi_p^\dual(\check\rho_{T_0})$.}}$$ 
Pour les m\^emes raisons que ci-dessus, $\mu(\rho_{T_0})$ 
est invariante par $\matrice{p}{0}{0}{1}_p$; elle d\'efinit donc, 
via la prop.~\ref{wasi2},
un \'el\'ement 
$${\bf z}^S_{\rm Iw}(\rho_{T_0})_p\in \rho_{T_0}^\diamond\otimes_{T_0}
 H^1(G_{\Q_p},\Lambda\wotimes_{\Z_p}\rho_{T_0}).$$

\subsubsection
{Lien entre ${\bf z}^S_{\rm Iw}(\rho_{T_0})_p$, ${\bf z}^S(\rho_T)_p$
et ${\bf z}^S_{\rm Iw}(\rho_{T})_p$}\label{lien1}
Comme nous
allons le voir, les \'el\'ements 
${\bf z}^S_{\rm Iw}(\rho_{T_0})_p$, ${\bf z}^S(\rho_T)_p$
et ${\bf z}^S_{\rm Iw}(\rho_{T})_p$ se d\'eterminent mutuellement.  Dans certaines
situations, il est plus agr\'eable de travailler avec ${\bf z}^S_{\rm Iw}(\rho_{T_0})_p$
(cela permet d'utiliser des techniques de th\'eorie d'Iwasawa),
et dans d'autres avec ${\bf z}^S(\rho_T)_p$.

Les identifications $\Lambda\wotimes_{\Z_p}T_0=T$
et $\Lambda\wotimes\rho_{T_0}^\diamond = \rho_T^\diamond$
permettent de d\'efinir un certain nombre de morphismes $G_{\Q,S}$-\'equivariant.

\vskip.1cm
\noindent$\diamond$
On note $(\Lambda\wotimes_{\Z_p}\rho_T)^{\rm lin}$ (res.~$(\Lambda\wotimes_{\Z_p}\rho_T)^{\rm semi}$)
le $G_{\Q,S}$-module
avec {\it action $\Lambda$-lin\'eaire} (resp.~{\it semi-lin\'eaire}) de $G_{\Q,S}$, 
i.e. $\sigma([a]\otimes v)=[a]\otimes\sigma(v)$ 
(resp.~$\sigma([a]\otimes v)=[\overline\sigma\,a]\otimes\sigma(v)$)
si $a\in\Z_p^\dual$ et $v\in\rho_T$, 
o\`u $\overline\sigma=\cyp (\sigma)$
est l'image de $\sigma$ dans $\Z_p^\dual$.

\vskip.1cm
\noindent$\diamond$
On note $\Delta: (\Lambda\wotimes_{\Z_p}\rho_T)^{\rm semi}\to (\Lambda\wotimes_{\Z_p}\rho_T)^{\rm lin}$
l'application $T_0$-lin\'eaire d\'efinie par
$$\Delta([a]\otimes[b]\otimes v)=[ab^{-1}]\otimes[b]\otimes v,
\quad{\text{si $a,b\in\Z_p^\dual$ et $v\in\rho_{T_0}$.}}$$

\vskip.1cm
\noindent$\diamond$
On \index{iotalin@\iotalin}d\'efinit $\iota^{\rm lin}:\rho_T\to (\Lambda\wotimes_{\Z_p}\rho_T)^{\rm lin}$ 
et $\iota^{\rm semi}:\rho_T\to (\Lambda\wotimes_{\Z_p}\rho_T)^{\rm semi}$ 
par 
$$\iota^{\rm lin}(z)=1\otimes z,\quad \iota^{\rm semi}([a]\otimes v)=[a]\otimes[a]\otimes v,
\ {\text{si $a\in\Z_p^\dual$ et $v\in\rho_{T_0}$.}}$$
\noindent$\diamond$
Enfin, on dispose de $p^{\rm lin}:(\Lambda\wotimes_{\Z_p}\rho_T)^{\rm lin}\to \rho_T$
et $p^{\rm semi}:(\Lambda\wotimes_{\Z_p}\rho_T)^{\rm semi}\to\rho_T$
induites par l'augmentation $\Lambda\to\Z_p$ envoyant $[a]$ sur $1$, si $a\in\Z_p^\dual$.

\vskip.2cm
Toutes ces applications sont $G_{\Q,S}$-\'equivariantes, et on a:
$$p^{\rm lin}\circ\iota^{\rm lin}={\rm id},\quad
p^{\rm semi}\circ\iota^{\rm semi}={\rm id}, \quad\Delta\circ\iota^{\rm semi}=\iota^{\rm lin}.$$
La $G_{\Q,S}$-\'equivariance de ces applications fournit des applications ayant les
m\^emes noms entre les $H^1(G_{\Q_p},-)$.
La discussion ci-dessus montre que $\mu\in \rho_T^\diamond\otimes_T
H^1(G_{\Q_p},(\Lambda\wotimes_{\Z_p}\rho_T)^{\rm lin})$ est dans l'image de $\iota^{\rm lin}$ si et seulement
l'\'el\'ement de $\rho_T^{\diamond}\otimes_T \Pi_p^\dual(\check\rho_T)$ qui lui correspond
est invariant par torsion par un caract\`ere. Il s'ensuit que:
\begin{align}\label{lien2}
{\bf z}^S(\rho_{T})_p&=p^{\rm semi}({\bf z}_{\rm Iw}^S(\rho_T)_p)\\
{\bf z}_{\rm Iw}^S(\rho_{T_0})_p&={\bf z}^S(\rho_{T})_p\notag\\
{\bf z}^S_{\rm Iw}(\rho_{T})_p&=\iota^{\rm semi}({\bf z}^S(\rho_T)_p)\notag\\
\Delta({\bf z}^S_{\rm Iw}(\rho_{T})_p)
&=\iota^{\rm lin}({\bf z}_{\rm Iw}^S(\rho_{T_0})_p)\notag
\end{align}
(La premi\`ere \'egalit\'e est une d\'efinition, les trois autres s'en d\'eduisent
en utilisant les relations entre les applications ci-dessus
et le fait que $\Delta({\bf z}^S_{\rm Iw}(\rho_{T})_p)$ est dans
l'image de $\iota^{\rm lin}$ puisque $(0,\infty)_{T,S}$ est invariant par
torsion par un caract\`ere.)

\subsubsection{Sp\'ecialisation en un point classique}\label{EUL5.1}
Si $x\in {\cal X}(\O_L)$, on note 
$${\bf z}^S_{\rm Iw}(\rho_x)_p\in 
\rho_{x}^\dual\otimes H^1(G_{\Q_p},\Lambda\otimes\rho_{x})$$
l'image de ${\bf z}^S_{\rm Iw}(\rho_{T})_p$.

Comme ${\cal X}(\O_L)={\cal X}_0(\O_L)\times{\cal W}(\O_L)$,
on peut \'ecrire $x$ sous la forme $x=(x_0,\eta)$.
Alors $\rho_{(x_0,\eta)}=\rho_{x_0}\otimes\eta$ et donc 
$$\rho_{(x_0,\eta)}^\dual=\rho_{x_0}^\dual\otimes\eta^{-1}
\quad{\rm et}\quad
H^1(G_{\Q_p},\Lambda\otimes\rho_{(x_0,\eta)})=H^1(G_{\Q_p},\Lambda\otimes\rho_{x_0})\otimes\eta.$$
On obtient donc une identification naturelle (les $\eta$ et $\eta^{-1}$ se compensent):
$$\rho_{(x_0,\eta)}^\dual\otimes H^1(G_{\Q_p},\Lambda\otimes\rho_{(x_0,\eta)})=
\rho_{x_0}^\dual\otimes H^1(G_{\Q_p},\Lambda\otimes\rho_{x_0}).$$
L'invariance de $(0,\infty)$ par torsion par un caract\`ere
implique, modulo cette identification, que
$${\bf z}^S_{\rm Iw}(\rho_{(x_0,\eta)})_p={\bf z}^S_{\rm Iw}(\rho_{x_0})_p,
\quad{\text{pour tout $\eta$.}}$$

Si $x$ est classique, et donc $\rho_x=m_{\eet}(\pi)$ pour une repr\'esentation
cohomologique $\pi$ de $\GG(\A^{]\infty[})$, on a d\'efini (cf.~rem.\,\ref{ZK})
des \'el\'ements
$${\bf z}_{\rm Iw}(\rho_x)_p\in \rho_{x}^\dual\otimes H^1(G_{\Q_p},\Lambda\otimes\rho_{x}),
\quad
{\bf z}_{\rm Iw}(\rho_x)\in \rho_{x}^\dual\otimes H^1(G_{\Q,S},\Lambda\otimes\rho_{x}).$$
Les \'el\'ements ${\bf z}^S_{\rm Iw}(\rho_x)_p$ 
et ${\bf z}_{\rm Iw}(\rho_x)_p$ sont tous les deux obtenus
\`a partir de $(0,\infty)$.
La diff\'erence entre leurs constructions 
est le remplacement du nouveau vecteur $v_\pi^{]p[}$ par $\psi_S$, ce qui
introduit un facteur d'Euler.
Plus pr\'ecis\'ement, on a le r\'esultat suivant,
dans lequel $P_\ell\in 1+X\O_L[X]$ est le polyn\^ome de la rem.\,\ref{eu100}
pour $\Pi_\ell(\check\rho_x)$.
\begin{prop}\phantomsection\label{ZK1}
Si $x\in {\cal X}^{\rm cl}$, alors
$${\bf z}^S_{\rm Iw}(\rho_x)_p=
\big(\prod_{\ell\in S\moins\{p\}}P_\ell([\sigma_\ell^{-1}])\big)
\cdot {\bf z}_{\rm Iw}(\rho_x)_p.$$
\end{prop}
\begin{proof}
D'apr\`es la rem.\,\ref{eu100}, la sp\'ecialisation $\psi_{x,S}$ de $\psi_S$ en $x$ est
$$\prod_{\ell\in S\moins\{p\}}P_\ell\big(\matrice{\ell^{-1}}{0}{0}{1}_\ell\big)\star v^{]p[}_\pi.$$
On a donc:
\begin{align*}
\langle {\bf z}^S_{\rm Iw}(\rho_x)_p,& v\otimes\phi_p\rangle=
\big\langle (0,\infty), (v\otimes\zeta_{\rm B}^{-1})\otimes 
\big(\prod_{\ell\in S\moins\{p\}}P_\ell\big(\matrice{\ell^{-1}}{0}{0}{1}_\ell\big)\star v^{]p[}_\pi\big)
\otimes\phi_p\big\rangle\\ &= \big\langle
\big(\prod_{\ell\in S\moins\{p\}}P_\ell\big(\matrice{\ell}{0}{0}{1}_\ell\big)\big)\star (0,\infty),
\iota_\pi((v\otimes\zeta_{\rm B}^{-1})\otimes v^{]p[}_\pi \otimes\phi_p)\big\rangle\\
&= \big\langle
\big(\prod_{\ell\in S\moins\{p\}}P_\ell\big(\matrice{\ell^{-1}}{0}{0}{1}_p\big)\big)\star (0,\infty),
\iota_\pi((v\otimes\zeta_{\rm B}^{-1})\otimes v^{]p[}_\pi \otimes\phi_p)\big\rangle\\
&= \big\langle \big(\prod_{\ell\in S\moins\{p\}}P_\ell([\sigma_\ell^{-1}])\big)
\cdot {\bf z}_{\rm Iw}(\rho_x)_p, v\otimes\phi_p\big\rangle
\end{align*}
(Le remplacement de $\matrice{\ell}{0}{0}{1}_\ell$ par $\matrice{\ell^{-1}}{0}{0}{1}_p$ 
vient de la troisi\`eme identit\'e du cor.~\ref{p=1}; 
celui de $\matrice{\ell^{-1}}{0}{0}{1}_p$ par $[\sigma_\ell^{-1}]$ vient de la
prop.\,\ref{wasi2}.)
\end{proof}

\begin{rema}
On a construit un objet local ${\bf z}^S_{\rm Iw}(\rho_{T})_p$ dont les sp\'ecialisations
en un ensemble zariski-dense sont des localisations de classes globales.
On peut donc esp\'erer que ${\bf z}_{\rm Iw}(\rho_{T})_p$ soit la localisation
d'un objet global interpolant les classes pr\'ec\'edentes.
C'est ce que nous prouvons aux \S\S\,\ref{EUL17} et~\ref{geni0}.
\end{rema}

\subsubsection{Construction de la tour d'\'el\'ements}\label{EUL13}
Soit $M$ premier \`a $Np$. 
\vskip.2cm
$\bullet$ {\it Le module $H^1_c(Np^\infty,M)$}.---
Soient:
\begin{align*}
\wGamma(Np^\infty,M)= &\ {\rm Ker}\big(
\wGamma(Np^\infty)\overset{\det}{\lra}(\cZp)^\dual\to (\Z/M)^\dual\big)\\
H^1_c(Np^\infty,M)= &\ H^1_c\big(\GG(\Q),{\cal C}\big(\GG(\A)/\wGamma(Np^\infty,M),\O_L\big)\big)
\end{align*}
Si $\alpha\in{\cal C}((\Z/M)^\dual,\O_L)$, 
notons $\tilde\alpha\in{\cal C}(\GG(\A),\O_L)$ la fonction
d\'efinie par $\tilde\alpha(g)=\alpha(\pi_M(\det g))$, o\`u
$\pi_M$ est la compos\'ee $\A^\dual\to\A^\dual/\Q^\dual\R_+^\dual\cong\cZ^\dual\to (\Z/M)^\dual$.
L'application $\phi\otimes\alpha\mapsto \phi\tilde\alpha$
induit un isomorphisme
\begin{align*}
{\cal C}\big(\GG(\A)/\wGamma(Np^\infty,M),\O_L\big)\cong &\ 
{\cal C}\big(\GG(\A)/\wGamma(Np^\infty),\O_L\big)\otimes_{\Z_p}{\cal C}\big((\Z/M)^\dual,\O_L\big)
\end{align*}
qui est $\GG(\Q)$-\'equivariant si on fait agir
$\GG(\Q)$ trivialement sur $(\Z/M)^\dual$.
On en d\'eduit un isomorphisme
$$H^1_c(Np^\infty,M)\cong
H^1_c(Np^\infty)\otimes_{\Z_p}{\cal C}\big((\Z/M)^\dual,\O_L\big).$$
Le membre de gauche est un sous-espace
de $H^1_c(NMp^\infty)$ qui est stable par
$T({NMp^\infty})$ agissant \`a travers 
$$T({Np^\infty,M})=
T({Np^\infty})[(\Z/M)^\dual]=T({Np^\infty})\otimes_{\Z_p}\Z_p[(\Z/M)^\dual],$$
et l'isomorphisme ci-dessus commute \`a l'action de
$T({Np^\infty,M})$.
Il commute aussi \`a l'action de $G_{\Q,MNp}$ agissant sur
${\cal C}((\Z/M)^\dual,\O_L)$ via l'identification
${\rm Gal}(\Q(\bmu_M)/\Q)= (\Z/M)^\dual$ fournie par le caract\`ere cyclotomique.

\vskip.2cm
$\bullet$ {\it L'alg\`ebre $T_M$ et les objets associ\'es}.---
Soit $S_M$ la r\'eunion de $S$ et des premiers divisant $M$.
\index{Tm2@\TTM}\index{rhom2@\rhoM}Posons 
\begin{align*}
T_M=T[(\Z/M)^\dual]&\quad{\rm et}\quad 
\rho_{T_M}=T_M\otimes_{T}\rho_T=\Z_p[(\Z/M)^\dual]\otimes_{\Z_p}\rho_T,\\
T_{0,M}=T_0[(\Z/M)^\dual]&\quad{\rm et}\quad 
\rho_{T_{0,M}}=T_{0,M}\otimes_{{T_0}}\rho_{T_0}=\Z_p[(\Z/M)^\dual]\otimes_{\Z_p}\rho_{T_0}.
\end{align*}
Alors $\rho_{T_M}$ est muni d'une action naturelle de $G_\Q$
par la recette du \no\ref{EUL6}
(via l'identification
${\rm Gal}(\Q(\bmu_M)/\Q)= (\Z/M)^\dual$), et cette action se factorise \`a travers
$G_{\Q,S_M}$.
Notons $H^1_c[\rho_{T_M}]_{S_M}$ le sous-$\GG(\Q_{S_M})$-module de
$H^1_c(\GG(\Q),{\cal C}(\GG(\A),\O_L))$ engendr\'e par 
$\big(\oplus_\eta H^1_c(Np^\infty)_{{\goth m}\otimes\eta}\big)
\otimes_{\Z_p}{\cal C}((\Z/M)^\dual,\O_L)$.
La factorisation de $H^1_c[\rho_T]_{S_M}$ induit une factorisation
$$H^1_c[\rho_{T_M}(-1)]_{S_M}\otimes\zeta_{\rm B}\cong
\rho_{T_M}\otimes\Pi_{S_M}(\check\rho_{T_M}),
\quad{\text{avec $\Pi_{S_M}(\check\rho_{T_M})=\otimes_{\ell\in {S_M}}\Pi_\ell(\check\rho_{T_M}),$}}$$
o\`u les produits tensoriels sont au-dessus de $T_M$, et
$\Pi_\ell(\check\rho_{T_M})$ est obtenue \`a partir de $\Pi_\ell(\check\rho_T)$
par la recette du \no\ref{EUL6}: en particulier,
$$\Pi_{S_M}^{]p[}(\check\rho_{T_M})=\Z_p[(\Z/M)^\dual]\otimes \Pi_{S_M}^{]p[}(\check\rho_{T}).$$

$\bullet$ {\it L'\'el\'ement $\psi_{S,M}$}.---
Rappelons que $\Pi_\ell(\check\rho_{T})\subset
{\rm LC}(\Q_\ell^\dual,\Z_p[\bmu_{\ell^\infty}]\otimes T)^{\sigma_p=1}$, si $\ell\neq p$,
qu'il
contient les fonctions \`a support compact dans $\Q_\ell^\dual$, et que
$v'_{T,\ell}={\bf 1}_{\Z_\ell}^\dual$.
Si $n\geq 1$, on pose
$$v_{T,\ell}^{(n)}={\bf 1}_{1+\ell^n\Z_\ell}\in\Z_p[\bmu_{\ell^n}]\otimes \Pi_\ell(\check\rho_{T})$$
et, si $(M,pN)=1$ admet pour factorisation $M=\prod_{\ell\mid M}\ell^{n_\ell}$, 
$$\psi_{S,M}=\sum
[a]_M\otimes\big(\matrice{a^{-1}}{0}{0}{1}^{]\infty,p[} \star
\big(\psi_S\cdot\big(\otimes_{\ell\mid M}
v_{T,\ell}^{(n_\ell)}\big)\big)\big)
\in \Z_p[\bmu_M]\otimes \Pi_{S_M}^{]p[}(\check\rho_{T}),$$
o\`u $a$ d\'ecrit un syst\`eme de repr\'esentants de $(\Z/M)^\dual$ dans $\Q$,
et $[a]_M$ est l'image de $a$ par $(\Z/M)^\dual\hookrightarrow \Z_p[(\Z/M)^\dual]$.
(Le r\'esultat ne d\'epend pas du choix du syst\`eme de repr\'esentants.)
\begin{lemm}\phantomsection\label{EU5}
{\rm (i)} $\psi_{S,M}\in \Pi_{S_M}^{]p[}(\check\rho_{T_M})$.

{\rm (ii)} $\matrice{p}{0}{0}{1}^{]\infty,p[}\star\psi_{S,M}=\psi_{S,M}$.
\end{lemm}
\begin{proof}
Pour prouver le (i), il s'agit de v\'erifier que $\sigma_p (\psi_{S,M}(x))=\psi_{S,M}(px)$.
Or, tout est invariant par $\sigma_p$ (car \`a valeurs dans $\Z_p$)
sauf $[a]_M$ pour lequel on a $\sigma_p ([a]_M)=[pa]_M$
(formule (\ref{EUL6.5})).
Par ailleurs, 
$$\matrice{a^{-1}}{0}{0}{1}^{]\infty,p[}\star\psi_S=\psi_S,\quad
\big(\matrice{a^{-1}}{0}{0}{1}^{]\infty,p[}\star v_{T,\ell}^{(n_\ell)}\big)(x)
=v_{T,\ell}^{(n_\ell)}(a^{-1}x)=v_{T,\ell}^{(n_\ell)}((pa)^{-1}px).$$
Le r\'esultat est donc une cons\'equence de ce que les $pa$ forment
aussi un syst\`eme de repr\'esentants de $(\Z/M)^\dual$.

Le (ii) r\'esulte des relations
(formule (\ref{EUL6.5}) pour la premi\`ere) 
$$\matrice{p}{0}{0}{1}^{]\infty,p[}\star[a]_M=[ap^{-1}]_M,\quad
\matrice{p}{0}{0}{1}^{]\infty,p[}\matrice{a^{-1}}{0}{0}{1}^{]\infty,p[}
=\matrice{pa^{-1}}{0}{0}{1}^{]\infty,p[}.\qedhere$$
\end{proof}

\vskip.2cm
$\bullet$ {\it L'\'el\'ement ${\bf z}^S_{{\rm Iw},M}(\rho_{T_0})_p$}.---
Le (ii) du lemme permet, comme pr\'ec\'edemment, de \index{Zbmem@\zmem}d\'efinir
\begin{align*}
{\bf z}^S_{{\rm Iw},M}(\rho_{T})_p&\in \rho_{T_{M}}^\diamond\otimes_{T_{M}}H^1(G_{\Q_p},\Lambda\wotimes\rho_{T_{M}})=
\rho_{T}^\diamond\otimes_{T}H^1(G_{\Q_p},\Lambda[(\Z/M)^\dual]\wotimes\rho_{T})\\
{\bf z}^S_{M}(\rho_{T})_p&\in 
\rho_{T}^\diamond\otimes_{T}H^1(G_{\Q_p},\Z_p[(\Z/M)^\dual]\otimes\rho_{T})\\
{\bf z}^S_{{\rm Iw},M}(\rho_{T_0})_p&\in 
\rho_{T_0}^\diamond\otimes_{T_0}H^1(G_{\Q_p},\Lambda[(\Z/M)^\dual]\wotimes\rho_{T_0})
\end{align*}
\`a partir de la forme bilin\'eaire
$v\otimes\phi\mapsto \langle(0,\infty),(v\otimes\zeta_{\rm B}^{-1})\otimes\psi_{S,M}\otimes\phi\rangle$.
Pour $M=1$, on retombe sur les \'el\'ements ${\bf z}^S_{{\rm Iw}}(\rho_{T})_p$,
${\bf z}^S(\rho_{T})_p$, ${\bf z}^S_{{\rm Iw}}(\rho_{T_0})_p$
du \no\ref{EUL5.1}.
Ces \'el\'ements v\'erifient les relations (\ref{lien2}); en particulier:
\begin{equation}\label{lien3}
{\bf z}^S_{{\rm Iw},M}(\rho_{T_0})_p={\bf z}^S_{M}(\rho_{T})_p,\quad
{\bf z}^S_{{\rm Iw},M}(\rho_{T})_p=\iota^{\rm semi}({\bf z}^S_{M}(\rho_{T})_p)
\end{equation}

\subsubsection{Relations de syst\`emes d'Euler}
Si $M\mid M'$, la projection naturelle $\Lambda[(\Z/M')^\dual]\to \Lambda[(\Z/M)^\dual]$
induit une application:
$${\rm cor}_{M'}^M:H^1(G_{\Q_p},\Z_p[(\Z/M')^\dual]\otimes_{\Z_p}\rho_{T})
\to H^1(G_{\Q_p},\Z_p[(\Z/M)^\dual]\otimes_{\Z_p}\rho_{T}).$$
Si $\ell\nmid pN$, on dispose\footnote{Rappelons
que ce sont les caract\`eres associ\'es \`a $\check\rho_T$ et pas $\rho_T$.}
 de $\chi_{\ell,i}:\Q_\ell^\dual\to T^\dual$, pour $i=1,2$,
cf.~\no\ref{EUL5}. On d\'efinit:
$$\chi_{\ell,i,M}:\Q_\ell^\dual\to T_M^\dual,
\quad\chi_{\ell,i,M}(x)=\chi_{\ell,i}(x)\,[x]_M,$$
 o\`u $[x]_M$ est obtenu via la projection 
$\Q_\ell^\dual\to\A^\dual\to \A^\dual/\Q^\dual\R_+^\dual\cong\cZ^\dual\to (\Z/M)^\dual$.
\begin{prop}\phantomsection\label{EU6}
{\rm (Relations de syst\`emes d'Euler)}

{\rm (i)} Si $\ell\nmid pNM$, alors
$${\rm cor}_{M\ell}^M{\bf z}^S_{M\ell}(\rho_{T})_p=
P_{\ell,M}([\sigma_\ell^{-1}])\cdot
{\bf z}^S_{M}(\rho_{T})_p,$$
o\`u l'on a pos\'e
$$P_{\ell,M}(X)=
1-(\chi_{\ell,1,M}(\ell)+\chi_{\ell,2,M}(\ell))X
+(\chi_{\ell,1,M}(\ell)\chi_{\ell,2,M}(\ell))X^2.$$

{\rm (ii)} Si $\ell\mid M$, alors
$${\rm cor}_{M\ell}^M{\bf z}^S_{M\ell}(\rho_{T})_p={\bf z}^S_{M}(\rho_{T})_p.$$
\end{prop}
\begin{proof}
La projection naturelle $(\Z/M')^\dual\to (\Z/M)^\dual$
induit d\'ej\`a une application
$${\rm cor}_{M'}^M:\Pi_{S_{M'}}^{]p[}(\check\rho_{T_{M'}})\to
\Pi_{S_M}^{]p[}(\check\rho_{T_{M}}),$$
et la formule~(\ref{eu1}) se traduit, si $\ell\nmid pNM$, par la relation:
\begin{align*}
{\rm cor}_{M\ell}^M\psi_{S,M\ell}=
P_{\ell,M}\big(\matrice{\ell^{-1}}{0}{0}{1}_\ell\big)
\star \psi_{S,M}
\end{align*}
(On utilise un syst\`eme de repr\'esentants de $(\Z/M\ell)^\dual$
adapt\'e \`a l'isomorphisme $(\Z/M\ell)^\dual\cong (\Z/\ell)^\dual\times(\Z/M)^\dual$
et la relation 
$\sum_{a\in \Z_\ell^\dual\,{\rm mod}\,1+\ell\Z_\ell}
\matrice{a}{0}{0}{1}v_{T,\ell}^{(1)}=v'_{T,\ell}$.)

On en d\'eduit le (i) en reprenant les calculs de la preuve de la prop.\,\ref{ZK1}.
Le (ii) se prouve en utilisant la relation ${\rm cor}_{M\ell}^M\psi_{S,M\ell}=\psi_{S,M}$,
qui repose sur l'identit\'e
$$\sum_{a\in 1+\ell^n\Z_\ell\,{\rm mod}\,1+\ell^{n+1}\Z_\ell}
\matrice{a}{0}{0}{1}v_{T,\ell}^{(n+1)}=v_{T,\ell}^{(n)},
\quad{\text{ si $n\geq 1$.}}\qedhere$$
\end{proof}
 
\begin{rema}\phantomsection\label{EU6.7}
On d\'eduit de (\ref{lien3}) des relations analogues entre
les ${\bf z}^S_{{\rm Iw},M}(\rho_{T_0})_p$
et entre les ${\bf z}^S_{{\rm Iw},M}(\rho_{T})_p$.
\end{rema}

\Subsection{Modules d'Iwasawa locaux}\label{EU60}
Soit $A$ une $\Z_p$-alg\`ebre locale compl\`ete (comme $\O_L$ ou $T_0$) et
$V$ une $A$-repr\'esentation de $G_{\Q,S}$ (comme $\rho_x$ ou $\rho_{T_0}$).
On va s'int\'eresser \`a la structure des modules
$H^1(G_{\Q_\ell},\Lambda\wotimes V)$, pour $\ell\in S$.

\subsubsection{Le cas $\ell\neq p$}\label{EU61}
Dans le cas $\ell\neq p$, on a le r\'esultat suivant ($I_\ell$ est le sous-groupe
d'inertie de $G_{\Q_\ell}$):
\begin{lemm}\phantomsection\label{EU62}
{\rm (i)}  $H^1(I_\ell,\Lambda\wotimes V)^{\sigma_\ell=1}=0$

{\rm (ii)} 
$H^1(G_{\Q_\ell},\Lambda\wotimes V)\cong (\Lambda\wotimes V^{I_\ell})/(\sigma_\ell-1)$.
En particulier,
$H^1(G_{\Q_\ell},\Lambda\wotimes V)$ est un $(\Lambda\wotimes A)$-module de torsion.
\end{lemm}
\begin{proof}
Comme $I_\ell$ agit trivialement sur $\Lambda$, on a 
$(\Lambda\wotimes V)^{I_\ell}=\Lambda\wotimes V^{I_\ell}$
et $H^1(I_\ell,\Lambda\wotimes V)=\Lambda\wotimes H^1(I_\ell,V)$.
La suite
d'inflation-restriction fournit donc une suite exacte
$$0\to H^1(\sigma_\ell^{\cZ},\Lambda\wotimes V^{I_\ell})\to H^1(G_{\Q_\ell},\Lambda\wotimes V)\to
H^1(I_\ell,\Lambda\wotimes V)^{\sigma_\ell=1},$$
et le (ii) est une cons\'equence du (i).

Pour prouver le (i), posons $M=H^1(I_\ell,V)$, et notons $M^\vee$ le dual de Pontryagin de $M$.
Alors $H^1(I_\ell,\Lambda\wotimes V)^{\sigma_\ell=1}$ est le dual de Pontryagin
de ${\cal C}(\Z_p^\dual,M^\vee)/(\sigma_\ell-1)$. Ce dernier module est nul car
il est aussi \'egal \`a $H^1(\Gamma_\ell,{\cal C}(\Z_p^\dual,M^\vee))$, o\`u
$\Gamma_\ell$ est l'image de $\sigma_\ell^{\cZ}$ dans $\Z_p^\dual$,
et ${\cal C}(\Z_p^\dual,M^\vee)$ est la somme directe d'un nombre fini
de copies de l'induite de $\{1\}$ \`a $\Gamma_\ell$ de $M^\vee$.

Ceci permet de conclure.
\end{proof}

\subsubsection{Le cas $\ell=p$}\label{EU64}
Dans le cas $\ell=p$, on utilise la th\'eorie des $(\varphi,\Gamma)$-modules.
D'apr\`es~\cite{CC99,Dee}, on a un isomorphisme naturel
$$H^1(G_{\Q_p},\Lambda\wotimes V)\cong D(V)^{\psi=1}$$
Si on note ${\cal C}(V)$ le module $(\varphi-1)\cdot D(V)^{\psi=1}$,
on a une suite exacte 
\begin{equation}\phantomsection\label{geni1.1}
0\to D(V)^{\varphi=1}\to D(V)^{\psi=1}\to {\cal C}(V)\to 0.
\end{equation}
\begin{theo}\phantomsection\label{geni1}
Le foncteur $V\mapsto {\cal C}(V)$ est exact.
\end{theo}
\begin{proof}
Il r\'esulte de~\cite[prop.\,VI.1.2]{mira} que l'on a un isomorphisme
$${\cal C}(V)\cong {\rm Hom}_\Lambda({\cal C}(V^\vee(1)),\Lambda[\tfrac{1}{p}]/\Lambda).$$
Par ailleurs, si $0\to V_1\to V\to V_2\to 0$ est une suite exacte
de repr\'esentations de $G_{\Q_p}$, on a des suites exactes 
\begin{align*}
&0\to D(V_1)^{\psi=1}\to D(V)^{\psi=1}\to D(V_2)^{\psi=1},\\
&0\to D(V_2^\vee(1))^{\psi=1}\to D(V^\vee(1))^{\psi=1}\to D(V_1^\vee(1))^{\psi=1}, 
\end{align*}
dont on d\'eduit des suites exactes
$$0\to {\cal C}(V_1)\to {\cal C}(V)\to {\cal C}(V_2),\quad
0\to {\cal C}(V_2^\vee(1))\to {\cal C}(V^\vee(1))\to {\cal C}(V_1^\vee(1)). $$
Maintenant, $D(V_1)$ et $D(V_2^\vee(1))$ sont orthogonaux; il en est
donc de m\^eme de ${\cal C}(V_1)$ et ${\cal C}(V_2^\vee)$ (cf.~la formule d\'efinissant
l'accouplement $\langle\ ,\ \rangle_{\rm Iw}$, \cite[prop.\,VI.1.2]{mira}).
Comme ${\cal C}(V)$ et ${\cal C}(V^\vee(1))$ sont en dualit\'e, on en d\'eduit
que ${\cal C}(V)$ se surjecte sur le dual de ${\cal C}(V_2^\vee(1))$, c'est-\`a-dire sur
${\cal C}(V_2)$, ce que l'on voulait prouver.
\end{proof}

\begin{prop}\phantomsection\label{geni2}
Si $A$ est une $\Z_p$-alg\`ebre locale et compl\`ete de corps r\'esiduel fini, 
et si $V$ est une $A$-repr\'esentation
de $G_{\Q_p}$ de rang~$d$,
on a une suite exacte
$$0\to H^0(G_{\Q_p(\bmu_{p^\infty})},V)\to H^1(G_{\Q_p},\Lambda\wotimes V)
\to {\cal C}(V)\to 0,$$
et ${\cal C}(V)$ est un $\Lambda\wotimes A$-module libre de rang~$d$ tandis que
$H^0(G_{\Q_p(\bmu_{p^\infty})},V)$ est 
un $\Lambda\wotimes A$-module de torsion.
\end{prop}
\begin{proof}
La suite exacte est une traduction de la suite (\ref{geni1.1}).
Que $H^0(G_{\Q_p(\bmu_{p^\infty})},V)$ soit de torsion est imm\'ediat.
Il reste donc \`a prouver que ${\cal C}(V)$ est libre.
Posons $B=\Lambda\wotimes A$.  Soit ${\goth m}$ l'id\'eal maximal de $A$; par hypoth\`ese,
$A/{\goth m}$ est un corps fini ${\bf F}_q$.  Soit $\overline V=(A/{\goth m})\otimes_AV$, de telle
sorte que $\overline V$ est une repr\'esentation de $G_{\Q_p}$, de rang $d$ sur ${\bf F}_q$
et donc ${\cal C}(\overline V)$ est libre de rang $d$ sur ${\bf F}_q\otimes_{\Z_p}\Lambda$.
Relevons une base de ${\cal C}(\overline V)$ sur ${\bf F}_q\otimes_{\Z_p}\Lambda$
en une famille d'\'el\'ements de $e_1,\dots,e_d$ de ${\cal C}(V)$ (c'est possible
puisque le foncteur ${\cal C}$ est exact).
Si $k\geq 1$, on a une suite exacte
$$0\to ({\goth m}^k/{\goth m}^{k+1})\otimes_{{\bf F}_q}\overline V\to
(A/{\goth m}^{k+1})\otimes_A V\to (A/{\goth m}^{k})\otimes_A V\to 0.$$
On en d\'eduit (par exactitude de ${\cal C}$)
que la seconde ligne du diagramme commutatif suivant est exacte
$$\xymatrix@R=.6cm@C=.4cm{
0\ar[r]&\frac{{\goth m}^k}{{\goth m}^{k+1}}e_1\oplus\cdots\oplus \frac{{\goth m}^k}{{\goth m}^{k+1}}e_d\ar[r]\ar[d]^-{\wr}
&\frac{B}{{\goth m}^{k+1}}e_1\oplus\cdots\oplus \frac{B}{{\goth m}^{k+1}}e_d\ar[r]\ar[d]
&\frac{B}{{\goth m}^{k}}e_1\oplus\cdots\oplus \frac{B}{{\goth m}^{k}}e_d\ar[r]\ar[d]^-{\wr}&0\\
0\ar[r]& \frac{{\goth m}^k}{{\goth m}^{k+1}}\otimes_{{\bf F}_q}{\cal C}(\overline V)\ar[r]&
{\cal C}(\frac{A}{{\goth m}^{k+1}}\otimes_A V)\ar[r]& {\cal C}(\frac{A}{{\goth m}^{k}}\otimes_A V)\ar[r]& 0}$$
ce qui permet de prouver, par r\'ecurrence sur $k$, que
$e_1,\dots,e_d$ est une base de ${\cal C}((A/{\goth m}^{k})\otimes_A V)$
sur $B/{\goth m}^k$ (l'isomorphisme vertical de gauche correspond au cas $k=1$, et
celui de droite est l'hypoth\`ese de r\'ecurrence).
Un passage \`a la limite permet de conclure.
\end{proof}

\begin{coro}
$H^1(G_{\Q_p},\Lambda\wotimes \rho_{T_0})$ est un $T$-module libre de rang~$2$.
\end{coro}
\begin{proof}
Cela r\'esulte de ce que $H^0(G_{\Q_p(\bmu_{p^\infty})},\rho_{T_0})=0$.
\end{proof}

\Subsection{Th\'eorie d'Iwasawa globale}\label{EUL15}
Le contenu de ce \S\  peut, essentiellement, se trouver dans~\cite[\S\,1.3]{PR95}.
\subsubsection{Suites exactes de Poitou-Tate et de localisation}
On fixe un g\'en\'erateur topologique $\gamma$ de $H_0=1+p\Z_p$ et, si $n\geq 1$,
on pose $\gamma_n=\gamma^{p^n}$.

Soit $V$ un $\Z_p$-module compact muni d'une action continue
de $G_{\Q,S}$.  
Si $i\in\N$, on pose
\begin{align*}
H^i_{{\rm Iw},S}(V)=&\ H^i(G_{\Q,S},\Lambda_0\wotimes V),\\
X^i_{{\rm Iw},S}(V)=&\ H^i(G_{\Q,S},{\cal C}(H_0 ,V^\vee(1)))^\vee=H^i(G_{\Q_\infty,S},V^\vee(1))^\vee,
\end{align*}
o\`u $M\mapsto M^\vee$ est la dualit\'e de Pontryagin.
On a $H^i_{{\rm Iw},S}(V)=0$ si $i\neq 1,2$ et
$X^i_{{\rm Iw},S}(V)=0$ si $i\neq 0,1,2$.
Par ailleurs, on dispose de la suite exacte de Poitou-Tate\footnote{Elle ne comporte
que $7$ termes au lieu de $9$ car les $H^0(G_{\Q_\ell},\Lambda_0\wotimes V)$ sont nuls.}
$$\xymatrix@R=.4cm@C=.5cm{
0\ar[r]& X^2_{{\rm Iw},S}(V)\ar[r]& H^1_{{\rm Iw},S}(V)\ar[r]&
\bigoplus_{\ell\in S}H^1(G_{\Q_\ell},\Lambda_0\wotimes V)\ar[r]& X^1_{{\rm Iw},S}(V)\ar[lld]\\
&& H^2_{{\rm Iw},S}(V)\ar[r]&
\bigoplus_{\ell\in S}H^2(G_{\Q_\ell},\Lambda_0\wotimes V)\ar[r]& X^0_{{\rm Iw},S}(V)\ar[r]&0.}$$
\begin{rema}\phantomsection\label{geni5}
$H^1_{{\rm Iw},S}(V)$ et ${\rm Ker}\big[H^2_{{\rm Iw},S}(V)\to 
\oplus_{\ell\in S}H^2(G_{\Q_\ell},\Lambda_0\wotimes V)\big]$ ne changent pas
si on augmente $S$.  En effet,
si $M=\Lambda_0\wotimes V$, 
 on a une suite exacte longue de localisation (les $H^0$ sont nuls)
$$\xymatrix@R=.4cm@C=.5cm{
0\ar[r] & H^1(G_{\Q,S},M)\ar[r] &
H^1(G_{\Q,S'},M)\ar[r] &
\bigoplus_{\ell\in S'\moins S} H^1(I_\ell,M)^{G_{{\bf F}_\ell}} \ar[dll]\\ 
& H^2(G_{\Q,S},M)\ar[r] & H^2(G_{\Q,S'},M)\ar[r] &
\bigoplus_{\ell\in S'\moins S} H^2(G_{\Q_\ell},M)\ar[r] & 0}$$
et $H^1(I_\ell,M)^{G_{{\bf F}_\ell}}=0$ ((i) du lemme~\ref{EU62}).
\end{rema}

\subsubsection{Le cas d'une repr\'esentation sur un corps fini}
On pose 
$$\overline\Lambda_0={\bf F}_q\otimes_{{\bf F}_p}(\Lambda_0/p)={\bf F}_q[[\gamma-1]].$$
Supposons maintenant que
$V$ est de dimension finie~$d$ sur~${\bf F}_q$.
Les $H^2(G_{\Q_\ell},\Lambda_0\otimes V)$ sont des 
$\overline\Lambda_0$-modules
de torsion, ainsi que les $H^1(G_{\Q_\ell},\Lambda_0\otimes V)$, pour $\ell\neq p$;
par contre, $H^1(G_{\Q_p},\Lambda_0\otimes V)$ est un 
$\overline\Lambda_0$-module de rang $d$.

Soit ${\rm Frob}_\infty\in G_\Q$ la conjugaison complexe.
Posons $d^+=\dim_{{\bf F}_q}V^{{\rm Frob}_\infty=1}$ et $d^-=\dim_{{\bf F}_q}V^{{\rm Frob}_\infty=-1}$. 
Les arguments habituels de caract\'eristique d'Euler-Poincar\'e
fournissent les relations (comme $S$ est fix\'e, on le supprime des $X^i_{{\rm Iw},S}$, etc.):
\begin{equation}
\label{iwas1}
{\rm rg}_{\overline\Lambda_0}H^1_{\rm Iw}( V )-
{\rm rg}_{\overline\Lambda_0}H^2_{\rm Iw}( V )=d-d^+=d^-,
\quad
{\rm rg}_{\overline\Lambda_0}X^1_{\rm Iw}( V )-
{\rm rg}_{\overline\Lambda_0}X^2_{\rm Iw}( V )=d^+
\end{equation}
\begin{lemm}\phantomsection\label{iwas3}
Si $ V $ est absolument irr\'eductible,
$H^1_{\rm Iw}( V )$ est un $\overline\Lambda_0$-module libre.
\end{lemm}
\begin{proof}
On a des suites exactes et isomorphismes
$$0\to H^1_{\rm Iw}( V )/(\gamma-1)\to H^1( V )\to
H^2_{\rm Iw}( V )[\gamma-1]\to 0,\quad
H^2_{\rm Iw}( V )/(\gamma-1)\cong H^2( V ).$$
Posons $k={\rm rg}_{\overline\Lambda_0}H^2_{\rm Iw}( V )$,
et donc $H^2_{\rm Iw}( V )=\overline\Lambda_0^k\oplus Z$,
o\`u $Z$ est un $\overline\Lambda_0$-module de torsion.
Alors
$$\dim_{{\bf F}_q}H^2_{\rm Iw}( V )[\gamma-1]=\dim_{{\bf F}_q}Z[\gamma-1]=\dim_{{\bf F}_q} Z/(\gamma-1)=
-k+\dim_{{\bf F}_q} H^2_{\rm Iw}( V )/(\gamma-1).$$
Comme $H^0( V )=0$ par hypoth\`ese,
la formule d'Euler-Poincar\'e 
$$\dim_{{\bf F}_q}H^1( V )-
\dim_{{\bf F}_q}H^0( V )-\dim_{{\bf F}_q}H^2( V )=
\dim_{{\bf F}_q} V -\dim_{{\bf F}_q} V ^{{\rm Frob}_\infty=1}
=d^-$$
implique que 
$\dim_{{\bf F}_q}H^1_{\rm Iw}( V )/(\gamma-1)=k+d^-$,
et comme ${\rm rg}_{\overline\Lambda_0}H^1_{\rm Iw}( V )=k+d^-$,
cela implique que $H^1_{\rm Iw}( V )$ est libre, ce que l'on
voulait prouver.
\end{proof}

\begin{lemm}\phantomsection\label{iwas2} 
Si $ V $ est absolument irr\'eductible,
les conditions suivantes sont \'equivalentes:

{\rm (i)} 
$H^2_{\rm Iw}( V )$ est un $\overline\Lambda_0$-module de torsion.

{\rm (ii)} $H^1_{\rm Iw}( V )$ est 
un $\overline\Lambda_0$-module libre de rang~$d^-$.

{\rm (iii)} $X^1_{\rm Iw}( V ^\vee(1))$ est 
un $\overline\Lambda_0$-module de rang~$d^+$ {\rm (pas n\'ecessairement libre)}.

{\rm (iv)} $X^2_{\rm Iw}( V ^\vee(1))$ est un $\overline\Lambda_0$-module de torsion.

{\rm (v)} $X^2_{\rm Iw}( V ^\vee(1))=0$.
\end{lemm}
\begin{proof}
Le lemme~\ref{iwas3} permet de prouver que (iv)$\Leftrightarrow$(v) 
puisque $X^2_{\rm Iw}(V ^\vee(1))$
s'injecte dans un module libre.

Les relations~(\ref{iwas1}) montrent
que (i)$\Leftrightarrow$(ii) et 
(pour $\rho^\vee(1)$)
que (iii)$\Leftrightarrow$(iv).

Pour conclure, il suffit donc de prouver 
que (i) \'equivaut \`a ce que $X^2_{\rm Iw}( V ^\vee(1))$ soit
de $\overline\Lambda_0$-torsion.
Or la suite exacte
$$0\to {\cal C}( H_0 / H_0 ^{p^n}, V )\to
{\cal C}( H_0 , V )\overset{\gamma_n-1}{\lra}
 {\cal C}( H_0 , V )\to 0$$
fournit, par passage \`a la cohomologie de $G_{\Q,S}$ et au 
${\bf F}_q$-dual,
la suite exacte
$$0\to X^2_{\rm Iw}( V ^\vee(1))/(\gamma_n-1)
\to H^2(G_{\Q,S},{\cal C}( H_0 / H_0 ^{p^n}, V ))^\vee\to
X^1_{\rm Iw}( V ^\vee(1))^{\gamma_n=1}\to 0.$$
Comme 
${\cal C}( H_0 / H_0 ^{p^n}, V )\cong 
\Z_p[ H_0 / H_0 ^{p^n}]\otimes V $, en tant
que $G_{\Q,S}$-module, et comme
$H^2_{\rm Iw}( V )/(\gamma_n-1)\cong H^2(G_{\Q,S},
\Z_p[ H_0 / H_0 ^{p^n}]\otimes V )$ puisque
$H^3(G_{\Q,S},-)=0$,
on a $$\dim_{{\bf F}_q}
H^2(G_{\Q,S},{\cal C}( H_0 / H_0 ^{p^n}, V ))^\vee
=\dim_{{\bf F}_q}H^2_{\rm Iw}( V )/(\gamma_n-1).$$
Par ailleurs, $X^1_{\rm Iw}( V ^\vee(1))^{\gamma_n=1}$ se stabilise
puisque $X^1_{\rm Iw}( V ^\vee(1))$ est de type fini
sur~$\overline\Lambda_0$.  Comme un $\overline\Lambda_0$-module
$M$ est de torsion si et seulement si $\dim_{{\bf F}_q}M/(\gamma_n-1)$
est major\'ee ind\'ependamment de $n$,
on voit que les conditions suivantes sont
\'equivalentes:

\quad $\bullet$ $X^2_{\rm Iw}( V ^\vee(1))$ est de $\overline\Lambda_0$-torsion,

\quad $\bullet$ $\dim_{{\bf F}_q}(X^2_{\rm Iw}( V ^\vee(1))/(\gamma_n-1))$
est major\'ee,

\quad $\bullet$ $\dim_{{\bf F}_q}(H^2(G_{\Q,S},{\cal C}( H_0 / H_0 ^{p^n}, V ))^\vee)$ 
est major\'ee

\quad $\bullet$ $\dim_{{\bf F}_q}H^2_{\rm Iw}( V )/(\gamma_n-1)$
est major\'ee,

\quad $\bullet$ $H^2_{\rm Iw}( V )$ est de $\overline\Lambda_0$-torsion.

Ceci permet de conclure.
\end{proof}

\Subsection{Th\'eorie d'Iwasawa de $\rho_{T_0}$}\label{EUL16}
\subsubsection{La condition $\mu=0$}
On va appliquer ce qui pr\'ec\`ede \`a $V=\rho_{T_0}$ et ses quotients.
Nous aurons besoin de mani\`ere intensive du r\'esultat fondamental suivant
de Kato~\cite[th.\,12.4]{Ka4} (Perrin-Riou~\cite[prop.\,1.3.2]{PR95} a prouv\'e que les quatre propri\'et\'es
sont essentiellement \'equivalentes et conjectur\'e 
-- sous le nom de {\og conjecture de Leopoldt faible\fg} -- qu'elles sont satisfaites):

\begin{theo}\phantomsection\label{KK0}
{\rm (Kato)}
Si $x$ est classique,
alors:

$\bullet$  $X^2_{\rm Iw}(\rho_x)=0$,

$\bullet$ $H^1_{\rm Iw}(\rho_x)$ est un $\Lambda_0$-module sans torsion, 
libre de rang~$1$
si on inverse~$p$ {\rm (et m\^eme sans inverser $p$ si $p\neq 2$)}, et si 
$\overline\rho_{\goth m}$
est irr\'eductible.

$\bullet$ $H^2_{\rm Iw}(\rho_x)$ est un $\Lambda_0$-module de torsion,

$\bullet$ $X^1_{\rm Iw}(\rho_x)$ est de rang~$1$ sur $\Lambda_0$
{\rm (mais peut avoir de la torsion)}.
\end{theo}

\begin{lemm}\phantomsection\label{iwas2.2} 
Les conditions suivantes sont \'equivalentes:

{\rm (i)} Il existe $x$ classique tel que
l'invariant $\mu$ du $\Lambda_0$-module de torsion
$H^2_{\rm Iw}(\rho_{x})$
soit nul.

{\rm (ii)} 
$H^2_{\rm Iw}(\overline\rho_{\goth m})$ est un $\overline\Lambda_0$-module de torsion.
\end{lemm}
\begin{proof}
Cela r\'esulte de
la nullit\'e
de $H^3_{\rm Iw}(\rho_{x})$ qui fournit un isomorphisme
$H^2_{\rm Iw}(\rho_{x})/\varpi_x\cong H^2_{\rm Iw}(\overline\rho_{\goth m})$, o\`u $\varpi_x$
est une uniformisante de $T_0/{\goth p}_x$.
(Comme $H^2_{\rm Iw}(\rho_{x})$ est de torsion sur $\Lambda_0$,
la nullit\'e de son invariant $\mu$ \'equivaut \`a ce que
$H^2_{\rm Iw}(\rho_{x})/\varpi_x$ soit de torsion sur 
$\overline\Lambda_0$).
\end{proof}

On dit que {\it $\rho_{T_0}$ v\'erifie $\mu=0$} 
si $\overline\rho_{\goth m}$ est absolument irr\'eductible et
si $H^2_{\rm Iw}(\overline\rho_{\goth m})$
est un $\overline\Lambda_0$-module de torsion
(et donc $\overline\rho_{\goth m}$ v\'erifie les conditions \'equivalentes du lemme~\ref{iwas2}).

\Subsubsection{La localisation en $p$}
\begin{prop}\phantomsection\label{geni3}
L'application de localisation 
$${\rm loc}_p: H^1_{\rm Iw}(\rho_{T_0})\to
H^1(G_{\Q_p},\Lambda_0\wotimes\rho_{T_0})$$
 est injective.
\end{prop}
\begin{proof}
Soit $z$ dans le noyau $M$ de cette application.  
Si $x\in X={\cal X}_0^{\rm cl}$,
soit $z_x$ l'image de $z$ dans 
$H^1_{\rm Iw}(\rho_x)$.  Alors ${\rm loc}_p(z_x)=0$, et $z_x$ appartient
au noyau $M_x$ de ${\rm loc}_p$.
Or $H^1_{\rm Iw}(\rho_x)$ est un $\Lambda_0$-module sans torsion (th.~\ref{KK0})
et donc $M_x$ est aussi sans $\Lambda_0$-torsion.  Comme $X^2_{\rm Iw}(\rho_x)=0$ (th.\,\ref{KK0}),
la suite exacte de Poitou-Tate fournit une injection de $M_x$ dans
$\oplus_{\ell\in S\moins\{p\}}H^1(G_{\Q_\ell},\Lambda_0\wotimes\rho_{T_0})$
qui est un $\Lambda_0$-module de torsion.  On en d\'eduit que $M_x=0$
et donc $z_x=0$.
Comme $\rho_{T_0}$ s'injecte dans $\prod_{x\in X}\rho_x$ puisque $T_0$ s'injecte
dans $\prod_{x\in X}{\goth p}_x$ par zariski-densit\'e de $X$,
et comme les $H^0_{\rm Iw}$ sont nuls, il s'ensuit
que $M$ s'injecte dans $\prod_{x\in X}M_x$, et donc $M=0$, ce que l'on voulait
d\'emontrer.
\end{proof}

\begin{coro}\phantomsection\label{geni4}
$X^2_{\rm Iw}(\rho_{T_0})=0$.
\end{coro}
\begin{proof}
Cela r\'esulte de la suite exacte de Poitou-Tate.
\end{proof}

\Subsection{Globalisation: le cas $\mu=0$}\label{EUL17}
On suppose dans ce paragraphe que l'on est dans le cas $\mu=0$,
et donc $\overline\rho_{\goth m}$ v\'erifie les conditions \'equivalentes du lemme~\ref{iwas2}.

\begin{lemm}\phantomsection\label{EU8}
Si $V=\rho_{T_0}$, $\prod_{x\in X}\rho_x$ ou $\big(\prod_{x\in X}\rho_x)/\rho_{T_0}$,
alors $X^2_{\rm Iw}(V)=0$.
\end{lemm}
\begin{proof}
Dans tous les cas, $V$ est un $\Z_p$-module compact (le dernier comme quotient
de modules compacts) et donc $V^\vee$ est discret et tous ses composants de
Jordan-H\"older sont isomorphes \`a $\overline\rho_{\goth m}$.
Comme on est dans le cas $\mu=0$, on a $X^2_{\rm Iw}(\overline\rho_{\goth m})=0$, ce qui permet
de conclure.
\end{proof}

\begin{lemm}\phantomsection\label{EU9}
Si $c_p\in H^1(G_{\Q_p},\Lambda_0\wotimes\rho_{T_0})$, les conditions suivantes
sont \'equivalentes:

{\rm (i)} Il existe $c\in H^1_{\rm Iw}(\rho_{T_0})$, unique,
tel que $c_{p}$ soit l'image de $c$ par l'application de localisation.

{\rm (ii)}  Il existe ${\cal X}'$ zariski-dense dans la fibre
g\'en\'erique de ${\cal X}_0$, tel que l'image $c_{p,x}$ de
$c_p$ dans $H^1(G_{\Q_p},\Lambda_0\otimes\rho_x)$ soit dans l'image
de la localisation $H^1_{\rm Iw}(\rho_x)\to H^1(G_{\Q_p},\Lambda_0\otimes\rho_x)$,
pour tout $x\in {\cal X}'$.
\end{lemm}
\begin{proof}
Il n'y a que (ii)$\Rightarrow$(i) \`a prouver.
L'unicit\'e de $c$ est une cons\'equence de la prop.\,\ref{geni3};
prouvons l'existence. 
Soit $Z$ le conoyau de $\rho_{T_0}\to\prod_{x\in{\cal X}'}\rho_x$.
On a un diagramme commutatif \`a lignes et colonnes exactes
$$\xymatrix@R=.5cm@C=.5cm{
&0\ar[d]&0\ar[d]&\\
0\ar[r]&H_{\rm Iw}^1(\rho_{T_0})\ar[r]\ar[d]&\oplus_{\ell\in S}
H^1(G_{\Q_\ell},\Lambda_0\wotimes\rho_{T_0})\ar[r]\ar[d]&X_{\rm Iw}^1(\rho_{T_0})\ar[d]\\
0\ar[r]&H_{\rm Iw}^1(\prod_x\rho_x)\ar[r]\ar[d]&\oplus_{\ell\in S}
H^1(G_{\Q_\ell},\Lambda_0\wotimes\prod_x\rho_x)\ar[r]\ar[d]&X_{\rm Iw}^1(\prod_x\rho_x)\ar[d]\\
0\ar[r]&H_{\rm Iw}^1(Z)\ar[r]&\oplus_{\ell\in S}
H^1(G_{\Q_\ell},\Lambda_0\wotimes Z)\ar[r]&X_{\rm Iw}^1(Z)
}$$
Les $0$ \`a gauche proviennent de la nullit\'e des $X^2_{\rm Iw}$ des termes
consid\'er\'es (cf.~lemme~\ref{EU8}); ceux du haut proviennent de la nullit\'e
des $H^0_{\rm Iw}$. 

L'hypoth\`ese implique que, pour tout $x\in{\cal X}'$,
il existe $(c_{\ell,x})_{\ell\in S\moins\{p\}}$, tel que
$\sum_{\ell\in S}c_{\ell,x}=0$ dans $X_{\rm Iw}^1(\rho_x)$.
Maintenant, $Z$ est un $\Lambda$-module compact qui admet une filtration
par des sous-$\Lambda$-modules compacts $Z_n$ tels que $Z_n/Z_{n+1}\cong
\overline\rho_{\goth m}$, pour tout $n$, et $Z=\varprojlim_n(Z/Z_n)$.
L'image de $((c_{\ell,x})_{\ell\in S})_{x}$ dans $\oplus_{\ell\in S}
H^1(G_{\Q_\ell},\Lambda_0\wotimes Z)$ a pour image $0$ dans $X_{\rm Iw}^1(Z)$ 
et, par construction,
tombe dans le sous-groupe $\oplus_{\ell\in S\moins\{p\}}
H^1(G_{\Q_\ell},\Lambda_0\wotimes Z)$.
La premi\`ere propri\'et\'e implique que cette image
provient d'un \'el\'ement $y$ de $H_{\rm Iw}^1(Z)$.
Pour conclure, il suffit de prouver que $y=0$ et,
pour cela il suffit de prouver que l'image
$y_n$ de $y$ dans $H^1_{\rm Iw}(Z/Z_n)$ est nulle pour tout $n$,
ce qui se fait par r\'ecurrence sur $n$ en utilisant
la suite exacte $0\to \overline\rho_{\goth m}\to
Z/Z_{n+1}\to Z/Z_n\to 0$:
l'hypoth\`ese de r\'ecurrence implique que $y_{n+1}$ appartient
au sous-groupe $H_{\rm Iw}^1(\overline\Lambda\otimes\overline\rho_{\goth m})$
de $H_{\rm Iw}^1(Z/Z_{n+1})$ (les $H_{\rm Iw}^0$ sont nuls), qui est libre sur $\overline\Lambda_0$,
et son image par l'injection
$H_{\rm Iw}^1(\overline\rho_{\goth m})\to
\oplus_{\ell\in S}
H^1(G_{\Q_\ell},\overline\Lambda_0\otimes\overline\rho_{\goth m})$ tombe
dans le sous-groupe $\oplus_{\ell\in S\moins\{p\}}
H^1(G_{\Q_\ell},\overline\Lambda_0\otimes\overline\rho_{\goth m})$
qui est de $\overline\Lambda_0$-torsion.
\end{proof}

Ce qui pr\'ec\`ede s'applique aux \'el\'ements ${\bf z}^S_{\rm Iw}(\rho_{T_0})_p$
et ${\bf z}^S_{{\rm Iw},M}(\rho_{T_0})_p$. 
Pour mettre le r\'esultat sous une forme plus standard, notons que le lemme de Shapiro fournit
un isomorphisme
\begin{equation}\label{shapi1}
H^1(G_{\Q,{S_M}},\Lambda\wotimes\rho_{T_{0,M}})=
H^1(G_{\Q(\zeta_M),{S_M}},\Lambda\wotimes\rho_{T_0}).
\end{equation}
\begin{theo}\phantomsection\label{EU10}
{\rm (i)} Si $\rho_{T_0}\otimes\eta $ v\'erifie $\mu=0$ pour tout caract\`ere $\eta$
de ${\rm Gal}(\Q(\bmu_p)/\Q)$, alors il \index{Zbem@\zem}existe
$${\bf z}^S_{\rm Iw}(\rho_{T_0})\in \rho_{T_0}^\diamond\otimes_{T_0}
H^1(G_{\Q,S},\Lambda\wotimes\rho_{T_0}),$$ 
unique, dont l'image par la localisation soit
${\bf z}^S_{\rm Iw}(\rho_{T_0})_p$.

{\rm (ii)} Plus g\'en\'eralement, si $(M,pN)=1$,
et si $\rho_{T_0}\otimes\eta $ v\'erifie $\mu=0$ pour tout caract\`ere $\eta$
de ${\rm Gal}(\Q(\bmu_{Mp})/\Q)$, alors il existe
$${\bf z}^S_{{\rm Iw},M}(\rho_{T_0})\in 
\rho_{T_0}^\diamond\otimes_{T_0}H^1(G_{\Q(\zeta_M),{S_M}},\Lambda\wotimes\rho_{T_0}),$$
unique, dont l'image par la localisation soit
${\bf z}^S_{{\rm Iw},M}(\rho_{T_0})_p$.

{\rm (iii)} Les ${\bf z}^S_{{\rm Iw},M}(\rho_{T_0})$ v\'erifient les relations de
syst\`emes d'Euler:

$\bullet$ Si $\ell\nmid pNM$, alors
$${\rm cor}_{M\ell}^M{\bf z}^S_{{\rm Iw},M\ell}(\rho_{T_0})=
P_{\ell,M}([\sigma_\ell^{-1}])\cdot
{\bf z}^S_{{\rm Iw},M}(\rho_{T_0}),$$
o\`u l'on a pos\'e
$$
P_{\ell,M}(X)=\big(1-(\chi_{\ell,1,M}(\ell)+\chi_{\ell,2,M}(\ell))X
+(\chi_{\ell,1,M}(\ell)\chi_{\ell,2,M}(\ell))X^2\big).$$

$\bullet$ Si $\ell\mid M$, alors
$${\rm cor}_{M\ell}^M{\bf z}^S_{{\rm Iw},M\ell}(\rho_{T_0})={\bf z}^S_{{\rm Iw},M}(\rho_{T_0}).$$

\end{theo}
\begin{proof}
Les (i) et (ii) sont une cons\'equence de la zariski-densit\'e des points classiques,
de la prop.~\ref{ZK1} et du lemme~\ref{EU9} (pour tous les $\rho_{T_0}\otimes\eta $ pertinents);
le groupe qui appara\^{\i}t naturellement dans le (ii) est
$\rho_{T_0}^\diamond\otimes_{T_0}H^1(G_{\Q,{S_M}},\Lambda\wotimes\rho_{T_{0,M}})$ et on utilise
l'isomorphisme~(\ref{shapi1}) pour le remplacer par le groupe de l'\'enonc\'e.

Le (iii) est une traduction de la prop.~\ref{EU6}.
\end{proof}

\begin{rema}\phantomsection\label{EU11}
{\rm (i)} On note 
$${\bf z}^S_M(\rho_T)\in\rho_T^\diamond\otimes 
H^1(G_{\Q(\zeta_M),{S_M}},\rho_T)$$ 
l'\'el\'ement correspondant \`a
${\bf z}^S_{{\rm Iw},M}(\rho_{T_0})$ via l'identification
$\rho_T=\Lambda\wotimes_{\Z_p}\rho_{T_0}$.
Les ${\bf z}^S_M(\rho_T)$ forment un syst\`eme d'Euler pour $\rho_T$.
 
{\rm (ii)} L'existence d'un tel syst\`eme d'Euler permet
d'utiliser les d\'eriv\'ees de Kolyvagin~\cite{Kol,PR98,Ka99,rubin}
pour \'etudier le groupe $H^2(G_{\Q,S},\check\rho_T)$.

{\rm (iii)} 
V\'erifier que
$\mu=0$ pour tous les tordus par des caract\`eres peut \^etre un peu compliqu\'e.
Pour les applications des syst\`emes d'Euler, seul le $p$-quotient
maximal $(\Z/M)^\dual_p$ de $(\Z/M)^\dual$ joue un r\^ole (on peut remplacer
les \'el\'ements initiaux par les corestrictions idoines), et 
les facteurs de Jordan-H\"older de
$\Z_p[(\Z/M)^\dual_p]\otimes\Lambda\wotimes\rho_{T_0}$ sont tous
isomorphes \`a $\overline\Lambda\otimes\overline\rho_{\goth m}$, ce qui fait qu'on a juste
besoin de $\mu=0$ pour $\overline\rho_{\goth m}$.
\end{rema}

\Subsection{Globalisation: le cas g\'en\'eral}\label{geni0}
On ne suppose plus que l'on est dans le cas $\mu=0$, mais on suppose toujours
que l'on est dans les conditions du th.~\ref{Ycano110} et, en particulier,
que $\overline\rho_{\goth m}$ est irr\'eductible (et comme $p\neq 2$, cela implique
que la restriction de $\overline\rho_{\goth m}$
\`a $G_{\Q_\infty}$ est aussi irr\'eductible).

\Subsubsection{La tour des ${\bf z}^S_{{\rm Iw},M}(\rho_{T_0})$}
\begin{lemm}\phantomsection\label{geni6}
Si $x\in {\cal X}_0(\O_L)$, l'application naturelle $X^1_{\rm Iw}(\rho_{T_0})/{\goth p}_x\to
X^1_{\rm Iw}(\rho_x)$ est un isomorphisme.
\end{lemm}
\begin{proof}
On note simplement $H^i(W)$ le groupe $H^i(G_{\Q_\infty,S},W)$.
Soient $a_1,\dots, a_r$ engendrant ${\goth p}_x$ comme $T_0$-module.
Si $1\leq n\leq r$, la multiplication par $a_n$ induit
une suite exacte $0\to\rho_{T_0}^\vee[a_1,\dots,a_{n}]\to\rho_{T_0}^\vee [a_1,\dots,a_{n-1}]
\overset{a_n}{\to}X\to 0$, o\`u $X$ est un sous-$T_0$-module de
$\rho_{T_0}^\vee[a_1,\dots,a_{n-1}]$.
L'hypoth\`ese $\overline\rho_{\goth m}$ irr\'eductible
implique que les $H^0$ sont nuls.  Il en r\'esulte que 
$H^1(X)\to H^1(\rho_{T_0}^\vee[a_1,\dots,a_{n-1}])$ est injective
et que $H^1(\rho_{T_0}^\vee[a_1,\dots,a_{n}])=H^1(\rho_{T_0}^\vee[a_1,\dots,a_{n-1}])[a_n]$.
Par r\'ecurrence, cela prouve que
$H^1(\rho_{T_0}^\vee[a_1,\dots,a_{n}])=H^1(\rho_{T_0}^\vee)[a_1,\dots,a_{n}]$.
Pour $n=r$, cela fournit, par dualit\'e de Pontryagin, le r\'esultat voulu.
\end{proof}

\begin{prop}\phantomsection\label{geni7}
Il existe $\alpha\in \Lambda_0\wotimes_{\Z_p}T_0$, non diviseur de $0$, tel que
la composante sur $\eta=1$ de
$\alpha\, {\bf z}^S_{\rm Iw}(\rho_{T_0})_p$ 
soit dans l'image de $\rho_{T_0}^\diamond\otimes_{T_0}H^1_{\rm Iw}(\rho_{T_0})$ par l'application de
localisation.
\end{prop}
\begin{proof}
Soient 
\begin{align*}
W_x=\rho_{x}^\dual\otimes &\big(X^1_{\rm Iw}(\rho_x)/\oplus_{\ell\in S\moins\{p\}}H^1(G_{\Q_\ell},\Lambda_0\otimes\rho_x)\big)\\
W=\rho_{T_0}^\diamond\otimes_{T_0}
 &\big(X^1_{\rm Iw}(\rho_{T_0})/\oplus_{\ell\in S\moins\{p\}}H^1(G_{\Q_\ell},\Lambda_0\otimes\rho_{T_0})\big)
\end{align*}
Si $x\in {\cal X}_0^{{\rm cl},+}$, l'image de la composante sur $\eta=1$ de 
${\bf z}^S_{\rm Iw}(\rho_{T_0})_p$ dans
$W_x$
est nulle d'apr\`es la d\'efinition de ${\bf z}^S_{\rm Iw}(\rho_{T_0})_p$,
le th.\,\ref{zk2} et la suite exacte de Poitou-Tate.  On en d\'eduit, en utilisant le lemme~\ref{geni6},
que l'image de la composante sur $\eta=1$ de ${\bf z}^S_{\rm Iw}(\rho_{T_0})_p$ dans $W$
appartient \`a ${\goth p}_xW$.  

Comme $W$ est de type fini sur $\Lambda_0\wotimes_{\Z_p}T_0$, il existe $\alpha\in \Lambda_0\wotimes_{\Z_p}T_0$, non diviseur de $0$, tel que
$\alpha$ tue le sous-module de torsion de $W$, et donc $\alpha W\subset W$ est sans $\Lambda_0\wotimes_{\Z_p}T_0$ torsion.
Mais alors l'image de $\alpha\, {\bf z}^S_{\rm Iw}(\rho_{T_0})_p$ appartient \`a ${\goth p}_x\alpha W$,
pour tout $x\in {\cal X}_0^{{\rm cl},+}$, et comme ${\cal X}_0^{{\rm cl},+}$ est
zariski-dense dans ${\cal X}_0$, on en d\'eduit que cette image est nulle (on peut injecter $\alpha W$
dans $(\Lambda_0\wotimes_{\Z_p}T_0)^d$ et $\cap_x {\goth p}_x\Lambda_0\wotimes_{\Z_p}T_0=0$ 
car $\cap_x{\goth p}_xT_0=0$ puisque 
${\cal X}_0^{{\rm cl},+}$ est
zariski-dense dans ${\cal X}_0$).
La suite exacte de Poitou-Tate permet d'en d\'eduire que
$\alpha\, {\bf z}^S_{\rm Iw}(\rho_{T_0})_p$ est la localisation d'un \'el\'ement de 
$\rho_{T_0}^\diamond\otimes_{T_0}H^1_{\rm Iw}(\rho_{T_0})$.
\end{proof}

\begin{rema}\phantomsection\label{geni8}
{\rm (i)} La prop.~\ref{geni7} pour $\rho_{T_0}$ et ses tordues par des caract\`eres de
$(\Z/p)^\dual$ permet de d\'efinir 
\begin{align*}
{\bf z}^S_{\rm Iw}(\rho_{T_0})&\in {\rm Fr}(T)
\otimes_T\big(\rho_{T_0}^\diamond \otimes_{T_0} H^1(G_{\Q,S},\Lambda\wotimes_{\Z_p}\rho_{T_0})\big)\\
{\bf z}^S(\rho_T)&\in {\rm Fr}(T)
\otimes_T\big(\rho_{T}^\diamond \otimes_{T} H^1(G_{\Q,S},\rho_{T})\big)
\end{align*}
le premier
 comme l'unique (cor.~\ref{geni4}) 
\'el\'ement ayant pour image
${\bf z}^S_{\rm Iw}(\rho_{T_0})_p$ par l'application de localisation,
le second par l'isomorphisme habituel.  

{\rm (ii)}
Plus g\'en\'eralement, on peut
\index{ZbM@\zm}d\'efinir, si $(M,pN)=1$,
\begin{align*}
{\bf z}^S_{{\rm Iw},M}(\rho_{T_0})& \in {\rm Fr}(T)\otimes_T\big(\rho_{T_0}^\diamond \otimes_{T_0}
 H^1(G_{\Q(\zeta_M),{S_M}},\Lambda\wotimes_{\Z_p}\rho_{T_0})\big)\\
{\bf z}^S_{M}(\rho_{T})& \in {\rm Fr}(T)\otimes_T\big(\rho_{T}^\diamond \otimes_{T}
 H^1(G_{\Q(\zeta_M),{S_M}},\rho_{T})\big)
\end{align*}
Les ${\bf z}^S_{M}(\rho_{T})$ v\'erifient les relations de syst\`emes d'Euler du th.\,\ref{EU10}
mais l'introduction des d\'enominateurs dans leur d\'efinition rend 
d\'elicate l'utilisation de la m\'ethode des syst\`emes d'Euler.

{\rm (iii)} L'exemple suivant montre qu'il est
 difficile de se d\'ebarrasser de ces d\'enominateurs en n'utilisant
que la zariski-densit\'e des points classiques.
Soit $T_0=\Z_p[[x,y]]$ et donc $T=\Z_p[[x,y,z]]$.
Soit $X=p\Z_p^\dual\times p\Z_p^\dual$: c'est un sous-ensemble zariski-dense de ${\rm Spec}\,T_0$;
si $a,b\in \Z_p^\dual$, on note ${\goth p}_{(a,b)}$ l'id\'eal $(x-pa,y-pb)$ de $T_0$ correspondant
\`a $(pa,pb)\in X$.
Enfin, soit $M=T/(zx-y)$.  Alors 
$$M/{\goth p}_{(a,b)}M=\Z_p[[x,z]]/(x-pa,xz-pb)=
\Z_p[[z]]/(paz-pb)={\bf F}_p[[z]]$$
 car $b-az$ est une unit\'e de $\Z_p[[z]]$.
On en d\'eduit que ${\goth p}_{(a,b)}M \supset pM$, pour tout $(a,b)$.

{\rm (iv)} Nous prouverons (th.\,\ref{facto2}) 
que ${\bf z}^S_M(\rho_{T})\in \rho_T^\diamond\otimes_T
H^1(G_{\Q,{S_M}},\rho_{T})$,
et donc qu'il n'y a pas de d\'enominateurs.
\end{rema}

\subsubsection{Comparaison entre ${\bf z}(\pi)$ et $(0,\infty)$: suite et fin}\label{geni9.4}
Le r\'esultat suivant, alli\'e \`a l'invariance
de $(0,\infty)$ par multiplication par un caract\`ere (pour passer de ${\cal X}$ \`a~${\cal X}_0$)
 permet (cf.~rem.\,\ref{ZK} et~prop.\,\ref{ZK1})
d'\'etendre le th.\,\ref{zk2} au cas g\'en\'eral (i.e.~sans supposer que $m_{\eet}(\pi)_p$
est absolument irr\'eductible).

\begin{prop}\phantomsection\label{geni9}
Si $x\in {\cal X}_0^{\rm cl}$, 
il existe $z\in\rho_x^\dual\otimes H^1(G_{\Q,S},\Lambda\otimes\rho_x)$ et
$\alpha_x\in\Lambda$, non diviseur de $0$, tels que ${\rm loc}_p(z)=\alpha_x\,{\bf z}^S_{\rm Iw}(\rho_x)_p$.
\end{prop}
\begin{proof}
On accouple avec un \'el\'ement g\'en\'erique
de $\rho_T$ pour ne pas tra\^{\i}ner des $\rho_T^\diamond$ ou $\rho_x^\dual$ dans la preuve,
et on fixe un caract\`ere de $(\Z/p)^\dual$ pour travailler avec $\Lambda_0$ plut\^ot
que $\Lambda$ et pouvoir utiliser la
prop.\,\ref{geni7} (le r\'esultat s'en d\'eduit en faisant la somme sur tous les $\eta$).

On sait qu'il existe $\alpha\in \Lambda_0\wotimes T_0$, non diviseur de $0$,
 tel que $\alpha\, {\bf z}^S_{\rm Iw}(\rho_{T_0})_p$
soit dans l'image de ${\rm loc}_p$.  Le probl\`eme est que $\alpha$ pourrait fort bien
\^etre un diviseur de $0$ sur $\{x\}\times {\cal W}_0$.
Par contre, $x$ \'etant classique,
est un point lisse de ${\cal X}_0$ et est 
un z\'ero isol\'e de $\alpha$ sur toute courbe suffisamment g\'en\'erique passant par $x$.

Soit donc $R=T_0/I$ un quotient de $T_0$, avec $\alpha\notin I$, $I\subset{\goth p}_x$,
et ${\rm Spec}(R[\frac{1}{p}])$ de dimension~$1$.
Notons ${\bf z}^S_{\rm Iw}(\rho_R)_p$ l'image de ${\bf z}^S_{\rm Iw}(\rho_{T_0})_p$ par
sp\'ecialisation.

Quitte \`a localiser, on peut supposer que $R$ contient un \'el\'ement $f$ dont
le diviseur est $(x)$ et il suffit de prouver que, si $\alpha$ s'annule en $x$,
alors $f^{-1}\alpha\, {\bf z}^S_{\rm Iw}(\rho_R)_p$ est encore dans l'image de ${\rm loc}_p$
(par r\'ecurrence, cela permet de supposer que $\alpha$ n'est pas identiquement 
nul sur $\{x\}\times {\cal W}_0$ et on peut
prendre pour $\alpha_x\in\Lambda$ la sp\'ecialisation de $\alpha$ \`a $\{x\}\times{\cal W}_0$).
Soit $z_R\in H^1_{\rm Iw}(\rho_R)$ l'image de $z_{T_0}\in H^1_{\rm Iw}(\rho_{T_0})$
v\'erifiant ${\rm loc}_p(z_{T_0})=\alpha\, {\bf z}^S_{\rm Iw}(\rho_{T_0})_p$, 
et $z_x$ son image dans $H^1_{\rm Iw}(\rho_x)$.
Comme $\alpha$ s'annule en $x$, on a ${\rm loc}_p(z_x)=0$, et donc $z_x=0$ puisque
${\rm loc}_p:H^1_{\rm Iw}(\rho_x)\to H^1(G_{\Q_p},\Lambda_0\wotimes \rho_x)$ est injective
(cf.~prop.\,\ref{geni3}).  

Soit $\sigma\mapsto c_\sigma$ un $1$-cocycle repr\'esentant $z_R$.
D'apr\`es ce qui pr\'ec\`ede, il existe $c\in\Lambda_0\otimes\rho_x$ tel
que $c_\sigma(x)=(\sigma-1)c$.  Soit $\tilde c$ un rel\`evement de
$c$ dans $\Lambda\wotimes \rho_R$.  Alors, par construction,
le $1$-cocycle $c_\sigma-(\sigma-1)\tilde c$ repr\'esente $z_R$ et est divisible
par $f$. Il s'ensuit que $f^{-1}\alpha\, {\bf z}^S_{\rm Iw}(\rho_R)_p$ est l'image
de $f^{-1}(c_\sigma-(\sigma-1)\tilde c)$ et donc est dans l'image de ${\rm loc}_p$.
\end{proof}

\subsubsection{Le syst\`eme d'Euler des ${\bf z}^S_M(\rho_T)$}
En anticipant un peu (i.e., en utilisant le th.\,\ref{facto2} pour $\rho_T$ et les $\rho_{T_M}$), 
on obtient finalement le r\'esultat suivant,
si ${\goth m}$ est g\'en\'erique.
\begin{theo}\phantomsection\label{geni21}
{\rm (i)} Si $(M,Np)=1$, 
$${\bf z}^S_{M}(\rho_{T})\in \rho_T^\diamond\otimes_T
H^1(G_{\Q(\zeta_M),{S_M}},\rho_{T}).$$

{\rm (ii)}
Les ${\bf z}^S_{M}(\rho_{T})$ v\'erifient les relations de syst\`emes d'Euler du th.\,\ref{EU10}.

{\rm (iii)} Si $x\in {\cal X}$, et si ${\bf z}^S_{M}(\rho_{x})\in \rho_x^\dual\otimes
H^1(G_{\Q(\zeta_M),{S_M}},\rho_{x})$ est la sp\'ecialisation de ${\bf z}^S_{M}(\rho_{T})$,
alors $({\bf z}^S_{M}(\rho_{x}))_M$ est un syst\`eme d'Euler pour $\rho_x$.

{\rm (iv)} 
Si $x\in {\cal X}^{\rm cl}$, alors $({\bf z}^S_{M}(\rho_{x}))_M$ est reli\'e au
syst\`eme d'Euler de Kato par la relation suivante\footnote{\label{euler}
$P_\ell$ est
le polyn\^ome $P_\ell$ de la prop.\,\ref{ZK1}. Comme $\sigma_\ell$ agit trivialement
sur $H^1(G_{\Q,S},\rho_x)$,  $P_\ell([\sigma_\ell^{-1}])$ agit par $P_\ell(1)$. On a
aussi $P_\ell(1)=L_\ell(\check\rho_x,0)^{-1}$ d'apr\`es la rem.\,\ref{eu100}}
 pour $M=1$
$${\bf z}^S(\rho_x)=\big(\prod_{\ell\in S\moins\{p\}}P_\ell(1)\big)
\cdot {\bf z}(\rho_x).$$
\end{theo}

\begin{rema}\phantomsection\label{geni8.5}
En \index{ZbM@\zm}posant
$${\bf z}^S_{{\rm Iw},M}(\rho_{T})=\iota^{\rm semi}({\bf z}^S_{M}(\rho_{T})),$$
o\`u $\iota^{\rm semi}:\rho_{T_M}\to \Lambda\wotimes_{\Z_p}\rho_{T_M}$
est l'application du \no\ref{lien1},
on obtient un syst\`eme d'Euler $({\bf z}^S_{{\rm Iw},M}(\rho_{T}))_M$
pour $\Lambda\wotimes_{\Z_p}\rho_T$. 

Si $x\in{\cal X}$, ce syst\`eme d'Euler
se sp\'ecialise en un syst\`eme d'Euler $({\bf z}^S_{{\rm Iw},M}(\rho_{x}))_M$
pour $\Lambda\wotimes_{\Z_p}\rho_x$. Si $x\in{\cal X}^{\rm cl}$,
ce syst\`eme d'Euler redonne, \`a un facteur d'Euler pr\`es, le syst\`eme d'Euler
de Kato qui permet d'\'etudier la th\'eorie d'Iwasawa de la forme modulaire associ\'ee \`a $x$.
\end{rema}

\section{Le syst\`eme des \'el\'ements de Beilinson-Kato}\label{chapi4}
Ce chapitre est un survol de la construction du syst\`eme de Beilinson-Kato
et de ses twists \`a la Soul\'e.

\Subsection{S\'eries d'Eisenstein-Kronecker}\label{s24}
\subsubsection{D\'efinition}\label{s25}
Si $(\tau,z)\in {\cal H}\times\C$, on
pose $q=e^{2i\pi\tau}$, $q_z=e^{2i\pi z}$ et on note
$\partial_z$ l'op\'erateur $\frac{1}{2i\pi}\frac{\partial}{\partial z}=
q_z\frac{\partial}{\partial q_z}$.  On pose aussi ${\bf e}_\infty(a)=e^{-2i\pi a}$.
Si $k\in\N$, $\tau\in{\cal H}$, $z,u\in\C$, la s\'erie
d'Eisenstein-Kronecker
$$H_k(s,\tau,z,u)=\frac{\Gamma(s)}{(-2i\pi)^k}
\big(\frac{\tau-\overline\tau}{2i\pi}\big)^{s-k}
\sum_{\omega\in\Z+\Z\tau}
\frac{\overline{\omega+z}^k}{|\omega+z|^{2s}}
{\bf e}_\infty\Big({\frac{\omega\overline u-u\overline \omega}{\overline\tau-\tau}}\Big),$$
qui converge\footnote{Si $z\in\Z+\Z\tau$, on supprime le terme
correspondant \`a $\omega=-z$ de la somme.}
si ${\rm Re}(s)>1+\frac{k}{2}$, poss\`ede un prolongement m\'eromorphe
\`a tout le plan complexe, holomorphe en dehors de p\^oles simples
en $s=1$ (si $k=0$ et $u\in\Z+\Z\tau$) et $s=0$ (si $k=0$
et $z\in\Z+\Z\tau$) et v\'erifie l'\'equation fonctionnelle
$$H_k(s,\tau,z,u)=
{\bf e}_\infty\Big({\frac{z\overline u-u\overline z}{\overline\tau-\tau}}\Big)
\cdot
H_k(k+1-s,\tau,u,z).$$
On d\'efinit les fonctions $E_k$ et $F_k$ par les formules
$$E_k(\tau,z)=H_k(k,\tau,z,0)\quad{\rm et}\quad
F_k(\tau,z)=H_k(k,\tau,0,z).$$
On a
$$E_{k+1}(\tau,z)=\partial_zE_k(\tau,z)\quad
{\rm et}\quad
E_0(\tau,z)=\log|\theta(\tau,z)|^2,\ {\rm si}\ z\notin\Z\tau+\Z,$$
o\`u $\theta(\tau,z)$ est donn\'ee par le produit
infini
$$\theta(\tau,z)=q^{1/12}(q_z^{{1}/{2}}-q_z^{{-1}/{2}})
\prod_{n\geq 1}\big((1-q^nq_z)(1-q^nq_z^{-1})\big).$$

\subsubsection{Les formes modulaires $E^{(k)}_{\alpha,\beta}$ et
$F^{(k)}_{\alpha,\beta}$}\label{s26}
Les fonctions $E_k(\tau,z)$ et $F_k(\tau,z)$ sont p\'eriodiques
en $z$ de p\'eriode $\Z\tau+\Z$.  Si $(\alpha,\beta)\in(\Q/\Z)^2$
et si $(a,b)\in\Q^2$ a pour image $(\alpha,\beta)$ dans
$(\Q/\Z)^2$, on pose
$$E^{(k)}_{\alpha,\beta}=E_k(\tau,a\tau+b)
\quad{\rm et}\quad
F^{(k)}_{\alpha,\beta}=F_k(\tau,a\tau+b).$$

\begin{prop}\phantomsection\label{p1}
{\rm (i)} $E^{(2)}_{0,0}=F^{(2)}_{0,0}=\frac{-1}{24}E_2^\dual$, o\`u
$$E_2^\dual=
\tfrac{6}{i\pi(\tau-\overline\tau)}
+{1}-24\sum_{n=1}^{+\infty}\sigma_1(n)q^n$$
est la s\'erie d'Eisenstein quasi-holomorphe
de poids $2$ habituelle.

\noindent 
{\rm (ii)} Si $N\alpha=N\beta=0$, alors

{\rm (a)} $E^{(2)}_{\alpha,\beta}-E^{(2)}_{0,0}\in M^{\rm cl}_2(\Gamma_N,\Q(\zeta_N))$
et $E^{(k)}_{\alpha,\beta}\in M^{\rm cl}_k(\Gamma(N),\Q(\zeta_N))$
si $k\geq 1$, $k\neq 2$.

{\rm (b)} $F^{(k)}_{\alpha,\beta}\in M^{\rm cl}_k(\Gamma(N),\Q(\zeta_N))$
si $k\geq 1$, $k\neq 2$ ou si $k=2$ et $(\alpha,\beta)\neq (0,0)$.
\end{prop}
On a les relations de distribution suivantes si $e$ est un entier~$\geq 1$:
\begin{align*}
\sum_{e\alpha'=\alpha,\,e\beta'=\beta}E_{\alpha',\beta'}^{(k)}=
e^kE^{(k)}_{\alpha,\beta}
\qquad & {\rm et}\qquad
\sum_{e\alpha'=\alpha,\,e\beta'=\beta}F_{\alpha',\beta'}^{(k)}=
e^{2-k}F^{(k)}_{\alpha,\beta}\\
\sum_{e\beta'=\beta}E_{\alpha,\beta'}^{(k)}(\frac{\tau}{e})=
e^kE^{(k)}_{\alpha,\beta}
\qquad & {\rm et}\qquad
\sum_{e\beta'=\beta}F_{\alpha,\beta'}^{(k)}(\frac{\tau}{e})=
e\,F^{(k)}_{\alpha,\beta}.
\end{align*}
Ces relations de distribution
peuvent se condenser agr\'eablement en l'\'enonc\'e suivant:
\begin{theo}\phantomsection\label{t1}
Si $k\geq 1$, il \index{Zeis@\zeis}existe 
$${\bf z}_{\rm Eis}(k), {\bf z}'_{\rm Eis}(k)\in
\dalg({\bf M}_{2,1}(\A^{]\infty[}),M_k^{\rm cl,\,qh}(\Q^{\rm cycl})),$$ telles que,
quels que soient $r\in\Q^\dual$ et $(a,b)\in\Q^2$, l'on
ait
\begin{align*}
\int_{\vecteur{a+r\cZ}{b+r\cZ}}{\bf z}_{\rm Eis}(k)=
r^{-k}E^{(k)}_{-\frac{b}{r},\frac{a}{r}},\qquad
\int_{\vecteur{a+r\cZ}{b+r\cZ}}{\bf z}'_{\rm Eis}(k)=
r^{k-2}F^{(k)}_{\frac{b}{r},-\frac{a}{r}}
\end{align*}
De plus, si $k\neq 2$, alors ${\bf z}_{\rm Eis}(k)$ et $ {\bf z}'_{\rm Eis}(k)$
sont \`a valeurs dans $M_k^{\rm cl}(\Q^{\rm cycl})$.
\end{theo}

\subsubsection{Quelques $q$-d\'eveloppements}\label{s30}
Si $\alpha\in\Q/\Z$, posons
$$\zeta(\alpha,s)=\sum_{n\in\Q_+^\dual,\ n\equiv\alpha\,{\rm mod}\,\Z}n^{-s}
\quad{\rm et}\quad
\zeta^\dual(\alpha,s)=\sum_{n=1}^{+\infty}e^{2i\pi\, n\alpha }n^{-s}.$$
\begin{prop}\phantomsection\label{qdev}
Si $k\geq 1$, et $\alpha,\beta\in\Q/\Z$, alors le $q$-d\'eveloppement
$\sum_{n\in\Q_+}a_nq^n$ de $F^{(k)}_{\alpha,\beta}$ est donn\'e par
$$\sum_{n\in\Q_+^\dual}\frac{a_n}{n^s}=
\zeta(\alpha,s-k+1)\zeta^\dual(\beta,s)+(-1)^k\zeta(-\alpha,s-k+1)\zeta^\dual(-\beta,s)$$
et $a_0=\zeta(\beta,1-k)$ {\rm [}resp.~$a_0=\frac{1}{2}(\zeta^\dual(\beta,0)
-\zeta^\dual(-\beta,0))${\rm ]} si $k\neq 1$ ou $\alpha\neq 0$
{\rm (}resp.~si $k=1$ et $\beta=0${\rm )}.
\end{prop}
\begin{rema}\phantomsection\label{r1}
  Il y a des formules similaires pour le $q$-d\'eveloppement
de $E^{(k)}_{\alpha,\beta}$, mais nous n'en aurons pas besoin.
\end{rema}

\Subsection{Unit\'es de Siegel}\label{s34}
\subsubsection{La distribution ${\bf z}_{\rm Siegel}$}\label{s35}
Voir~\cite[chap.\,1]{Ka4} pour ce qui suit (ainsi que~\cite{KL}).
\index{O@\sO}Soit $$\iO=\varinjlim_N\O(Y(N)).$$  
On identifie $\iO$ aux fonctions $\phi$ sur ${\cal H}\times \GG(\A^{]\infty[})$, invariantes
par $\GG(\Q)$ (pour l'action $*$), holomorphes
en $\tau$ et avec un p\^ole d'ordre fini aux pointes, et dont le $q$-d\'eveloppement
v\'erifie ${\cal K}(\phi,u)\in\Q^{\rm cycl}$ et $\sigma_a({\cal K}(\phi,u))={\cal K}(\phi,au)$
quels que soient $a\in\cZ^\dual$ et $u\in\Aidu$.
Comme d'habitude (lemme~\ref{emrt1}), l'action naturelle de $\GG(\A^{]\infty[})$ sur $\iO$
devient l'action $\star$ via cette identification.

La fonction $\theta(\tau,z)$ n'est pas p\'eriodique en $z$ de p\'eriode
$\Z\tau+\Z$, mais si $c\geq 2$ est un entier premier \`a $6$, alors la fonction
$q_z^{\frac{c-c^2}{2}}\theta(\tau,z)^{c^2}\theta(\tau,cz)^{-1}$ est p\'eriodique.
Si $(\alpha,\beta)\in(\Q/\Z)^2-(0,0)$, et si
$(a,b)\in\Q^2$ a pour image $(\alpha,\beta)\in(\Q/\Z)^2$,
posons
$$ g_{c,\alpha,\beta}=\theta(\tau,a\tau+b)^{c^2}
\theta(\tau,ca\tau+cb)^{-1}.$$

\begin{prop}\phantomsection\label{p15}
Soient $\alpha,\beta\in\frac{1}{N}\Z/\Z$.

{\rm (i)} Si $c\in\N$ est premier \`a $6$ et si $(c\alpha,c\beta)\neq(0,0)$, alors
il existe $\tilde g_{c,\alpha,\beta}\in\O(Y(N))^\dual$, unique, dont le pull-back par
${\cal H}^+\to (\Gamma(N)\backslash{\cal H}^+)\times\{1\}\hookrightarrow Y(N)(\C)$ est
$g_{c,\alpha,\beta}$.

{\rm (ii)} L'\'el\'ement $\tilde g_{\alpha,\beta}=\tilde g_{c,\alpha,\beta}^{1/(c^2-1)}$
de $\Q\otimes \iO^\dual$ ne d\'epend pas du choix
de $c$ congru \`a $1$ modulo~$N$.  De plus,
quel que soit $c$ premier \`a $6$,
on a $\tilde g_{c,\alpha,\beta}=\tilde g_{\alpha,\beta}^{c^2}\,\tilde g_{c\alpha,c\beta}^{-1}$.
\end{prop}

\begin{rema}\phantomsection\label{t3.1}
{\rm (i)}
Si $\alpha$ et $\beta$ sont les
images de $\frac{a}{N}$ et $\frac{b}{N}$, avec $0\leq \frac{a}{N}<1$,
le $q$-d\'eveloppement de 
$\tilde g_{\alpha,\beta}$ sur $(\Gamma(N)\backslash{\cal H}^+)\times\{1\}$
est: 
$$g_{\alpha,\beta}=q^{B_2({\frac{a}{N})}}
\prod_{n\geq 0}(1-q^nq^{\frac{a}{N}}\zeta_N^b)
\prod_{n\geq 1}(1-q^nq^{-\frac{a}{N}}\zeta_N^{-b}),
\quad{\text{o\`u $B_2(x)=\tfrac{x^2}{2}-\tfrac{x}{2}+\tfrac{1}{12}$.}}$$

{\rm (ii)} On a 
$$\tilde g_{\alpha,\beta}=\tilde g_{-\alpha,-\beta}$$
et
les relations de distribution suivantes:
$$\prod_{e\alpha'=\alpha,\,e\beta'=\beta}\tilde g_{\alpha',\beta'}=\tilde g_{\alpha,\beta},
\quad
\prod_{e\beta'=\beta}\tilde g_{\alpha,\beta'}=\matrice{1}{0}{0}{e}^{]\infty[}\star \tilde g_{\alpha,\beta}.$$

{\rm (iii)} Le $\Q$-module $\iO^\dual\otimes\Q$ est engendr\'e par les $\tilde g_{\alpha,\beta}$,
le module des relations \'etant engendr\'e par les relations du (ii).

{\rm (iv)}
On a
$$\matrice{a}{b}{c}{d}\star \tilde g_{\alpha,\beta}=\tilde g_{d\alpha-c\beta,-b\alpha+a\beta},
\hskip.2cm {\text{si $\matrice{a}{b}{c}{d}\in \GG(\cZ)$.}}$$
\end{rema}

\index{Mat@\MAT}Soit 
${\bf M}'_{2,1}(\A^{]\infty[})={\bf M}_{2,1}(\A^{]\infty[})\moins\big\{\vecteur{0}{0}\big\}$.
Alors ${\bf M}'_{2,1}(\A^{]\infty[})$ est stable par multiplication \`a gauche par $\GG(\A^{]\infty[})$. 

\begin{theo}\phantomsection\label{t3}
Il \index{Zsie@\zsie}existe\footnote{Si $X$ est un espace localement profini, et si $\Lambda$ est un anneau,
on note ${\rm LC}_c(X,\Lambda)$ le $\Lambda$-module des fonctions localement constantes
\`a support compact et, si $V$ est un $\Lambda$-module, alors
$\dalg(X,V)={\rm Hom}({\rm LC}_c(X,\Lambda),V)$ est l'espace des distributions
alg\'ebriques sur $X$ \`a valeurs dans $V$.} 
$${\bf z}_{\rm Siegel}\in \dalg({\bf M}'_{2,1}(\A^{]\infty[}),
\Q\otimes\iO^\dual),$$
telle que, quels que soient $r\in\Q_+^\dual$ et $(a,b)\in\Q^2-r\Z^2$, on
ait
$$\int_{\vecteur{a+r\cZ}{b+r\cZ}}{\bf z}_{\rm Siegel}= \tilde g_{\frac{-b}{r},\,\frac{a}{r}}
= \tilde g_{\frac{b}{r},\,\frac{-a}{r}}.$$
De plus, ${\bf z}_{\rm Siegel}$ est invariante sous l'action
naturelle de $\GG(\A^{]\infty[})$.
\end{theo}
\begin{proof}
L'existence de ${\bf z}_{\rm Siegel}$ r\'esulte de la premi\`ere des relations de
distribution du (ii) de la rem.~\ref{t3.1}.

Pour prouver l'invariance par $\GG(\A^{]\infty[})$, il s'agit de prouver que, si $U$ est un
ouvert de ${\bf M}'_{2,1}(\A^{]\infty[})$ et si $\gamma\in \GG(\A^{]\infty[})$,
alors $$\int_{\gamma U}{\bf z}_{\rm Siegel}=\gamma\star\int_U{\bf z}_{\rm Siegel}.$$
Par lin\'earit\'e, il suffit de le prouver pour $U=\vecteur{x+r\cZ}{y+r\cZ}$, avec $\vecteur{x}{y}\notin
r{\bf M}_{2,1}(\cZ)$, et comme $\GG(\A^{]\infty[})$ est engendr\'e, d'apr\`es la th\'eorie des diviseurs
\'el\'ementaires, par
$\GG(\cZ)$ et par 
$\big\{\matrice{a}{0}{0}{b}^{]\infty[},\ a,b\in\Q_+^\dual,\  \frac{b}{a}\in\N\big\}$,
il suffit de le prouver pour $\gamma$ d'une des formes ci-dessus.

$\bullet$ Si $\gamma=\matrice{a}{b}{c}{d}\in \GG(\cZ)$, alors $\gamma U=\vecteur{ax+by+r\cZ}{cx+dy+r\cZ}
$, et donc $\int_{\gamma U}=\tilde g_{\frac{-cx-dy}{r},\frac{ax+by}{r}}$.
Tandis que $\gamma\star\int_U{\bf z}_{\rm Siegel}=
\gamma\star \tilde g_{\frac{-y}{r},\frac{x}{r}}=\tilde g_{\frac{-dy-cx}{r},\frac{by+ax}{r}}$.

$\bullet$ Si $\gamma=\matrice{a}{0}{0}{b}^{]\infty[}$, 
on peut supposer $a=1$ car le r\'esultat est trivial pour
$\matrice{a}{0}{0}{a}$, et alors $b\in\N$.  Alors
$$\gamma U=\vecteur{x+r\cZ}{by+br\cZ}=\sqcup_{i=0}^{b-1}
\vecteur{x+i+br\cZ}{by+br\cZ},$$
et donc, d'apr\`es la seconde relation de distribution du (ii) de la remarque~\ref{t3.1},
$$\int_{\gamma U}{\bf z}_{\rm Siegel}=
\prod_{i=0}^{b-1}\tilde g_{\frac{-y}{r},\frac{x+i}{br}}=\matrice{1}{0}{0}{b}\star \tilde g_{\frac{-y}{r},\frac{x}{r}},$$
ce qui permet de conclure.
\end{proof}

\subsubsection{Th\'eorie de Kummer et l'\'el\'ement ${\bf z}_{\rm Kato}$}\label{s43}
Comme la distribution ${\bf z}_{\rm Siegel}$ est invariante
par $\GG(\A^{]\infty[})$, elle l'est aussi par $\Pi'_\Q$ 
qui agit \`a travers
$\GG(\A^{]\infty[})$ et on note $${\rm Kum}({\bf z}_{\rm Siegel})\in
H^1(\Pi'_\Q,\dalg({\bf M}'_{2,1}\big(\A^{]\infty[}\big),\Q_p(1)))$$
son image par
l'application de Kummer\footnote{\label{kummer}
Soit $Z^0=\{(x_n)_{n\in\N},\  x_n\in {\cal M}(\Qbar)^\dual,
\ x_{n+1}^p=x_n
{\text{ si $n\in\N$}}\}$.  Soit $Z=\Q\otimes Z^0$.
Alors $Z$ est muni d'une action de $\Pi'_\Q$ et la suite
$0\to \Q_p(1)\to Z\to \iO^\dual\otimes\Q\to 0$
est une suite exacte de $\Pi'_\Q$-modules.  Posons
$X={\bf M}_{2,1}(\A^{]\infty[})$ et soit $(\phi_i)_{i\in I}$
une base de ${\rm LC}_c(X,\Z)$ sur $\Z$.  On peut fabriquer une distribution
alg\'ebrique $\mu$ sur $X$, \`a valeurs dans $Z$, en prenant pour
$\int_X\phi_i\,\mu$ n'importe quel rel\`evement dans $Z$ de
$\int_X\phi_i\, {\bf z}_{\rm Siegel}$ et alors
${\rm Kum}({\bf z}_{\rm Siegel})$ est l'image du cocycle
$\sigma\mapsto\mu\star\sigma-\mu$.}.

\index{Mat@\MAT}Soit $${\bf M}_2'\big(\A^{]\infty[}\big)=\GG(\Q_p)\times {\bf M}_2\big(\A^{]\infty,p[}\big)
\subset {\bf M}'_{2,1}\big(\A^{]\infty[}\big)\times
{\bf M}'_{2,1}\big(\A^{]\infty[}\big)\subset {\bf M}_{2}\big(\A^{]\infty[}\big).$$
Alors $\GG(\A^{]\infty[})$ agit par multiplication \`a gauche et \`a droite sur
${\bf M}_2'\big(\A^{]\infty[}\big)$.
On \index{Zkato@\zkato}d\'efinit
$${\bf z}_{\rm Kato}=-{\rm Kum}({\bf z}_{\rm Siegel})\otimes {\rm Kum}({\bf z}_{\rm Siegel})
\in  H^2(\Pi'_\Q, \dalg({\bf M}_2'(\A^{]\infty[}),\Q_p(2))).$$
Par construction, on a
$$\int_{\matrice{a+M\cZ}{b+N\cZ}{c+M\cZ}{d+N\cZ}}{\bf z}_{\rm Kato}=
{\rm Kum}(\tilde g_{\frac{d}{N},\frac{-b}{N}})\cup {\rm Kum}(\tilde g_{\frac{-c}{M},\frac{a}{M}}),
\quad{\text{si $\matrice{a+M\cZ}{b+N\cZ}{c+M\cZ}{d+N\cZ}\subset {\bf M}_2'\big(\A^{]\infty[}\big)$.}}$$

\begin{prop}\phantomsection\label{inva1}
On a
\begin{align*}\matrice{u}{0}{0}{v}^{]\infty[}\star {\bf z}_{\rm Kato}=&\ {\bf z}_{\rm Kato},\quad
{\text{si $u,v\in \Q^\dual$.}}\\
\matrice{0}{1}{1}{0}\star {\bf z}_{\rm Kato}=&\ -{\bf z}_{\rm Kato}.
\end{align*}
\end{prop}
\begin{proof}
On a $\int_U g\star {\bf z}_{\rm Kato}=\int_{Ug}{\bf z}_{\rm Kato}$.
Comme 
$$\matrice{a+N\cZ}{b+M\cZ}{c+N\cZ}{d+M\cZ}\matrice{u}{0}{0}{v}=
\matrice{ua+uN\cZ}{vb+vM\cZ}{uc+uN\cZ}{vd+vM\cZ},$$ 
la premi\`ere \'egalit\'e se d\'eduit de l'identit\'e
$$\int_{\matrice{ua+uN\cZ}{vb+vM\cZ}{uc+uN\cZ}{vd+vM\cZ}}={\rm Kum}(\tilde g_{\frac{vd}{vM},\frac{-vb}{vM}})
\cup {\rm Kum}(\tilde g_{\frac{-uc}{uN},\frac{ua}{uN}})=
\int_{\matrice{a+N\cZ}{b+M\cZ}{c+N\cZ}{d+M\cZ}}{\bf z}_{\rm Kato}$$
La seconde \'egalit\'e se prouve de m\^eme, le signe venant de ce que le cup-produit
est altern\'e.
\end{proof}

\subsubsection{L'\'el\'ement ${\bf z}_{\rm Kato}^{c,d}$}\label{entier}
Pour aller plus loin, il faut se d\'ebarrasser des d\'enominateurs
dans la distribution~${\bf z}_{\rm Kato}$.
Notons que $\cZ$ agit naturellement sur $\Q/\Z=\A^{]\infty[}/\cZ$.
Maintenant, il r\'esulte du (i) de la prop.~\ref{p15} que, si $c\in\cZ^\dual$,
alors 
$c^2\otimes \tilde g_{\alpha,\beta}-\tilde g_{c\alpha,c\beta}\in\cZ\otimes\iO^\dual$
(appliquer la prop.~\ref{p15} \`a une suite
de $c_n\in\N$ v\'erifiant $c_n\to c$ dans $\cZ$),
et donc
$$_c\tilde g_{\alpha,\beta}=c_p^2\otimes \tilde g_{\alpha,\beta}-\tilde g_{c\alpha,c\beta}
\in\Z_p\otimes\iO^\dual\subset
\Q_p\otimes\iO^\dual.$$
La matrice $(c)\in{\bf M}_1(\A^{]\infty[})$ agit sur ${\bf M}'_{2,1}(\A^{]\infty[})$
par multiplication \`a droite, et donc aussi de mani\`ere compatible sur les distributions
sur cet espace.
Il r\'esulte de ce qui pr\'ec\`ede
que $$(c_p^2-(c))\star{\bf z}_{\rm Siegel}\in
\dalg({\bf M}'_{2,1}(\A^{]\infty[}),\Z_p\otimes\iO^\dual)$$
(l'int\'egrale sur $\vecteur{a+r\cZ}{b+r\cZ}$ est $_c\tilde g_{\frac{-b}{r},\frac{a}{r}}$),
et donc
que 
$$(c_p^2-(c))\star {\rm Kum}({\bf z}_{\rm Siegel})\in 
H^1(\Pi'_\Q,\dalg({\bf M}'_{2,1}(\A^{]\infty[}),\Z_p(1))).$$ 
Il s'ensuit \index{Zkato@\zkato}que
$${\bf z}_{{\rm Kato}}^{c,d}:=
\big(c_p^2-\matrice{c}{0}{0}{1}\big)\big(d_p^2-\matrice{1}{0}{0}{d}\big)
\star{\bf z}_{\rm Kato}
\in H^2\big(\Pi'_\Q, \dalg\big({\bf M}'_2(\A^{]\infty[}),\Z_p(2)\big)\big).$$
L'int\'er\^et d'avoir supprim\'e les d\'enominateurs est de
pouvoir int\'egrer des fonctions continues \`a support compact et
pas seulement des fonctions localement constantes (en d'autres termes,
$\dalg({\bf M}'_2(\A^{]\infty[}),\Z_p(2))$ est l'espace des mesures
$\cdo({\bf M}'_2(\A^{]\infty[}),\Z_p(2))$ sur ${\bf M}'_2(\A^{]\infty[})$,
\`a valeurs dans $\Z_p(2)$).

\Subsection{Torsion \`a la Soul\'e}\label{sou1}
\subsubsection{Formalisme g\'en\'eral}\label{sou2}
Soient $G'$ un groupe et $X'$ un ensemble muni d'actions \`a droite
$(g',x')\mapsto x'g'$ et \`a gauche $(g',x')\mapsto g'x'$ de $G'$
commutant entre elles (i.e.~si $x'\in X'$ et $g'_1,g'_2,g'_3,g'_4\in G'$,
alors $g'_1g'_2x'g'_3g'_4$ ne d\'epend pas de l'ordre
dans lequel on fait les op\'erations).

Soit $G_p$ un groupe, et soient $G=G_p\times G'$, $X=G_p\times X'$.
On note $g=(g_p,g')$ et $x=(x_p,x')$ les \'el\'ements de $G$ et $X$,
et on fait agir $G$ sur $X$ \`a gauche par
$((g_p,g'),(x_p,x'))\mapsto(g_px_p,g'x')$ et \`a droite par
$((g_p,g'),(x_p,x'))\mapsto(x_pg_p,x'g')$; ces deux actions commutent.

Soient $H_1,H_2$ des groupes, munis de morphismes $\iota_1:H_1\to G$ et $\iota_2:H_2\to G$.
Si $h_i\in H_i$, on note simplement $h_i$ l'\'el\'ement $\iota_i(h_i)$ de $G$.
Soit $H=H_1\times H_2$.

Soit $ A=\Z_p,\Q_p,\O_L, L,...$.  Soit $\Phi={\cal C}_c(X,A)$
et $\Phi^\dual$ son dual (ou un sous-espace de son dual).
Si $\alpha\in{\cal C}(X,A)$,
on d\'efinit $\alpha\mu\in\Phi^\dual$, si $\mu\in\Phi^\dual$
par la formule $\langle \alpha\mu,\phi\rangle=\langle \mu,\alpha\phi\rangle$.

On munit $\Phi$ d'une action de $H=H_1\times H_2$, en posant $\big((h_1,h_2)\cdot\phi\big)(x)=\phi(h_1^{-1}xh_2)$.
On munit $\Phi^\dual$ de l'action duale: $\langle h\cdot\mu,\phi\rangle=\langle \mu,h^{-1}\cdot \phi\rangle$,
si $h\in H$, $\mu\in\Phi^\dual$ et $\phi\in\Phi$.

Soit $W$ une $ A$-repr\'esentation de rang fini de $G_p$, et soit $W^\dual$ sa duale:
$$\langle g\cdot v^\dual,v\rangle=\langle v^\dual,g^{-1}\cdot v\rangle, \quad{\text{
si $g\in G_p$, $v^\dual\in W^\dual$ et $v\in W$.}}$$
On note $W_{H_i}, W^\dual_{H_i}$, pour $i=1,2$, les repr\'esentations de $H_i$ obtenues
via le morphisme $\iota_i:H_i\to G\to G_p$.  On les voit aussi comme des repr\'esentations
de $H$ en faisant agir $H_{3-i}$ trivialement.

Fixons une base $e_i^\dual$ de $W^\dual$ sur $A$.
Si $w^\dual\in W^\dual$, soient $$\phi_{w^\dual}:X\to W^\dual,\quad
\phi_{w^\dual,i}:X\to A,\qquad
\phi_{w^\dual}(x)=x_p\cdot w^\dual=\sum \phi_{w^\dual,i}(x)e_i^\dual.$$  

Soit $V$ une repr\'esentation de $H_1$; on consid\`ere
$V$ comme une repr\'esentation de $H$ en faisant agir $H_2$ trivialement.
Si 
$\mu\in \Phi^\dual$, $v\in V$, et $w^\dual\in W_{H_2}^\dual$, on d\'efinit
$$\alpha_W(\mu\otimes v\otimes w^\dual)=\sum_i(\phi_{w^\dual,i}\mu)\otimes v\otimes e_i^\dual
\in \Phi^\dual\otimes(V\otimes W_{H_1}^\dual).$$
\begin{lemm}\phantomsection\label{sou3}
{\rm (i)}
$\alpha_W:(\Phi^\dual\otimes V)\otimes W_{H_2}^\dual\to \Phi^\dual\otimes(V\otimes W_{H_1}^\dual)$
est $H$-\'equivariante.

{\rm (ii)} $\alpha_W$ induit un morphisme $H_2$-\'equivariant
$$\alpha_W:H^i(H_1,\Phi^\dual\otimes V)\otimes W_{H_2}^\dual\to H^i(H_1,\Phi^\dual\otimes(V\otimes W_{H_1}^\dual)).$$
\end{lemm}
\begin{proof}
Le (ii) est une cons\'equence imm\'ediate du (i).
Pour prouver le~(i),
il suffit de regarder
ce qui se passe sur les tenseurs \'el\'ementaires de la forme
$\delta_a\otimes v\otimes w^\dual$, o\`u $\delta_a$ est la masse de Dirac en $a=(a_p,a')\in X$.
On note $w^\dual_{H_i}$ l'\'el\'ement $w^\dual$ de $W^\dual_{H_i}$.
On a $$\alpha_W(\delta_a\otimes v\otimes w^\dual_{H_2})
=\delta_a\otimes v\otimes (a_p\cdot w^\dual)_{H_1}$$
On en d\'eduit
\begin{align*}
\alpha_W\big((h_1,h_2)\cdot(\delta_a\otimes v\otimes w^\dual_{H_2})\big)=&\
\alpha_W\big(\delta_{h_1ah_2^{-1}}\otimes (h_1\cdot v)\otimes (h_2\cdot w^\dual)_{H_2}\big)\\ =&\
\delta_{h_1ah_2^{-1}}\otimes (h_1\cdot v)\otimes ((h_1ah_2^{-1})_ph_2\cdot w^\dual)_{H_1}\\
(h_1,h_2)\cdot\alpha_W(\delta_a\otimes v\otimes w^\dual_{H_2}))=&\
(h_1,h_2)\cdot\big(\delta_a\otimes v\otimes (a_p\cdot w^\dual)_{H_1}\big)\\ =&\ 
\delta_{h_1ah_2^{-1}}\otimes (h_1\cdot v)\otimes (h_1a_p\cdot w^\dual)_{H_1}
\end{align*}
D'o\`u le r\'esultat (on a $(h_1ah_2^{-1})_p=h_1a_ph_2^{-1}$).
\end{proof}

Si $v\in W$ et $v^\dual\in W^\dual$, soit 
$$\phi_{v^\dual,v}\in\Phi, \qquad
\phi_{v^\dual,v}(g)=\langle g\cdot v^\dual, v\rangle=\langle \phi_{v^\dual}(g),v\rangle.$$
L'application $v^\dual\otimes v\mapsto \phi_{v^\dual,v}$ induit un morphisme
$H$-\'equivariant de $W^\dual_{H_2}\otimes W_{H_1}$ dans $\Phi$: si $h=(h_1,h_2)$, alors
on a $h\cdot (v^\dual\otimes v)=(h_2\cdot v^\dual)\otimes (h_1\cdot v)$,
et donc 
$$\phi_{h\cdot (v^\dual\otimes v)}(g)=\langle gh_2\cdot v^\dual,h_1\cdot v\rangle=
\langle h_1^{-1}gh_2\cdot v^\dual, v\rangle=\big(h\cdot\phi_{v^\dual\otimes v}\big)(g).$$

\begin{coro}\phantomsection\label{sou4}
Les deux morphismes $H$-\'equivariants ci-dessous sont \'egaux:
$$
\xymatrix@R=.2cm@C=.3cm
{(\Phi^\dual{\otimes} V){\otimes} W^\dual_{H_2}{\otimes} W_{H_1}{\otimes}\Phi\ar[r] &
(\Phi^\dual{\otimes} V){\otimes}\Phi\ar[r] & V,\\ (\mu{\otimes} v){\otimes} w^\dual{\otimes} w{\otimes} \phi\ar@{|->}[r] &
(\mu{\otimes} v){\otimes}(\phi_{w^\dual,w}\phi)\ar@{|->}[r] & \langle\mu,\phi_{w^\dual,w}\phi\rangle\,v\\
(\Phi^\dual{\otimes} V){\otimes} W^\dual_{H_2}{\otimes} W_{H_1}{\otimes}\Phi\ar[r] &
(\Phi^\dual{\otimes} V{\otimes} W^\dual_{H_2}){\otimes}(\Phi{\otimes} W_{H_1})\ar[r] & V,
\\ (\mu{\otimes} v){\otimes} w^\dual{\otimes} w{\otimes} \phi\ar@{|->}[r] &
\alpha_W(\mu{\otimes} v{\otimes} w^\dual){\otimes}(\phi{\otimes} w)\ar@{|->}[r] 
& \langle\alpha_W(\mu{\otimes} w^\dual),\phi{\otimes} w\rangle\,v}
$$
\end{coro}

\subsubsection{La distribution ${\bf z}_{\rm Kato}^{c,d}(k,j)$}\label{s46}
On va appliquer ce qui pr\'ec\`ede \`a 
\begin{align*}
X'={\bf M}_2(\A^{]\infty,p[}),\quad
X={\bf M}'_2(\A^{]\infty[}),\quad &
G'=\GG(\A^{]\infty,p[}), 
\quad G_p=\GG(\Q_p)\\
G=H_2=\GG(\A^{]\infty[}),\quad H_1=\Pi'_{\Q}, \quad & W=W_{k,j}
\quad{\rm et} \quad V=\Q_p(2)\\
\Phi={\cal C}_c({\bf M}'_2(\A^{]\infty[}),L),\quad & \Phi^\dual=\cdo({\bf M}'_2(\A^{]\infty[}),L)
\end{align*}

On \index{Zkatokj@\zkatokj}d\'efinit
$${\bf z}_{{\rm Kato}}^{c,d}(k,j):=
\alpha_{W_{k,j}}\big({\bf z}_{{\rm Kato}}^{c,d}\otimes\tfrac{(e_2^\dual)^k}{(e_1^\dual\wedge e_2^\dual)^j}\big)
\in H^2(\Pi'_\Q,\cdo({\bf M}'_2(\A^{]\infty[}),W^\dual_{k,j}(2))).$$
On note encore ${\bf z}_{{\rm Kato}}^{c,d}(k,j)$ la restriction \`a $\Pi_\Q\subset\Pi'_\Q$.

\vskip.2cm
On a $H^4_c(\Pi_{\Q_p},\Q_p(2))=H^4_{{\eet},c}(Y(1)_{\Q_p},\Q_p(2))=\Q_p$; on en d\'eduit
des accouplements
\begin{align*}
\langle\ ,\rangle:&\ \big(H^2(\Pi_{\Q_p},\Phi^\dual\otimes \Q_p(2))\otimes W^\dual_{H_2}\big) \times 
\big(H^2_c(\Pi_{\Q_p},\Phi\otimes W_{H_1})\big)\to L\\
\langle\ ,\rangle:&\ \big(H^2(\Pi_{\Q_p},\Phi^\dual\otimes \Q_p(2))\big) \times 
\big(H^2_c(\Pi_{\Q_p},\Phi\otimes W_{H_1})\otimes W^\dual_{H_2}\big)\to L
\end{align*}
Le cor.~\ref{sou4} fournit alors
le r\'esultat suivant:
\begin{lemm}\phantomsection\label{expo1}
Si $\phi\in H^2_c(\Pi_\Q,{\cal C}_c({\bf M}'_2(\A^{]\infty[}),W_{k,j}))$, alors
$$\langle {\bf z}_{{\rm Kato}}^{c,d}, \phi\otimes \tfrac{(e_2^\dual)^k}{(e_1^\dual\wedge e_2^\dual)^j}\rangle=
\langle {\bf z}_{{\rm Kato}}^{c,d}(k,j), \phi\rangle.$$
\end{lemm}

Si $S$ est un ensemble fini de nombres premiers contenant~$p$,
on \index{Zkato@\zkato}\index{Zkatokj@\zkatokj}d\'efinit
\begin{align*}
{\bf z}_{{\rm Kato}}^{S,c,d} &\in H^2(\Pi'_\Q,\cdo(\GG(\Q_S),\Q_p(2)))\\
{\bf z}_{{\rm Kato}}^{S,c,d}(k,j) &\in H^2(\Pi'_\Q,\cdo(\GG(\Q_S),W^\dual_{k,j}(2)))
\end{align*}
par les formules (o\`u $\phi_S\in {\cal C}_c(\GG(\Q_S),L)$)
\begin{align*}
\int_{\GG(\Q_S)} \phi_S\, {\bf z}_{{\rm Kato}}^{S,c,d}
=&\ \int \phi_S\otimes {\bf 1}_{{\bf M}_2(\cZ^{]S[})}\,{\bf z}_{{\rm Kato}}^{S,c,d}\\
\int_{\GG(\Q_S)} \phi_S\, {\bf z}_{{\rm Kato}}^{S,c,d}(k,j)
=&\ \int \phi_S\otimes {\bf 1}_{{\bf M}_2(\cZ^{]S[})}
\,{\bf z}_{{\rm Kato}}^{S,c,d}(k,j)
\end{align*}

\begin{rema}\phantomsection\label{expo1.1}
On peut aussi d\'efinir ${\bf z}^S_{{\rm Siegel}}$,  ${\bf z}^S_{{\rm Eis}}$,
 ${\bf z}^{S,'}_{{\rm Eis}}$,
${\bf z}_{{\rm Kato}}^{S,c,d}$ et ${\bf z}_{{\rm Kato}}^{S,c,d}(k,j)$ directement,
en partant des formules:
$$\int_{\vecteur{a+r\Z_S}{b+r\Z_S}}{\bf z}_{{\rm Siegel}}^S= \tilde g_{\frac{-b}{r},\,\frac{a}{r}}
= \tilde g_{\frac{b}{r},\,\frac{-a}{r}},\quad{\text {etc...}}$$
si $r\in \Z[\frac{1}{S}]^\dual$, $r>0$, et $(a,b)\in\Z[\frac{1}{S}]^2\moins r\Z^2$.
Cela permet de prouver qu'en fait toutes ces distributions sont obtenues par inflation
\`a partir de $\Pi'_{\Q,S}$
\end{rema}

\Subsection{La loi de r\'eciprocit\'e explicite de Kato}\label{recex3}
\subsubsection{La suite spectrale de Hochschild-Serre}\label{recex4}
La suite spectrale de Hochshild-Serre fournit des isomorphismes
\begin{align*}
H^2(\Pi'_\Q,\cdo({\bf M}'(\A^{]\infty[}),\Z_p(2)))&\cong
H^1(G_{\Q},H^1(\Pi'_{\Qbar},\cdo({\bf M}'(\A^{]\infty[}),\Z_p(2))))\\
H^2(\Pi'_{\Q,S},\cdo(\GG(\Q_S),\Z_p(2)))&\cong
H^1(G_{\Q,S},H^1(\Pi'_{{\Qbar},S},\cdo(\GG(\Q_S),\Z_p(2))))
\end{align*}
On peut donc consid\'erer \index{Zkato@\zkato}\index{Zkatokj@\zkatokj}que:
\begin{align*}
{\bf z}_{\rm Kato}^{c,d}&\in
H^1(G_{\Q},H^1(\Pi'_{\Qbar},\cdo({\bf M}'(\A^{]\infty[}),\Z_p(2))))\\
{\bf z}_{\rm Kato}^{S,c,d}&\in
H^1(G_{\Q,S},H^1(\Pi'_{{\Qbar},S},\cdo(\GG(\Q_S),\Z_p(2))))
\end{align*}

\subsubsection{La distribution $\tilde{\bf z}_{\rm Eis}(k,j)$}\label{recex5}
On voit $E^{(k)}_{\alpha,\beta}$ comme un \'el\'ement de $M^{\rm cl}_{k,0}(\Q^{\rm cycl})$
et $F^{(k)}_{\alpha,\beta}$ comme un \'el\'ement de $M^{\rm cl}_{k,k-1}(\Q^{\rm cycl})$.
On \index{Zeis@\zeis}note $\widetilde E^{(k)}_{\alpha,\beta}\in \iH^0(\omega^{k,0})$
et $\widetilde F^{(k)}_{\alpha,\beta}\in \iH^0(\omega^{k,k-1})$
les \'el\'ements correspondant \`a $(-2i\pi)^kE^{(k)}_{\alpha,\beta}$
et $(-2i\pi)\,F^{(k)}_{\alpha,\beta}$ via les isomorphismes de la rem.\,\ref{ratio1}
et du \no\ref{como4}.
(Si $k=2$, il faut supposer $(\alpha,\beta)\neq (0,0)$
pour $F^{(k)}$ et remplacer $\iH^0(\omega^{2,0})$ par $\iH^0(W_{\rm dR}^{2,0})$ pour $E^{(k)}$.)
\begin{prop}\phantomsection\label{p2}
Si $\matrice{a}{b}{c}{d}\in
{\bf GL}_2(\cZ)$, 
 et si $(\alpha,\beta)\in(\Q/\Z)^2$, alors
$$\matrice{a}{b}{c}{d}\star \widetilde E^{(k,j)}_{\alpha,\beta}= \widetilde E^{(k,j)}_{d\alpha-c\beta,-b\alpha+a\beta}
\quad{\rm et}\quad
\matrice{a}{b}{c}{d}\star \widetilde F^{(k,j)}_{\alpha,\beta}= \widetilde F^{(k,j)}_{d\alpha-c\beta,-b\alpha+a\beta}.$$
\end{prop}

On note $\tilde {\bf z}_{\rm Eis}(k)$ et $\tilde {\bf z}'_{\rm Eis}(k)$
les distributions obtenues en rempla\c{c}ant $E^{(k)}_{\alpha,\beta}$
et $F^{(k)}_{\alpha,\beta}$ par $\widetilde E^{(k)}_{\alpha,\beta}$ et
$\widetilde F^{(k)}_{\alpha,\beta}$ dans les d\'efinitions de
${\bf z}_{\rm Eis}(k)$ et ${\bf z}'_{\rm Eis}(k)$.
La preuve du th.\,\ref{t3} s'adapte et nous donne le r\'esultat suivant.

\begin{theo}\phantomsection\label{t1.5}
Les distributions $\tilde {\bf z}_{\rm Eis}(k)$ et $\tilde {\bf z}'_{\rm Eis}(k)$
sont invariantes sous l'action naturelle de $\GG(\A^{]\infty[})$.
\end{theo}

Si $k\geq 1$, alors (y compris si $k=2$):
\begin{align*}
\big(c_p^2-c_p^k(c))\star \tilde{\bf z}_{\rm Eis}(k)
&\in \dalg({\bf M}_{2,1}(\A^{]\infty[}),\iH^0(\omega^{k,0}))\\
\big(c_p^2-c_p^{2-k}(c))\star \tilde{\bf z}_{\rm Eis}'(k)
&\in \dalg({\bf M}'_{2,1}(\A^{]\infty[}),\iH^0(\omega^{k,k-1}))
\end{align*}
Le produit fournit des fl\`eches
\begin{align*}
\iH^0(\omega^{j,0})\times \iH^0(\omega^{k+2-j,k+1-j})\to \iH^0(\omega^{k+2,k+1-j})\\ 
\iH^0(W_{\rm dR}^{j,0})\times \iH^0(\omega^{k+2-j,k+1-j})\to \iH^0(W_{\rm dR}^{k+2,k+1-j})
\end{align*}
En utilisant l'inclusion 
${\bf M}'_{2}(\A^{]\infty[})\subset {\bf M}_{2,1}(\A^{]\infty[})\times {\bf M}'_{2,1}(\A^{]\infty[})$,
cela permet de construire, si $k\geq 2$ et $1\leq j\leq k+1$, des distributions
\begin{align*}
\tilde {\bf z}_{\rm Eis}(k,j)  =
\tfrac{(-1)^j}{(j-1)!}{\bf z}_{\rm Eis}(j)\otimes{\bf z}_{\rm Eis}'(k+2-j)
&\in
\dalg({\bf M}'_2(\A^{]\infty[}),\iH^0(W_{\rm dR}^{k+2,k+1-j}))\\
\tilde {\bf z}^{c,d}_{\rm Eis}(k,j)  =\big(c_p^2-c_p^{j}\matrice{c}{0}{0}{1}\big)
\big(d_p^2-d_p^{j-k}\matrice{1}{0}{0}{d}\big)&\star {\bf z}_{\rm Eis}(k,j)\\ &\in
\dalg({\bf M}'_2(\A^{]\infty[}),\iH^0(\omega^{k+2,k+1-j}))
\end{align*}
La valeur de ces distributions sur une fonction localement
constante \`a support compact est une combinaison lin\'eaire de
produits de s\'eries d'Eisenstein.
Comme pour ${\bf z}_{\rm Kato}$, si $S$ est un ensemble fini de nombres premiers
contenant $p$, on d\'efinit:
\begin{align*}
\tilde {\bf z}^S_{\rm Eis}(k,j)&\in \dalg(\GG(\Q_S),\iH^0(W_{\rm dR}^{k+2,k+1-j})_S),\\
\tilde {\bf z}^{S,c,d}_{\rm Eis}(k,j)&\in \dalg(\GG(\Q_S),\iH^0(\omega^{k+2,k+1-j})_S)
\end{align*}

\subsubsection{L'exponentielle duale de Bloch-Kato}\label{recex6}
Accoupler ${\bf z}_{\rm Kato}^{c,d}(k,j)$ avec une fonction $\phi$ localement constante,
\`a support compact dans ${\bf M}'_2(\A^{]\infty[})$, produit 
$$\langle {\bf z}_{\rm Kato}^{c,d}(k,j),\phi\rangle
\in H^1(G_{\Q},\iH^1_{\eet}((W^{\eet}_{k,j})^\dual(2)))$$ 

On a aussi
$(W^{\eet}_{k,j})^\dual(2)=W^{\eet}_{k,k-j+2}$.
L'exponentielle duale de Bloch-Kato fournit une fl\`eche
$$H^1(G_{\Q_p},\iH^1_{\eet}(W^{\eet}_{k,k-j+2}))\to
{\rm Fil}^0 D_{\rm dR}(\iH^1_{\eet}(W_{k,k-j+2}^{\eet})).$$
On a 
$$D_{\rm dR}(\iH^1_{\eet}(W_{k,k-j+2}^{\eet}))=\Q_p\otimes_\Q\iH^1_{\rm dR}(W_{k,k-j+2}^{\rm dR})
=\Q_p\otimes_\Q\iH^1_{\rm dR}(W_{k,k-j}^{\rm dR})\otimes\zeta_{\rm dR}^2$$
comme module filtr\'e, et la filtration sur $\iH^1_{\rm dR}(W_{k,k-j}^{\rm dR})$
a deux cran, le sous-objet \'etant $\iH^0(\omega^{k+2,k+1-j})$.
Les poids de Hodge-Tate de $\iH^1_{\eet}(W_{k,k+2-j}^{\eet})$ \'etant $k+2-j$ et $1-j$,
on a ${\rm Fil}^0(\iH^1_{\rm dR}(W_{k,k+2-j}^{\rm dR}))=
\iH^0(\omega^{k+2,k+1-j})\otimes\zeta_{\rm dR}^2$ si et seulement si
$1\leq j\leq k+1$.

La loi de r\'eciprocit\'e explicite de Kato~\cite[prop.\,10.10]{Ka4} (et~\cite[prop.\,11.7]{Ka4} pour
faire la comparaison entre deux application $\exp^\dual$) nous donne
\footnote{Si on veut utiliser les calculs de~\cite{bbk-Kato}, il faut faire attention
que ${\bf z}_{\rm Kato}$ a chang\'e ce qui \'echange les r\^oles de $b$ et $c$ (et rajoute des signes), 
et qu'on int\`egre $(be_1^\dual+de_2^\dual)^k$ au lieu
de $(ae_1^\dual+be_2^\dual)^k$, ce qui \'echange les r\^oles de $E$ et $F$. A la fin
on retombe sur la m\^eme formule.}:
\begin{theo}\phantomsection\label{recex1} Si $0\leq j\leq k$, alors
$$\exp^*({\bf z}_{\rm Kato}^{c,d}(k,j))=\tilde{\bf z}_{\rm Eis}^{c,d}(k,j)\otimes\zeta_{\rm dR}^2,\quad
\exp^*({\bf z}_{\rm Kato}^{S,c,d}(k,j))=\tilde{\bf z}_{\rm Eis}^{S,c,d}(k,j)\otimes\zeta_{\rm dR}^2.$$
\end{theo}

\section{Un avatar alg\'ebrique des produits de Rankin}\label{chapi5}
Dans ce chapitre, nous factorisons (rem.\,\ref{fact27}) le syst\`eme de Beilinson-Kato comme un produit
de deux symboles $(0,\infty)$, et nous terminons la preuve du th.\,\ref{geni21}
en prouvant (th.\,\ref{facto2}) que l'\'el\'ement ${\bf z}^S(\rho_T)$ n'a pas de d\'enominateur.
La preuve repose sur
l'\'evaluation de ${\bf z}_{\rm Kato}^{S,c,d}$ en des fonctions tests bien choisies (th.\,\ref{shi4},
avatar de~\cite[th.\,6.6]{Ka4},
dont la preuve repose sur la m\'ethode de Rankin); le r\'esultat fait appara\^{\i}tre un produit
de deux valeurs de fonctions~$L$. On en d\'eduit la factorisation 
voulue en un point classique (th.\,\ref{fact2}), et on obtient une factorisation en famille
(lemme~\ref{facto1}) par densit\'e des points classiques.
\Subsection{L'application exponentielle}\label{expo2}
\subsubsection{Une formule explicite g\'en\'erale}\label{expo3}
Si $X$ est un espace fonctionnel comme ${\rm LC}$, ${\rm LP}$, ${\cal C}$, ${\rm Mes}$, etc.,
et si $\sharp\in\{\ ,c,{\rm par}\}$, on pose
\begin{align*}
H_\sharp^1(\GG,X_L)&:= H^1_\sharp(\GG(\Q),X(\GG(\A),L))\\
H_\sharp^1(\GG,X_L)_S&:= H^1_\sharp(\GG(\Z[\tfrac{1}{S}]),X(\GG(\Q_S),L))
\end{align*}

On suppose $1\leq j\leq k+1$ (les poids de $\iH^1_{{\eet},c}(W_{k,j}^{\eet})_S$ sont
$j>0$ et $j-k-1\leq 0$).
On a
$$D_{\rm dR}(\iH^1_{{\eet},c}(W_{k,j}^{\eet})_S)=\Q_p\otimes_\Q\iH^1_{{\rm dR},c}(W_{k,j}^{\rm dR})_S$$
On en d\'eduit une application exponentielle de Bloch-Kato:
$$\exp:\iH^1_{{\rm dR},c}(W_{k,j}^{\rm dR})_S/{\rm Fil}^0\to H^1(G_{\Q_p},\iH^1_{{\eet},c}(W_{k,j}^{\eet})_S).$$
On peut utiliser l'isomorphisme 
$\iH^1_{{\eet},c}(W_{k,j}^{\eet})_S\cong H^1(\GG,{\rm LC}_L\otimes W_{k,j})_S$,
tensoriser par $W_{k,j}^\dual$ et utiliser l'injection naturelle
${\rm LC}_L\otimes W_{k,j}\otimes W_{k,j}^\dual\hookrightarrow {\cal C}$
pour fabriquer une application, encore appel\'ee
exponentielle de Bloch-Kato:
$$\exp:L\otimes_\Q(\iH^1_{{\rm dR},c}(W^{\rm dR}_{k,j})_S/{\rm Fil}^0)\otimes W_{k,j}^\dual
\to H^1(G_{\Q_p},H^1_c(\GG,{\cal C}_L)_S)$$
Notons que $H^1_c(\GG,{\cal C}_L)_S$ est la cohomologie compl\'et\'ee de la tour
des courbes modulaires de niveaux \`a support dans $S$.

Comme on l'a vu $H^1_c(\GG,{\cal C}_L)_S$ et $H^1(\GG,{\rm Mes}_L)_S$ sont en dualit\'e (rem.\,\ref{dual1}).
Si on veut inclure l'action de Galois, la dualit\'e naturelle est \`a valeurs dans
$H^2_c(X(1)^\times_\Qbar,\Q_p)=\Q_p(-1)$.  Il s'ensuit que
$H^1(\GG,{\rm Mes}_L)_S(2)$ est le dual de Tate de $H^1_c(\GG,{\cal C}_L)_S$.
Combin\'e avec la dualit\'e locale de Poitou-Tate, cela fournit une dualit\'e
$$H^1(G_{\Q_p},H^1(\GG,{\rm Mes}_L)_S(2))\times H^1(G_{\Q_p},H^1_c(\GG,{\cal C}_L)_S)\to \Q_p$$
En particulier, on peut accoupler
${\bf z}_{{\rm Kato}}^{S,c,d}$ (plut\^ot sa restriction \`a $G_{\Q_p}$) 
et $\exp(\Phi)\otimes \check v$, 
si 
$$\Phi\in \iH^1_{{\rm dR},c}(W^{\rm dR}_{k,j})_S/{\rm Fil}^0
\quad{\rm et}\quad \check v\in W_{k,j}^\dual.$$

\begin{prop}\phantomsection\label{expo3.1}
 Soit $\Phi\in \iH^1_{{\rm dR},c}(W^{\rm dR}_{k,j})_S/{\rm Fil}^0$.
Alors,
pour tout $N$ \`a support dans $S$
et tel que $\Phi\in H^1_{{\rm dR},c}(Y(N),W^{\rm dR}_{k,j})$, on a
\begin{align*}
\langle {\bf z}_{{\rm Kato}}^{S,c,d}, \exp(\Phi)\otimes \tfrac{(e_2^\dual)^k}{(e_1^\dual\wedge e_2^\dual)^j}\rangle_{\eet}=&\ \tfrac{(-1)^j}{(j-1)!}N^{k-2j}
\langle \Phi^{c,d},\widetilde E_{0,\frac{1}{N}}^{(j)} 
\widetilde F_{\frac{1}{N},0}^{(k+2-j)}\rangle_{{\rm dR},Y(N)}\otimes\zeta_{\rm dR}
\end{align*}
avec 
$$\Phi^{c,d}=\big(c_p^2-c_p^{j}\matrice{c^{-1}}{0}{0}{1}\big)
\big(d_p^2-d_p^{j-k}\matrice{1}{0}{0}{d^{-1}}\big)\star\Phi.$$
\end{prop}
\begin{proof}
Le lemme~\ref{expo1} fournit l'identit\'e:
$$\langle {\bf z}_{{\rm Kato}}^{S,c,d}, \exp(\Phi)\otimes \tfrac{(e_2^\dual)^k}{(e_1^\dual\wedge e_2^\dual)^j}\rangle_{\eet}=
\langle {\bf z}_{{\rm Kato}}^{S,c,d}(k,j), \exp(\Phi)\rangle_{\eet}.$$
Le th.\,\ref{recex1} se traduit par
\begin{align*}
\big\langle {\bf z}^{S,c,d}_{{\rm Kato}}(k,j), \exp(\Phi)\big\rangle_{\eet}
&= \big\langle \tilde{\bf z}^{S,c,d}_{{\rm Eis}}(k,j)\otimes\zeta_{\rm dR}^2, \Phi\big\rangle_{\rm dR}\otimes\zeta_{\rm dR}^{-1}\\
&= \big\langle \big(c_p^2-c_p^{j}\matrice{c}{0}{0}{1}\big)\big(d_p^2-d_p^{j-k}\matrice{1}{0}{0}{d}\big)
\star\tilde{\bf z}^{S}_{{\rm Eis}}(k,j), \Phi\big\rangle_{\rm dR}\otimes\zeta_{\rm dR}\\
&= \big\langle \tilde{\bf z}^S_{{\rm Eis}}(k,j), 
\big(c_p^2-c_p^{j}\matrice{c^{-1}}{0}{0}{1}\big)\big(d_p^2-d_p^{j-k}\matrice{1}{0}{0}{d^{-1}}\big)\star\Phi\big\rangle_{\rm dR}\otimes\zeta_{\rm dR}
\end{align*}
On conclut en utilisant la formule (\ref{limite}) 
et l'identit\'e
$$\int_{1+N{\rm M}_2(\Z_S)}\tilde {\bf z}^S_{{\rm Eis}}(k,j)=
\tfrac{(-1)^j}{(j-1)!}N^{k-2j}
\tilde F_{\frac{1}{N},0}^{(k+2-j)}\tilde E_{0,\frac{1}{N}}^{(j)}.
\qedhere$$
\end{proof}

\subsubsection{La m\'ethode de Rankin-Selberg}\label{shi1}
Un r\'esultat classique de Shimura~\cite{Shi} (prop.~\ref{shi3} ci-dessous),
qui se d\'emontre en utilisant la m\'ethode de Rankin, 
permet d'exprimer
$\langle\Phi,\widetilde E_{0,\frac{1}{N}}^{(j)} \widetilde F_{\frac{1}{N},0}^{(k+2-j)}\rangle_{{\rm dR},Y(N)}$ en termes de valeurs sp\'eciales de fonctions $L$.  

Soit
$$E_{0,\frac{1}{N}}^{(j,s)}(\tau)=\sum_{\omega\in\Z+\Z\tau}\frac{1}{\big(\omega+\frac{1}{N}\big)^j}
\big(\frac{{\rm Im}\,\tau}{\big|\omega+\frac{1}{N}\big|^2}\big)^{s-1-k}.$$

\begin{prop}\phantomsection\label{shi3}
Soient 
$$f=\sum_{n\in\frac{1}{M}\Z,\,n>0}a_nq^n\in M^{\rm par,\,cl}_{k+2}(\Gamma(M))
\quad{\rm et}\quad
g=\sum_{n\in\frac{1}{M}\Z,\,n\geq 0}b_nq^n\in M^{\rm cl}_{k+2-j}(\Gamma(M))$$
dont les fonctions $L$ se factorisent sous la forme:
\begin{align*}
L(f,s)=&\ \big(\sum_{n\in \Z[\frac{1}{M}]^\dual}\tfrac{a_n}{n^s}\big)
\prod_{\ell\nmid M}\big({(1-\tfrac{\alpha_{\ell,1}}{\ell^{s}})(1-\tfrac{\alpha_{\ell,2}}{\ell^{s}})}\big)^{-1}\\
L(g,s)=&\ \big(\sum_{n\in \Z[\frac{1}{M}]^\dual}\tfrac{b_n}{n^s}\big)
\prod_{\ell\nmid M}\big({(1-\tfrac{\beta_{\ell,1}}{\ell^{s}})(1-\tfrac{\beta_{\ell,2}}{\ell^{s}})}\big)^{-1}
\end{align*}
Alors
$$\int_{\Gamma(M)\backslash{\cal H}^+}{f(-\overline\tau)}g(\tau)E^{(j,s)}_{0,\frac{1}{M}}(\tau)
\,y^{k+2}\tfrac{dx\,dy}{y^2}=M^{1+j+2(s-1-k)}
\tfrac{\Gamma(s)}{(4\pi)^s}
D(f,g,s),$$
o\`u l'on a pos\'e
$$D(f,g,s)=
\big(\sum_{n\in \Z[\frac{1}{M}]^\dual}\hskip-.3cm\tfrac{{a_n}b_n}{n^s}\big)
\prod_{\ell\nmid M}\big({(1-\tfrac{{\alpha_{\ell,1}}\beta_{\ell,1}}{\ell^{s}})
(1-\tfrac{{\alpha_{\ell,1}}\beta_{\ell,2}}{\ell^{s}})
(1-\tfrac{{\alpha_{\ell,2}}\beta_{\ell,1}}{\ell^{s}})
(1-\tfrac{{\alpha_{\ell,2}}\beta_{\ell,2}}{\ell^{s}})}\big)^{-1}$$
\end{prop}

\subsubsection{Syst\`eme d'Euler de Kato et valeurs sp\'eciales de fonctions $L$}\label{expo4}
Soit maintenant $\pi$ une repr\'esentation cohomologique encadr\'ee de $\GG(\A^{]\infty[})$,
de poids $(k+2,j+1)$, avec $1\leq j\leq k+1$, de conducteur $N$ et caract\`ere central $\omega_\pi$.
Soit $\tilde\omega_\pi$ le caract\`ere de Dirichlet associ\'e \`a $\omega_\pi$ (\'etendu en une fonction
multiplicative sur $\Z$, en posant $\tilde\omega_\pi(n)=0$ si $(n,N)\neq 1$).
Notons que $\tilde\omega_\pi(-1)=(-1)^k$.

On renvoie au \no\ref{como12} et suivants pour la d\'efinition de $v_\pi$, $\iota_{{\rm dR},\pi}^+$,
$f_\pi$, $f^\dual_\pi$, $\lambda(\pi)$, etc.
Notons que $f_\pi$ est une forme primitive, ce qui implique
$$L(f_\pi,s)=
\prod_{\ell}
\big({1-\tfrac{a_\ell}{\ell^s}+\tfrac{\tilde\omega_\pi^{-1}(\ell)\ell^{k+1}}{\ell^{2s}}}\big)^{-1}
$$

On \index{pis@\piS}note $\pi_S$ le produit tensoriel $\otimes_{\ell\in S}\pi_\ell$; c'est un sous-espace
de ${\rm LC}(\Q_S^\dual,\Q(\bmu_{S^\infty}))^{\Z_S^\dual}$.  
On peut utiliser l'action de $\matrice{\Z_S^\dual}{0}{0}{1}$ pour d\'ecouper $\pi_S$: si
$\eta:\Z_S^\dual\to L^\dual$ est un caract\'ere localement constant, on note
$\pi_S[\eta]$ l'espace des $\phi_S\in\Q(\pi,\eta)\otimes_{\Q(\pi)}\pi_S$ 
v\'erifiant $\matrice{a}{0}{0}{1}\star \phi_S=\eta(a)\phi_S$,
pour tout $a\in\Z_S^\dual$. (Ceci implique que $\phi_S$ est \`a valeurs dans $\Q(\pi,\eta)G(\eta^{-1})$.)

Soit $\alpha_{\pi,S}:\pi_S\to 
\Q(\pi)\otimes_\Q\iH^1_{{\rm dR},c}(W_{k,j}^{\rm dR})/{\rm Fil}^0$ l'application lin\'eaire
$$\phi\mapsto \alpha_{\pi,S}(\phi)=\iota_{{\rm dR},\pi}^-(v_{\pi}^{]S[}\otimes\phi)$$
Alors,
dans $\C\otimes(\iH^1_{{\rm dR},c}(W_{k,j}^{\rm dR})/{\rm Fil}^0)$,
on a aussi
$$\alpha_{\pi,S}(\phi)=
\lambda(\pi)\,
w_\infty\star\iota_{{\rm dR},\pi}^+(v_\pi^{]S[}\otimes \phi).$$

\begin{theo}\phantomsection\label{shi4}
Soient $\eta:\Z_S^\dual\to L^\dual$ un caract\`ere localement constant
et $\phi\hskip.5mm{\in}\hskip.5mm \pi_S[\eta]$. Alors\footnote{Voir le cor.\,\ref{shi4.1} pour une
formule plus g\'en\'erale, mais une preuve plus d\'etourn\'ee.}
$$
\langle {\bf z}^{S,c,d}_{{\rm Kato}},\exp(\alpha_{\pi,S}(\phi))
\otimes\tfrac{(e_2^\dual)^k}{k!\,(e_1^\dual\wedge e_2^\dual)^j}
\rangle_{\eet}= a_{\pi}^{c,d}(\phi X^{-j})
L^{]S[}(f_\pi^\dual,j)L(v_\pi^{]S[}\otimes\phi X^{-j},0),$$
o\`u:
$$ L^{]S[}(f_\pi^\dual,s)=\sum_{n\geq 1}{\bf 1}_{\Z_S^\dual}(n)\overline{a_n}\,n^{-s},\quad
a_{\pi}^{c,d}(\phi X^\ell)=
\lambda(\pi)\,
a_{\pi,\infty}(\phi X^\ell)\, a_{\pi,p}^{c,d}(\phi X^\ell).$$
avec\footnote{On voit $\eta$ comme un caract\`ere de $\cZ^\dual$ via la projection naturelle
$\cZ^\dual\to\Z_S^\dual$.}
\begin{align*}
a_{\pi,\infty}(\phi X^\ell)&=(1+(-1)^{-\ell}\eta(-1))\\
a_{\pi, p}^{c,d}(\phi X^\ell)&=\big(c_p^2-c_p^{-\ell}\eta(c^{-1})\big)
\big(d_p^2-d_p^{{-\ell}-k}\eta(d)\tilde\omega^{-1}_\pi(d)\big)
\end{align*}
\end{theo}
\begin{rema} 
On trouvera au cor.\,\ref{shi4.1} une formule plus g\'en\'erale.
\end{rema}
\begin{proof}
On utilise la prop.\,\ref{expo3.1} pour faire le calcul.
Le facteur $a_{\pi,p}^{c,d}(\phi X^{-j})$
provient du passage de $\phi$ \`a $\phi^{c,d}$ dans les notations
de cette proposition: si $a\in\cZ^\dual$, on a
$$\matrice{a}{0}{0}{1}^{]\infty[}\star(v_\pi^{]S[}\otimes\phi)=\eta(a)(v_\pi^{]S[}\otimes\phi)
\quad{\rm car}\ 
\matrice{a}{0}{0}{1}^{]S[}\star v_\pi^{]S[}=v_\pi^{]S[}
\ {\rm et}\ 
\matrice{a}{0}{0}{1}_S\star \phi
=\eta(a)\phi.$$
On en d\'eduit, en utilisant le fait que $\omega_\pi$ est le caract\`ere central,
que
\begin{align*}
\big(c_p^2-c_p^j\matrice{c^{-1}}{0}{0}{1}\big)&
\big(d_p^2-d_p^{j-k}\matrice{d^{-1}}{0}{0}{d^{-1}}\matrice{d}{0}{0}{1}\big)
\star(v_\pi^{]S[}\otimes\phi)\\ &=
\big(c_p^2-c_p^j\eta(c^{-1})\big)
\big(d_p^2-d_p^{j-k}\tilde\omega^{-1}_\pi(d)\eta(d)\big)
(v_\pi^{]S[}\otimes\phi).
\end{align*}
On continue donc le calcul avec $\phi$. On choisit $M$ tel que $\phi$ soit fixe par $1+M{\bf M}_2(\Z_S)$.

\noindent
$\bullet$ On commence par appliquer le lemme~\ref{shi2} avec 
$$\Phi_1^{\rm Har}=\alpha_{\pi,S}(\phi) \quad{\rm et}\quad
\Phi_2^{\rm Har}=\widetilde E^{(j)}_{0,\frac{1}{M}}
\widetilde F^{(k+2-j)}_{\frac{1}{M},0}.$$
Pour all\'eger un peu les notations, on pose:
$$\langle \Phi_1,\Phi_2\rangle:=
\langle \Phi_1,\Phi_2\rangle_{{\rm dR},Y(M)}\otimes\zeta_{\rm dR}$$

\noindent
$\diamond$
Puisque $\matrice{1}{0}{0}{a}^{]\infty[}\star\widetilde E^{(j)}_{0,\frac{1}{M}}=\widetilde E^{(j)}_{0,\frac{1}{M}}$
et $\matrice{1}{0}{0}{a}^{]\infty[}\star\widetilde F^{(j)}_{\frac{1}{M},0}=\widetilde F^{(j)}_{\frac{a}{M},0}$,
on a $$\big(\matrice{1}{0}{0}{a}^{]\infty[}\star\Phi^{\rm Har}_2\big)(\tau,1^{]\infty[})=
(-2i\pi)^{j+1}
E^{(j)}_{0,\frac{1}{M}}F^{(k+2-j)}_{\frac{a}{M},0}\otimes\tfrac{(\tau e_2-e_1)^k}{(e_1\wedge e_2)^{k-j}}d\tau.$$

\noindent
$\diamond$ Comme on l'a vu plus haut,
si $a\in\cZ^\dual$, alors 
\begin{align*}
\matrice{1}{0}{0}{a}^{]\infty[}\star \big(v_\pi^{]S[}\otimes\phi)
= \tilde\omega_{\pi}(a)\eta(a^{-1})\big(v_\pi^{]S[}\otimes\phi\big)
\end{align*}
On en d\'eduit que 
$$\matrice{1}{0}{0}{a}^{]\infty[}\star\Phi^{\rm Har}_1=\tilde\omega_\pi(a)\eta^{-1}(a)\Phi^{\rm Har}_1.$$

\noindent
$\diamond$
On obtient donc, gr\^ace au lemme~\ref{shi2},
$$\langle \Phi_1,\Phi_2\rangle= \tfrac{1}{2i\pi}
\int_{\Gamma(M)\backslash{\cal H}^+}\hskip-1cm
\Phi^{\rm Har}_1(\tau,1^{]\infty[})\wedge \big((-2i\pi)^{j+1}E^{(j)}_{0,\frac{1}{M}}
F^{(k+2-j)}_{M,\tilde\omega_\pi\eta^{-1}}\tfrac{(\tau e_2-e_1)^k}{(e_1\wedge e_2)^{k-j}}d\tau\big),$$
o\`u, si $\chi:(\Z/M)^\dual\to\C^\dual$ est un caract\`ere de Dirichlet,
on a pos\'e 
$$F^{(k+2-j)}_{M,\chi}=
\sum_{a\in(\Z/M)^\dual}\chi(a)F^{(k+2-j)}_{\frac{a}{M},0},$$
et, d'apr\`es le lemme~\ref{lambda1}, 
$$\Phi^{\rm Har}_1(\tau,1^{]\infty[})=
\eta(-1)\lambda(\pi)(2i\pi)^{k+1-j}f_\phi(-\overline\tau)
\otimes\tfrac{(\overline\tau e_2-e_1)^k}{(e_1\wedge e_2)^{j}}d\overline\tau$$
avec 
$$f_\phi=\sum_{n\in\Q,\,n>0}\hskip-.4cm n^{k+1-j}(v_\pi^{]S[}\otimes\phi)(n^{]\infty[})\,{\bf e}_\infty(-n\tau).$$

\noindent
$\bullet$ Ensuite, on applique la prop.~\ref{shi3} \`a
$$s=k+1, \quad f=f_\phi,
\quad g=F^{(k+2-j)}_{M,\chi},\ {\text{ avec $\chi=\tilde\omega_\pi\eta^{-1}$.}}$$

\noindent
$\diamond$
On a 
\begin{align*}
L(f,s)&=\big(\sum_{n\in\Z[\frac{1}{S}]_+^\dual}\tfrac{\phi(n_S)}{n^{s-(k+1-j)}}\big)
L^{]S[}(f_\pi\otimes\eta,s)\\
 &= \big(\sum_{n\in\Z[\frac{1}{S}]_+^\dual}\tfrac{\phi(n_S)}{n^{s-(k+1-j)}}\big)
\prod_{\ell\notin S}
(1-\tfrac{{a_\ell}\eta(\ell)}{\ell^s}+\tfrac{\tilde\omega_\pi^{-1}(\ell)\eta^2(\ell)\ell^{k+1}}{\ell^{2s}})^{-1}
\end{align*}
qui se factorise sous la forme voulue, avec $a_n=n^{k+1-j}\phi(n_S)$ si $n\in\Z[\frac{1}{S}]_+^\dual$.

\noindent
$\diamond$
On d\'eduit de la prop.~\ref{qdev} l'identit\'e
$$L(g,s)=c(\eta)M^{s-(k+1-j)}\zeta(s)L^{]S[}(\chi,s-(k+1-j)), $$
avec 
$$c(\eta)=1+(-1)^{k-j}\tilde\omega_\pi(-1)\eta(-1)=a_{\pi,\infty}(\phi X^{-j}).$$
qui fournit une factorisation sous la forme voulue,
avec $\beta_{\ell,1}=1$ et $\beta_{\ell,2}=\chi(\ell)\ell^{k+1-j}$,
et $b_n=c(\eta)M^{j-k-1}$, si $n\in\frac{1}{M}\N\cap\Z[\frac{1}{S}]_+^\dual$.

\noindent 
$\diamond$
 En utilisant les identit\'es
\begin{align*}
\big(\tfrac{(\overline\tau e_2-e_1)^k}{(e_1\wedge e_2)^j}d\overline\tau\big)\wedge
\big(\tfrac{(\tau e_2-e_1)^k}{(e_1\wedge e_2)^{k-j}}d\tau\big)=&\ 
-(-2i)^{k+1}y^{k+2}\tfrac{dx\wedge dy}{y^2},\\
E_{0,\frac{1}{M}}^{(j)}=
\tfrac{(j-1)!}{(-2i\pi)^j}E_{0,\frac{1}{M}}^{(j,k+1)},\quad
&\quad \tfrac{-(-2i)^{k+1}\Gamma(k+1)}{(4\pi)^{k+1}}=
\tfrac{-k!}{(2i\pi)^{k+1}}
\end{align*}
on en tire, gr\^ace \`a la prop.\,\ref{shi3},
\begin{align*}
\langle \Phi_1,\Phi_2\rangle&=\tfrac{1}{2i\pi}
\eta(-1)\lambda(\pi)(2i\pi)^{k+1-j}(-2i\pi)^{j+1}
\tfrac{(j-1)!}{(-2i\pi)^j}\tfrac{-k!}{(2i\pi)^{k+1}} M^{1+j}D(f,g,k+1)\\
&=\eta(-1)\lambda(\pi)k!\tfrac{(j-1)!}{(2i\pi)^j}M^{1+j}D(f,g,k+1)
\end{align*}

\noindent
$\bullet$
Par ailleurs, les factorisations de $L(f,s)$ et $L(g,s)$ ci-dessus nous donnent:
$$D(f,g,s)=c(\eta)M^{j-k-1}
L(f_\phi,s)
L^{]S[}((f_\pi\otimes\eta)\otimes\eta^{-1}\tilde\omega_\pi,s-(k+1-j)),$$
ce que l'on peut r\'e\'ecrire, en utilisant les
relations
$$f_\pi\otimes\tilde\omega_\pi=f_\pi^\dual,\quad
L(f_\phi,s)=L(v_\pi^{]S[}\otimes\phi,s-(k+1-j))=(2i\pi)^jL(v_\pi^{]S[}\otimes\phi X^{-j},s-(k+1)),$$
sous la forme
$$D(f,g,s)=(2i\pi)^j c(\eta)M^{j-k-1}
L(v_\pi^{]S[}\otimes\phi X^{-j},s-(k+1))
L^{]S[}(f^\dual_\pi,s-(k+1-j)).$$
On en tire 
$$\langle \Phi_1,\Phi_2\rangle= \eta(-1) c(\eta)
\lambda(\pi)k!{(j-1)!}M^{2j-k}
L^{]S[}(f_\pi^\dual,j)L(v_\pi^{]S[}\otimes\phi X^{-j},0).$$
Le r\'esultat s'en d\'eduit en utilisant l'identit\'e
$(-1)^j\eta(-1)c(\eta)=a_{\pi,\infty}(\phi X^{-j})$ et la prop.~\ref{expo3.1}; le $k!$ dispara\^{\i}t
car on utilise $\frac{(e_2^\dual)^k}{k!(e_1^\dual\wedge e_2^\dual)^j}$ (qui correspond \`a
$X^{-j}$ dans le mod\`ele de Kirillov de $\pi_p^{\rm alg}$) au lieu de
$\frac{(e_2^\dual)^k}{(e_1^\dual\wedge e_2^\dual)^j}$.
\end{proof}

\Subsection{Factorisation du syst\`eme de Beilinson-Kato}\label{fact1}

\subsubsection{Les fonctionnelles ${\bf z}_{{\rm Kato}}^{S,c,d}(\check m_{\eet}(\pi))$
et ${\bf z}^{S}(\check m_{\eet}(\pi))$.}
Soit $\pi$ une repr\'esentation de $\GG(\A^{]\infty[})$,
cohomologique, $\Pi_p$-compatible,
de poids $(k+2,j+1)$ avec $0\leq j\leq k$.
\vskip.2cm
$\bullet$ {\it L'\'el\'ement ${\bf z}^S(\check m_{\eet}(\pi))$}.--- 
On note 
$$\check m_{\eet}(\pi):=m_{\eet}^\dual(\pi)(1)=m_{\eet}(\check\pi\otimes|\ |_\A^{-2})$$ 
la duale de Tate de $m_{\eet}(\pi)$.  On a aussi
$$\check m_{\eet}(\pi)\cong\rho_{f_\pi^\dual}^\dual\otimes\cyp ^{1-j}
=\rho_{f_\pi}\otimes\cyp ^{k+2-j}.$$
Comme $m(\check\pi\otimes|\ |_\A^{-2})^\dual=m(\pi)\otimes\zeta_{\rm B}^{-1}$,
on dispose \index{Zmpi2@\Zmpi}de
\begin{align*}
{\bf z}^S(\check m_{\eet}(\pi))&\in 
\check m_{\eet}(\pi)^\dual\otimes H^1(G_{\Q,S},\check m_{\eet}(\pi))
= m(\pi)\otimes\zeta_{\rm B}^{-1}\otimes H^1(G_{\Q,S},\check m_{\eet}(\pi))
\end{align*}

\vskip.2cm
$\bullet$ {\it L'\'el\'ement
${\bf z}_{{\rm Kato}}^{S,c,d}(\check m_{\eet}(\pi))$}.---
On a une application naturelle
$$H^1(G_{\Q,S},H^1(\Pi'_{\Qbar,S},\cdo (\GG(\Q_S),\Q_p(2)))\to
H^1(G_{\Q,S},H^1(\GG,{\rm Mes}_L)_S(2))$$
En effet, $H^1(\Pi'_{\Qbar,S},\cdo (\GG(\Q_S),\Q_p))=H^1(\Pi_{\Qbar,S},\cdo (\GG(\Z_S),\Q_p))$
par le lemme de Shapiro et le membre de droite s'injecte dans
$H^1(\Pi_{\Qbar},\cdo (\GG(\Z_S),\Q_p))$ par inflation.
Comme $\Pi_{\Qbar}$ est le compl\'et\'e profini de $\Gamma(1)$ qui est presque libre,
on a $H^1(\Pi_{\Qbar},\cdo (\GG(\Z_S),\Q_p))=H^1(\Gamma(1),\cdo (\GG(\Z_S),\Q_p))$,
et ce dernier groupe est isomorphe \`a $H^1(\GG,{\rm Mes}_{\Q_p})_S$ par le lemme de Shapiro.

On note encore 
$${\bf z}_{{\rm Kato}}^{S,c,d}\in H^1(G_{\Q,S},H^1(\GG,{\rm Mes}_L)_S(2))$$
l'image de ${\bf z}_{{\rm Kato}}^{S,c,d}\in H^1(G_{\Q,S},H^1(\Pi'_{\Qbar,S},\cdo (\GG(\Q_S),\Q_p(2)))$ par
l'application ci-dessus.
Maintenant, $H^1(\GG,{\rm Mes}_L)_S(2)$ est le dual de Tate 
$H^1_c(\GG,{\cal C}_L)_S$.
On dispose d'injections naturelles (cf.~\no\ref{como17}, du moins si la restriction
$m_{\eet}(\pi)_p$ de $m_{\eet}(\pi)$ \`a $G_{\Q_p}$ n'est pas somme de deux caract\`eres,
cf.~rem.\,\ref{passomme1}):
\begin{align*}
m(\pi)\otimes v_\pi^{]S[}\otimes\pi_S^{\rm alg}{\hooklra} 
m_{\eet}(\pi)\otimes v_\pi^{]S[}\otimes\Pi_S(m_{\eet}^\dual(\pi))
\overset{\iota_{\pi,S}}{\hooklra} \Q(\pi)\otimes H^1_c(\GG,{\cal C}_L)_S,
\end{align*}

Par dualit\'e, cela fournit une fl\`eche $G_{\Q,S}\times \GG(\Q_p)$-\'equivariante
$$\iota_{\pi,S}^\dual:
\Q(\pi)\otimes H^1(\GG,{\rm Mes}_L)_S(2)\to \check m_{\eet}(\pi)\otimes\Pi_S^\dual(m_{\eet}^\dual(\pi)),$$
On \index{Zmpi2@\Zmpi}note 
$${\bf z}_{{\rm Kato}}^{S,c,d}(\check m_{\eet}(\pi))
\in H^1(G_{\Q,S},\check m_{\eet}(\pi)\otimes\Pi_S^\dual(m_{\eet}^\dual(\pi)))=
H^1(G_{\Q,S},\check m_{\eet}(\pi))\otimes\Pi_S^\dual(m_{\eet}^\dual(\pi))$$
l'image, par $\iota_{\pi,S}^\dual$,
 de ${\bf z}_{{\rm Kato}}^{S,c,d}\in H^1(G_{\Q,S},H^1(\GG,{\rm Mes}_L)_S(2))$.

\Subsubsection{La fonctionnelle $(0,\infty)_{\pi,S}^{c,d}$}

\vskip.2cm
$\bullet$ {\it Les op\'erateurs $B_p^{c,d}$, $B_\infty$, $A_{\pi,p}^{c,d}$ et $A_{\infty}$}.---
Soient $B_p^{c,d}$ et $B_\infty$ les op\'erateurs
$$B_p^{c,d}=\big(c_p^2-\matrice{c}{0}{0}{1}\big)
\big(d_p^2-\matrice{1}{0}{0}{d}\big)
\quad{\rm et}\quad
B_\infty=\tfrac{1}{2}\big(1+\matrice{-1}{0}{0}{1}^{]\infty[}\big).$$
Les op\'erateurs adjoints $(B_p^{c,d})^\dual$ et $B_\infty^\dual$
 s'obtiennent en inversant les matrices qui interviennent.
Comme $\matrice{1}{0}{0}{d^{-1}}=\matrice{d}{0}{0}{1}\matrice{d}{0}{0}{d}^{-1}$
et comme $\matrice{d}{0}{0}{d}$ agit par multiplication
par $\omega_{\pi^{\rm alg}}(d)=d_p^{k-2j}\tilde\omega_\pi(d)$, 
on a
\begin{align*}
&(B_p^{c,d})^\dual=\big(c_p^2-\matrice{c^{-1}}{0}{0}{1}\big)
\big(d_p^2-\tilde\omega_\pi^{-1}(d)d_p^{2j-k}\matrice{d}{0}{0}{1}\big),
&B_\infty^\dual=
\tfrac{1}{2}\big(1+\matrice{-1}{0}{0}{1}^{]\infty[}\big)
\end{align*}
Enfin, soient $A_{\pi,p}^{c,d}$ et $A_{\infty}$ les op\'erateurs
\begin{align*}
&A_{\pi,p}^{c,d}=\big(c_p^2-\matrice{c_p^{-1}}{0}{0}{1}_p\big)
\big(d_p^2-\tilde\omega_\pi^{-1}(d)d_p^{2j-k}\matrice{d_p}{0}{0}{1}_p\big)
&A_{\infty}=
\tfrac{1}{2}\big(1+\matrice{-1}{0}{0}{1}_S\big)
\end{align*}

$\bullet$ {\it Les fonctionnelles $(0,\infty)_{\pi,S}$ et $(0,\infty)_{\pi,S}^{c,d}$}.---
La restriction de $(0,\infty)$ \`a $m(\pi)\otimes  v_\pi^{]S[}\otimes \Pi_S(m_{\eet}^\dual(\pi))$
fournit
\begin{align*}
(0,\infty)_{\pi,S}&\in m(\pi)^\dual\otimes \Pi_S^\dual(m_{\eet}^\dual(\pi))
\end{align*}
On \index{inf2@\oits}note $(0,\infty)_{\pi,S}^{c,d}$
la fonctionnelle 
$$(0,\infty)_{\pi,S}^{c,d}=B_\infty B_p^{c,d}\star (0,\infty)_{\pi,S}$$ 
sur $m(\pi)\otimes\Pi_S(m_{\eet}^\dual(\pi))$.
Par d\'efinition,
$$\langle (0,\infty)^{c,d}_{\pi,S},\gamma\otimes\phi\rangle=
\big\langle (0,\infty),\gamma\otimes\big(B_\infty^\dual(B_p^{c,d})^\dual\star( v_\pi^{]S[}\otimes\phi)\big)\big\rangle.$$
Comme $\matrice{u}{0}{0}{1}^{]\infty,p[}$ fixe $ v_\pi^{]S[}$, si $u\in\cZ^\dual$,
on a
\begin{align}
B_\infty^\dual(B_p^{c,d})^\dual\star( v_\pi^{]S[}\otimes\phi)&=
 v_\pi^{]S[}\otimes (A_{\infty}A_{\pi,p}^{c,d}\star\phi) \notag\\
\langle (0,\infty)^{c,d}_{\pi,S},\gamma\otimes\phi\rangle&=
\langle (0,\infty)_{\pi,S},\gamma\otimes (A_{\infty}A_{\pi,p}^{c,d}\star\phi)\rangle
\label{cons1}
\end{align}

\subsubsection{Factorisation en un point classique}\label{fact1.1}
Remarquons que
$$m(\check\pi)\otimes\zeta_{\rm B}^2\hookrightarrow m_{\eet}(\check\pi) \otimes\zeta_{\rm B}^2=
m_{\eet}^\dual(\pi)(-1) \otimes\zeta_{\rm B}^2=\check m_{\eet}(\pi).$$
\begin{theo}\phantomsection\label{fact2}
On suppose que\footnote{\label{fact2.7}
C'est automatique si $j\neq\frac{k}{2}+1$ (centre de la bande critique)
et si $j\neq\frac{k+1}{2}$ (o\`u les facteurs d'Euler en $\ell\in S$
peuvent s'annuler).}
$L^{]S[}(f_\pi^\dual,j)\neq 0$.
Alors, 
pour tous $\gamma\in m(\pi)$ et $\check\gamma\in m(\check\pi)$,
on a l'identit\'e suivante dans $H^1(G_{\Q,S},\check m_{\eet}(\pi))\otimes\Pi_S^\dual(m_{\eet}^\dual(\pi))$,
$$(\langle\check\gamma^-,\gamma^+\rangle_{\rm B}\otimes\zeta_{\rm B})
\cdot {\bf z}_{{\rm Kato}}^{S,c,d}(\check m_{\eet}(\pi))=
 \langle{\bf z}^S(\check m_{\eet}(\pi)),\check\gamma\otimes\zeta_{\rm B}^2\rangle\otimes 
\langle (0,\infty)^{c,d}_{\pi,S},\gamma\rangle.$$
\end{theo}
\begin{proof}
On va d\'emontrer le r\'esultat dans le cas o\`u $m_{\eet}(\pi)_p$
est irr\'eductible. Le cas g\'en\'eral
s'en d\'eduit par prolongement analytique, cf.~(ii) de la rem.\,\ref{fact27}.
On fait les calculs dans le mod\`ele de Kirillov de $\pi_S^{\rm alg}$ dans lequel
$\phi\otimes\frac{(e_2^\dual)^k}{k!(e_1^\dual\wedge e_2^\dual)^j}$
devient~$\phi X^{-j}$.

$\bullet$ L'hypoth\`ese d'irr\'eductibilit\'e de $m_{\eet}(\pi)_p$ implique que 
le sous-espace de $\pi_p^{\rm alg}$ des
$\phi X^{-j}$,
pour $\phi\in {\rm LC}_c(\Q_p^\dual,\Q(\bmu_{p^\infty}))^{\Z_p^\dual}$,
est dense dans $\Pi_p(m_{\eet}^\dual(\pi))$: en effet, il est stable par 
$\BB(\Q_p)$ et
$\Pi_p(m_{\eet}^\dual(\pi))$ est topologiquement irr\'eductible comme $\BB(\Q_p)$-module
(cf.~\cite{gl2}, cor.\,II.2.9 et~\cite{mira}, cor.\,IV.5.6).

$\bullet$ L'hypoth\`ese $L^{]S[}(f_\pi^\dual,j)\neq 0$ implique, 
d'apr\`es~\cite{Ka4}, que $H^1(G_{\Q,S},\check m_{\eet}(\pi))$
est de dimension~$1$ et s'injecte dans $H^1(G_{\Q_p},\check m_{\eet}(\pi))$.
Ce dernier espace est le dual de $H^1(G_{\Q_p},m_{\eet}(\pi))$ 
qui contient $\exp(\iota_{{\rm dR},\pi}^-)$.
Il suffit donc de v\'erifier le m\^eme \'enonc\'e en rempla\c{c}ant ${\bf z}_{{\rm Kato}}^{S,c,d}(\check m_{\eet}(\pi))$
et ${\bf z}(\check m_{\eet}(\pi))$ par leurs restrictions \`a $G_{\Q_p}$, et il suffit alors
de v\'erifier l'\'egalit\'e souhait\'ee
par accouplement avec $\exp(\iota^-_{{\rm dR},\pi})\otimes \phi$,
avec $\phi=\phi_0 X^{-j}$ et $\phi_0\in\pi_S$.

$\bullet$ 
Par lin\'earit\'e, on peut supposer de plus que $\phi_0\in\pi_S[\eta]$,
o\`u $\eta$ est un caract\`ere localement constant de $\Z_S^\dual$.
Dans ce cas, on a $\matrice{u}{0}{0}{1}_S\star\phi=\eta(u)u^{-j}\phi$ si $u\in\Z_S^\dual$;
il en r\'esulte que (cf.~th.\,\ref{shi4} pour $a_{\pi,\infty}(\phi)$ et $a_{\pi,p}^{c,d}(\phi)$)
\begin{align}\label{cons}
A_{\infty}\star \phi =\tfrac{1}{2}\,a_{\pi,\infty}(\phi)\,\phi ,
\quad
A_{\pi,p}^{c,d}\star \phi =a_{\pi,p}^{c,d}(\phi)\,\phi.
\end{align}

$\bullet$ D'apr\`es le th.\,\ref{shi4},
\begin{align*}
\big\langle {\bf z}_{{\rm Kato}}^{S,c,d}(\check m_{\eet}(\pi)), \exp(\iota^-_{{\rm dR},\pi})\otimes 
\phi\big\rangle 
&= \big\langle {\bf z}_{{\rm Kato}}^{S,c,d}, \exp(\alpha_{\pi,S}(\phi_0)\otimes
\tfrac{(e_2^\dual)^k}{k!(e_1^\dual\wedge e_2^\dual)^j})\big\rangle\\ 
&= 
a_\pi^{c,d}(\phi)L^{]S[}(f_\pi^\dual,j)
L(v_\pi^{]S[}\otimes\phi,0).
\end{align*}

$\bullet$ D'un autre c\^ot\'e,
\begin{align*}
\big\langle \langle{\bf z}(\check m_{\eet}(\pi)), \check\gamma\otimes\zeta_{\rm B}^2\rangle
&\otimes \langle(0,\infty)^{c,d}_{\pi,S},\gamma\rangle,
\exp(\iota^-_{{\rm dR},\pi})\otimes \phi \big\rangle\\
=&\ 
\big\langle \langle{\bf z}(\check m_{\eet}(\pi)),\check\gamma\otimes\zeta_{\rm B}^2\rangle,\exp(\iota_{{\rm dR},\pi}^-)\big\rangle\,
\big\langle (0,\infty)^{c,d}_{\pi,S},\gamma\otimes \phi\big\rangle
\end{align*}

$\bullet$ Il r\'esulte des formules (\ref{cons1}) et (\ref{cons}) que
$$
\big\langle (0,\infty)^{c,d}_{\pi,S},\gamma\otimes \phi\big\rangle=a_{\pi,\infty}(\phi)
\,a_{\pi,p}^{c,d}(\phi)
\,\big\langle (0,\infty)_{\pi,S},\gamma\otimes \phi\big\rangle.$$
En particulier, si $(-1)^j\eta(-1)=-1$, alors $a_{\pi,\infty}(\phi)=0$ et les deux
membres de l'identit\'e \`a v\'erifier sont nuls.  On suppose donc $(-1)^j\eta(-1)=1$
dans ce qui suit (et donc $a_{\pi,\infty}(\phi)=2$, mais nous n'aurons pas besoin de ce fait).

Le cor.~\ref{es4} fournit (cf.~formule\,(\ref{ZK19}), la condition $(-1)^j\eta(-1)=1$ implique
que $\phi=\phi^+$) l'identit\'e
\begin{align*}
\big\langle (0,\infty)_{\pi,S},\gamma\otimes \phi \big\rangle=
\Omega_\pi^+(\gamma)L( v_\pi^{]S[}\otimes\phi,0)
\end{align*}

$\bullet$ Enfin, on d\'eduit de la prop.\,\ref{zk20} utilis\'ee pour 
$\check\pi\otimes|\ |_\A^{-2}$ au lieu de $\pi$ (ce qui change $j$ en $k+2-j$),
et pour $\phi={\bf 1}_{\Z_p^\dual}$ et $r=0$, que
\begin{align*}
\big\langle \langle{\bf z}^S(\check m_{\eet}(\pi)),\check\gamma\otimes\zeta_{\rm B}^2\rangle,
\exp(\iota_{{\rm dR},\pi}^-)\big\rangle &=
\big\langle \exp^\dual\big(\langle{\bf z}^S(\check m_{\eet}(\pi)),\check\gamma\otimes\zeta_{\rm B}^2\rangle\big),
\iota_{{\rm dR},\pi}^-\big\rangle\\
&= (2i\pi)^{-1}\Omega_{\check\pi\otimes|\ |_\A^{-2}}^-(\check\gamma\otimes \zeta_{\rm B}^2)L^{]S[}(f_\pi^\dual,j)\\
&=2i\pi\, \Omega_{\check\pi}^-(\check\gamma)L^{]S[}(f_\pi^\dual,j)
\end{align*}
(On obtient $L^{]S[}(f_\pi^\dual,j)$ gr\^ace au lien entre ${\bf z}^S$ et ${\bf z}$,
cf.~prop.\,\ref{ZK1}, rem.\,\ref{eu100} et note~\ref{euler}.
 Le passage de la deuxi\`eme ligne \`a la troisi\`eme utilise le lemme~\ref{twi1}.)

$\bullet$ Le lemme~\ref{como11.1}, 
selon lequel $\Omega_{\check\pi}^-(\check\gamma)\Omega_{\pi}^+(\gamma)=
2\frac{\lambda(\pi)}{2i\pi}(\langle\check\gamma^-,\gamma^+\rangle_{\rm B}\otimes\zeta_B)$, fournit l'identit\'e
\begin{align*}
\Omega_\pi^+(\gamma)(2i\pi)\Omega_{\check\pi}^-(\check\gamma)
\,\tfrac{1}{2}\, a_{\pi,\infty}(\phi)a_{\pi,p}^{c,d}(\phi)
&=
2(\langle\check\gamma^-,\gamma^+\rangle_{\rm B} \otimes\zeta_B)\,
\lambda(\pi)\,\tfrac{1}{2}\,a_{\pi,\infty}(\phi)a_{\pi,p}^{c,d}(\phi)\\
&=(\langle\check\gamma^-,\gamma^+\rangle_{\rm B}\otimes\zeta_B)
a_\pi^{c,d}(\phi)
\end{align*}
Ceci permet de conclure.
\end{proof}

\begin{coro}\phantomsection\label{shi4.1}
Si $\phi\in \pi_S[\eta]$, et si $0\leq \ell\leq k$, alors
$$
\langle {\bf z}^{S,c,d}_{{\rm Kato}},\exp(\alpha_{\pi,S}(\phi))
\otimes\tfrac{(e_1^\dual)^\ell(e_2^\dual)^{k-\ell}}
{(k-\ell)!\,(e_1^\dual\wedge e_2^\dual)^j}
\rangle_{\eet}=a_\pi^{c,d}(\phi X^{\ell-j})
L^{]S[}(f_\pi^\dual,j)L(v_\pi^{]S[}\otimes\phi X^{\ell-j},0),$$
\end{coro}
\begin{proof}
Cela r\'esulte de la formule (\ref{ZK19}) pour 
$\big\langle (0,\infty),\gamma\otimes  v_\pi^{]S[}\otimes \phi X^{\ell-j}\big\rangle$.
\end{proof}

\subsubsection{Factorisation en famille}\label{fact4}
Soit $T'$ la localis\'ee de $T(Np^\infty)$ correspondant \`a $\rho_T^\diamond(2)$,
et ${\cal X}'={\rm Spec}(T')$.  L'application $\rho\mapsto\rho\otimes (\det\rho)^{-1}\cy^2$
induit des isomorphismes ${\cal X}\cong {\cal X}'$, et $T'\cong T$,
ce qui permet de voir tous les objets attach\'es \`a $\rho_{T'}=\rho_T^\diamond(2)$ comme des
$T$-modules. On note $H^1_c[\rho_T^\diamond(2)]_S\subset H^1_c[\rho_T^\diamond(2)]$ 
le sous-$\GG(\Q_S)$-module
correspondant \`a $\rho_T^\diamond(1)\otimes_T v_T^{]S[}\otimes_T \Pi_S(\rho_T(-1))$
dans la factorisation du th.\,\ref{Ycano110}.
On note 
$$(0,\infty)_{T',S},\ 
(0,\infty)_{T',S}^{c,d}\in \rho_T(-1)\otimes_T\Pi_S^\dual(\rho_T(-1))$$
les restrictions de $(0,\infty)$ 
et $\big(B_\infty B_p^{c,d}\star (0,\infty)\big)$ \`a
$\rho_T^\diamond(1)\otimes_T v_T^{]S[}\otimes_T \Pi_S(\rho_T(-1))$.

Par dualit\'e, la factorisation fournit une fl\`eche 
$$H^1(\GG,{\rm Mes}_L)_S(2)\to
(\rho_T^\diamond(1)\otimes_T \Pi_S(\rho_T(-1)))^\dual(1)
=\rho_T\otimes_T\Pi^\dual_S(\rho_T(-1))$$
qui est $G_{\Q,S}\times \GG(\Q_S)$-\'equivariante.
On note $${\bf z}_{{\rm Kato}}^{S,c,d}(\rho_T)\in 
H^1(G_{\Q,S}, \rho_T\otimes_T\Pi_S^\dual(\rho_T(-1)))
\cong H^1(G_{\Q,S}, \rho_T)\otimes_T\Pi_S^\dual(\rho_T(-1))$$ 
l'image de ${\bf z}_{{\rm Kato}}^{S,c,d}$. 

Rappelons (cf.~rem.\,\ref{geni8}) que l'on a d\'efini 
$${\bf z}^S(\rho_T)\in {\rm Fr}(T)\otimes_T(\rho_T^\diamond\otimes_T H^1(G_{\Q,S},\rho_T))$$
\begin{lemm}\phantomsection\label{facto1}
Pour tous $\check\gamma\in\rho_T$, 
$\gamma\in \rho_T^\diamond(1)$ et $v\in \Pi_S^\dual(\rho_T(-1))^\diamond$,
on a
$$(\langle \check\gamma^-,\gamma^+\rangle\otimes\zeta_{\rm B})
 \,\langle{\bf z}_{{\rm Kato}}^{S,c,d}(\rho_T),v\rangle
=\langle (0,\infty)^{c,d}_{T',S},
\gamma\otimes v\rangle\,\langle {\bf z}^S(\rho_T),\check\gamma\rangle.$$
\end{lemm}
\begin{proof}
Les deux membres appartiennent \`a ${\rm Fr}(T)\otimes_TH^1(G_{\Q,S},\rho_T)$,
qui est un ${\rm Fr}(T)$-module de rang~$1$. Si $x\in{\cal X}^{{\rm cl},+}$ n'est pas un p\^ole,
on peut tout r\'eduire modulo~${\goth p}_x$. Si $\rho_x=\check m_{\eet}(\pi)$, o\`u $\pi$
est cohomologique, alors

$\bullet$ $\gamma(x)\in \check m_{\eet}(\pi)^\dual\otimes\zeta_{\rm B}=m_{\eet}(\pi)$,

$\bullet$ $\check\gamma(x)\in \check m_{\eet}(\pi)=(\Q_p\otimes m(\check\pi))\otimes\zeta_{\rm B}^2$,

$\bullet$ $v(x)\in \Pi_S(m_{\eet}^\dual(\pi))$ car $v(x)$ est l'image de $v$
par 
\begin{align*}
{\rm Hom}_T(\Pi_S^\dual(\rho_T(-1)),T)\to 
{\rm Hom}_T(\Pi_S^\dual(\rho_T(-1))&,T/{\goth p}_x)=
{\rm Hom}(\Pi_S^\dual(\rho_T(-1)),L)[{\goth p}_x]\\
&=\Pi_S(\rho_T(-1))[{\goth p}_x]=
\Pi_S(m_{\eet}^\dual(\pi)).
\end{align*}

$\bullet$ ${\bf z}_{{\rm Kato}}^{S,c,d}(\rho_T)$ se sp\'ecialise en
${\bf z}_{{\rm Kato}}^{S,c,d}(\check m_{\eet}(\pi))$,

$\bullet$ $(0,\infty)^{c,d}_{T',S}$ se sp\'ecialise en
$(0,\infty)^{c,d}_{\pi,S}$.

D'apr\`es le th.\,\ref{fact2},
la diff\'erence des deux membres est donc nulle en tout $x \in{\cal X}^{{\rm cl},+}$,
de poids $(k+2,j+1)$ avec $1\leq j\leq k+1$, mais
$j\neq\frac{k}{2}+1,\frac{k+1}{2}$ (cf.~note\,\ref{fact2.7}), et
qui n'est pas un p\^ole
de ${\bf z}^S(\rho_T)$.  Comme ces points sont zariski-denses dans ${\cal X}$, cela prouve
que la diff\'erence est nulle.
\end{proof}

\begin{theo}\phantomsection\label{facto2}
{\rm (i)} ${\bf z}^S(\rho_T)\in \rho_T^\diamond\otimes_T H^1(G_{\Q,S},\rho_T)$.

{\rm (ii)} Si $x$ est un point classique $x$, et si
${\bf z}(\rho_x)$ est l'\'el\'ement de Kato du \no\ref{ZK99}.
la sp\'ecialisation de ${\bf z}^S(\rho_T)$
en $x$ est $\big(\prod_{\ell\in S\moins\{p\}}P_\ell(1)\big)\, {\bf z}(\rho_x)$,
o\`u $P_\ell$ est le polyn\^ome de la rem.\,\ref{eu100}.
\end{theo}
\begin{proof}
Compte-tenu du lemme~\ref{facto1}, il suffit, pour d\'emontrer le (i),
de prouver que, si $x\in{\cal X}$, on peut trouver 
$v,\gamma,c,d$ tels que $\langle (0,\infty)^{c,d}_{T',S},\gamma\otimes v\rangle$ ne s'annule
pas en $x$.  Dans le cas contraire, on a
$$\big\langle (0,\infty),\gamma\otimes ((B_p^{c,d})^\dual\star v)\big\rangle=0,
\quad{\rm avec}\ 
B_p^{c,d}=\big(c_p^2-\matrice{c}{0}{0}{1}\big)\big(d_p^2-\matrice{1}{0}{0}{d}\big),$$
pour tous $\gamma\in\rho_x$ et $v\in\Pi_S(\rho_x)$ et $c,d\in\cZ^\dual$.
Comme $p\neq 2$, le groupe $\Z_p^\dual$ est procyclique; prenons donc $c\in\Z_p^\dual\subset\cZ^\dual$
un g\'en\'erateur de $\Z_p^\dual$ et $d=c^{-1}$.
Alors $$(B_p^{c,d})^\dual\star(\otimes_{\ell\in S}v_\ell)=\big(\otimes_{\ell\in S\moins\{p\}}v_\ell\big)
\otimes ((B_p^{c,d})^\dual\star v_p).$$
Maintenant, si $\rho_{x,p}$ est irr\'eductible ou, plus g\'en\'eralement, si $\rho_{x,p}$
n'a pas de quotient non nul sur lequel l'inertie agit par $\cy^{-1}$ 
ou par $\cy^2\det\rho_{x,p}$,
le sous-espace engendr\'e par les $(B_p^{c,d})^\dual\star v_p$ est dense dans $\Pi_p(\rho_x)$
(cf.~lemme\,\ref{facto3}).
On en d\'eduit que $(0,\infty)$ est identiquement nul sur $\Pi_S(\rho_x^\dual)$, ce qui contredit
le th\'eor\`eme d'Ash-Stevens~\cite{AS}
 (pour les coefficients constants) dont une cons\'equence est
qu'un sous-espace ferm\'e de $H^1_c(\GG,{\cal C}_L)_S$, 
qui est stable par $\GG(\Q_S)$ (en fait, $\GG(\Z_S)$ suffit)
et tu\'e par $(0,\infty)$, est nul.

Il s'ensuit que les p\^oles \'eventuels de ${\bf z}^S(\rho_T)$ sont de codimension au moins~$2$
(la condition que $\rho_{x,p}$ est non irr\'eductible implique que $x$ est dans le lieu ordinaire
de ${\cal X}$, qui est de codimension~$1$; la condition que l'un des caract\`eres soit comme ci-dessus
est de codimension~$1$ dans le lieu ordinaire car une des dimensions
du lieu ordinaire est la
torsion par un caract\`ere).  Comme ${\cal X}$ est fini et plat au-dessus d'une boule\footnote{C'est un
sous-produit des th\'eor\`emes {\og big $R$ = big $T$\fg}, 
cf.~\cite[th.\,3.1]{boc} et~\cite[\S\,7.3]{Em08}
pour l'extension des m\'ethodes de~\cite{boc}.}
de dimension~$3$, il en r\'esulte que
${\bf z}^S(\rho_T)$ n'a pas de p\^ole.

Ceci prouve le (i).  Le (ii) est alors une cons\'equence du th.\,\ref{fact2}.
\end{proof}

\begin{rema}\phantomsection\label{fact27}
{\rm (i)}
Il r\'esulte du lemme~\ref{facto1}
que la projection naturelle $\rho_T(-1)\otimes_T\rho_T^\diamond\to T\otimes\zeta_{\rm B}^{-1}$
induit une factorisation dans $\Pi_S^\dual(\rho_T(-1))\otimes_T H^1(G_{\Q,S},\rho_T)$:
$${\bf z}_{{\rm Kato}}^{S,c,d}(\rho_T)\otimes\zeta_{\rm B}^{-1}
= (0,\infty)^{c,d}_{T',S}\otimes {\bf z}^S(\rho_T)$$

{\rm (ii)} Une sp\'ecialisation de la factorisation du (i) en un point classique g\'en\'eral $x$
permet de terminer la preuve du th.\,\ref{fact2} (qui n'avait \'et\'e faite que sous
l'hypoth\`ese $x\in{\cal X}^{{\rm cl},+}$).

{\rm (iii)} On peut diviser ${\bf z}_{{\rm Kato}}^{S,c,d}(\rho_T)$ par le facteur
$B^{c,d}_p$ que l'on avait d\^u introduire pour rendre ${\bf z}_{\rm Kato}$ entier 
(cf.~\no\ref{entier}).
En effet, si on pose
$${\bf z}_{{\rm Kato}}^S(\rho_T):=(B_\infty\star (0,\infty)_{T',S})\otimes {\bf z}^S(\rho_T)\otimes\zeta_B,$$
alors, dans $\Pi_S^\dual(\rho_T(-1))\otimes_T H^1(G_{\Q,S},\rho_T)$, on a la relation
${\bf z}^{S,c,d}_{{\rm Kato}}(\rho_T)=B^{c,d}_p\star {\bf z}_{{\rm Kato}}^S(\rho_T)$
gr\^ace \`a la relation $(0,\infty)_{T',S}^{c,d}=B^{c,d}_p\star(B_\infty\star (0,\infty)_{T',S})$
et au (i).

{\rm (iv)}
Posons, \index{Zkatol@\zkatol}comme dans la rem.\,\ref{geni8.5},
$${\bf z}_{{\rm Kato}}^S(\Lambda\otimes\rho_T)=\iota^{\rm semi}({\bf z}_{{\rm Kato}}^S(\rho_T)).$$
On a une factorisation de ${\bf z}_{{\rm Kato}}^S(\Lambda\otimes\rho_T)$
analogue \`a la 
factorisation ci-dessus avec ${\bf z}^S_{\rm Iw}(\rho_T)$
au lieu de ${\bf z}^S(\rho_T)$. Or la localisation en $p$ de
${\bf z}^S_{\rm Iw}(\rho_T)$ s'identifie par construction, modulo les identifications habituelles,
 \`a $(0,\infty)_{T,S}$.
La factorisation ci-dessus devient donc
$${\rm loc}_p({\bf z}_{{\rm Kato}}^S(\Lambda\otimes\rho_T))
=B_\infty\star \big((0,\infty)_{T',S}\otimes (0,\infty)_{T,S}\big),$$
formule que l'on peut voir comme un avatar alg\'ebrique du produit de Rankin.
(Le $\matrice{-1}{0}{0}{1}_S$ intervenant dans $B_\infty$ agit sur les deux facteurs: pour que
les p\'eriodes se combinent pour donner des nombres alg\'ebriques, il faut que les deux
fonctions soient paires ou soient impaires.)
\end{rema}

\printindex


\end{document}